\documentclass[final,noinfo,textrefs,12pt]{nddiss2e}
\usepackage{etex,amsmath,amsthm}
\usepackage{dcpic,pictex}
\usepackage{xspace}
\usepackage{graphicx}
\usepackage{makeidx}
\makeindex

\usepackage{ifthen}
\newboolean{doexecute}
\setboolean{doexecute}{true}

\theoremstyle{definition}
\newtheorem{ass}[equation]{Assumption}
\newtheorem{cor}[equation]{Corollary}
\newtheorem{defn}[equation]{Definition}
\newtheorem{eg}[equation]{Example}

\newtheorem{lem}[equation]{Lemma}
\newtheorem{prop}[equation]{Proposition}
\newtheorem{thm}[equation]{Theorem}

\theoremstyle{remark}

\newtheorem{notat}[equation]{Notation}
\newtheorem{note}[equation]{Note}
 
\newcommand{\binphantom}[1]{\mathbin{\phantom{#1}}}

\newcommand{\Frechet}{Fr\'{e}chet\xspace}

\newcommand{\CC}{\mathbb{C}}      %complex numbers
\newcommand{\QH}{\mathbb{H}}      %quaternions
\newcommand{\qone}{{}\mathbf{1}}  %quaternion 1
\newcommand{\qi}{{}\mathbf{i}}    %quaternion i
\newcommand{\qj}{{}\mathbf{j}}    %quaternion j
\newcommand{\qk}{{}\mathbf{k}}    %quaternion k
\newcommand{\NN}{\mathbb{N}}      %natural numbers
\newcommand{\PP}{\mathbb{P}}      %projective space
\newcommand{\QQ}{\mathbb{Q}}      %rational numbers
\newcommand{\RR}{\mathbb{R}}      %real numbers
\newcommand{\ZZ}{\mathbb{Z}}      %integers

\newcommand{\contract}{\mathrel{\lrcorner}}         %contraction
\newcommand{\cross}{\times}                         %cartesian prod
\newcommand{\dirsum}{\oplus}                        %direct sum
\newcommand{\isomfrom}{\overset{\sim}{\leftarrow}}  %(l) isomorphism
\newcommand{\isomto}{\overset{\sim}{\rightarrow}}   %(r) isomorphism
\newcommand{\tensor}{\otimes}                       %tensor product
\newcommand{\Tensor}{\bigotimes}                    %tensor product

\newcommand{\abs}[1]{\lvert#1 \rvert}               %absolute value
\newcommand{\norm}[1]{\lVert#1 \rVert}              %norm
\newcommand{\st}{\mid}                              %such that

\DeclareMathOperator{\Aut}{Aut}         %group of automorphisms
\DeclareMathOperator{\B}{B}             %bounded morphisms algebra
\DeclareMathOperator{\CH}{\check{H}}    %\v{C}ech (co)homology
\DeclareMathOperator{\cl}{cl}           %closure or Clifford algebra
\DeclareMathOperator{\Cl}{Cl}           %$C^{*}$ Clifford algebra
\DeclareMathOperator{\coker}{coker}     %cokernel
\DeclareMathOperator{\Det}{Det}         %determinant
\DeclareMathOperator{\Diff}{Diff}       %set of diffeomorphisms
\DeclareMathOperator{\ev}{ev}           %evaluation map
\DeclareMathOperator{\F}{F}             %Fock space
\DeclareMathOperator{\GL}{GL}           %general linear group
\DeclareMathOperator{\Gr}{Gr}           %Grassmannian
\DeclareMathOperator{\HH}{H}            %(co)homology
\DeclareMathOperator{\Hom}{Hom}         %set of morphisms
\DeclareMathOperator{\Homeo}{Homeo}     %set of homeomorphisms
\DeclareMathOperator{\Hot}{\pi}         %homotopy group
\DeclareMathOperator{\ident}{id}        %identity morphism
\DeclareMathOperator{\im}{im}           %image = range
\DeclareMathOperator{\incl}{incl}       %inclusion
\DeclareMathOperator{\inv}{inv}         %inverse
\DeclareMathOperator{\Lagr}{Lagr}       %set of Lagrangian subspaces
\DeclareMathOperator{\Mor}{Mor}         %morphisms of a category
\DeclareMathOperator{\Obj}{Obj}         %objects of a category
\DeclareMathOperator{\Orth}{O}          %orthogonal group
\DeclareMathOperator{\pcoker}{Pcoker}   %presheaf cokernel
\DeclareMathOperator{\pim}{Pim}         %presheaf image
\DeclareMathOperator{\PU}{\PP U}        %projective unitary group
\DeclareMathOperator{\rank}{rank}       %rank
\DeclareMathOperator{\scoker}{Scoker}   %sheaf cokernel
\DeclareMathOperator{\sign}{sign}       %+1 arg > 0, -1 arg < 0; 0
\DeclareMathOperator{\Sym}{\mathbb{S}}  %symmetric space
\DeclareMathOperator{\SO}{SO}           %special orthogonal group
\DeclareMathOperator{\Symp}{Sp}         %symplectic group
\DeclareMathOperator{\Spin}{Spin}       %spin group
\DeclareMathOperator{\SU}{SU}           %special unitary group
\DeclareMathOperator{\UU}{U}            %unitary group
\DeclareMathOperator{\US}{\mathbb{U}}   %space of unitary structures
\begin{document}

\frontmatter

\title{A BUNDLE GERBE CONSTRUCTION OF A SPINOR BUNDLE FROM THE SMOOTH FREE LOOP OF A VECTOR BUNDLE}
\author{Stuart Ambler}
\work{Dissertation}
\degaward{%\vskip 24pt
Doctor of Philosophy
%\\in\\Mathematics
}
\advisor{Stephan Stolz}
\department{Mathematics}
\degdate{April 2012}
\maketitle

\copyrightholder{Stuart Ambler}
\copyrightyear{2012}
\makecopyright

\begin{abstract}
A bundle gerbe is constructed from an oriented smooth vector bundle of even rank with a fiberwise inner product, over a compact connected orientable smooth manifold with Riemannian metric.  From a trivialization of the bundle gerbe is constructed an irreducible Clifford module bundle, a spinor bundle over the smooth free loop space of the manifold.

First, a Clifford algebra bundle over the loop space is constructed from the vector bundle.  A polarization class bundle is constructed, choosing continuously over each point of the loop space a polarization class of Lagrangian subspaces of the complexification of the real vector space from which the Clifford algebra is made.  Being unable to choose a Lagrangian subspace continuously from the polarization class over each point, the thesis constructs a bundle gerbe over the loop space of the base manifold to encode over each loop all such subspaces, along with the isomorphisms between the Fock spaces made from them, resulting from their being in the same polarization class.

The vanishing of the Dixmier-Douady class of the bundle gerbe implies that the latter has a trivialization, from which is constructed a spinor bundle.
\end{abstract}

\renewcommand{\dedicationname}{Dedication}
\begin{dedication}
To whoever reads it.
\end{dedication}

\setcounter{tocdepth}{2}
\tableofcontents

\begin{symbols}
  \sym{\CC}{complex numbers}
  \sym{\CC^{\cross}}{nonzero complex numbers}
  \sym{\QH}{quaternions}
  \sym{\NN}{natural numbers}
% \sym{\PP}{projective space}
  \sym{\QQ}{rational numbers}
  \sym{\RR}{real numbers}
  \sym{\ZZ}{integers}
  \sym{\ZZ_n}{integers modulo $n$}

  \sym{\contract}{contraction}
  \sym{\cross}{cartesian product}
  \sym{\dirsum}{direct sum}
  \sym{\cong}{isomorphic}
  \sym{\isomfrom}{isomorphism}
  \sym{\isomto}{isomorphism}
  \sym{\tensor}{tensor product}

  \sym{\abs{z}}{absolute value of $z$}
  \sym{\norm{v}}{norm of $v$}
  \sym{(, )}{real inner product}
  \sym{\langle , \rangle}{usually complex, Hermitian inner product}
  \sym{\langle , \rangle_J}{Hermitian inner product on $V_J$}
  \sym{\st}{such that}
  \sym{[\quad]}{often used to denote an equivalence class}
  \sym{[1,n]}{the set of integers from $1$ to $n$}

  \sym{A_g}{part of $g$ that anticommutes with (implied) $J$}
  \sym{\Aut(A)}{group of automorphisms of $A$}
  \sym{\B(A)}{set of bounded morphisms of $A$}
  \sym{B_{res, J}}{restricted bounded operators}
  \sym{\B(A,C)}{set of bounded morphisms from $A$ to $C$}
  \sym{\B_{\delta}(x)}{open ball of radius $\delta$ about $x$}
  \sym{C_g}{part of $g$ that commutes with (implied) $J$}
  \sym{C(X,Y)}{set of continuous maps from $X$ to $Y$}
  \sym{C^{\infty}(X,Y)}{set of smooth maps from $X$ to $Y$}
  \sym{c_1 (Q)}{first Chern class of $Q$}
  \sym{\cl(V)}{uncompleted Clifford algebra of $V$}
  \sym{\Cl(V)}{$C^{*}$ Clifford algebra of $V$}
  \sym{\coker(T)}{cokernel of $T$}
  \sym{\delta}{e.g. coboundary operator}
  \sym{DD(P, Y, X)}{Dixmier-Douady class of bundle gerbe (P, Y, X)}
  \sym{\Det(T)}{determinant of $T$}
  \sym{\Diff(X,Y)}{diffeomorphisms $X \rightarrow Y$}
  \sym{\dim(V)}{dimension of $V$}
  \sym{E}{a real vector bundle}
  \sym{E_Q}{tautological quaternionic vector bundle, viewed as real}
  \sym{F}{standard Fock space bundle}
  \sym{\F(L)}{Fock space constructed from $L$}
  \sym{FY}{Fock space bundle}
  \sym{\gamma}{e.g. loop in $M$}
  \sym{\Gamma(P)}{set of sections of $P$}
  \sym{\GL(n)}{general linear group of $\RR^n$}
  \sym{H}{e.g. Hilbert space, complexification of $V$}
  \sym{\HH_p(X;G)}{$p$-th homology of $X$, coeff. in $G$}
  \sym{\HH^q(X;G)}{$q$-th cohomology of $X$, coeff. in $G$}
  \sym{\CH^q(X;G)}{$q$-th \v{C}ech cohomology of $X$, coeff. in $G$}
  \sym{H_{J,\pm i}}{$\pm i$ eigenspaces of $J$}
  \sym{\Hom(A)}{set of endomorphisms of $A$}
  \sym{\Hom(A,B)}{set of morphisms from $A$ to $B$}
  \sym{\Homeo(X,Y)}{homeomorphisms $X \rightarrow Y$}
  \sym{\Hot_p(X)}{$p$-th homotopy group of $X$}
  \sym{\iota}{e.g. identity section}
  \sym{\ident}{identity morphism, $\ident(x) = x$}
  \sym{\Im(z)}{imaginary part of $z$}
  \sym{\im(T)}{image = range of $T$}
  \sym{\incl}{inclusion morphism, $\incl(x) = x$}
  \sym{\inv}{inverse, $\inv(x) = x^{-1}$}
  \sym{J}{a unitary structure}
  \sym{\ker(T)}{kernel of $T$}
  \sym{L}{Lagrangian subspace}
  \sym{L X}{free loop of $X$}
  \sym{\Lagr(H, \Sigma)}{Lagrangian subspaces of $H$ with real structure $\Sigma$}
  \sym{\Lambda^k(L)}{$k$-th exterior product of $L$}
  \sym{M}{manifold; e.g., basis for $E$}
  \sym{m}{e.g., dimension of $M$}
  \sym{\Mor(A,B)}{morphisms from $A$ to $B$}
  \sym{n}{e.g., rank of $E$}
  \sym{\Obj(\mathcal{C})}{objects of a category $\mathcal{C}$}
  \sym{\Omega X}{based loop of $X$}
  \sym{\Omega}{vacuum vector}
  \sym{\Orth(V)}{orthogonal group of $V$}
  \sym{\Orth_{res}(V)}{restricted orthogonal group of $V$}
  \sym{P}{e.g. principal bundle, in particular of a bundle gerbe}
  \sym{p_1 (E)}{first Pontryagin class of $E$}
  \sym{Presh}{category of presheaves}
  \sym{Presh/X}{category of presheaves over $X$}
  \sym{\PU(V)}{projective unitary group of $V$}
  \sym{\QH P^1}{quaternionic projective space}
  \sym{\rank(T)}{rank of $T$}
  \sym{\Re(z)}{real part of $z \in \CC$ or $\in \QH$}
  \sym{\Sigma}{real structure; cx. conj. on complexified real vector space}
  \sym{\Sigma X}{unreduced suspension of $X$}
  \sym{S X}{reduced suspension of $X$}
  \sym{\Sym(\Sigma(L))}{space of operators with a skew-symmetry property}
  \sym{Shv/X}{category of sheaves over $X$}
  \sym{\sign(r)}{$+1$ if $r > 0$, $-1$ if $r < 0$, else $0$}
  \sym{\sim}{used for equivalence relations}
  \sym{\SO(n)}{special orthogonal group of $\RR^n$}
  \sym{\SO(E)}{oriented orthogonal frame bundle of $E$}
  \sym{\Spin(n)}{spin group of $\RR^n$}
  \sym{\Spin(E)}{spin bundle (structure) for $E$}
  \sym{\Symp(n)}{symplectic group of $\QH^n$}
  \sym{\Symp(E)}{symplectic frame bundle of $E$}
  \sym{\SU(n)}{special unitary group of $\CC^n$}
  \sym{\tau}{e.g., transgression or transposition}
  \sym{\theta_g}{Bogoliubov automorphism}
  \sym{\tensor^k(V)}{$k$-th tensor product of $V$}
  \sym{T}{e.g. linear operator, intertwiner, standard intertwiner bundle}
  \sym{\UU(n)}{unitary group of $\CC^n$}
  \sym{\UU(V)}{unitary group of $V$}
  \sym{\US(V)}{set of unitary structures of $V$}
  \sym{\US_{res}(V)}{restricted set of unitary structures of $V$}
  \sym{V}{generally, a real vector space}
  \sym{V_J}{$V$ viewed as complex using $J$}
  \sym{Y}{e.g. $Y$ space of a bundle gerbe}
  \sym{note \colon}{also see the index}
\end{symbols}

\begin{preface}
This document is intended for a person with the background of a mathematics graduate student who has taken basic year courses in algebra, point set and algebraic topology, real and functional analysis, and has some acquaintance with differential geometry and Lie groups.  It is not a place to start learning about manifolds or bundles, but it contains a lot of material unnecessary for experts.

Beyond a more elementary level, it tries to give full statements of definitions and results from other sources, and a few details about translation from different conventions in the sources, though most proofs of facts from the sources are omitted.  Some material that isn't logically necessary is included for greater understandability, internal coherence, or interest.

A number of global assumptions and conventions are collected, first at the beginning of the document, and then at the beginnings of a few chapters or sections.  The intent is to make them as clear and visible as possible, avoiding some need for constantly referring to other places in the text that could be difficult to find.

I might have liked to tell a story, but the typical mathematical definitions, lemmas, and theorems, presented in a logical sequence, necessitate finding motivation in comments, from reading further, or from oneself.

My apologies for errors, and please let me know.  I'd also like to hear from you if you found this thesis useful.  One way to contact me is via \url{www.zulazon.com}.
\end{preface}

%\setboolean{doexecute}{false}
\setboolean{doexecute}{true}
\ifthenelse{\boolean{doexecute}}{

\begin{acknowledge}
My debt to my adviser, Stephan Stolz, for help with this thesis is very great.  It was a long thing to read, longer than either of us wanted.  Many thanks to the readers, Laurence Taylor, Bruce Williams, and Liviu Nicolaescu, for reading and for conversations.  Daniel Cibotaru gave me some of the ideas of the proof of lemma \ref{l-ev-star-h4s4-h4os}.  Discussions with Ryan Grady and Steven Broad were helpful.

Some of the basic ideas of the thesis came from my adviser:  for instance, constructing as concretely as possible a spinor bundle on the loop space of a manifold, using a bundle gerbe, using associated bundles; and for what has ended up as chapter \ref{c-futu-work}, relating the transgression of the first Pontryagin class to the Dixmier-Douady class, that this in general could be a consequence of the same relation starting with the tautological quaternionic line bundle, that there could be a connection to a Pfaffian line bundle, and that a spin structure on the original vector bundle might be helpful.

My stay in graduate school certainly importantly involved other things than mathematics, but this is not the place to expound on them.  However, let me thank, in addition to the people already mentioned and my long-distance friends, some fellow graduate students, undergraduates, professors in and out of the department, University officials and employees, Catholic priests, and South Bend residents.  Also, I'll very much miss Notre Dame's excellent athletic facilities.
\end{acknowledge}

}{}
\setboolean{doexecute}{true}

\mainmatter

\chapter{INTRODUCTION}\label{c-intr}

As a graduate student past the first or second year, the intended reader of the preface may not have the background to understand all the terms used in the abstract.  The introduction in section \ref{s-intr} provides a way to understand something about the problem to be solved in the thesis, without getting into technicalities.

Section \ref{s-intr-note} contains a few examples of things in the thesis that might be useful to someone aside from the solution to the main problem, a brief listing summarizing the contents of the other chapters of the thesis, and a note on citation of references.

After the introduction is a more technical overview of the overall plan, in section \ref{s-tech-over}.

\section{Introduction}\label{s-intr}

The story begins with spinors.  A student of physics may be introduced to spinors as $4$-tuples of complex valued functions, meeting them first, as used for solutions to the Dirac equation of quantum field theory for a free electron.  This is a partial differential equation whose coefficients include factors that are $4 \cross 4$ constant matrices, elements of a Clifford algebra.

Starting from a vector space with a bilinear form, which in this thesis will be an inner product, Clifford algebras can be thought of as giving a way to multiply vectors, obtaining scalars, vectors, and indicated products of vectors:  the Clifford algebra.  The convention of the thesis is that in a Clifford algebra, the square of a vector equals the square of its norm, a real number.

Clifford algebras, exterior algebras, and quaternions, all of which are widely used in geometric investigations, play roles in this thesis.  The Clifford algebras here start from real vector spaces, but the resulting algebras are complex.  They are $C^{*}$-algebras.

Representations of Clifford algebras may be thought of as a ways to see these algebras in the $C^{*}$-algebras of bounded operators on complex Hilbert spaces.  A prototypical irreducible representation, which the thesis will use, is a Fock representation.  A Fock space is built by choosing a Lagrangian subspace of the complexification of the vector space the Clifford algebra is built on, taking the exterior algebra of that Lagrangian subspace, and making it into a Hilbert space.  A Lagrangian subspace in the thesis is a subspace whose orthogonal complement is its complex conjugate.

To get an idea how Fock representations work, start by noticing that both the Clifford algebra and the exterior algebra are made by some kind of multiplication of vectors.  By thinking of Clifford algebra elements as built up from vectors, one can make them act on exterior algebra elements by using the wedge product and by using a contraction operator.  From this is constructed an algebra homomorphism from the Clifford algebra to the bounded operators on the Fock space.  A Fock representation is irreducible; the Hilbert space is no larger than it need be.  Elements of a Fock space are examples of spinors.

The next step in understanding the abstract is to make similar constructions in bundles.  Hoping it's easier to understand at first, we start with a finite dimensional case (no loops), although the thesis didn't develop that way.  The introduction being written when the thesis is nearly complete, in retrospect it looks like it would be possible to use very close analogs to the methods of the thesis, for the finite-dimensional case about to be described; but this development is not carried out in the thesis, so it should be taken as an introductory analogy.

Over a point $x$ of a smooth manifold $M$, the tangent space $T_x M$ is a real vector space, and if $M$ is a Riemannian manifold, there is an inner product on $T_x M$.  Thus there is a Clifford algebra $\Cl(T_x M)$ over $x$; and over varying $x$, these form a Clifford algebra bundle, a fiber bundle.  Similarly, given a vector bundle $E$ over $M$ with fiberwise inner product, over each $x$ we can build a Clifford algebra $\Cl(E_x)$ from the fiber $E_x$, and get a Clifford algebra bundle.

Over each $x$ we can also pick a Lagrangian subspace $L$ of $\CC \tensor E_x$, supposing that $E$ is of even rank, and form the corresponding Fock space $\F(L)$, obtaining a Fock representation of $\Cl(E_x)$ on $\F(L)$.  However, choosing a Lagrangian subspace is not so easy to do continuously as $x$ varies.  It's possible to choose continuously a set of Lagrangian subspaces, called a polarization class, each element of which results in an equivalent Fock representation, but choosing one Lagrangian subspace out of the polarization class, is in general not possible to do continuously.

The idea of the thesis would construct an object called a bundle gerbe over $M$, which encodes the entire polarization class of Lagrangian subspaces over each point, and all the isomorphisms between the resulting equivalent Fock representations.  To each bundle gerbe over $M$ is associated a cohomology class in $\HH^3 (M; \ZZ)$, the Dixmier-Douady class.  Triviality of the bundle gerbe is equivalent to the vanishing of the Dixmier-Douady class.

When the bundle gerbe is trivial, the data in the bundle gerbe, together with the data of a trivialization, can be used to construct an irreducible representation of $\Cl(E_x)$ over each $x$, in such a way that the Hilbert spaces $S_x$ for the representation fit together to form a fiber bundle $S$, which we call a spinor bundle.

Now to move to the loop spaces the thesis is concerned with.  Starting with a smooth manifold $M$, we ``loop it'' to get its smooth free loop space $LM$, the set of smooth functions $S^1 \rightarrow M$.  Although infinite-dimensional, $LM$ is a kind of manifold, a \Frechet manifold, with local trivializations taking values in \Frechet spaces.

Given a smooth map $M \xrightarrow{f} N$ of smooth manifolds, we can loop the whole thing to get a smooth map of \Frechet manifolds, $LM \xrightarrow{Lf} LN$; in fact smooth looping is a functor.  Further, given some conditions, we can loop a smooth principal bundle $P \rightarrow M$ to get a \Frechet principal bundle $LP \rightarrow LM$, and similarly with associated bundles.  (At about this point, to avoid technical difficulties, since the construction of the thesis didn't depend on smoothness of the bundle gerbes, it starts working with topological rather than smooth objects.)

The thesis starts with a vector bundle $E \rightarrow M$, loops that to get $LE \rightarrow LM$, constructs the Clifford algebra bundle $\Cl(LE)$ over $LM$, and constructs a bundle gerbe using a continuous choice, for all $\gamma \in LM$, of a polarization class of $\CC \tensor LE_{\gamma}$.  The bundle gerbe encodes also, all the isomorphisms between equivalent Fock representations built from all the Lagrangian subspaces in that polarization class over $\gamma$, and does this in a continuous way for all $\gamma$.

Finally, if the bundle gerbe is trivial, the thesis uses the data of the bundle gerbe and of a trivialization, to construct a spinor bundle $S$ over $LM$.

\section{Introductory Notes}\label{s-intr-note}

Besides a solution to the main problem, the thesis also contains exposition that was difficult to find in the literature in the form needed.  Examples are continuity of some constructions involving Clifford algebras and Fock representations, some details relating to bundle gerbes, and an exposition of Cech cohomology with $\UU(1)$ sheaf coefficients, that includes the map in cohomology induced by a map of spaces.  Although certainly not a textbook, the thesis could be useful for some of this material.

Another subsidiary item that might be of independent use is proposition \ref{p-dd-bg-susp-c1-pb}.  From a bundle gerbe over the unreduced suspension $\Sigma X$ of a space $X$, this proposition constructs a principal $\UU(1)$ bundle over the space $X$ itself; and the first Chern class of the principal bundle is isomorphic via the suspension isomorphism to the Dixmier-Douady class of the bundle gerbe.

Other things that might be useful may be found by browsing the table of contents and the thesis itself.  Chapter \ref{c-con-text} is background on bundles, \Frechet spaces and manifolds, loop spaces, and point set topology.  Chapter \ref{c-loop-bndl} is about loops of bundles.  Chapter \ref{c-clif-alg-fock-rep} is on Clifford algebras and Fock representations.  Chapter \ref{c-calg-bndl} is about the Clifford algebra bundle over $LM$, including some point set topology.  Chapter \ref{c-res-orth-grp} goes a little deeper into Fock representations and answers more point set topological questions.  Chapter \ref{c-pol-clas-bndl} defines the bundle that has a polarization class over each point in $LM$, and gives some translation between a more geometric viewpoint and the associated bundle methods used in the thesis.  Chapter \ref{c-std-fock-spac-bndl} is about a standard Fock space bundle, not over $LM$, but over a polarization class.  Chapter \ref{c-fock-spac-bndl} is another Fock space bundle, not over $LM$, but over the total space of the polarization class bundle.  Chapter \ref{c-cont-bndl-gerb} is about the continuous bundle gerbes used in the thesis, with background on $\UU(1)$ torsors and principal bundles, sheaf theory and \v{C}ech cohomology, and detailed proofs of a number of properties of bundle gerbes.  Chapter \ref{c-cnst-bndl-gerb} gives the construction of the particular bundle gerbe used in the thesis, and chapter \ref{c-bndl-gerb-cnst-func} shows that the construction is functorial.  Chapter \ref{c-clif-alg-modu-bndl} gives the main result about the bundle gerbe constructed in the thesis, constructing from the bundle gerbe and its trivialization a spinor bundle.  Chapter \ref{c-futu-work} contains notes for possible further work, mainly the possibility of using a spin structure for the original vector bundle, and the relating of its first Pontryagin class to the bundle gerbe's Dixmier-Douady class.

Please note that originality is not necessarily claimed for items without a cited reference.  Effort was taken to provide many citations, but some things were considered too small, or part of general mathematical knowledge, and some things were not found in references, or not in the form needed.  As an example, mentioned shortly before, the thesis proves continuity of some constructions involving Clifford algebras and Fock representations.  The constructions were found in the references but the continuity was not.  The continuity statements surely would not be news to an expert in that area.

\section{A More Technical Overview}\label{s-tech-over}

The intent of this section is to help the reader understand in a more technical way, the overall plan used to come up with the various bundle constructions that follow.  Also, it contains a brief introduction to associated bundles, as used in the thesis.

\index{spinors}
\index{ClLE@$\Cl(LE)$}
\index{Clifford algebra bundle}
\index{bundle!Clifford algebra}
\index{LSOE@$L\SO(E)$}
\index{orthonormal frame bundle}
\index{bundle!orthonormal frame}
Since the goal of the thesis is to construct a spinor bundle over the loop space $LM$ of the base of the vector bundle $E \rightarrow M$ of assumption \ref{a-mfld-vb}, we first need the Clifford algebra bundle $\Cl(LE)$ over $LM$; over each $\gamma \in LM$, $\Cl(LE_{\gamma})$.

We will not use this directly as the definition, but instead use an associated bundle construction.  In the development of the thesis, using associated bundle constructions allowed us not to confront some technical difficulties, and they make the functoriality of the bundle gerbe construction natural to see.

\index{associated bundle}
\index{bundle!associated}
\index{LE@$LE$}
\index{LRn@$L\RR^n$}
\index{frame!orthonormal!oriented}
\index{oriented orthonormal frame}
\index{LSOn@$L\SO(n)$}
\index{LSOE@$L\SO(E)$}
\index{bundle!oriented orthonormal frame}
For those not very familiar with associated bundles, here is a brief introduction to their use in the thesis.  As in example \ref{e-loop-soe}, given a fiber $LE_{\gamma} = \Gamma(\gamma^{*}E)$, the associated bundle construction for $LE$ refers analysis from $\Gamma(\gamma^{*}E)$ to the standard \Frechet space $L\RR^n = C^{\infty}(S^1, \RR^n)$, using an oriented orthonormal frame (for each $x \in M$, an orientation preserving orthogonal map $\RR^n \rightarrow E_x$) that moves along or above $\gamma$; or in other words, $\widetilde{\gamma} \in L\SO(E)$ over $\gamma \in LM$.  Picking a loop $\sigma$ in $E$ over $\gamma$, for every oriented orthonormal frame $\psi$ moving above $\gamma$, $t \mapsto (\psi(t))^{-1} (\sigma(t))$ is a function $v \colon S^1 \rightarrow \RR^n$.  The inverse $v \mapsto \sigma$ also is straightforward.

Rather than trying to choose a particular moving frame by which to refer analysis, we allow all of them, using the associated bundle construction $L\SO(E) \cross_{L\SO(n)} L\RR^n$.  The $\cross_{L\SO(n)}$ serves to remove the redundancy caused by the multiple possible moving frames.  The group $L\SO(n)$ acts freely and transitively on $L\SO(E)$, so along a given loop $\gamma$, there are as many moving frames as elements of $L\SO(n)$, and considering the mappings by all of them is equivalent to considering the mapping by one and acting on this by precomposing with all elements of $L\SO(n)$.  This is a \Frechet fiber bundle, though we won't use its smooth structure.

We will use similar associated topological bundle constructions to take a standard or model object related to $L\RR^n$ or its Hilbert space completion $L^2(S^1, \RR^n)$, with a continuous action of $L\SO(n)$ on the standard or model space, and make a bundle associated to $L\SO(E)$, thereby related to $LE$.

To answer a possible question, our use of the oriented orthonormal frame bundle of $E$, which has a connected structure group $\SO(n)$ and hence can be looped as mentioned following proposition \ref{p-loop-smth-pb-is-frec-pb}, is the reason for the assumption that $E$ is oriented.  (The alternative condition of that proposition, that $M$ be one-connected, implies that $E$ is orientable.)

Continuing now with the plan, make $\Cl(LE)$ an associated bundle for $L\SO(E)$, with fiber a standard Clifford algebra.  This construction is in chapter \ref{c-calg-bndl}.

\index{Clifford module bundle}
\index{bundle!Clifford module}
\index{Fock space}
\index{Ores@$\Orth_{res}$}
\index{restricted Orthogonal group}
\index{Orthogonal group!restricted}
\index{Lagrres@$\Lagr_{res}$}
\index{standard polarization class}
\index{polarization class!standard}
\index{Y@$Y$}
\index{polarization class bundle}
\index{bundle!polarization class}
Though our spinor bundle, or irreducible Clifford module bundle, is not a bundle of Fock spaces, yet over each $\gamma$ it will be built from Fock spaces that are built upon Lagrangian subspaces of a Hilbert space made from the fiber $LE_{\gamma}$.  This isn't done in those terms, but rather using associated bundles.

To relate different Lagrangian subspaces of a standard Hilbert space $H$ made from the standard \Frechet space for $LE$, chapter \ref{c-res-orth-grp} discusses the restricted orthogonal group.

In chapter \ref{c-pol-clas-bndl} we will create a bundle $Y$ over $LM$, giving a fixed polarization class of Lagrangian subspaces over each point $\gamma \in LM$; i.e., over each loop $\gamma$.  The polarization class bundle $Y$ is again made as a bundle associated to $L\SO(E)$, with fiber $\Lagr_{res}$, the standard polarization class we will fix in definition \ref{d-lagr-res}.

\index{F@$F$}
\index{standard Fock space bundle}
\index{bundle!standard Fock space}
\index{bundle!Fock space!standard}
\index{Fock space!standard bundle}
We then retreat a moment from bundles over $LM$ and create in chapter \ref{c-std-fock-spac-bndl} a standard Fock space bundle $F$ for intermediate use, a bundle of Fock spaces, one over each Lagrangian subspace in the polarization class $\Lagr_{res}$.  This is not created using an associated bundle construction.

\index{FY@$FY$}
\index{Fock space!bundle}
\index{bundle!Fock space}
In chapter \ref{c-fock-spac-bndl} we construct the Fock space bundle $FY$ over the polarization class bundle $Y$ (not over $LM$).  Each fiber of the Fock space bundle $FY$ over each point of the fiber of the polarization class bundle $Y$ over a particular loop $\gamma \in LM$, is a module for the fiber over $\gamma$ of the Clifford algebra bundle $\Cl(LE)$.  It is another bundle associated to $L\SO(E)$, with fiber over each Lagrangian subspace being the Fock space associated to that subspace.

\index{bundle gerbe!continuous}
Since the intertwiners between two equivalent Fock spaces form a $\UU(1)$ torsor, one can speak of the phase difference between two such intertwiners.  After defining continuous bundle gerbes (with ``band $\UU(1)$'') in chapter \ref{c-cont-bndl-gerb} we construct in chapter \ref{c-cnst-bndl-gerb} a particular bundle gerbe, which will eventually allow us to give a natural condition allowing us to relate $\UU(1)$ phases of intertwiners so as to get an irreducible Clifford module bundle, a spinor bundle $S$ over all of $LM$.  This bundle gerbe construction involves three more bundles.  One is a fiber product $Y^{[2]}$ of the polarization class bundle with itself.

\index{T@$T$}
\index{intertwiner bundle!standard}
\index{bundle!standard intertwiner}
\index{standard intertwiner bundle}
Another, the standard intertwiner bundle $T$, is a principal bundle over the cartesian product of the standard polarization class with itself, $\Lagr_{res} \cross \Lagr_{res}$, with total space consisting of pairs of Lagrangian subspaces and intertwiners between Fock spaces corresponding to them.

\index{P@$P$}
The fiber of the third, $P$, over a point in $Y^{[2]}$, is the set of intertwiners between two Fock spaces built from the Lagrangian subspaces that are points of $Y$.  It is constructed as a bundle associated to $L\SO(E)$, with fiber the total space of the standard intertwiner bundle $T$, projecting to the first of the three bundles, $Y^{[2]}$.

That's it for the bundles needed to construct the bundle gerbe.  In chapter \ref{c-clif-alg-modu-bndl} we construct the desired Clifford module bundle $S$ from a trivialization of the bundle gerbe.  We show in chapter \ref{c-bndl-gerb-cnst-func} that the bundle gerbe construction is functorial and that it satisfies a stability property.

\index{Dixmier-Douady class}
\index{Pontryagin class}
Chapter \ref{c-futu-work} briefly presents the idea of starting with a vector bundle that has a spin structure, then building the bundle gerbe from bundles associated to $L\Spin(E)$ rather than $L\SO(E)$, allowing $L\Spin(n)$ to act on the other factors of the associated bundles through $L\SO(n)$.  This should allow the spinor bundle constructed at the end, to be $\ZZ_2$ graded.  Also, it likely would allow the proof of the conjecture that the transgression of the first Pontryagin class of the vector bundle, is plus or minus twice the Dixmier-Douady class of the bundle gerbe, if the conjecture is true for a particular vector bundle that thus could be thought of as the universal case.  The chapter also discusses and presents some results toward a proof of the conjecture in the universal case.

\chapter{CONTEXT}\label{c-con-text}

This background chapter defines and gives notation for the context of the thesis, and presents results on the \Frechet topology the thesis will use for the smooth free loop space and various bundles; in particular looping a smooth manifold to get a \Frechet manifold.  The individual items are here for various reasons, one being direct use in the thesis, another to clarify which common variant definitions or requirements we use, and to present a logical sequence of the most important definitions and results about \Frechet manifolds supporting what we use.

\section{Assumptions}\label{s-assu}

First, a few general assumptions.
\tolerance=300
\begin{ass}\label{a-boun}
\index{operators}
\index{assumptions!operators bounded}
(Linear Operators on Normed Linear Spaces are Bounded).
Unless otherwise specified, terms such as linear transformations, linear operators, or operators on a Banach or Hilbert space mean bounded, or equivalently continuous operators.  For such a space $H$, the set of such is denoted $\B(H)$.
\end{ass}

\begin{ass}\label{a-herm-innr-prod}
\index{Hermitian inner product}
\index{assumptions!Hermitian inner product}
(Hermitian Inner Products).
A Hermitian inner product will be complex-linear in its first argument and antilinear in its second.
\end{ass}

\begin{ass}\label{a-mfld}
\index{manifolds}
\index{assumptions!manifolds}
(Manifolds).
All manifolds are assumed Hausdorff and without boundary.  The \emph{unqualified} term smooth manifold implies second countable and finite-dimensional.  We will define \Frechet manifolds as in \citet[page~85]{Hami82}; they may be infinite-dimensional smooth manifolds. The unqualified term diffeomorphism will imply smooth.
\end{ass}
The second countability of a smooth manifold implies paracompactness, and thus by the Smirnov metrization theorem \citep[page~261]{Munk00}, metrizability.

\index{canonical}
\index{natural}
``Canonical'' generally is used to indicate preferred, natural, made without arbitrary choice, or sometimes well-known or called canonical by others.  ``Natural'' sometimes is used in the technical sense of natural transformation, but other times it may have only the ordinary English meaning.

\section{Bundles}\label{s-bundles}

\begin{ass}\label{a-bndl}
\index{bundles}
\index{assumptions!bundles}
(Bundles).
We assume throughout that smooth principal bundles, smooth vector bundles, and in general smooth fiber bundles are finite-dimensional smooth locally trivial bundles over smooth manifolds.  The projection maps are surjections.  The structure group of a smooth principal bundle is assumed to be a finite-dimensional Lie group.  Without the adjective smooth, the terms refer to continuous bundles involving topological groups and spaces, that are not required to be manifolds.  \Frechet vector bundles, Lie groups, and principal bundles are as defined in \citet{Hami82}.  Isomorphisms of bundles over the same base space are assumed to be over the identity unless otherwise noted, and the same for bundle gerbes; see assumption \ref{a-bndl-gerb}.  We use a constructive definition of pullback of a bundle rather than defining it only up to unique isomorphism using the universal property: given $\pi \colon P \rightarrow B$ and $f \colon X \rightarrow B$, $f^{*} P = \{ (x, p) \st f(x) = \pi (p) \}$.
\end{ass}
To reiterate, our convention is that a smooth bundle is finite-dimensional.  The infinite-dimensional bundles we deal with that could be called smooth if we had a different convention, we call \Frechet.

\begin{defn}\label{d-fb-vb}
\index{fiber bundle}
\index{bundle!fiber}
\index{vector bundle}
\index{bundle!vector}
(Fiber and Vector Bundles).
A topological fiber bundle is a morphism of topological spaces $\pi \colon T \rightarrow B$, with an open covering $\{ U_i \}$ of $B$ and a set of homeomorphisms $\phi_i \colon \pi^{-1}(U_i) \isomto U_i \cross F_i$, where the $F_i$ are topological spaces; such that $\pi_1 \colon \phi_i = \pi_{|U_i}$.  (Thus $\pi$ is surjective.)  We call $T$ the total space, $B$ the base space, $\pi$ the projection, $F_i$ the fibers (not to be confused with the fiber $\pi^{-1}(b)$ for $b \in B$), and either the pairs $(U_i, \phi_s)$ or just the maps $\phi_i$ are called local trivializations, coordinate charts, or charts.  It is said to have standard fiber $F$ if all the $F_i = F$.
A smooth fiber bundle is a topological fiber bundle for which all the spaces are smooth manifolds and the local trivializations are diffeomorphisms.

A topological or smooth vector bundle is a topological or smooth fiber bundle with standard fiber a vector space, and local trivializations that over each point of the base are vector space isomorphisms.  We don't require that the spaces for topological vector bundles be manifolds.

Given $U_i \cap U_j \ne \emptyset$, define the change of coordinates map $\phi_{ji} \colon (U_i \cap U_j) \cross F_i \rightarrow (U_i \cap U_j) \cross F_j$ by $\phi_{ji} = \phi_j \circ \phi_i^{-1}$, and the transition function $\phi_{ji,2} = \pi_2 \circ \phi_{ji} \colon (U_i \cap U_j) \cross F_i \rightarrow F_j$.

Morphisms of fiber bundles are maps of total spaces that are continuous or smooth as appropriate, and that cover respectively continuous or smooth maps of base spaces (i.e. commute with the projections).  Morphisms of vector bundles are also fiberwise linear maps.
\end{defn}
Our topological fiber bundles with standard fiber are the same as the locally trivial bundles with (standard) fiber of \citet[pages~9--11]{HJJS08}.  \citet[pages~1--3,~12]{Poor07} defines smooth fiber and vector bundles.  \citet[page~24]{HJJS08} has a somewhat broken-up definition of topological vector bundle.

\begin{note}\label{n-tran-func-fb-vb}
\index{transition functions}
(Fiber and Vector Bundle Transition Functions).
As a consequence of the definition, the change of coordinates maps and transition functions are continuous or smooth as appropriate. Note that transition functions do not refer to maps $U_i \cap U_j \rightarrow \Homeo(F_i, F_j) \text{ or } \Diff(F_i, F_j)$; continuity or smoothness of these maps can be a different question, depending on circumstances.
\end{note}

\begin{lem}\label{l-fb-cons}
\index{fiber bundle!construction}
\index{bundle!fiber!construction}
(Construction of Standard Fiber, Fiber and Vector Bundles).
Given a standard fiber and set of transition functions suitable for a fiber or vector bundle, satisfying a cocycle condition, there is a functorial construction of a fiber or vector bundle from them, that has them as its transition functions.  Furthermore, there is a naturally constructed bijection between equivalence classes of transition functions and isomorphism classes of bundles. Thus, if one can choose in a natural way, transition functions using the same open covering of the same base space for two fiber or vector bundles with standard fiber, and if the transition functions from the two bundles are equal, then the bundles are isomorphic in a natural way.

Given the elements of a topological fiber or vector bundle except no topology on the total space, and consequently the local trivializations being just bijections, if the change of coordinates maps or equivalently the transition functions are continuous, then a unique topology can be defined on the total space resulting in a topological fiber or vector bundle.  The analogous facts are true for a smooth fiber  or vector bundle, when what is lacking is a smooth structure and possibly also a topology on the total space.
\end{lem}
\begin{proof}
For the first paragraph, \citet[pages~78--80]{DK01} address the continuous case, where our fiber bundle is their locally trivial bundle and their definition of transition functions in exercise 57 needs to be changed to ours (see note \ref{n-tran-func-fb-vb}).  Their note that a locally trivial bundle is the same as a fiber bundle with group $\Homeo(F)$ depends on assumptions we do not make.  \citet[page~108--109,~121]{Lee03} (problem 5-4) covers the continuous case for vector bundles.  Although he defines transition functions differently than we do, his proof doesn't use the continuity of his version of them, only that of ours. \citet[page~100--101,~200]{Mich08} discusses the smooth case; the same comment as for the previous reference applies to his discussion of vector bundles.

If the transition functions take values in a topological group, for the first paragraph, \citet[pages~191--195]{Swit75} has a thorough discussion of the equivalence of isomorphism classes of continuous principal bundles and their isomorphism classes of sets of continuous transition functions, and talks about the connection with vector bundles.

For the smooth case of the second paragraph, see \citet[page~3]{Poor07} for the fiber bundle case.
\end{proof}

\begin{defn}\label{d-grp-act}
\index{group!action}
\index{action!group}
(Action of a Topological Group).
Given a topological group $G$ (a group with a Hausdorff topology under which inversion and multiplication are continuous), a continuous left action of $G$ on a topological space $X$ is a map
\begin{align}
\rho \colon G \cross X &\rightarrow X, \text{ sometimes indicated by } \notag \\
(g, x) &\mapsto g x , \notag
\end{align}
such that for $e$ the identity of $G$ and every $g, h \in G$ and $x \in X$,
\begin{align}
\rho (e, x) &= x \notag \\
\rho (g, \rho (h, x)) &= \rho (g h , x). \notag
\end{align}
A continuous right action is defined similarly, with
\begin{align}
\rho \colon X \cross G &\rightarrow X, \text{ sometimes indicated by } \notag \\
(x, g) &\mapsto x g , \notag
\end{align}
such that for $e$ the identity of $G$ and every $g, h \in G$ and $x \in X$,
\begin{align}
\rho (x, e) &= x \notag \\
\rho (\rho (g, x), h) &= \rho (x, g h). \notag
\end{align}
For smooth manifolds $G$ and $X$, $\rho$ is called a smooth action if it is a smooth map.
\end{defn}
For $G$ commutative the notions of right and left actions are equivalent.  To consider only the algebraic structure, without requiring that the map be continuous, we will say "action ignoring continuity".

\begin{defn}\label{d-g-tors}
\index{group!torsor}
\index{torsor!group}
\index{G torsor@$G$ torsor}
($G$ Torsors).
Given a topological group $G$, a (left) $G$ torsor $X$ is a nonempty Hausdorff space $X$ and a continuous, free and transitive action $\rho \colon G \cross X \rightarrow X$, such that for every $x \in X$, the map
\begin{align}
\rho_x \colon G &\rightarrow X \notag \\
\rho_x \colon g &\mapsto \rho(g, x) = g x \notag
\end{align}
is an open map (and thus is a homeomorphism).  Without the topological conditions, $X$ is called a (left) algebraic $G$ torsor.  Corresponding to right actions we have right torsors.
\end{defn}
The torsor $X$ has a particular action, not shown in the notation $X$, in the same way that the topological space has a particular topology, also not shown in the notation.  Whether $X$ means the torsor, the topological space, or the set, is inferred from the context.  $G$ acts on itself via left (respectively right) multiplication and is a left (respectively right) $G$ torsor.

\begin{defn}\label{d-g-equi}
\index{equivariant maps}
(Equivariant Maps).
Given a group $G$, a map of $G$ torsors is called equivariant if it commutes with the $G$ actions.  Corresponding to left and right actions we have left and right equivariance.
\end{defn}

\begin{defn}\label{d-g-tors-isom}
\index{group!torsor!isomorphism}
\index{isomorphism!torsor!group}
\index{G torsor@$G$ torsor!isomorphism}
(Isomorphisms of $G$ Torsors).
Given a topological group $G$, a $G$-equivariant homeomorphism of $G$ torsors is called an isomorphism of $G$ torsors.
\end{defn}

\begin{lem}\label{l-g-tors-home-g}
\index{group!torsor!homeomorphic to group}
\index{torsor!group!homeomorphic}
\index{G torsor@$G$ torsor!homeomorphic to group}
(For Compact $G$, Continuous Hausdorff Algebraic Torsors are Torsors).
Suppose given a compact topological group $G$ acting continuously on the left via $\rho$ on a Hausdorff space $X$, making it an algebraic $G$ torsor.  (In the title, we call the torsor continuous since the action is.)  Picking an arbitrary $x \in X$, the map $\rho_x$ of definition \ref{d-g-tors} is an an isomorphism of $G$ torsors.  Similarly for a right action.
\end{lem}
\begin{proof}
$\rho_x$ is a $G$-equivariant continuous map. Since it is from a compact space to a Hausdorff space, it is a closed map.  Since the action $\rho$ is free, $\rho_x$ is injective, and since the action is transitive, $\rho_x$ is surjective.  Thus $\rho_x$ is a $G$-equivariant homeomorphism, an isomorphism of $G$ torsors.
\end{proof}

\begin{cor}\label{co-g-tors-alg-top}
\index{group!torsor!homeomorphic to group}
\index{torsor!group!homeomorphic}
\index{G torsor@$G$ torsor!homeomorphic to group}
(Giving an Algebraic $G$ Torsor a Topology).
Given a topological group $G$ (compact or not) and a nonempty set $X$, if $X$ is an algebraic $G$ torsor, it can be given by choice of equivariant bijection as in lemma \ref{l-g-tors-home-g} a topology making $X$ a $G$ torsor isomorphic to $G$ via that bijection.
\end{cor}

\begin{cor}\label{co-g-equi-cont}
\index{group!torsor}
\index{torsor!group}
\index{equivariant map}
\index{G torsor@$G$ torsor}
(Equivariant Maps of Torsors are Isomorphisms).
Given a topological group $G$, torsors $S$, $T$, and a $G$-equivariant map $\phi \colon S \rightarrow T$, then $\phi$ is an isomorphism of $G$ torsors; in particular it is a homeomorphism.
\end{cor}
\begin{proof}
Name the actions $\rho_S, \rho_T$, pick $s \in S$, $t \in T$, and let $\rho_{S,s}, \rho_{T,t}$ be the isomorphisms of lemma \ref{l-g-tors-home-g}.  Then $\widetilde{\phi} = \rho_{T,t}^{-1} \circ \phi \circ \rho_{S,s} \colon G \rightarrow G$ being equivariant, it is given by right translation by $\widetilde{\phi} (e)$, an automorphism of the topological group.  Thus $\phi$ is a $G$-equivariant homeomorphism, an isomorphism of $G$ torsors.  Since additionally, $\rho_{S,s}$ and $\rho_{T,t}$ are isomorphisms of $G$ torsors, so is $\phi$.
\end{proof}

\begin{defn}\label{d-pb-afb}
\index{principal bundle}
\index{bundle!principal}
\index{associated fiber bundle}
\index{bundle!fiber!associated}
(Principal and Associated Fiber Bundles).
Given a topological or Lie group $G$, a topological \citep[pages~54-55]{tomD87} or smooth principal $G$ bundle is a topological or smooth fiber bundle $P \rightarrow B$ with standard fiber $G$ (definition \ref{d-fb-vb}), together with a continuous or smooth right action (definition \ref{d-grp-act}) of $G$ on $P$ that commutes with the projection map and is free and transitive over each point in $B$, for which the coordinate charts are $G$-equivariant (definition \ref{d-g-equi}).  (Thus the fibers are $G$ torsors.)

Given a continuous or smooth left action of $G$ on a topological space or smooth manifold $F$, the associated fiber bundle with fiber $F$ for the principal $G$ bundle $\pi \colon P \rightarrow B$ is defined by setting its total space to $P \cross_G F = (P \cross F) /\sim$, where $(x, f) \sim (x g, g^{-1} f)$, the equivalence class of $(x, f)$ is denoted $[(x, f)]$ or $[x, f]$, and the projection map $\pi_{P \cross_G F} \colon P \cross_G F \rightarrow B$ is defined by $\pi_{P \cross_G F} ([x, f]) = \pi (x)$.

Morphisms of principal bundles are maps of total spaces that are continuous or smooth as appropriate, that cover respectively continuous or smooth maps of base spaces (i.e. commute with the projections), and that preserve the structure of fibers (are equivariant, i.e., commute with the group action; see definition \ref{d-g-equi}).
\end{defn}
\begin{note}\label{n-vb-afb}
\index{associated fiber bundle}
\index{bundle!fiber!associated}
\index{vector bundle}
\index{bundle!vector}
(Associated Fiber Bundles Are Fiber Bundles.)
Smooth vector bundles, whose transition functions' automatically define continuous maps into topological groups $G$, or continuous vector bundles that happen to satisfy that requirement, are associated smooth or continuous fiber bundles for smooth or continuous principal $G$ bundles with fiber a vector space $X$.
\end{note}
See \citet[pages~25--31]{Poor07} for the smooth case; also note \ref{n-tran-func-pb} and lemma \ref{l-tran-func-pb}.  The continuous case is similar except that for the various flavors of topological bundles, we do not require the spaces to be manifolds.  See also \citet[pages~84--87]{DK01}, and \citet[pages~55--59]{HJJS08} (see lemma \ref{l-pb-tran-func-tors} and note that all our bundles are locally trivial in their terms).

\begin{note}\label{n-tran-func-pb}
\index{transition functions!principal bundle}
(Principal Bundle Transition Functions).
As a consequence of the definition, the change of coordinates maps and transition functions of fiber bundles or vector bundles are continuous or smooth as appropriate.  Note that these transition functions do not refer to maps $U_i \cap U_j \rightarrow \Homeo(F_i, F_j) \text{ or } \Diff(F_i, F_j)$; see definition \ref{d-fb-vb}. Continuity or smoothness of these maps can be a different question, depending on circumstances, but for a principal bundle things are easy, as follows.
\end{note}
\begin{lem}\label{l-tran-func-pb}
\index{transition functions!principal bundle}
(Principal Bundle Transition Functions).
For a principal $G$ bundle, transition functions as in definition \ref{d-fb-vb} give rise to maps $g_{ji} \colon U_i \cap U_j \rightarrow G$, continuous or smooth as the case may be, which are called the transition functions of the principal bundle.
\end{lem}
\begin{proof}
To see how this works, use the notation of definition \ref{d-fb-vb}.

Let $e \in G$ be the identity, and for $x \in U_i \cap U_j$ let $p_{i, x} = \phi_i^{-1} (x, e)$ and $p_{j, x} = \phi_j^{-1} (x, e)$.  Define $g_{ji}(x)$ by $p_{i, x} = p_{j, x} g_{ji}(x)$ (see lemma \ref{l-pb-tran-func-tors}).  We have $\phi_{ji}(x,e) = \phi_j \circ \phi_i^{-1} (x,e) = \phi_j (p_{i, x}) = \phi_j (p_{j, x} g_{ji}(x)) = \phi_j (p_{j, x}) g_{ji}(x) = (x, e) g_{ji}(x) = (x, g_{ji}(x))$.  Then by equivariance of the local trivialization maps $\phi_i$, $\phi_j$ and hence of $\phi_{ji}$, $\phi_{ji}(x,z) = (x, g_{ji}(x) z)$; or $\phi_{ji,2} (x,z) = g_{ji} (x) z$.

Since $\phi_{ji,2}$ is jointly continuous with respect to $x$ and $z$, or a smooth function of $x$ and $z$, so is $g_{ji} (x) z$, and thus so is $g_{ji} (x) = g_{ji} (x) z z^{-1}$.
\end{proof}
The context should make clear which definition of transition function is being used.

\begin{lem}\label{l-pb-cons}
\index{principal bundle!construction}
\index{bundle!principal!construction}
(Principal Bundle Constructions).
Given set of transition functions suitable for a principal bundle, satisfying a cocycle condition, there is a functorial construction of a principal bundle from them, that has them as its transition functions.  Furthermore, there is a naturally constructed bijection between equivalence classes of transition functions and isomorphism classes of bundles. Thus, if one can choose in a natural way, transition functions using the same open covering of the same base space for two principal bundles for the same group, and if the transition functions from the two bundles are equal, then the principal bundles are isomorphic in a natural way.

Given the elements of a topological principal bundle except no topology on the total space, and consequently the local trivializations being just bijections, if the change of coordinates maps or equivalently the transition functions are continuous, then a unique topology can be defined on the total space resulting in a topological principal bundle.  The analogous facts are true for a smooth principal bundle, when what is lacking is a smooth structure and possibly also a topology on the total space.

Somewhat more generally, given an open cover of what will become the base space, given a principal bundle over each element of the open cover and a principal bundle isomorphism between each pair of principal bundles, such that the isomorphisms satisfy a cocycle condition, then there is a natural construction of a total space with topology and smooth structure as appropriate, resulting in a principal bundle that when restricted to each open set of the cover, gives a principal bundle canonically isomorphic to the corresponding one of the original principal bundles.

In addition, if from each of the original bundles there is a principal bundle morphism to a fixed bundle, and these morphisms are compatible with the isomorphisms between the original bundles, then there is a naturally constructed morphism from the resulting bundle to the fixed bundle, compatible with the original morphism.
\end{lem}
\begin{proof}
For the first paragraph, \citet[pages~191--195]{Swit75} has a thorough discussion of the equivalence of isomorphism classes of continuous principal bundles and isomorphism classes of sets of continuous transition functions.  \citet[pages~58--60]{HJJS08} also covers the continuous case, from a different point of view.  \citet[pages~210--211]{Mich08} discusses the smooth case.

For the smooth case in the second paragraph see \citet[page~3]{Poor07}; the group can be ignored for this question.

For the third paragraph, see \citet[pages~58--60]{HJJS08} for a discussion of the continuous case, which is what we use.  For the third and fourth paragraphs, one could use as the total space of the constructed bundle the disjoint union of the total spaces of the individual bundles, modulo an equivalence relation based on isomorphisms, much as in a standard proof of the similar fact when the individual bundles are trivial.
\end{proof}

\begin{lem}\label{l-eqho-tspc-pb-home-bspc}
\index{bundle!principal!homeomorphic}
(Equivariant Homeomorphisms of Total Spaces of Principal Bundles Induce Homeomorphisms of Base Spaces).
Given a topological group $G$, topological principal bundles $\pi_1 \colon P_1 \rightarrow B_1$, $\pi_2 \colon P_2 \rightarrow B_2$, a $G$-equivariant homeomorphism $\psi \colon P_1 \rightarrow P_2$ induces a homeomorphism $\overline{\psi} \colon B_1 \rightarrow B_2$ making $(\psi, \overline{\psi})$ an isomorphism of topological principal $G$ bundles.
\end{lem}
This is so because a continuous equivariant map descends to a continuous map of the orbit spaces; see \citet[page~4]{tomD87}.

\begin{lem}\label{l-pb-tran-func-tors}
\index{translation function}
\index{torsor}
\index{principal bundle!translation function}
\index{principal bundle!fiber}
\index{principal bundle!torsor}
(Principal Bundle Translation Functions).
Given a topological group $G$ and a topological principal $G$ bundle $\pi \colon P \rightarrow X$, defining $P \cross_{\pi} P = \{ (p_1, p_2) \in P \cross P \st \pi(p_1) = \pi(p_2) \}$, there is a unique translation function $\tau \colon P \cross_{\pi} P \rightarrow G$ such that for all $p \in P$ and $g \in G$, $\tau (p, pg) = g$; and $\tau$ is continuous \citep[page~54--55]{tomD87}.  Stated differently, for all $p_1, p_2 \in P_x$, $p_1 \tau (p_1, p_2) = p_2$.
\end{lem}

\begin{lem}\label{l-pb-mor-cov-id-isom}
\index{bundle!morphism}
(Morphisms of Principal Bundles Covering the Identity are Isomorphisms).
\citep[pages~54--56]{tomD87} Given a topological group $G$, topological principal bundles $\pi_1 \colon P_1 \rightarrow B$, $\pi_2 \colon P_2 \rightarrow B$, a $G$-equivariant continuous map $\psi \colon P_1 \rightarrow P_2$ covering the identity is a homeomorphism and hence an isomorphism of topological principal $G$ bundles.
\end{lem}

\section{The Vector Bundle}\label{s-the-vb}

\begin{ass}\label{a-mfld-vb}
\index{manifold}
\index{M@$M$}
\index{m@$m$}
\index{bundle!vector}
\index{vector bundle}
\index{E@$E$}
\index{n@$n$}
\index{fiberwise inner product!vector bundle}
\index{SO(n)@$\SO(n)$}
\index{SO(E)@$\SO(E)$}
\index{O(n)@$\Orth(n)$}
\index{assumptions!M@$M$}
\index{assumptions!m@$m$}
\index{assumptions!E@$E$}
\index{assumptions!n@$n$}
\index{assumptions!base manifold}
\index{assumptions!vector bundle}
\index{compact}
\index{connected}
\index{orientable}
\index{oriented}
\index{Riemannian metric}
(The Vector Bundle and Base Manifold).
We suppose given a compact connected orientable smooth manifold $M$ of dimension $m$ with Riemannian metric; and an smooth oriented real vector bundle $\pi_E \colon E \rightarrow M$ of even rank $n$, with fiberwise inner product denoted $(,)$, the inner product on fiber $E_x$ depending smoothly on $x \in M$ (``smoothly'' defined using the notions of \citet[pages~58--63]{Lang99}).
\[
\begindc{\commdiag}[5]
\obj(10,20)[objE]{$E^n$}
\obj(10,10)[objM]{$M^m$}
\mor{objE}{objM}{$\pi_E$}
\enddc
\]
\end{ass}
The condition of orientability of $M$ is specified only because \citet{Stac05}, referenced here in proposition \ref{p-smth-free-loop-spac-frec-mfld}, specifies orientability; though with language indicating that the results may be true without that assumption.  The Riemannian metric on $M$ is used in proposition \ref{p-smth-free-loop-spac-frec-mfld} to construct the \Frechet manifold structure for $LM$, and in proposition \ref{p-lm-c-coho-isom}, which references the first proposition.

\begin{ass}\label{a-bndl-orient}
\index{E@$E$}
\index{SO(E)@$\SO(E)$}
\index{SO(n)@$\SO(n)$}
\index{vector bundle!orientation}
\index{assumptions!vector bundle!orientation}
(Orientation for The Vector Bundle).
Let $\Orth(n) \rightarrow \Orth(E) \rightarrow M$ denote the orthogonal frame bundle $\Orth(E)$ over $M$, a principal $\Orth(n)$ bundle.  The first arrow means that the fibers are diffeomorphic to $\Orth(n)$.  We assume given a choice of orientation of $E$; this lets us define the orthonormal frame bundle $\SO(n) \rightarrow \SO(E) \rightarrow M$.  The following diagram of smooth principal bundles commutes, where $\nu$ is the inclusion and where the arrows from fibers to total spaces emulate the convention of \citet[page~29]{Whit78}.  Each upper vertical arrow is a $G$-equivariant homeomorphism from a group $G$ to a standard fiber, a $G$ torsor (see definition \ref{d-g-tors}).

\[
\begindc{\commdiag}[5]
\obj(30,30)[objSOn]{$\SO(n)$}
\obj(30,20)[objSOE]{$\SO(E)$}
\obj(30,10)[objM2]{$M$}
\obj(50,30)[objOn]{$\Orth(n)$}
\obj(50,20)[objOE]{$\Orth(E)$}
\obj(50,10)[objM3]{$M$}
\mor{objSOn}{objSOE}{}
\mor{objSOE}{objM2}{}
\mor{objOn}{objOE}{}
\mor{objOE}{objM3}{}
\mor{objSOn}{objOn}{$\nu$}
\mor{objSOE}{objOE}{$\widehat{\nu}$}
\mor{objM2}{objM3}{$=$}

\enddc
\]
\end{ass}

\section{Loop Spaces}\label{s-loop-spac}

\begin{ass}\label{a-parm-s1-u1}
\index{S1@$S^1$!parametrization}
\index{U1@$\UU(1)$!parametrization}
\index{assumptions!parametrization S1, U(1)@parametrization $S^1$, $\UU(1)$}
(Parametrizations of $S^1$ and $\UU(1)$).
We will identify $S^1$, $\UU(1)$, the unit circle in $\CC$, and $\RR / \ZZ$.  Sometimes we will mix inconsistently the additive notation that would be associated with a general abelian group, in our case $\RR / \ZZ$, and multiplicative notation for $S^1$ or $\UU(1)$, particularly when we say something is $0$ or say it is $1$, meaning the same thing, the identity element of the group.  We will try not to let this confuse us into thinking we are dealing with a ring.

We may use for $S^1$ at times, various intervals of $\RR$ with endpoints identified:  $[0, 2 \pi]$, convenient for writing Fourier series, $[0, 1]$, consistent with $\RR/\ZZ$, and $[-1, 1]$, convenient for suspensions.  The reader should supply affine transformations when needed.
\end{ass}

\begin{ass}\label{a-smooth-loops}
\index{smooth loops}
\index{assumptions!smooth loops}
(Smooth Loops).
We will use smooth free loops, choosing smooth loops because they are used by \citet[page~27]{PS86}, which was an important reference in this investigation, are used by \citet{SW07}, and fall into the context of \citet{Hami82}, whose results we use.
\end{ass}
\begin{note}\label{n-smth-loop}
\index{smooth loops}
\index{loops!smooth}
(Smooth Loops).
A broader class of loops such as Sobolev half-differentiable or absolutely continuous might work.  While they use smooth loops for their book, \citet[pages~26--27,~84]{PS86} discuss larger classes of loops than smooth, mentioning Sobolev half-differentiable functions and the group $L_{\frac{1}{2}}GL_n(C)$ which ``is the largest group for which the crucial central extension can be constructed and the basic representation defined''.  For other more differential geometric approaches to some parts of the thesis, absolutely continuous loops, which are Sobolev half-differentiable and along which parallel translation can be defined, might be useful.
\end{note}

\begin{defn}\label{d-smth-free-loop-spac}
\index{Frechet@\Frechet!topology}
\index{topology!Frechet@\Frechet}
\index{loop space}
(Smooth Free Loop Spaces have the \Frechet Topology).
Given a smooth manifold $X$, the smooth free loop space of $X$ as a set is $LX = C^{\infty} (S^1, X) = \{ \text{smooth }\gamma \colon S^1 \rightarrow X \}$.  It is given the \Frechet manifold topology \citep[page~85]{Hami82}.
\end{defn}
The definition as a set is effective immediately; the topology needs justifying, done in proposition \ref{p-smth-free-loop-spac-frec-mfld}.  $X$ may be, for example, the base manifold $M$ or total space $E$ of the smooth vector bundle, $\RR^n$, or the structure group $\SO(n)$ of the smooth vector bundle.

\begin{defn}\label{d-loop-smth-map}
\index{loop!smooth map}
(Loop of a Smooth Map $f \colon Y \rightarrow X$).
Given two smooth manifolds $X$, $Y$ and a smooth map $f \colon Y \rightarrow X$, define $Lf \colon LY \rightarrow LX$ for $\gamma \in LY$ as $Lf \colon \gamma \mapsto f \circ \gamma$.
\end{defn}
For example, $f \colon Y \rightarrow X$ may be $\pi_E \colon E \rightarrow M$.

\section{\Frechet Spaces and Manifolds}\label{s-frec-spac-mfld}

We will include in this and following sections excerpts from \citet{Hami82} sections I.1 and I.4, which is highly recommended, and \citet{SW07} section 1.3 and Appendix, specialized to our needs.  For general topological questions about loop spaces we will use material from \citet{Stac05}.

\begin{defn}\label{d-frec-spac}
\index{Frechet@\Frechet!space}
\index{space!Frechet@\Frechet}
(\Frechet Spaces).
A \Frechet space is a complete topological vector space whose topology is defined by a countable collection of seminorms $\norm{}_k$ which are all simultaneously $0$ only for the $0$ vector.
\end{defn}

\begin{defn}\label{d-para-cpct}
\index{paracompact}
(Paracompactness).
A topological space is called paracompact when it is Hausdorff and every open covering has a refinement that is locally finite, meaning that each point has a neighborhood that intersects only finitely many open sets in the refinement \citep[pages~81,~162]{Dugu66} \citep[page~21]{Bred97}.
\end{defn}

\begin{lem}\label{l-frec-spac}
\index{Frechet@\Frechet!space}
\index{space!Frechet@\Frechet}
\index{paracompact}
\index{paracompact!hereditarily}
\index{good cover}
(\Frechet Space Properties).
A \Frechet space $F$ is a complete metrizable (thus Hausdorff) locally convex vector space.  The topology is defined by the requirement that a sequence $f_j \rightarrow f  \Leftrightarrow \forall k, \norm{f_j - f}_k \rightarrow 0$.  See section I.1.1 of \citet[page~67]{Hami82}.  By \citet[page~979]{Ston48} every metric space is paracompact, so \Frechet spaces are paracompact, and moreover, hereditarily paracompact; i.e., all their open sets, and hence all their topological subspaces are paracompact.

Since it is locally convex and the vector space operations are continuous, it is locally path-connected.  What is more, every open cover of a \Frechet space can be refined to a good cover, an open cover with all nonempty finite intersections contractible, as follows.  (In fact, even nonempty infinite intersections will be contractible.) Define by \citet[page~11]{Rudi91} a translation-invariant metric on $F$ that induces the \Frechet space topology by choosing a sequence of positive numbers $c_k$ that converges to $0$ - for concreteness let us choose $c_k = 2^{-k}$ - and for $\alpha, \beta \in F$, defining
\[
d(\alpha, \beta) = \max_k \frac{c_k \norm{\alpha - \beta}_k}{1 + \norm{\alpha - \beta}_k}. \notag
\]
Then by \citet[page~29]{Rudi91} the metric is translation invariant, and the open balls $\B_r (0)$ for $r > 0$ form a convex local basis $\{V_k\}$ at $0$ for the \Frechet space topology. Their translations are convex and form a cover of $F$.  Taking a nonempty intersection of these and fixing some point in it, $x_0 \in \cap (V_k + x_k)$, then
\begin{align}
H \colon [0,1] \cross (\bigcap (V_k + x_k)) &\rightarrow \bigcap (V_k + x_k) \notag \\
(s, x) &\mapsto s x_0 + (1 - s) x \notag
\end{align}
is a deformation retraction of $\cap (V_k + x_k)$ onto $x_0$, since intersections of convex sets are convex.  Note that looking ahead to definition \ref{d-frec-deri} and using material from \citet[pages~73--84]{Hami82}, $H$ is smooth.  Since all nonempty intersections are contractible, the set of open balls of this metric form a good cover of $F$.  Since any open set in $F$ is the union of the open balls contained in it, if we have any cover of $F$, we may refine it to a good cover consisting of all the open balls in all the open sets of the cover.
\end{lem}

\begin{eg}\label{e-lrn-l-vect-spac}
\index{Frechet@\Frechet!space!LRn@$L\RR^b$}
\index{LRn@$L\RR^n$}
\index{loop!finite-dimensional vector space}
($L\RR^n$, the Loop of a Finite-Dimensional Vector Space).
The \Frechet space $L\RR^n = C^{\infty} (S^1, \RR^n)$ is the space of smooth loops in $\RR^n$ with topology defined by the collection of seminorms $\norm{f}_k = \max_{t \in S^1}$ $\norm{f^{(k)}(t)}$; that is, for a sequence of loops to converge, the loops themselves and all their derivatives must converge in the uniform norm.  This is implied by a special case of example I.1.1.5 in \citet[page~68]{Hami82} as in our lemma \ref{l-sect-smoo-vb-form-frec-spac}.  This definition of $L\RR^n$ is equivalent to that in \citet[page~9]{Stac05}, who states that $L\RR^n$ is separable.

More generally, if $X$ is a finite-dimensional vector space, then $LX$ is a \Frechet space, with topology defined as for $L\RR^n$.
\end{eg}

\begin{defn}\label{d-frec-deri}
\index{Frechet@\Frechet!derivatives}
(\Frechet Derivatives).
Given \Frechet spaces $F$, $G$, open $U \subset F$, and continuous $P \colon U \rightarrow G$, the derivative of $P$ at the point $f \in U$ in the direction $h \in F$ is defined by
\[
DP(f)h = \lim_{t \to 0} \frac{P(f + th) - P(f)}{t}.
\]
$P$ is differentiable at $f$ in the direction $h$ if the limit exists.  It is continuously differentiable or $C^1$ on $U$ if the limit exists for all $f \in U$ and $h \in F$ and $DP \colon U \cross F \rightarrow G$ is continuous using the product topology for the domain.

The second derivative of $P$ at the point $f$ in $U$ in the directions $h, k \in F$ is defined by
\[
D^2P(f)(h,k) = \lim_{t \to 0} \frac{DP(f + tk)h - DP(f)h}{t}.
\]
$P$ is twice continuously differentiable or $C^2$ on $U$ if it is $C^1$, the limit exists for all $f \in U$ and $h, k \in F$, and $D^2P \colon U \cross F \cross F \rightarrow G$ is continuous using the product topology for the domain.
Higher derivatives are defined similarly.
\end{defn}
Note that the definition of differentiability \citep[page~73]{Hami82} assumes continuity.  Remember that the limits in the definition are in the \Frechet topology.  Note that $DP(f)h$ is linear in h, but $C^1$ does not imply that the induced map $U \rightarrow \{ \text{linear maps }F \rightarrow G \}$ is continuous, whatever that might mean; the set of linear maps may not even be a \Frechet space.

\begin{lem}\label{l-sect-smoo-vb-form-frec-spac}
\index{vector bundle!smooth sections!Frechet space@\Frechet space}
\index{Frechet space@\Frechet space!vector bundle!smooth sections}
(Sections of Smooth Vector Bundle with Compact Base are a \Frechet Space).
Given a smooth vector bundle $V \rightarrow X$, $X$ a compact smooth manifold, choosing a Riemannian metric and connection on $TX$ and smooth fiberwise inner product and connection on $V$, then the set of its smooth sections, $\Gamma(V)$, is a \Frechet space.  The resulting topology does not depend on the Riemannian metrics or connections chosen.
\end{lem}
The first statement is example I.1.1.5 in \citet[page~68]{Hami82}.  That the topology is independent of the choices seems to be assumed by \citet[page~85]{Hami82} example I.4.1.2, though it's not explicitly stated.  It follows from reasoning similar to that in \citep[pages~123--125]{Pete06}, facts about the difference of two covariant derivatives, and the product rule for covariant differentiation.

\begin{defn}\label{d-frec-mfld}
\index{Frechet@\Frechet!Manifold}
\index{Manifold!Frechet@\Frechet}
(\Frechet Manifolds).
A \Frechet manifold is a Hausdorff topological space with an atlas of coordinate charts with values in a \Frechet space, such that the change of coordinates maps are all smooth maps.
\end{defn}
\begin{note}\label{n-frec-mfld}
\index{Frechet@\Frechet!Manifold!definitions}
(\Frechet Manifold Definitions).
This is almost the same as definition I.4.1.1 of \citet[page~85]{Hami82} except that he allows different \Frechet spaces for different charts whereas we don't.
\end{note}

\begin{lem}\label{l-sect-fb-form-frec-mfld}
\index{fiber bundle!smooth sections!Frechet manifold@\Frechet manifold}
\index{Frechet manifold@\Frechet manifold!fiber bundle!smooth sections}
(Sections of Smooth Fiber Bundle with Compact Base form a \Frechet Manifold).
Given a smooth fiber bundle $V \rightarrow X$, $X$ compact, the set of its smooth sections, $\Gamma(V)$ is a \Frechet manifold.
\end{lem}
This is example I.4.1.2 in \citet[pages~85--86]{Hami82}; our local triviality condition implies his condition of surjectivity of the derivative of the projection map of the fiber bundle.

Next we will go into some detail about the free loop space as a \Frechet manifold.  This will require a little preparation.

\section{Compact-Open Topology}\label{s-cpct-open}

\begin{defn}\label{d-co-top}
\index{C(X,Y)@$C(X,Y)$}
\index{topology!compact-open}
\index{compact-open topology}
\index{mapping space!compact-open topology}
(The Compact-Open Topology).
\citep[page~257]{Dugu66} Given topological spaces $X$, $Y$, the compact-open topology on $C(X,Y)$, the space of continuous maps $X \rightarrow Y$, is given by a sub-basis consisting of the sets
\begin{align}
(K,U) = \{ \gamma \in C(X,Y) \st \gamma (K) \subset U \}, \notag \\
K \text{ compact } \subset X \text { and } U \text{ open } \subset Y. \notag
\end{align}
When a specific topology for $C(X,Y)$ is not mentioned, it will have the compact-open topology.  The same name, compact-open topology, will be given to the subspace topology for subsets of $C(X,Y)$.
\end{defn}

\begin{lem}\label{d-alt-char-frec-top-lrn}
\index{Frechet@\Frechet!space!LRn@$L\RR^n$}
(An Alternate Characterization of the \Frechet Topology on $L\RR^n$).
\citep[page~9]{Stac05} The \Frechet topology on $L\RR^n$ of example \ref{e-lrn-l-vect-spac} can be defined as the projective or initial topology (i.e. the coarsest topology that makes all the maps continuous) for the maps $L\RR^n \rightarrow C(S^1, \RR^{kn})$ given by
\[
\gamma \mapsto (t \mapsto (\gamma(t), \gamma'(t), \dots, \gamma^{(k-1)}(t))). \notag
\]
\end{lem}

\begin{lem}\label{l-co-top}
\index{C(X,Y)@$C(X,Y)$}
\index{topology!compact-open}
\index{compact-open topology}
\index{mapping space!compact-open topology}
(Compact-Open Topology Properties).
% ~ removed from citation to cure overfull box
\citep[pages 259--261,~276]{Dugu66} Given topological spaces $X$, $Y$, and $Z$, using the compact-open topologies, for fixed $g \in C(Y, Z)$, the map $g_{\sharp} \colon C(X, Y) \rightarrow C(X, Z)$ given by $g_{\sharp}(f) = g \circ f$ is continuous.  Thus if $g$ is a homeomorphism, so is $g_{\sharp}$, which has inverse $(g^{-1})_{\sharp}$.  Similarly, for fixed $h \in C(X, Y)$, the map $h^{\sharp} \colon C(Y, Z) \rightarrow C(X, Z)$ given by $h^{\sharp}(f) = f \circ h$ is continuous.  Thus if $h$ is a homeomorphism, so is $h_{\sharp}$, which has inverse $(h^{-1})^{\sharp}$.

The natural bijection $C(X, Y \cross Z) \rightarrow C(X, Y) \cross C(X, Z)$ is a homeomorphism.

If $Y$ is a subspace of $Z$, the inclusion $Y \xrightarrow{\incl} Z$ is a continuous map that is a homeomorphism onto its image.  Then if $X$ is compact, $(\incl)_{\sharp} \colon C(X, Y) \rightarrow C(X, Z)$ is a homeomorphism onto its image.

Defining the evaluation map $\ev \colon C(Y, Z) \cross Y \rightarrow Z$ by $\ev (g, y) = g(y)$, for each $y_0 \in Y$, $\ev_{y_0} = \ev (\cdot, y_0)$ is continuous, and if $Y$ is locally compact, $\ev$ is continuous.

Given $\alpha \colon X \cross Y \rightarrow Z$, the adjoint of $\alpha$, $\widehat{\alpha} \colon X \rightarrow C(Y, Z)$, defined by $(\widehat{\alpha} (x)) (y) = \alpha (x, y)$, is continuous.  Conversely, given $\widehat{\alpha}$, the equation defines $\alpha$, which is continuous if $\widehat{\alpha}$ is continuous and $Y$ is locally compact.  Example: $\widehat{\ev} = \ident_{C(Y, Z)}$.
\end{lem}
\begin{proof}
The assertions in the first paragraph and those about evaluation maps and adjoints come from \citet[pages~259--261,~276]{Dugu66}.

For the second paragraph, use \citet[page~264]{Dugu66} and choose a sub-basis for the cartesian product $Y \cross Z$ consisting of sets of the form $U \cross Z$ and $Y \cross V$.

For the third paragraph, let $U$ be an open subset of $Y$; equivalently, there is some $V$, an open subset of $Z$, such that $U = Y \cap V$.  Then for $K$ a compact set in $X$, the sub-basic open set $(K, U) \subset C(X, Y)$ is carried by $(\incl)_{\sharp}$ (remember that the meaning of $(K, U)$ depends on which space it is in) to $(X, Y) \cap (K, U) = (X, Y) \cap (K, Y \cap V) = (X, Y) \cap (K, Y) \cap (K, V) = (X, Y) \cap (K, V)$, the intersection of $(\incl)_{\sharp} (C(X, Y))$ and the set $(K, V)$ open in $C(X, Z)$.
\end{proof}

\begin{defn}\label{d-lspc-metr-top}
\index{topology!mapping space metric}
\index{metric!mapping space}
\index{mapping space!metric topology}
% ~ removed from citation to cure overfull box
(The Metric Topology on a Mapping Space).
\citep[page 269]{Dugu66} Given $(Z,d)$ a metric space and $Y$ a topological space, letting $C(Y,Z,d)$ be the set of all $d$-bounded continuous maps $Y \rightarrow Z$, define a metric on $C(Y,Z,d)$ by
\[
d^{+}(f,g) = \sup_{y \in Y} (d(f(y), g(y))). \notag
\]
\end{defn}
Another phrase for this metric topology is the topology of uniform convergence.

\begin{lem}\label{l-co-metr-top}
\index{topology!mapping space metric}
\index{metric!mapping space}
\index{mapping space!metric topology}
\index{loop space!compact open topology}
\index{loop space!metric topology}
(With a Compact Domain, the Compact-Open and Metric Topologies are the Same).
By 8.2(3) of \citet[pages~269--271]{Dugu66}, Given $(Z,d)$ a metric space and $Y$ a compact space, $C(Y,Z,d) = C(Y,Z)$, and the metric topology equals the compact-open topology.  Thus the topology is independent of the particular metric chosen.

In particular, if $Z$ is a compact smooth manifold, any Riemannian metric will make it a complete metric space \citep[page~125]{Pete06}, and the induced metric topology on $C(S^1, Z)$ equals the compact-open topology.  Further, by \citet[pages~18--19]{Stac05}, $C(S^1, Z)$ is separable.
\end{lem}

\section{Loop Space \Frechet Manifold}\label{s-loop-spac-frec-mfld}

Although many of the results concerning loop spaces in this chapter and the next do not depend on the space being looped being compact and are stated without that requirement, our use of these results is to loop compact spaces.  In the following, when $X$ is compact, specific choices are made to be used later.
\begin{lem}\label{l-loc-addn}
\index{exponential map}
(The Exponential Map).
Given a smooth manifold $X$ with a Riemannian metric and the Levi-Civita connection, denoting its tangent bundle by $\pi \colon TX \rightarrow X$, there is an open neighborhood $N$ of the image $Z = \zeta (X) \subset TX$ of the zero section $\zeta \in \Gamma (TX)$, an open neighborhood $V$ of the diagonal $\Delta$ in $X \cross X$, and a smooth map $\sigma \colon N \rightarrow X$ given by $\sigma (v) = \exp_{\pi (v)} (v)$ \citep[pages~130--134]{Pete06}, such that $\sigma_{|Z} = \pi_{|Z}$ and $(\pi, \sigma) \colon N \rightarrow V$ is a diffeomorphism.

Furthermore, there is a diffeomorphism $\psi \colon TX \rightarrow N$ such that $\pi \circ \psi = \pi$, $\psi_{|Z} = \ident_Z$, and thus, defining $\eta = \sigma \circ \psi$, $\eta_{|Z} = \pi_{|Z}$, and $(\pi, \eta) \colon TX \rightarrow V$ is a diffeomorphism.

If $X$ is compact, let $\epsilon > 0$ equal half the convexity radius of $X$ \citep[pages~86--87]{GHL04}.  Then choose $N = \bigcup_{x \in X} \B_{\epsilon} (\zeta (x))$, where for each $x \in X$, $\B_{\epsilon} (\zeta (x)) \subset T_x X$ and $\exp_x \colon \B_{\epsilon} (\zeta (x)) \rightarrow \B_{\epsilon} (x)$ is a diffeomorphism.  Let $\psi$ multiply vectors in $TX$ by a function of their lengths, depending on $\epsilon$; then the action of $\psi$ on fibers commutes with parallel translation.
\end{lem}
\begin{note}\label{n-conv-rad}
\index{convexity radius}
(The Convexity Radius).
The convexity radius of $X$ is a number such that geodesic balls of radius smaller than it are geodesically convex.  It's no larger than the injectivity radius \citep[page~278]{Berg03}, and is used in the proof of lemmas \ref{p-lm-c-coho-isom}, \ref{l-good-cove-c-lm}, and \ref{l-lm-c-nerv-bij-non-empt}.
\end{note}
\begin{note}\label{n-loc-addn}
\index{local addition!exponential map}
\index{exponential map!local addition}
(Local Additions).
The above lemma is proposition 3.14 of \citet[page~13]{Stac05} with additions for our circumstances.  He refers to a ``local addition'', which can be understood by looking at the exponential map when $X$ is a finite-dimensional vector space.
\end{note}

\begin{prop}\label{p-smth-free-loop-spac-frec-mfld}
\index{loop space!Frechet manifold@\Frechet manifold}
\index{Frechet manifold@\Frechet manifold!loop space}
\index{paracompact}
\index{metrizable}
\index{orientable}
\index{good cover}
(A Smooth Free Loop Space is a \Frechet Manifold).
Given an orientable smooth manifold $X$ with Riemannian metric, $LX$ is a \Frechet manifold, with topology defined by charts with domains named $U_{\alpha}$, $\alpha \in LX$, whose definitions depend on the data used in lemma \ref{l-loc-addn}, including choices made in the case when $X$ is compact.  The open sets of $LX$ are defined as sets whose intersections with all $U_{\alpha}$ are open.

By lemma \ref{l-frec-spac}, any covering of $LX$ has a refinement made by refinements of intersections with chart domains, that is a good cover, i.e. for which all nonempty finite intersections are contractible (even nonempty infinite intersections are contractible).  Similarly or as a consequence, $LX$ is locally path-connected.

The inclusion $LX \subset C(S^1, X)$ is continuous, the $U_{\alpha}$ are open in their induced topologies, $LX$ is metrizable, hence hereditarily paracompact.
\end{prop}
\begin{proof}
To make $\alpha^{*} TX$ a \Frechet space as in  \citet[page~68]{Hami82} requires a choice of Riemannian metric and connection on $X$, and that source seems to gloss over the question of whether these choices affect the topology or smooth structure; in some contexts it seemed implicitly to assume they do not.  However, we will assume a metric given; it is also used in proposition \ref{p-lm-c-coho-isom} and related lemmas, which refer back to this proof and use the local charts constructed here.  We choose the Levi-Civita connection to remove that possible ambiguity of choice.

That $LX$ is a \Frechet manifold is example I.4.1.3 in \citet[pages~86]{Hami82}, a special case of proposition \ref{l-sect-fb-form-frec-mfld} using the product bundle $S^1 \cross X \rightarrow S^1$.  That example gives almost the same local chart about $\alpha \in LX$ we and \citet[pages~12--21]{Stac05} use, except that the former uses a diffeomorphism from a neighborhood of the zero section of what appears to be $\alpha^{*}TX$, whereas the latter uses a diffeomorphism from the whole space of sections of $\alpha^{*}TX$.  This is an inessential difference; although \citet{Hami82} doesn't say how he chooses that neighborhood, it could be one that is diffeomorphic to the whole space.

That $LX$ is Hausdorff is a consequence of local homeomorphisms to open sets in \Frechet spaces.

For the rest of the point set topological properties, we will use \citet[pages~12--20]{Stac05}.  That paper assumes for its exposition and proofs that $X$ is orientable, though it seems to imply by saying that this assumption is for convenience, not necessity, and not actually used in ``the analysis'', that the results are true even if $X$ is not orientable.  It appears that this assumption is used, for $\alpha \in LX$, to allow identification of $\alpha^{*} TX$ with $L\RR^m$, which has a standard \Frechet space structure.  This is the only place in this thesis this assumption is used.

To help clear up the question of whether more of present proposition might depend on orientability of $X$, and more importantly to make more understandable the local charts used later, we will repeat the proof of lemma 3.22 of \citet[page~19]{Stac05} with modifications, mainly not trivializing $\alpha^{*}TX$.  In this proof we don't use \citet{Stac05} section 3.4 on $LTX$ and we don't need to know the relationship between the smooth structures defined by the two references.

\begin{notat}\label{n-smth-free-loop-spac-frec-mfld}
\index{loop space!Frechet manifold@\Frechet manifold!charts}
(Manifold Charts and Proof Details Used Elsewhere).
\par
In brief, the chart for $U_{\alpha}$ maps a loop near $\alpha$, via the exponential map, to a loop in $TX$ over $\alpha$, and thence to a section of $\alpha^{*} TX$.  If $X$ is compact we choose $\epsilon$ as in lemma \ref{l-loc-addn}, and in any case obtain the neighborhood $V$ of the diagonal in $X \cross X$ and map $\eta$ such that $(\pi, \eta) \colon TX \rightarrow V$ is a diffeomorphism.  Given $\alpha \in LX$, $\alpha^{*} TX = \{ (t, v) \in S^1 \cross TX \st \alpha(t) = \pi(v) \}$, and using the Levi-Civita connection on $TX$, $\Gamma (\alpha^{*} TX)$ is a \Frechet space by lemma \ref{l-sect-smoo-vb-form-frec-spac}.  Define $\theta \colon \Gamma (\alpha^{*} TX) \rightarrow LTX$ as the natural injection, for $\beta \in \Gamma (\alpha^{*} TX)$,  $\beta \mapsto \pi_2 \circ \beta$; i.e. $\theta = L\pi_2$ restricted to $\Gamma (\alpha^{*} TX)$, where $\pi_2 \colon S^1 \cross TX \rightarrow TX$.  Define, using another $\pi_2 \colon X \cross X \rightarrow X$,
\begin{align}
U_{\alpha} &= \{ \gamma \in LX \st (\alpha, \gamma) \in LV \} \notag \\
\Psi_{\alpha} \colon U_{\alpha} &\rightarrow \Gamma(\alpha^{*} TX) \notag \\
\Psi_{\alpha}^{-1} \colon \Gamma(\alpha^{*} TX) &\rightarrow LTX \rightarrow LX \cross LX \rightarrow LX \notag \\
\Psi_{\alpha}^{-1} \colon \beta &\mapsto L\pi_2 (L(\pi, \eta) (\theta(\beta))) = L\eta (\theta(\beta)) = \gamma \notag \\
\Psi_{\alpha} \colon \gamma &\mapsto \theta^{-1} ((L(\pi, \eta))^{-1} (\alpha, \gamma)) = \beta, \notag
\end{align}
where although $\theta$ isn't surjective, $(L(\pi, \eta))^{-1} (\alpha, \gamma)$ is in its range; a loop in $TX$ over $\alpha$.
Thus $\Psi_{\alpha}$ is a bijection.  The change of coordinates functions $\Psi_{\alpha_2} \circ \Psi_{\alpha_1}^{-1}$ are diffeomorphisms, as in lemma 3.15, definition 3.16, and proposition 3.18 of \citet[pages~13--15]{Stac05}, where the metric, the Levi-Civita connection, and hypothesis that $X$ is orientable are used to identify $\Gamma (\alpha^{*} TX)$ with $L\RR^m$.

For a concrete example unravelling the \Frechet manifold chart definition, see the alternative (''Another way'') proof of continuity of $i_3$ in lemma \ref{l-incl-s3-ls4-cont}.
\begin{flushleft}
\emph{End Notation}
\end{flushleft}
\end{notat}

First we show using definition \ref{d-co-top} and lemma \ref{l-co-top} that the $U_{\alpha}$ are open in the topology induced by the inclusion $LX \xrightarrow{\incl} C(S^1,X)$.  Define $U_{\alpha}^0 = \{ \gamma \in C(S^1, X) \st (\alpha, \gamma) \in C(S^1, V) \}$.  We will show that $U_{\alpha}^0$ is an open subset of $C(S^1,X)$, whence the desired statement follows from $U_{\alpha} = U_{\alpha}^0 \cap LX$.  Because $V$ is open in $X \cross X$, $C(S^1, V) = (S^1, V)$ is open in $C(S^1, X \cross X)$.  Thus $(\{ \alpha \} \cross C(S^1, X)) \cap C(S^1, V)$ is an open subset of $\{ \alpha \} \cross C(S^1, X)$, giving the latter the topology of a subspace of $C(S^1, X \cross X)$.  Defining the natural homeomorphism $\Phi \colon C(S^1, X \cross X) \rightarrow C(S^1, X) \cross C(S^1, X)$ and the natural bijection $\Phi_{\alpha} \colon \{ \alpha \} \cross C(S^1, X) \rightarrow C(S^1, X)$, the latter is continuous and open because $\Phi_{\alpha} = \pi_2 \circ \Phi_{|\{\alpha \} \cross C(S^1, X)}$, so is a homeomorphism.  But $U_{\alpha}^0 = \Phi_{\alpha} ((\{ \alpha \} \cross C(S^1, X)) \cap C(S^1, V))$.

Then we show that the inclusion $LX \xrightarrow{\incl} C(S^1,X)$ is continuous.  Given $\alpha \in LX$ and any open neighborhood $U^0$ of $\alpha$ in $C(S^1, X)$, we will find an open neighborhood $W$ of $\alpha$ in $LX$ such that $W = \incl(W) \subset U^0$.  Corresponding to the bijection
\begin{align}
\Psi_{\alpha}^{-1} \colon \Gamma(\alpha^{*} TX) &\xrightarrow{\theta} \theta(\Gamma(\alpha^{*} TX)) \notag \\
 &\xrightarrow{L(\pi, \eta)} (\{ \alpha \} \cross LX) \cap LV \notag \\
 &\xrightarrow{L\pi_2} U_{\alpha}, \notag
\end{align}
defining $\Gamma^0(\alpha^{*}TX) = \{ \beta \in C(S^1, \alpha^{*} TX) \st \pi_1 \circ \beta = \ident_{S^1} \}$, we have the bijection
\begin{align}
{\Psi_{\alpha}^0}^{-1} \colon \Gamma^0(\alpha^{*}TX) &\xrightarrow{\theta^0} \theta^0(C(S^1, \alpha^{*} TX)) \notag \\
 &\xrightarrow{(\pi, \eta)_{\sharp}} (\{ \alpha \} \cross C(S^1, X)) \cap C(S^1, V) \notag \\
 &\xrightarrow{\Phi_{\alpha}} U_{\alpha}^0. \notag
\end{align}
Let us look more closely at $\theta^0$ and the spaces involved with it.  The set $\alpha^{*} TX$ gets its topology as a subspace of $S^1 \cross TX$, so the corresponding inclusion is continuous, a homeomorphism onto its image.  From that inclusion by lemma \ref{l-co-top} we get a continuous map $C(S^1, \alpha^{*} TX) \rightarrow C(S^1, S^1 \cross TX)$ that is a homeomorphism onto its image.  The set $\Gamma^0 (\alpha^{*} TX)$ gets its topology as a subspace of $C(S^1, \alpha^{*} TX)$, so that inclusion also is continuous, a homeomorphism onto its image. Altogether
\begin{align}
\Gamma^0 (\alpha^{*} TX) &\rightarrow C(S^1, \alpha^{*} TX) \rightarrow C(S^1, S^1 \cross TX) \notag \\
 &\cong C(S^1, S^1) \cross C(S^1, TX) \xrightarrow{\pi_2} C(S^1, TX), \notag
\end{align}
a sequence of continuous maps, each of which is a homeomorphism onto its image, except $\pi_2$, which is continuous and open.  We define the composition as $\theta^0$, a homeomorphism onto its image.

That $\Phi_{\alpha}$ is a homeomorphism has been shown already.  Since $(\pi, \eta)$ is a homeomorphism, by lemma \ref{l-co-top} so is $(\pi, \eta)_{\sharp}$ with unrestricted domain and codomain, and hence also restricted as here.  Then the composition ${\Psi_{\alpha}^0}^{-1}$ is a homeomorphism.  (Homeomorphism seemed convenient; only continuity is needed.)

Thus $\Psi_{\alpha}^0 (U^0 \cap U_{\alpha}^0)$ is an open set in $\Gamma^0(\alpha^{*}TX)$.  From the metric on $\alpha^{*} TX$ coming from the Riemannian metric on $TX$ and the standard metric on $S^1$, we have by lemma \ref{l-co-metr-top} a metric $d^0$ on $C(S^1, \alpha^{*} TX)$ compatible with the compact-open topology; write its open balls with $\B^0$.  Thus there is some $\delta' > 0$ such that $\B_{\delta'}^0 (0) \subset \Psi_{\alpha}^0 (U^0 \cap U_{\alpha}^0)$.  The metric $d^0$ in turn is used to define the \Frechet space seminorms $\norm{}_0, \norm{}_1$, etc. on $\Gamma(\alpha^{*} TX)$.  Since \Frechet spaces are metrizable, let $d^F$ denote a translation-invariant metric from lemma \ref{l-frec-spac} that induces the \Frechet topology; write its open balls with $\B^F$.

Now $\Gamma(\alpha^{*} TX) \subset C(S^1, \alpha^{*} TX)$, so $\B_{\delta'}^F (0) \subset \{ \beta \in \Gamma(\alpha^{*} TX) \st \frac{\norm{\beta}_0}{1 + \norm{\beta}_0} \le \norm{\beta}_0 < \delta' \} \subset \B_{\delta'}^0 (0)$.  Since $\B_{\delta'}^F (0)$ is an open set in $\Gamma(\alpha^{*} TX)$, $W = \Psi_{\alpha}^{-1} (\B_{\delta'}^F (0))$ is an open set in $U_{\alpha} \subset LX$, with $\incl(W) \subset U^0$.

As a consequence of the continuity of the inclusion $LX \subset C(S^1,X)$ and the $U_{\alpha}$ being open in the induced topology, corollary 3.23 of \citet[pages~19--20]{Stac05} states that $LX$ is Hausdorff, regular, second countable, paracompact, and hence he concludes it is metrizable.
\end{proof}
\begin{note}\label{n-frec-spac-frec-mfld}
\index{loop of vector space!Frechet manifold@\Frechet manifold}
\index{Frechet manifold@\Frechet manifold!loop of vector space}
(The Loop of a Vector Space as a \Frechet Space or \Frechet Manifold).
If $X$ is a finite-dimensional vector space, then the topology of $LX$ as a \Frechet space (see example \ref{e-lrn-l-vect-spac}) is the same as the topology as a \Frechet manifold in proposition \ref{p-smth-free-loop-spac-frec-mfld}.  To see this, use note \ref{n-loc-addn}.
\end{note}

\section{\Frechet Bundles}\label{s-frec-bndl}

\begin{defn}\label{d-frec-fb}
\index{Frechet@\Frechet!fiber bundle}
\index{bundle!fiber!Frechet@\Frechet}
\index{fiber bundle!Frechet@\Frechet}
(\Frechet Fiber Bundles).
A \Frechet fiber bundle is a topological fiber bundle as in definition \ref{d-fb-vb} where the spaces are \Frechet manifolds and the local trivializations are \Frechet diffeomorphisms.
\end{defn}

\begin{defn}\label{d-frec-vb}
\index{Frechet@\Frechet!vector bundle}
\index{vector bundle!Frechet@\Frechet}
\index{bundle!vector!Frechet@\Frechet}
(\Frechet Vector Bundles With Standard Fiber).
A \Frechet vector bundle $\pi \colon T \rightarrow B$ with standard  fiber the \Frechet space $F$ is a \Frechet fiber bundle with standard  fiber $F$ when each fiber $\pi^{-1}(b)$, $b \in B$, is a \Frechet space, and the local trivializations induce \Frechet space isomorphisms on fibers.
\end{defn}

\begin{defn}\label{d-frec-vb-hami}
\index{Frechet@\Frechet!vector bundle}
\index{vector bundle!Frechet@\Frechet}
\index{bundle!vector!Frechet@\Frechet}
(Hamilton's \Frechet Vector Bundles With Standard Fiber).
Definition I.4.3.1 of \citet[page~88]{Hami82}, re-worded for compatibility and explicitness, altered to require that each manifold has charts in a \Frechet space, and the \Frechet vector bundle has a standard fiber, is as follows.

A \Frechet vector bundle $\pi \colon T \rightarrow B$ consists of \Frechet manifolds $T$, $B$ with charts taking values in \Frechet spaces $C$, $D$ respectively; a standard fiber $F$ that is a \Frechet space, and a surjection $\pi$ such that each fiber $\pi^{-1}(b)$, $b \in B$, is a vector space, is a Hamilton's \Frechet vector bundle with standard fiber if the following holds:

Given any point $b \in B$, there is an open neighborhood $U$ of $b$ and a \Frechet manifold coordinate chart $\phi \colon U \rightarrow \phi(U) \subset D$ such that there is another \Frechet manifold coordinate chart $\psi \colon \pi^{-1}(U) \rightarrow \phi(U) \cross F \subset C$ such that $\psi$ induces a \Frechet space isomorphism on fibers and the following diagram commutes:
\[
\begindc{\commdiag}[5]
\obj(10,25)[objPiInvU]{$\pi^{-1}(U)$}
\obj(25,25)[objPhiUCrossF]{$\phi(U) \cross F$}
\obj(10,13)[objU]{$U$}
\obj(10,0)[objPhiU]{$\phi(U)$}
\mor{objPiInvU}{objU}{$\pi$}[\atright, \solidarrow]
\mor{objU}{objPhiU}{$\phi$}[\atright, \solidarrow]
\mor{objPiInvU}{objPhiUCrossF}{$\psi$}
\mor{objPhiUCrossF}{objPhiU}{$\pi_1$}
\enddc
\]
\end{defn}

\begin{lem}\label{l-frec-vb-def}
\index{Frechet@\Frechet!vector bundle}
\index{vector bundle!Frechet@\Frechet}
\index{bundle!vector!Frechet@\Frechet}
(Equivalence of Definitions of \Frechet Vector Bundle).
Definitions \ref{d-frec-vb} and \ref{d-frec-vb-hami} are equivalent.
\end{lem}
This can be seen by taking suitable \Frechet manifold charts.

% NOTE!  There seems to be a bug in Miktex 2.9 as of 02/02/2011,
% whereby the construct \begin{note}\label{...} \citet[]{} results
% in no "\emph{Note} number" in the compiled document.  My
% workaround is to give titles to notes.  The bug might also affect
% \begin{lem} etc.

\begin{note}\label{n-frec-mfld-tang-bndl-spac}
\index{Frechet@\Frechet!manifold tangent bundle}
(A \Frechet Manifold's Tangent Bundle and Tangent Space).
\citet[page~89]{Hami82} states that the tangent bundle $TX$ of a \Frechet manifold $X$ is a \Frechet vector bundle, and gives its transition functions.  He relates the tangent space $T_x X$ at a point $x \in X$ to parametrized at-least-differentiable curves in $X$ passing through $x$, giving the identification we will see in note \ref{n-loop-spac-tang-spac-bndl}:  for $X$ a smooth manifold and $\gamma \in LX$, we identify $T_{\gamma}LX$ with the space of smooth vector fields along $\gamma$, $\Gamma(\gamma^{*}TX)$.
\end{note}

\begin{prop}\label{p-loop-smth-map-smth}
\index{loop of smooth map}
(The Loop of a Smooth Map $Lf \colon LT \rightarrow LB$ is Smooth).
Given smooth manifolds $T$, $B$, and a smooth map $f \colon T \rightarrow B$, $Lf \colon LT \rightarrow LB$ is smooth.
\end{prop}
This is the second part of example I.4.4.5 in \citet[page~91]{Hami82}.

\begin{cor}\label{co-smth-loop-func}
\index{looping!functor}
(Smooth Looping is a Functor).
Smooth looping is a functor from the category of orientable smooth manifolds with Riemannian metric and smooth maps, to the category of \Frechet manifolds and smooth maps.
\end{cor}
\begin{proof}
This follows from from propositions  \ref{p-smth-free-loop-spac-frec-mfld} and \ref{p-loop-smth-map-smth}, and definition \ref{d-loop-smth-map}.
\end{proof}

\begin{defn}\label{d-frec-lgrp}
\index{Frechet@\Frechet!Lie group}
\index{Lie group!Frechet@\Frechet}
\index{group!Lie!Frechet@\Frechet}
(\Frechet Lie Groups).
A \Frechet Lie group is a \Frechet manifold with group structure such that the multiplication and inverse maps are smooth maps of \Frechet manifolds.
\end{defn}
This is definition I.4.6.1 of \citet[page~98]{Hami82}.

\begin{prop}\label{p-loop-lgrp-frec-lgrp}
\index{loop!Lie group!Frechet Lie group@\Frechet Lie group}
(The Loop of a Lie Group is a \Frechet Lie Group).
Given a Lie group $G$, its loop $LG$ is a \Frechet Lie group.
\end{prop}
\begin{proof}
See \citet[pages~27--28]{PS86}.
\end{proof}
By proposition \ref{p-smth-free-loop-spac-frec-mfld}, since $G$ is
orientable \citep[page~92]{Bump04}, $LG$ is a \Frechet manifold. Defining the looped multiplication and inversion maps pointwise, if it were proved that the \Frechet manifold $L(G \cross G)$ is diffeomorphic to $LG \cross LG$, proposition \ref{p-loop-smth-map-smth} would imply that $LG$ is a \Frechet Lie group.

\begin{defn}\label{d-frec-lgrp-act}
\index{Frechet@\Frechet!Lie group action}
(\Frechet Lie Group Actions).
A \Frechet Lie group $G$ acts on a \Frechet manifold $X$ if there is a group action that is a smooth map of \Frechet manifolds (hence by definition is continuous).
\end{defn}
This is definition I.4.6.4 of \citet[page~98]{Hami82}.

\begin{prop}\label{p-loop-smth-lgrp-act-frec-lrgp-act}
\index{loop!Lie group action!Frechet Lie group action@\Frechet Lie group action}
\index{smooth!loop!projection}
(The Loop of a Smooth Lie Group Action is a \Frechet Lie Group Action).
Given a smooth action of a Lie group $G$ on a smooth manifold $P$, the action of $LG$ on $LP$ given pointwise by the action of $G$ on $P$ is a \Frechet Lie group action.
\end{prop}
\begin{proof}
This is Lemma 1.6 of \citet[page~811]{SW07}.
\end{proof}
This would also follow from propositions \ref{p-loop-lgrp-frec-lgrp} and \ref{p-loop-smth-map-smth}, if it were proved that $L(G \cross P)$ is diffeomorphic to $LG \cross LP$.

\begin{defn}\label{d-frec-pb}
\index{Frechet@\Frechet!principal bundle}
\index{bundle!principal!Frechet@\Frechet}
(\Frechet Principal Bundles).
Given a \Frechet Lie group $G$, a \Frechet fiber bundle $\pi \colon P \rightarrow B$ with standard fiber $G$ is a \Frechet principal $G$ bundle if $G$ acts smoothly on $P$, the action commutes with the projection, is free and transitive over each point of $B$, and the local trivialization maps are $G$-equivariant.
\end{defn}
This is modified from but equivalent to definition I.4.6.5 in \citet[pages~98--99]{Hami82}.

\begin{prop}\label{p-frec-pb-assb-frec-fb}
\index{Frechet@\Frechet!associated bundle}
\index{bundle!associated!Frechet@\Frechet}
(\Frechet Principal Bundle Associated Bundles are \Frechet Fiber Bundles).
Given a \Frechet Lie group $G$, \Frechet principal $G$ bundle $P \rightarrow B$, and a left \Frechet Lie group action of $G$ on a \Frechet manifold $F$, the associated bundle $P \cross_G F$ is a \Frechet fiber bundle over $B$.  If $F$ is a \Frechet space, then $P \cross_G F$ is a \Frechet vector bundle.
\end{prop}
The proof in \citet[pages~28--29]{Poor07} applies by adding \Frechet appropriately.

\chapter{LOOPS OF BUNDLES}\label{c-loop-bndl}

This background chapter discusses the loops of smooth principal bundles and associated smooth fiber bundles, in particular vector bundles, and fiberwise inner products on the last.  The results from chapter \ref{c-con-text} apply to the loops of the total space, base space, projection, and structure group, imply that the loop of the projection is a smooth map, and that the loop of a Lie group smooth action is a \Frechet Lie group action.

\begin{prop}\label{p-loop-smth-pb-is-frec-pb}
\index{principal bundle!loop}
\index{loop!principal bundle}
\index{bundle!principal!loop}
(The Loop of a Principal Bundle is a \Frechet Principal Bundle).
Given a smooth principal $G$ bundle $\pi \colon P \rightarrow X$, its loop $L\pi \colon LP \rightarrow L\pi(LP) \subset LX$ is a \Frechet principal $LG$ bundle.  If $X$ is one-connected or $G$ is connected, $L\pi (LP) = LX$.
\end{prop}
The first statement is Proposition 1.9 of % ~ removed from citation to cure overfull box
\citet[pages 812, 839--840]{SW07}; the second is in Remark 2 under Proposition 1.8 of \citet[pages~811--812,~838]{SW07}.  Our use will be with $G = \SO(n)$, which is connected.  Also note, using proposition \ref{p-smth-free-loop-spac-frec-mfld}, that the looped group, base, and total space have nice point set topological properties, being metrizable.

\begin{prop}\label{p-loop-fb-is-frec-fb}
\index{fiber bundle!loop}
\index{loop!fiber bundle}
\index{bundle!fiber!loop}
(The Loop of a Smooth Associated Fiber Bundle is a \Frechet Fiber Bundle).
Given a connected Lie group $G$, a smooth principal $G$ bundle $\pi_P \colon P \rightarrow X$, and a smooth manifold $F$ upon which $G$ acts smoothly, the associated bundle construction used in finite dimensions works also for \Frechet Lie groups, \Frechet principal bundles, and \Frechet manifolds, so that $\pi_{LP \cross_{LG} LF} \colon LP \cross_{LG} LF \rightarrow LX$ is a \Frechet fiber bundle.  The loop of $\pi_{P \cross_G F} \colon P \cross_G F \rightarrow X$, $L\pi_{P \cross_G F} \colon L(P \cross_G F) \rightarrow L\pi_{P \cross_G F}(L(P \cross_G F)) = LX$, is a \Frechet fiber bundle.  It and $LP \cross_{LG} LF \rightarrow LX$ are isomorphic as topological fiber bundles. If $F$ is a finite-dimensional vector space so that $P \cross_G F$ is a smooth vector bundle, then $LP \cross_{LG} LF$ is a \Frechet vector bundle, with vector operations on loops defined pointwise.
\end{prop}
\begin{proof}
Using proposition \ref{p-loop-smth-pb-is-frec-pb}, since $G$ is connected, $L\pi_P (LP) = LX$ and proposition \ref{p-frec-pb-assb-frec-fb} applies.  To show that the loop of the associated bundle is a \Frechet fiber bundle, by propositions \ref{p-smth-free-loop-spac-frec-mfld} and \ref{p-loop-smth-map-smth} the loops of the total space, base space, and projection are respectively \Frechet manifolds and a map thereof.  The map $L\pi_{P \cross_G F}$ is surjective as follows:  given any $\gamma \in LX$, take $\widetilde{\gamma} \in LP$ such that $\pi_P \circ \widetilde{\gamma} = \gamma$.  Then $t \mapsto [(\widetilde{\gamma}(t),0)]$ is an element of $L(P \cross_G F)$ that $L\pi_{P \cross_G F}$ maps to $\gamma$.  The second proof of Proposition 1.9 of \citet[pages~812,~838--840]{SW07} works here also; referring to \citet[pages~290--293]{Poor07} for parallel transport on associated fiber bundles.

To show the isomorphism as topological fiber bundles of the two \Frechet associated fiber bundles, consider the following diagram, where the first row of vertical arrows on the left comes from looping the quotient map and on the right is directly the quotient map, and the second row of vertical on the left comes from looping the projection and on the right is directly the projection.
\[
\begindc{\commdiag}[5]
\obj(10,30)[objLPxF]{$L(P \cross F)$}
\obj(30,30)[objLPxLF]{$LP \cross LF$}
\obj(10,15)[objLPxgF]{$L(P \cross_G F)$}
\obj(30,15)[objLPxLGLF]{$LP \cross_{LG} LF$}
\obj(10,0)[objLX1]{$LX$}
\obj(30,0)[objLX2]{$LX$}
\mor{objLPxLF}{objLPxF}{$\psi$}
\mor{objLPxF}{objLPxgF}{$L\pi_{P \cross F}$}
\mor{objLPxLF}{objLPxLGLF}{$\pi_{LP \cross LF}$}
\mor{objLPxLGLF}{objLPxgF}{$\overline{\psi}$}
\mor{objLPxgF}{objLX1}{$L\pi_{P \cross_G F}$}
\mor{objLPxLGLF}{objLX2}{$\pi_{LP \cross_{LG} LF}$}
\mor{objLX2}{objLX1}{$\ident_{LX}$}
\enddc
\]
Define the $LG$-equivariant continuous map $\psi$ in a natural way by mapping $(p,f) \in LP \cross LF$ to the element of $L(P \cross F)$ defined for all $t \in S^1$ by $\psi (p, f)(t) = (p(t), f(t))$.  Its continuous inverse is $\sigma \mapsto (\pi_1 \circ \sigma, \pi_2 \circ \sigma)$.  The map $\psi$ induces the homeomorphism $\overline{\psi}$ that it covers, because a continuous equivariant map descends to a continuous map of the orbit spaces; see \citet[page~4]{tomD87}.  In turn, $\overline{\psi}$ covers the identity of $LX$ because of the way the projection maps of the associated bundles are defined.  The diagram commutes.  Thus the two associated bundles are isomorphic as topological fiber bundles.
\end{proof}
Note, using proposition \ref{p-smth-free-loop-spac-frec-mfld}, that the looped group, base, and total space have nice point set topological properties, being metrizable.  Also, it seems plausible that $L(P \cross_G F) \rightarrow LX$ and $LP \cross_{LG} LF \rightarrow LX$ are isomorphic as \Frechet fiber bundles via the natural map $\overline{\psi}$, which can also be defined without reference to $\psi$.  However, that kind of isomorphism isn't needed here.

The following is a convenient view of the fiber of a looped smooth vector bundle.
\begin{lem}\label{l-fibr-loop-vb}
\index{fiber!loop!vector bundle}
\index{loop!vector bundle!fiber}
(The Fiber of a Looped Vector Bundle).
Given a smooth vector bundle $\pi \colon V \rightarrow X$ and its loop $L\pi \colon LV \rightarrow LX$, and $\gamma \in LX$, we may identify $LY_{\gamma} = (L\pi)^{-1} (\gamma) = \Gamma(\gamma^{*}V) = \Gamma_{\gamma} V$, the smooth sections of $\pi \colon V \rightarrow X$ along $\gamma$.
\end{lem}
\begin{proof}
$\widetilde{\gamma} \in (L\pi)^{-1} (\gamma) \Leftrightarrow \pi \circ \widetilde{\gamma} = \gamma \Leftrightarrow (\ident_{S^1}, \widetilde{\gamma}) \in \Gamma(\gamma^{*}V) = C^{\infty}(S^1, \gamma^{*}V) \Leftrightarrow \widetilde{\gamma} \in \Gamma_{\gamma} V$.  For more detail, see \citet[pages~36--37]{Poor07}.
\end{proof}

\begin{note}\label{n-loop-spac-tang-spac-bndl}
\index{loop space!tangent bundle}
(The Loop Space's Tangent Bundle).
For $V = E = TM$, we identify $TLM$ with $LTM$:  for $\gamma \in LM$, identify $T_{\gamma}LM$ with the space of smooth vector fields along $\gamma$, $\Gamma(\gamma^{*}TM) = (LTM)_{\gamma}$.  Imagine varying $\gamma$ by varying its value at each point of its domain, in any direction in $TM$.  See \citet[page~89]{Hami82}.
\end{note}

\begin{defn}\label{d-fibr-innr-prod-loop-vb}
\index{loop!fiberwise inner product}
\index{fiberwise inner product!loop}
\index{LEgamma@$LE_{\gamma}$}
(The Loop of a Fiberwise Inner Product for a Vector Bundle).
Given our smooth real vector bundle $\pi_E \colon E \rightarrow M$, and fiberwise inner product $(,)$ on $E$, for $\gamma \in LM$, denote the pullback fiberwise inner product on $\gamma^{*}E$ by $\gamma^{*}(,)$. For every $t \in S^1$, $\gamma^{*}(,)_t \colon ((\gamma^{*}E) \tensor (\gamma^{*}E))_t  = (\gamma^{*}E)_t \tensor (\gamma^{*}E)_t \rightarrow \RR$; for $v_1, v_2 \in E_{\gamma(t)}$, $\gamma^{*}((t, v_1), (t, v_2))_t = (v_1, v_2)$.  We may consider that $\gamma^{*}(,) \colon \Gamma(\gamma^{*}E) \tensor \Gamma(\gamma^{*}E) = LE_{\gamma} \tensor LE_{\gamma} \rightarrow L\RR$.  Using $\int_{S^1} \colon L\RR \rightarrow \RR$, define $(,)_{LE_{\gamma}} \colon LE_{\gamma} \cross LE_{\gamma} \rightarrow \RR$ as $(\int_{S^1}) \circ \gamma^{*}(,)$.  I.e., for $\sigma, \tau \in LE_{\gamma}$, let $(\sigma, \tau)_{LE_{\gamma}} = \int_0^1 (\sigma (t), \tau (t))_{\gamma} dt$.

$\overline{LE_{\gamma}}$ will denote the real Hilbert space completion for this inner product.
\end{defn}
The looped fiberwise inner product can be defined globally in a way that allows talking about its continuity or smoothness, using ideas from \citet[pages~58--63,~173--175]{Lang99}.  Presumably it is continuous and even smooth, but we don't prove this, as it is only used for the more geometric viewpoint that is not used for proofs, of constructions that are actually made using associated bundles.

\begin{defn}\label{d-herm-ext-fibr-innr-prod}
\index{fiberwise inner product!Hermitian extension}
\index{complexification!vector bundle}
\index{vector bundle!complexification}
\index{bundle!vector!complexification}
(The Hermitian Extension of a Fiberwise Inner Product).
Consistent with definition \ref{d-v-h} and assumption \ref{a-herm-innr-prod}, given a fiberwise inner product $(,)$ on a smooth real vector bundle (e.g. a Riemannian metric on a tangent bundle), denote by $\langle , \rangle$ the Hermitian fiberwise inner product on the complexification of the vector bundle, defined by $\langle \alpha x, \beta y \rangle = \alpha \overline{\beta} (x, y)$, for all $x$, $y$ in the same fiber of the total space of the bundle and all $\alpha$, $\beta$ in $\CC$.
\end{defn}

\begin{eg}\label{e-loop-soe}
\index{loop!SO(E)@$\SO(E)$}
\index{SO(E)@$\SO(E)$!loop}
\index{LSO(E)@$L\SO(E)$}
(The Loop of $\SO(E)$).
The group $\SO(n)$ is connected, so $\pi_{L\SO(E)}$ is surjective.  The right-hand isomorphism below the diagrams respects fiberwise inner products, is proved as in special case $E = TM$, $TLM = LTM$ in \citet[pages~828--829]{SW07}, or for topological bundles by looping the left-hand isomorphism.  See section \ref{s-tech-over} of chapter \ref{c-intr} for an introduction to our use of associated bundles.
\[
\begindc{\commdiag}[5]
\obj(10,20)[objSOn]{$\SO(n)$}
\obj(23,20)[objSOE]{$\SO(E)$}
\obj(23,10)[objM]{$M$}
\mor{objSOn}{objSOE}{}
\mor{objSOE}{objM}{$\pi_{\SO(E)}$}

\obj(40,20)[objLSOn]{$L\SO(n)$}
\obj(55,20)[objLSOE]{$L\SO(E)$}
\obj(55,10)[objLM]{$LM$}
\mor{objLSOn}{objLSOE}{}
\mor{objLSOE}{objLM}{$\pi_{L\SO(E)}$}
\enddc
\]
\begin{align}
E \cong \SO(E) \cross_{\SO(n)} \RR^n && LE \cong L\SO(E) \cross_{L\SO(n)} L\RR^n \notag
\end{align}
\end{eg}

\chapter{CLIFFORD ALGEBRAS AND FOCK REPRESENTATIONS}\label{c-clif-alg-fock-rep}

This long background chapter discusses Clifford algebras and some of their representations, which will play a vital role hereafter.  Mixed in are occasional facts from other subjects used in the context of the moment.

The elements of the Fock spaces used to represent Clifford algebras are generalizations of the spinors that Paul Dirac used in about 1927 in the Dirac equation, initially to formulate the relativistic quantum mechanics of a free electron including its spin.  This equation led to the prediction of the positron, and to quantum field theories, in which spinors are used.

In Dirac's theory, the Clifford algebras that act on the Fock spaces, and the Fock spaces themselves, are defined by the pseudo-Riemannian metric of spacetime.  In this thesis, spacetime is replaced by the loop space of a smooth manifold, and the Clifford algebras come from the loop of an even rank smooth vector bundle with fiberwise inner product over the manifold.  We won't be able to define a Fock representation over the whole looped manifold, but will construct an irreducible representation related to Fock spaces.  The algebraic basis of Clifford algebras and Fock spaces that we use, is for a flat infinite-dimensional spacetime.  Since our Clifford algebras and the Fock spaces are over $\CC$, we can start with the standard metric on $\RR^n$ before looping, not worrying with the signature of the metric, as we would with in the Minkowski metric.

We will look at Clifford algebras first purely algebraically, then with a norm, as $C^{*}$-algebras.  We will not use Clifford algebras built on finite-dimensional real inner product spaces, since looping introduces infinite-dimensionality.  Nor will we much need the flexibility for our real inner product space to be incomplete; besides which \citet[page~28]{PR94} give a proposition that the completed $C^{*}$ Clifford algebra built from an incomplete vector space and that built from its completion are naturally isomorphic.  \citet{PR94} do not require separability of the inner product space the Clifford algebra is built from, but our actual use is of a separable real Hilbert space.

Many definitions and results from \citet{PR94} are included, often with changed notation, and with some additions such as proving continuity.  Another general reference for this and other chapters that should be mentioned, because I spent a lot of time trying to understand more precisely some of the host of interesting things in it, though very little is taken directly without proof, is \citet{ PS86}.

\begin{ass}\label{a-alg}
\index{algebras}
\index{assumptions!algebras}
(Algebras).
All algebras in this document are unital, associative, and complex.  The unit will be denoted $1$.
\end{ass}

\begin{ass}\label{a-v-real-hilb-spac}
\index{V@$V$!real Hilbert space}
\index{assumptions!V@$V$ real Hilbert space}
\index{Hilbert space!separable}
\index{assumptions!Hilbert space separable}
($V$ is a Real Hilbert Space).
Henceforth, unless otherwise specified, the real vector space named $V$ with inner product $(,)$ upon which the Clifford algebra $\Cl(V)$ and the Fock spaces on which it acts are built, is a separable infinite-dimensional real Hilbert space.
\end{ass}
Much of our background exposition would work for even-dimensional spaces; but our applications to loop spaces will all be infinite-dimensional.

\section{Clifford Algebras}\label{s-clif-alg}

The Clifford algebra of a real vector space $V$ with inner product $(,)$ allows multiplication of the elements of the vector space, subject to the Clifford relations $v^2 = (v,v) = \norm{v}^2$.  Some authors use $v^2 = - \norm{v}^2$, but with a complex algebra this is not an essential difference.

\begin{defn}\label{d-clif-alg}
\index{Clifford algebra!construction}
(The Clifford Algebra Construction).
\citep[pages~2--4]{PR94} Define the Clifford algebra $\cl(V)$ associated to $V$ as follows:
\begin{align}
T(V) &= \dirsum_{r=0}^{\infty} \tensor_{\CC}^r (\CC \tensor_{\RR} V) \text{, with } 1 \in \CC = \tensor^0 (\CC \tensor V) \text{ as multiplicative identity}, \notag \\
I(V) &= (\{v \tensor v - (v,v)1 \st v \in V \subset \CC \tensor V\}) \text{, a two-sided ideal, and}\notag \\
\cl(V) &= T(V) / I(V). \notag
\end{align}
We may write products in $\cl(V)$ by juxtaposition, as usual for an algebra, omitting the tensor symbols, so that a generic monomial element, using a canonical isomorphism of tensor products, would be $z v_1 \cdots v_k$, with $z \in \CC$ and $v_1 \dots v_k \in V$.
\end{defn}

\begin{lem}\label{l-clif-alg-unio-fini-dim}
\index{Clifford algebra!union finitely generated algebras}
(A Clifford Algebra is a Union of Finitely Generated Algebras).
The Clifford algebra of $V$,
\[
\cl(V) = \bigcup \{\cl(W) \st W \text{ is a finite-dimensional subspace of } V \}. \notag
\]
\end{lem}
This is theorem 1.1.13 of \citet[page~17]{PR94}.

\section{$C^{*}$ Clifford Algebras}\label{s-c-star-clif-alg}

In general, $C^{*}$-algebras can be thought of as abstractions of closed subalgebras of the algebras of bounded operators on Hilbert spaces, with operator norm.  We will get there in several steps.

\begin{defn}\label{d-star-alg}
\index{*-algebra@${}^{*}$-algebra}
(${}^{*}$-Algebras).
A ${}^{*}$-algebra $A$ is an algebra with an involution ${}^{*}$, a complex-antilinear anti-automorphism; i.e., for $a, b \in A$ and $\alpha \in \CC$:
\begin{align}
(a^{*})^{*}    &= a \notag \\
(a + b)^{*}    &= a^{*} + b^{*} \notag \\
(\alpha a)^{*} &= \overline{\alpha} a^{*} \notag \\
(a b)^{*}      &= b^{*} a^{*}. \notag
\end{align}
\end{defn}

\begin{defn}\label{d-star-alg-z2-grad}
\index{*-algebra@${}^{*}$-algebra!Z2 graded@$\ZZ_2$ graded}
\index{Z2 graded@$\ZZ_2$ graded!*-algebra@${}^{*}$-algebra}
($\ZZ_2$ Graded ${}^{*}$-Algebras).
A $\ZZ_2$ graded ${}^{*}$-algebra $A$ is a ${}^{*}$-algebra that is a direct sum $A = A_0 \dirsum A_1$, with $\partial \colon A_0 \cup A_1 \rightarrow \ZZ_2$ defined to have value $0$ on $A_0$, the set of even elements, and $1$ on $A_1$, the set of odd elements, satisfying for $a, b \in A_0 \cup A_1$:
\begin{align}
\partial a^*   &= \partial a \notag \\
\partial a b   &= \partial a + \partial b. \notag
\end{align}
\end{defn}

\begin{defn}\label{d-clif-alg-invo}
\index{Clifford algebra!involution}
(The Clifford Algebra Involution).
\citep[pages~5--7]{PR94} Define the involution ${}^{*}$ of $\cl(V)$, making it a ${}^{*}$-algebra, as the complex-antilinear antiautomorphism such that for any monomial $a = z v_1 \cdots v_k$, $z \in \CC$, $v_1 \dots v_k \in V$, $a^{*}= \overline{z} v_k \cdots v_1$, the complex conjugate of $a$ with the order of factors reversed.
\end{defn}
\begin{note}\label{d-clif-alg-z2-grad}
\index{Clifford algebra!Z2 graded@$\ZZ_2$ graded}
\index{Z2 graded@$\ZZ_2$ graded!Clifford algebra}
($\ZZ_2$ Graded Clifford Algebras).
The Clifford algebra $\cl(V)$ is $\ZZ_2$ graded, with $\cl(V)_0$ the scalars and products of even numbers of vectors, and $\cl(V)_1$ the products of odd numbers of vectors.
\end{note}

\begin{defn}\label{d-hilb-spac-z2-grad}
\index{Hilbert space!Z2 graded@$\ZZ_2$ graded}
\index{Z2 graded@$\ZZ_2$ graded!Hilbert space}
($\ZZ_2$ Graded Hilbert Spaces).
A $\ZZ_2$ graded Hilbert space is an orthogonal direct sum $H = H_0 \dirsum H_1$, with $\partial \colon H_0 \cup H_1 \rightarrow \ZZ_2$ defined to have value $0$ on $H_0$, the set of even elements, and $1$ on $H_1$, the set of odd elements.
\end{defn}
\index{parity!Z2 grading@$\ZZ_2$ grading!Hilbert space}
Given an orthogonal direct sum $H = H' \dirsum H''$ to be used as a $\ZZ_2$ grading of $H$, we speak of the parity of the grading in connection with the choice of whether $H'$ or $H''$ is defined as $H_0$.

\begin{eg}\label{e-star-alg-boun-op-hilb-spac}
\index{*-algebra@${}^{*}$-algebra!B(H)@$\B(H)$}
(The ${}^{*}$-Algebra of Bounded Operators on Hilbert Space).
$\B(H)$, the bounded operators on a Hilbert space, form a ${}^{*}$-algebra with the Hilbert space adjoint as the involution.

The set of bounded operators $\B(H)$ is a $\ZZ_2$ graded ${}^{*}$-algebra with $\B(H)_0$ the set of operators that map $H_0 \rightarrow H_0$ and $H_1 \rightarrow H_1$, and $\B(H)_1$ those that map $H_0 \rightarrow H_1$ and $H_1 \rightarrow H_0$.
\end{eg}

\begin{defn}\label{d-star-mor}
\index{*-morphism@${}^{*}$-morphism}
\index{*-algebra@${}^{*}$-algebra!morphism}
(${}^{*}$-Morphisms).
A ${}^{*}$-morphism is an algebra morphism of ${}^{*}$-algebras that preserves the involution.  That is, an algebra morphism $\phi \colon A \rightarrow B$ of ${}^{*}$-algebras is a ${}^{*}$-morphism if for all $a \in A$:
\[
\phi(a^{*}) = (\phi(a))^{*}.
\]
\end{defn}
A ($\ZZ_2$ graded) ${}^{*}$-morphism of $\ZZ_2$ graded ${}^{*}$-algebras preserves the gradings.

\begin{defn}\label{d-star-alg-rep}
\index{*-algebra@${}^{*}$-algebra!representation}
\index{representation!C*-algebra@${}^{*}$-algebra}
\index{*-algebra@${}^{*}$-algebra!module}
\index{module!C*-algebra@${}^{*}$-algebra}
(${}^{*}$-Algebra Representations or Modules).
A representation, also called a ${}^{*}$-representation or a ${}^{*}$-module, of a ${}^{*}$-algebra $A$ on a Hilbert space $H$ is a ${}^{*}$-morphism $\rho \colon A \rightarrow \B(H)$.

A representation is irreducible when there are no nontrivial closed subspaces of its Hilbert space that are invariant under it.
\end{defn}
\begin{note}\label{n-star-alg-rep}
\index{representation!vs. module}
\index{module!vs. representation}
(Representation vs. Module Terminology).
The two terms representation and module are used interchangeably for convenience.  We may use representation when thinking more of the ${}^{*}$-morphism, and module when thinking more of the Hilbert space.  Note that a ${}^{*}$-module is not just a module for the algebra as a ring.  If we needed this concept, we could let, say, $R_A$ denote $A$ considered as a ring, ignoring complex multiplication, the involution, and the norm.  Then, given a representation of $A$ on a Hilbert space $H$, speaking of $H$ as an $R_A$ module would imply ignoring the extra structure on $A$ and the inner product structure on $H$.

A $\ZZ_2$ graded representation $\rho$ of a $\ZZ_2$ graded ${}^{*}$-algebra $A = A_0 \dirsum A_1$ on a $\ZZ_2$ graded Hilbert space $H = H_0 \dirsum H_1$ is a $\ZZ_2$ graded ${}^{*}$-morphism to the $\ZZ_2$ graded ${}^{*}$-algebra $\B(H)$.  Working out the details, for $a \in A_0 \cup A_1$ and $v \in H_0 \cup H_1$,
\[
\partial \rho(a) (v) = \partial a + \partial v. \notag
\]
\end{note}

\begin{defn}\label{d-ban-alg}
\index{Banach algebra}
(Banach Algebras).
\citep[pages~7,~128]{Pede89} A Banach algebra $A$ is an algebra, complete as a normed vector space, whose norm satisfies
\[
a, b \in A \Rightarrow \norm{a b} \le \norm{a} \norm{b}.
\]
\end{defn}

\begin{defn}\label{d-c-star-alg}
\index{C*-algebra@$C^{*}$-algebra}
($C^{*}$-Algebras).
\citep[page~2]{Take02} A $C^{*}$-algebra $A$ is a Banach ${}^{*}$-algebra such that:
\[
a \in A \Rightarrow \norm{a^{*} a} = \norm{a}^2.
\]
\end{defn}
\begin{note}\label{n-c-star-alg}
\index{C*-algebra@$C^{*}$-algebra!norm a star@$\norm{a^{*}}$}
($C^{*}$-Algebras).
If $A$ is a $C^{*}$-algebra, $a \in A \Rightarrow \norm{a^{*}} = \norm{a}$.
\end{note}

\begin{defn}\label{d-c-star-alg-z2-grad}
\index{C*-algebra@$C^{*}$-algebra!Z2 graded@$\ZZ_2$ graded}
\index{Z2 graded@$\ZZ_2$ graded!C*-algebra@$C^{*}$-algebra}
($\ZZ_2$ Graded $C^{*}$-Algebras).
A $\ZZ_2$ graded $C^{*}$-algebra $A$ is one that is graded $A = A_0 \dirsum A_1$ as a ${}^{*}$-algebra, with $A_0$ and $A_1$ closed subspaces.
\end{defn}

\begin{eg}\label{e-c-star-alg-boun-op-hilb-spac}
\index{C*-algebra@$C^{*}$-algebra!B(H)@$\B(H)$}
(The $C^{*}$-Algebra of Bounded Operators on Hilbert Space).
$\B(H)$, the bounded operators on a Hilbert space, form a $C^{*}$-algebra with the Hilbert space adjoint as involution, for in the operator norm, $\B(H)$ is complete, and $\norm{a^{*}} = \norm{a}$.

In the $\ZZ_2$ grading of example \ref{e-star-alg-boun-op-hilb-spac}, $\B(H)_0$ and $\B(H)_1$ are closed subspaces, so $\B(H)$ is $\ZZ_2$ graded as a $C^{*}$-algebra.
\end{eg}

\begin{defn}\label{d-c-star-alg-mor}
\index{C*-algebra@$C^{*}$-algebra!morphism}
% (A title for this definition seemed redundant).
A morphism of $C^{*}$-algebras is a ${}^{*}$-morphism of ${}^{*}$-algebras.
\end{defn}
A $\ZZ_2$ graded morphism of $\ZZ_2$ graded $C^{*}$-algebras is defined like one for $\ZZ_2$ graded ${}^{*}$-algebras, the only difference being the requirement that the even and odd subspaces of the algebras be closed.

\begin{lem}\label{l-c-star-alg-mor-cont}
\index{C*-algebra@$C^{*}$-algebra!morphism!continuous}
($C^{*}$-Algebra Morphisms are Continuous).
\citep[pages 21--22]{Take02} Every morphism $f \colon A \rightarrow B$ between two $C^{*}$-algebras is continuous, and $\norm{f} \le 1$.  If $f$ is injective, it is an isometry; i.e., $a \in A \Rightarrow \norm{f(a)} = \norm{a}$.
\end{lem}
Thus any representation of a $C^{*}$-algebra is continuous.  Note: unlike \citet{Take02}, we call $f$ an isomorphism only if it is bijective.

\begin{lem}\label{l-clif-alg-rep}
\index{C*-algebra@$C^{*}$-algebra!Clifford}
\index{C*-algebra@$C^{*}$-algebra!Clifford!norm} \index{C*-algebra@$C^{*}$-algebra!Clifford!representation}
(Representations and Completion of the Clifford Algebra $\cl(V)$).
\citep[pages~22--26]{PR94} Nonzero ${}^{*}$-representations $\pi$ of the ${}^{*}$-algebra $\cl(V)$ exist.  For any such $\pi$, $a \in \cl(V) \Rightarrow \norm{a} = \norm{\pi(a)}$ defines a norm on $\cl(V)$, and all such norms are equal.  The completion of $\cl(V)$ with respect to that norm is a $C^{*}$-algebra.
\end{lem}

\begin{defn}\label{d-c-star-clif-alg}
\index{CL(V)@$\Cl(V)$}
\index{C*-algebra@$C^{*}$-algebra!Clifford}
\index{C*-algebra@$C^{*}$-algebra!Clifford!norm}
(The $C^{*}$ Clifford Algebra $\Cl(V)$).
\citep[page~24]{PR94} Let $\Cl(V)$ denote the completion of $\cl(V)$ with respect to the norm of lemma \ref{l-clif-alg-rep}.
\end{defn}

\begin{defn}\label{d-clif-map-self-adj}
\index{Clifford!map!self-adjoint}
(Self-Adjoint Clifford Maps).
\citep[pages~2,~7]{PR94} A self-adjoint Clifford map on $V$ is a real-linear map $f \colon V \rightarrow B$ into a ${}^{*}$-algebra $B$ such that $v \in V \Rightarrow f(v)^2 = (v,v) 1 \text{ and } f(v)^{*} = f(v)$.  For $B = \Cl(V)$, we define its Clifford map $\phi (v) = v$, considered as an element of $\Cl(V)$.
\end{defn}

\begin{prop}\label{p-c-star-clif-alg}
\index{CL(V)@$\Cl(V)$}
\index{C*-algebra@$C^{*}$-algebra!Clifford}
\index{C*-algebra@$C^{*}$-algebra!Clifford!map}
\index{C*-algebra@$C^{*}$-algebra!Clifford!simple}
(The $C^{*}$ Clifford Algebra $\Cl(V)$ is Simple; the Clifford Map is an Isometry).
\citep[pages~24--26]{PR94} $\Cl(V)$ is simple (has no nontrivial two-sided ideals).  The self-adjoint Clifford map $\phi \colon V \rightarrow \Cl(V)$ is an isometry.
\end{prop}

\begin{defn}\label{d-c-star-clif-alg-rep}
\index{CL(V)@$\Cl(V)$}
\index{Clifford algebra!representation}
\index{representation!Clifford algebra}
\index{Clifford algebra!module}
\index{module!Clifford algebra}
\index{C*-algebra@$C^{*}$-algebra!Clifford!representation}
($C^{*}$ Clifford Algebra $\Cl(V)$ Representations or Modules). \linebreak
\citep[page~35]{Take02} A representation of $\Cl(V)$, or a $\Cl(V)$ module, is a ${}^{*}$-representation of it as a $C^{*}$-algebra.  It may be called a Clifford algebra representation, a Clifford or Clifford algebra module.
\end{defn}

\begin{prop}\label{p-c-star-clif-alg-univ}
\index{CL(V)@$\Cl(V)$}
\index{C*-algebra@$C^{*}$-algebra!Clifford!universal property}
(The $C^{*}$ Clifford Algebra $\Cl(V)$ Universal Property).
The Clifford algebra $\Cl(V)$ over $V$ and its self-adjoint Clifford map $\phi$ satisfy the universal property that if $f \colon V \rightarrow B$ is any self-adjoint Clifford map to a $C^{*}$-algebra $B$, there is a unique morphism of $C^{*}$-algebras (a unique ${}^{*}$-morphism) $F \colon \Cl(V) \rightarrow B$ such that $F \circ \phi = f$:
\[
\begindc{\commdiag}[5]
\obj(10,25)[objV]{$V$}
\obj(25,25)[objClV]{$\Cl(V)$}
\obj(25,10)[objB]{$B$}
\mor{objV}{objClV}{$\phi$}
\mor{objV}{objB}{$f$}
\mor{objClV}{objB}{$F$}
\enddc
\]
$F$ maps each monomial in $\cl(V) \subset \Cl(V)$ to the monomial in $B$ obtained by replacing each vector in the monomial with its image under f.  The resulting map is extended by continuity to $\Cl(V)$.
\end{prop}
\begin{proof}
See \citep[pages~2--7,~26]{PR94}.
\end{proof}

\begin{prop}\label{p-c-star-clif-alg-func}
\index{CL(V)@$\Cl(V)$}
\index{Clifford algebra!functor}
\index{C*-algebra@$C^{*}$-algebra!Clifford!functor}
(The $C^{*}$ Clifford Algebra Functor).
We may define a functor $\Cl$ from the category $\mathcal{V}_{\mathcal{I}}$ whose objects are separable infinite-dimensional real Hilbert spaces whose morphisms are isometric real-linear maps, to the category $\mathcal{C^{*}}$ of $C^{*}$-algebras and $C^{*}$-algebra morphisms (which are isometric), given by $V \mapsto \Cl(V)$ and $g \mapsto \Cl(g) = \theta_g$.  If $g$ is an isomorphism, so is $\theta_g$.  As with proposition \ref{p-c-star-clif-alg-univ}, $\theta_g$ replaces each vector in a monomial with its image under $g$.
\end{prop}
\begin{proof}
The basis is proposition \ref{p-c-star-clif-alg-univ}.  See \citep[pages~26--27]{PR94}.  That the functor maps composition of morphisms of $\mathcal{V}_{\mathcal{I}}$ to composition of morphisms of $\mathcal{A}$, and the fact about isomorphisms, follow from uniqueness in their theorem 1.2.5.
\end{proof}

\begin{defn}\label{d-c-star-bogo-auto}
\index{thetag@$\theta_g$}
\index{Bogoliubov automorphism}
\index{Clifford algebra!Bogoliubov automorphism}
\index{Bogoliubov automorphism!C*-algebra@$C^{*}$-algebra}
\index{C*-algebra@$C^{*}$-algebra!Clifford!Bogoliubov automorphism}
\index{Bogoliubov map}
\index{O(V)@$\Orth(V)$}
($C^{*}$ Bogoliubov Automorphisms and the Bogoliubov Map).
\citep[page~27]{PR94} Denote by $\Orth(V)$ the group of real-linear isometric automorphisms of $V$, and by $\Aut(\Cl(V))$ the group of $C^{*}$-algebra automorphisms of the Clifford algebra of $V$.  Then for $g \in \Orth(V)$, the map $\theta_g \in \Aut(\Cl(V))$ from proposition \ref{p-c-star-clif-alg-func} is called the Bogoliubov automorphism corresponding to $g$.  The map
\begin{align}
\theta \colon \Orth(V) &\rightarrow \Aut(\Cl(V)) \notag \\
g \mapsto \theta(g) &= \theta_g, \notag
\end{align}
is a group homomorphism we will call here the Bogoliubov map.  If $\Orth(V)$ is treated as a topological group, it has the operator norm topology unless otherwise stated.
\end{defn}

\begin{lem}\label{l-z2-grad-clif-c-star-alg}
\index{CL(V)@$\Cl(V)$}
\index{C*-algebra@$C^{*}$-algebra!Clifford!Z2 graded@$\ZZ_2$ graded}
\index{Clifford C*-algebra@Clifford $C^{*}$-algebra!Z2 graded@$\ZZ_2$ graded}
\index{Z2 graded@$\ZZ_2$ graded!Clifford C*-algebra@Clifford $C^{*}$-algebra}
(Clifford $C^{*}$-Algebras are $\ZZ_2$ Graded).
\citep[page~27]{PR94} For $V$ a real vector space with inner product, the Clifford $C^{*}$-algebra $\Cl(V)$ is $\ZZ_2$ graded by the Bogoliubov automorphism $\theta_{-1}$ corresponding to $- \ident \in \Orth(V)$; that is, $\Cl(V) = \Cl(V)_0 \dirsum \Cl(V)_1$, where $\Cl(V)_0$ is the $+1$ eigenspace and $\Cl(V)_1$ the $-1$ eigenspace of the operator.  The scalars $\CC$ and products of even numbers of vectors are in $\Cl(V)_0$, the vectors and products of odd numbers of vectors are in $\Cl(V)_1$, and those spaces are the closures in $\Cl(V)$ of the spans of those elements.
\end{lem}

\begin{defn}\label{d-spin-modu}
\index{module!spinor}
\index{spinor!module}
(Spinor Modules).
\citep[page~106]{BGV04}  A spinor module is an irreducible Clifford algebra module, an irreducible Clifford algebra representation.
\end{defn}
A fiber bundle of spinor modules is called a spinor bundle.

\section{Bogoliubov Map Continuity}\label{s-bogo-map-cont}

\begin{defn}\label{d-stro-op-top}
\index{topology!strong operator}
(The Strong Operator Topology).
\citep[page~169]{Foll99} Given two normed linear spaces $X$, $Y$, the strong operator topology on the linear space $\B(X,Y)$ of bounded operators is given by the seminorms $T \mapsto \norm{Tx}$ for all $x \in X$.  It is also given in terms of convergence by $T_i \rightarrow T \Leftrightarrow \forall x \in X, \text{ } T_i x \rightarrow T x$.

The term ``strong operator topology'' will also be applied to topological subspaces of spaces of bounded operators, with convergence defined the same way, since linear operations on the operators aren't involved.  For example, the term will be used for unitary groups and the group of automorphisms of $C^{*}$-algebras. \end{defn}
\vline

\begin{lem}\label{l-cont-stro-op-top-adj}
\index{topology!strong operator}
(Continuity of a Linear Action vs. Continuity of a Map to $\B(Y)$).
Let $X$ be a topological space and $Y$ a normed linear space.  Give $\B(Y)$ the strong operator topology.  Then the continuity of
\[
f \colon X \cross Y \rightarrow Y \notag
\]
in each variable separately, with $f$ linear in the second variable, is equivalent to continuity of
\[
g \colon X \rightarrow \B(Y), \notag
\]
defining one map from the other by $g(x)(y) = f(x,y)$.

Furthermore, if there is a bound $M$ for the operator norms of all the $g(x)$, then continuity of $g$ implies joint continuity of $f$, and hence, separate continuity of $f$ implies joint continuity of $f$.
\end{lem}
\begin{proof}
For the first part, continuity of $g$ is equivalent to continuity of $f$ in its first variable directly from definition \ref{d-stro-op-top}.  Continuity and linearity of $f$ in its second variable is equivalent to each $g(x) \in \B(Y)$. For the second part, suppose $\norm{g(x)} \le M$, $\forall x \in X$, and suppose $g$ is continuous; then $f$ is jointly continuous at any $(x_0, y_0)$:
\begin{align}
\norm{f(x, y) - f(x_0, y_0)} &\le \norm{f(x, y) - f(x, y_0)} + \norm{f(x, y_0) - f(x_0, y_0)} \notag \\
 &\le \norm{g(x)(y - y_0)} + \norm{g(x)(y_0) - g(x_0)(y_0)} \notag \\
 &\le \norm{g(x)} \norm{y - y_0} + \norm{g(x)(y_0) - g(x_0)(y_0)} \notag \\
 &\le M \norm{y - y_0} + \norm{g(x)(y_0) - g(x_0)(y_0)}. \notag
\end{align}
The first term can be made as small as desired by making $y$ close enough to $y_0$; and the second, by making $x$ close enough to $x_0$, by definition \ref{d-stro-op-top}.
\end{proof}

\begin{lem}\label{l-grp-act-sep-join-cont}
\index{Baire space}
\index{group!action!separate and joint continuity}
(The Separate Continuity of a Group Action Implies Joint Continuity).
\citep[page~1044]{CM70} Let $M$ be a metric space, and let $G$ be a Baire space with a group structure in which multiplication is separately continuous.  Let $\pi \colon G \cross M \rightarrow M$ be an action which is separately continuous.  Then $\pi$ is jointly continuous.
\end{lem}
\begin{note}\label{n-grp-act-sep-join-cont}
\index{Baire space!examples}
(Examples of Baire Spaces).
With the definition of Baire space from \citet[page~249]{Dugu66}, \citet[page~43]{Rudi91} says that a complete metric space, e.g. a closed subset of a Banach or \Frechet space, is a Baire space.
\end{note}

\begin{prop}\label{p-bogo-map-cont}
\index{thetag@$\theta_g$}
\index{Bogoliubov map!continuity}
\index{Bogoliubov automorphism}
\index{Clifford algebra!Bogoliubov automorphism}
\index{Bogoliubov automorphism!C*-algebra@$C^{*}$-algebra}
\index{C*-algebra@$C^{*}$-algebra!Clifford!Bogoliubov automorphism}
(The Continuity of the Bogoliubov Map).
The action of $\Orth(V)$ on $\Cl(V)$ via Bogoliubov automorphisms; i.e. the map $\Orth(V) \cross \Cl(V) \rightarrow \Cl(V)$ defined by $(g,a) \mapsto \theta(g)(a) = \theta_g (a) = g a$ (for notational simplicity) is (jointly) continuous and a fortiori separately continuous, with the operator norm topology on $\Orth(V)$ and $C^{*}$-norm topology on $\Cl(V)$.  Thus by the first part of lemma \ref{l-cont-stro-op-top-adj}, the Bogoliubov map $\theta \colon \Orth(V) \rightarrow \Aut(\Cl(V))$ is continuous, where $\Aut(\Cl(V))$ is given the strong operator topology.
\end{prop}
\begin{proof}
Let $g, g_0 \in \Orth(V)$, $a, a_o \in \Cl(V)$.  Use homogeneity:  since $g a - g_0 a_0 = g_0 (g_0^{-1} g a - a_0)$, and since $\theta_{g_0}$ is continuous (see definition \ref{d-c-star-bogo-auto}), it suffices to consider continuity at $g_0 = \ident \in \Orth(V)$.  Given $\epsilon > 0$ choose $a_0^u$ in $\cl(V)$ such that $\norm{a_0^u - a_0} < \epsilon / 4$ (possible by definition of completion), and require that $\norm{a - a_0} < \epsilon / 4$.  Then recalling that $\theta_g$ is an isometry,
\begin{align}
\norm{g a - a_0} &\le \norm{g (a - a_0)} + \norm{g (a_0 - a_0^u)} + \norm{g a_0^u - a_0^u} + \norm{a_0^u - a_0} \notag \\
          &\le \norm{a - a_0} + \norm{a_0 - a_0^u} + \norm{g a_0^u - a_0^u} + \norm{a_0^u - a_0} \notag \\
          &\le \frac{3}{4}\epsilon + \norm{g a_0^u - a_0^u}. \notag
\end{align}
By lemma \ref{l-clif-alg-unio-fini-dim} there is some finite-dimensional vector subspace $W$ of $V$ such that $a_0^u \in \Cl(W) \subset$ $\cl(V)$.  Letting $d = \dim(W)$, $a_0^u$ can be expressed as a sum of at most $2^d$ monomials, each of which is a product with complex coefficient of some subset of a chosen orthonormal basis of $W$, hence consisting of at most $d$ elements (see definition \ref{d-clif-alg} and \citet[pages~7--10,~23]{PR94}).

Referring to proposition \ref{p-c-star-clif-alg-func}, $g \prod w_j = \theta_g (\prod w_j) = \prod g(w_j)$, where $g (w_j)$ is the application of the orthogonal operator $g$ to $w_j$.  By using the triangle inequality and the Banach algebra norm inequality for $\Cl(V)$, and recalling from proposition \ref{p-c-star-clif-alg} that the $C^{*}$-algebra norm of a vector equals the inner product norm of the vector, we can make $\norm{g a_0^u - a_0^u} \le \epsilon / 4$, if by choosing $g$ close enough to $\ident$ we can make the following quantity arbitrarily small, for any set of $k \le d$ elements of the fixed orthonormal basis of $W$.  The notation assumes the products are in indexed order, and $\norm{}_{*}$, $\norm{}_V$, $\norm{}_{op}$ denote respectively the norms on $\Cl(V)$, $V$, and the operator norm on $\Orth(V)$.
\begin{align}
\norm{g \prod_{j=1}^k w_j - \prod_{j=1}^k w_j}_{*} &\le \sum_{i=1}^k \norm{\prod_{j=1}^{k - i + 1} g(w_j) \prod_{j=k - i + 2}^k w_j - \prod_{j=1}^{k - i} g(w_j) \prod_{j=k - i + 1}^k w_j}_{*} \notag \\
 &= \sum_{i=1}^k \norm{(\prod_{j=1}^{k - i} g(w_j)) (g(w_{k - i + 1}) - w_{k - i + 1}) (\prod_{j=k - i + 2}^k w_j)}_{*} \notag \\
 &\le \sum_{i=1}^k (\prod_{j=1}^{k - i} \norm{g(w_j)}_{*}) (\norm{g(w_{k - i + 1}) - w_{k - i + 1}}_{*}) (\prod_{j=k - i + 2}^k \norm{w_j}_{*}) \notag \\
 &= \sum_{i=1}^k \norm{g(w_{k - i + 1}) - w_{k - i + 1}}_{*} \notag \\
 &= \sum_{i=1}^k \norm{g(w_{k - i + 1}) - w_{k - i + 1}}_V \notag \\
 &= \sum_{i=1}^k \norm{(g - \ident) (w_{k - i + 1})}_V \le d \norm{g - \ident}_{op}, \notag
\end{align}
where products with starting index greater than ending index are taken as $1$.
\end{proof}
$\Aut(\Cl(V))$ is a topological group in the chosen topology, which is also called the point-norm topology, but that's not needed here.

\section{Unitary Structures and Lagrangian Subspaces}\label{s-unit-stru-lagr-subs}

This build-up to our Fock space construction starts with a separable infinite-dimensional real Hilbert space $V$.  In our main application of these concepts we start with $V$ and real-linear operators on it, and we use these a lot.  Essential results rely on analysis relating to real orthogonal operators on $V$.  \citet{PR94} start with $V$, and this approach makes it easier to use their results.  So it seems natural to think of $V$.

Complex numbers come in quickly: $\Cl(V)$ is a complex algebra, and we complexify $V$ to get $H$, where we find eigenspaces of complex-linear extensions of operators on $V$, obtaining Lagrangian subspaces of $H$, used directly in our Fock space construction.  We use these eigenspaces extensively, and they were used even more in initial attempts at proofs.  Their elements are concrete and are eigenvectors of derivative operators, relating to another way, which we shall not explore, of viewing the main constructions of the thesis.  So it also seems natural to think of $H$.

Instead of starting with $V$, we could have started with a combination of $H$ and a real structure $\Sigma$ on $H$, resulting in $V$ as the fixed point set of $\Sigma$.  However, this was not our choice.

In the end we use both $V$ and $H$, resisting the temptation to reduce everything to one or the other.  It takes a little time to become familiar with both, and with ways of relating them, but it isn't difficult, and this section shows how.

\begin{defn}\label{d-v-h}
\index{CL(V)@$\Cl(V)$!contains H@contains $H$}
\index{inner product!Hermitian extension}
\index{H@$H$!complexification of V@complexification of $V$}
\index{H@$H$!subset Cl(V)@$\subset \Cl(V)$!}
($H$ is the Complexification of $V$).
Consistent with assumption \ref{a-herm-innr-prod} and definition \ref{d-herm-ext-fibr-innr-prod}, the complexification of $V$ is $H = \CC \tensor V$, with Hermitian inner product $\langle , \rangle$ given for $\alpha, \beta \in \CC$ and $v, w \in V$ by:
\[
\langle \alpha v, \beta w \rangle = \alpha \overline{\beta} (v,w).
\]
\end{defn}
We identify $V$ with $\{ 1 \tensor v \st v \in V \} = \RR \tensor V \subset \CC \tensor V = H \subset \Cl(V)$.  Since we assume $V$ is a real Hilbert space, $H$ is a complex Hilbert space.

\begin{defn}\label{d-cx-conj}
\index{H@$H$!complex conjugation}
\index{Cl(V)@$\Cl(V)$!complex conjugation}
(Complex Conjugation in $H$ and $\Cl(V)$).
\citep[page~57]{PR94} Denote complex conjugation in $H = \CC \tensor V$ by $\Sigma$, where $\Sigma (z \tensor v) = \overline{z} \tensor v$.  Use $\Sigma$ also for the natural complex conjugation in $\Cl(V)$.
\end{defn}
We might on occasion write, e.g., $\overline{x}$ in place of $\Sigma(x)$.

\begin{defn}\label{d-real-structure}
\index{real structure}
(Real Structures).
\citep[pages~368--369]{Atiy66}
\begin{enumerate}
    \item Given a complex Hilbert space $H$, a real structure is a complex-antilinear isometric involution $\Sigma$; i.e. $\Sigma \colon H \rightarrow H$ such that for $x, y \in H$ and $\alpha \in \CC$:
        \begin{align}
        \Sigma (x + y)    &= \Sigma (x) + \Sigma (y) \notag \\
        \Sigma (\alpha x) &= \overline{\alpha} \Sigma (x) \notag \\
        \Sigma^2          &= \ident_H. \notag
        \end{align}
    \item If the complex Hilbert space $H = \CC \tensor V$, the complexification of a real Hilbert space $V$, we associate to $H$ the particular real structure $\Sigma$ given, for $\alpha \in \CC$ and $v \in V$, by:
        \[
        \Sigma (\alpha \tensor v) = \overline{\alpha} \tensor v.
        \]
\end{enumerate}
\end{defn}
$V$ may be recovered from $H$ as the fixed point set of $\Sigma$.

\begin{lem}\label{l-oph-opv}
\index{Operators on H vs V@Operators on $H$ vs $V$}
(Operators on $H$ vs. Operators on $V$).
Given $H = \CC \tensor V$ with real structure $\Sigma$ whose fixed point set is $V$, real-linear operators on $V$ extend to complex-linear operators on $H$ that commute with $\Sigma$; and conversely, complex-linear operators on $H$ that commute with $\Sigma$ can be restricted to (arise from) real operators on $V$.  The operator norm is unchanged by this extension or restriction.
\end{lem}
\begin{proof}
Given a real-linear operator $\phi$ on $V$, extend $\phi$ to a complex-linear operator on $H$ by defining for $v \in V$ and $\alpha \in \CC$, $\phi (\alpha \tensor v) = \alpha \tensor \phi (v)$.  Then $\Sigma (\phi (\alpha \tensor v)) = \Sigma (\alpha \tensor \phi (v)) = \overline{\alpha} \tensor \phi (v)$, whereas also $\phi (\Sigma (\alpha \tensor v)) = \phi (\overline{\alpha} \tensor v) = \overline{\alpha} \tensor \phi (v)$; i.e. $\phi$ extended to $H$ commutes with $\Sigma$.

Conversely, given a complex-linear operator $\phi$ on $H$ that commutes with $\Sigma$, and given $v \in V$, $\Sigma (\phi (1 \tensor v)) = \phi (\Sigma (1 \tensor v)) = \phi (1 \tensor v)$, so $\phi$ maps the fixed point set of $\Sigma$, which we identify with $V$, into itself.  Thus we may restrict $\phi$ to $V$.
\end{proof}

\index{polarization}
The Fock space construction starts with a complex Hilbert space that can be thought of as ``half'' of $H$, half of the complexification of $V$, which is also called a ``polarization'' of $H$.  We will be more precise shortly.  If, contrary to our assumption, the real dimension of $V$ were a finite even number $n$, half of $H$ would have complex dimension $\frac{n}{2}$.  We will show two ways of constructing this half, related by proposition \ref{p-vj-hp-isom}.

\begin{defn}\label{d-unit-stru}
\index{unitary structure}
\index{complex structure}
\index{J@$J$}
\index{VJ@$V_J$}
\index{US(V)@$\US(V)$}
\index{O(V)@$\Orth(V)$}
\index{<,>J@$\langle , \rangle_J$}
(A Unitary Structure $J$ and the Corresponding $V_J$).
\citep[pages~55--56]{PR94} A unitary structure, or complex structure, on $V$ is an orthogonal transformation $J \in \Orth(V)$ such that $J^2 = -1$.  Let $\US(V)$ be the space of unitary structures on $V$, with the subspace topology from $\Orth(V)$, which has the operator norm topology.

Denote by $V_J$ the complex Hilbert space equal to $V$ as a set and with the same additive and real multiplicative structure, with complex multiplication induced by $v \in V \Rightarrow i v = J(v)$, and with the Hermitian inner product given by $x, y \in V \Rightarrow \langle x, y \rangle_J = (x, y) + i (x, J(y))$.
\end{defn}
$V_J$ is one form of the ``half'' of $H$ that we will need to construct Fock space.  As a set, $V_J = V$.  Proposition \ref{p-vj-hp-isom} gives an isometric complex isomorphism placing $V_J$ in the set $H$, but differently from the way $V$ is placed in $H$.

\begin{defn}\label{d-uvj}
\index{U(VJ)@$\UU(V_J)$}
Define $\UU(V_J) \subset \Orth(V)$ as the set of elements of $\Orth(V)$ that commute with $J$.
\end{defn}
$\UU(V_J)$ is the unitary group of the complex Hilbert space $V_J$, since as a set, $V_J = V$.

\begin{prop}\label{p-uv-ov}
\index{homogeneous structure}
\index{US(V)@$\US(V)$!homogeneous structure}
\index{phi(O/U(VJ),US)@$\phi_{\Orth/\UU(V_J), \US}$}
\index{pi(O/U(VJ))@$\pi_{\Orth/\UU(V_J)}$}
\index{O(V)@$\Orth(V)$}
($\US(V)$ has Homogeneous Structures).
\citep[pages~56--57]{PR94} Given a unitary structure $J$ on $V$, the topological group $\Orth(V)$ acts transitively on $\US(V)$ by conjugation, and the stabilizer of $J$ under this action is the closed subgroup $\UU(V_J) \subset \Orth(V)$.  Thus there is a continuous equivariant bijection from the homogenous space whose projection is
\begin{align}
\pi_{\Orth/\UU(V_J)} \colon \Orth(V) &\rightarrow \Orth(V) / \UU(V_J), \text{ namely} \notag \\
\phi_{\Orth/\UU(V_J), \US} \colon \Orth(V) / \UU(V_J) &\rightarrow \US(V) \notag \\
g \UU(V_J) &\mapsto K = g J g^{-1} \notag
\end{align}
\citep[page~3]{tomD87}; and $\phi_{\Orth/\UU(V_J), \US} ([\ident_{\Orth(V)}]) = J$.

Suppose that $K \in \US(V)$ is another unitary structure, with $K = k J k^{-1}$ for some $k \in \Orth(V)$.  Then $\UU(V_K) = k \UU(V_J) k^{-1}$.  Thus in general $\Orth(V) / \UU(V_J) \ne \Orth(V) / \UU(V_K)$ as a set, though there is an equivariant homeomorphism \citep[page~5]{tomD87} $\phi_{\Orth/\UU(V_K), \US}^{-1} \circ \phi_{\Orth/\UU(V_J), \US}$ between them that maps $g \UU(V_J) \mapsto g k^{-1} \UU(V_K)$, with inverse $h \UU(V_K) \mapsto h k \UU(V_J)$.
\end{prop}
It may be worth repeating that $\Orth(V)$ acts on $\US(V)$ by conjugation. Given two unitary structures $J$, $K$, and an isometric complex linear isomorphism $g \colon V_J \rightarrow V_K$, the complex linearity implies $K = g J g^{-1}$ and isometry implies $g \in \Orth(V)$.  The equivalence class $g \UU(V_J)$ in the quotient corresponds to the one unitary structure $K = g J g^{-1}$.  Any other $g'$ such that $K = g' J {g'}^{-1}$ is related to $g$ by $g' = g u$ for some $u \in \UU(V_J)$.

\begin{prop}\label{p-uv-ov-cano-g}
\index{US(V)@$\US(V)$!homogeneous space!continuously chosen g}
(For $K$ Near $J$ there is a Canonical $g$ with $K = g J g^{-1}$).
\citep[pages~101--102]{PR94} Given unitary structures $J$ and $K$ on $V$, if $\norm{K - J} < 2$ then there is a canonical choice of $g \in \Orth(V)$ such that $K = g J g^{-1}$ and for $K = J$, $g = \ident$.  This chosen $g$ depends continuously in the operator norm topology on $\Orth(V)$, on $J$ and $K$ in the product topology of subspace of operator norm topologies on $\US(V) \cross \US(V)$.
\end{prop}
\begin{proof}
The reference proves the existence of a canonical $g$ by obtaining it as the unitary part of the polar decomposition $\ident - K J = g \abs{\ident - K J}$, where $\ident - K J$ is invertible since $\norm{K - J} < 2$.  But $\ident - K J$ is a continuous function of $J$ and $K$, inversion in a Banach algebra is continuous \citep[page~224]{LS68}, and $\abs{}$ is a continuous function of its argument \citep[page~197]{RS80}.
\end{proof}

\begin{lem}\label{l-ext-ov}
\index{O(V)@$\Orth(V)$}
\index{Sigma@$\Sigma$}
\index{J@$J$!extends to H@extends to $H$}
\index{J@$J$!commutes with Sigma@commutes with $\Sigma$}
\index{J@$J$!in U(H)@$\in \UU(H)$}
($\Orth(V)$, $\UU(H)$, $J$, and $\Sigma$; $J$ is in $\UU(H)$).
\citep[pages~56--58,~102,~109]{PR94} Given $V$'s complexification $H$ with real structure $\Sigma$, elements of the group $\Orth(V)$ extended by complex linearity to act on $H$ lie in $\UU(H)$ and commute with $\Sigma$; and conversely, all elements of $\UU(H)$ commuting with $\Sigma$ arise from $\Orth(V)$ by complexification.

Given a unitary structure $J$ on $V$, $J$ extends uniquely to a complex-linear automorphism on $H$, also denoted by $J$.  This has the properties $J^2 = - 1$, $J \Sigma = \Sigma J$, and $J \in \UU(H)$.
\end{lem}
\begin{proof}
Aside from $\UU(H)$, see \ref{l-oph-opv}.  To see that $J \in \UU(H)$, recall $J \in \Orth(V)$, and then for $x, y \in V$ and $\mu, \nu \in \CC$,
\begin{align}
\langle J(\mu \tensor x), J(\nu \tensor y) \rangle &= \langle \mu \tensor J(x), \nu \tensor J(y) \rangle \notag \\
&= \mu \overline{\nu} (J(x), J(y)) \notag \\
&= \mu \overline{\nu} (x, y) = \langle \mu \tensor x, \nu \tensor y \rangle. \notag
\end{align}
\end{proof}
\begin{note}\label{n-ext-ov-uvj}
\index{U(VJ)@$\UU(V_J)$!interpretation}
(An Interpretation of $\UU(V_J)$).
We may interpret $\UU(V_J)$ as the subset of the complexified $\Orth(V)$, whose elements commute with $J$; or as the subset of $\UU(H)$, whose elements commute with $\Sigma$ and with $J$.
\end{note}

The idea of Lagrangian subspace gives our other way of defining ``half'' of $H$.
\begin{defn}\label{d-lagr-subs}
\index{L@$L$}
\index{Lagrangian subspace}
(Lagrangian Subspaces).
Given a complex Hilbert space $H$ with real structure $\Sigma$, a Lagrangian subspace $L \subset H$ is a (complex) subspace of $H$ for which $L^{\perp} = \Sigma (L)$.
\end{defn}

\begin{defn}\label{d-lagh}
\index{Lagrangian subspace!set of}
\index{Lagr(H)@$\Lagr(H, \Sigma)$}
(The Set of Lagrangian Subspaces $\Lagr(H, \Sigma)$).
Given a complex Hilbert space $H$ with real structure $\Sigma$, let $\Lagr(H, \Sigma)$ be the set of Lagrangian subspaces of $H$.
\end{defn}

\begin{lem}\label{l-lagr-subs}
\index{L@$L$}
\index{Lagrangian subspace!orthogonal decomposition of H@orthogonal decomposition of $H$}
(The Lagrangian Subspace Orthogonal Decomposition of $H$).
If $L$ is a Lagrangian subspace of a complex Hilbert space $H$ with real structure $\Sigma$, $L = \Sigma (L^{\perp})$ is closed, and $H = L \dirsum \Sigma (L)$.
\end{lem}

\begin{defn}\label{d-eige-spac}
\index{HTpmi@$H_{T,\pm i}$}
(Notation for Eigenspaces).
% removed ~ to cure overfull box
\citet[page 58]{PR94} Given an operator $T$ on a complex Hilbert space $H$, denote the eigenspace of $T$ in $H$ with eigenvalue $\lambda$, by $H_{T, \lambda}$.  In particular, if $T^2 = -1$, denote the $+i$ eigenspace by $H_{T,+i}$ and the $-i$ eigenspace by $H_{T,-i}$.
\end{defn}

\begin{lem}\label{l-h-split}
\index{HJpmi@$H_{J,\pm i}$}
\index{unitary structure!orthogonal decomposition of H@orthogonal decomposition of $H$}
(The Unitary Structure Orthogonal Decomposition of $H$).
\citet[page~58]{PR94} Given $V$'s complexification $H$ with real structure $\Sigma$, and a unitary structure $J$ on $V$, complex linearly extended to $H$, $H = H_{J,+i} \dirsum H_{J,-i}$, an orthogonal direct sum of Lagrangian subspaces.
\end{lem}
\begin{proof}
If $w \in H$, $w = \frac{1}{2} (w - i J w) + \frac{1}{2} (w + i J w)$, where $\frac{1}{2} (w \mp i J w) \in H_{J,\pm i}$. The subspaces $H_{J,\pm i}$ are closed because $J$ is bounded, and $\Sigma (H_{J,\pm i}) = H_{J,\mp i}$.  Letting $w_1, w_2 \in H_{J,+i},$ since by lemma \ref{l-ext-ov} $J \in \UU(H)$, $\langle w_1, \Sigma (w_2) \rangle = \langle J(w_1), J(\Sigma (w_2)) \rangle = \langle i w_1, -i \Sigma (w_2) \rangle =  i \langle w_1, \Sigma (w_2) \rangle = 0$, and so $H_{J,+i} \perp H_{J,-i}$.
\end{proof}

Definition \ref{d-eige-spac} and lemma \ref{l-h-split} give from a unitary structure $J$ on $V$, a Lagrangian subspace of $H$.
\begin{defn}\label{d-lj}
\index{LJ@$L_J$}
(The Lagrangian Subspace Corresponding to a Unitary Structure).
Given a unitary structure $J$ on $V$, let $L_J = H_{J,+i} \subset H$ denote the Lagrangian subspace corresponding to $J$; its $+i$ eigenspace.
\end{defn}
Many times, we will refer just to a unitary structure $J$ and its corresponding $+i$ eigenspace $L$, without the subscript.

Conversely, given a Lagrangian subspace $L$ of $H$, the space $V$ can be recovered as $V = \{ w + \Sigma(w) \st w \in L \}$; and $J$ can be recovered also.
\begin{lem}\label{l-lagr-give-unit-stru}
\index{Lagrangian subspace!unitary structure}
(The Unitary Structure Corresponding to a Lagrangian Subspace).
\citet[page~59]{PR94} Given $V$'s complexification $H$ with real structure $\Sigma$, a Lagrangian subspace $L \subset H$ defines a unitary structure $J$ on $V$, as follows.  Define the complex-linear automorphism $J$ on $H$ as multiplication by $i$ on $L$ and by $-i$ on $\Sigma (L)$.  Then $J$ restricts to a unitary structure on $V$. Also, $H_{J,+i} = L$ and $H_{J,-i} = \Sigma (L)$.
\end{lem}
\begin{proof}
$J^2 = -1$.  Since $H_{J,\pm i}$ are conjugate, $J$ commutes with $\Sigma$ and so by lemma \ref{l-oph-opv} restricts to an automorphism of $V$, also called $J$. Since $H_{J,\pm i}$ are orthogonal, by lemma \ref{l-ext-ov} $J$ on $H$ is in $\UU(H)$; hence the restriction of $J$ is in $\Orth(V)$.
\end{proof}

\begin{prop}\label{p-bij-usv-lagh}
\index{Lagrangian subspace!unitary structure}
(The Bijection between Unitary Structures and Lagrangian Subspaces).
Given $V$'s complexification $H$ with real structure $\Sigma$,
there is a naturally given bijection $\US(V) \leftrightarrow \Lagr(H, \Sigma)$, a unitary structure $J$ corresponding to its $+i$ eigenspace $L_J$.
\end{prop}

\begin{lem}\label{l-ov-acts-lagr}
\index{Lagrangian subspace!action of Orth(V) on@action of $\Orth(V)$ on}
(The Action of $\Orth(V)$ on $\Lagr(H, \Sigma)$).
Given $V$'s complexification $H$ with real structure $\Sigma$, $g \in \Orth(V)$, and two unitary structures $J$, $K$ on $V$ with corresponding Lagrangian subspaces $L_J, L_K \in \Lagr(H, \Sigma)$, identifying $g$ with its complex-linear extension to $H$, $L_K = g L_J \Leftrightarrow K = g J g^{-1}$.
\end{lem}
\begin{proof}
This can be seen using the fact that the Lagrangian subspaces are the $+i$ eigenspaces of the corresponding unitary structures.
\end{proof}

The following proposition connects the two ways we have given of finding ''half'' of $H$.  Some results we will use from \citet{PR94} are in terms of $V_J$, but we most often will think in terms of $L_J$; thus it is necessary to connect the two ways of thinking.
\begin{prop}\label{p-vj-hp-isom}
\index{VJ@$V_J$!isomorphic to LJ@isomorphic to $L_J$}
\index{LJ@$L_J$!isomorphic to VJ@isomorphic to $V_J$}
($V_J$ is Isomorphic to $L_J$).
\citet[page~60]{PR94} Given $V$'s complexification $H$ with real structure $\Sigma$, and a unitary structure $J$ on $V$, there is a naturally given isomorphism $\phi$ of the complex Hilbert spaces $V_J$ and $L_J$, where $L_J$, the $+i$ eigenspace of $J$, obtains its Hilbert space structure as a closed subspace of $H$:
\begin{align}
\phi \colon V_J &\rightarrow L_J \notag \\
\phi \colon v &\mapsto \frac{1}{\sqrt{2}} (v - i J v) \notag \\
\phi^{-1} \colon w &\mapsto \frac{1}{\sqrt{2}} (w + \Sigma(w)). \notag
\end{align}
\end{prop}

\begin{prop}\label{p-vj-hp-isom-op}
\index{VJ@$V_J$!isomorphic to LJ@isomorphic to $L_J$}
\index{LJ@$L_J$!isomorphic to VJ@isomorphic to $V_J$}
(The Relation of Operators on $V_J$ to Operators on $L_J$).
Given $Z_H \in \UU(H)$ that commutes with $\Sigma$ and commutes with $J$, the corresponding $V_J$ complex-linear map $Z_V = \phi^{-1} Z_H \circ \phi \in \Orth(V)$ is also given by $Z_V = {Z_H}_{|V}$.

Given $Z_H \in \UU(H)$ that commutes with $\Sigma$ but anticommutes with $J$, the corresponding $V_J$ complex-antilinear map $Z_V = \phi^{-1} \circ \Sigma \circ Z_H \circ \phi \in \Orth(V)$ is also given by $Z_V = {Z_H}_{|V}$.

Conversely, given $Z_V \in \UU(V_J) \subset \Orth(V)$, i.e. that commutes with $J$, $\phi \circ Z_V \circ \phi^{-1}$ equals the restriction to $L_J$ of the complex-linear extension of $Z_V$ to an element $Z_H \in \UU(H)$ that commutes with $\Sigma$ and $J$.

Again conversely, given $Z_V \in \Orth(V)$ that anticommutes with $J$, $\Sigma \circ \phi \circ Z_V \circ \phi^{-1}$ equals the restriction to $L_J$ of the complex-linear extension of $Z_V$ to an element $Z_H \in \UU(H)$ that commutes with $\Sigma$ and anticommutes with $J$.
\end{prop}
\begin{proof}
If $Z_H$ commutes with $J$, then $Z_H \colon L_J \rightarrow L_J$.  Calculations show that $Z_H \circ \phi = \phi \circ Z_H$.  If $Z_H$ anticommutes with $J$, then $Z_H \colon L_J \rightarrow \Sigma (L_J)$, and then that $\Sigma \circ Z_H \circ \phi = \phi \circ Z_H$.  Calculations using commutation properties and values of $J$ on $L_J$ and $\overline{L_J}$, also show the converses.
\end{proof}

\begin{lem}\label{l-hp-proj}
\index{LJ@$L_J$!projection}
(Orthogonal Projection onto $L_J$).
\citet[pages~59--60]{PR94} Given $V$'s complexification $H$ with real structure $\Sigma$, and a unitary structure $J$ on $V$, another way to specify $L_J$ is by the orthogonal projection operator on $H$ whose image it is:
\[
P_J = \frac{1}{2} (\ident - i J) \Rightarrow
P_J^2 = \ident \text{, } P^{*} = P \text{, } P_J + P_{-J} = P_J + \Sigma P_J \Sigma = \ident  \text{, and } P_J P_{-J} = 0 \notag
\]
Using $L = L_J$, we also call these $P_L = P_J$ and $P_{\overline{L}} = P_{-J}$.
\end{lem}

\section{Fock Representations and Their Equivalence}\label{s-fock-rep-equi}

Given a Lagrangian subspace $L \subset H$, we will represent elements of the $C^{*}$-Clifford algebra $\Cl(V)$ as bounded operators on the Fock space $\F(L)$, the exterior algebra of $L$ completed as a Hilbert space.
\begin{defn}\label{d-fock-spac}
\index{L@$L$}
\index{F(L)@$\F(L)$}
\index{CL(V)@$\Cl(V)$}
\index{Fock space}
(Fock Spaces).
\citep[pages~61--62]{PR94}
\begin{enumerate}
    \item Define the inner product on $\Lambda^0 L = \CC$ by $z_1, z_2 \in \CC \Rightarrow \langle z_1, z_2 \rangle = z_1 \overline{z_2}$.
    \item For $k > 0$ define the inner product on $\Lambda^k L$ on decomposables by \\
        $x_1, y_1, \dots, x_k, y_k \in L \Rightarrow \langle x_1 \wedge \cdots \wedge x_k, y_1 \wedge \cdots \wedge y_k \rangle = \Det([\langle x_i, y_j \rangle]_{1 \le i, j \le k})$. \label{it-gram-det}
    \item Let the Fock space $\F(L) = \dirsum_{k=0}^{\infty} \Lambda^k L$, where each summand is the Hilbert space completion of the algebraic wedge product, the inner product of elements in $\Lambda^{k_1}$ and $\Lambda^{k_2}$ is defined as $0$ when $k_1 \neq k_2$, and the whole sum is the Hilbert space completion of the algebraic direct sum.
\end{enumerate}
\end{defn}
\begin{note}\label{n-gram-det}
\index{Grammian determinant}
(Grammian determinants).
The determinant in item \ref{it-gram-det} is called a Grammian determinant, and can be thought of as the determinant of the product of the $k \cross \infty$ matrix whose rows are the vectors $x_i$, and the $\infty \cross k$ matrix whose columns are the vectors $y_j$.
\end{note}

\begin{prop}\label{p-fock-rep}
\index{Fock representation}
(Fock Representations).
% ~ removed from citation to cure overfull box
\citep[pages 75--76,~157]{PR94} Given a Lagrangian subspace $L \subset H$, we define as follows a self-adjoint Clifford map $\pi_L \colon V \rightarrow \B(\F(L))$.

For $v \in V \subset H$, write $v = l + \overline{l} = P_L (v) + P_{\overline{L}} (v)$, recalling lemma \ref{l-hp-proj}.

Let $l \in L$ operate on $\Lambda^{*} L$ by $\sqrt{2}$ times wedging with $l$ on the left, to get the creator $c \colon V \rightarrow \B(\F(L))$, $c (v) (w) = \sqrt{2} \thinspace P_L(v) \wedge w$, for $w \in \F(L)$.

Let $l$ operate on $\Lambda^{*} L$ by $\sqrt{2}$ times the contraction with $l$ on the left, to get the annihilator.  Define contraction, for a decomposable $l_1 \wedge \cdots l_k \in \Lambda^k L$, as $0$ for $k = 0$, and otherwise by:
\[
l \contract l_1 \wedge \cdots l_k = \sum_{r=1}^k (-1)^{r - 1} \langle l_r , l \rangle l_1 \wedge \cdots \wedge \widehat{l_r} \wedge \cdots \wedge l_k
\]
where $\widehat{l_r}$ indicates omission of that factor.  This gives rise to the annihilator $a \colon V \rightarrow \B(\F(L))$, $a (v) (w) = \sqrt{2} \thinspace P_L(v) \contract w$, for $w \in \F(L)$.

Let $V$ act on $\Lambda^{*} L$ as the sum of the creator and the annihilator.  The resulting self-adjoint Clifford map $\pi_L = c + a \colon V \rightarrow \B(\F(L))$ extends by proposition \ref{p-c-star-clif-alg-univ} to an isometric representation of $\Cl(V)$; i.e., $v \in V \Rightarrow \norm{\pi_L (v)} = \norm{v}$).

This is called a Fock representation, the Fock representation defined by $L$.  It is irreducible, viewed as a $\Cl(V)$ module; $\F(L)$ has no nontrivial submodules.  Equivalently the bounded operators on $\F(L)$ that commute with the image of $\pi_L$ are precisely the complex scalars.  As a consequence of irreducibility, any nonzero $w \in \F(L)$ is cyclic; that is, $\pi_L (\Cl(V)) w = \F(L)$.

$\F(L)$ is a spinor module for $\Cl(V)$.  It is $\ZZ_2$ graded as $\F(L) = \F(L)_0 \dirsum \F(L)_1$, where $\F(L)_0$, $\F(L)_1$ are the closures in $\F(L)$ of $\dirsum_{k=0}^{\infty} \Lambda^{2 k} L$, $\dirsum_{k=0}^{\infty} \Lambda^{2 k + 1} L$, respectively.  For $v$ in $\Lambda^k L$, by which is meant the Hilbert space completion as in definition \ref{d-fock-spac}, and $w$ in $\Lambda^l L$, $v \wedge w \in \Lambda^{k + l} L$ as in \citet[page~62]{PR94}, and so
\[
\partial (v \wedge w) = \partial v + \partial w. \notag
\]
\end{prop}
\begin{proof}
Sketch of proof for the self-adjoint Clifford map.  Contraction twice with the same $\overline{l}$ gives $0$, which can be seen by direct calculation.  Since the creator and annihilator are adjoints of each other, proved the same way as in \citet[page~69]{PR94}, their sum is self-adjoint.  Note that $\norm{v} = \sqrt{2} \norm{l}$.  To see that the sum is a Clifford map, calculate that operating twice with $v$ on $\Lambda^{*} L$ equals applying $2 \langle l, l \rangle \ident_{\Lambda^{*} L} = \norm{v}^2 \ident_{\Lambda^{*} L}$.
\end{proof}
Some of these things can be seen differently, as in \citet[page~232]{PS86}, mutatis mutandis.  For the statements not proved, see \citet[pages~61--62,~68--69,~75--76,~84--85]{PR94}, adapted to our definitions as will follow using the concept of equivalence of Fock representations.

The question of when two Fock representations are equivalent is vital to this thesis.  We can choose continuously a Clifford algebra and a polarization class of Lagrangian subspaces over each point of $LM$, but are unable to choose continuously a Lagrangian subspace from that polarization class, to construct a Fock representation.  It is the structure of equivalences that will allow us nonetheless to construct some continuously chosen (not Fock) representation.
\begin{defn}\label{d-equi-rep}
\index{Fock representation!equivalence}
\index{intertwiner}
\index{intertwiner!set of}
\index{T@$T$}
(Equivalence of Fock Representations).
\citep[page~91]{PR94} Given two Lagrangian subspaces $L$ and $K$ of $H$, the two Fock representations $\pi_L$ and $\pi_K$ are called unitarily equivalent when there is a unitary isomorphism $T \colon \F(L) \isomto \F(K)$ such that the following diagram commutes for every $a \in \Cl(V)$.  The isometric isomorphism $T$ is said to intertwine $\pi_L$ and $\pi_K$, and is called a unitary intertwiner or just intertwiner.  We may consider $\F(L)$ and $\F(K)$ as $\Cl(V)$ modules, and from this viewpoint $T$ is a $\Cl(V)$ linear, we will say Clifford linear, unitary isomorphism.  Denote by $T(L, K)$ the set of all such $T$.
\[
\begindc{\commdiag}[5]
\obj(10,25)[objFL]{$\F(L)$}
\obj(25,25)[objFK]{$\F(K)$}
\obj(10,10)[objFL2]{$\F(L)$}
\obj(25,10)[objFK2]{$\F(K)$}
\mor{objFL}{objFK}{$T$}
\mor{objFL2}{objFK2}{$T$}
\mor{objFL}{objFL2}{$\pi_L(a)$}
\mor{objFK}{objFK2}{$\pi_K(a)$}
\enddc
\]

$\pi_L$ and $\pi_K$ are equivalent as $\ZZ_2$ graded representations if in addition, $T$ preserves the grading:
\begin{align}
\forall v \in \F(L)_0 \cup \F(L)_1, \text{ } \partial T(v) &= \partial v. \notag
\end{align}
\end{defn}
\begin{note}\label{n-equi-rep-z2-grad}
\index{Fock representation!Z2 graded@$\ZZ_2$ graded!equivalence}
(Equivalence of Representations that are $\ZZ_2$ Graded).
From \citet[pages~116--117]{PR94}, which might make more sense after reading the rest of this chapter, equivalence in the ordinary sense implies that either $T$ gives an equivalence as $\ZZ_2$ graded representations by this definition, or that $T$ reverses the grading; i.e. $\forall v \in \F(L)_0 \cup \F(L)_1, \text{ } \partial T(v) = 1 + \partial v$.
\end{note}

We apply the concept of equivalence of Fock representations to connect our definition with that of our primary reference for this purpose.
\begin{note}\label{n-fock-spac-vj-l-trans}
\index{Fock space!VJ-L translation@$V_J$-$L$ translation}
\index{VJ@$V_J$!L translation@$L$ translation}
\index{L@$L$!VJ translation@$V_J$ translation}
(The Fock Space $V_J$-$L$ translation).
\citet[pages~61--62,~68--69,~75--76]{PR94} define an action of $\Cl(V)$ on a Hilbert space we might call $\F(V_J)$, defined just like $\F(L)$ except using $V_J$, equal to $V$ as a set, in place of $L$.  Their representation, call it $\pi_J$, and ours, $\pi_L$, are equivalent (see definition \ref{d-equi-rep}), using the isomorphism $\phi \colon V_J \isomto L$ of proposition \ref{p-vj-hp-isom}, as follows.

They take $v \in V$, and using the equality of sets $V_J = V$, let $v$ act on elements of $\Lambda^{*} V_J$.  We express $v = l + \overline{l}$, $l \in L$ and $\norm{l} = \frac{1}{\sqrt{2}} \norm{v}$, and let $l$ act on elements of $\Lambda^{*} L$ in the same way they let $v$ act on elements of $\Lambda^{*} V_J$, except we multiply the result by $\sqrt(2)$.  Now, $\phi(v) = \phi(l + \overline{l}) = \sqrt(2) l$, so what we do is to let $\phi(v)$ act on elements of $\Lambda^{*} L$ in exactly the same way they let $v$ act on elements of $\Lambda^{*} V_J$.

There's another Clifford algebra we haven't mentioned, that will help show that $\phi$ is the intertwiner we say it is.  The isometric isomorphism $\phi$ induces further isometric isomorphisms we will give the same name, $\phi \colon V \rightarrow L$ of real Hilbert spaces, ignoring the complex structures, with inner product $\Re (\langle , \rangle_J)$ for $V_J$ corresponding to $\Re (\langle , \rangle)$ for $L$; and $\phi \colon \F(V_J) \rightarrow \F(L)$.  Thus we can use $\phi$ to move the whole apparatus, real Hilbert space, Clifford algebra made from it, complex Hilbert space, and Fock space made from that, from $V$ to $L$.

Let $\pi_{L_{\RR}} \colon L \rightarrow \B(\F(L))$ be the map $\phi$ moves over to $L$ from $V$, starting from the self-adjoint Clifford map $\pi_{V_J} \colon V \rightarrow \B(\F(V_J))$.  For $l \in L$, $w \in \F(L)$, $\pi_{L_{\RR}} (l) (w) = \phi (\pi_{V_J} (\phi^{-1} (l)) (\phi^{-1} (w)))$; $\pi_{L_{\RR}}$ is a self-adjoint Clifford map because $\pi_{V_J}$ is.  It gives a Fock representation we call $\pi_{L_{\RR}}$ of $\Cl(L)$ on $\F(L)$ defined just as is the Fock representation of $\Cl(V)$ on $\F(V_J)$.

Our self-adjoint Clifford map is the composition $\pi_L \colon V \xrightarrow{\phi} L \xrightarrow{\pi_{L_{\RR}}} \B(\F(L))$; it's a Clifford map because $\phi$ induces a isometry of real Hilbert spaces as we noted.  Thus $\pi_L (v) (w) = \pi_{L_{\RR}} (\phi(v)) (w) = \phi (\pi_{V_J} (v) (\phi^{-1} (w)))$, and $\phi$ is a naturally given unitary equivalence between the representations $\pi_L$ and $\pi_{V_J}$.
\end{note}

We will use repeatedly, that if two Fock representations are equivalent, then any two intertwiners (in one direction) differ only by a $\UU(1)$ factor.  Moreover:
\begin{prop}\label{p-set-intw-u1-tors}
\index{intertwiner!set of!U(1) torsor@$\UU(1)$ torsor}
\index{T@$T$!U(1) torsor@$\UU(1)$ torsor}
(The Set of Intertwiners is a $\UU(1)$ torsor).
Let $T(L, K)$ be the set of unitary intertwiners as in definition \ref{d-equi-rep}, between two equivalent Fock representations corresponding to Lagrangian subspaces $L$, $K$.  Then $T(L, K)$ is a $\UU(1)$ torsor.  The topology $T(L, K)$ has as a $\UU(1)$ torsor, is the same as it has as a topological subspace of $\UU(\F(L), \F(K))$ with the operator norm topology or the strong operator topology.

The $\UU(1)$ torsors $T(L, L)$ and $T(K, K)$ can be identified canonically with $\UU(1)$, using $\ident \mapsto 1$.
\end{prop}
\begin{proof}
\index{Schur's Lemma!Fock representations}
If $\phi_a$, $\phi_b$ $\in T(L, K)$, then $\phi_b^{-1} \circ \phi_a \colon \F(L) \rightarrow \F(L)$ commutes with the representation $\pi_L$ and thus is given by multiplication by some $z \in \UU(1)$, by Schur's Lemma (\citet[page~68]{Fabe00} Proposition II.11(b)), since $\F(L)$ is an irreducible $\Cl(V)$ module.  Multiplication by $z$ gives a free, and because of Schur's Lemma, transitive action of $\UU(1)$ on $T(L, K)$.  Give $T(L, K)$ by corollary \ref{co-g-tors-alg-top} the topology making it a $\UU(1)$ torsor, $\UU(1)$-equivariantly homeomorphic to $\UU(1) = S^1$ by some map $\psi \colon T(L, K) \rightarrow \UU(1)$.  Pick some $\phi_0 \in T(L, K)$ and define a $\UU(1)$-equivariant map $\theta \colon T(L, K) \rightarrow \UU(\F(L), \F(K))$ by $\theta(\phi) = \psi(\phi) (\psi(\phi_0))^{-1} \phi_0$, which is the inclusion as a map of sets.  The map $\theta$ is continuous since $\psi$ is continuous and multiplication by elements of $\UU(1)$ is continuous in $\UU(\F(L), \F(K))$.  Similarly and because composition of operators is continuous in the operator norm topology, $\theta^{-1} \colon \xi \mapsto \psi^{-1}(\xi \phi_0^{-1} \psi(\phi_0))$ is continuous, and thus $\theta$ is a homeomorphism onto its image.  Convergence in the strong operator topology implies convergence in the operator norm topology by evaluation at a fixed element of the domain, since all elements of $T(L, K)$ are $\UU(1)$ multiples of each other.

Considering just $L$, the elements of $T(L, L)$ are again by Schur's Lemma, simply multiplication by elements of $\UU(1)$.  The identity map is multiplication by $1$.  Thus the $\UU(1)$ torsor $T(L, L)$ can be identified canonically with $\UU(1)$.
\end{proof}
Keep in mind also lemma \ref{l-g-tors-home-g}, whereby any $\UU(1)$ torsor is equivariantly homeomorphic to $\UU(1)$.

The following term vacuum vector comes from quantum field theory in physics, and is closely tied to the Fock representations.
\begin{defn}\label{d-j-vac-vec}
\index{Fock space!vacuum vector}
\index{vacuum vector}
\index{Omega@$\Omega$}
($L$- or $J$-Vacuum Vectors for $\pi$).
\citep[pages~79--80]{PR94} Given Lagrangian subspace $L$ of $H$ with Fock representation $\pi_L$ on Fock space $\F(L)$, and in addition some ${}^{*}$-representation $\pi \colon \Cl(V) \rightarrow \B(F)$ for some Hilbert space $F$, then a nonzero vector $\Omega \in F$ is called a $L$-vacuum vector for $\pi$ when it satisfies the first following condition, called the $L$-vacuum condition.  If $L$ is the Lagrangian subspace corresponding to the unitary structure $J$, the second condition, called the $J$-vacuum condition, is equivalent by proposition \ref{p-vj-hp-isom} to the first, and $\Omega$ can be called a $J$-vacuum vector.
\begin{align}
x \in L^{\perp} &\Rightarrow \pi(x) \Omega = 0 \text{, or equivalently} \notag \\
v \in V &\Rightarrow \pi(v + i J v) \Omega = 0.  \notag
\end{align}
\end{defn}
A special case is of the Fock representation for another Lagrangian subspace $K$ of $H$, where $\pi = \pi_K$ is a ${}^{*}$-representation on $F = \F(K)$.

\begin{prop}\label{p-j-vac-vec-unit-equi}
\index{Fock representation!equivalence!vacuum vector}
\index{intertwiner!vacuum vector}
\index{vacuum vector!equivalent representations}
(The Existence of a Cyclic $L$-Vacuum Vector Implies Equivalence).
\citep[pages~79--80]{PR94} Note that ``equivalence'' here is between a Fock representation and some other, not necessarily Fock, ${}^{*}$-representation of $\Cl(V)$ on some Hilbert space.

Given a Lagrangian subspace $L$ of $H$ with corresponding Fock representation $\pi_L$ on Fock space $\F(L)$, and in addition some ${}^{*}$-representation $\pi \colon \Cl(V) \rightarrow \B(F)$ for some Hilbert space $F$, if there exists a unit $L$-vacuum vector $\Omega \in F$ for $\pi$, then $\pi_L$ and $\pi$ are unitarily equivalent, and there is a unique unitary intertwiner $T \colon \F(L) \rightarrow F$ such that $T(\Omega_L) = \Omega$, where $\Omega_L$ (a.k.a. $1$) is the standard unit or Fock vacuum vector for $\F(L)$.

Furthermore, if $F$ is fixed, but $\Omega = \Omega_w$ and $\pi = \pi_w \in \B(\Cl(V), \B(F))$ depend continuously on a parameter $w$ in some first countable (so that continuity can be defined by sequential convergence) topological space $W$, the dependence of $T = T_w \in T(\pi_L, \pi) \subset \UU(\F(L), F)$ on $w$ is also continuous, where we let $T(\pi_L, \pi)$ denote the set of unitary intertwiners between the two representations, and we use the strong operator topologies on both $\B(\Cl(V), \B(F))$ and $\UU(\F(L), F)$.

That is, taking $w_i \rightarrow w \in W$, if in addition to $\Omega_{w_i} \rightarrow \Omega_w$ we have that for every $a \in \Cl(V)$, $w_i \rightarrow w \Rightarrow \pi_{w_i} (a) \rightarrow \pi_w (a)$ in the operator norm topology for $\B(F)$, then for every $x \in \F(L)$, $T_{w_i}(x) \rightarrow T_w(x)$.
\end{prop}
\begin{proof}
Only the statements about continuity are not proved in the reference. The ``that is'' consists of two applications of definition \ref{d-stro-op-top}.  To show that the statement is true, start with the definition of $T_w$ by
\[
T_w(\pi_L(a) \Omega_L) = \pi_w(a) \Omega_w \notag
\]
for every fixed $a \in \Cl(V)$.  Replacing $x$ with $\pi_L(a) \Omega_L$, we want the following to go to $0$:
\begin{align}
\norm {\pi_w (a) \Omega_w - \pi_{w_i} (a) \Omega_{w_i}} &\le \norm {\pi_w (a) \Omega_w - \pi_w (a) \Omega_{w_i}} + \norm {\pi_w (a) \Omega_{w_i} - \pi_{w_i} (a) \Omega_{w_i}} \notag \\
 &\le \norm {\pi_w (a) (\Omega_w - \Omega_{w_i})} + \norm {(\pi_w (a) - \pi_{w_i} (a)) \Omega_{w_i}} \notag \\
 &\le \norm {\pi_w (a)} \norm{(\Omega_w - \Omega_{w_i})} + \norm {(\pi_w (a) - \pi_{w_i} (a))} \norm{\Omega_{w_i}} \notag \\
 &\le \norm {a} \norm{(\Omega_w - \Omega_{w_i})} + \norm {(\pi_w (a) - \pi_{w_i} (a))} \norm{\Omega_{w_i}} \notag
\end{align}
where the second factor of the first term and the first factor of the second term go to $0$, and the second factor of the second term is bounded because $\Omega_{w_i}$ converges.
\end{proof}
\begin{note}\label{n-j-vac-vec-unit-equi}
\index{vacuum vector!unitary equivalence}
(Vacuum Vectors and Unitary Equivalence).
In our applications, the parameter space $W$ will have a metric and hence be first countable.

Note that if $\pi = \pi_L$, $\Omega = \Omega_L$, and $F = \F(L)$, then $T = \ident_{\F(L)}$.  Note the special case of another Lagrangian subspace $K$ of $H$ with Fock representation $\pi = \pi_K$ on $F = \F(K)$.  Here, $K = L$ implies that $\Omega$ can be chosen equal to $\Omega_L$, in which case $T = \ident_{\F(L)}$.  Also note the case of some $g \in \Orth(V)$, $\pi = \pi_L \circ \theta_g$, $F = \F(L)$; we use this to obtain unitary implementers, to be defined soon.  Here, $g = \ident$ implies that $\Omega$ can be chosen equal to $\Omega_L$, in which case $T = \ident_{\F(L)}$.

The continuity argument here would work just as well for the intertwiner defined for uniqueness of the general Gelfand-Naimark-Segal construction \citep[pages~39--41]{Take02}.
\end{note}

\section{Intertwiners and Bogoliubov Automorphism Implementers}\label{s-impl-bogo-auto}

\citet{PR94} solve the question of equivalence of Fock representations on different Fock spaces, by solving another question, that of ``implementation'', which amounts to finding an intertwiner between a Fock representation and another related representation on the same Fock space, that involves a Bogoliubov automorphism.  There are other possible paths to the result on equivalence, but this is the one we use.  We will use implementers not only to get the result on equivalence, but on occasion when proving some sort of continuity.
\begin{defn}\label{d-impl-rep}
\index{Bogoliubov automorphism!implementation}
\index{implementer}
(Bogoliubov Automorphism Implementers).
(See \citet[page~91]{PR94}).  Given a Lagrangian subspace $L$ of $H$, and some $g \in \Orth(V)$, the Bogoliubov automorphism $\theta_g$ of $\Cl(V)$ is said to be unitarily implemented in the Fock representation $\pi_L$ when there is a unitary isomorphism $U$ on $\F(L)$ such that the following diagram commutes for every $a \in \Cl(V)$.  The unitary isomorphism $U$ is said to implement $\theta_g$ in $\pi_L$, and is called a unitary implementer.  From the viewpoint of modules, $U$ is a $\Cl(V)$ linear unitary operator from $\F(L)$ with $\Cl(V)$ multiplication given by $\pi_L$, to $\F(L)$ with $\Cl(V)$ multiplication given by $\pi_L \circ \theta_g$.
\[
\begindc{\commdiag}[5]
\obj(10,25)[objFL]{$\F(L)$}
\obj(25,25)[objFL2]{$\F(L)$}
\obj(10,10)[objFL3]{$\F(L)$}
\obj(25,10)[objFL4]{$\F(L)$}
\mor{objFL}{objFL2}{$U$}
\mor{objFL3}{objFL4}{$U$}
\mor{objFL}{objFL3}{$\pi_L(a)$}
\mor{objFL2}{objFL4}{$\pi_L \circ \theta_g (a)$}
\enddc
\]
\end{defn}

When intertwiners and implementers are related, a third kind of map comes into play: $\Lambda$, given by functoriality of the exterior algebra construction.
\begin{prop}\label{p-bij-intw-impl}
\index{bijection!intertwiners and implementers}
\index{implementer!bijection with intertwiner}
\index{intertwiner!bijection with implementer}
\index{Lambda@$\Lambda$}
(The Bijection between the Set of Intertwiners and the Set of Implementers).
\citep[page~101]{PR94} Given a Lagrangian subspace $L$ of $H$, and some $g \in \Orth(V)$, $K = g L$ also is a Lagrangian subspace of $H$.  Letting $\Lambda_g \in \B(\F(L), \F(K))$ denote the unitary isomorphism induced by $g$ (see definition \ref{d-fock-spac} to see that it's an isometry), the equation
\[
T \circ U = \Lambda_g
\]
gives a bijection between the set of unitary intertwiners $T \colon \F(L) \rightarrow \F(K)$ and the set of unitary implementers $U \colon \F(L) \rightarrow \F(L)$ of $\theta_g$ in $\pi_L$ (which may be empty).  For fixed $g$, this bijection is in fact a $\UU(1)$-equivariant homeomorphism, an isomorphism of $\UU(1)$ torsors, with the operator norm topologies on $\UU(\F(L), \F(K))$ and $\UU(\F(L))$.  In the future, we may use the symbol $\Lambda_{g, L}$ to indicate $\F(L)$.
\end{prop}
\begin{proof}
In the following diagram, the outermost four edges form a commutative square by functoriality and the way the action of the Clifford algebra on the Fock space is defined \citep[page~99--100]{PR94}.  Given $U$ for which the left hand smaller square commutes, we can use it to define $T$ such that the right hand smaller square commutes; and vice-versa.
\[
\begindc{\commdiag}[5]
\obj(10,25)[objFL]{$\F(L)$}
\obj(30,25)[objFL2]{$\F(L)$}
\obj(10,10)[objFL3]{$\F(L)$}
\obj(30,10)[objFL4]{$\F(L)$}
\obj(55,25)[objFK]{$\F(K) = \F(g L)$}
\obj(55,10)[objFK2]{$\F(K) = \F(g L)$}
\mor{objFL}{objFL2}{$U$}
\mor{objFL3}{objFL4}{$U$}
\mor{objFL}{objFL3}{$\pi_L(a)$}
\mor{objFL2}{objFL4}{$\pi_L (\theta_g (a))$}
\mor{objFL2}{objFK}{$T$}
\mor{objFL4}{objFK2}{$T$}
\mor{objFK}{objFK2}{$\pi_K (\theta_g (a)) = \pi_{g L} (\theta_g(a))$}
\cmor((12,27)(30,33)(53,27)) \pright(30,35){$\Lambda_g$}
\cmor((12,8)(30,2)(53,8)) \pright(30,0){$\Lambda_g$}
\enddc
\]
\end{proof}
If $K = L$, which is to say $g \in \UU(V_J)$ where $J$ is the unitary structure corresponding to $L$, then $\pi_K = \pi_L$, we may take $T = \ident$, and then $U = \Lambda_g$.
\begin{cor}\label{co-uvj-cano-impl}
\index{implementer!canonical}
\index{U(VJ)@$\UU(V_J)$!canonical implementation}
(The Canonical Implementer for $g \in \UU(V_J)$).
\citep[page~102]{PR94} Given a Lagrangian subspace $L$ of $H$, with corresponding unitary structure $J$ on $V$, if $g \in \UU(V_J) \subset \Orth(V)$, then $g$ is canonically implemented in $\pi_L$ by the unitary automorphism $\Lambda_g$ of $\F(L)$.
\end{cor}

The following fact will be useful later.
\begin{lem}\label{l-lamb-conj-intw}
\index{intertwiner!conjugation with Lambda@conjugation with $\Lambda$}
(The Conjugation of an Intertwiner with $\Lambda$ is an Intertwiner).
Given Lagrangian subspaces $L$ and $K$ of $H$, some $g \in \Orth(V)$, and an intertwiner $T \in T(L, K)$, $\Lambda_{g, K} \circ T \circ \Lambda_{g, L}^{*} \in T(g L, g K)$.
\end{lem}
\begin{proof}
\[
\begindc{\commdiag}[5]
\obj(35,40)[objFL]{$\F(L)$}
\obj(65,40)[objgFL]{$\F(g L)$}
\obj(10,34)[objFL2]{$\F(L)$}
\obj(40,34)[objgFL2]{$\F(g L)$}
\mor{objFL}{objgFL}{$\Lambda_{g, L}$}
\mor{objFL2}{objgFL2}{$\Lambda_{g, L}$}[\atright,\solidarrow]
\mor{objFL}{objFL2}{$\pi_L(a)$}[\atright,\solidarrow]
\mor{objgFL}{objgFL2}{$\pi_{gL} \circ \theta_g (a)$}
\obj(35,16)[objFK]{$\F(K)$}
\obj(65,16)[objgFK]{$\F(g K)$}
\obj(10,10)[objFK2]{$\F(K)$}
\obj(40,10)[objgFK2]{$\F(g K)$}
\mor{objFK}{objgFK}{$\Lambda_{g, K}$}
\mor{objFK2}{objgFK2}{$\Lambda_{g, K}$}[\atright,\solidarrow]
\mor{objFK}{objFK2}{$\pi_K(a)$}[\atright,\solidarrow]
\mor{objgFK}{objgFK2}{$\pi_{gK} \circ \theta_g (a)$}
\mor{objFL}{objFK}{$T$}
\mor{objgFL}{objgFK}{$\Lambda_{g, K} \circ T \circ \Lambda_{g, L}^{*}$}
\mor{objFL2}{objFK2}{$T$}
\mor{objgFL2}{objgFK2}{$\Lambda_{g, K} \circ T \circ \Lambda_{g, L}^{*}$}
\enddc
\]
The left face commutes by definition \ref{d-equi-rep}, and the top and bottom faces commute by functoriality as in the diagram for the proof of proposition \ref{p-bij-intw-impl}.  The front and back faces commute by definition.  The commutativity of the right face then follows from commutativity of the other faces.  Since $\theta_g$ is an automorphism, if the composition with $\theta_g$ is deleted from both maps in the right face, the diagram still commutes for all $a$, and by definition \ref{d-equi-rep}, $\Lambda_{g, K} \circ T \circ \Lambda_{g, L}^{*} \in T(g L, g K)$.
\end{proof}

Implementers are maps on one Fock space.  The following shows how implementers on one Fock space relate to implementers on another.
\begin{lem}\label{l-impl-ch-base}
\index{implementer!change of base}
(Implementer Change of Base).
Given Lagrangian subspaces $L$ and $K$ of $H$, some $g \in \Orth(V)$, an implementer $U_{g,L}$ in $\pi_L$ of $\theta_g$, and an intertwiner $T_{LK} \in T(L, K)$, $U_{g,K} = T_{LK} \circ U_{g,L} \circ T_{LK}^{*}$ is an implementer in $\pi_K$ of $\theta_g$, the same for any $T_{LK} \in T(L, K)$.
\end{lem}
\begin{proof}
``The same'' follows from proposition \ref{p-set-intw-u1-tors}.  The top face of the following diagram commutes from definition \ref{d-impl-rep}, the left and right faces commute from definition \ref{d-equi-rep}, and the front and back faces commute by definition of $U_{g,K}$.  The bottom face commutes, i.e. $U_{g,K}$ is an implementer as claimed, by following the maps in the diagram or because
\begin{align}
\pi_K \circ \theta_g (a) &= T_{LK} \circ (\pi_L \circ \theta_g (a)) \circ T_{LK}^{*} \notag \\
&= T_{LK} \circ U_{g,L} \circ (\pi_L (a)) \circ U_{g,L}^{*} \circ T_{LK}^{*} \notag \\
&= T_{LK} \circ U_{g,L} \circ T_{LK}^{*} \circ (\pi_K (a)) \circ T_{LK} \circ U_{g,L}^{*} \circ T_{LK}^{*} \notag \\
&= U_{g,K} \circ (\pi_K (a)) \circ U_{g,K}^{*}. \notag
\end{align}
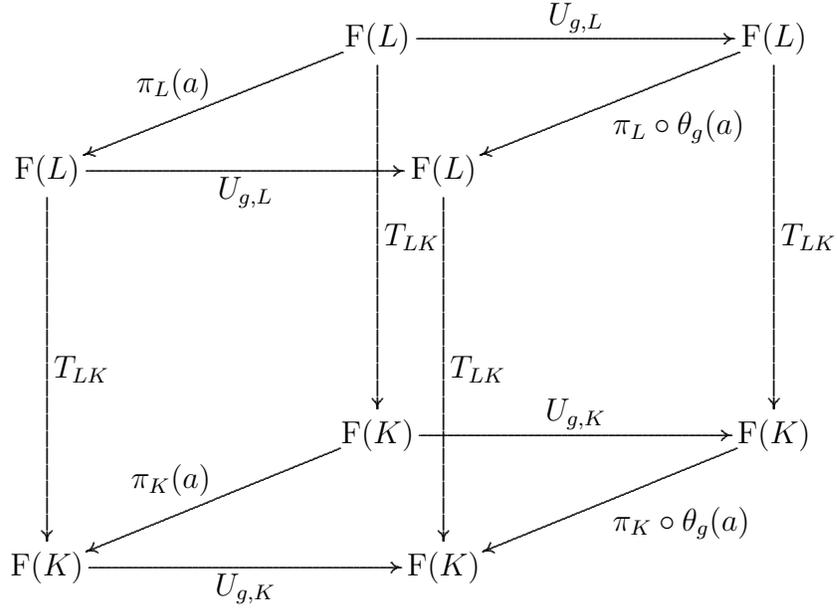
\[
\begindc{\commdiag}[5]
\obj(35,50)[objFL]{$\F(L)$}
\obj(65,50)[objFL2]{$\F(L)$}
\obj(10,40)[objFL3]{$\F(L)$}
\obj(40,40)[objFL4]{$\F(L)$}
\mor{objFL}{objFL2}{$U_{g,L}$}
\mor{objFL3}{objFL4}{$U_{g,L}$}[\atright,\solidarrow]
\mor{objFL}{objFL3}{$\pi_L(a)$}[\atright,\solidarrow]
\mor{objFL2}{objFL4}{$\pi_L \circ \theta_g (a)$}
\obj(35,20)[objFK]{$\F(K)$}
\obj(65,20)[objFK2]{$\F(K)$}
\obj(10,10)[objFK3]{$\F(K)$}
\obj(40,10)[objFK4]{$\F(K)$}
\mor{objFK}{objFK2}{$U_{g,K}$}
\mor{objFK3}{objFK4}{$U_{g,K}$}[\atright,\solidarrow]
\mor{objFK}{objFK3}{$\pi_K(a)$}[\atright,\solidarrow]
\mor{objFK2}{objFK4}{$\pi_K \circ \theta_g (a)$}
\mor{objFL}{objFK}{$T_{LK}$}
\mor{objFL2}{objFK2}{$T_{LK}$}
\mor{objFL3}{objFK3}{$T_{LK}$}
\mor{objFL4}{objFK4}{$T_{LK}$}
\enddc
\]
\end{proof}

\begin{cor}\label{co-set-impl-u1-tors}
\index{implementer!set of!U(1) torsor@$\UU(1)$ torsor}
(The Set of Implementers is a $\UU(1)$ torsor).
Given a Lagrangian subspace $L$ of $H$ and some $g \in \Orth(V)$, the set of unitary implementers $U \colon \F(L) \rightarrow \F(L)$ of $\theta_g$ in $\pi_L$ is a $\UU(1)$ torsor with the same topology it would have as a subspace of $\UU(\F(L))$.  If $g \in \UU(V_J)$, this torsor can be identified with $\UU(1)$, by letting $\Lambda_g$ correspond to $1$.
\end{cor}
\begin{proof}
This is a consequence of propositions \ref{p-set-intw-u1-tors} and \ref{p-bij-intw-impl}, and corollary \ref{co-uvj-cano-impl}.
\end{proof}

The correspondence of orthogonal maps to their implementers is used later.
\begin{lem}\label{l-impl-homo}
\index{implementer!homomorphism}
($\Orth(V) \rightarrow \{ \text{ implementers } \}$ is a Homomorphism Up to a $\UU(1)$ Factor).
Suppose given a Lagrangian subspace $L$ of $H$, and $g_1, g_2 \in \Orth(V)$ that are implemented in $\pi_L$ by $U_{g_1}$, $U_{g_2}$, respectively.  Then $g_1 g_2$ is implemented in $\pi_L$ by $U_{g_1} U_{g_2}$.  Thus if $U_{g_1 g_2}$ is any implementer of $g_1 g_2$ in $\pi_L$, $U_{g_1} U_{g_2} = z U_{g_1 g_2}$ for some $z \in \UU(1)$.
\end{lem}
\begin{proof}
The loops in the following diagram are commutative by definition.

\[
\begindc{\commdiag}[5]
\obj(10,25)[objFL]{$\F(L)$}
\obj(30,25)[objFL2]{$\F(L)$}
\obj(10,10)[objFL3]{$\F(L)$}
\obj(30,10)[objFL4]{$\F(L)$}
\obj(55,25)[objFL5]{$\F(L)$}
\obj(55,10)[objFL6]{$\F(L)$}
\mor{objFL}{objFL2}{$U_{g_2}$}
\mor{objFL3}{objFL4}{$U_{g_2}$}
\mor{objFL}{objFL3}{$\pi_L(a)$}
\mor{objFL2}{objFL4}{$\pi_L (\theta_{g_2} (a))$}
\mor{objFL2}{objFL5}{$U_{g_1}$}
\mor{objFL4}{objFL6}{$U_{g_1}$}
\mor{objFL5}{objFL6}{$\pi_L (\theta_{g_1} \circ \theta_{g_2} (a)) = \pi_L (\theta_{g_1 g_2} (a))$}
\cmor((12,27)(30,33)(53,27)) \pright(30,35){$U_{g_1} U_{g_2} = z U_{g_1 g_2}$}
\cmor((12,8)(30,2)(53,8)) \pright(30,0){$U_{g_1} U_{g_2} = z U_{g_1 g_2}$}
\enddc
\]
\end{proof}

\begin{lem}\label{l-lamb-homo}
\index{Lambda@$\Lambda$!homomorphism}
(The $\Orth(V) \rightarrow \{ \Lambda \}$ Homomorphism).
Suppose given Lagrangian subspaces $L_1, L_2, L_3$, of $H$, and $g_1, g_2 \in \Orth(V)$ such that $L_2 = g_1 L_1$, $L_3 = g_2 L_2 = (g_2 g_1) L_1$. Let $\Lambda_{g_1}, \Lambda_{g_2}, \Lambda_{g_2 g_1}$ be the isometric isomorphisms induced by $g_1, g_2, g_2 g_1$.  Then $\Lambda_{g_2} \circ \Lambda_{g_1} = \Lambda_{g_2 g_1}$.

\[
\begindc{\commdiag}[5]
\obj(10,25)[objFL1T]{$\F(L_1)$}
\obj(30,25)[objFL2T]{$\F(L_2)$}
\obj(55,25)[objFL3T]{$\F(L_3)$}
\obj(10,10)[objFL1B]{$\F(L_1)$}
\obj(30,10)[objFL2B]{$\F(L_2)$}
\obj(55,10)[objFL3B]{$\F(L_3)$}
\mor{objFL1T}{objFL2T}{$\Lambda_{g_1}$}
\mor{objFL2T}{objFL3T}{$\Lambda_{g_2}$}
\mor{objFL1B}{objFL2B}{$\Lambda_{g_1}$}
\mor{objFL2B}{objFL3B}{$\Lambda_{g_2}$}
\mor{objFL1T}{objFL1B}{$\pi_{L_1}(a)$}
\mor{objFL2T}{objFL2B}{$\pi_{L_2} (\theta_{g_1} (a))$}
\mor{objFL3T}{objFL3B}{$\pi_{L_3} (\theta_{g_2} \circ \theta_{g_1} (a)) = \pi_{L_3} (\theta_{g_2 g_1} (a))$}
\cmor((12,27)(30,33)(53,27)) \pright(30,35){$\Lambda_{g_2 g_1} = \Lambda_{g_2} \Lambda_{g_1}$}
\cmor((12,8)(30,2)(53,8)) \pright(30,0){$\Lambda_{g_2 g_1} = \Lambda_{g_2} \Lambda_{g_1}$}
\enddc
\]
\end{lem}

It's important later that the correspondence of orthogonal transformation and $\Lambda$ is continuous, in a sense defined now.
\begin{lem}\label{l-lamb-homo-cont}
\index{Lambda@$\Lambda$!homomorphism!continuous}
(The $\UU(V_J) \rightarrow \{ \Lambda \}$ Homomorphism is Continuous).
Suppose given a Lagrangian subspace $L$ of $H$, with corresponding unitary structure $J$ on $V$.  Then the map $\Lambda \colon \UU(V_J) \rightarrow \UU(\F(L))$ is continuous, with the operator norm topology on $\UU(V_J)$ and the strong operator topology on $\UU(\F(L))$.
\end{lem}
\begin{proof}
Lemma \ref{l-lamb-homo} showed that $\Lambda$ is an algebraic homomorphism.  We want to show that for $g, g_0 \in \UU(V_J)$ and fixed $w \in \F(L)$, $g \rightarrow g_0 \Rightarrow \norm{(g - g_0) w} \rightarrow 0$.  Reducing to a finite-dimensional computation similarly to what was done in the proof of proposition \ref{p-bogo-map-cont}, given $\epsilon > 0$ we may choose $w_1$ the sum of decomposables of degree at most $d_1$ such that $\norm{w - w_1} < \epsilon / 6$, whence because $\norm{g - g_0} \le 2$, $\norm{(g - g_0) (w - w_1)} < \epsilon / 3$, and may choose $w_2$ the sum of decomposables of degree at most $d_1$, and each decomposable consisting of vectors from the same subspace of $L$ of finite dimension $d_2$, such that $\norm{w_1 - w_2} < \epsilon / 6$, whence similarly $\norm{(g - g_0) (w_1 - w_2)} < \epsilon / 3$.  Then it suffices to show that by making $\norm{g - g_0}$ small enough, we can make $\norm{(g - g_0) w_2} < \epsilon / 3$.  Now, $\norm{(g - g_0) w_2}^2$ is the sum of no more than some natural number $N$ determinants of matrices of size at most $d_2 \cross d_2$, one for each decomposable in $w_2$.  The entries of each matrix are inner products of $g - g_0$ applied to two of the vectors in the decomposable.  Let $W$ be the largest norm of any vector in any of the decomposables; then each determinant of a matrix of size $k \cross k$ will be no larger than $(k \norm{(g - g_0) w_2}^2 W^2)^k$ (using the determinantal inequality $\abs{(a_{ij})} \le \prod_{i=1}^k \sum_{j=1}^k \abs{a_{ij}}$).  Thus the total is no more than
\[
N \max_{k=1}^{d_2} (k \norm{(g - g_0) w_2}^2 W^2)^k, \notag
\]
(since $g$ and $g_0$ don't affect degree $0$ decomposables), which can be made as small as desired.
\end{proof}
$\UU(V_J)$ is a topological group in its operator norm topology, and $\UU(\F(L))$ is a topological group in its strong operator topology; but we didn't need those facts.

\section{Solution to the Questions of Implementation and Equivalence}\label{s-soln-ques-impl-equi}

To understand conditions for existence of unitary implementers, or equivalently by proposition \ref{p-bij-intw-impl} conditions for two Fock representations to be equivalent, we need the concept of Hilbert-Schmidt operators.
\begin{lem}\label{l-hsop}
\index{Hilbert-Schmidt!operator}
\index{Hilbert-Schmidt!inner product}
\index{Hilbert-Schmidt!norm}
(Hilbert-Schmidt Operators).
For any real (respectively complex) Hilbert space $W$ with inner product $\langle , \rangle$ and norm $\norm{}$, the set $\B_2(W) \subset \B(W)$ of Hilbert-Schmidt operators, i.e. those $A \in \B(W)$ for which, given any orthonormal basis $\{ e_i \}$ of $W$, $\norm{A}_2 = \sum \norm{A e_i}^2 < \infty$, is a two-sided self-adjoint ideal.  With inner product defined for $A, B \in \B_2(W)$ as $\langle A, B \rangle_2 = \sum \langle A e_i, B e_i \rangle_2$, $\B_2(W)$ is a real (respectively complex) Hilbert space.  The definitions don't depend on the basis.  The following norm properties hold for any $A, B \in \B_2(W)$, any $T \in \B(W)$, and any $U \in \Orth(W)$:
\begin{align}
\norm{A^{*}}_2 &= \norm{A}_2 \notag \\
\norm{U A}_2 = \norm{A U}_2 &= \norm{A}_2 \notag \\
\norm{T A}_2 &\le \norm{T} \norm{A}_2 \notag \\
\norm{A T}_2 &\le \norm{A}_2 \norm{T} \notag \\
\norm{A} &\le \norm{A}_2 \notag \\
\norm{U A U^{*} - A}_2 &\le 2 \norm{A}_2. \notag
\end{align}
Defining $\B_2(V,W)$ analogously, where $V$ and $W$ are real (respectively complex) Hilbert spaces, analogous properties hold.
\end{lem}
\begin{proof}
That $\B_2(W)$ is a two-sided ideal, is self-adjoint, and that the various norm properties hold, follow from \citet[page~267]{Conw90} exercise 19 (a), (b), (c), (d), as does the verification that $\langle , \rangle_2$ is an inner product.  At that point in the book, scalars are assumed complex, but these things don't depend on that.  The last norm property can be seen from $\norm{U A U^{*} - A}_2 = \norm{U A - A U}_2$.  That $\B_2(W)$ is complete and hence is a real (respectively complex) Hilbert space can be proved as in \citet[pages~118--119]{Pede89}.  The proof works for real scalars, although the book uses the polarization identity to prove some other things nearby.

The analogous definition of $B_2(V,W)$ and its properties are proved analogously, taking care that $T$ and $U$ operate on $V$ or $W$ depending on whether their composition with $A$ is to the right or left.  In the case of complex Hilbert spaces, the definition and properties are consequences of \citet[pages~265--268]{Fabe00}.
\end{proof}

The concept of Hilbert-Schmidt operators lets us define several related notions whose uses permeate the thesis, including letting us state the solutions to the implementation and equivalence problems, which we will do shortly.  We used unitary structures to introduce Lagrangian subspaces, and recently have cast things more in terms of the latter; but now again use unitary structures to define something important.  We will state what seems logical here and needed to give solutions to our two related problems; then these subjects will be taken up again in chapter \ref{c-res-orth-grp}.

The condition $J_2 - J_1$ Hilbert-Schmidt is an equivalence relation for unitary structures (see lemma \ref{l-hsop}); it partitions $\US(V)$ into equivalence classes.  By proposition \ref{p-bij-usv-lagh}, $\US(V)$ and $\Lagr(H,\Sigma)$ are in bijection, so the partition of the former gives us a partition of the latter.  Recall that Lagrangian subspaces are called polarizations of $H$, as in the comment before definition \ref{d-unit-stru}.
\begin{defn}\label{d-pol-clas}
\index{polarization class}
(Polarization Classes).
The equivalence classes of $\Lagr(H,\Sigma)$ corresponding under the bijection of proposition \ref{p-bij-usv-lagh} to the equivalence classes of $\US(V)$ by the relation $J_2 - J_1$ Hilbert-Schmidt, $J_1, J_2 \in \US(V)$, are called polarization classes.  We may also refer to the equivalence classes of $\US(V)$ as polarization classes.  For $L \in \Lagr(H,\Sigma)$ and $J \in \US(V)$, denote by $[L]$, $[J]$ the polarization classes containing $L$, $J$ respectively.
\end{defn}

By proposition \ref{p-uv-ov}, $\US(V)$ is in bijection with a homogeneous space of $\Orth(V)$.  That proposition's continuous equivariant bijection depends on a chosen unitary structure $J$, or equivalently, a chosen corresponding Lagrangian subspace, the choice of which lets us use the bijection and the quotient map to partition $\Orth(V)$ into equivalence classes corresponding to the equivalence classes for $\US(V)$.  There is a preferred equivalence class of elements of $\Orth(V)$, the one that contains the identity, and this class we will call $\Orth_{res, J}$.
\begin{defn}\label{d-ores-ures-lagr-res}
\index{OresJ@$\Orth_{res, J}$}
\index{UresJ@$\US_{res, J}$}
\index{LagrresJ@$\Lagr_{res, J}$}
\index{Hilbert-Schmidt!operator}
\index{restricted group!orthogonal}
\index{orthogonal group!restricted}
($\Orth_{res, J}$, $\US_{res, J}$, $\Lagr_{res, J}$).
Partition $\US(V)$ into equivalence classes as in definition \ref{d-pol-clas}.  Given a unitary structure $J$ on $V$, use inverse images via the equivariant bijection and quotient map of proposition \ref{p-uv-ov} to partition $\Orth(V)$ into equivalence classes.  Denote by $\Orth_{res, J}$, the restricted orthogonal group for $J$, the equivalence class containing $\ident_{\Orth(V)}$.  More explicitly, $\Orth_{res, J} = \{ g \in \Orth(V) \st g J g^{-1} \text{ is in the same polarization class as } J \}$.

Denote by $\US_{res, J}$ the corresponding equivalence class in $\US(V)$, or equivalently the equivalence class containing $J$; and denote by $\Lagr_{res, J}$ the corresponding polarization class in $\Lagr(H,\Sigma)$ under the bijection of proposition \ref{p-bij-usv-lagh}.
\end{defn}
In the preceding definition or the succeeding propositions, the choice of a unitary structure can be made, equivalently, by the choice of a corresponding Lagrangian subspace, since the bijection between unitary structures and Lagrangian subspaces is canonical.

Although the definition has the choice of a unitary structure or equivalently a Lagrangian subspace, it only depends on the choice of a polarization class.
\begin{prop}\label{p-ores-dep-pol-clas}
\index{OresJ@$\Orth_{res, J}$}
\index{Hilbert-Schmidt!operator}
\index{restricted group!orthogonal}
\index{orthogonal group!restricted}
($\Orth_{res, J}$ Depends Only on the Polarization Class of $J$).  Given a unitary structure $J$ on $V$, we can write $\Orth_{res, [L]} = \Orth_{res, [J]} = \Orth_{res, J}$, and similarly for $\US_{res, [J]}$ and $\Lagr_{res, [J]}$.  We have $\US_{res, [J]} = [J]$.  If $L$ is the Lagrangian subspace of $H$ corresponding to $J$, $\Lagr_{res, [J]} = [L]$, the polarization class of $L$; $\Lagr_{res, [L]} = [L]$.
\end{prop}
\begin{proof}
Suppose $g_0 J g_0^{-1}$, for some $g_0 \in \Orth(V)$, is another element of the polarization class containing $J$; that is, $\norm{g_0 J g_0^{-1} - J}_2 < \infty$.  For $g \in \Orth(V)$, $g \in \Orth_{res, J} \Leftrightarrow g \in \Orth_{res, g_0 J g_0^{-1}}$ is equivalent to $\norm{g J g^{-1} - J}_2 < \infty \Leftrightarrow \norm{g g_0 J g_0^{-1} g^{-1} - g_0 J g_0^{-1}}_2 < \infty$.  Note that for unitary operators $A, B$, $\norm{A B A^{-1} - B}_2 = \norm{A^{-1} B A - B}_2$.

``$\Rightarrow$'': Suppose $\norm{g J g^{-1} - J}_2 < \infty$.  Then
\begin{align}
\norm{g g_0 J g_0^{-1} g^{-1} - g_0 J g_0^{-1}}_2 &\le \norm{g g_0 J g_0^{-1} g^{-1} - J}_2 + \norm{J - g_0 J g_0^{-1}}_2 \notag \\
 &= \norm{g_0 J g_0^{-1} - g^{-1} J g}_2 + \norm{J - g_0 J g_0^{-1}}_2 \notag \\
 &\le \norm{g_0 J g_0^{-1} - J}_2 + \norm{J - g^{-1} J g}_2 + \norm{J - g_0 J g_0^{-1}}_2 < \infty. \notag
\end{align}
``$\Leftarrow$'': Suppose $\norm{g g_0 J g_0^{-1} g^{-1} - g_0 J g_0^{-1}}_2 < \infty$.  Then
\begin{align}
\norm{g J g^{-1} - J}_2 &\le \norm{g J g^{-1} - g_0 J g_0^{-1}}_2 + \norm{g_0 J g_0^{-1} - J}_2 \notag \\
 &= \norm{J - g^{-1} g_0 J g_0^{-1} g}_2 + \norm{g_0 J g_0^{-1} - J}_2 \notag \\
 &= \norm{J - g_0 J g_0^{-1}}_2 + \norm{g_0 J g_0^{-1} - g^{-1} g_0 J g_0^{-1} g}_2 + \norm{g_0 J g_0^{-1} - J}_2 < \infty. \notag
\end{align}
\end{proof}

\begin{prop}\label{p-ojv-ores}
\index{OJ(V)@$\Orth_J(V)$}
\index{OresJ@$\Orth_{res, J}$}
\index{Hilbert-Schmidt!norm}
\index{Hilbert-Schmidt!operator}
\index{restricted group!orthogonal}
\index{orthogonal group!restricted}
(Alternative Definition of $\Orth_{res, J}$ and $\US_{res, J}$).
\citep[page~109]{PR94} Given a unitary structure $J$ on $V$, define
\[
\Orth_J(V) = \{ g \in \Orth(V) \st \norm{A_g}_2 = \norm{\frac{J}{2}[g, J]}_2 = \frac{1}{2} \norm{[g, J]}_2 < \infty \},
\]
$J$ being implicit in the symbol $A_g$, referring ahead to definition \ref{d-g-cgag} for its definition as shown in the second equality, and to lemma \ref{l-hsop} for the Hilbert-Schmidt norm $\norm{}_2$.  The last equality follows from $J$ being orthogonal.  Then $\Orth_{res, J} = \Orth_J(V)$.
\end{prop}
\begin{proof}
Referring to definition \ref{d-pol-clas} and proposition \ref{p-ores-dep-pol-clas}, it suffices to show that $\norm{\frac{J}{2}[g, J]}_2 < \infty \Leftrightarrow \norm{g J g^{-1} - J}_2 < \infty$.  But $\norm{J [g, J]}_2 = \norm{[g, J]}_2 = \norm{[g, J] g^{-1}}_2 = \norm{g J g^{-1}- J}_2$, since $J, g \in \Orth(V)$.
\end{proof}
(\citet[page~109]{PR94} use $\Orth_J(V)$, but we use $\Orth_{res, J} = \Orth_{res, L}$ since we use Lagrangian subspaces more than they do. \citet[page~244]{PS86} uses $\Orth_{res} (H_{\CC})$, depending on $J$ without indicating that in the notation, from which we take the first part, leaving implicit the Hilbert space, but additionally showing $J$.  When, in chapter \ref{c-res-orth-grp}, we fix $V$ and a polarization class $\Lagr_{res}$ of $H$, we will then use $\Orth_{res}$ alone, and $\US_{res}$.)

\begin{thm}\label{t-impl-soln}
\index{implementer!existence}
(A Condition for the Existence of Unitary Implementers). \linebreak
\citep[page~108--111]{PR94} Given a unitary structure $J$ on $V$, with corresponding Lagrangian subspace $L$ of $H$ and Fock representation $\pi_L$, and given any $g \in \Orth(V)$, extended complex linearly to $H$:
\begin{align}
g \in \Orth_{res, J} &\Leftrightarrow g \text{ has Hilbert-Schmidt off-diagonals for } H = L \dirsum \Sigma (L) \notag \\
 &\Leftrightarrow \theta_g \text{ is unitarily implemented in } \pi_L \notag
\end{align}
\end{thm}

Now, the condition for equivalence of two Fock representations of $\Cl(V)$ can be stated easily:  they are equivalent if and only if the Lagrangian subspaces are in the same polarization class.
\begin{cor}\label{co-equi-soln}
\index{intertwiner!existence}
(A Condition for Unitary Equivalence).
\citep[page~111--112]{PR94} Given two unitary structures $J_1$ and $J_2 = g J_1 g ^{-1}$ on $V$, $g \in \Orth(V)$, with corresponding Lagrangian subspaces $L_1, L_2$ and Fock representations $\pi_1, \pi_2$:
\begin{align}
g \in \Orth_{res, J_1} = \Orth_{res, J_2} &\Leftrightarrow J_2 - J_1 \text{ is Hilbert-Schmidt } \notag \\
 &\Leftrightarrow L_1, L_2 \text{ are in the same polarization class} \notag \\
 &\Leftrightarrow \text{Fock representations } \pi_1,\text{ }\pi_2 \text{ unitarily equivalent.} \notag
\end{align}
\end{cor}

\chapter{THE CLIFFORD ALGEBRA BUNDLE}\label{c-calg-bndl}

For those who have just read through the lengthy background chapters, it might be a good idea to review the technical overview of the rest of the thesis, in section \ref{s-tech-over} of chapter \ref{c-intr}, which may make more sense now.  For example, it has an introduction to our use of associated bundles.

This chapter uses an associated bundle construction to get the Clifford algebra bundle, a topological fiber bundle associated to the \Frechet principal bundle $L\SO(E) \rightarrow LM$ via a continuous action of $L\SO(n)$ on the standard or model $C^{*}$ Clifford algebra $\Cl(L\RR^n)$, on which we do not define a smooth structure.  This bundle is denoted $\Cl(LE) \rightarrow LM$, with $\Cl(LE)_{\gamma} \cong \Cl(LE_{\gamma})$, defined more precisely using associated bundles.  For the geometric viewpoint, we can take $LE_{\gamma}$ as the uncompleted real inner product space of lemma \ref{l-fibr-loop-vb} and definition \ref{d-fibr-innr-prod-loop-vb}, but by \citet[page~28]{PR94}, an isomorphic Clifford algebra results from using the completion, $\overline{LE_{\gamma}}$; but again the actual definition of $\Cl(LE)$ is as an associated bundle.  We begin by defining a real Hilbert space structure on $L\RR^n$, used to define its Clifford algebra and Fock spaces.  Other than for this purpose and as otherwise noted, $L\RR^n$ will still denote, as in example \ref{e-lrn-l-vect-spac}, the \Frechet space.

\begin{defn}\label{d-lrn-preh}
\index{loop space!inner product}
\index{LRn@$L\RR^n$}
\index{Cl(LRn)@$\Cl(L\RR^n)$}
\index{Cl(L2S1Rn)@$\Cl(L\RR^n)$}
($L\RR^n$, $\Cl(L\RR^n)$).
Define the inner product for $L\RR^n$ by $(\alpha, \beta) = \int_{S^1} (\alpha(t), \beta(t)) dt$ using the standard inner product on $\RR^n$.  The completion of $L\RR^n$ for this inner product is $L^2(S^1, \RR^n)$.  Use definition \ref{d-c-star-clif-alg} to define the model Clifford $C^{*}$-algebra over $L^2(S^1, \RR^n)$, $\Cl(L\RR^n) = \Cl(L^2(S^1, \RR^n))$.

Notice that the symbol $L\RR^n$ will still indicate the \Frechet topology, not the one from the inner product just defined, except in a few cases when noted otherwise.
\end{defn}
This is consistent with definition \ref{d-fibr-innr-prod-loop-vb}, taking $L\RR^n$ as the smooth sections of the trivial smooth vector bundle $S^1 \cross \RR^n \rightarrow S^1$ with fiberwise inner product from the standard inner product on $\RR^n$.

\begin{lem}\label{l-id-lrn-lrnh-cont}
\index{LRn@$L\RR^n$!identity map to L2S1Rn@identity map to $L^2(S^1, \RR^n)$}
($L\RR^n \xrightarrow{\ident} L^2(S^1, \RR^n)$ is continuous).
\end{lem}
\begin{proof}
Given any $\sigma \in L\RR^n$,
\[
\norm{\sigma}^2 = \int_{S^1} (\sigma(t), \sigma(t)) dt \le \max_{t \in S^1} (\sigma(t), \sigma(t)) = \norm{\sigma}_0^2, \notag
\]
where we consider $S^1 = \RR / \ZZ$ and integrate from $0$ to $1$, the norm on the left hand side being the norm for $L^2(S^1, \RR^n)$, the inner product being the standard one for $\RR^n$, and $\norm{\sigma}_0$, the 0-th seminorm for the \Frechet space $L\RR^n$.
\end{proof}

We use the \Frechet manifold topology for the loop of a smooth principal bundle and on the left hand side of associated bundle constructions, but when we consider the right hand side, the action of $L\SO(n)$ on some space, factored through the action of an orthogonal group on a Hilbert space, we will use the first two seminorms of the subspace \Frechet topology.  This will support our use of Fourier series for elements of $L\SO(n)$ considered as elements of $L\B(\RR^n)$, where $\B(R^n)$ is the space of bounded operators on $\RR^n$.

\begin{lem}\label{l-lson-two-top}
\index{LSOn@$L\SO(n)$!two \Frechet topologies}
(The Relation Between Two \Frechet-related topologies for $L\SO(n)$).
One topology for $L\SO(n)$ is that of a \Frechet manifold from proposition \ref{p-loop-lgrp-frec-lgrp}.  It also has a topology as a subspace of the \Frechet space $L\B(\RR^n)$ (see lemma \ref{e-lrn-l-vect-spac}). Convergence in the \Frechet manifold topology implies convergence in the $0$-th and $1$-st seminorms, $\norm{}_0$, $\norm{}_1$, of the subspace of \Frechet space topology.
\end{lem}
\begin{proof}
We have not already specified a topology for $\SO(n)$, but from \citet[page~22]{Hall03}, which says that every matrix Lie group is a smooth embedded real submanifold of $\B(\CC^n)$, its manifold topology is the same as the one from its realization as a matrix group; i.e., since $\SO(n)$ consists of real matrices, as a subspace of $\B(\RR^n)$, a finite-dimensional vector space so that all norms on it are equivalent.  We can use the operator norm, with its associated metric that induces the compact-open topology on $C(S^1, \SO(n))$.  By proposition \ref{p-smth-free-loop-spac-frec-mfld}, the inclusion $L\SO(n) \rightarrow C(S^1, \SO(n))$ is continuous.  Thus convergence in $L\SO(n)$ in the \Frechet manifold topology implies convergence in the topology induced by inclusion into $C(S^1, \SO(n))$; i.e. uniform convergence in the operator norm; that is to say, convergence in the $0$-th seminorm of the \Frechet space topology for $L\B(\RR^n)$.

The statement about the $1$-st seminorm follows from proposition 3.25 of \citet[page~21]{Stac05}, that the topology on $LX$ is the projective topology for the family of maps $LX \rightarrow C(S^1, T^{(k)} X)$ defined by taking successive derivatives of loops.  I.e., for $k = 1$, he defines a section $\tau \colon S^1 \rightarrow TS^1$, $\tau \colon t \mapsto (t, \frac{\partial}{\partial t})$, which induces $\widehat{\tau} \colon LX \rightarrow LTX$, $\widehat{\tau}(\alpha) = d\alpha \circ \tau$, and notes that composition with the inclusion of smooth loops into continuous ones results in continuity of the first map in the stated family.  The background for his proof assumes that $X$ is orientable, and though that may not be a necessary assumption, as discussed in the proof of our proposition \ref{p-smth-free-loop-spac-frec-mfld}, it is satisfied for $X = \SO(n)$, since all Lie groups are orientable \citep[page~92]{Bump04}.

$X = \SO(n) \subset \B(\RR^n)$ is an embedding, a homeomorphism onto its image that induces smooth bundle map $TX \rightarrow T\B(\RR^n) \cong \B(\RR^n) \cross \B(\RR^n)$.  That map, in turn, induces by lemma \ref{l-co-top} a continuous map $C(S^1, TX) \rightarrow C(S^1, \B(\RR^n))$ (the second factor, made up of the tangent spaces).  Composing with the continuous inclusion $LTX \subset C(S^1, TX)$, $LX \rightarrow LTX \rightarrow C(S^1, TX) \rightarrow C(S^1, \B(\RR^n))$ is continuous, and amounts to $\alpha \mapsto \alpha'$.
\end{proof}

\begin{defn}\label{d-actn-lson-clrn}
\index{Cl(LRn)@$\Cl(L\RR^n)$!action of LSO(n) on@action of $L\SO(n)$ on!}
\index{LSO(n) action on Cl(LRn)$L\SO(n)$ @ action on $\Cl(L\RR^n)$}
(The Action of $L\SO(n)$ on $\Cl(L\RR^n)$).
Define the action of $L\SO(n)$ on $L^2(S^1, \RR^n)$ pointwise over $S^1$.
\end{defn}

\begin{defn}\label{l-actn-lson-clrn}
\index{Cl(LRn)@$\Cl(L\RR^n)$!action of LSO(n) on@action of $L\SO(n)$ on!}
\index{LSO(n) action on Cl(LRn)$L\SO(n)$ @ action on $\Cl(L\RR^n)$}
(Properties of the Action of $L\SO(n)$ on $\Cl(L\RR^n)$).
Since $\SO(n)$ preserves the inner product on $\RR^n$, an element of $L\SO(n)$ preserves the inner product at each point $t \in S^1$, and hence the integral over $S^1$.  Thus $L\SO(n) \subset \Orth(L^2(S^1, \RR^n))$, and so $g \in L\SO(n)$ acts on $\Cl(L\RR^n)$ via the Bogoliubov automorphism $\theta_g$ of $\Cl(L\RR^n)$.
\end{defn}

\begin{lem}\label{l-actn-lson-clrn-cont}
\index{Cl(LRn)@$\Cl(L\RR^n)$!action of LSO(n) on@action of $L\SO(n)$ on!continuous}
\index{LSO(n) action on Cl(LRn)$L\SO(n)$ @ action on $\Cl(L\RR^n)$!continuous}
\index{LSOn@$L\SO(n)$!inclusion in O(L2S1Rn)continuous@inclusion in $\Orth(L^2(S^1, \RR^n))$ continuous}
The inclusion $L\SO(n) \hookrightarrow \Orth(L^2(S^1, \RR^n))$ is continuous, as is the action of $L\SO(n)$ on $\Cl(L\RR^n)$.
\end{lem}
\begin{proof}
We will see that $L\SO(n) \hookrightarrow \Orth(L^2(S^1, \RR^n))$ is continuous for the operator norm topology on $\Orth(L^2(S^1, \RR^n))$.  For $g, g_0 \in L\SO(n)$, $\sigma \in L^2(S^1, \RR^n)$,
\begin{align}
\norm{g - g_0}^2 &= \sup_{\norm{\sigma} = 1} (\norm{(g - g_0) \sigma}^2) \notag \\
 &= \sup_{\norm{\sigma} = 1} (\int_{S^1} \norm{(g - g_0)(t) \sigma(t)}^2 dt) \notag \\
 &\le \sup_{\norm{\sigma} = 1} (\int_{S^1} \norm{(g - g_0)(t)}^2 \norm{\sigma (t)}^2 dt) \notag \\
 &\le \sup_{\norm{\sigma} = 1} (\max_{S^1} (\norm{(g - g_0)(t)}^2) \int_{S^1} \norm{\sigma (t)}^2) \notag \\
 &= \sup_{\norm{\sigma} = 1} (\norm{g - g_0}_0^2 \norm{\sigma}^2) = \norm{g - g_0}_0^2. \notag
\end{align}
So, by making $g$ close enough to $g_0$ in the \Frechet topology, $\norm{g - g_0}$ can be made as small as desired, and the inclusion is continuous.  Then since proposition \ref{p-bogo-map-cont} says that the action of $\Orth(L^2(S^1, \RR^n))$ with operator norm on $\Cl(L\RR^n)$ is continuous, the action of $L\SO(n)$ factoring through the action of $\Orth(L^2(S^1, \RR^n))$ on $\Cl(L\RR^n)$ is continuous.
\end{proof}

\begin{defn}\label{d-clif-alg-bndl}
\index{Cl(LE)@$\Cl(LE)$}
\index{Clifford algebra!bundle}
(The Clifford Algebra Bundle).
\begin{align}
\Cl(LE) = L\SO(E) \cross_{L\SO(n)} \Cl(L\RR^n) \xrightarrow{\pi_{\Cl(LE)}} LM \notag
\end{align}
\end{defn}
\begin{note}\label{n-clif-alg-bndl}
\index{associated bundle!geometric viewpoint}
(Associated Bundle vs. Geometric Viewpoints; Smoothness vs. Completeness).
To aid in intuition, one may use the correspondence
\begin{align}
\Cl(LE) &= \{ [(\widetilde{\gamma}, a)] \st \widetilde{\gamma} \in L\SO(E) \text{ and } a \in \Cl(L\RR^n) \} \notag \\
 &\leftrightarrow \{ (\gamma, a_{\gamma}) \st a_{\gamma} \in \Cl(LE_{\gamma}) \}, \notag
\end{align}
where $\gamma = \pi_{L\SO(E)} (\widetilde{\gamma})$ and $LE_{\gamma} = \Gamma(\gamma^{*} E)$.

We are saved from questions of the smooth loops in $\Cl(L\RR^n)$ versus the completeness of $L^2(S^1, \RR^n)$ because the Clifford algebra construction gives the same complete $C^{*}$-algebra regardless of whether the vector space it builds upon is complete (see the introduction to chapter \ref{c-clif-alg-fock-rep}).
\end{note}

\chapter{THE RESTRICTED ORTHOGONAL GROUP}\label{c-res-orth-grp}

On the way to our goal of creating a Clifford module bundle over $LM$, over each point $\gamma$ of $LM$ we pick a fixed polarization class of Lagrangian subspaces of the complexification of $LE_{\gamma}$, any of which would give rise to an equivalent representation of the Clifford algebra $\Cl(LE_{\gamma})$ over that point.  This can be done consistently so as to form a bundle.  However, we can't then choose consistently one of these Lagrangian subspaces over each point, to use in a Fock space construction.  To deal with this problem, we consider $\Orth(LE_{\gamma})$.  Our chosen polarization class of Lagrangian subspaces corresponds to a restricted orthogonal group, which as a set is a subgroup of $\Orth(LE_{\gamma})$.

To facilitate proofs, we use associated bundle constructions to refer this analysis to a standard real Hilbert space, a standard polarization class of Lagrangian subspaces of its complexification, and the standard restricted orthogonal group corresponding to that polarization class.

This chapter discusses the restricted Orthogonal group of automorphisms of the real Hilbert space $V = L^2 (S^1, \RR^n)$, $\Orth_{res}$ = $\Orth_{res, \Lagr_{res}}$, whose definition depends on a standard polarization class $\Lagr_{res}$ of Lagrangian subspaces of $H = \CC \tensor_{\RR} V$.  Recall from section \ref{s-soln-ques-impl-equi}, definitions \ref{d-pol-clas}, \ref{d-ores-ures-lagr-res} and propositions \ref{p-ores-dep-pol-clas}, \ref{p-ojv-ores}.

We will also discuss in this chapter a little about topological group representations.  We will give a projective representation of $\Orth_{res}$.  Although the thesis doesn't use the representation, work that leads to it will be needed later.

\index{multiplication operator}
The structure group of the smooth vector bundle we start with is $\SO(n)$, and the loop of that group, $L\SO(n)$, as a ``multiplication operator'' acting pointwise for each point on $S^1$, is included in the group $\Orth(L\RR^n)$ defined using the inner product of definition \ref{d-lrn-preh}, and by continuous extension of the operators to act on the completion of the pre-Hilbert space, $\Orth (V) = \Orth(L^2 (S^1, \RR^n))$.  The inclusion of $L\SO(n)$ actually lies in $\Orth_{res}$ for the $[L]$ we will define, which is good news for forming a bundle of polarization classes associated to the principal $L\SO(n)$ bundle $L\SO(E)$.

\section{Restricted Orthogonal Group and Homogeneous Space}\label{s-ores-homo-spac}

In order to define $\Orth_{res}$ we fix the polarization class $\Lagr_{res}$ of Lagrangian subspaces of $H$, and thus the corresponding polarization class $\US_{res}$ of unitary structures on $V$.
\begin{defn}\label{d-lagr-res}
\index{Lagrres@$\Lagr_{res}$}
\index{USres@$\US_{res}$}
\index{polarization class!standard}
(The Standard Polarization Class: for Lagrangian Subspaces, $\Lagr_{res}$; for Unitary Structures, $\US_{res}$).
Let $\Lagr_{res}$ be the polarization class of Lagrangian subspaces that contains the set of Lagrangian subspaces $L$ of the form specified by choosing any Lagrangian subspace $L_{finite} \subset \CC^n$ (recall that $n$, the rank of $E$, is even), and setting $L \subset H$ to the direct sum of the constant functions with values in $L_{finite}$, and the functions with only positive frequency Fourier series terms.

Let $\US_{res} = \US_{res, \Lagr_{res}}$.
\end{defn}
A calculation shows that $L^{\perp} = \Sigma (L)$, so $L$ is a Lagrangian subspace of $H$, and $J$ restricted to $V$ is a unitary structure on $V$.  Any two such $L_{finite, 1}, L_{finite, 2}$ give $L_1, L_2$ in the same polarization class by corollary \ref{co-equi-soln}, since the difference of the corresponding $J_1, J_2$ is zero except on a finite-dimensional subspace, and hence is Hilbert-Schmidt.  Thus all $L$ of the form specified are in the same polarization class, and so the definition is valid.  Note that $\Lagr_{res}$ does contain other $L$ not of the form specified.

For each $L \in \Lagr_{res}$, the corresponding $J \in \US_{res}$ is given by multiplication by $i$ on $L$, and multiplication by $-i$ on $L^{\perp}$.

The Fourier series here is the eigenspace decomposition of the derivative operator, which relates to the covariant derivative operator on vector fields along a loop, used in initial attempts at proving results in this thesis.

\begin{ass}\label{a-res-grp}
\index{assumptions!restricted group related}
\index{restricted group}
\index{L2@$L^2$}
\index{V@$V$}
\index{H@$H$}
\index{Lagrres@$\Lagr_{res}$}
\index{Ores@$\Orth_{res}$}
\index{orthogonal group!restricted}
(Fix $V$, $H$, $\Lagr_{res}$, $\US_{res}$, and $\Orth_{res}$).
Henceforth we let $V = L^2 (S^1, \RR^n)$. Then $H = \CC \tensor V = L^2 (S^1, \CC^n)$. We assume given the standard polarization class $\Lagr_{res}$ of Lagrangian subspaces of $H$, and from proposition \ref{p-bij-usv-lagh} get the corresponding polarization class $\US_{res}$ of unitary structures on $V$.  When needed, we can choose a particular $L_{finite}$ as in the definition, obtaining a particular $L \in \Lagr_{res}$ of the form specified, with corresponding $J \in \US_{res}$.  Using proposition \ref{p-ores-dep-pol-clas} we define the standard restricted orthogonal group of $V$, $\Orth_{res} = \Orth_{res, \Lagr_{res}}$.
\end{ass}

\begin{prop}\label{p-ores-top-grp}
\index{BresJ@$\B_{res, J}$!Banach algebra}
\index{Ores@$\Orth_{res}$!topological group}
\index{restricted group!orthogonal!topology}
($\Orth_{res}$ is a Topological Group).
\citep[pages~80,~244]{PS86} Given any $J \in \US_{res}$, $\Orth_{res}$ is a closed topological subspace of the real Banach algebra $\B_{res, J} = \{ A \in \B(V) \st \norm{[J,A]}_2 < \infty \}$ with norm $\norm{A}_J = \norm{A} + \norm{[J,A]}_2$, with $\norm{}_2$ as in lemma \ref{l-hsop}.  Only the norm of $\B_{res, J}$ depends on the specific $J$, and all such norms are equivalent, giving the same topology.

Further, $\Orth_{res}$ is a topological group, and its topology is the subspace topology inherited from any $\B_{res, J}$.  Its subgroup $\UU(V_J)$ is closed in it.
\end{prop}
\begin{proof}
Given $A, B \in \B_{res, J}$, $A B \in \B_{res, J}$ and $\norm{A B}_J \le \norm{A}_J \norm{B}_J$, because
\begin{align}
\norm{[J, A B]}_2 &= \norm{J A B - A B J}_2 \notag \\
 &\le \norm{J A B - A J B}_2 + \norm{A J B - A B J}_2 \notag \\
 &\le \norm{[J, A]}_2 \norm{B} + \norm{A} \norm{[J, B]}_2 < \infty \text{ and } \notag \\
\norm{A B} + \norm{[J, A B]}_2 &\le \norm{A} \norm{B} + \norm{[J, A]}_2 \norm{B} + \norm{A} \norm{[J, B]}_2 \notag \\
 &\le (\norm{A} + \norm{[J, A]}_2 ) (\norm{B} + \norm{[J, B]}_2 ). \notag
\end{align}

$\B_{res, J}$ is complete as follows.  If $A_i \in \B_{res, J}$ is a Cauchy sequence for $\norm{}_J$, then it's also a Cauchy sequence in $\B(V)$ for $\norm{}$.  Since $\B(V)$ with the operator norm is a Banach algebra, there is some $A \in \B(V)$ such that $\norm{A_i - A} \rightarrow 0$.  Again, $[J, A_i]$ is a Cauchy sequence in the Hilbert space (see lemma \ref{l-hsop}) of Hilbert-Schmidt operators on $V$, an ideal in $\B(V)$, so there is some Hilbert-Schmidt operator $B \in \B(V)$ such that $\norm{[J, A_i] - B}_2 \rightarrow 0$.  Then $\norm{[J, A_i] - B} \le \norm{[J, A_i] - B}_2 \Rightarrow \norm{[J, A_i] - B} \rightarrow 0$.  But $\norm{[J, A_i] - [J, A]} \rightarrow 0$, so $B = [J, A]$, $\norm{[J,A]}_2 = \norm{B}_2 < \infty$, $A \in \B_{res, J}$, and $\norm{A_i - A}_J \rightarrow 0$.  Thus $\B_{res, J}$ is a Banach algebra.

Suppose $J_1, J_2 \in \US_{res}$; i.e., $\norm{J_2 - J_1}_2 < \infty$.  Then
\begin{align}
\norm{[J_2, A]}_2 &= \norm{J_2 A - A J_2}_2 \notag \\
 &\le \norm{J_2 A - J_1 A}_2 + \norm{J_1 A - A J_1}_2 + \norm{A J_1 - A J_2}_2 \notag \\
 &\le 2 \norm{A} \norm{J_2 - J_1}_2 + \norm{[J_1, A]}_2, \text{ so} \notag \\
\norm{A}_{J_2} &\le (1 + 2 \norm{J_2 - J_1}_2) \norm{A}_{J_1}. \notag
\end{align}
The roles of $J_1, J_2$ may be reversed.  Thus $\norm{[J_1, A]}_2 < \infty \Leftrightarrow \norm{[J_2, A]}_2 < \infty$, so the sets $\B_{res, J_1}, \B_{res, J_2}$ are equal.  The norms of $\B_{res, J_1}, \B_{res, J_2}$ are equivalent.

$\Orth_{res}$ is a closed subset of $\B_{res, J}$, for if $x_i \in \Orth_{res}$ and $x_i \rightarrow x \in \B_{res, J}$, using the operator norm part of $\norm{}_J$ we have $x_i \rightarrow x \in \B(V)$.  Since $\Orth(V)$ is closed in $\B(V)$, $x \in \Orth(V)$ and hence $x \in \Orth_{res}$.

Since $J^{*} = -J$, for $A \in \B_{res, J}$, $\norm{[J, A^{*}]}_2 = \norm{[J, A]^{*}}_2 = \norm{[J, A]}_2$ by lemma \ref{l-hsop}, and for $A \in \Orth(V)$, $A^{-1} = A^{*}$, the inverse operation in the group $\Orth_{res}$ is continuous.  The multiplication operation also is continuous, as it is in any Banach algebra.  Thus $\Orth_{res}$ is a topological group.

Referring to definition \ref{d-ores-ures-lagr-res}, since for $g \in \UU(V_J)$, $[g, J] = 0$, $\UU(V_J) \subset \Orth_{res}$, and is closed by reasoning similar to that showing $\Orth_{res}$ closed in $\B_{res, J}$.
\end{proof}

Note that the subspace topology of $\UU(V_J) \subset \Orth_{res}$ itself is the same as its operator norm topology, since its elements commute with $J$.

\begin{cor}\label{co-ores-bair-acts-join-cont}
\index{Baire space!Ores@$\Orth_{res}$}
\index{restricted group!orthogonal!topology}
(The $\Orth_{res}$ Action Separate Continuity implies Joint Continuity).
Separately continuous actions of $\Orth_{res}$ on Hilbert, Banach, \Frechet, or other metric spaces (e.g. $\F(L)$) are jointly continuous by lemma \ref{l-grp-act-sep-join-cont}, since $\Orth_{res}$ is a closed subset of the Banach algebra $\B_{res, J}$.
\end{cor}

We will give $\US_{res}$ a topology different from what it would have as a subspace of $\US$, deriving instead from the topology of $\Orth_{res}$ in proposition \ref{p-ores-top-grp}.  Then we will induce from the new topology on $\US_{res}$, the topology of $\Lagr_{res}$, by requiring that the restriction of the bijection of proposition \ref{p-bij-usv-lagh} be a homeomorphism.  In the following, the restrictions of bijections have images as implied, because of the use of the same bijections in definition \ref{d-ores-ures-lagr-res}.
\begin{defn}\label{d-ores-uvj-homo}
\index{homogeneous space}
\index{Ores/UVJ@$\Orth_{res}/\UU(V_J)$}
\index{USres@$\US_{res}$!homogeneous space}
\index{Lagrres@$\Lagr_{res}$!homogeneous space}
($\Orth_{res}/\UU(V_J)$ Homogeneous Space Topologies for $\Lagr_{res}$, $\US_{res}$).
Suppose given $J \in \US_{res}$.  Denote by $\pi_{\Orth_{res}/\UU(V_J)} \colon \Orth_{res} \rightarrow \Orth_{res}/\UU(V_J)$ the quotient map, with codomain given the quotient topology.  Since the codomain, as a set, is contained in $\Orth(V) / \UU(V_J)$, restrict to it the equivariant bijection of proposition \ref{p-uv-ov}, considered just as a map of sets, to obtain the map
\[
\phi_{\Orth_{res}/\UU(V_J), \US_{res}} \colon \Orth_{res}/\UU(V_J) \rightarrow \US_{res}, \notag
\]
which for $g \in \Orth_{res}$, maps $g \UU(V_J) \mapsto g J g^{-1} \in \US_{res}$.  Give $\US_{res}$ the topology that makes this map a homeomorphism.

Denote by $\phi_{\US_{res}, \Lagr_{res}} \colon \US_{res} \rightarrow \Lagr_{res}$ the restriction of the bijection of proposition \ref{p-bij-usv-lagh}, which maps $J_1 \mapsto +i \text{ eigenspace of } J_1$, not depending on $J$.  Give $\Lagr_{res}$ the topology that makes this a homeomorphism.

Summarizing for convenient reference, with notation for the composition:
\begin{align}
\pi_{\Orth_{res}/\UU(V_J)} \colon \Orth_{res} &\rightarrow \Orth_{res}/\UU(V_J) \notag \\
\phi_{\Orth_{res}/\UU(V_J), \US_{res}} \colon \Orth_{res}/\UU(V_J) &\rightarrow \US_{res} \notag \\
g \UU(V_J) &\mapsto g J g^{-1} \notag \\
\phi_{\US_{res}, \Lagr_{res}} \colon \US_{res} &\rightarrow \Lagr_{res} \notag \\
J_1 &\mapsto +i \text{ eigenspace of } J_1 \notag \\
\phi_{\Orth_{res}/\UU(V_J), \Lagr_{res}} &= \phi_{\US_{res}, \Lagr_{res}} \circ \phi_{\Orth_{res}/\UU(V_J), \US_{res}}. \notag
\end{align}
\end{defn}
Some facts that will be used frequently later are as follows, given $J \in \US_{res}$ with corresponding $L \in \Lagr_{res}$.  For $J_1 \in \US(V)$, corresponding $L_1 \in \Lagr(H, \Sigma)$, $J_1 \in \US_{res}$ $\Leftrightarrow$ $L_1 \in \Lagr_{res}$ $\Leftrightarrow$ $g \in \Orth(V)$ such that $L_1 = g L$, is in $\Orth_{res}$ $\Leftrightarrow$ the Fock representation on $\F(L_1)$ is unitarily equivalent to that on $\F(L)$, by corollary \ref{co-equi-soln}.

The definition might lead one to expect that the choice of $J$ might influence the topology on $\US_{res}$ defined by that of $\Orth_{res}/\UU(V_J)$, and hence the topology on $\Lagr_{res}$; but this is not so.
\begin{lem}\label{l-ures-top-not-dep-j}
\index{USres@$\US_{res}$!homogeneous space}
(The Topologies of $\Lagr_{res}$, $\US_{res}$ Depend Only on Polarization Class).
Given $J, K \in \US_{res}(V)$, the composition of bijections
\[
\phi_{\Orth_{res}/\UU(V_K), \US_{res}}^{-1} \circ \phi_{\Orth_{res}/\UU(V_J), \US_{res}} \colon \Orth_{res}/\UU(V_J) \rightarrow \Orth_{res}/\UU(V_K) \notag
\]
is a homeomorphism.
\end{lem}
\begin{proof}
This is the same bijection as in the second part of proposition \ref{p-uv-ov}, and can be justified as there, by \citet[page~5]{tomD87}.  Each direction descends from a right translation of $\Orth_{res}$ with certain properties.
\end{proof}

The action of a topological group on its homogeneous space is continuous (see \citet[pages~2,3]{tomD87}).  For future reference, we record the following facts related to definition \ref{d-ores-uvj-homo}.
\begin{lem}\label{l-ores-uvj-homo}
\index{homogeneous space}
\index{Ores/UVJ@$\Orth_{res}/\UU(V_J)$}
\index{USres@$\US_{res}$!homogeneous space}
\index{Lagrres@$\Lagr_{res}$!homogeneous space}
(The Action of $\Orth_{res}$ on $\Orth_{res}/\UU(V_J)$ is Continuous).
The action of $\Orth_{res}$ by left translation on $\Orth_{res}/\UU(V_J)$ is continuous. The homeomorphisms between $\Orth_{res}/\UU(V_J)$, $\Lagr_{res}$, and $\US_{res}$ are $\Orth_{res}$-equivariant, with actions by left translation for the first two, and by conjugation for the third.
\end{lem}

We will use later the local section of the projection $\pi_{\Orth_{res}/\UU(V_J)}$, more precisely given instead with domain $\US_{res}$, a section of $\phi_{\Orth_{res}/\UU(V_J), \US_{res}} \circ \pi_{\Orth_{res}/\UU(V_J)}$.
\begin{lem}\label{l-uv-ov-sect-res}
\index{USres@$\US_{res}$!homogeneous space!local sections}
(Given $J$, there is a Canonical $g$ for $J_1$ Near $J \in \US_{res}$).
\citep[page~102]{PR94} Suppose given $J \in \US_{res}$. There is an open neighborhood of $J$, $V_1 \subset \US_{res}$, and a continuous function $\xi \colon V_1 \rightarrow \Orth_{res}$, such that for every $J_1 \in V_1$, $J_1 = \xi(J_1) J (\xi(J_1))^{-1}$, and $\xi(J) = \ident$.  For $J_1 \in V_1$, $\norm{J_1 - J} < 2$ in the operator norm on $\US(V)$, considering $\US_{res} \subset \US(V)$ as a set.
\end{lem}
\begin{proof}
We will use proposition \ref{p-uv-ov-cano-g}, which requires $\norm{J_1 - J} < 2$ using the operator norm on $\US(V)$.  However, the desired neighborhood is in the topology $\US_{res}$ has as a quotient of $\Orth_{res}$.  Looking at the following chain of relations, for $J, J_1 \in \US_{res} \subset \US(V)$, with $J_1 = g J g^{-1}$ for some $g \in \Orth_{res}$,
\begin{align}
\norm{J_1 - J} &= \norm{g J g^{-1} - J} = \norm{(g J - J g)g^{-1}} = \norm{g J - J g} \le \norm{[g, J]}_2 \notag \\
 &\le \norm{g - \ident} + \norm{[g, J]}_2 = \norm{g - \ident} + \norm{[J, g - \ident]}_2 = \norm{g - \ident}_J, \text{ so} \notag \\
g \in \B_{2} (\ident) &\Rightarrow \norm{J_1 - J} < 2, \notag
\notag
\end{align}
where $\B_{2} (\ident) \subset \Orth_{res}$ is defined with $\norm{}_J$.  Since the quotient map of a homogeneous space is open \citep[page~22]{tomD87}, $V_1 = \pi_{\Orth_{res}/\UU(V_J)} (B_2 (\ident))$ is open.  It's true that $\pi_{\Orth_{res}/\UU(V_J)}^{-1} (V_1) = B_2 (\ident) \UU(V_J)$, but for $u \in \UU(V_J)$, $\norm{(g u) J (gu)^{-1} - J} = \norm{g J g^{-1} - J}$, since $u$ commutes with $J$. Thus $J_1 \in V_1 \Rightarrow \norm{J_1 - J} < 2$.  Now apply proposition \ref{p-uv-ov-cano-g} to obtain a canonical $g \in \Orth(V)$ depending continuously on $J_1$, such that $J_1 = g J g^{-1}$.
\end{proof}

\section{Skew Symmetry, Quadratic Exponentials, Operator Decomposition}\label{s-skew-symm-quad-exp-op-dcmp}

Now we will explore a connection between orthogonal and skew symmetric operators.  Lemma \ref{l-map-lagr-symv} has uses later, as does the following preparation for it.
\begin{defn}\label{d-sl}
\index{anti-skew}
\index{skew symmetric!space}
\index{SSL@$\Sym(\Sigma(L))$}
($\Sym(\Sigma(L))$).
Given $L \in \Lagr(H)$, with corresponding $J \in \US(V)$, let $\Sym(\Sigma(L))$ denote the subset of $\B_2(H)$, the Hilbert-Schmidt operators on $H$, that satisfy the following properties, for $Z \in \Sym(\Sigma(L))$:
\begin{align}
\Sigma Z &= Z \Sigma \notag \\
Z^{*} &= -Z \notag \\
J Z &= - Z J. \notag
\end{align}
\end{defn}
\begin{note}\label{n-anti-skew}
\index{anti-skew}
(An Alternative Definition of $\Sym(\Sigma(L))$).
For $Z \in \B(H)$, the third property implies that $Z \colon \Sigma(L) \rightarrow L$.  We could define $\Sym(\Sigma(L))$ as a space of Hilbert-Schmidt operators $\Sigma(L) \rightarrow L$, touching on the tensor product $L \tensor L$ as in \citet[page~269]{Fabe00}, though it's more convenient for us later on to consider operators $H \rightarrow H$.
\end{note}

The following relates inner products and the norms of operators in $\Sym(\Sigma(L))$ to whether they are considered as operators on $H$, $V$, or $\Sigma(L)$.
\begin{lem}\label{l-ssl-cx-bana}
\index{Hilbert-Schmidt!operator}
\index{Hilbert-Schmidt!inner product}
\index{Hilbert-Schmidt!norm}
(Inner Products and Norms on $\Sym(\Sigma(L))$).
Given $L \in \Lagr(H)$, for $Z, Z_1, Z_2 \in \B(H)$ satisfying the properties of definition \ref{d-sl}, $\langle Z_1, Z_2 \rangle_2 = \langle (Z_1)_{|V}, (Z_2)_{|V} \rangle_2 = 2 \thinspace \Re( \langle (Z_1)_{|\Sigma(L)}, (Z_2)_{|\Sigma(L)} \rangle_2)$.  Thus $\norm{Z}_2 = \norm{Z_{|V}}_2 = \sqrt{2} \thinspace \norm{Z_{|\Sigma{L}}}_2$; and $Z \in \B_2(H) \Leftrightarrow Z_{|V} \in \B_2(V) \Leftrightarrow Z_{|\Sigma(L)} \in B_2(\Sigma(L), L)$.
\end{lem}
\begin{proof}
Orthonormal bases $\{ l_k \}$ for $L$ and $\{ \Sigma(l_k) \}$ for $\Sigma(L)$ together give an orthonormal basis $\{ \frac{1}{\sqrt{2}} (l_k + \Sigma(l_k)) \}$ for $V$ and an orthonormal basis $\{ l_k \} \cup \{ \Sigma{l_k} \}$ for $H$.  Then $\langle Z_1 (l_k + \overline{l_k}), Z_2 (l_k + \overline{l_k}) \rangle = \langle Z_1 l_k, Z_2 l_k \rangle + \langle Z_1 \overline{l_k}, Z_2 \overline{l_k} \rangle + \langle Z_1 \overline{l_k}, Z_2 l_k \rangle + \langle Z_1 l_k, Z_2 \overline{l_k} \rangle = \langle Z_1 l_k, Z_2 l_k \rangle + \langle Z_1 \overline{l_k}, Z_2 \overline{l_k} \rangle = 2 \thinspace \Re( \langle Z_1 \overline{l_k}, Z_2 \overline{l_k} \rangle )$, since $L \perp \overline L$.  Thus the inner product of $Z_1$ and $Z_2$ considered as operators on $V$, equals their inner product considered as operators on $H$, and is is twice the real part of their inner product considered as operators on $\Sigma(L)$.  The statements about norms follow.
\end{proof}

\begin{lem}\label{l-sl-alt}
\index{anti-skew}
\index{skew symmetric!space}
\index{skew symmetric!operators}
\index{SSL@$\Sym(\Sigma(L))$}
(An Alternative Characterization of $\Sym(\Sigma(L))$).
Given $L \in \Lagr(H)$, the restriction to $\Sigma(L)$ of $Z \in \Sym(\Sigma(L))$ satisfies the following properties:
\begin{align}
Z &\colon \Sigma(L) \rightarrow L \notag \\
Z &\in \B_2 (\Sigma(L), L) \notag \\
\forall x, y \in L, \quad &\langle Z \Sigma x, y \rangle = - \langle Z \Sigma y, x \rangle. \notag
\end{align}
Conversely, given $Z \colon \Sigma(L) \rightarrow L$ satisfying these properties, its complex-linear extension to $H$ by extending it to $L$ as $\Sigma Z \Sigma$, is in $\Sym(\Sigma(L))$.
\end{lem}
\begin{proof}
If $Z \in \Sym(\Sigma(L))$, the third property holds because, for $x, y \in L$, $\langle Z \Sigma x, y \rangle = \langle \Sigma x, -Z y \rangle = -\langle Z \Sigma y, x \rangle$.  If $Z \colon \Sigma(L) \rightarrow L$ satisfies these properties, its given complex-linear extension to $H$, which we will still call $Z$, satisfies $Z^{*} = -Z$ because $\langle Z \Sigma x, y \rangle = -\langle Z \Sigma y, x \rangle = \langle \Sigma x, -Z y \rangle$, and similarly for $x, y \in \Sigma(L)$.
\end{proof}

\citep[page~65]{PR94} define $\Sym(V_J)$ somewhat differently, requiring that elements of $V_J$ be Hilbert-Schmidt as operators on the real vector space $V$, and use the complex Hilbert space $V_J$ with Hermitian inner product $\langle , \rangle_J$.  The maps in our $\Sym(\Sigma(L))$ are simply complex-linear extensions to $H$ of theirs, though the subsequent restrictions to maps $\Sigma(L) \rightarrow L$ can be constructed from $\Sigma$ and the isometric isomorphism $\phi \colon V_J \rightarrow L$ of proposition \ref{p-vj-hp-isom}.  Their $\norm{Z}_{HS}$ equals twice our $\norm{Z_{\Sigma(L)}}_2$, since their basis for the real vector space $V$ is a real basis $\{ v_k \} \cup \{ J v_k\}$ where the first set by itself is a complex basis of $V_J$ and so corresponds via $\phi$ to a complex basis such as $\{ l_k \}$ of $L$ (similarly for $\Sigma(L)$), so that the sum for their Hilbert-Schmidt norm has two summands, equal since $J$ is unitary and anticommutes with elements of $\Sym(V_J)$, for every one of our $\norm{Z_{\Sigma(L)}}_2$.  Thus their $\norm{Z}_{HS} = \norm{Z_{|V}}_2$.

The following result and proof follow \citet[pages~64--66]{PR94} somewhat but go a slightly further in showing that $\Sym(\Sigma(L))$ can be made a complex Hilbert space, isometrically isomorphic to $\Lambda^2 V$.
\begin{prop}\label{p-ssl-cx-bana}
\index{skew symmetric!space!Hilbert space}
\index{quadratic exponentials!skew symmetric space}
\index{SSL@$\Sym(\Sigma(L))$!Hilbert space}
(Properties of $\Sym(\Sigma(L))$).
Given $L \in \Lagr(H)$, with corresponding $J \in \US(V)$,
$\Sym(\Sigma(L))$ is a complex Hilbert space, using the alternative characterization of lemma \ref{l-sl-alt}, with scalar multiplication induced by that on $\Sigma(L)$, and inner product the Hilbert-Schmidt inner product for maps $\Sigma(L) \rightarrow L$.  For $\zeta \in \Lambda^2 L$, $x, y \in L$, the map
\begin{align}
\tau &\colon \Lambda^2 L \rightarrow \Sym(\Sigma(L)) \notag \\
\tau &\colon \zeta \mapsto \frac{1}{\sqrt{2}} Z_{\zeta} \text{ where } Z_{\zeta} \text{ is defined by} \notag \\
\langle \zeta, x \wedge y \rangle &= \langle Z_{\zeta} \Sigma(x), y \rangle, \notag
\end{align}
is an isometric isomorphism of complex Hilbert spaces, using the alternative characterization's inner product on $\Sym(\Sigma(L))$.  The notation $\Lambda^2 L$ indicates the Hilbert space completion.
\end{prop}
\begin{proof}
The symmetry property of lemma \ref{l-sl-alt} is unaffected by multiplying an operator on $\Sigma(L)$ by $i$ (for operators on $H$ we would need to multiply by $J$ instead).  By lemma \ref{l-ssl-cx-bana}, $\Sym(\Sigma(L))$ is a complex Hilbert space using the Hilbert-Schmidt inner product $\langle (Z_1)_{|\Sigma(L)}, (Z_2)_{|\Sigma(L)} \rangle_2$.

The equation defining $\tau$ has on the right side a complex bounded sesquilinear form considered as a function of $(\Sigma x, y) \in \Sigma(L) \cross L$, so by \citet[page~31]{Conw90}, there is a unique $Z_{\zeta} \in B(\Sigma(L), L)$ satisfying the defining equation for all $(\Sigma(x), y) \in \Sigma(L) \cross L$, or equivalently all $(x, y) \in L \cross L$.  Since the right hand side changes sign when $x$ and $y$ are transposed, so does the left.  Thus, if we can show that $Z_{\zeta} \in B_2(\Sigma(L), L)$, it will follow that $Z_{\zeta} \in \Sym(\Sigma(L))$.  We will do this by showing that $\tau$, which is a complex-linear function, is an isometry.

To see that $\tau$ is an isometry, suppose $Z_{\zeta_1}, Z_{\zeta_2} \in \Sym(\Sigma(L))$ correspond to $\zeta_1, \zeta_2 \in \Lambda^2 L$, and $\{ l_k \}$ is an orthonormal basis of $L$.  Using the defining equation twice,
\begin{align}
Z_{\zeta_2} \overline{l_j} &= \sum_p \langle Z_{\zeta_2} \overline{l_j}, l_p \rangle l_p \notag \\
 &= \sum_p \langle \zeta_2, l_j \wedge l_p \rangle l_p, \text{ so} \notag \\
\langle \tau(\zeta_1), \tau(\zeta_2) \rangle = \frac{1}{2} \langle Z_{\zeta_1}, Z_{\zeta_2} \rangle &= \frac{1}{2} \sum_j \langle Z_1 \overline{l_j}, Z_2 \overline{l_j} \rangle \rangle \notag \\
 &= \frac{1}{2} \sum_j \langle \zeta_1, l_j \wedge Z_{\zeta_2} \overline{l_j} \rangle \notag \\
 &= \frac{1}{2} \sum_j \langle \zeta_1, l_j \wedge \sum_p \langle \zeta_2, l_j \wedge l_p \rangle l_p \rangle \notag \\
 &= \frac{1}{2} \sum_{j,p} \langle \zeta_1, l_j \wedge l_p \rangle \overline{\langle \zeta_2, l_j \wedge l_p \rangle} \notag \\
 &= \langle \zeta_1, \zeta_2 \rangle, \notag
\end{align}
since $\{ l_j \wedge l_p \st j < p \}$ is an orthonormal basis for $\Lambda^2 L$.  In the last sum, terms with $j = p$ are zero, and terms with $j > p$ equal those with $j$ and $p$ transposed.  Thus $\tau$ is a complex-linear isometry into $\Sym(\Sigma(L))$.

For surjectivity of $\tau$, suppose given $Z \in \Sym(\Sigma(L))$; then each finite sum in $\sqrt{2} \sum_{j < p} \langle Z \overline{l_j}, l_p \rangle l_j \wedge l_p$ has norm squared equal to the corresponding finite sum in $2 \sum_{j < p} \langle Z \overline{l_j}, l_p \rangle \overline{\langle Z \overline{l_j}, l_p \rangle} = \sum_{j,p} \langle Z \overline{l_j}, l_p \rangle \langle l_p, Z \overline{l_j} \rangle = \sum_j \langle Z \overline{l_j}, Z \overline{l_j} \rangle$, which is bounded by the finite $\norm{Z}_2^2$.  Since we have taken $\Lambda^2 L$ to be complete, the first sum converges to some $\zeta \in \Lambda^2 L$, and by continuity the image under $\tau$ of the first sum converges to $\tau(\zeta) \in \Sym(\Sigma(L))$.  By definition of $\tau$, for each $j$, $p$, $\langle \sqrt{2} \tau(\zeta) \overline{l_j}, l_p \rangle = \langle \zeta, l_j \wedge l_p \rangle = \sqrt{2} \langle Z \overline{l_j}, l_p \rangle$, so $\tau(\zeta) = Z$ and $\tau$ is surjective.
\end{proof}

\begin{prop}\label{p-quad-exp}
\index{quadratic exponentials}
(Quadratic Exponentials).
\citep[page~68]{PR94} Given a Lagrangian subspace $L$ of $H$, if $\zeta \in \Lambda^2 L \subset \F(L)$, then the following ``quadratic exponential'' converges and satisfies the inequality:
\[
\exp (\zeta) = \sum_{k \ge 0} \frac{1}{k!} \zeta^k \Rightarrow \norm{\exp(\zeta)}^2 \le \exp(\norm{\zeta}^2). \notag
\]
\end{prop}

\sloppy
\begin{defn}\label{d-g-cgag}
\index{unitary structure!decomposition of an endomorphism}
\index{Cg@$C_g$}
\index{Ag@$A_g$}
(The Unitary Structure Decomposition of an Endomorphism).
\citep[page~92]{PR94} Given a unitary structure $J \in \US(V)$, and any $g \in \B(V)$, let
\begin{align}
C_g   &= \frac{1}{2} (g - J g J) = \frac{J}{2} \{ g, J \} \in \B(V) \notag \\
A_g   &= \frac{1}{2} (g + J g J) = \frac{J}{2} [g, J] \in \B(V), \notag
\end{align}
$J$ implicit in the symbols $C_g$ and $A_g$, where $\{a,b\} = ab + ba$ is the anticommutator and $[a,b] = ab - ba$ the commutator.  We call $A_g$ the antisymmetric and $C_g$ the symmetric parts of $g$, the parts that anticommute and commute with $J$, and give the same names to the complex-linear extensions to $H$ of $g$, $A_g$, and $C_g$.
\end{defn}
\fussy

\begin{lem}\label{l-g-cgag}
\index{unitary structure!decomposition of an endomorphism}
\index{Cg@$C_g$}
\index{Ag@$A_g$}
(Properties of the Unitary Structure Decomposition of an Endomorphism).
\citep[pages~92--94]{PR94} Given a unitary structure $J \in \US(V)$, any $g \in \B(V)$, optionally complex linearly extended to $H$, in which case we consider $A_g$ and $C_g$ also as their complex-linear extensions to $H$, and $u \in \UU(V_J)$,
\begin{align}
g       &= C_g + A_g \notag \\
C_g J   &=   J C_g \notag \\
A_g J   &= - J A_g \notag \\
C_{g h} &= C_g C_h + A_g A_h \notag \\
A_{g h} &= C_g A_h + A_g C_h \notag \\
C_{g u} &= C_g \notag \\
A_{g u} &= A_g \notag \\
C_g^{*} &= C_{g^{*}} \notag \\
A_g^{*} &= A_{g^{*}}. \notag
\end{align}
If $g \in \Aut(V)$,
\begin{align}
\ident &= C_{g^{-1}} C_g + A_{g^{-1}} A_g  \notag \\
0      &= C_{g^{-1}} A_g + A_{g^{-1}} C_g.  \notag
\end{align}
$g \in \Aut(V)$ is in $\Orth(V)$ if and only if both of the conditions below are true:
\begin{align}
\ident &= C_g^{*} C_g + A_g^{*} A_g \notag \\
0      &= C_g^{*} A_g + A_g^{*} C_g. \notag
\end{align}
If $g \in \Orth(V)$, then $C_g^{*} = C_{g^{-1}}$ and $A_g^{*} = A_{g^{-1}}$, where the adjoints are with respect to the inner product on $V$ or its Hermitian extension to $H$.
\end{lem}
\begin{proof}
$g \in \Orth(V) \Leftrightarrow g^{*} = g^{-1}$.  The necessary and sufficient condition for $g \in \Aut(V)$ to be in $\Orth(V)$ is equivalent to $C_{g^{*} g} = \ident$ and $A_{g^{*} g} = 0$, which add together to give $g^{*} g = \ident$, so the condition is sufficient.  That it is necessary follows from $C_{\ident} = \ident$, $A_{\ident} = 0$.
\end{proof}

\section{Restricted Orthogonal Neighborhood to Skew Symmetric Operators}\label{s-ores-nhbd-skew-symm-op}

\begin{lem}\label{l-map-lagr-symv}
\index{skew symmetric!operators}
\index{Lagrangian subspace!set of!map to SSL@map to $\Sym(\Sigma(L))$}
\index{SSL@$\Sym(\Sigma(L))$!set of Lagrangian subspaces mapped to}
(Maps from Neighborhoods of $\ident \in \Orth_{res}$ and $L \in \Lagr_{res}$ to $\Sym(\Sigma(L))$).
Given $L \in \Lagr_{res}$ with corresponding $J \in \US_{res}$, there is an open neighborhood $\widetilde{V_2}$ of $\ident \in \Orth_{res}$ on which the map $\widetilde{\sigma} \colon \widetilde{V_2} \rightarrow \Sym(\Sigma(L))$ given by $g \mapsto Z_g = - A_g C_g^{-1}$ is continuous.  For $g_1, g_2 \in \widetilde{V_2}$, $\widetilde{\sigma}(g_1) = \widetilde{\sigma}(g_2) \Leftrightarrow g_1^{-1} g_2 \in \UU(V_J)$.

Furthermore, recalling definition \ref{d-ores-uvj-homo} and letting $V_2$ denote the corresponding neighborhood of $L \in \Lagr_{res}$, there is a homeomorphism $\sigma \colon V_2 \rightarrow \widetilde{\sigma} (\widetilde{V_2}))$ such that the triangle in the following diagram commutes.
\[
\begindc{\commdiag}[5]
\obj(0,25)[objOres]{$\Orth_{res}$}
\obj(10,25)[objWTV2]{$\widetilde{V_2}$}
\obj(25,25)[objImsigma]{$\widetilde{\sigma} (\widetilde{V_2})$}
\obj(40,25)[objSSL]{$\Sym(\Sigma(L))$}
\obj(0,10)[objLagrres]{$\Lagr_{res}$}
\obj(10,10)[objV2]{$V_2$}
\mor{objWTV2}{objOres}{$\incl$}
\mor{objWTV2}{objImsigma}{$\widetilde{\sigma}$}
\mor{objWTV2}{objV2}{}[\atright, \solidarrow]
\mor{objV2}{objLagrres}{$\incl$}
\mor{objV2}{objImsigma}{$\sigma$}[\atleft,\solidarrow]
\mor{objImsigma}{objSSL}{$\incl$}
\cmor((25,28)(18,31)(10,28)) \pleft(18,32){$\widetilde{\sigma}_{r}$}
\enddc
\]
There is a continuous right inverse of $\widetilde{\sigma}$; i.e., $\widetilde{\sigma} \circ \widetilde{\sigma}_r = \ident$, given by
\begin{align}
\widetilde{\sigma}_r \colon \widetilde{\sigma} (\widetilde{V_2}) &\rightarrow \Orth_{res} \notag \\
 z &\mapsto (\ident - z) (\ident + z^{*} z)^{-\frac{1}{2}} \notag
\end{align}
\end{lem}
\begin{proof}
Let $\widetilde{V_2} = \B_{\delta} (\ident) \UU(V_J)$, and let $V_2 = \phi_{\Orth_{res}/\UU(V_J), \Lagr_{res}} \circ \pi_{\Orth_{res}/\UU(V_J)} (\widetilde{V_2})$, an open neighborhood of $L \in \Lagr_{res}$.  Then $\pi_{\Orth_{res}/\UU(V_J)}^{-1} \circ \phi_{\Orth_{res}/\UU(V_J), \Lagr_{res}}^{-1} (V_2) = \widetilde{V_2}$.

To define $\delta > 0$, note that in any Banach algebra $B$, $b \in B$ and $\norm{1 - b} < \beta < 1$ imply that $b$ is invertible, $\norm{b^{-1}} < \frac{1}{1 - \beta}$, and $\norm{1 - b^{-1}} < \frac{\beta}{1 - \beta}$ (see \citet[page~224]{LS68} or \citet[page~250]{Rudi91}).  For all $g \in \Orth_{res} \subset \B_{res, J}$, $C_g  = \frac{1}{2} (g - J g J)$ and $A_g  = \frac{1}{2} (g + J g J)$ are continuous as functions of $g$, since of course $J \in \B_{res, J}$.  Recall that $C_{\ident} = \ident$ and $A_{\ident} = 0$.  Choose $\delta_0 > 0$ such that
\begin{align}
\delta_0 &< 1 \notag \\
g \in \B_{\delta_0} (\ident) &\Rightarrow \norm{\ident - C_g}_J < \frac{1}{2} \text{ and } \norm{A_g}_J < \frac{1}{4}, \text{ and define} \notag \\
\widetilde{V_{2,0}} &= \B_{\delta_0} (\ident) \UU(V_J) \notag \\
V_{2,0} &= \phi_{\Orth_{res}/\UU(V_J), \Lagr_{res}} \circ \pi_{\Orth_{res}/\UU(V_J)} (\widetilde{V_{2,0}}), \notag
\end{align}
whence for such $g$, $C_g$ is invertible and $\norm{C_g^{-1}}_J < 2$.

Take a right translation of $g$ by $u \in \UU(V_J)$.  Decomposing $g u = A_{g u} + C_{g u}$, using the fact that $u \in \UU(V_J)$ commutes with $J$,
\begin{align}
C_{g u} &= \frac{1}{2} (g u - J g u J) \notag \\
 &= \frac{1}{2} (g u - J g J u) \notag \\
 &= \frac{1}{2} (g - J g J) u \notag \\
 &= C_g u, \notag
\end{align}
so
\begin{align}
\norm{C_{gu}}_J &= \norm{C_g u} + \norm{J C_g u - C_g u J}_2 \notag \\
 &= \norm{C_g u} + \norm{J C_g u - C_g J u}_2 \notag \\
 &= \norm{C_g u} + \norm{(J C_g - C_g J) u}_2 \notag \\
 &= \norm{C_g}_J. \notag
\end{align}
Since $C_g u$ is invertible when $C_g$ is, $C_g$ is invertible for all $g \in \widetilde{V_{2,0}}$.  As in any Banach algebra, when $C_g$ is invertible, $C_g^{-1}$ is a continuous function of $C_g$, and thence of $g$.  For all $g \in \widetilde{V_{2,0}}$, $\norm{C_g^{-1}}_J < 2$.

Similarly $A_{g u} = A_g u$, $\norm{A_{g u}}_J = \norm{A_g}_J$, and $\norm{A_g}_J < \frac{1}{4}$ for all $g \in \widetilde{V_{2,0}}$.

The facts given in the proof so far for right translations of $g$ by $u \in \UU(V_J)$ are also true for left translations and both simultaneously.  Thus the phrase ``for all $g \in \widetilde{V_{2,0}}$'' in the last two paragraphs could be replaced by ``for all $g \in \UU(V_J) \widetilde{V_{2,0}}$''.

% ~ removed from citation to cure overfull box
Letting $Z_g = -A_g C_g^{-1}$, a calculation \citep[pages 93, 105]{PR94} shows that $Z_g \in \Sym(\Sigma(L))$.  To wit, for $x, y \in H$, using $A_g^{*} = A_{g^{-1}}$, $C_g^{*} = C_{g^{-1}}$ and another of the statements of lemma \ref{l-g-cgag}, that we will transform to start with:
\begin{align}
0 &= C_{g^{-1}} A_g + A_{g^{-1}} C_g \Rightarrow \notag \\
0 &= A_g C_g^{-1} + C_{g^{-1}}^{-1} A_{g^{-1}}, \text{ so} \notag \\
\langle Z_g x, y \rangle &= -\langle A_g C_g^{-1} x, y \rangle \notag \\
 &= -\langle x, C_{g^{-1}}^{-1} A_{g^{-1}} y \rangle \notag \\
 &=  \langle x, A_g C_g^{-1} y \rangle \notag \\
 &= -\langle x, Z_g y \rangle. \notag
\end{align}

$Z_g$ is continuous using $\norm{}_2$, as a function of $g$ using norm $\norm{}_J$, for fixing $g \in \widetilde{V_{2,0}}$ for the moment and supposing $g' \in \widetilde{V_{2,0}}$,
\begin{align}
\norm{Z_{g'} - Z_g}_2 &= \norm{ -A_{g'} C_{g'}^{-1} + A_g C_g^{-1}}_2 \notag \\
 &\le \norm{ -A_{g'} C_{g'}^{-1} + A_{g'} C_g^{-1}}_2 + \norm{ -A_{g'} C_g^{-1} + A_g C_g^{-1}}_2 \notag \\
 &\le \norm{A_{g'}}_2 \norm{-C_{g'}^{-1} + C_g^{-1}} + \norm{-A_{g'} + A_g}_2 \norm{C_g^{-1}}. \notag
\end{align}
By making $\norm{g' - g}_J$ small enough, because $C_g^{-1}$ and $A_g$ are continuous functions of $g$, we can ensure that $\norm{-C_{g'}^{-1} + C_g^{-1}} \le \norm{-C_{g'}^{-1} + C_g^{-1}}_J$ is as small as desired, $\norm{A_g'}_2 \le \frac{1}{2} \norm{g'}_J$ is bounded, $\norm{-A_{g'} + A_g}_2 \le \frac{1}{2} \norm{-g' + g}_J$ is as small as desired; and $\norm{C_g^{-1}}$ is constant.  Thus $Z_g$ is a continuous function of $g$.

To prove the next item in the statement, if we now supposed that $g_1, g_2 \in \widetilde{V_{2,0}}$ and $\widetilde{\sigma}(g_1) = \widetilde{\sigma}(g_2)$, we would run into a problem. To avoid excessive subscripts let $g = g_1$, $g' = g_2$, and $h = g^{-1} g'$, so that $g' = g h$.  Then looking ahead,
\[
f, f' \in \B_{\delta_0}(\ident) \Rightarrow f^{-1} f' \in \B_{\frac{2 \delta_0}{1 - \delta_0}}(\ident) \notag
\]
because supposing $f = \ident + f_{\delta_0}$ and $f' = \ident + {f'}_{\delta_0}$, with $\norm{f_{\delta_0}}_J < \delta_0$ and $\norm{{f'}_{\delta_0}}_J < \delta_0$, then $f^{-1} = \ident + (f^{-1})_{\frac{\delta_0}{1 - \delta_0}}$ with $\norm{(f^{-1})_{\frac{\delta_0}{1 - \delta_0}}}_J < \frac{\delta_0}{1 - \delta_0}$, and
\[
\norm{\ident - f^{-1} f'}_J = \norm{\ident - (\ident + (f^{-1})_{\frac{\delta_0}{1 - \delta_0}}) (\ident + {f'}_{\delta_0})}_J \le \delta_0 + \frac{\delta_0}{1 - \delta_0} + \frac{\delta_0^2}{1 - \delta_0} = \frac{2 \delta_0}{1 - \delta_0}. \notag
\]
We will want $\delta$ small enough that products and inverses of elements of $\B_{\delta} (\ident)$ will be in $B_{\delta_0}(\ident)$, so set
\[
\delta < \frac{\delta_0}{2 + \delta_0}, \text{ whence } \frac{2 \delta}{1 - \delta} < \delta_0, \notag
\]
and define $\widetilde{V_2}$, $V_2$ as in the lemma statement.  Now suppose that $g, g' \in \widetilde{V_2}$, and $\widetilde{\sigma}(g) = \widetilde{\sigma}(g')$, or in other words, $Z_{g} = Z_{g'}$.  Take $f \in (g \UU(V_J)) \cap \B_{\delta}(\ident)$ and $f' \in (g' \UU(V_J)) \cap \B_{\delta}(\ident)$, so that $g = f u$ and $g' = f' u'$ for some $u, u' \in \UU(V_J)$.  Then
\[
h = g^{-1}g' = (f u)^{-1} (f' u') = u^{-1} f^{-1} f u' \in \UU(V_J) B_{\delta_0}(\ident) \UU(V_J) \notag
\]
and $Z_h$ exists.  To answer when $Z_{g h}$ equals $Z_g$, using lemma \ref{l-g-cgag} we have
\begin{align}
Z_{g h} &= -(C_g A_h + A_g C_h) (C_g C_h + A_g A_h)^{-1} \notag \\
 &= -(C_g A_h + A_g C_h) (C_h^{-1} C_g^{-1}) (\ident + A_g A_h C_h^{-1} C_g^{-1})^{-1} \notag \\
 &= (C_g Z_h C_g^{-1} + Z_g) (\ident - A_g Z_h C_g^{-1})^{-1}, \notag
\end{align}
so
\begin{align}
Z_g = Z_{g h} &\Leftrightarrow Z_g (\ident - A_g Z_h C_g^{-1}) = C_g Z_h C_g^{-1} + Z_g \notag \\
 &\Leftrightarrow - Z_g A_g Z_h C_g^{-1} = C_g Z_h C_g^{-1} \notag \\
 &\Leftrightarrow Z_h = - C_g^{-1} Z_g A_g Z_h \notag \\
 &\Leftrightarrow 0 = Z_h (\ident + C_g^{-1} Z_g A_g) = Z_h (\ident - (C_g^{-1} A_g)^2). \notag
\end{align}
For $g \in \widetilde{V_2}$, because $\B_{res, J}$ is a Banach algebra, $\norm{C_g^{-1} A_g}_J < \frac{1}{2}$.  Thus $\ident - (C_g^{-1} A_g)^2$ is invertible, and $Z_{g h} = Z_g \Leftrightarrow Z_h = 0 \Leftrightarrow g^{-1} g' = h \in \UU(V_J)$.

Continuing now to prove the second paragraph of the statement, $\pi_{\Orth_{res}/\UU(V_J)}$ is an open map, $V_2$ is open, and it contains $L$, which corresponds to $\ident_{\Orth_{res}}$.
Given $g \in \widetilde{V_2}$, $K = g L$, we define $\sigma(K)$ by choosing any $g' \in \Orth_{res}$ such that $K = g' L$, and letting $\sigma(K) = \widetilde{\sigma}(g') = Z_{g'}$.  Since $g' L = K = g L$, $g' = g u$ for some $u \in \UU(V_J)$, and $Z_{g'} = -A_g u (C_g u)^{-1} = Z_g$.  Since $\widetilde{\sigma}$ is $\UU(V_J)$-equivariant for the right action on $\Orth_{res}$ and the trivial action on its image, and a continuous equivariant map descends to a continuous map of the orbit spaces \citep[page~4]{tomD87}, $Z_g$ is a continuous function of $[g]$, and hence of $K$.

Since $\widetilde{\sigma}(g_1) = \widetilde{\sigma}(g_2) \Leftrightarrow g_1^{-1} g_2 \in \UU(V_J) \Leftrightarrow g_1 L = g_2 L$, it follows that $\sigma$ is a bijection onto its image, which because of the definition of $\sigma$ from $\widetilde{\sigma}$, equals the image of the latter.  The facts about $\widetilde{\sigma}_r$ will help show that the inverse of $\sigma$ is continuous.

Using \citet[page~92]{Pede89}, which shows the existence of the square root of a positive self-adjoint operator on a real (or complex) Hilbert space, noting that the limit of a uniformly convergent sequence of continuous functions is continuous, $c \colon \B(V) \rightarrow \B(V)$ defined as $c \colon z \mapsto (\ident - z) (\ident + z^{*} z)^{-\frac{1}{2}}$ is continuous    in the operator norm.  Since $\norm{z} \le \norm{z}_2$, $c$ is also continuous from $\Sym(V)$ with $\norm{}_2$, to $\B(V)$ with $\norm{}$.  The values of $c$ are in $\B_{res, J}$ because $J$ anticommutes with $z$ and both are skew-adjoint, so $J$ commutes with $z^{*} z$ and hence with $(\ident + z^{*} z)^{-\frac{1}{2}}$, and thus
\begin{align}
\norm{[J, \widetilde{\sigma}_r (z)]}_2 &= \norm{[J, (\ident - z) (\ident + z^{*} z)^{-\frac{1}{2}}]}_2 \notag \\
 &= \norm{[J, (\ident - z)] (\ident + z^{*} z)^{-\frac{1}{2}}}_2 \notag \\
 &= \norm{[J, z] (\ident + z^{*} z)^{-\frac{1}{2}}}_2 \notag \\
 &= \norm{2 J z (\ident + z^{*} z)^{-\frac{1}{2}}}_2 \notag \\
 &\le 2 \norm{z}_2 \norm{(\ident + z^{*} z)^{-\frac{1}{2}}}. \label{e-norm-j-sigma-2}
\end{align}
When restricted to $\widetilde{\sigma} (\widetilde{V_2})$ to give $\widetilde{\sigma}_r$, the values of $c$ are in $\Orth_{res}$, as follows. With the restriction of domain, $z = Z_g = - A_g C_g^{-1}$ for some $g \in \widetilde{V_2}$.  Let $w = (\ident - z) (\ident + z^{*} z)^{-\frac{1}{2}} \in \B_{res, J}$.  Since $(\ident + z^{*} z)^{\frac{1}{2}}$ commutes with $J$ it's in $\B_{res, J}$, and thus $w$ is invertible in $\B_{res, J}$ when $\ident - z$ is, which will be true if $\norm{z}_J < 1$; but the definition of $\delta$ ensures that $\norm{z}_J = \norm{A_g C_g^{-1}}_J < \frac{1}{2}$.  Thus we can apply lemma \ref{l-g-cgag}:
\begin{align}
C_w &= (\ident + z^{*} z)^{-\frac{1}{2}} \notag \\
A_w &= - z (\ident + z^{*} z)^{-\frac{1}{2}} \notag \\
C_w^{*} C_w + A_w^{*} A_w &= (\ident + z^{*} z)^{-\frac{1}{2}} (\ident + (z^{*} z)) (\ident + z^{*} z)^{-\frac{1}{2}} = \ident \notag \\
C_w^{*} A_w + A_w^{*} C_w &= - (\ident + z^{*} z)^{-\frac{1}{2}} (z + z^{*}) (\ident + z^{*} z)^{-\frac{1}{2}} = 0, \notag
\end{align}
whence $w \in \Orth_{res}$.  For $z \in \widetilde{\sigma}(\widetilde{V_2})$, $Z_w = - A_w C_w^{-1} = z$, so $\widetilde{\sigma} (\widetilde{\sigma}_r (z)) = z$.

$c$ is continuous from $\Sym(V)$ to $\B_{res, J}$: estimating as in \ref{e-norm-j-sigma-2}, for fixed $z_1$, and $z_2$ in a bounded neighborhood of $z_1$, there are constants $c_1, c_2 \ge 0$ such that
\begin{align}
\norm{[J, \widetilde{\sigma}_r (z_2) - \widetilde{\sigma}_r (z_1)]}_2 \le &2 \norm{z_2}_2 \norm{(1 + z_2^{*} z_2)^{-\frac{1}{2}} - (1 + z_1^{*} z_1)^{-\frac{1}{2}}} \notag \\
 + &2 \norm{z_2 - z_1}_2 \norm{(1 + z_1^{*} z_1)^{-\frac{1}{2}}} \notag \\
 \le &c_2 \norm{(1 + z_2^{*} z_2)^{-\frac{1}{2}} - (1 + z_1^{*} z_1)^{-\frac{1}{2}}} + c_1 \norm{z_2 - z_1}_2. \notag
\end{align}

$\sigma^{-1}$ is continuous, as the composition of continuous maps $\phi_{\Orth_{res}/\UU(V_J), \Lagr_{res}} \circ \pi_{\Orth_{res}/\UU(V_J) | \widetilde{V_2}} \circ \widetilde{\sigma}_{r | \widetilde{\sigma} (\widetilde{V_2})}$ (see the diagram in the lemma statement).
\end{proof}
\begin{note}
\index{skew symmetric!operators!graphs}
(Notes on $\sigma$ and on Proof).
The thesis doesn't use $\sigma$, but it may help in understanding the situation of the diagram for the lemma, since its inverse is a restriction of the correspondence from skew-symmetric operators mapping $L \rightarrow \Sigma(L)$, to their graphs thought of as subsets of $L \dirsum \Sigma(L)$, which are Lagrangian subspaces of $H$.

It might be possible to choose a larger $\delta$ so as not to ''waste'' as much of the $V_1$ of lemma \ref{l-uv-ov-sect-res} as we did.
\end{note}

\section{Restricted Orthogonal Group Projective Representation}\label{s-ores-grou-rep}

Although our goal in the thesis is concerned with representations of algebras, we also make essential use of some facts related to group representations, because of $\Orth_{res}$.  This section is written as though the end were the projective representation of $\Orth_{res}$, although that representation itself is not used elsewhere in the thesis; only the two lemmas leading up to it are used.
\begin{defn}\label{d-prj-unit-rep}
\index{representation!projective}
\index{group!projective representation}
\index{projective representation}
(Projective Unitary Representations of Groups).
\citep[page~833]{FD88b} A projective unitary representation of a topological group $G$ on a Hilbert space $K$ is a continuous homomorphism $\rho \colon G \rightarrow \PU(K)$, with the strong operator topology on $\PU(K)$.  The topological group $\PU(K)$ is defined as the quotient of $\UU(K)$ with its strong operator topology, by its closed subgroup $\UU(1)$, denoting the projection by $\pi_{\PU(K)} \colon \UU(K) \rightarrow \UU(K) / \UU(1)$.
\end{defn}

Now, given $L \in \Lagr_{res}$, consider how we might construct a projective representation of $\Orth_{res}$ on $\F(L)$.  You can get an idea where it could come from by considering the Bogoliubov map that goes from orthogonal transformations to automorphisms of the Clifford algebra, which when composed with the Clifford algebra representation, are implemented by unitary automorphisms of $\F(L)$, and recalling lemma \ref{l-impl-homo}, which says that the Bogoliubov map is a homomorphism up to $\UU(1)$ multiples.  Thus we will get into the details of construction of these implementers.

First, some more background on vacuum vectors.
\begin{lem}\label{l-vac-vec-exp-bkg}
\index{creator}
\index{annihilator}
\index{vacuum vector}
\index{quadratic exponential}
\index{exponential construction!vacuum vector}
\index{vacuum vector!exponential construction}
(Background for the Vacuum Vector Construction).
Suppose given $L \in \Lagr_{res}$ with corresponding $J \in \US_{res}$, and $g \in \Orth(V)$ such that $C_g$ is invertible and $A_g$ is Hilbert-Schmidt.  Then $\exp(\tau^{-1}(Z_g))$ is a cyclic $L$-vacuum vector for $\pi_L \circ \theta_g$, where $\tau$ is the isometric isomorphism of proposition \ref{p-ssl-cx-bana}.
\end{lem}
\begin{proof}
Looking at \citet[page~105]{PR94}, their $Z_g$ is the same map as ours considered as an operator on $V$, so their quadratic exponential, transferred by the map induced by the $\phi$ of proposition \ref{p-vj-hp-isom} as in note \ref{n-fock-spac-vj-l-trans} to $\F(L)$, equals ours.  Since $\phi$ also is the intertwiner between their representation and ours, the fact that their quadratic exponential is a cyclic $J$-vacuum vector for their $\pi_J \circ \theta_g$, implies that our quadratic exponential is a cyclic $L$-vacuum vector for $\pi_L \circ \theta_g$.
\end{proof}

\begin{lem}\label{l-vac-vec-exp}
\index{exponential construction!vacuum vector}
\index{vacuum vector!exponential construction}
(The Continuous Exponential Construction of a Vacuum Vector).
Suppose given $L \in \Lagr_{res}$ with corresponding $J \in \US_{res}$.  For every $g \in \widetilde{V_2}$, the open neighborhood of $\ident \in \Orth_{res}$ of lemma \ref{l-map-lagr-symv}, there is a cyclic $L$-vacuum vector for $\pi_L \circ \theta_g$, chosen by a continuous function of $g$.  When $g \in \UU(V_J)$, e.g. $g = \ident$, the vector is $\Omega_L = 1$.
\end{lem}
\begin{proof}
We use $\widetilde{V_2} = \B_{\delta} (\ident) \UU(V_J)$ and the map $\widetilde{\sigma}$ of lemma \ref{l-map-lagr-symv}, the maps $\tau$ and $\exp$ of propositions \ref{p-ssl-cx-bana} and \ref{p-quad-exp}.

We want to show that the map $\exp \circ \tau^{-1} \circ \widetilde{\sigma} \colon \widetilde{V_2} \rightarrow \F(L)$ is continuous, that for every $g \in \widetilde{V_2}$, $\exp (\tau^{-1}(\widetilde{\sigma}(g)))$ is a cyclic $L$-vacuum vector for $\pi_L \circ \theta_{g}$, and that this vector is $\Omega_L = 1$ for $g \in \UU(V_J)$.

Since $\widetilde{\sigma}(g) = Z_g$ is a continuous function of $g$, the isometric isomorphism of complex Banach spaces $\tau \colon \Lambda^2 L \rightarrow \Sym(\Sigma(L))$ of proposition \ref{p-ssl-cx-bana} gives $\tau^{-1}(\widetilde{\sigma}(g)) = \tau^{-1}(Z_g)$ as a continuous function of $g$.  The exponential map $\exp \colon \Lambda^2 L \rightarrow \F(L)$ of proposition \ref{p-quad-exp} is continuous as a consequence of the inequality in that proposition.  Our continuity follows.

Lemma \ref{l-vac-vec-exp-bkg} shows that $\exp(\tau^{-1}(Z_g))$ is a cyclic $L$-vacuum vector for $\pi_L \circ \theta_g$.  We have $g \in \UU(V_J) \Rightarrow Z_g = 0 \Rightarrow \tau^{-1}(Z_g) = 0 \Rightarrow \exp(\tau^{-1}(Z_g)) = 1$.
\end{proof}
Though not used in this thesis, it's possible to do the analogous thing for every $K$ in the open neighborhood $V_2$ of $L \in \Lagr_{res}$ and every $g \in \Orth_{res}$ such that $K = g L$, by using the neighborhood $V_2$ of $L \in \Lagr_{res}$ and the map $\sigma$ of lemma \ref{l-map-lagr-symv}.

The following lemma is important for future use, when we will be given a Lagrangian subspace instead of an orthogonal transformation, as well as to construct the projective representation.
\begin{lem}\label{l-loc-impl}
\index{implementer!continuous choice near idOres(V)@continuous choice near $\ident_{\Orth_{res}}$}
\index{implementer!continuous choice near chosen L@continuous choice near chosen $L$}
(Given $L \in \Lagr_{res}$, there is a Continuous Local Choice of Implementer).
Suppose given $L \in \Lagr_{res}$ with corresponding $J \in \US_{res}$.  There is an open neighborhood $\widetilde{V_2} \subset \Orth_{res}$ of $\ident$, and a function $\widetilde{\eta} \colon \widetilde{V_2} \rightarrow \UU(\F(L))$, continuous with respect to the strong operator topology on $\UU(\F(L))$, such that $U_g = \widetilde{\eta}(g)$ is a unitary implementer of $\theta_{g}$ in $\pi_L$, with $\widetilde{\eta}(\ident) = \ident$.

Further, there is an open neighborhood $V$ of $L$ and a continuous function $\xi \colon V \rightarrow \widetilde{V_2}$, equal to $\ident$ for $K = L$.  Then $\widetilde{\eta} \circ \xi$ gives $U_g$ depending on $K$, with properties as before.
\end{lem}
\begin{proof}
Let $\widetilde{V_2}$ be the neighborhood from lemma \ref{l-map-lagr-symv}.  Then using lemma \ref{l-vac-vec-exp}, there is a cyclic $L$-vacuum vector, depending continuously on $g \in \widetilde{V_2}$, for $\pi_L \circ \theta_{g}$.  In addition, $\pi_L \circ \theta_{g}$ depends continuously on $g$ in the strong operator topology on $\B(\Cl(V), \B(\F(L)))$:  since the topology on $\Orth_{res}$ is even stronger than the operator norm topology, we can factor its action through the action of $\Orth(V)$ with operator norm topology on $\Cl(V)$, to conclude by proposition \ref{p-bogo-map-cont} that for any fixed $a \in \Cl(V)$, $\theta_{g} (a)$, and hence since $\pi_L$ is a ${}^{*}$-morphism, $\pi_L \circ \theta_{g} (a)$, is a continuous function of $g$.  Then we use proposition \ref{p-j-vac-vec-unit-equi} to obtain a unitary implementer of $\theta_{g}$ in $\pi_L$, $U_g \colon \F(L) \rightarrow \F(L)$; i.e., $\pi_L \circ \theta_{g} (a) = U_g \pi_L (a) U_g^{*}$ for every $a \in \Cl(V)$; and $U_g$ depends continuously on $g$ in the strong operator topology for $\UU(\F(L))$.  By note \ref{n-j-vac-vec-unit-equi}, since if $g = \ident$, the vacuum vector is $\Omega_L = 1$ and $\theta_g = \ident$, $U_g = \ident$. Define $\widetilde{\eta}(g) = U_g$.

Let $V = V_2 \cap \phi_{\US_{res}, \Lagr_{res}}(V_1) \subset \Lagr_{res}$, where $V_2 = \phi_{\Orth_{res}/\UU(V_J), \Lagr_{res}} \circ \pi_{\Orth_{res}/\UU(V_J)} (\widetilde{V_2})$ is the open neighborhood of $L$ of lemma \ref{l-map-lagr-symv}, and $V_1$ is the open neighborhood of $J$ of lemma \ref{l-uv-ov-sect-res}.  Let $\xi$ be the restriction of the $\xi$ of lemma \ref{l-uv-ov-sect-res}.
\end{proof}

\begin{prop}\label{p-pu-rep-ores}
\index{representation!projective!Ores@$\Orth_{res}$}
\index{projective representation!Ores@$\Orth_{res}$}
\index{Ores@$\Orth_{res}$!projective representation}
(A Projective Representation of $\Orth_{res}$ on $\F(L)$).
Suppose given $L \in \Lagr_{res}$.  There is a projective unitary representation $\sigma \colon \Orth_{res} \rightarrow \PU(\F(L))$; for $g \in \Orth_{res}$, $\sigma (g) = \pi_{\PU(\F(L)) (U_g)}$, with $U_g$ an implementer in $\pi_L$ of $\theta_g$.
\end{prop}
\begin{proof}
Given $g \in \Orth_{res}$, by theorem \ref{t-impl-soln} let $U_g$ be an implementer in $\pi_L$ of $\theta_g$; by corollary \ref{co-set-impl-u1-tors} $U_g$ is well-defined up to a $\UU(1)$ factor.  By lemma \ref{l-impl-homo} the correspondence $g \mapsto U_g$ is a homomorphism up to $\UU(1)$ factors, from $\Orth_{res} \subset \Orth(V)$ to $\UU(\F(L))$, and if the induced map $\sigma \colon g \mapsto U_g \UU(1)$ is continuous, it will be a projective unitary representation of $\Orth(V)$ (see definition \ref{d-prj-unit-rep}).

$\PU(\F(L))$ has the quotient (by $\UU(1)$) topology of the strong operator topology on $\UU(\F(L))$.  We will show continuity at any $g \in \Orth_{res}$.  Choose some implementer $U_g$.  For $g' \in g \widetilde{V_2}$ of lemma \ref{l-loc-impl}, $\widetilde{\eta} (g^{-1} g')$ is an implementer of $g^{-1} g'$, and is a continuous function of $g'$, with respect to the strong operator topology on $\UU(\F(L))$.  Define a continuous function $\widehat{\sigma} \colon g \widetilde{V_2} \rightarrow \UU(\F(L))$ by $\widehat{\sigma} (g') = U_g \widetilde{\eta} (g^{-1} g')$, which is an implementer of $g'$, and equals $U_g$ when $g' = g$.

Suppose given an open neighborhood $W$ of $\sigma (g) = U_g \UU_1$, and let $V = \pi_{\PU(\F(L))}^{-1} (W)  = U_g \UU_1$ with another meaning.  Then for $g' \in \widehat{\sigma}^{-1} (V)$, an open neighborhood of $g$, any implementer $U_{g'}$ of $g'$ is an element of $\widehat{\sigma}(g') \UU(1) \subset V$, and hence $\sigma (g') = U_{g'} \UU(1) = \widehat{\sigma}(g') \UU(1) \in W$.  That is, $\sigma$ is continuous.
\end{proof}

\section{$L$SO($n$) Continuous Inclusion into Restricted Orthogonal Group}\label{s-lson-cont-incl-ores}

\begin{prop}\label{p-lson-in-ores}
\index{LSOn Ores Continuous@$L\SO(n) \hookrightarrow \Orth_{res}$ Continuous}
($L\SO(n) \rightarrow \Orth_{res}$ is Continuous).
The inclusion $L\SO(n) \hookrightarrow \Orth_{res}$ is continuous; the \Frechet topology is finer than the subspace topology of its image in $\Orth_{res}$.
\end{prop}
\begin{proof}
We have from lemma \ref{l-actn-lson-clrn} the inclusion of $L\SO(n)$ into $\Orth(V)$.  The inclusion into $\Orth_{res}$ is proved briefly in \citet[pages~82--84]{PS86} by considering a particular $L \in \Lagr_{res}$.  We prove it in more detail, following the plan of the first proof of their proposition 6.3.1, using our proposition \ref{p-ojv-ores}, and go on to show that the inclusion is continuous.

To calculate, we use the unitary structure $J \in \US_{res}$ corresponding to a chosen Lagrangian subspace $L \in \Lagr_{res}$ of the form specified in our definition \ref{d-lagr-res}.  Since neither $L\SO(n)$ nor $\Orth_{res}$ (see proposition \ref{p-ores-dep-pol-clas}) nor the inclusion map depend on the choice of $L$, the result obtained using $L$ is independent of that choice.

To see that $\sigma \in L\SO(n)$ is in $\Orth_{res}$, identifying $S^1$ here with $[0, 2 \pi]$ with endpoints identified so that the exponentials in Fourier series have the simplest form, $v \in H$ may be written
\[
v(t) = \sum_{j \in \ZZ} v_j e^{i j t} = \sum_{j \in \ZZ} \sum_{b = 1, 2} v_{j b} e^{i j t}
\]
with $v_j \in \CC^n$, $v \in V \Leftrightarrow \overline{v_j} = v_{-j}$ for all $j$, and $v_j = v_{j 1} + v_{j 2}$ with respect to the decomposition $\CC^n = L_{finite} \dirsum \overline{L_{finite}}$, $b = 1$ corresponding to the $L_{finite}$ summand.  Consider any
\[
\sigma(t) = \sum_{k \in \ZZ} \sigma_k e^{i k t}
\]
where the $\sigma_k$ are $n \cross n$ complex matrices, given in block form as
\begin{align}
\sigma_k =
\left[
\sigma_{k a b}
\right] =
\left[
\begin{matrix}
 \sigma_{k 1 1} & \sigma_{k 1 2} \\
 \sigma_{k 2 1} & \sigma_{k 2 2}
\end{matrix}
\right], \notag
\end{align}
the $\sigma_{k a b}$ being $\frac{n}{2} \cross \frac{n}{2}$ complex matrices, $a, b = 1$ corresponding to $L_{finite}$ and $a, b = 2$ to $\overline{L_{finite}}$.  Then $\sigma$ acts on $v$ as
\[
\sigma (v) (t) = \sum_{k \in \ZZ} \sum_{j \in \ZZ} \sigma_k v_j e^{i (k + j) t} = \sum_{k \in \ZZ} \sum_{j \in \ZZ} \sum_{a = 1, 2} \sum_{b = 1, 2} \sigma_{k a b} v_{j b} e^{i (k + j) t}.
\]
$\sigma \in L\GL(n) \Leftrightarrow \overline{\sigma_k} = \sigma_{-k}$ for all $k$.  To compute $\norm{[J, \sigma]}_2$, to avoid writing needless exponentials, use $L^2 (S^1, \CC^n) \cong l^2 (\CC^n)$ to represent $v$ as a $\ZZ$ dimensional column vector with entries in $\CC^n$.  Represent $\sigma$ as a $\ZZ \cross \ZZ$ matrix $M_{pq}$ whose entries are $n \cross n$ complex matrices, with $M_{p q} = \sigma_{p - q}$, since $p = k + j$ and $q = j$, whence $k = p - q$.  We will refer, e.g., to the $p$-th component of the column vector representing $\sigma v$, meaning the coefficient of $e^{i p t}$ in the Fourier series for $\sigma v$, as $(\sigma v)_p$, rather than using $W$ to name the column vector for $v$ and writing $(M W)_p$.  Represent $J$ also as a $\ZZ \cross \ZZ$ matrix $N = (J_{p q})$ whose entries are $n \cross n$ complex matrices in block form as follows:  $J_{p q} = 0$ for $p \ne q$, and
\begin{align}
J_{p p} =
\left[
\begin{matrix}
 i & 0 \\
 0 & i
\end{matrix}
\right] p > 0,
\left[
\begin{matrix}
 i &  0 \\
 0 & -i
\end{matrix}
\right] p = 0, \text{ and }
\left[
\begin{matrix}
 -i &  0 \\
  0 & -i
\end{matrix}
\right] p < 0. \notag
\end{align}
In the case of $J$ there are no exponentials implicitly tagging along (though it operates on them), as with $\sigma$.  Then
\begin{align}
(J \sigma - \sigma J)_{p r} &= \sum_q J_{p q} \sigma_{q - r} - \sum_q \sigma_{p - q} J_{q r} \notag \\
 &= J_{p p} \sigma_{p - r} - \sigma_{p - r} J_{r r}, \notag
\end{align}
which is $0$ for $p r > 0$, $2 i \sign(p) \sigma_{p - r}$ for $p r < 0$, $i (J_{0 0} - \sign(r) \ident) \sigma_{-r}$ for $p = 0, r \ne 0$, $i \sigma_{p} (\sign(p) \ident - J_{0 0})$ for $p \ne 0, r = 0$, and $J_{0 0} \sigma_0 - \sigma_0 J_{0 0}$ for $p = r = 0$.

\begin{align}
J_{0 0} - \sign(r) \ident &=
\left[
\begin{matrix}
 0 &   0 \\
 0 & -2i
\end{matrix}
\right] r > 0,
\left[
\begin{matrix}
 2i & 0 \\
  0 & 0
\end{matrix}
\right] r < 0, \notag \\
\sign(p) \ident - J_{0 0} &=
\left[
\begin{matrix}
 2i & 0 \\
  0 & 0
\end{matrix}
\right] p > 0,
\left[
\begin{matrix}
 0 &   0 \\
 0 & -2i
\end{matrix}
\right] p < 0. \notag
\end{align}
To show that $\norm{[J, \sigma]}_2$ is finite, we can use the fact that all norms on a finite-dimensional vector space ($\B(\CC^n)$ here) are equivalent; in particular, $\norm{B}_2 \le c \norm{B}$, for some positive constant $c$ and any operator $B$ on $\CC^n$.  Let $e_j$ be the standard basis of $\CC^n$.  Let $A \in \B(H)$ be represented as a $\ZZ \cross \ZZ$ matrix $(A_{p q})$ of $n \cross n$ matrices.  Then $\norm{A}_2^2 = \sum_{p, q \in \ZZ} \sum_j \norm{A_{p q} e_j}^2 = \sum_{p, q \in \ZZ} \norm{A_{p q}}_2^2 \le c \sum_{p, q \in \ZZ} \norm{A_{p q}}^2$.  Applying this to the results of the computations above,
\begin{align}
\norm{[J, \sigma]}_2^2 &\le 2 \sum_{p r < 0} \norm{\sigma_{p - r}}^2 + 2 \sum_{p = 0, r \ne 0} \norm{\sigma_{-r}}^2 + 2 \sum_{p \ne 0, r = 0} \norm{\sigma_p}^2 + 2 \norm{\sigma_0}^2 \notag \\
 &= 2 \sum_{(q + r)(r) < 0} \norm{\sigma_{q}}^2 + 4 \sum_{p \ne 0} \norm{\sigma_p}^2 + 2 \norm{\sigma_0}^2 \notag \\
 &= 2 \sum_{\abs{q} > 0} (\abs{q} - 1) \norm{\sigma_{q}}^2 + 4 \sum_{\abs{q} > 0} \norm{\sigma_q}^2 + 2 \norm{\sigma_0}^2 \notag \\
 &= 2 \sum_{q \in \ZZ} (\abs{q} + 1) \norm{\sigma_{q}}^2 \notag \\
 &\le 2 \sum_{q \in \ZZ} \norm{q \sigma_{q}}^2 + 2 \sum_{q \in \ZZ} \norm{\sigma_{q}}^2, \notag
\end{align}
which is finite because since $\sigma$ is smooth, the Fourier series for its derivative and for $\sigma$ itself converge in the $L^2$ norm, by which we mean to treat $L^2(S^1, \B(\CC^n))$ as $L^2(S^1, \CC^{n^2})$, and make use again of equivalence of norms on finite-dimensional spaces.

Thus we have shown the inclusion of $L\SO(n)$ into $\Orth_{res}$.

Since for functions on a compact set, $\norm{}_{L_2} \le \norm{}_{L^{\infty}}$, and $\abs{q} \le \abs{q}^2$ for $q \in \ZZ$, it follows that $\norm{[J, \sigma]}_2^2$ is bounded by $2 (\norm{\sigma}_0^2 + \norm{\sigma}_1^2)$, referring to the $0$th and $1$st seminorms of the \Frechet space $L\B(\RR^n)$; and the first term of $\norm{\sigma}_J = \norm{\sigma} + \norm{[J, \sigma]}_2$ is bounded by $\norm{\sigma}_0$.  Thus by lemma \ref{l-lson-two-top} the inclusion $L\SO(n) \rightarrow \Orth_{res}$ is continuous.

Let $M = \max_{t \in S^1} \norm{\sigma(t)}$, $t_0$ some $t$ at which the maximum is attained, $w_k \in \RR^n$ a sequence such that $\norm{w_k} = 1$ for each $k$ and $\sigma(t_0)(w_k) \rightarrow M$, $d_k \in L\RR$ a sequence such that $d_k(t) \ge 0$ for all $k$ and $t$, $d_k(t) = 0$ for all $t \in S^1 \setminus \B_{1/(k+1)} (t_0)$, and $\norm{d_k}_{L^2} = 1$, then defining $v_k = d_k w_k \in L\RR^n$ and using the inner product of definition \ref{d-lrn-preh},
\begin{align}
\int_{S^1} \norm{\sigma(t) v_k(t)}^2 dt &\rightarrow \sup_{\int_{S^1} \norm{v(t)}^2 dt = 1} \int_{S^1} \norm{\sigma(t) v(t)}^2 dt = \norm{\sigma}, \notag
\end{align}
so it can be seen that $\norm{\sigma} = \max_{t \in S^1} \norm{\sigma(t)} = \norm{\sigma}_0$, the $0$-th seminorm of the \Frechet space $L\B(\RR^n)$.  Thus the statement about the subspace topology is true.
\end{proof}
\begin{note}\label{n-top-imag-lson-ores}
\index{LSOn Ores topology@$L\SO(n) \hookrightarrow \Orth_{res}$ topology}
(Notes on Proof).
The topology of the image of $L\SO(n)$ in $\Orth_{res}$ is given by a norm equal to the sum of the uniform norm (over $S^1$, of the operator norm of matrices) and the Sobolev norm for $\frac{1}{2}$-differentiable functions $S^1 \rightarrow \B(\CC^n)$.  We adapt to our functions the definition of the Sobolev norm in \citet[page~181]{LL01}.

Note that smoothness was not necessary, as the penultimate expression bounding $\norm{[J, \sigma]}_2^2$ is twice the square of the norm of $\sigma$ in the Sobolev space of half differentiable functions on $S^1$.  See note \ref{n-smth-loop}.
\end{note}

\chapter{THE POLARIZATION CLASS BUNDLE}\label{c-pol-clas-bndl}

This chapter will again use an associated bundle construction, to define $Y \rightarrow LM$.  We will suggest some details for an alternative, more geometric way to look at fibers of $Y$, and a way to use $Y_{\gamma}$ to define a fixed polarization class of Lagrangian subspaces of $\CC \tensor \overline{LE_{\gamma}}$, where $\overline{LE_{\gamma}}$ is the Hilbert space completion of $LE_{\gamma}$, as in definition \ref{d-fibr-innr-prod-loop-vb}.
\begin{defn}\label{d-pol-clas-bndl}
\index{polarization class bundle}
\index{bundle!polarization class}
\index{Y@$Y$}
(The Polarization Class Bundle).
\[
Y = L\SO(E) \cross_{L\SO(n)} \Lagr_{res} \xrightarrow{\pi} LM.
\]
where the continuous action of $L\SO(n)$ on $\Lagr_{res}$ is given by definition \ref{d-ores-uvj-homo}, lemma \ref{l-ores-uvj-homo}, and proposition \ref{p-lson-in-ores}.  See example \ref{e-loop-soe} and section \ref{s-tech-over} of chapter \ref{c-intr} for context for this construction and an introduction to our use of associated bundles.
\end{defn}
For a more geometric viewpoint, we could write for the fiber of $Y$, $[(\widetilde{\gamma}, K)] \leftrightarrow (\gamma, K_{\gamma})$, where $K_{\gamma}$ is a Lagrangian subspace of $\CC \tensor \overline{LE_{\gamma}}$ that depends on $[(\widetilde{\gamma}, K)]$. Suggestions for how to do this follow.  See also \citet[page~829]{SW07}.

As a set, $Y = \{ [(\widetilde{\gamma}, K)] \st \widetilde{\gamma} \in L\SO(E) \text{ and } K \in \Lagr_{res} \}$.  The set $\Lagr_{res} \subset \Lagr(H)$, the set of Lagrangian subspaces of $H = \CC \tensor_{\RR} V = L^2 (S^1, \CC^n)$, and $\Lagr(H) \subset \Gr(H)$, the set of all subspaces of $H$.  Define the set $YY = \{ [(\widetilde{\gamma}, K)] \st \widetilde{\gamma} \in L\SO(E) \text{ and } K \in \Gr(H) \}$, possible since there is an action ignoring continuity, of $L\SO(n)$ on $\Gr(H)$.  Also, as a set, $\Lagr(\CC \tensor  \overline{LE_{\gamma}}) \subset \Gr(\CC \tensor  \overline{LE_{\gamma}})$.  Define the set $\Gr(\CC \tensor  \overline{LE}) = \{ (\gamma, K) \st \gamma \in LM \text{ and } K \in \Gr(\CC \tensor  \overline{LE_{\gamma}}) \}$.

Define the map of sets $\Upsilon \colon YY \rightarrow \Gr(\CC \tensor  \overline{LE})$ by $\Upsilon \colon [(\widetilde{\gamma}, K)] \mapsto (\gamma, \{ \widetilde{\gamma} (k) \st k \in K \})$, where $\gamma = \pi_{L\SO(E)} (\widetilde{\gamma})$, and the complex linear extension of an element of $L\SO(E)$ over $\gamma$, $\widetilde{\gamma} \colon \CC \tensor L\RR^n \rightarrow \CC \tensor LE_{\gamma}$, using on $L\RR^n$ the inner product of definition \ref{d-lrn-preh}, is a bounded map of pre-Hilbert spaces that extends by continuity to a map of their completions, $H \rightarrow \CC \tensor  \overline{LE_{\gamma}}$.

The map $\Upsilon$ respects the equivalence relation forming its domain, since if $g \in L\SO(n)$, $[(\widetilde{\gamma} g, g^{-1} K)] \mapsto (\gamma, \{ \widetilde{\gamma} \circ g (g^{-1} (k)) \st k \in K \}) = (\gamma, \{ \widetilde{\gamma} (k) \st k \in K \})$.  Thus we can write $\Upsilon \colon [(\widetilde{\gamma}, K)] \mapsto (\gamma, K_{\gamma})$, where $K_{\gamma}$ is a subspace of $\CC \tensor \overline{LE_{\gamma}}$ that depends on $[(\widetilde{\gamma}, K)]$.

$\Upsilon$ is injective because it is so on each fiber.  To see that $\Upsilon$ is injective over any fixed $\gamma$, suppose $(\gamma, \{ \widetilde{\gamma}' (k') \st k' \in K' \}) = (\gamma, \{ \widetilde{\gamma} (k) \st k \in K \})$, where $k', k \in H$.  Since $L\SO(n)$ acts transitively on the fibers of $L\SO(E)$ by proposition \ref{p-loop-smth-pb-is-frec-pb}, there is some $g \in L\SO(n)$ such that $\widetilde{\gamma}' = \widetilde{\gamma} g$.  Since $\widetilde{\gamma}' (k') = \widetilde{\gamma}'  \circ g  (g^{-1} (k')) = \widetilde{\gamma}  (g^{-1} (k'))$, the equality of values of $\Upsilon$ implies that $g^{-1} K' = K$, whence $[(\widetilde{\gamma}', K')] = [(\widetilde{\gamma}, K)]$.

To see that $\Upsilon$ is surjective, take any $(\gamma, S) \in \Gr(\CC \tensor  \overline{LE})$, and any $s \in S \subset \CC \tensor  \overline{LE_{\gamma}})$.  Since smooth functions are dense in the $L^2$ completion, let $s_i$ be a sequence in $\CC \tensor LE_{\gamma}$ converging in $L^2$ to $s$, and choose $\widetilde{\gamma} \in L\SO(E)$ such that $s_i = \widetilde{\gamma} (k_i)$ for every $i$, for some sequence $k_i$ of smooth loops in $\CC^n$.  Since $\widetilde{\gamma}$ is an isometric isomorphism of Hilbert spaces, the $k_i$ converge in $L^2$ to some $k \in H$, and $s = \widetilde{\gamma} (k)$.

We claim but don't prove that $\Upsilon$ preserves subspaces, complex conjugation, orthogonality, and orthogonal complement relations between subspaces in each fiber, because it's a bijection and because of how it uses $(L\SO(E))_{\gamma} \subset \Orth(LE_{\gamma})$.  Thus $\Upsilon$ can be used to define, from the standard polarization class $\Lagr_{res}$ of $H$, a fixed polarization class of $\CC \tensor LE_{\gamma}$ for each $\gamma$.

\chapter{THE STANDARD FOCK SPACE BUNDLE}\label{c-std-fock-spac-bndl}

Before we return to associated bundles in the next chapter, this chapter discusses a standard bundle constructed by other means over a standard space.  The standard Fock space bundle $F \rightarrow \Lagr_{res}$ will be used in the construction of the (not standard) Fock space bundle $FY$ over a different space in chapter \ref{c-fock-spac-bndl}, and the local trivializations of $F$ will be used to construct the local trivializations of the intertwiner bundle $T$ in proposition \ref{p-bg-t}.

Given the infinite-dimensional real inner product space $V$ and a chosen Lagrangian subspace $L$ of the Hilbert space $H$ defined as the completion (if needed) of $V \tensor \CC$, as in assumption \ref{a-res-grp}, one can construct the Fock space $\F(L)$ as in definition \ref{d-fock-spac}, a complex Hilbert space upon which the Clifford algebra $\Cl(V)$, a $C^{*}$-algebra, acts via a ${}^{*}$-representation as in proposition \ref{p-fock-rep}; thus $\F(L)$ is a Clifford module.

Trying to do the analogous thing over each $\gamma \in LM$, although as in chapter \ref{c-pol-clas-bndl}, we can choose continuously, a polarization class of Lagrangian subspaces $L_{\gamma}$ of $\CC \tensor \overline{LE_{\gamma}}$, for each of which $\F(L_{\gamma})$ is a Clifford module bundle over $LE$, it's not possible in general (\citet{Stol12} gave a proof outline) to choose continuously one Lagrangian subspace of the polarization class for each $\gamma$.

In this chapter, however, we will come from a different angle and do another more limited thing that will prove useful later:  show that there is a Clifford module bundle $F \rightarrow \Lagr_{res}$, the base space of which, $\Lagr_{res}$, is the standard polarization class of our standard Hilbert space $H = L^2(S^1, \CC^n)$.  The fiber of this bundle over a Lagrangian subspace $K$ is $F_K = \F(K)$.

Although $F$ is defined as a set without choosing any particular $L \in \Lagr_{res}$, the local bundle trivializations used to define the topology of the total space, do depend on a choice of $L$.  They use maps that are Clifford linear on fibers, to standard fiber $\F(L)$.  We will show that the topology of the total space and the Clifford module structure on the fibers are independent of the choice of $L$.

\begin{prop}\label{p-std-fock-spac-bndl}
\index{Fock space!standard bundle}
\index{bundle!standard!Fock space}
\index{F@$F$}
(The Standard Fock Space Bundle).
Let $F$ be the set of pairs $(K, w)$ such that $K \in \Lagr_{res}$ and $w \in \F(K)$, with the projection $\pi_F \colon F \rightarrow \Lagr_{res}$, $(K, w) \mapsto K$.  Suppose given $L \in \Lagr_{res}$; then we can define fiber bundle local trivializations with standard fiber $\F(L)$ as follows, that restricted to fibers are $\Cl(V)$ linear unitary isomorphisms, i.e. intertwiners.  These local trivializations induce a topology on $F$ making it a continuous fiber bundle.  Local trivializations defined by another choice of $L$ are compatible via homeomorphisms that are fiberwise Clifford linear unitary isomorphisms; in particular, the topology doesn't depend on the choice of $L$.
\end{prop}
\begin{proof}
\begin{notat}\label{n-std-fock-spac-bndl}
\index{Fock space!standard bundle!local trivializations}
Given a choice of $L$, each $K \in \Lagr_{res}$ has an open neighborhood $V_K = V_{K, L}$ and local trivialization $\Theta_K = \Theta_{K, L}$ as in the commutative diagram:
\[
\begindc{\commdiag}[5]
\obj(10,10)[objLR]{$\Lagr_{res}$}
\obj(10,20)[objF]{$F$}
\obj(30,10)[objVKL]{$V_K$}
\obj(30,20)[objPVKL]{$\pi_F^{-1} (V_K)$}
\obj(55,20)[objVKLFL]{$V_K \cross \F(L)$}
\mor{objVKL}{objLR}{$\incl$}
\mor{objPVKL}{objF}{$\incl$}
\mor{objF}{objLR}{$\pi_F$}
\mor{objPVKL}{objVKL}{$\pi_F$}
\mor{objPVKL}{objVKLFL}{$\Theta_K$}
\mor{objVKLFL}{objVKL}{$\pi_1$}
\enddc
\]
For $K' \in V_K$, $\pi_2 \circ (\Theta_K)_{|K'} = T_{K', L} \in T(K', L)$ such that $T_{K', L} = U_{g_{K'}} \Lambda_{g_{K'}, L}^{*}$, where $g_{K'}$ depends continuously on $K'$, $\Lambda_{g_{K'}, L} \colon \F(L) \rightarrow \F(K')$, and $U_{g_{K'}}$ depends continuously in the $\UU(\F(L))$ strong operator topology, on $g_{K'}$ and hence on $K'$.
\begin{flushleft}
\emph{End Notation}
\end{flushleft}
\end{notat}
Define the topology on $F$ that makes the following system of local fiber bundle trivializations homeomorphisms.  Suppose $L$ given and omit it from the notation for $\Theta$ until the proof that the topology doesn't depend on $L$.  Recall definition \ref{d-ores-uvj-homo}.

Let $V$ be the open neighborhood of $L$ from the second paragraph of lemma \ref{l-loc-impl}, and let
\[
W = \pi_{\Orth_{res}/\UU(V_J)}^{-1} \circ \phi_{\Orth_{res}/\UU(V_J), \Lagr_{res}}^{-1} (V). \notag
\]

Given $K \in \Lagr_{res}$, construct a local trivialization over a neighborhood about $K$ as follows.  Choose a fixed $g_K \in \Orth_{res}$ such that $K = g_K L$, and by theorem \ref{t-impl-soln} let $U_{g_K} \in \UU(\F(L))$ be a fixed implementer of $\theta_{g_K}$ in $\pi_L$.  Define the open neighborhood $V_K$ of $K$ by
\begin{align}
W_K &= g_K W \notag \\
V_K &= \phi_{\Orth_{res}/\UU(V_J), \Lagr_{res}} \circ \pi_{\Orth_{res}/\UU(V_J)} (W_K) = g_K V. \notag
\end{align}
For any $K' \in V_K$, define $L' = g_K^{-1} K' \in V$, and let $g' = \xi(L') \in \Orth_{res}$ be the canonical element from lemma \ref{l-uv-ov-sect-res} such that $L' = g' L$, depending continuously on $L'$ and thus on $K'$.
\begin{notat}\label{n-std-fock-spac-bndl-g-prim}
\index{Fock space!standard bundle!choices}
Alternatively, for later use, in addition to free choice of any $g_K \in \Orth_{res}$ such that $K = g_K L$, any open neighborhood $V$ of $L$, throughout which $g'$ can be found, and any $g' \in \Orth_{res}$ such that $L' = g' L$, depending continuously on $L' \in V$, will do as well as the choices set forth above.  Other properties of the $V$ of lemma \ref{l-loc-impl} and the $g' = \xi(L')$ of lemma \ref{l-uv-ov-sect-res} are not used.

For specialized use later - not when constructing a topology as in this proposition - it is in fact not necessary that $V$ be open, only that it suit the need for which the local trivializations are constructed.
\begin{flushleft}
\emph{End Notation}
\end{flushleft}
\end{notat}
Letting $g_{K'} = g_K g'$, also a continuous function of $K'$, $K' = g_{K'} L$.
\[
\begindc{\commdiag}[5]
\obj(10,8)[objID]{$\ident$}
\obj(30,8)[objGK]{$g_K$}
\obj(10,14)[objW]{$W$}
\obj(30,14)[objW_K]{$W_K$}
\obj(20,20)[objORES]{$\Orth_{res}$}
\obj(60,0)[objLP]{$L'$}
\obj(80,0)[objKP]{$K'$}
\obj(60,8)[objL]{$L$}
\obj(80,8)[objK]{$K$}
\obj(60,14)[objV]{$V$}
\obj(80,14)[objV_K]{$V_K$}
\obj(70,20)[objLRES]{$\Lagr_{res}$}
\mor{objW}{objW_K}{$g_K \cdot$}
\mor{objID}{objGK}{$g_K \cdot$}[\atright, \aplicationarrow]
\mor{objV}{objV_K}{$g_K \cdot$}
\mor{objL}{objK}{$g_K \cdot$}[\atright, \aplicationarrow]
\mor{objLP}{objKP}{$g_K \cdot$}[\atright, \aplicationarrow]
\mor{objL}{objLP}{$g' \cdot$}[\atright, \aplicationarrow]
\mor{objL}{objKP}{$g_{K'} \cdot$}[\atright, \aplicationarrow]
\enddc
\]
Let $U_{g'} \in \UU(\F(L))$ be the implementer of $\theta_{g'}$ in $\pi_L$ from lemma \ref{l-loc-impl}, depending continuously on $K'$, in the strong operator topology on $\UU(\F(L))$.  Referring to lemma \ref{l-impl-homo}, define $U_{g_{K'}} = U_{g_K} U_{g'}$, an implementer for $g_{K'} = g_K g'$ in $\pi_L$, depending continuously on $K'$; and referring to
proposition \ref{p-bij-intw-impl} and lemma \ref{l-lamb-homo}, intertwiner $T_{L, K'} \colon \F(L) \rightarrow \F(K')$, $T_{L, K'} = \Lambda_{g_{K'}} U_{g_{K'}}^{*}$, and its inverse, $T_{K', L} = T_{L, K'}^{*} = U_{g_{K'}} \Lambda_{g_{K'}}^{*}$.

Use this intertwiner to define the local fiber bundle trivialization
\begin{align}
\Theta_K \colon \pi_F^{-1} (V_K) &\rightarrow V_K \cross \F(L) \notag \\
\Theta_K \colon (K', x) &\mapsto (K', T_{K', L} (x)). \notag
\end{align}
Being built from an intertwiner, which is a bijection, the local trivialization is bijective on fibers, and hence bijective.
If $V_{K_1} \cap V_{K_2} \ne \emptyset$, then the change of coordinates map is
\begin{align}
\Theta_{K_2} \circ \Theta_{K_1}^{-1} \colon V_{K_1} \cap V_{K_2} \cross \F(L) &\rightarrow V_{K_1} \cap V_{K_2} \cross \F(L) \notag \\
\Theta_{K_2} \circ \Theta_{K_1}^{-1} \colon (K',y) &\mapsto (K', T_{2, K', L} \circ T_{1, K', L}^{-1} (y)) \notag
\end{align}
where the $T$'s are dependent on $K'$, and the transition function that needs to be continuous as a function of $K'$ and $y$, in order for us to be able to define a topology on the total space is:
\begin{align}
T_{K', L, 2} \circ T_{K', L, 1}^{-1} = U_{g_{K',2}} \Lambda_{g_{K',2}}^{*} \Lambda_{g_{K',1}} U_{g_{K',1}}^{*} = U_{g_{K',2}} \Lambda_{g_{K',2}^{-1} g_{K',1}} U_{g_{K',1}}^{*}, \notag
\end{align}
referring to lemma \ref{l-lamb-homo}.

What we need to show is that $T_{K', L, 2} \circ T_{K', L, 1}^{-1} (y)$ is a jointly continuous function of $K'$ and $y$.  It is separately continuous in $y$ since the linear operators are bounded.

We know that in the strong operator topology on $\UU(\F(L))$, $U_{g_{K',2}}$ and $U_{g_{K',1}}^{*}$ (adjoint for unitary operators is continuous in the strong operator topology) were constructed as continuous functions of $K'$, and $\Lambda_{g_{K',2}^{-1} g_{K',1}} \colon \F(L) \rightarrow \F(L)$ is a continuous function of $K'$ by lemma \ref{l-lamb-homo-cont}, because $g_{K',2}^{-1} g_{K',1}$, which is in $\Orth_{res}$ and maps $L \rightarrow L$, hence is in $\UU(V_J)$, depends continuously on $K'$.  Although in general, operator multiplication (composition) is not continuous in the strong operator topology, all these operators are unitary and hence have norm $1$, so as in \citet[page~171]{Pede89}, their product $T_{K', L, 2} \circ T_{K', L, 1}^{-1}$ is a continuous function of $K'$ in the strong operator topology, and so $T_{K', L, 2} \circ T_{K', L, 1}^{-1} (y)$ is separately continuous in $K'$. Finally, by lemma \ref{l-cont-stro-op-top-adj}, $T_{K', L, 2} \circ T_{K', L, 1}^{-1} (y)$ is jointly continuous in $K'$ and $y$.

That a system of local trivializations defined using another choice $L' \in \Lagr_{res}$ (don't confuse this with $L'$ in the proof above) in place of $L$, is compatible with those defined using $L$, can be seen as follows, adapting the notation of the foregoing.  Given open neighborhoods $V_{K_1, L}, V_{K_2, L'}$ with nonempty intersection $V$, we need to show that $\Theta_{K_2, L'} \circ \Theta_{K_1, L}^{-1} \colon V \cross \F(L) \rightarrow V \cross \F(L')$ is continuous.  Its first component is the identity, and its second is $T_{K_2, L'} \circ T_{K_1, L}^{-1}$, which for each point $K' \in V$ is a Clifford linear unitary isomorphism, given almost as before but with added notation to track $L'$ as well as $L$, by
\begin{align}
T_{K', L', 2} \circ T_{K', L, 1}^{-1} = U_{g_{K',L',2}} \Lambda_{g_{K',L',2}}^{*} \Lambda_{g_{K',L,1}} U_{g_{K',L,1}}^{*} = U_{g_{K',L',2}} \Lambda_{g_{K',L',2}^{-1} g_{K',L,1}} U_{g_{K',L,1}}^{*}. \notag
\end{align}
The same reasoning for continuity holds as before, with one noticeable difference:  $g_{K',L',2}^{-1} g_{K',L,1} \colon L \rightarrow L'$, not $L \rightarrow L$.  However, taking any fixed $g_{L',L} \in \Orth_{res}$ such that $g_{L',L} L' = L$, $g_{L',L} g_{K',L',2}^{-1} g_{K',L,1} \colon L \rightarrow L$ and hence is a continuous function of $K'$ as before.  Then, since $g_{L',L}$ is constant and hence continuous with respect to $K'$, and is a unitary operator, composing it on the left with the triple product gives a continuous function, $g_{K',L',2}^{-1} g_{K',L,1}$, of $K'$.
\end{proof}

\begin{lem}\label{l-clif-act-std-fock-spac-bndl}
\index{Fock space!standard bundle!Clifford algebra action}
\index{bundle!standard Fock space!Clifford algebra action}
\index{Cl(LRn)@$\Cl(L\RR^n)$}
\index{F@$F$}
(The Clifford Algebra Action on the Standard Fock Space Bundle).
The fiberwise action of $\Cl(L\RR^n)$ defines a continuous map $\Cl(L\RR^n) \cross F \rightarrow F$.
\end{lem}
\begin{proof}
To show continuity at $(a, (K, w))$, use local trivialization $\Theta_K$ as in proposition \ref{p-std-fock-spac-bndl} notation \ref{n-std-fock-spac-bndl}.  Let $a'$ lie in some ball about $a$ and let $K'$ be some element of $V_K$. Via the local trivialization, $w' \in \F(K')$ corresponds to $\widehat{w'} = T_{K', L} (w') \in \F(L)$.  Then continuity of the action is implied by the continuity of the following function of $a'$, $K'$, and $\widehat{w'}$, which is a consequence of definition \ref{d-equi-rep}, proposition \ref{p-fock-rep}, and lemma \ref{l-cont-stro-op-top-adj}.  To get the bound on norms needed to apply the second part of the lemma, note that $C^{*}$-algebra representations are isometric algebra morphisms.
\[
T_{K', L} (\pi_{K'} (a') (T_{K', L}^{*} (\widehat{w'}))) = \pi_L (a') (\widehat{w'}). \notag
\]
\end{proof}

\chapter{THE FOCK SPACE BUNDLE}\label{c-fock-spac-bndl}

We construct the Fock space bundle $FY$ over the polarization class bundle $Y$.  The fiber of $Y$ over a loop $\gamma \in LM$ is a fixed polarization class, a set of Lagrangian subspaces $L_{\gamma}$ of $LE_{\gamma}$ with unitarily equivalent Fock representations.  The fiber of $FY$ over each $L_{\gamma}$ is the Fock space of that representation; the standard fiber for our associated bundle construction is the standard Fock space bundle $F$.

\begin{defn}\label{d-fock-spac-bndl}
\index{bundle!Fock space}
\index{Fock space!bundle}
\index{FY@$FY$}
(The Fock space bundle $FY \xrightarrow{\pi_{FY}} Y$).
Define an action of $\Orth_{res}$, and hence of $L\SO(n)$ on $F$, and then define $FY$ via associated bundle construction:
\begin{align}
&g \in \Orth_{res}, K \in \Lagr_{res}, w \in \F(K) \Rightarrow (g, (K, w)) \mapsto (g K, \Lambda_g (w)) \notag \\
&FY = L\SO(E) \cross_{L\SO(n)} F, \text{ with } \pi_{FY} \colon [(\widetilde{\gamma}, (K, w))] \mapsto [(\widetilde{\gamma}, K)]. \notag
\end{align}
\end{defn}
\begin{note}\label{n-fock-spac-bndl}
\index{associated bundle!geometric viewpoint}
(Associated Bundle vs. Geometric Viewpoints).
To aid in intuition, one may use the correspondence
\begin{align}
FY &= \{ [(\widetilde{\gamma}, (K, w))] & &\st \widetilde{\gamma} \in L\SO(E) \text{, } K \in \Lagr_{res} \text{ and } w \in \F(K) \} \notag \\
&\leftrightarrow \{ (\gamma, (L_{\gamma}, w_{\gamma})) & &\st L_{\gamma} \text{ is in the fixed polarization class of } \CC \tensor \overline{LE_{\gamma}} \notag \\
& & &\text{ and } w_{\gamma} \in \F(L_{\gamma})\}, \notag
\end{align}
$\gamma = \pi_{L\SO(E)} (\widetilde{\gamma})$.  In a sense, $\gamma$ may be unnecessary in the notation $(\gamma, (L_{\gamma}, w_{\gamma})))$, since elements of the subspace $L_{\gamma}$ are $L^2$ loops, including continuous and even smooth loops that via $L\pi_E$ project to $\gamma$.
\end{note}

\begin{lem}\label{l-clle-act-fy}
\index{Fock space bundle}
\index{bundle!Fock space}
\index{Cl(LE) module@$\Cl(LE)$ module}
\index{FY@$FY$}
(The Definition of $FY$ is Valid, and $FY$ is a $\pi^{*} \Cl(LE)$ module).
Here, $\pi$ is the projection $Y \rightarrow LM$ (see definition \ref{d-pol-clas-bndl}).  See the end of the proof for the definition of the Clifford action.
\end{lem}
\begin{proof}
The continuity of the action of $\Orth_{res}$ on $F$ can be seen by choosing $L \in \Lagr_{res}$ for local trivializations in proposition \ref{p-std-fock-spac-bndl} notation \ref{n-std-fock-spac-bndl}, using continuity of the action of $\Orth_{res}$ on $\Lagr_{res}$ as in lemma \ref{l-ores-uvj-homo} by using definition \ref{d-ores-uvj-homo}.  Fix $K \in \Lagr_{res}$ and let $(g', K') \in W_g \cross W_K$, where open neighborhoods $W_g$ of $g \in \Orth_{res}$ and $W_K$ of $K \in V_K$, the open neighborhood of $K$ for the local trivialization $\Theta_K \colon \pi_F^{-1} (V_K) \rightarrow V_K \cross \F(L)$, are chosen so that $g' K' \in V_{g K}$, the open neighborhood of $g K$ for the local trivialization $\Theta_{g K}$.  The continuity of the action at $(g, (K, w)) \in \Orth_{res} \cross F$ is then equivalent to that of the following map, where $\widehat{w'} \in \F(L)$.
\begin{align}
(g', (K', \widehat{w'})) &\mapsto \Theta_{g K} (g' \pi_1, \Lambda_{g'} \circ \pi_2) (\Theta_K^{-1} (K', \widehat{w'})), \text{ which since} \notag \\
\Theta_{g K} (g' K', x) &= (g' K', U_{g_{g' K'}} \circ \Lambda_{g_{g' K'}, L}^{*} (x)) \text{ and} \notag \\
\Theta_K^{-1} (K', \widehat{w'}) &= (K', \Lambda_{g_{K'}, L} \circ U_{g_{K'}}^{*} (\widehat{w'})), \text{ is given by} \notag \\
(g', (K', \widehat{w'})) &\mapsto (g' K', U_{g_{g' K'}} \circ \Lambda_{g_{g' K'}, L}^{*} \circ \Lambda_{g', K'} \circ \Lambda_{g_{K'}, L} \circ U_{g_{K'}}^{*} (\widehat{w'})) \notag \\
&= (g' K', U_{g_{g' K'}} \circ \Lambda_{g_{g' K'}^{-1} g' g_{K'}, L} \circ U_{g_{K'}}^{*} (\widehat{w'})), \notag
\end{align}
where the symbol $\Lambda$ with two subscripts indicates the group element that induces it and the Lagrangian subspace for the Fock space that is its domain.  Now $g' K'$ is a continuous function of $g'$ and $K'$ jointly, so are $U_{g_{g' K'}} \in \UU(\F(L))$, $\Lambda_{g_{g' K'}^{-1} g' g_{K'}, L} \in \UU(\F(L))$, $U_{g_{K'}}^{*} \in \UU(\F(L))$, and hence their composition is also, in the strong operator topology on $\UU(\F(L))$.  Also the operators applied to $\widehat{w'}$ are bounded.  Holding $g'$ fixed for the moment, reasoning as near the end of the proof of proposition \ref{p-std-fock-spac-bndl}, our map is separately continuous in $K'$ and $\widehat{w'}$ and hence in them jointly.  Holding $(K', \widehat{w'})$ fixed, our map is continuous in $g'$.  Thus since $\Lagr_{res}$ is metrizable and $\F(L)$ has the norm metric, by corollary \ref{co-ores-bair-acts-join-cont} the action is jointly continuous, and the associated bundle definition is possible.

We haven't given $FY$ local trivializations that are Clifford linear, but don't need to.  For each $y \in Y$, $FY_{y}$ is an irreducible $\Cl(LE)_{\pi(y)}$ module, or equivalently, $(\pi^{*} \Cl(LE))_y$ module, as can be seen by comparing the definitions
\begin{align}
F &= \{ (K, \F(K)) \st K \in \Lagr_{res} \} \xrightarrow{\pi_F} \Lagr_{res} \notag \\
Y &= L\SO(E) \cross_{L\SO(n)} \Lagr_{res} \xrightarrow{\pi} LM \notag \\
FY &= L\SO(E) \cross_{L\SO(n)} F \xrightarrow{\pi_{FY}} Y \notag \\
\Cl(LE) &= L\SO(E) \cross_{L\SO(n)} \Cl(L\RR^n) \xrightarrow{\pi_{\Cl(LE)}} LM. \notag
\end{align}
Let $y = [(\widetilde{\gamma}, K)]$, $\widetilde{\gamma} \in L\SO(E)$ lying over $\gamma \in LM$, $K \in \Lagr_{res}$, $[(\widetilde{\gamma}, (K, w))] \in FY_y$, $w \in \F(K)$, $[(\widetilde{\gamma}, a)] \in \Cl(LE)_{\pi(y)}$, $a \in \Cl(L\RR^n)$.  The Clifford action is defined using equal first components of the pairs in the equivalence classes and the continuous map $\Cl(L\RR^n) \cross F \rightarrow F$ defined in proposition \ref{l-clif-act-std-fock-spac-bndl}, $g \in L\SO(n)$:
\begin{align}
([(\widetilde{\gamma}, a)], [(\widetilde{\gamma}, (K, w))]) &\mapsto [(\widetilde{\gamma}, (K, \pi_K (a) (w)))], \notag
\end{align}
which is well-defined because
\begin{align}
([(\widetilde{\gamma} g^{-1}, \theta_g(a))], [(\widetilde{\gamma} g^{-1}, (g K, \Lambda_g (w)))]) &\mapsto [(\widetilde{\gamma} g^{-1}, (g K, \pi_{g K} (\theta_g(a)) (\Lambda_g(w))))] \notag \\
&= [(\widetilde{\gamma}, (K, \Lambda_{g^{-1}}(\pi_{g K} (\theta_g(a)) (\Lambda_g(w)))))] \notag \\
&= [(\widetilde{\gamma}, (K, \pi_K (a) (w)))], \text{ by lemma \ref{l-lamb-homo}}. \notag
\end{align}
The continuity of the action follows from proposition \ref{l-clif-act-std-fock-spac-bndl}; that is, we define a continuous map on the cross products from which the associated bundles are formed, which because it respects the equivalence relation, descends to a continuous map of the quotients.  That is, a continuous equivariant map descends to a continuous map of the orbit spaces \citep[page~4]{tomD87}.
\end{proof}

\chapter{CONTINUOUS BUNDLE GERBES}\label{c-cont-bndl-gerb}

This chapter defines ``continuous'' bundle gerbes, their Dixmier-Douady classes and standard constructions we will use.  Our references assume smoothness, which we don't need.  Recall from assumption \ref{a-bndl} that unless bundles are called smooth or \Frechet, they are only continuous or topological.

\begin{ass}\label{a-cont-bg}
\index{continuous bundle gerbes}
\index{bundle gerbes!continuous}
(Continuous Bundle Gerbes).
The term bundle gerbe will refer to a continuous bundle gerbe unless otherwise noted.
\end{ass}

Our bundle gerbes have ``band'' $\UU(1)$; part of their definition is a principal $\UU(1)$ bundle.  Other abelian Lie groups could be used for other purposes.  Bundle gerbes with band $\UU(1)$ are useful to us because the intertwiners for two equivalent Fock representations form a $\UU(1)$ torsor (recall definition \ref{d-g-tors}).  Fibers of a principal $\UU(1)$ bundle are $\UU(1)$ torsors, and their being in a bundle reflects their being continuously related to each other.

Bundle gerbes with band $\UU(1)$ have been defined using principal $\UU(1)$ bundles in \citet[page~243]{Murr07}.  Equivalent definitions have been made with Hermitian line bundles \citep[page~3]{Wald07}, complex line bundles \citep[page~280]{HJJS08}, and $\CC^{\cross}$ principal bundles \citep[page~18]{Stev00}.

We now give some details about $\UU(1)$ torsors and principal bundles; those familiar with this material can skim through it or skip to section \ref{s-cont-bg-def}.

\section{U($1$) Torsors}\label{s-u1-tors}

\begin{defn}\label{d-u1t-tens}
\index{U(1) torsor@$\UU(1)$ torsor!tensor product}
\index{tensor product!U(1) torsors@$\UU(1)$ torsors}
\index{torsor}
(Tensor Products of $\UU(1)$ Torsors).
Given two $\UU(1)$ torsors $T_1$, $T_2$ with respective actions $\rho_1, \rho_2$, define a $\UU(1)$ action on $T_1 \cross T_2$ by
\begin{align}
\rho_Q \colon \UU(1) \cross (T_1 \cross T_2) &\rightarrow (T_1 \cross T_2) \notag \\
(z, (t_1, t_2)) &\mapsto (z t_1, z^{-1} t_2), \notag
\end{align}
and define the tensor product of $T_1$ and $T_2$, $T_1 \tensor T_2$, as the corresponding quotient, $(T_1 \cross T_2) / \UU(1)$ \citep[page~2]{tomD87}.  Define the map
\begin{align}
\widetilde{\rho_1 \tensor \rho_2} \colon \UU(1) \cross (T_1 \cross T_2) &\rightarrow (T_1 \cross T_2) \notag \\
(z, (t_1, t_2)) &\mapsto (z t_1, t_2). \notag
\end{align}
This is $\UU(1)$-equivariant with respect to the trivial action of $\UU(1)$ on itself and the action $\rho_Q$ on $T_1 \cross T_2$.  Define $\rho_1 \tensor \rho_2 \colon \UU(1) \cross (T_1 \tensor T_2) \rightarrow T_1 \tensor T_2$ as the continuous map induced by $\widetilde{\rho_1 \tensor \rho_2}$ on the quotients; a continuous equivariant map descends to a continuous map of the orbit spaces \citep[page~4]{tomD87}.  Let $t_1 \tensor t_2$ denote $[(t_1, t_2)]$.

This definition is written for a left action, but, because $\UU(1)$ is commutative, may be used mutatis mutandis for a right action also.
\end{defn}

\begin{lem}\label{l-u1t-tens}
\index{U(1) torsor@$\UU(1)$ torsor!tensor product}
\index{tensor product!U(1) torsors@$\UU(1)$ torsors}
\index{torsor}
(Tensor Products of $\UU(1)$ Torsors).
Given $\UU(1)$ torsors $T_1$, $T_2$, the tensor product $T_1 \tensor T_2$ of definition \ref{d-u1t-tens} is a $\UU(1)$ torsor.
\end{lem}
\begin{proof}
By \citet[pages~22--23]{tomD87}, since $\UU(1)$ is compact and $T_1$ and $T_2$, hence $T_1 \cross T_2$ are Hausdorff, $T_1 \tensor T_2$ is Hausdorff.

Let the actions for $T_1$, $T_2$ be respectively $\rho_1$, $\rho_2$. $\rho_1 \tensor \rho_2$ is given by $\rho_1 \tensor \rho_2 (z, [(t_1, t_2)]) = [(z t_1, t_2)] = [(z^{-1} z t_1, z t_2)] = [(t_1, z t_2)]$ using the quotient equivalence relation from the action $\rho$.  This satisfies the algebraic properties of a group action as follows because $\rho_1$ is a group action: $[(z_1 (z_2 t_1), t_2)] = [((z_1 z_2) t_1, t_2)]$ and $[(1 t_1, t_2)] = [(t_1, t_2)]$.

$\rho_1 \tensor \rho_2$ is free because $[(z t_1, t_2)] = [(t_1, t_2)] \Leftrightarrow (z t_1, t_2) = (w t_1, w^{-1} t_2)$ for some $w \in \UU(1)$, which implies since $\rho_2$ is free that $w = 1$, and hence because $\rho_1$ is free, $z = 1$.  It is transitive because fixing some $[(t_{1 0}, t_{2 0})] \in T_1 \tensor T_2$, given any $[(t_1, t_2)]$, there is some $w \in \UU(1)$ such that $t_2 = w t_{2 0}$ because $\rho_2$ is transitive, and there is some $z \in \UU(1)$ such that $t_1 = z (w^{-1} t_{1 0})$ because $\rho_1$ is transitive, whence $[(z t_{1 0}, t_{2 0})] = [(w^{-1} z t_{1 0}, w t_{2 0})] = [(t_1, t_2)]$.  Thus $T_1 \tensor T_2$ with the action $\rho_1 \tensor \rho_2$ is a $\UU(1)$ torsor.
\end{proof}

\begin{defn}\label{d-u1t-dual}
\index{U(1) torsor@$\UU(1)$ torsor!dual}
\index{dual!U(1) torsor@$\UU(1)$ torsor}
\index{torsor}
(The Dual of a $\UU(1)$ Torsor).
Given a $\UU(1)$ torsor $T$ with action $\rho$, define its dual $T^{*}$, also called its inverse $T^{-1}$, as the same topological space and set, so that as topological spaces and sets we can write $T^{*} = T$.  For $t \in T$, let $t^{*} \in T^{*}$ denote the same element $t$.  That is, as an element of the topological space for $T$, which is also the topological space for $T^{*}$, $t^{*} = t$.

The action for $T^{*}$ is $\rho^{*}$, given for $z \in \UU(1)$, $t \in T = T^{*}$ by $z t^{*} = \rho^{*} (z, t^{*}) = (\rho (\overline{z}, t))^{*} = (\overline{z} t)^{*}$, where the first equality $T = T^{*}$ is of sets, the second and fourth are previously defined notational equalities of elements of torsors, the third is the present definitional equality of elements of torsors; and $\overline{z} = z^{-1}$ on $\UU(1)$.

Among its other meanings, let ${}^{*} \colon T \rightarrow T^{*}$ denote the identity map on the topological spaces, written ``exponentially'' as $t \mapsto t^{*}$.  When the argument of this map isn't a single symbol, use parentheses to enclose the argument, as in $(\overline{z} t)^{*}$, where $t \in T$ as a torsor, which implies that $\overline{z} t = \rho (\overline{z}, t) \in T$.

When working with duals of $\UU(1)$ torsors we use $\overline{z}$ to denote the complex conjugate of $z$; we do not use $z^{*}$ to denote the complex conjugate of $Z$ or equivalently the adjoint of $z \in \UU(1)$ as a unitary operator on $\CC$.  Great care should be taken with this to avoid a conflict with the definition, for $z \in \UU(1)$, of $z^{*} \in \UU(1)^{*}$ as the same set element $z^{*} = z$.

This definition is written for a left action, but, because $\UU(1)$ is commutative, may be used mutatis mutandis for a right action also.
\end{defn}
The definition could be rewritten to make explicit in the notations the distinction between the $\UU(1)$ torsor, and the topological space and set, $T$, but this would be more cumbersome in practical use later.  Often, ${}^{*}$ applied to an element will be effectively a sneaky way to remind us when to interpret the result as being in the dual torsor to the torsor in which lies the element of the unadorned symbol; in particular to remind us which action to use.

The possible confusion between ${}^{*}$ as torsor dual and ${}^{*}$ as complex conjugation or operator adjoint appears to be the price for going with this common notation, which otherwise seems useful.  Alternatives would be some otherwise unused symbolic device, or ordinary functional notation.

\begin{lem}\label{l-u1t-dual}
\index{U(1) torsor@$\UU(1)$ torsor!dual}
\index{dual!U(1) torsor@$\UU(1)$ torsor}
\index{torsor}
(The Dual of a $\UU(1)$ Torsor).
Given a $\UU(1)$ torsor $T$, the dual $T^{*}$ of definition \ref{d-u1t-dual} is a $\UU(1)$ torsor.
\end{lem}
\begin{proof}
The same topological space as $T$, $T^{*}$ is Hausdorff.  Let the action for $T$ be $\rho$.  The action $\rho^{*} \colon \UU(1) \cross T^{*} \rightarrow T^{*}$ is continuous, since as a map of topological spaces it is the composition $\UU(1) \cross T \xrightarrow{inv \cross \ident} \UU(1) \cross T \xrightarrow{\rho} T$ of continuous maps, where $inv \colon z \mapsto z^{-1}$.  The action $\rho^{*}$ is free and transitive because $\rho$ is free and transitive and $inv$ is a bijection.  Thus $T^{*}$ is a $\UU(1)$ torsor.
\end{proof}

\begin{lem}\label{l-u1t-dual-cano-doub-alt}
\index{U(1) torsor@$\UU(1)$ torsor!dual}
\index{dual!U(1) torsor@$\UU(1)$ torsor}
\index{omega@$\omega$}
\index{torsor}
(Uses of ${}^{*}$; the Double Dual; an Alternative Dual of a $\UU(1)$ Torsor).
We denote the map $({}^{*})^{-1}$ also by ${}^{*}$.  Then given a $\UU(1)$ torsor $T$, $T^{**} = T$.

Given $z \in \UU(1)$ and $t \in T$,
\[
\overline{z} t^{*} = \rho^{*} (\overline{z}, t^{*}) = (\rho (z, t))^{*} = (z t)^{*}, \notag
\]
as elements of torsors; ${}^{*}$ is an inverse-equivariant homeomorphism of the topological spaces underlying the $\UU(1)$ torsors.

An alternative dual to $T^{*}$ is $\Hom(T, \UU(1))$:  the continuous bi-equivariant pairing $ T \cross \Hom(T, \UU(1)) \rightarrow \UU(1)$, given for $\sigma \in \Hom(T, \UU(1))$ by $(\sigma, t) \mapsto \sigma(t)$, induces a natural isomorphism of $\UU(1)$ torsors, $\omega \colon \Hom(T, \UU(1)) \rightarrow T^{*}$, $\sigma \mapsto t^{*}$ where $\sigma(t) = 1$, or $\sigma ((\omega(\sigma))^{*}) = 1$.
\end{lem}
\begin{proof}
That the natural pairing is bi-equivariant follows from the definitions.  That it is continuous follows from $\abs{\sigma_1(t_1) - \sigma_2(t_2)} \le \abs{\sigma_1(t_1) - \sigma_1(t_2)} + \abs{\sigma_1(t_2) - \sigma_2(t_2)} \le \abs{\sigma_1(t_1) - \sigma_1(t_2)} + d(\sigma_1, \sigma_2)$ and continuity of $\sigma_1$.

To reason about $\omega$, denote the action for $T$ by $\rho$ and take any $t_0 \in T$; then $\sigma ((\sigma(t_0))^{-1} t_0) = 1$, and we can let $t = (\sigma(t_0))^{-1} t_0$.  Well-definedness of $\omega$ follows because elements of $\Hom(T, \UU(1))$ are injective.  Injectivity of $\omega$ holds because, for $\sigma_1, \sigma_2 \in \Hom(T, \UU(1))$, $\sigma_1(t) = \sigma_2(t) \Rightarrow \sigma_2^{-1} \circ \sigma_1(t) = t$, and being equivariant, $\sigma_2^{-1} \circ \sigma_1$ is the identity.  Surjectivity of $\omega$ holds because given $t$, choose any $\sigma_0 \in \Hom(T, \UU(1))$; then $((\sigma_0(t))^{-1} \sigma_0)(t) = 1$.  Since $\omega$ by using its defining equation is $\UU(1)$-equivariant, by corollary \ref{co-g-equi-cont} it is an isomorphism of $\UU(1)$ torsors.
\end{proof}

\begin{defn}\label{d-u1t-hom}
\index{U(1) torsor@$\UU(1)$ torsor!Hom}
\index{Hom@$\Hom$!U(1) torsors@$\UU(1)$ torsors}
\index{torsor}
(The $\Hom$ of $\UU(1)$ Torsors).
Given two $\UU(1)$ torsors $T_1$, $T_2$ with respective actions $\rho_1, \rho_2$, define
\[
\Hom(T_1, T_2) = \{ \phi \colon T_1 \rightarrow T_2 \st \phi \text{ is } \UU(1)\text{-equivariant} \}, \notag
\]
with the compact-open topology.  Use the natural action $\rho_H \colon \UU(1) \cross \Hom(T_1, T_2) \rightarrow \Hom(T_1, T_2)$ given for $z \in \UU(1)$, $t_1 \in T_1$ by $(z, \phi) \mapsto (t_1 \mapsto \rho_2(z, \phi(t_1)))$.

This definition is written for a left action, but, because $\UU(1)$ is commutative, may be used mutatis mutandis for a right action also.
\end{defn}

\begin{lem}\label{l-u1t-hom}
\index{U(1) torsor@$\UU(1)$ torsor!Hom@$\Hom$}
\index{Hom@$\Hom$!U(1) torsors@$\UU(1)$ torsors}
\index{torsor}
(The $\Hom$ of $\UU(1)$ Torsors).
Given $\UU(1)$ torsors $T_1$, $T_2$, $\Hom(T_1, T_2)$ of definition \ref{d-u1t-hom} is a $\UU(1)$ torsor.
\end{lem}
\begin{proof}
By corollary \ref{co-g-equi-cont}, elements of $\Hom(T_1, T_2)$ are isomorphisms of $\UU(1)$ torsors, and in particular, continuous.  By \citet[page~258]{Dugu66}, since $T_2$ is Hausdorff and $\Hom(T_1, T_2)$ has the compact-open topology, $\Hom(T_1, T_2)$ is Hausdorff.  Giving $\Hom(T_1, T_2)$ the compact-open topology is equivalent, using any isomorphism of $\UU(1)$ torsors $T_2 \rightarrow \UU(1)$ and the metric on $T_2$ thus induced by the absolute value on $\CC$, to the topology of uniform convergence.  Given $z_1, z_2 \in \UU(1)$, $\sigma_1, \sigma_2 \in \Hom(T_1, T_2)$, and choosing some $t_0 \in T$, we have $d(z_1 \sigma_1, z_2 \sigma_2) = \max_{t \in T} (\abs{z_1 \sigma_1(t) - z_2 \sigma_2(t)}) \le \max_{t \in T} (\abs{z_1 \sigma_1(t) - z_1 \sigma_2(t)}) + \max_{t \in T} (\abs{z_1 \sigma_2(t) - z_2 \sigma_2(t)}) = \max_{t \in T} (\abs{\sigma_1(t) - \sigma_2(t)}) + \abs{z_1 - z_2} = d(\sigma_1, \sigma_2) + \abs{z_1 - z_2}$, whence $\rho_H$ is continuous.

Freeness of the action $\rho_H$ follows from freeness of the action for $T_2$ by evaluating elements of $\Hom(T_1, T_2)$ at points in $T_1$.  Transitivity is a result of transitivity of the action for $T_2$ and the fact that an element of $\Hom(T_1, T_2)$, being $\UU(1)$-equivariant, is determined by its value at one element of $T_1$.

Note that by equivariance of $\phi$, $\rho_{H} (z, \phi) (t_1) =  \rho_2(z, \phi(t_1)) = \phi(\rho_1(z, t_1))$ so the action $\rho_1$ has not been neglected in the definition.
\end{proof}

\begin{lem}\label{l-u1t-cano-isom}
\index{U(1) torsor@$\UU(1)$ torsor!canonical isomorphisms}
\index{canonical isomorphisms!U(1) torsors@$\UU(1)$ torsors}
\index{isomorphisms!U(1) torsors@$\UU(1)$ torsors}
\index{torsor}
(Canonical Isomorphisms of $\UU(1)$ Torsors).
Given $\UU(1)$ torsors $T, T_1, T_2, T_3$, and $\UU(1)$ the trivial torsor with action given by multiplication, there are natural isomorphisms of $\UU(1)$ torsors, for $z, z_{st} \in \UU(1)$, $s, t \in T$, $t_1 \in T_1$, $t_2 \in T_2$, $\sigma \in \Hom(T, \UU(1))$, $\sigma_1 \in \Hom(T_1, \UU(1))$, $\phi_{12} \in \Hom(T_1, T_2)$, $\phi_{23} \in \Hom(T_2, T_3)$, defining the natural pairing $\langle , \rangle$:
\begin{align}
T_1 \tensor (T_2 \tensor T_3) &\cong (T_1 \tensor T_2) \tensor T_3 \notag \\
T_1 \tensor T_2 &\cong T_2 \tensor T_1, \text{ } t_1 \tensor t_2 \mapsto t_2 \tensor t_1 \notag \\
\UU(1) \tensor T &\cong T, \text{ } z \tensor t \mapsto z t \notag \\
T^{*} \tensor T &\cong \UU(1), \text{ } s^{*} \tensor t \mapsto z_{st} \text{ where } t = z_{st} s = \langle t, s \rangle s \notag \\
T &\cong T^{**} \text{ (they are equal), } t \mapsto t \notag \\
\UU(1)^{*} &\cong \UU(1), \text{ } z^{*} \text{ (} = z \text{ as a set element) }\mapsto \overline{z} \notag \\
(T_1 \tensor T_2)^{*} &\cong T_1^{*} \tensor T_2^{*}, \text{ } (t_1 \tensor t_2)^{*} \mapsto t_1^{*} \tensor t_2^{*} \notag \\
\Hom(T_1, T_2) \tensor \Hom(T_2, T_3) &\cong \Hom(T_1, T_3), \text{ } \phi_{12} \tensor \phi_{23} \mapsto \phi_{23} \circ \phi_{12} \notag \\
\Hom(T, \UU(1)) &\cong T^{*}, \text{ } \sigma \mapsto \omega(\sigma) \notag \\
\Hom(\UU(1), T) &\overset{\sim}{\underset{\ev}\rightarrow} T, \text{ } \sigma^{-1} \mapsto \sigma^{-1}(1) \notag \\
T_1^{*} \tensor T_2 &\cong \Hom(T_1, T_2), \text{ } \sigma_1 \tensor t_2 \mapsto \ev^{-1} (t_2) \circ \omega^{-1} (\sigma_1) \notag \\
\Hom (T, T) &\cong \UU(1), \text{ } \phi \mapsto t^{-1} \phi(t) \notag.
\end{align}
We will generally treat the first two isomorphisms, for associativity and commutativity of tensor products, as identifications.
\end{lem}
\begin{proof}
The tensor product of $\UU(1)$ torsors satisfies a universal property analogous to that for modules over commutative rings, from which follow associativity, commutativity and ``functoriality'' (we will use the notation of tensor products of maps without other comment) as in \citet[pages~601--607]{Lang93}, mutatis mutandis, especially replacing linearity with equivariance.

To prove $\UU(1) \tensor T \cong T$, define $\phi \colon \UU(1) \tensor T \rightarrow T$ by $\phi ([(z, t)]) = \rho (z, t) = z t$.  It is well-defined because $\rho (z_1 z, z_1^{-1} t) = \rho (z_1 z, \rho (z_1^{-1}, t)) = \rho (z_1 z z_1^{-1}, t) = \rho (z, t)$, and it is similarly equivariant.  Set $t$ to any fixed element of $T$ to see that $\phi$ is surjective, and it's injective because if $z_1 t_1 = z_2 t_2$, then defining $w \in \UU(1)$ by $w t_1 = t_2$ or $t_1 = w^{-1} t_2$, we have $z_1 t_1 = z_2 t_2 = w z_2 w^{-1} t_2 = w z_2 t_1$, whence $z_1 = w z_2$, and $[(z_1, t_1)] = [(w z_2, w^{-1} t_2)] = [(z_2, t_2)]$.  To see that $\phi$ is continuous at an arbitrary point $[(z_0, t_0)]$, let $U$ be any open neighborhood of $\phi([(z_0, t_0)]) = \rho(z_0, t_0)$.  Since $\rho$ is continuous, there are open neighborhoods $V$ of $z_0$ and $W$ of $t_0$ such that $\rho (V \cross W) \subset U$.  Let $\pi_Q \colon T \cross \UU(1) \rightarrow T \tensor \UU(1)$ be the quotient map of definition \ref{d-u1t-tens}; then $\phi^{-1} (U) = \pi_Q (\{ 1 \} \cross U) \supset \pi_Q (\{ 1 \} \cross \rho (V \cross W)) = \pi_Q (V \cross W)$, which contains $[(z_0, t_0)]$, and is open because $\pi_Q$ is open \citep[page~22]{tomD87}.  To see that $\phi^{-1}$ is continuous, note that $\phi^{-1} (t) = [(t, 1)]$, which is the composition $\pi_Q \circ (t \mapsto (t, 1))$.

To prove that $T^{*} \tensor T \cong \UU(1)$, note that the special case of lemma \ref{l-pb-tran-func-tors} where the base space is a point, implies that there is a continuous unique function $\tau \colon T \cross T \rightarrow \UU(1)$ such that $\tau(s, t) s = t$, whence the $z_{st}$ of the statement is $\tau(s, t)$.  It is inverse-equivariant with respect to the first factor and equivariant with respect to the second.  Then $\tau \circ (({}^{*})^{-1} \cross \ident) \colon T^{*} \cross T \rightarrow \UU(1)$ is a bi-equivariant map that descends to the continuous equivariant map of of the statement, which is thus an isomorphism of $\UU(1)$ torsors.

To see that $\UU(1)^{*} \cong \UU(1)$, note that complex conjugation is a homeomorphism and the resulting map here is equivariant:  for $w, z \in \UU(1)$, $z w^{*} = (\overline{z} w)^{*} \mapsto \overline{\overline{z} w} = z \overline{w}$.  We have not tried to define functors to show naturality in the technical sense elsewhere, but to allay concerns of unnaturality when identifying this torsor and its dual, note that it is natural in the technical sense, between two constant functors with respective values $\UU(1)^{*}$ and $\UU(1)$.

To see that $(T_1 \tensor T_2)^{*} \cong T_1^{*} \tensor T_2^{*}$, define the related map in the reverse direction, $T_1^{*} \cross T_2^{*} \rightarrow (T_1 \tensor T_2)^{*}$, by $t_1^{*} \cross t_2^{*} \mapsto (t_1 \tensor t_2)^{*}$, which being bi-equivariant, descends to the tensor product.

To check the isomorphisms relating to $\Hom$, note that they are $\UU(1)$-equivariant, and hence by corollary \ref{co-g-equi-cont} are isomorphisms of $\UU(1)$ torsors.  The definition of the last one doesn't depend on the element $t \in T$.
\end{proof}

\begin{note}\label{n-u1t-tens-dual-hom}
\index{U(1) torsors@$\UU(1)$ torsors!category}
\index{U(1) torsors@$\UU(1)$ torsors!identifications}
(Categorical Notes on $\UU(1)$ Torsors).
We could think of an abelian group structure up to canonical isomorphism for $\UU(1)$ torsors, with the tensor product as the group operation, $\UU(1)$ as the identity element, and the inverse given by $T^{-1} = T^{*}$.  Various properties that one would expect would follow from this structure.  To deal with elements of the torsors we might still need to know the specific canonical isomorphisms.

This could be made more precise.  We've gone some way in proving that $\UU(1)$ torsors and their morphisms form a closed symmetric monoidal category \citep[pages~162--163,~184,~251--252]{MacL00}, since our canonical isomorphisms are natural transformations of certain functors.  Our category has canonical isomorphisms $T^{**} \cong T$ (in fact, equality), so \citet[pages~301--304]{Solo91} gives the result that all the allowable diagrams commute, that are made from allowable natural transformations of allowable functors.  He says that allowable functors are those obtained from the following by composition:  the identity functor, the constant functor with value on objects $\UU(1)$ and on morphisms $\ident_{\UU(1)}$, the bifunctors $\tensor$ and $\Hom$; and allowable natural transformations are those that can be obtained by applications of $\tensor$, $\Hom$, $\pi$, $\pi^{-1}$, where $\pi \colon \Hom(T_1 \tensor T_2, T_3) \isomto \Hom(T_1, \Hom(T_2, T_3))$, and composition including with allowable functors, from the natural transformations $a, b, c, a^{-1}, b^{-1}$, where $a \colon (T_1 \tensor T_2) \tensor T_3 \rightarrow T_1 \tensor (T_2 \tensor T_3)$, $b \colon T \tensor \UU(1) \rightarrow T$, and $c \colon T_1 \tensor T_2 \isomto T_2 \tensor T_1$.
\end{note}
However, we leave this at a more informal level.

\section{Principal U($1$) Bundles}\label{s-u1-pb}

\begin{defn}\label{d-u1pb-tens}
\index{principal U(1) bundle@principal $\UU(1)$ bundle!tensor product}
\index{tensor product!principal U(1) bundles@principal $\UU(1)$ bundles}
\index{principal bundle}
(Tensor Products of Principal $\UU(1)$ Bundles).
Given two principal $\UU(1)$ bundles $\pi_1 \colon P_1 \rightarrow B$, $\pi_2 \colon P_2 \rightarrow B$ over the same base, with respective actions $\rho_1, \rho_2$, define a $\UU(1)$ action on $P_1 \cross_B P_2$ by
\begin{align}
\rho_Q \colon (P_1 \cross_B P_2) \cross \UU(1) &\rightarrow (P_1 \cross_B P_2) \notag \\
((p_1, p_2), z) &\mapsto (p_1 z, p_2 z^{-1}), \notag
\end{align}
and define the tensor product of $P_1$ and $P_2$, $P_1 \tensor P_2$, as the corresponding quotient, $(P_1 \cross_B P_2) / \UU(1)$ \citep[page~2]{tomD87}.  Define the map
\begin{align}
\widetilde{\rho_1 \tensor \rho_2} \colon (P_1 \cross_B P_2) \cross \UU(1) &\rightarrow (P_1 \cross_B P_2) \notag \\
((p_1, p_2), z) &\mapsto (p_1, p_2 z). \notag
\end{align}
This is $\UU(1)$-equivariant with respect to the action $\rho_Q$ on $P_1 \cross_B P_2$ and the trivial action of $\UU(1)$ on itself.  Define $\rho_1 \tensor \rho_2 \colon (P_1 \tensor P_2) \cross \UU(1) \rightarrow P_1 \tensor P_2$ as the continuous map induced by $\widetilde{\rho_1 \tensor \rho_2}$ on the quotients; a continuous equivariant map descends to a continuous map of the orbit spaces \citep[page~4]{tomD87}.  Let $p_1 \tensor p_2$ denote $[(p_1, p_2)]$.

Define $P_1 \tensor P_2$ as the total space, $\rho_1 \tensor \rho_2$ as the $\UU(1)$ action, and $\pi_{P_1 \tensor P_2} \colon P_1 \tensor P_2 \rightarrow B$ as the map that descends from the map on the fiber product, $(p_1, p_2) \mapsto \pi_1 (p_1)$, as the projection of a principal $\UU(1)$ bundle.
\end{defn}

\begin{lem}\label{l-u1pb-tens}
\index{principal U(1) bundle@principal $\UU(1)$ bundle!tensor product}
\index{tensor product!principal U(1) bundles@principal $\UU(1)$ bundles}
\index{principal bundle}
(Tensor Products of Principal $\UU(1)$ Bundles).
Given principal $\UU(1)$ bundles $\pi_1 \colon P_1 \rightarrow B$, $\pi_2 \colon P_2 \rightarrow B$ over the same base, their tensor product $\pi_{P_1 \tensor P_2} \colon P_1 \tensor P_2 \rightarrow B$ as in definition \ref{d-u1pb-tens}, is a principal $\UU(1)$ bundle.
\end{lem}
\begin{proof}
Use the notation of the definition.  As sets, definition \ref{d-u1pb-tens} on each fiber, and definition \ref{d-u1t-tens} give the same torsor, and continuity of the action in the bundle definition implies continuity of its restriction to each fiber.  Over each $b \in B$, the map $\rho_1 \tensor \rho_2$ is a free and transitive $\UU(1)$ action by reasoning analogous to that in lemma \ref{l-u1t-tens}.

Let the projections of the original bundles be $\pi_1, \pi_2$, and let $\{ U_i \}$ be an open cover of $B$ giving equivariant local trivializations for both original bundles: $\phi_{i1} \colon \pi_1^{-1} (U_i) \rightarrow U_i \cross \UU(1)$, $\phi_{i2} \colon \pi_2^{-1} (U_i) \rightarrow U_i \cross \UU(1)$.  Then $P_1 \cross_B P_2$ has local trivializations $\phi_{i1} \cross_B \phi_{i2} \colon (\pi_1 \cross_B \pi_2)^{-1} (U_i) \isomto U_i \cross \UU(1) \cross \UU(1)$, homeomorphisms that are equivariant for each of the fiber product factors.  Given $b \in U_i$, $\phi_{i1} \cross_B \phi_{i2} \colon (P_1 \cross_B P_2)_b \isomto \{ b \} \cross \UU(1) \cross \UU(1)$.

The equivariance and fiber-preserving properties of the homeomorphism $\phi_{i1} \cross_B \phi_{i2}$ imply that it carries (in particular is equivariant for) the bundle action $\rho_Q$ on $(\pi_1 \cross_B \pi_2)^{-1} (U_i)$ to the product of the trivial action on $U_i$ and the torsor action $\rho_Q$ on $\UU(1) \cross \UU(1)$.  Thus the map $(\phi_{i1} \cross_B \phi_{i2})_{\rho_Q} \colon \pi_{P_1 \tensor P_2}^{-1} (U_i) \rightarrow U_i \cross \UU(1) \tensor \UU(1)$ that it descends to, a bijection, is a homeomorphism of the quotients \citep[page~4]{tomD87}, and in particular gives a homeomorphism $(P_1 \tensor P_2)_b \isomto \{ b \} \cross \UU(1) \cross \UU(1)$.

Again because of the equivariance properties of $\phi_{i1} \cross_B \phi_{i2}$, $(\phi_{i1} \cross_B \phi_{i2})_{\rho_Q}$ is $\UU(1)$-equivariant with respect to the bundle action $\rho_1 \tensor \rho_2$ on $\pi_{P_1 \tensor P_2}^{-1} (U_i)$ and the product of the trivial action on $U_i$ and the torsor action $\rho_1 \tensor \rho_2$ on $\UU(1) \tensor \UU(1)$.  Restricting, the bundle action on $(P_1 \tensor P_2)_b$ corresponds exactly to the torsor action on $\UU(1) \tensor \UU(1)$.

We would like the final local trivialization to have standard fiber $\UU(1)$.  Compose the canonical torsor isomorphism $\UU(1) \tensor \UU(1) \cong \UU(1)$ with the second component of $(\phi_{i1} \cross_B \phi_{i2})_{\rho_Q}$ to get $\phi_{i1} \tensor \phi_{i2} \colon \pi_{P_1 \tensor P_2}^{-1} (U_i) \rightarrow U_i \cross \UU(1)$, which is $\UU(1)$-equivariant for the bundle action $\rho_1 \tensor \rho_2$ on $\pi_{P_1 \tensor P_2}^{-1} (U_i)$ and the product of the trivial action on $U_i$ and the torsor action $\rho_1 \tensor \rho_2$ on $\UU(1) \tensor \UU(1)$.
\end{proof}
\begin{note}\label{n-u1pb-tens}
\index{tensor product!principal U(1) bundles@principal $\UU(1)$ bundles!transition functions}
(Example Alternative Proof Using Transition Functions).
Another viewpoint that can be useful is that of transition functions.  Here is an indication of how specifying the topology on the total space of the tensor product could be approached by taking as the transition functions for the tensor product of the bundles, the product in $\UU(1)$ of transition functions for the two tensor product factors.

Use the preceding notation, adding $U_j$ such that $U_i \cap U_j \ne \emptyset$, denote the transition functions by $g_{j i q} \colon (U_j \cap U_i) \cross \UU(1) \rightarrow \UU(1)$, $q = 1, 2$. Let $b \in U_i \cap U_j$; then $(P_1)_b$, $(P_2)_b$, and $(P_1 \tensor P_2)_b = (P_1)_b \tensor (P_2)_b$ are $\UU(1)$ torsors homeomorphic to $\UU(1)$.  Consider the following commutative diagram of $\UU(1)$ torsors, which omits the $\{b\} \cross$'s of the local trivializations, and omits the $j$ trivializations:
\[
\begindc{\commdiag}[5]
\obj(10,25)[objP1b]{$(P_1)_b$}
\obj(18,25)[objCtop]{$\cross_B$}
\obj(26,25)[objP2b]{$(P_2)_b$}
\obj(46,25)[objP1cP2b]{$(P_1 \cross_B P_2)_b$}
\obj(66,25)[objP1tP2b]{$(P_1 \tensor P_2)_b$}
\obj(10,10)[objU1b]{$\UU(1)$}
\obj(18,10)[objCbot]{$\cross$}
\obj(26,10)[objU2b]{$\UU(1)$}
\obj(46,10)[objU1cU2b]{$\UU(1) \cross \UU(1)$}
\obj(66,10)[objU1tU2b]{$\UU(1) \tensor \UU(1)$}
\obj(86,10)[objUb]{$\UU(1)$}
\mor{objP1b}{objU1b}{$\phi_{i1}$}
\mor{objP2b}{objU2b}{$\phi_{i2}$}
\mor{objP1cP2b}{objU1cU2b}{$\phi_{i1} \cross_B \phi_{i2}$}
\mor{objP1tP2b}{objU1tU2b}{$\phi_{i1} \tensor \phi_{i2}$}
\mor(30,25)(40,25){}[\atleft,\dashArrow]
\mor{objP1cP2b}{objP1tP2b}{$\pi$}
\mor(30,10)(40,10){}[\atleft,\dashArrow]
\mor{objU1cU2b}{objU1tU2b}{$\pi$}
\mor{objU1tU2b}{objUb}{$\cong$}
\cmor((7,9)(6,8)(5,6)(5,4)(6,2)(7,1)(9,0)(10,0)(11,0)(13,1)(14,2)(15,4)(15,6)(14,8)(13,9)) \pup(10,-2){$g_{j i 1}$}
\cmor((23,9)(22,8)(21,6)(21,4)(22,2)(23,1)(25,0)(26,0)(27,0)(29,1)(30,2)(31,4)(31,6)(30,8)(29,9)) \pup(26,-2){$g_{j i 2}$}
\cmor((43,9)(42,8)(41,6)(41,4)(42,2)(43,1)(45,0)(46,0)(47,0)(49,1)(50,2)(51,4)(51,6)(50,8)(49,9)) \pup(46,-2){$g_{j i 1} \cross g_{j i 2}$}
\cmor((63,9)(62,8)(61,6)(61,4)(62,2)(63,1)(65,0)(66,0)(67,0)(69,1)(70,2)(71,4)(71,6)(70,8)(69,9)) \pup(66,-2){$g_{j i 1} \tensor g_{j i 2}$}
\cmor((83,9)(82,8)(81,6)(81,4)(82,2)(83,1)(85,0)(86,0)(87,0)(89,1)(90,2)(91,4)(91,6)(90,8)(89,9)) \pup(86,-2){$g_{j i 1} g_{j i 2}$}
\enddc
\]
The diagram pretty much tells the story of constructing the local trivializations, using the shorthand notation of fiber and tensor products of local trivializations, when what are meant are fiber and tensor products of the $\UU(1)$ portions of them.  That the resulting change of coordinates functions for the tensor product of the bundles are continuous, follows from continuity of the transition functions, which follows from their formulae.  Thus we could conclude that despite the fiberwise definition, the gluing of the fibers is all right and the trivializations define the topology of the total space.
\end{note}

\begin{defn}\label{d-u1pb-dual}
\index{principal U(1) bundle@principal $\UU(1)$ bundle!dual}
\index{dual!principal U(1) bundle@principal $\UU(1)$ bundle}
\index{principal bundle}
(The Dual of a Principal $\UU(1)$ Bundle).
Given a principal $\UU(1)$ bundle $\pi \colon P \rightarrow B$ with action $\rho \colon P \cross \UU(1) \rightarrow P$, define its dual $P^{*}$, also called its inverse $P^{-1}$, as the same topological space and set, so that as topological spaces and sets we can write $P^{*} = P$.  For $p \in P$, let $p^{*} \in P^{*}$ denote the same element $p$.  That is, as an element of the topological space for $P$, which is also the topological space for $P^{*}$, $p^{*} = p$.

The action for $P^{*}$ is $\rho^{*}$, given for $z \in \UU(1)$, $p \in P = P^{*}$ by $p^{*} z= \rho^{*} (p^{*}, z) = (\rho (p, \overline{z}))^{*} = (p \overline{z})^{*}$, where the first equality $P = P^{*}$ is of sets, the second and fourth are previously defined notational equalities of elements of principal bundles, the third is the present definitional equality of elements of principal bundles; and $\overline{z} = z^{-1}$ on $\UU(1)$.  Denote the projection map for $P^{*}$ by $\pi^{*}$, equal to $\pi$ as a map of topological spaces.

Among its other meanings, let ${}^{*} \colon P \rightarrow P^{*}$ denote the identity map on the topological spaces, written ``exponentially'' as $p \mapsto p^{*}$.  When the argument of this map isn't a single symbol, use parentheses to enclose the argument, as in $(p \overline{z})^{*}$, with $p \in P$ as a principal bundle, which implies that $p \overline{z}= \rho (p, \overline{z}) \in P$.
\end{defn}
As with the $\UU(1)$ torsor definition, this  could be rewritten to make explicit in the notations the distinction between the principal $\UU(1)$ bundle, and the topological space and set, $P$, but this would be more cumbersome in practical use later.  Often, ${}^{*}$ applied to an element will be effectively a sneaky way to remind us when to interpret the result as being in the dual principal bundle to the principal bundle in which lies the element of the unadorned symbol; in particular to remind us which action to use.

\begin{lem}\label{l-u1pb-dual}
\index{principal U(1) bundle@principal $\UU(1)$ bundle!dual}
\index{dual!principal U(1) bundle@principal $\UU(1)$ bundle}
\index{principal bundle}
(The Dual of a Principal $\UU(1)$ Bundle).
Given a principal $\UU(1)$ bundle $\pi \colon P \rightarrow B$, its dual $\pi^{*} \colon P^{*} \rightarrow B$ as in definition \ref{d-u1pb-dual}, is a principal $\UU(1)$ bundle.
\end{lem}
\begin{proof}
Using the notation of the definition, since $\UU(1)$ is commutative and $P^{*} = P$ as a topological space, the definition's construction of $\rho^{*}$ on $P^{*}$ from that on $P$ gives an action, with continuity following from that of $\rho$ and complex conjugation.  Over each $b \in B$, the map $\rho^{*}$ is a free and transitive $\UU(1)$ action by reasoning analogous to that in lemma \ref{l-u1t-dual}.

As sets, definition \ref{d-u1pb-dual} on each fiber, and definition \ref{d-u1t-dual} give the same construction.  The continuity of $\rho^{*}$ implies continuity of its restrictions to fibers.   Over each $b \in B$, $\rho^{*}$ coincides with the action of the torsor dual of the fiber.

Let the projection of the original bundle be $\pi$, and let $\{ U_i \}$ be an open cover of $B$ giving equivariant local trivializations for the original bundle: $\phi_i \colon \pi^{-1} (U_i) \rightarrow U_i \cross \UU(1)$.  Then the same local trivializations, as maps of topological spaces, work for $P^{*}$, since it is the same topological space as $P$.  Let us re-label them $\psi^{*}_i \colon (\pi^{*})^{-1} (U_i) \rightarrow U_i \cross \UU(1)^{*}$ for use with $P^{*}$; $\psi^{*}_i (p^{*}) = (\phi_i(p))^{*}$, where we identify $(U_i \cross \UU(1))^{*}$ with $U_i \cross \UU(1)^{*}$.

Letting $b = \pi(p)$, $\phi_i(p) = (b, z_{p,i})$ with $z_{p,i} \in \UU(1)$, and equivariance of $\phi$ is formulated as $\phi_i(p z) = \phi_i(p) z = (b, z_{p,i} z)$.  Equivariance of $\psi^{*}_i$ follows: $\psi^{*}_i (p^{*} z) = \psi^{*}_i ((p \overline{z})^{*}) = (\phi_i (p \overline{z}))^{*} = ((b, z_{p,i}) \overline{z})^{*} = (b, (z_{p,i} \overline{z})^{*}) = (b, z_{p,i}^{*} z)$.

We would like the final local trivialization to have standard fiber $\UU(1)$.  Compose the canonical torsor isomorphism $\UU(1)^{*} \cong \UU(1)$ with the second component of $\psi^{*}_i$ to get what we call $\phi^{*}_i \colon (\pi^{*})^{-1} (U_i) \rightarrow U_i \cross \UU(1)$, which is $\UU(1)$-equivariant for the bundle action $\rho^{*}$ on $(\pi^{*})^{-1} (U_i)$ and the product of the trivial action on $U_i$ and the action of $\UU(1)^{*}$.
\end{proof}
If the original bundle's transition functions are $g_{j i}$, then those of the dual are $\overline{g_{j i}}$.

\begin{lem}\label{l-u1pb-dual-cano-doub-alt}
\index{principal U(1) bundle@principal $\UU(1)$ bundle!dual}
\index{dual!principal U(1) bundle@principal $\UU(1)$ bundle}
\index{principal bundle}
\index{Omega@$\Omega$}
(Uses of ${}^{*}$; the Double Dual of a Principal $\UU(1)$ Bundle).
We denote the map $({}^{*})^{-1}$ also by ${}^{*}$.  Then given a principal $\UU(1)$ bundle $P$, $P^{**} = P$.

Given $z \in \UU(1)$ and $p \in P$, $p^{*} \overline{z} = \rho^{*} (p^{*}, \overline{z}) = (\rho (p, z))^{*} = (p z)^{*}$, as elements of principal bundles; ${}^{*}$ is an inverse-equivariant homeomorphism of the topological spaces that are the total spaces of the principal $\UU(1)$ bundles.
\end{lem}
\begin{proof}
That the natural pairing is bi-equivariant follows from the definitions.  That it is continuous follows by using local trivializations and lemma \ref{l-u1t-dual-cano-doub-alt}, the result for torsors analogous to this one.
\end{proof}

We won't define $\Hom$ for principal bundles.  The situation is not as simple as for torsors since the topology of the base space comes into play.

\begin{lem}\label{l-u1pb-cano-isom}
\index{principal U(1) bundle@principal $\UU(1)$ bundle!canonical isomorphisms}
\index{canonical isomorphisms!principal U(1) bundles@principal $\UU(1)$ bundles}
\index{isomorphisms!principal U(1) bundles@principal $\UU(1)$ bundles}
\index{principal bundle}
(Canonical Isomorphisms of Principal $\UU(1)$ Bundles).
Given principal $\UU(1)$ bundles $P, P_1, P_2, P_3$, and $B \cross \UU(1)$ the trivial principal bundle with action given by multiplication, all over the same base $B$, there are canonical natural isomorphisms of principal $\UU(1)$ bundles, given on fibers by the canonical isomorphisms for $\UU(1)$ torsors of lemma \ref{l-u1t-cano-isom}:
\begin{align}
P_1 \tensor (P_2 \tensor P_3) &\cong (P_1 \tensor P_2) \tensor P_3 \notag \\
P_1 \tensor P_2 &\cong P_2 \tensor P_1 \notag \\
(B \cross \UU(1)) \tensor P &\cong P \notag \\
P^{*} \tensor P &\cong B \cross \UU(1) \notag \\
P &\cong P^{**} \text{ (by } \ident \text{; they are equal)} \notag \\
(P_1 \tensor P_2)^{*} &\cong P_1^{*} \tensor P_2^{*}. \notag
\end{align}
We will generally treat the first two isomorphisms, for associativity and commutativity of tensor products, as identifications.
\end{lem}
\begin{proof}
Associativity and commutativity follow from a universal property of principal $\UU(1)$ bundles; see the proof of lemma \ref{l-u1t-cano-isom} for more words and a reference.

These isomorphisms follow directly from the fiberwise canonical isomorphisms; each one is a $\UU(1)$-equivariant map covering the identity.  Further, using local trivializations and canonical isomorphisms for various expressions in the trivial torsor $\UU(1)$, each one is continuous.  Then by lemma \ref{l-pb-mor-cov-id-isom} it's a principal $\UU(1)$ bundle isomorphism.
\end{proof}
The existence of the isomorphisms can also be obtained by looking at transition functions.  For example, if $g_{ji}$ is the transition function for $P$ over $U_i \cap U_j$, then $g_{ji}^{*}$, its complex conjugate, is that for $P^{*}$, whence that for $P^{*} \tensor P$ is their product, which is the constant function $1$, corresponding to the trivial bundle.
\begin{note}\label{n-u1pb-tens-dual}
\index{principal U(1) bundles@principal $\UU(1)$ bundles!identifications}
(Properties and Categorical Notes for Principal $\UU(1)$ Bundles).
The analogs of the general statements not involving $\Hom$ for torsors in note \ref{n-u1t-tens-dual-hom} are valid for principal bundles.  (The more precise category theoretical statements in that note involve $\Hom$.)  In addition, up to canonical isomorphisms that we will treat as identifications, dual and tensor product commute with pullbacks.  We will also treat as identifications the canonical isomorphisms of pullbacks via compositions of maps.  None of these isomorphisms, whether for associativity or commutativity of tensor products or related to pullbacks, introduces $\UU(1)$ phase factors.
\end{note}

\section{Continuous Bundle Gerbe Definition}\label{s-cont-bg-def}

\begin{defn}\label{d-y-n}
\index{Y[n]@$Y^{[n]}$}
\index{fiber product}
(Fiber Product Powers of a Space).
Denote by $\pi_i$ the projections from the $r$-fold cartesian product of topological spaces to the $i$-th factor, by $\widehat{\pi}_i$ the projection to the $(r-1)$-fold cartesian product omitting the $i$-th factor, and by $\pi_{i_1 \dots i_p}$ the projections onto the $(r-p)$-fold cartesian product of factors $i_1 \dots i_p$, that is, omitting all factors except $i_1 \dots i_p$.

For $Y$ a topological space, $p < q \in \NN$, and a sequence $i_1 \dots i_q$ of $q$ integers each in the range $[1,p]$, define an extended version of the diagonal map,
\begin{align}
\Delta_p^{i_1 \dots i_q} \colon Y^p &\rightarrow Y^q \notag \\
     (y_1, \dots, y_p) &\mapsto (y_{i_1}, \dots, y_{i_q}). \notag
\end{align}

Given a continuous map $\pi \colon Y \rightarrow X$, for $r \in \NN$, the $r$-fold fiber product of $Y$, denoted $Y^{[r]}$, is the subspace of $r$-tuples $(y_1, \dots, y_r)$ in the $r$-fold cartesian product of $Y$ such that $\pi (y_i) = \pi (y_j)$ for all $i, j$.  Let $\pi_i$, $\widehat{\pi}_i$, $\pi_{i_1 \dots i_r}$, and the $\Delta_p^{i_1 \dots i_q}$ denote also the maps induced for fiber products.
\end{defn}
Note that the names for the projections are duplicated for different domains; e.g. $\pi_1 \colon Y^{[2]} \rightarrow Y$, but also $\pi_1 \colon Y^{[3]} \rightarrow Y^{[2]}$, so the context must be kept in mind.

\begin{defn}\label{d-cont-bndl-gerb}
\index{bundle gerbe!continous}
\index{paracompact}
\index{good cover}
(Continuous Bundle Gerbes).
Suppose $X$ and $Y$ are topological spaces, $X$ paracompact, and such that each open cover of $X$ has a refinement that is a good cover, meaning an open cover with all nonempty finite intersections contractible.  The object denoted variously $(P, Y)$, $(P, Y, X)$, $(P, p, Y, \pi, X)$, $(m, P, p, Y, \pi, X)$, or
\[
\begindc{\commdiag}[5]
\obj(10,30)[objP]{$P$}
\obj(10,20)[objYY]{$Y^{[2]}$}
\obj(20,20)[objY]{$Y$}
\obj(20,10)[objX]{$X$}
\mor{objP}{objYY}{$p$}
\mor(10,21)(20,21){$\pi_1$}
\mor(10,19)(20,19){$\pi_2$}
\mor{objY}{objX}{$\pi$}
\enddc
\]
is a ``continuous'' bundle gerbe over $X$ if
\begin{enumerate}
    \item $Y \rightarrow X$ is a continuous surjection with local sections.  Here, a local section on an open neighborhood $U$ of $x \in X$ doesn't need to start at any particular given point in $Y$ above $x$; it is just some continuous map $\sigma \colon U \rightarrow Y$ such that $\pi \circ \sigma = \ident_U$.
    \item $P \rightarrow Y^{[2]}$ is a continuous principal $\UU(1)$ bundle
    \item There is an associative bundle gerbe multiplication $m$, an isomorphism of continuous principal $\UU(1)$ bundles over $Y^{[3]}$,
        \[
        m \colon \pi_{12}^{*} P \tensor \pi_{23}^{*} P \isomto \pi_{13}^{*} P
        \]
        such that the following diagram of continuous principal $\UU(1)$ bundles over $Y^{[4]}$ commutes up to canonical natural isomorphism that we treat as identifications:
        \[
        \begindc{\commdiag}[5]
        \obj(10,30)[obj1234]{$\pi_{12}^{*} P \tensor \pi_{23}^{*} P \tensor \pi_{34}^{*} P$}
        \obj(50,30)[obj134]{$\pi_{13}^{*} P \tensor \pi_{34}^{*} P$}
        \obj(10,10)[obj124]{$\pi_{12}^{*} P \tensor \pi_{24}^{*} P$}
        \obj(50,10)[obj14]{$\pi_{14}^{*} P$}
        \mor{obj1234}{obj134}{$\pi_{123}^{*} m \tensor \ident$}
        \mor{obj1234}{obj124}{$\ident \tensor \pi_{234}^{*} m$}
        \mor{obj134}{obj14}{$\pi_{134}^{*} m$}
        \mor{obj124}{obj14}{$\pi_{124}^{*} m$}
        \enddc
        \]
        where the isomorphisms not shown are for associativity of tensor products and, in the construction of the pullbacks of $m$, composition of pullbacks; as indicated by long equal signs in the diagram of \citet[page~18]{Stev00}.
\end{enumerate}
\end{defn}
\begin{note}\label{n-bndl-gerb-mult-fibe}
\index{bundle gerbe!multiplication}
(Bundle Gerbe Multiplication on Fibers).
In terms of fibers $P_{y_i, y_j}$, which are $\UU(1)$ torsors, the bundle gerbe multiplication defines for any $(y_1, y_2, y_3) \in Y^{[3]}$, an isomorphism of $\UU(1)$ torsors, by abuse of notation also called $m$,
\[
m \colon P_{y_1, y_2} \tensor P_{y_2, y_3} \isomto P_{y_1, y_3}, \notag
\]
and similarly the associativity diagram of bundles leads, for any $(y_1, y_2, y_3, y_4) \in Y^{[4]}$, to a commutative diagram of isomorphisms of $\UU(1)$ torsors.

We will also abuse notation frequently by calling pullbacks of $m$, $m$.

On a practical note for proofs, it often seems easier to deal with bundle gerbe multiplication on fibers rather than using pullbacks to get bundle maps.  Reasoning about fibers can give the values of maps, but to verify continuity across fibers, it's necessary, whether explicitly or not, to use bundle maps in order to take advantage of the fact that bundle gerbe multiplication is one.
\end{note}

To understand the concept of a bundle gerbe it may help to suppose for a moment that $X$ consists of one point, $X = \{x\}$.  The picture to follow will also work for the part of any bundle gerbe that lies above a particular point $x \in X$.  For a larger set $X$, the ``fibers'' are glued together by the topology of $Y$.  What lies above $x$ can be viewed as a particular kind of groupoid, here a small category for which the set of objects can be identified with $Y_x$, such that for every pair of objects $(y_1, y_2) \in (Y^{[2]})_x = (Y_x)^2$ there is a set of morphisms $P_{y_1, y_2}$ - all isomorphisms - that is a $\UU(1)$ torsor (see definition \ref{d-g-tors}).  What lies above $x$ has been termed a groupoid with band $\UU(1)$.

The $\UU(1)$ torsors over some pairs of points of $(Y^{[2]})_x$, those pairs of the form $((y_1, y_2), (y_2, y_3))$, are related together by the bundle gerbe multiplication, which allows us to compose morphisms $P_{y_1, y_2}$ with those in $P_{y_2, y_3}$. \textit{If} we were to take morphisms in $P_{y_1, y_2}$ as mapping $y_1 \leftarrow y_2$, in the way we have defined subscripts for bundle transition functions, the notation for bundle gerbe multiplication and the usual notation for composition of morphisms would be in the same order.  However, sticking with the other notation, more closely allied to European languages than Arabic, we take $\phi_{12} \in P_{y_1, y_2}$ as an arrow $y_1 \rightarrow y_2$ and similarly with $\phi_{23} \in P_{y_2, y_3}$, and so $m(\phi_{12} \tensor \phi_{23}) \in P_{y_1, y_3}$ corresponds to $\phi_{23} \circ \phi_{12}$. Associativity of bundle gerbe multiplication translates to associativity of composition of morphisms.  The canonical isomorphisms in lemma \ref{l-bg-mult-cano-isom} to come, $P_{y_1, y_2} \cong P_{y_2, y_1}^{*}$ and $P_{y, y} \cong \UU(1)$, make sense in terms of inverses of isomorphisms and the identity map.  For the latter, if the morphisms between the pair of objects $(y,y)$ were actual maps, since the identity morphism would be distinguished, there would be a canonical identification of the $\UU(1)$ torsor with $\UU(1)$.

As applied in this thesis, the $\UU(1)$ torsors of the groupoid will be sets of unitary intertwiners between two Fock representations, and the bundle gerbe multiplication will be composition of unitary operators.  It will be useful to know not only if there is an isomorphism of a torsor with $\UU(1)$, but whether that isomorphism maps the identity intertwiner to $1$.  To allow stating this for general continuous bundle gerbes, we use the property $m(\iota \tensor \iota) = \iota$, which for an invertible (in particular, unitary) operator on a Hilbert space is satisfied only by the identity.  Let us make more formal some of the connections with groupoids just discussed.

\section{Multiplicative Identity, Inverse, and Related Isomorphisms}\label{s-bg-id-inv-isom}

A bundle gerbe's multiplication is related to the abelian group structures up to canonical isomorphism on $\UU(1)$ torsors and principal bundles mentioned in notes \ref{n-u1t-tens-dual-hom} and \ref{n-u1pb-tens-dual}.  This confirms the groupoid viewpoint discussed at the end of section \ref{s-cont-bg-def}.  Given a bundle gerbe $(P, Y, X)$, there are bundle gerbe multiplicative identity elements in each fiber $P_{y,y}$, and inverse elements in $P_{y_2, y_1}$ for elements in $P_{y_1, y_2}$, and continuity in a sense related to sections of suitable pullbacks of $P$.  We start with relevant torsor and bundle canonical isomorphisms.
\begin{lem}\label{l-bg-mult-cano-isom}
\index{bundle gerbe!multiplication related isomorphisms}
(Canonical Multiplication Related Isomorphisms).
Given a bundle gerbe $(P, Y, X)$, and defining $\tau \colon Y^{[2]} \rightarrow Y^{[2]}$ by $\tau(y_1, y_2) = \Delta_2^{21} (y_1, y_2) = (y_2, y_1)$,
\begin{align}
P_{y, y} &\cong \UU(1) \label{eq-u1t-yy-triv} \\
P_{y_2, y_1} &\cong P_{y_1, y_2}^{*} \label{eq-u1t-21-12-star} \\
(\Delta_2^{11})^{*} P &\cong Y^{[2]} \cross \UU(1) \label{eq-u1pb-yy-triv} \\
(\Delta_2^{22})^{*} P &\cong Y^{[2]} \cross \UU(1) \notag \\
\tau^{*} P &\cong P^{*}. \label{eq-u1pb-tau-star}
\end{align}
\end{lem}
\begin{proof}
From multiplication and lemma \ref{l-u1t-cano-isom}, we have for isomorphism \ref{eq-u1t-yy-triv}
\begin{align}
P_{y, y} \tensor P_{y, y} &\xrightarrow{m} P_{y, y} \notag \\
(P_{y, y} \tensor P_{y, y}) \tensor P_{y, y}^{*} &\cong P_{y, y} \tensor P_{y, y}^{*} \notag \\
P_{y, y} \tensor \UU(1) &\cong \UU(1) \notag \\
P_{y, y} &\cong \UU(1),  \notag
\end{align}
and for isomorphism \ref{eq-u1t-21-12-star},
\begin{align}
P_{y_1, y_2} \tensor P_{y_2, y_1} &\xrightarrow{m} P_{y_1, y_1} \cong \UU(1) \notag \\
P_{y_1, y_2}^{*} \tensor (P_{y_1, y_2} \tensor P_{y_2, y_1}) &\cong P_{y_1, y_2}^{*} \tensor \UU(1) \notag \\
P_{y_2, y_1} &\cong P_{y_1, y_2}^{*}. \notag
\end{align}
To deal with the bundles, consider what spaces they are over.  Both sides of the bundle ``equations'' in the lemma statement are principal $\UU(1)$ bundles over $Y^{[2]}$.  Since we use bundle gerbe multiplication, we will need to consider bundles over $Y^{[3]}$, but we won't need bundle gerbe multiplication associativity and $Y^{[4]}$.  Otherwise we will follow the corresponding proofs for the torsors.

Using the method of proof for \ref{eq-u1t-yy-triv} as a model for \ref{eq-u1pb-yy-triv}, we see that it should start by using bundle gerbe multiplication, which involves pullbacks of $P$ to $Y^{[3]}$, using it for fibers with all the $y$'s equal, but giving a result about bundles over $Y^{[2]}$.  As we will do many times, we will pull back $m$ and still call it $m$.  Pull back the bundle multiplication isomorphism by $\Delta_2^{111} \colon Y^{[2]} \rightarrow Y^{[3]}$:
\begin{align}
{\Delta_2^{111}}^{*} \pi_{12}^{*} P \tensor {\Delta_2^{111}}^{*} \pi_{23}^{*} P &\cong {\Delta_2^{111}}^{*} \pi_{13}^{*} P, \text{ or} \notag \\
(\Delta_2^{11})^{*} P \tensor (\Delta_2^{11})^{*} P &\cong (\Delta_2^{11})^{*} P, \notag
\end{align}
since $\pi_{12} \circ \Delta_2^{111} = \pi_{23} \circ \Delta_2^{111} = \pi_{13} \circ \Delta_2^{111} = \Delta_2^{11}$.  Using lemma \ref{l-u1pb-cano-isom}:
\begin{align}
((\Delta_2^{11})^{*} P \tensor (\Delta_2^{11})^{*} P) \tensor ((\Delta_2^{11})^{*} P)^{*} &\cong (\Delta_2^{11})^{*} P \tensor ((\Delta_2^{11})^{*} P)^{*} \notag \\
(\Delta_2^{11})^{*} P &\cong Y^{[2]} \cross \UU(1). \notag
\end{align}
The isomorphism for $\Delta_2^{22}$ is exactly similar.  To use \ref{eq-u1t-21-12-star} as a model for \ref{eq-u1pb-tau-star}, pull back the multiplication isomorphism by $\Delta_2^{121}$ and note that $\pi_{12} \circ \Delta_2^{121} = \ident_{Y^{[2]}}$, $\pi_{23} \circ \Delta_2^{121} = \tau$, and $\pi_{13} \circ \Delta_2^{121} = \Delta_2^{11}$; the rest is analogous to the case of torsors:
\begin{align}
{\Delta_2^{121}}^{*} \pi_{12}^{*} P \tensor {\Delta_2^{121}}^{*} \pi_{23}^{*} P &\cong {\Delta_2^{121}}^{*} \pi_{13}^{*} P, \text{ or} \notag \\
P \tensor \tau^{*} P &\cong (\Delta_2^{11})^{*} P \cong Y^{[2]} \cross \UU(1), \notag \\
P^{*} \tensor P \tensor \tau^{*} P &\cong P^{*} \tensor (Y^{[2]} \cross \UU(1)), \notag \\
\tau^{*} P &\cong P^{*}. \notag
\end{align}
\end{proof}

It turns out that the ideas of multiplicative identities and inverses, part of the conceptual groupoid viewpoint, are also useful in calculations in proofs later, providing some actual equations amidst many canonical isomorphisms.
\begin{defn}\label{d-bg-mult-id-inv}
\index{bundle gerbe!multiplicative identity, inverse}
\index{identity!element}
\index{inverse!element}
\index{identity!section}
\index{inverse!section}
\index{iota@$\iota$}
(Bundle Gerbe Multiplicative Identity and Inverse Elements and Sections).
Given a bundle gerbe $(P, Y, X)$ and $y \in Y$, call an element $\iota_{y,y} \in P_{y, y}$ such that $m(\iota_{y,y} \tensor \iota_{y,y}) = \iota_{y,y}$ a (multiplicative) identity element of $P_{y,y}$.

Call a continuous section $\iota \colon Y^{[2]} \rightarrow (\Delta_2^{11})^{*} P$ a (multiplicative) identity section of $(\Delta_2^{11})^{*} P$ if for every $y \in Y$, $\iota(y_1, y_2)$ is an identity element of $P_{\Delta_2^{11} (y_1, y_2)} = P_{y_1, y_1}$.  Similarly for $(\Delta_2^{22})^{*} P$.

Given a continuous map $f \colon W \rightarrow Y^{[2]}$, call a continuous section $\iota_{f} \colon W \rightarrow f^{*} (\Delta_2^{11})^{*} P$ a (multiplicative) identity section of $f^{*} (\Delta_2^{11})^{*} P$ if for every $w \in W$, $\iota_f (w)$, or more precisely its second component since $(f^{*} (\Delta_2^{11})^{*} P)_w = \{ w \} \cross P_{\Delta_2^{11} \circ f (w)}$, is an identity element of $P_{\Delta_2^{11} \circ f (w)}$.  Similarly for $f^{*} (\Delta_2^{22})^{*} P$.

Given $(y_1, y_2) \in Y^{[2]}$ and $p_{12} \in P_{y_1, y_2}$, call $p_{21} \in P_{y_2, y_1}$ a (multiplicative) inverse (element) of $p_{12}$ if $m(p_{12} \tensor p_{21}) = \iota_{11}$, with $\iota_{11}$ an identity element of $P_{y_1, y_1}$.  If so, we may write $p_{21} = p_{12}^{-1}$.

Define $\tau \colon Y^{[2]} \rightarrow Y^{[2]}$ by $\tau(y_1, y_2) = \Delta_2^{21} (y_1, y_2) = (y_2, y_1)$.  Given a continuous map $f \colon W \rightarrow Y^{[2]}$ and a continuous section $\alpha \colon W \rightarrow f^{*} P$, call a section $\beta \colon W \rightarrow f^{*} \tau^{*} P$ a (multiplicative) inverse (section) of $\alpha$ if $m(\alpha \tensor \beta) = \iota_f$, for $\iota_f$ a (multiplicative) identity section of $f^{*} (\Delta_2^{11})^{*} P$.  If so, we may write $\beta = \alpha^{-1}$.
\end{defn}
\begin{note}\label{n-bg-mult-fibe}
\index{bundle gerbe!multiplication}
(The Fiberwise Viewpoint is Helpful with Bundle Gerbe Multiplication).
It can be helpful to recall what points in $Y^{[2]}$ are involved in use of bundle gerbe multiplication; to look at a situation fiberwise at first.
\end{note}

The following innocuous or perhaps trivial facts were very helpful in constructing proofs of facts relating to the Dixmier-Douady class.
\begin{lem}\label{l-bg-mult-id-inv}
\index{bundle gerbe!multiplicative identity, inverse}
\index{identity!element}
\index{inverse!element}
\index{identity!section}
\index{inverse!section}
\index{iota@$\iota$}
(Bundle Gerbe Multiplicative Identity and Inverse Elements and Sections).
Given a bundle gerbe $(P, Y, X)$, for each $y \in Y$ there is a unique identity element $\iota_{y,y} \in P_{y, y}$.  There is a continuous unique identity section $\iota$ of $(\Delta_2^{11})^{*} P$, $\iota(y_1, y_2) = ((y_1, y_2), \iota_{\Delta_2^{11} (y_1, y_2)}) = (y_1, y_2, \iota_{y_1, y_1})$, obtained by multiplying any given section by a continuous $\UU(1)$-valued function.  Similarly for the continuous unique identity section $\kappa$ of $(\Delta_2^{22})^{*} P$.

Given $(y_1, y_2) \in Y^{[2]}$ and $p_{12} \in P_{y_1, y_2}$, $m(\iota_{11} \tensor p_{12}) = p_{12}$ and $m(p_{12} \tensor \iota_{22}) = p_{12}$.

Given a continuous map $f \colon W \rightarrow Y^{[2]}$, there is a unique identity section $\iota_f$ of $f^{*} (\Delta_2^{11})^{*} P$, and $\iota_f = \ident_W \cross (\iota \circ f)$.  Given a continuous section $\alpha \colon W \rightarrow f^{*} P$, $m(\iota_f \tensor \alpha) = \alpha$.  Similarly there is a unique identity section $\kappa_f = \ident_W \cross (\kappa \circ f)$ of $f^{*} (\Delta_2^{22})^{*} P$, and $m(\alpha \tensor \kappa) = \alpha$.

Given $(y_1, y_2) \in Y^{[2]}$ and $p_{12} \in P_{y_1, y_2}$, there is a unique multiplicative inverse $p_{12}^{-1} = p_{21} \in P_{y_2, y_1}$ of $p_{12}$, which satisfies as well as the defining property, $m(p_{21} \tensor p_{12}) = \iota_{22}$.  Similarly, there is a unique inverse section $\alpha^{-1} = \beta$ corresponding to $\alpha$, and as well as the defining property $m(\alpha, \beta) = \iota_f$, we have also $m(\beta, \alpha) = \kappa_f$.

Given a continuous $\zeta \colon W \rightarrow \UU(1)$, $(\zeta \alpha)^{-1} = \zeta^{-1} \alpha^{-1}$.

Suppose $f \colon W \rightarrow Y^{[3]}$ is a continuous map and $\alpha_{12}$, $\alpha_{23}$ are continuous sections of $f^{*} \pi_{12}^{*} P$, $f^{*} \pi_{23}^{*} P$ respectively.  Then $m(\alpha_{23}^{-1} \tensor \alpha_{12}^{-1}) = (m(\alpha_{12} \tensor \alpha_{23}))^{-1}$.
\end{lem}
\begin{proof}
We use the facts that $P$ is a principal $\UU(1)$ bundle and $P_{y, y}$ is a $\UU(1)$ torsor, properties of such torsors and bundles as in lemmas \ref{l-pb-tran-func-tors}, \ref{l-u1t-cano-isom} and \ref{l-u1pb-cano-isom}, the fact that $m$ is an isomorphism of such torsors or bundles as the case may be, and lemma \ref{l-bg-mult-cano-isom}.  First we prove the results for torsors.

Take any $p_y \in P_{y, y}$; then $m(p_y \tensor p_y) = z p_y$ for some $z \in \UU(1)$.  Hence $\iota_{y,y} = z^{-1} p_y$ is a multiplicative identity in $P_{y,y }$.  Conversely, since $m((v \iota_{y,y}) \tensor (v \iota_{y,y})) = v (v \iota_{y,y})$, $\iota_{y,y}$ is the only multiplicative identity in $P_{y, y}$.

Let $m(\iota_{11} \tensor p_{12}) = z p_{12}$; then $z p_{12} = m(\iota_{11} \tensor p_{12}) = m(m(\iota_{11} \tensor \iota_{11}) \tensor p_{12}) = m(\iota_{11} \tensor m(\iota_{11} \tensor p_{12})) = m(\iota_{11} \tensor (z p_12)) = z m(\iota_{11} \tensor p_{12}) = z^2 p_{12}$, whence $z = 1$ and $m(\iota_{11} \tensor p_{12}) = p_{12}$.  Similarly $m(p_{12} \tensor \iota_{22}) = p_{12}$.

Take any $p_{21} \in P_{y_2, y_1}$; then $m(p_{12} \tensor p_{21}) = z \iota_{11}$ for some $z \in \UU(1)$, whence $z^{-1} p_{21}$ is a multiplicative inverse of $p_{12}$.  Conversely, since $m(p_{12} \tensor (v (z^{-1} p_{21}))) = v \iota_{11}$, $z^{-1} p_{21}$ is the only multiplicative inverse of $p_{12}$.

To see that if $p_{21}$ is a (right) inverse for $p_{12}$, it is also a left inverse of $p_{12}$, note that $m(p_{21} \tensor p_{12}) = z \iota_{22}$ for some $z \in \UU(1)$.  Then $p_{12} = m(\iota_1 \tensor p_{12}) = m(m(p_{12} \tensor p_{21}) \tensor p_{12}) = m(p_{12} \tensor m(p_{21} \tensor p_{12})) = m(p_{12} \tensor (z \iota_{22})) = z p_{12}$, whence $z = 1$ and $m(p_{21} \tensor p_{12}) = \iota_{22}$.

The bundle statements can be proved somewhat similarly, replacing $z, v \in \UU(1)$ with continuous $\UU(1)$-valued functions $\zeta, \upsilon$, or by relying on the torsor statements.

The existence and uniqueness of $\iota$ can be seen using the method of proof of isomorphism \ref{eq-u1pb-yy-triv} of lemma \ref{l-bg-mult-cano-isom} by noting that  since $\pi_{ij} \circ \Delta_2^{111} = \Delta_2^{11}$ for $ij \in \{ 12, 23, 13 \}$ we can pull back the bundle gerbe multiplication isomorphism by $\Delta_2^{111}$ to obtain the following isomorphism of principal $\UU(1)$ bundles:
\[
(\Delta_2^{11})^{*} P \tensor (\Delta_2^{11})^{*} P \xrightarrow{m} (\Delta_2^{11})^{*} P . \notag
\]
Take any section $\eta$ of $(\Delta_2^{11})^{*} P$, which has global sections since it is a trivializable principal $\UU(1)$ bundle by lemma \ref{l-bg-mult-cano-isom}.  Since $m(\eta \tensor \eta)$ and $\eta$ are sections of the same principal $\UU(1)$ bundle $(\Delta_2^{11})^{*} P$, they differ by a factor of a continuous function $\zeta \colon Y^{[2]} \rightarrow \UU(1)$, whence by the reasoning used for torsors, $\iota = \zeta^{-1} \eta$ (here $\zeta^{-1} (y_1, y_2)$ means $(\zeta (y_1, y_2))^{-1}$) is an identity section, in fact the unique one.

For $\iota_f$, first we verify that for every $w \in W$, the rightmost component (the one that is an element of $P$) of $\iota_f (w) = (w, (\iota \circ f) (w))$ is an identity element of $P_{\Delta_2^{11} \circ f (w)}$.  Unraveling the definitions this is true since $\iota$ is an identity section of $(\Delta_2^{11})^{*} P$.  Furthermore, since the identity element of each $P_{\Delta_2^{11} \circ f (w)} = P_{\pi_1 \circ f(w), \pi_1 \circ f(w)}$ is unique and the other components of $\iota_f$ must be what they are, $\iota_f$ is unique.  Its continuity follows from its formula.

Then we check that $m(\iota_f, \alpha) = \alpha$.  First pull back the multiplication isomorphism $\pi_{12}^{*} P \tensor \pi_{23}^{*} P \xrightarrow{m} \pi_{13}^{*} P$ by $\Delta_2^{112}$ to obtain, since $\pi_{23} \circ \Delta_2^{112} = \ident_{Y^{[2]}} = \pi_{13} \circ \Delta_2^{112}$, $(\Delta_2^{11})^{*} P \tensor P \xrightarrow{m} P$, then pull that back by $f$ to get $(f^{*} (\Delta_2^{11})^{*} P) \tensor (f^{*} P) \xrightarrow{m} (f^{*} P)$.  We have that $\iota_f$ is a section of the first factor on the left, and $\alpha$ of the second.  Since $f^{*} P$ has a section, it is trivializable, and so the right hand side is some $\zeta \colon W \rightarrow \UU(1)$ times $\alpha$.  As in previous argument, $\zeta \alpha = m(\iota_f \tensor \alpha) = m(m(\iota_f \tensor \iota_f) \tensor \alpha)  = m(\iota_f \tensor m(\iota_f \tensor \alpha)) = m(\iota_f \tensor (\zeta \alpha)) = \zeta^2 \alpha$, where the second equality follows by evaluating at $w \in W$ from the fact that $\iota_f (w)$ is an identity element.  Thus $\zeta = \underline{1}$.  The proof for $\kappa_f$ is similar.

For $\beta$, pull back the bundle gerbe multiplication isomorphism by $\Delta_2^{121}$ to obtain $P \tensor (\tau^{*} P) \xrightarrow{m} (\Delta_2{11})^{*} P$, and pull that back by $f$ to get $(f^{*} P) \tensor (f^{*} \tau^{*} P) \xrightarrow{m} (f^{*} \Delta_2{11})^{*} P$.  We have that $\alpha$ is a section of the first factor on the left, $\iota_f$ is a section of the right, and $\beta$ will be a section of the second factor on the left.  Proceed as in previous argument, letting $\widetilde{\beta}$ be any section of $f^{*} \tau^{*} P$, which has sections because $f^{*} P$ does; because $f^{*} \tau^{*} P \cong f^{*} P^{*} \cong (f^{*} P)^{*}$ using isomorphism \ref{eq-u1pb-tau-star} of lemma \ref{l-bg-mult-cano-isom} and note \ref{n-u1pb-tens-dual}.  Then $m(\alpha, \widetilde{\beta}) = \zeta \iota_f$ for some continuous $\zeta \colon W \rightarrow \UU(1)$; and we set $\beta = \zeta^{-1} \widetilde{\beta}$, where the inverse is of the value, not an inverse map.

Note: $\alpha$ itself gives rise via the canonical isomorphisms in the last paragraph to a section $\widetilde{\beta}$, which looking to the action on elements of isomorphism \ref{eq-u1pb-tau-star} as discussed in corollary \ref{co-bg-mult-cano-isom-elem} to come, in fact is equal to $\beta$; we have a sort of formula for the inverse of $\alpha$, namely a canonical isomorphism composed with $\alpha^{*}$.

The statement about $\beta$ and $\kappa$ follows from reasoning like that used for torsors, pulling back the bundle gerbe multiplication associativity diagram by $\Delta_2^{1212}$ for the multiplications one would write.

The statement about $\zeta \alpha$ follows from evaluation at $w \in W$.

For the inverse of a product, the result is equivalent to $m(m(\alpha_{12} \tensor \alpha_{23}) \tensor m(\alpha_{23}^{-1} \tensor \alpha_{12}^{-1})) = m(m(\alpha_{12} \tensor m(\alpha_{23} \tensor \alpha_{23}^{-1})) \tensor \alpha_{12}^{-1}) = \iota_{\pi_{12} \circ f}$.  For those with patience, the validity of the multiplication and associativity could be seen by pulling back an appropriately constructed diagram by something like $\Delta_3^{12321}$.
\end{proof}

The following results would naturally appear in lemma \ref{l-bg-mult-cano-isom}.  It is to avoid the appearance of circular reasoning that they are here, since their proof uses parts of lemma \ref{l-bg-mult-id-inv}, and part of that lemma in turn depends on part of lemma \ref{l-bg-mult-cano-isom}.
\begin{cor}\label{co-bg-mult-cano-isom-elem}
\index{bundle gerbe!multiplicative identity, inverse}
(Multiplication Related Isomorphisms on Elements)
Isomorphism \ref{eq-u1t-yy-triv} of lemma \ref{l-bg-mult-cano-isom}, $P_{y, y} \cong \UU(1)$, is given on elements by $\iota_{y,y} \mapsto 1$.  Isomorphism \ref{eq-u1t-21-12-star}, $P_{y_2, y_1} \cong P_{y_1, y_2}^{*}$, is given on elements as follows.  For $p_{21} \in P_{y_2, y_1}$ let $p_{12} = p_{21}^{-1} \in P_{y_1, y_2}$ denote its multiplicative inverse element.  Then $p_{21} \mapsto p_{12}^{*} = (p_{21}^{-1})^{*}$.  On fibers isomorphism \ref{eq-u1pb-yy-triv} acts like \ref{eq-u1t-yy-triv}, and \ref{eq-u1pb-tau-star} like \ref{eq-u1t-21-12-star}.
\end{cor}
\begin{proof}
Recalling the proof of isomorphism \ref{eq-u1t-yy-triv},
\begin{align}
P_{y, y} \tensor P_{y, y} &\xrightarrow{m} P_{y, y} \notag \\
(P_{y, y} \tensor P_{y, y}) \tensor P_{y, y}^{*} &\cong P_{y, y} \tensor P_{y, y}^{*} \notag \\
P_{y, y} \tensor \UU(1) &\cong \UU(1) \notag \\
P_{y, y} &\cong \UU(1), \text{ that is,} \notag \\
p_{y,y} &\mapsto p_{y,y} \tensor 1 \mapsto p_{y,y} \tensor (p_{y,y} \tensor p_{y,y}^{*}) \mapsto m(p_{y,y}, p_{y,y}) \tensor p_{y,y}^{*}, \notag
\end{align}
following the left side up to the second line, then across. Substituting $\iota_{y,y}$ for $p_{y,y}$, and going down on the right side, $\iota_{y,y} \mapsto m(\iota_{y,y}, \iota_{y,y}) \tensor \iota_{y,y}^{*} = \iota_{y,y} \tensor \iota_{y,y}^{*} \mapsto 1$, referring to maps on elements of lemma \ref{l-u1t-cano-isom}.

Recalling the proof of isomorphism \ref{eq-u1t-21-12-star},
\begin{align}
P_{y_1, y_2} \tensor P_{y_2, y_1} &\xrightarrow{m} P_{y_1, y_1} \cong \UU(1) \notag \\
P_{y_1, y_2}^{*} \tensor (P_{y_1, y_2} \tensor P_{y_2, y_1}) &\cong P_{y_1, y_2}^{*} \tensor \UU(1) \notag \\
P_{y_2, y_1} &\cong P_{y_1, y_2}^{*}, \notag
\end{align}
the statement about elements can be seen by following the left side up from bottom to middle line in two steps, then across and down the right side: $p_{21} \mapsto 1 \tensor p_{21} \mapsto p_{12}^{*} \tensor p_{12} \tensor p_{21} \mapsto p_{12}^{*} \tensor \iota_{11} \mapsto p_{12}^{*} \tensor 1 \mapsto p_{12}^{*}$, using the first part of the corollary statement.

The statements about bundles follow from the statements about torsors, since the proofs of the bundle isomorphisms in lemma \ref{l-bg-mult-cano-isom} were done by applying to bundles the reasoning used there for torsors.
\end{proof}

\begin{note}\label{n-bg-triv-fore}
\index{bundle gerbe!triviality}
(Motivation for the Definition of Triviality).
Next we would like to define triviality of a bundle gerbe.  We reason by analogy with triviality of a principal $\UU(1)$ bundle, but don't define a canonically trivial bundle gerbe as a product bundle gerbe.

A principal $\UU(1)$ bundle is isomorphic to the canonical trivial bundle, the product bundle, if and only if it has a global section, which gives rise by the local trivializations on open sets $U_i$ to functions $h_i \colon U_i \rightarrow \UU(1)$ that give a simple form for the transition functions $g_{j i} = h_j h_i^{-1}$ on nonempty intersections of two open sets.  These transition functions are a \v{C}ech $1$-cocycle, in fact for a good cover are the first Chern cohomology class; and the simple form says that this class is zero, i.e. that the cocycle is a coboundary.  Conversely such a set of $h_i$ gives rise to a global section.

A canonical trivial bundle gerbe is one in which the $P$ bundle is of the form $\delta(R)$, for $R$ a principal $\UU(1)$ bundle over $Y$, the coboundary map $\delta$ of definition \ref{d-bg-delt}, and with bundle gerbe multiplication defined in a natural way.  That a bundle gerbe is isomorphic to a canonical trivial bundle gerbe with the same base and $Y$ space, is equivalent to the existence of a principal $\UU(1)$ bundle isomorphism between $P$ and $\delta(R)$ that respects bundle gerbe multiplication.  Bundle gerbe multiplication and local sections of $Y$ give rise to a \v{C}ech $2$-cocycle, defined on nonempty intersections of three open sets, which is used to define the bundle gerbe's Dixmier-Douady cohomology class, the vanishing of which is equivalent to triviality.  We will be more precise; but first, to prepare for using \v{C}ech cohomology, an interlude of sheaf theory.
\end{note}

\section{Direct Limits, Category and Sheaf Theory}\label{s-dire-limi-cate-shea-theo}

The reader is not expected to know much about categories apart from what categories, functors, natural transformations, initial and terminal objects are, nor anything about sheaves.  The definitions and results we use are collected here.  However, these are just the bare bones needed for the small part of \v{C}ech cohomology needed in the thesis.  Read in the references for better understanding.  On the other hand, this section can be skipped and referred to later if desired.

References differ on the definition of the direct limit used in the definition of \v{C}ech cohomology, though in cases examined, the results appear equivalent.  A set theoretic problem arises from using indexed covers of a space, the collection of which, without some sort of limitation, is a proper class.  References deal with this in various ways, including ignoring it.  Some references require antisymmetry in the order relation on the index sets, and others don't.  There is the question of whether empty intersections are allowed in the definition of the cochains; we need the possibility because of using the inverse image, or pullback, of a cover.  The cochains may be alternating functions of index tuples, or not.  To suit our needs, we combine aspects of more than one reference's definition.  We start here by defining the kind of order relation we require, direct systems, their morphisms, and direct limits.  Then after material on category theory, presheaves and sheaves, the next section will define \v{C}ech cohomology, with coefficients in presheaves and hence in sheaves.

\begin{defn}\label{d-preo}
\index{pre-order}
(Pre-Orders).
\citep[page~3]{Tenn75} A pre-order $\le$ on a set $\Lambda$ is a relation satisfying:
\begin{enumerate}
  \item for every $\alpha \in \Lambda$, $\alpha \le \alpha$ (reflexivity)
  \item for $\alpha, \beta, \gamma \in \Lambda$, if $\alpha \le \beta$ and $\beta \le \gamma$, then $\alpha \le \gamma$ (transitivity).
\end{enumerate}
We will also view a set $\Lambda$ with a pre-order as an abstract category with one object per element of $\Lambda$, and for $\alpha, \beta \in \Lambda$, exactly one morphism $\alpha \rightarrow \beta$ when $\alpha \le \beta$.

\citep[page~212]{ES52} $\Lambda' \subset \Lambda$ is a cofinal subset when given any $\alpha \in \Lambda$, there is a $\beta \in \Lambda'$ such that $\alpha \le \beta$.
\end{defn}
We do not require antisymmetry.

\begin{defn}\label{d-dire-set}
\index{directed set}
(Directed Sets).
\citep[page~3]{Tenn75} A directed set $\Lambda$ is a set with a pre-order for which given any $\alpha, \beta \in \Lambda$, there is a $\gamma \in \Lambda$ such that $\alpha \le \gamma$ and $\beta \le \gamma$.

Viewing $\Lambda$ as a category as in definition \ref{d-preo}, the additional requirement for a directed set is that for any objects $\alpha, \beta$, there is an object $\gamma$ such that there are morphisms $\alpha \rightarrow \gamma$, $\beta \rightarrow \gamma$.

$\Lambda' \subset \Lambda$ is a cofinal subset when it is a cofinal subset as a set with a pre-order.
\end{defn}
This implies that $\Lambda'$ also, is a directed set.

\begin{defn}\label{d-dire-syst}
\index{direct system}
(Direct Systems of Abelian Groups).
% removed ~ from citation to cure overfull box
\citep[pages 3, 7]{Tenn75} A direct system $(A, \rho)$ of abelian groups over, or indexed by a directed set $\Lambda$ is a family $\{ A_\alpha \st \alpha \in \Lambda \}$ of abelian groups together with, for each $\alpha \le \beta$ in $\Lambda$, a homomorphism $\rho_{\alpha, \beta} \colon A_{\alpha} \rightarrow A_{\beta}$, satisfying
\begin{enumerate}
  \item for every $\alpha \in \Lambda$, $\rho_{\alpha, \alpha} = \ident$
  \item for $\alpha, \beta, \gamma \in \Lambda$, if $\alpha \le \beta$ and $\beta \le \gamma$, then $\rho_{\alpha, \gamma} = \rho_{\beta, \gamma} \circ \rho_{\alpha, \beta}$.
\end{enumerate}
We may indicate the directed set just by $A$, or also by $(\Lambda, A, \rho)$.

\citep[page~237]{Rotm09} Viewing $\Lambda$ as a category, a direct system of abelian groups indexed by a directed set $\Lambda$ is a covariant functor $A \colon \Lambda \rightarrow Ab$, the category of abelian groups; $A (\alpha) = A_{\alpha}$ and $A (\alpha \rightarrow \beta) = \rho_{\alpha, \beta}$.

\citep[page~214]{ES52} \citep[page~255]{Rotm09} If $\Lambda' \subset \Lambda$ is a cofinal set, the abelian groups and and homomorphisms corresponding to elements and relations in $\Lambda'$ form a direct system $(\Lambda', A', \rho')$, called a subsystem of $(\Lambda, A, \rho)$.  If $\Lambda'$ is cofinal in $\Lambda$, the subsystem is called cofinal.
\end{defn}

\begin{defn}\label{d-dire-syst-morp}
\index{direct system!morphism}
\index{morphism!direct system}
(Morphisms of Direct Systems of Abelian Groups).
\citep[pages~212--214]{ES52} A morphism $(\Psi, \psi)$ of direct systems $(\Lambda, A, \rho)$, $(\Upsilon, B, \sigma)$ of abelian groups consists of
\begin{enumerate}
  \item a function $\Psi \colon \Lambda \rightarrow \Upsilon$ such that for $\alpha, \beta \in \Lambda$, $\alpha \le \beta \Rightarrow \Psi(\alpha) \le \Psi(\beta)$
  \item a collection of homomorphisms $\{ \psi_{\alpha} \colon A_{\alpha} \rightarrow B_{\Psi (\alpha)} \st \alpha \in \Lambda \}$ such that for $\alpha, \beta \in \Lambda$ and $\alpha \le \beta$, $\sigma_{\Psi (\alpha), \Psi (\beta)} \circ \psi_{\alpha} = \psi_{\beta} \circ \rho_{\alpha, \beta}$.
\end{enumerate}

\citep[pages~246--247]{Rotm09} From the categorical viewpoint, a morphism from the direct system of abelian groups $(\Lambda, A, \rho)$, denoted here by $F \colon \Lambda \rightarrow Ab$, to the direct system of abelian groups $(\Upsilon, B, \sigma)$ or $G \colon \Upsilon \rightarrow Ab$, consists of a functor $\Psi \colon \Lambda \rightarrow \Upsilon$ and a natural transformation $\psi$ from $F$ to $G \circ \Psi$.

The composition of morphisms of direct systems of abelian groups is defined by the compositions of the functors and the natural transformations.  More precisely, given direct systems of abelian groups $F_1 \colon \Lambda_1 \rightarrow Ab$, $F_2 \colon \Lambda_2 \rightarrow Ab$, $F_3 \colon \Lambda_3 \rightarrow Ab$, and morphisms $(\Psi_{12}, \psi_{12}) \colon F_1 \rightarrow F_2$ and $(\Psi_{23}, \psi_{23}) \colon F_2 \rightarrow F_3$, define
\[
(\Psi_{23}, \psi_{23}) \circ (\Psi_{12}, \psi_{12}) = (\Psi_{23} \circ \Psi_{12}, (\psi_{23})_{\Psi_{12} \circ} \circ \psi_{12}), \notag
\]
where the notation $(\psi_{23})_{\Psi_{12} \circ}$ means that for an object $\alpha \in \Lambda_1$, we use the homomorphism $F_2 (\Psi_{12} (\alpha)) \rightarrow F_2 \circ \Psi_{23} (\Psi_{12} (\alpha))$ of the natural transformation $\psi_{23}$.

\citep[page~214]{ES52} Given a subsystem of a direct system of abelian groups, the inclusion $\Lambda' \rightarrow \Lambda$ and identity maps on the abelian groups, constitute a morphism of direct systems called the injection.
\end{defn}
\begin{note}\label{n-dire-syst-morp}
\index{direct system!abelian group}
(Abelian Groups can be Made Into Direct Systems).
Any abelian group $B$ can be made into a direct system indexed by a set with one element.  Then a morphism from a direct system $(\Lambda, A, \rho)$ into the abelian group amounts to a collection of homomorphisms $\psi_{\alpha} \colon A_{\alpha} \rightarrow B$ for $\alpha \in \Lambda$, such that $\psi_{\alpha} = \psi_{\beta} \circ \rho_{\alpha, \beta}$ for $\alpha \le \beta$ in $\Lambda$.
\end{note}

\begin{defn}\label{d-dire-syst-targ}
\index{direct system!target}
\index{target!direct system}
(Targets of a Direct System of Abelian Groups).
\citep[pages~4,~7]{Tenn75} Given a direct system $(A, \rho)$ of abelian groups, a target $(B, \sigma)$ for the system is an abelian group $B$ and a collection of homomorphisms $\{ \sigma_{\alpha} \colon A_{\alpha} \rightarrow B \st \alpha \in \Lambda \}$ such that for $\alpha, \beta \in \Lambda$ with $\alpha \le \beta$, $\sigma_{\beta} = \sigma_{\alpha} \circ \rho_{\alpha, \beta}$.

\citep[page~238]{Rotm09} Viewing $\Lambda$ as a category and the direct system as a functor $A$, a target is a constant functor $C_B \colon \Lambda \rightarrow Ab$, value on objects $B$ and on morphisms $\ident$, together with a natural transformation $\sigma$ from $A$ to $C_B$.
\end{defn}

For concreteness, rather than using its universal property to define the direct limit of a direct system of abelian groups only up to unique isomorphism, we define it by its construction and give the universal property as a consequence.
\begin{defn}\label{d-dire-limi}
\index{direct limit}
\index{limit!direct}
(The Direct Limit of a Direct System of Abelian Groups). \linebreak
\citep[pages~6,~8]{Tenn75} Suppose given a direct system $(A, \rho)$ of abelian groups.  Define the direct limit of the system, $(\varinjlim_{\alpha \in \Lambda} A_{\alpha}, \tau)$, usually referred to as just $\varinjlim_{\alpha \in \Lambda} A_{\alpha}$, as the target for the direct system, given by
\begin{enumerate}
   \item for the abelian group, the quotient of $\dirsum_{\alpha \in \Lambda} A_{\alpha}$ (with natural injections for each $\beta \in \Lambda$, $\incl_{\beta} \colon A_{\beta} \rightarrow \dirsum_{\alpha \in \Lambda} A_{\alpha}$) by the subgroup of the direct sum generated by $\incl_{\alpha} (g) - \incl_{\beta} \circ \rho_{\alpha, \beta} (g)$ for all $\alpha \le \beta$ and all $g \in A_{\alpha}$
   \item the collection of insertion homomorphisms, $\tau_{\beta} = \pi \circ \incl_{\beta}
       \colon A_{\beta} \rightarrow \varinjlim_{\alpha \in \Lambda} A_{\alpha}$, for $\beta \in \Lambda$, where $\pi$ is the quotient homomorphism.
\end{enumerate}
\end{defn}
\begin{note}\label{n-dire-limi}
\index{direct limit}
\index{limit!direct}
(Direct Limit Properties).
\citep[pages~4,~7]{Tenn75} The direct limit satisfies a universal property that could be used to define it up to unique isomorphism: given a target $(B, \sigma)$ for the direct system, there is a unique homomorphism $\mu \colon \varinjlim_{\alpha \in \Lambda} A_{\alpha} \rightarrow B$ such that for every $\alpha \in \Lambda$, $\sigma_{\alpha} = \mu \circ \tau_{\alpha}$.

\citep[pages~5--8]{Tenn75} Each element of the direct limit has a representative in a single $A_\alpha$.  The image in the direct limit of an element of $A_\alpha$ is zero exactly when that element is mapped to zero by some $\rho_{\alpha, \beta}$ for $\beta \ge \alpha$.

\citep[pages~222--223]{ES52} A morphism of direct systems induces a homomorphism between the resulting direct limits, compatible with the homomorphisms of the morphism and the insertion homomorphisms of the direct limits.  That is, a morphism $(\Psi, \psi)$ from $(\Lambda, A, \rho)$ to $(\Upsilon, B, \sigma)$ induces a homomorphism $\mu \colon (\varinjlim_{\alpha \in \Lambda} A_{\alpha}, \sigma) \rightarrow (\varinjlim_{\beta \in \Upsilon} B_{\beta}, \tau)$ such that $\mu \circ \sigma_{\alpha} = \tau_{\Psi(\alpha)} \circ \psi_{\alpha}$.  This homomorphism is called the direct limit of the morphism.

The operation of assigning the direct limits to each direct system and to each morphism of direct systems is a covariant functor from the category of direct systems and morphisms, to that of abelian groups.

If $\Lambda'$ is cofinal in $\Lambda$, the injection morphism induces an isomorphism between the direct limits.  In the case of note \ref{n-dire-syst-morp}, this recovers the universal property of direct limits.

Suppose given a morphism of direct systems $(\Psi, \psi)$ from $(\Lambda, A, \rho)$ to $(\Upsilon, B, \sigma)$.  If there exists a directed set $\Delta \subset \Lambda$ that is cofinal in $\Lambda$, such that $\Psi(\Delta)$ is cofinal in $\Upsilon$, and $\psi_{\alpha}$ is an isomorphism for each $\alpha \in \Lambda$, then the direct limit of the morphism is an isomorphism.

\citep[pages~246--247]{Rotm09} \citep[pages~224--225]{ES52} Defining a direct system of sequences of homomorphisms of abelian groups as a collection of sequences, one sequence for each element of a directed set, each position of each sequence belonging to a direct system of abelian groups over the directed set, with the sequence and direct system homomorphisms commuting, then if each sequence is a (co)chain complex, or exact, so is the direct limit sequence.
\end{note}

\begin{note}\label{n-dire-clas}
\index{directed class}
(Directed Classes, Cofinality, and Set Theoretical Issues).
It's possible to define a directed class, a cofinal directed class, and a direct system over a directed class, analogously to the cases with directed sets, as in \citet[pages~390--391]{Rotm09}.  Without defining the direct limit of a direct system over a directed class, one still has the following result.

Given a directed class $\Lambda$ and a direct system $(\Lambda, A, \rho)$, if $\Delta_1$ and $\Delta_2$ are cofinal in $\Lambda$ and both $\Delta_1$ and $\Delta_2$ are sets, then $\varinjlim_{\delta \in \Delta_1} A_{\delta} \cong \varinjlim_{\delta \in \Delta_2} A_{\delta}$.

This idea allows the comparison of various schemes, each of which makes \v{C}ech cohomology well defined by reducing what one might at first come up with as a directed class, to a directed set.  Thus we can reconcile various definitions of \v{C}ech cohomology that address a set theoretic problem.
\end{note}

\begin{defn}\label{d-mono-epim}
\index{monomorphism}
\index{epimorphism}
(Monomorphisms and Epimorphisms).
% removed ~ to cure overfull box
\citep[pages 304--305]{Rotm09} Given objects $E, F$ of a category $C$, a morphism $u \colon E \rightarrow F$ is a monomorphism when it can be canceled from the left; that is, when for all objects $D$ and morphisms $d, e \colon D \rightarrow E$, $u d = u e \Rightarrow d = e$.

Given objects $E, F$ of a category $C$, a morphism $v \colon E \rightarrow F$ is an epimorphism when it can be canceled from the right; that is, when for all objects $G$ and morphisms $f, g \colon F \rightarrow G$, $f v = g v \Rightarrow f = g$.
\end{defn}

\begin{defn}\label{d-subo-quot}
\index{subobject}
\index{quotient object}
(Subobjects and Quotient Objects).
\citep[pages 306--307]{Rotm09} Given an object $E$ of a category $C$, a subobject of $E$ is an equivalence class of pairs $(D, d)$, $D$ an object of $C$ and $d \colon D \rightarrow E$ a monomorphism, where $(D, d) \sim (D', d')$ when there is an isomorphism $\phi \colon D \rightarrow D'$ with $d = d' \phi$.

A quotient object of $E$ is an equivalence class of pairs $(G, f)$, $G$ an object of $C$ and $f \colon F \rightarrow G$ an epimorphism, where $(G, f) \sim (G', f')$ when there is an isomorphism $\psi \colon G' \rightarrow G$ with $f = \psi f'$.
\end{defn}

\begin{defn}\label{d-add-cat}
\index{category!additive}
\index{additive category}
(Additive Categories).
\citep[page~303]{Rotm09} A category $C$ is additive if
\begin{enumerate}
 \item for objects $A, B$ of $C$, $\Hom(A, B)$ is an additive abelian group
 \item for objects $A, B, X, Y$ of $C$, given morphisms
    \[
    \begindc{\commdiag}[5]
    \obj(10,10)[objX]{$X$}
    \obj(20,10)[objA]{$A$}
    \obj(30,10)[objB]{$B$}
    \obj(40,10)[objY]{$Y$}
    \mor{objX}{objA}{$a$}
    \mor(20,11)(30,11){$f$}[\atleft,\solidarrow]
    \mor(20,9)(30,9){$g$}[\atright,\solidarrow]
    \mor{objB}{objY}{$b$}
    \enddc
    \]
    \[
    b (f + g) = b f + b g \text{ and } (f + g) a = f a  + g a. \notag
    \]
 \item $C$ has a zero object $0$ (an object that is both initial and terminal)
 \item $C$ has finite products and finite coproducts \citep[pages~220--221]{Rotm09}
\end{enumerate}
\end{defn}

\begin{defn}\label{d-ker-coke}
\index{kernel}
\index{cokernel}
(Kernels and Cokernels).
\citep[page~305]{Rotm09} Given objects $E, F$ of an additive category $C$ and a morphism $u \colon E \rightarrow F$, its kernel $\ker(u)$ is (an object $D$ and) a morphism $i \colon D \rightarrow E$ that satisfies the following universal mapping property: $u i = 0$ and, for every object $X$ and morphism $e \colon X \rightarrow E$ with $u e = 0$, there is a unique morphism $\theta \colon X \rightarrow D$ with $i \theta = e$.
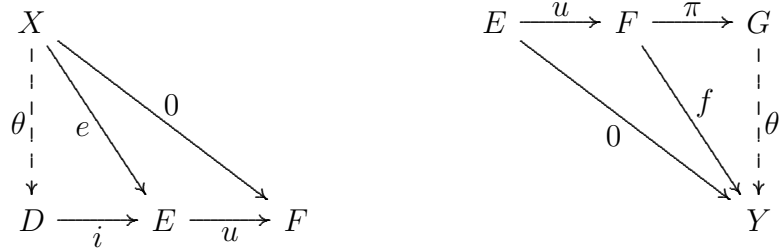
\[
\begindc{\commdiag}[5]
\obj(10,30)[objX]{$X$}
\obj(10,15)[objD]{$D$}
\obj(20,15)[objE]{$E$}
\obj(30,15)[objF]{$F$}
\mor{objD}{objE}{$i$}[\atright,\solidarrow]
\mor{objE}{objF}{$u$}[\atright,\solidarrow]
\mor{objX}{objD}{$\theta$}[\atright,\dashArrow]
\mor{objX}{objE}{$e$}[\atright,\solidarrow]
\mor{objX}{objF}{$0$}[\atleft,\solidarrow]

\obj(45,30)[objE]{$E$}
\obj(55,30)[objF]{$F$}
\obj(65,30)[objG]{$G$}
\obj(65,15)[objY]{$Y$}
\mor{objE}{objF}{$u$}[\atleft,\solidarrow]
\mor{objF}{objG}{$\pi$}[\atleft,\solidarrow]
\mor{objG}{objY}{$\theta$}[\atleft,\dashArrow]
\mor{objF}{objY}{$f$}[\atleft,\solidarrow]
\mor{objE}{objY}{$0$}[\atright,\solidarrow]
\enddc
\]
Given objects $E, F$ of an additive category $C$ and a morphism $u \colon E \rightarrow F$, its cokernel $\coker(u)$ is (an object $G$ and) a morphism $\pi \colon F \rightarrow G$ that satisfies the following universal mapping property: $\pi u = 0$ and, for every object $Y$ and morphism $f \colon F \rightarrow Y$ with $f u = 0$, there is a unique morphism $\theta \colon G \rightarrow Y$ with $\theta \pi = f$.
\end{defn}

\begin{defn}\label{d-abel-cat}
\index{category!abelian}
\index{abelian category}
(Abelian Categories).
\citep[pages~307]{Rotm09} An additive category $C$ is an abelian category if every morphism has a kernel and a cokernel, every monomorphism is a kernel, and every epimorphism is a cokernel.

Given objects $D, E$ of $C$ and a morphism $D \xrightarrow{d} E$, define the image of $d$ as the subobject of $E$ defined by $\im (d) = \ker (\coker (d))$.  A sequence $D \xrightarrow{d} E \xrightarrow{e} F$ in $C$ is exact when
\[
\im (d) = \ker (\coker (d)) = \ker (e), \notag
\]
where equality is of subobjects, and the quotient object $\coker(d)$ also is an equivalence class.

\citep[page~51]{Tenn75} A functor $G$ between abelian categories $C, C'$, is called exact when it maps exact sequences to exact sequences, or equivalently, for $D, E, F \in C$,
\[
0 \rightarrow D \rightarrow E \rightarrow F \rightarrow 0 \text{ exact } \Rightarrow \text{ } 0 \rightarrow GD \rightarrow GE \rightarrow GF \rightarrow 0 \text{ exact.} \notag
\]
It is called left exact when the implication is only
\[
0 \rightarrow D \rightarrow E \rightarrow F \rightarrow 0 \text{ exact } \Rightarrow \text{ } 0 \rightarrow GD \rightarrow GE \rightarrow GF \text{ exact.} \notag
\]
\end{defn}
The category of abelian groups is an abelian category.  If $J$ is a category and $A$ is an abelian category, then the category of functors from $J$ to $A$ is an abelian category.  In an abelian category, a morphism is an isomorphism when it is both a monomorphism and an epimorphism. \citep[page~199]{MacL00}

\begin{defn}\label{d-exac-delt-func}
\index{delta functor@$\partial$-functor}
\index{exact delta functor@exact $\partial$-functor}
(Exact $\partial$-Functors).
\citep[page~124]{Tenn75} Given abelian categories $K, K'$, a $\partial$-functor $T \colon K \rightarrow K'$ is a sequence of functors $T^p \colon K \rightarrow K'$ for integers $p \ge 0$ together with an assignment to each short exact sequence
\[
0 \rightarrow A \rightarrow B \rightarrow C \rightarrow 0 \notag
\]
in $K$ of a collection of morphisms $\partial = \partial_T^p \colon T^p C \rightarrow T^{p+1} A$, such that the associated long sequence
\[
0 \rightarrow T^0 A \rightarrow T^0 B \rightarrow T^0 C \xrightarrow{\partial} T^1 A \rightarrow \dots \rightarrow T^p C \xrightarrow{\partial} T^{p+1} A \rightarrow \dots \notag
\]
is exact, and whenever
\[
\begindc{\commdiag}[5]
\obj(10,30)[obj0UL]{$0$}
\obj(20,30)[objA]{$A$}
\obj(30,30)[objB]{$B$}
\obj(40,30)[objC]{$C$}
\obj(50,30)[obj0UR]{$0$}
\obj(10,20)[obj0LL]{$0$}
\obj(20,20)[objAP]{$A'$}
\obj(30,20)[objBP]{$B'$}
\obj(40,20)[objCP]{$C'$}
\obj(50,20)[obj0LR]{$0$}
\mor{obj0UL}{objA}{}
\mor{objA}{objB}{}
\mor{objB}{objC}{}
\mor{objC}{obj0UR}{}
\mor{obj0LL}{objAP}{}
\mor{objAP}{objBP}{}
\mor{objBP}{objCP}{}
\mor{objCP}{obj0LR}{}
\mor{objA}{objAP}{$f$}
\mor{objB}{objBP}{$g$}
\mor{objC}{objCP}{$h$}
\enddc\]
commutes in $K$ and has exact rows, then the corresponding diagrams
\[
\begindc{\commdiag}[5]
\obj(0,30)[objTBL]{$T^p B$}
\obj(0,15)[objTBPL]{$T^p B'$}
\obj(15,30)[objTC]{$T^p C$}
\obj(15,15)[objTCP]{$T^p C'$}
\obj(30,30)[objTA]{$T^{p+1} A$}
\obj(30,15)[objTAP]{$T^{p+1} A'$}
\obj(45,30)[objTBR]{$T^{p+1} B$}
\obj(45,15)[objTBPR]{$T^{p+1} B'$}
\mor{objTBL}{objTBPL}{$T^p g$}
\mor{objTC}{objTCP}{$T^p h$}
\mor{objTA}{objTAP}{$T^{p+1} f$}
\mor{objTBR}{objTBPR}{$T^{p+1} g$}
\mor{objTBL}{objTC}{}
\mor{objTBPL}{objTCP}{}
\mor{objTC}{objTA}{$\partial$}
\mor{objTCP}{objTAP}{$\partial$}
\mor{objTA}{objTBR}{}
\mor{objTAP}{objTBPR}{}
\enddc
\]
commute; i.e. $\partial$ is natural.
\end{defn}

\begin{defn}\label{d-prsh}
\index{presheaf}
\index{morphism!presheaves}
\index{morphism!Presh/X@$Presh/X$}
(Presheaves of Abelian Groups).
\citep[page~1]{Tenn75} A presheaf $A$ of abelian groups over a topological space $X$ is a contravariant functor from the category of open subsets of $X$ and inclusions to the category of abelian groups.  The value of $A$ on an open set $U \subset X$, $A(U)$, is an abelian group whose elements $s \in A(U)$ are called sections of $A$ over $U$.  The value of $A$ on an inclusion $V \subset U$ of open sets in $X$, $A(V \subset U) = r_{V,U} \colon A(U) \rightarrow A(V)$, is a group homomorphism called a restriction map, and for $s \in A(U)$, $r_{V,U} (s)$ is called the restriction of the section $s$ to $V$.  The elements of $A(X)$ are called global sections.

\citep[pages~9--10]{Tenn75} Given presheaves $A, B$ of abelian groups over $X$, denoting the corresponding restriction maps by $r, s$, a morphism of presheaves $\phi \colon A \rightarrow B$ over $X$ is a natural transformation of the presheaves (functors).  The definition of natural transformation here is that to each open $U \subset X$ corresponds a group homomorphism $\phi_U \colon A(U) \rightarrow B(U)$ such that whenever $V \subset U$, the following diagram commutes:
\[
\begindc{\commdiag}[5]
\obj(10,25)[objAU]{$A(U)$}
\obj(25,25)[objBU]{$B(U)$}
\obj(10,10)[objAV]{$A(V)$}
\obj(25,10)[objBV]{$B(V)$}
\mor{objAU}{objBU}{$\phi_U$}
\mor{objAV}{objBV}{$\phi_V$}
\mor{objAU}{objAV}{$r_{V, U}$}
\mor{objBU}{objBV}{$s_{V, U}$}
\enddc
\]
\citep[pages~8--9]{Tenn75} Given a presheaf $A$ of abelian groups over $X$, and a point $x \in X$, define a direct system ordering open sets of $X$ in the opposite order of inclusion (i.e. $V \subset U \Leftrightarrow U \le V$), with abelian groups $A(U)$ and homomorphisms $\rho_{U,V} = r_{V,U}  \colon A(U) \rightarrow A(V)$.  Define the stalk of $A$ at $x$, $A_x = \varinjlim_{U \st x \in U} A(U)$, and denote the insertion homomorphisms of the direct limit for the stalk at $x$, $A(U) \rightarrow A_x$, on elements by $s \mapsto s_x$.
\end{defn}
\citep[page~2]{Tenn75} We will use presheaves of continuous functions (into $\ZZ$, $\RR$, and $\UU(1)$) defined on open sets of a topological space given by the context; then restriction has the usual meaning, and is given by composition on the right with the inclusion mapping, $f \mapsto f \circ \incl$.

\begin{defn}\label{d-shea}
\index{sheaf}
(Sheaves of Abelian Groups).
\citep[pages~14--15]{Tenn75} A sheaf $A$ of abelian groups over a topological space $X$ is a presheaf of abelian groups over $X$ that satisfies two additional conditions, given an open $U \subset X$, and an open cover $\{ U_{\lambda} \}_{\lambda \in \Lambda}$ of $U$:
\begin{enumerate}
 \item given sections $s_1, s_2 \in A(U)$ such that for every $\lambda \in \Lambda$, $r_{U_{\lambda}, U} (s_1) = r_{U_{\lambda}, U} (s_2)$, then $s_1 = s_2$ (with this condition alone $A$ is called a monopresheaf)
 \item given a family $\{ s_{\lambda} \}_{\lambda \in \Lambda}$ of sections with $s_{\lambda} \in A(U_{\lambda})$ for every $\lambda \in \Lambda$, such that for every $\lambda, \mu \in \Lambda$, $r_{U_{\lambda} \cap U_{\mu}, U_{\lambda}} (s_{\lambda}) = r_{U_{\lambda} \cap U_{\mu}, U_{\mu}} (s_{\mu})$, then there is an $s \in A(U)$ such that for every $\lambda \in \Lambda$, $r_{U_{\lambda}, U} (s) = s_{\lambda}$.
\end{enumerate}
In other words, with both conditions, if sections are given on the open sets of a covering of $U$ that are consistent on all the overlaps, then there is exactly one section on all of $U$ from which they arise by restriction.

Since a sheaf over $X$ is a presheaf over $X$, it inherits the definitions of morphism of presheaves over $X$, and stalk, from definition \ref{d-prsh}.
\end{defn}

\begin{eg}\label{e-sheaves}
(Examples of Sheaves).
\citep[pages~2,~17]{Tenn75} Examples of sheaves of abelian groups over a topological space $X$ are presheaves of continuous functions. We denote by $\underline{\UU(1)}$, $\underline{\ZZ}$, and $\underline{\RR}$, the continuous $\UU(1)$, $\ZZ$, and $\RR$ valued functions defined on open subsets of $X$, using the discrete topology for $\ZZ$ and the usual topologies for $\UU(1)$ and $\RR$.  The sheaf $\underline{\UU(1)}$ will be written multiplicatively and its identity section can be written $\underline{1}$.
\end{eg}

\begin{lem}\label{l-pres-shea}
\index{sheaf}
\index{presheaf}
(Properties of Presheaves and Sheaves).
\citep[page~49]{Tenn75} Presheaves of abelian groups over a given topological space $X$ form an abelian category $Presh/X$.  Likewise sheaves of abelian groups over $X$ form an abelian category $Shv/X$, which is a full subcategory of $Presh/X$ \citep[pages~309]{Rotm09}.  The inclusion functor $Shv/X \xrightarrow{\incl} Presh/X$ preserves kernels; is left exact.

\citep[pages~22--23,~34--35,~52]{Tenn75} There is a ``sheafification'' functor $\Gamma L \colon Presh/X \rightarrow Shv/X$, with $(\Gamma L)_{|Shv/X}$ naturally equivalent to $\ident_{Shv/X}$.  There is a natural transformation $n$ from $\ident_{Presh/X}$ to $\Gamma L$, and each morphism of presheaves over $X$ from a presheaf $F$ to a sheaf $G$, $f \colon F \rightarrow G$, factors uniquely through $n_F \colon F \rightarrow \Gamma L F$.  The ``sheafify'' functor $\Gamma L$ is exact.

\citep[pages~37--38,~41--43,~49--51]{Tenn75} Given a morphism of presheaves over $X$, $f \colon F \rightarrow G$, its kernel $\ker(f)$, cokernel $\pcoker(f)$, and image $\pim(f)$, the P's to distinguish between the $Presh/X$ and $Shv/X$ categories, are presheaves over $X$ constructed by applying the usual constructions for the abelian groups the presheaves produce for every open set in of $X$.  Likewise, injectivity and surjectivity, equivalently the properties of being a monomorphism, respectively epimorphism, are equivalent to the same properties for the abelian groups for every open set.  A sequence of morphisms of presheaves over $X$ is exact when the corresponding sequence of morphisms of abelian groups is exact, for every open set of $X$.

Given a morphism of sheaves over$X$, $f \colon F \rightarrow G$, the sheaf cokernel of $f$, $\scoker(f)$, or more precisely the object that is the domain of the morphism for the cokernel, is the sheafification $\Gamma L \pcoker(f)$, and there is a natural morphism of presheaves (and hence sheaves) over $X$, $G \rightarrow \scoker(f)$, given by $G \rightarrow \pcoker(f) \xrightarrow{n_{\pcoker(f)}} \scoker(f)$, such that the composite $F \rightarrow G \rightarrow \scoker(f)$ is zero.
\end{lem}

\begin{defn}\label{d-morp-over}
\index{morphism!presheaves over a map}
\index{morphism!over}
\index{morphism!Presh@$Presh$}
(Morphisms of Presheaves Over a Map of Spaces).
\citep[pages~55--56]{Tenn75} Given a continuous map of topological spaces $f \colon X \rightarrow Y$ and presheaves of abelian groups $A$ over $X$ and $B$ over $Y$, a morphism over $f$, or $f$-morphism $\overline{f} \colon B \rightarrow A$ is a collection of abelian group homomorphisms $\overline{f} (U, V)$, one for each open $V \subset Y$ and each open $U \subset f^{-1} (V)$, such that if $V' \subset V$, $U' \subset U$, and $U' \subset f^{-1} (V')$, then the following diagram commutes:
\[
\begindc{\commdiag}[5]
\obj(10,25)[objAU]{$A(U)$}
\obj(30,25)[objBV]{$B(V)$}
\obj(10,10)[objAUP]{$A(U')$}
\obj(30,10)[objBVP]{$B(V')$}
\mor{objBV}{objAU}{$\overline{f} (U, V)$}
\mor{objBVP}{objAUP}{$\overline{f} (U', V')$}
\mor{objBV}{objBVP}{$r_{V', V}$}
\mor{objAU}{objAUP}{$r_{U', U}$}
\enddc
\]

For another viewpoint, we define the category of presheaves $Presh$ (not over a particular topological space $X$) with objects the objects of $Presh/X$ for all topological spaces $X$, and for $A \in Presh/X$ and $B \in Presh/Y$, morphisms $B \rightarrow A$ in $Presh$ consisting of the morphisms over continuous maps $X \rightarrow Y$.
\end{defn}
\begin{note}\label{n-morp-over}
\index{morphism!over}
(Notes on Morphisms Over).
An $f$-morphism is determined by the homomorphisms $\overline{f} (f^{-1} (V), V)$, using the diagram with $U = f^{-1} (V)$ and $V' = V$.  Given a sequence of two maps of topological spaces, $X \xrightarrow{f} Y \xrightarrow{g} Z$, morphisms $\overline{f}$ over $f$ and $\overline{g}$ over $g$, define $(\overline{f} \circ \overline{g}) ((g \circ f)^{-1}W, W) = \overline{f} (f^{-1} g^{-1} W, g^{-1} W) \circ \overline{g} (g^{-1} W, W)$ for all open $W \subset Z$, requiring compatibility with restrictions, to get a morphism over $g \circ f$.

If $Y = X$ and $f = \ident_X$, then a morphism over $f$ is just a morphism of $Presh/X$.
\end{note}

\section{\v{C}ech Cohomology}\label{s-cech-coho}

\begin{defn}\label{d-inde-cove}
\index{cover!indexed}
\index{indexed cover}
(Indexed Covers).
An indexed cover $\{ U_i \st i \in I \}$ for a topological space $X$ with topology $T$, is a function $I \rightarrow T$, with value on $i \in I$ written $U_i$, such that $X \subset \bigcup_{i \in I} U_i$.
\end{defn}

\begin{defn}\label{d-inde-cove-refi}
\index{cover!indexed!refinement}
\index{index cover refinement}
\index{refinement!indexed cover}
(Refinements of Indexed Covers).
\citep[page~141]{Tenn75} An indexed cover $\{ V_j \st j \in J \}$ for a topological space $X$ with topology $T$ refines another indexed cover $\{ U_i \st i \in I \}$ for $X$ with topology $T$ when there is a function $\phi \colon J \rightarrow I$ such that for every $j \in J$, $V_j \subset U_{\phi(j)}$.
\end{defn}

\begin{defn}\label{d-orde-tupl}
\index{tuple!ordered}
\index{ordered tuple}
(Ordered Tuples).
Given a set $I$ and an integer $p \ge 0$, an ordered $(p + 1)$-tuple $\sigma$ of elements of $I$ is a function $\sigma \colon [0, p] = \{ q \in \ZZ \st 0 \le q \le p \} \rightarrow I$.  It may also be written by listing, in order, $i_0 = \sigma(0), \dots. i_p = \sigma(p)$, so that we may write $\sigma = (i_0, \dots, i_p)$.  Denote the set of $(p+1)$-tuples of elements of $I$ by $I_p$.
\end{defn}
There is no order relation implied on $I$.

\begin{defn}\label{d-orde-cech-coch}
\index{Cech cochain@\v{C}ech cochain}
\index{cochain!Cech@\v{C}ech}
\index{alternating!cochains}
(Alternating \v{C}ech Cochains).
\citep[pages~140,~153]{Tenn75} Suppose given an indexed open cover $\{ U_i \st i \in I \}$ for a topological space $X$, and a presheaf $A$ of abelian groups on $X$.  For each $(p+1)$-tuple $\sigma \in I_p$, let $U_{\sigma} = \bigcap_{q \in [0, p]} U_{\sigma(q)}.$ Using the following notation, define the abelian group $\check{C}^p (U, A)$ of alternating $p$-cochains for $U$ with coefficients in $A$, letting $S_{p+1}$ denote the group of permutations of $[0,p]$, and for $\tau \in S_{p+1}$, letting $\epsilon(\tau)$ denote the sign of the permutation:
\begin{align}
s &= (s (\sigma)) \in \prod_{\sigma \in I_p} A(U_{\sigma}) \notag \\
s (\sigma) &\in A(U_{\sigma}) \notag \\
\check{C}^p (U, A) &= \{ s \text{ satisfying the following:} \} \notag \\
\tau \in S_{p+1} &\Rightarrow s (\sigma \circ \tau) = \epsilon (\tau) s (\sigma) \notag \\
\sigma \text{ not injective} &\Rightarrow s (\sigma) = 0. \notag
\end{align}
\end{defn}
Some references define \v{C}ech cohomology with cochains that aren't necessarily alternating.  However, \citet[page~153]{Tenn75}, \citet[page~218]{Hart77} and, for \v{C}ech cohomology with constant coefficients, \citet[pages~162--176]{ES52} imply that an equivalent theories result.  In some places, e.g. the proof of proposition \ref{p-dd-bg-susp-c1-pb}, this thesis depends on alternating cochains defining the Dixmier-Douady class, but it might work just as well not to require them in the definition, and instead just to make certain choices in places where it matters, to ensure that they're alternating.

Likewise, some references only define cochains to have values for index tuples corresponding to nonempty intersections of open sets, but since we use inverse images of covers, empty intersections in one space may arise naturally from nonempty intersections in another.  For us, $U_{\sigma} = \emptyset \Rightarrow s (\sigma) \in A(\emptyset) = \{ 0 \}$.

\begin{defn}\label{d-cech-coho}
\index{Cech cohomology@\v{C}ech cohomology}
\index{cohomology!Cech@\v{C}ech}
(\v{C}ech Cohomology).
\citep[pages~140--142,~153]{Tenn75} Given a topological space $X$, an indexed open cover $U = \{ U_i \}_{i \in I}$ of $X$, and a presheaf $A$ of abelian groups over $X$, define the coboundary operator
\begin{align}
\delta \colon \check{C}^p(U;A) &\rightarrow \check{C}^{p+1}(U;A) \text{ by} \notag \\
\delta c(i_0,\dots,i_p) &= \sum_{k=0}^{p+1} (-1)^k r_{U_{i_0,\dots,i_{p+1}}, U_{i_0,\dots,\widehat{i_k},\dots,i_{p+1}}} (c(i_0,\dots,\widehat{i_k},\dots,i_{p+1})), \notag
\end{align}
where the $r$ are the restriction homomorphisms of $A$.  Denote the cohomology groups from the cochain complex defined by these cochains and coboundary operators by $\CH^p(U;A)$; we will call them \v{C}ech cohomology groups for the (indexed) cover $U$.

Define the directed class $\Lambda$ of indexed covers of $X$, ordered by $U \preceq V$ when $V$ refines $U$, which induces a cochain map and resulting refinement homomorphism  $h_{VU} \colon \CH^p(U;A) \rightarrow \CH^p(V;A)$.  Define the directed set $\Lambda_T \subset 2^T$, of covers indexed by subsets of $T$, the topology of $X$.  Define the $p$-th \v{C}ech cohomology group of the space $X$ as the direct limit
\[
\CH^p(X;A) = \varinjlim_{U \in \Lambda_T} \CH^p(U;A) \notag
\]
and denote the insertion homomorphisms of the direct limit by $h_{XU} \colon \CH^p(U;A) \rightarrow \CH^p(X;A)$; the same symbols are used for the homomorphisms for any $p$.
\end{defn}
\begin{note}\label{n-cech-coho}
\index{Cech cohomology@\v{C}ech cohomology!notes}
(\v{C}ech Cohomology).
From the properties of direct limits, given three covers $U \preceq V \preceq W$, $h_{WU} = h_{WV} \circ h_{VU}$; also $h_{XU} = h_{XV} \circ h_{VU}$.

The definition applies to sheaves as well as presheaves.  All the steps of these constructions are functorial in the presheaf $A$.

$\Lambda_T$ is a directed set, since any two indexed covers have a common refinement, an indexed cover that refines both of them, that is an indexed cover without repetitions, indexed by itself; i.e. the index of each open set is that open set.  It is a cofinal class of $\Lambda$ \citep[pages~390--391]{Rotm09}.  Given any indexed cover $U = \{ U_i \st i \in I \}$ of $X$, let $V = U$; that is, let $V$ be the set of values $U_i$, a subset of $T$.  Index $V$ by $J = V$, a subset of $T$; i.e. $V_j = j$.  Then define a refinement map $\phi \colon J \rightarrow I$ by $\phi(j) = \text{ some } i \text{ for which } V_j = U_i$, using the axiom of choice.

Two refinement maps from a given indexed cover to another may produce different cochain maps, but the choice of refinement map disappears in cohomology: \citet[page~142]{Tenn75} shows that they induce the same homomorphism in \v{C}ech cohomology of covers.

A similar argument with a cofinal directed set lets us use results involving constant coefficients from \citet{Dowk50}, which does not appear to index covers, since a cover without repetitions can be indexed by itself to satisfy our definition.

Again, in understanding why the alternative restriction of \citet[page~143]{Tenn75} to covers indexed by $X$, with $x \in U_x$ for all $x \in X$, results in a cofinal directed set in the class of all indexed covers, one can refine any indexed cover to one of these by refining it using Zorn's Lemma to a inverse-inclusion maximal subset that is still a cover, which will have the property that no element is contained in the union of other elements, and so is a cover of the first type.
\end{note}

Although sheaves are presheaves so that the definition of \v{C}ech cohomology applies to them, and \v{C}ech cohomology for presheaves over a topological space is an exact $\partial$-functor, it requires a condition on the space to ensure that the same is true for sheaves, because exact sequences of sheaves over $X$  aren't necessarily exact in the category of presheaves over $X$.
\begin{prop}\label{p-cech-coho-pres-shea}
\index{Cech cohomology@\v{C}ech cohomology!presheaf}
\index{cohomology!Cech@\v{C}ech!presheaf}
\index{Cech cohomology@\v{C}ech cohomology!sheaf}
\index{cohomology!Cech@\v{C}ech!sheaf}
(\v{C}ech Cohomology of Presheaves and Sheaves).
\citep[pages~143--145]{Tenn75} \v{C}ech cohomology of presheaves is an exact $\partial$-functor (as in definition \ref{d-exac-delt-func}) from presheaves over a topological space $X$, $Presh/X$, to the category of abelian groups.

\citep[pages~144--147]{Tenn75} If $X$ is paracompact, the same is true of \v{C}ech cohomology of sheaves, the salient difference being that the exact sequences in definition \ref{d-exac-delt-func} are now exact in the category $Shv/X$ rather than $Presh/X$.
\end{prop}
\begin{proof}
Most of the proof is in \citet[pages~144--145]{Tenn75}, who starts by saying that a short exact sequence of presheaves gives, for any cover, a short exact sequence of \v{C}ech cochain complexes for that cover.  Then lemma 2.12 of \citet[page~126]{Tenn75} gives from the short exact sequence of cochain complexes of abelian groups, a long exact sequence in homology with connecting homomorphisms that commute with homomorphisms induced from a morphism of sequences.

Given a short exact sequence of sheaves over $X$, $0 \rightarrow A \xrightarrow{f} B \xrightarrow{g} C \rightarrow 0$, use the factorization of $g$ through $\pcoker(f)$ to obtain the following, where $n_{\pcoker (f)}$ is the natural map from sheafification (see lemma \ref{l-pres-shea}) and we abuse notation to call $\pcoker(f)$, of lemma \ref{l-pres-shea}, a presheaf object:
\[
\begindc{\commdiag}[5]
\obj(10,10)[obj0L]{$0$}
\obj(20,10)[objA]{$A$}
\obj(30,10)[objB]{$B$}
\obj(50,10)[objC]{$C$}
\obj(50,25)[objPCK]{$\pcoker(f)$}
\obj(65,25)[obj0RU]{$0$}
\obj(65,10)[obj0RL]{$0$}
\mor{obj0L}{objA}{}
\mor{objA}{objB}{}
\mor{objB}{objPCK}{}
\mor{objB}{objC}{}
\mor{objPCK}{objC}{$n_{\pcoker (f)}$}
\mor{objPCK}{obj0RU}{}
\mor{objC}{obj0RL}{}
\enddc
\]
yielding the long exact sequence in cohomology for the short exact sequence of presheaves in the row through the top branch, with natural isomorphisms $\CH^p (X; \pcoker(f)) \xrightarrow{n_{\pcoker (f)}^*} \CH^p (X; C)$ for all $p \ge 0$.

\citep[pages~145--147]{Tenn75} explains that if $X$ is paracompact, then the map from sheafification, $n_{\pcoker (f)}$, induces an isomorphism in cohomology.  However, he doesn't show why, whenever
\[
\begindc{\commdiag}[5]
\obj(10,30)[obj0UL]{$0$}
\obj(20,30)[objA]{$A$}
\obj(30,30)[objB]{$B$}
\obj(40,30)[objC]{$C$}
\obj(50,30)[obj0UR]{$0$}
\obj(10,20)[obj0LL]{$0$}
\obj(20,20)[objAP]{$A'$}
\obj(30,20)[objBP]{$B'$}
\obj(40,20)[objCP]{$C'$}
\obj(50,20)[obj0LR]{$0$}
\mor{obj0UL}{objA}{}
\mor{objA}{objB}{$f$}
\mor{objB}{objC}{$g$}
\mor{objC}{obj0UR}{}
\mor{obj0LL}{objAP}{}
\mor{objAP}{objBP}{$f'$}[\atright,\solidarrow]
\mor{objBP}{objCP}{$g'$}[\atright,\solidarrow]
\mor{objCP}{obj0LR}{}
\mor{objA}{objAP}{$a$}
\mor{objB}{objBP}{$b$}
\mor{objC}{objCP}{$c$}
\enddc\]
commutes and has exact rows in $Shv/X$, then the corresponding diagrams
\[
\begindc{\commdiag}[5]
\obj(15,30)[objCHC]{$\CH^p (X; C)$}
\obj(15,15)[objCHCP]{$\CH^p (X; C')$}
\obj(35,30)[objCHA]{$\CH^{p+1} (X; A)$}
\obj(35,15)[objCHAP]{$\CH^{p+1} (X; A')$}
\mor{objCHC}{objCHCP}{$\CH^p (c) = c^*$}[\atright,\solidarrow]
\mor{objCHA}{objCHAP}{$a^* = \CH^{p+1} (a)$}
\mor{objCHC}{objCHA}{$\partial$}
\mor{objCHCP}{objCHAP}{$\partial'$}
\enddc
\]
for all $p \ge 0$ commute; i.e. $\partial$ is natural.  We will construct a morphism
\[
\widetilde{c} \colon \pcoker(f) \rightarrow \pcoker(f') \notag
\]
to make the diagram below commute, then define $\partial$ for the long sequence using sheaves $C$ and $C'$ so that the sequence is exact and we get the desired naturality, following from naturality of $\partial$ for the long exact sequence using the presheaf cokernels, and from functoriality of the \v{C}ech cohomology functor.
\[
\begindc{\commdiag}[5]
\obj(10,30)[obj0UL]{$0$}
\obj(20,30)[objA]{$A$}
\obj(30,30)[objB]{$B$}
\obj(40,30)[objC]{$C$}
\obj(50,40)[objPCK]{$\pcoker(f)$}
\obj(60,40)[obj0PCK]{$0$}
\obj(60,30)[obj0UR]{$0$}
\obj(10,20)[obj0LL]{$0$}
\obj(20,20)[objAP]{$A'$}
\obj(30,20)[objBP]{$B'$}
\obj(40,20)[objCP]{$C'$}
\obj(50,10)[objPCKP]{$\pcoker(f')$}
\obj(60,10)[obj0PCKP]{$0$}
\obj(60,20)[obj0LR]{$0$}
\mor{obj0UL}{objA}{}
\mor{objA}{objB}{$f$}[\atright,\solidarrow]
\mor{objB}{objC}{$g$}[\atright,\solidarrow]
\mor{objB}{objPCK}{$\pi$}
\mor{objPCK}{objC}{$n_{\pcoker(f)}$}
\mor{objPCK}{obj0PCK}{}
\mor{objC}{obj0UR}{}
\mor{obj0LL}{objAP}{}
\mor{objAP}{objBP}{$f'$}
\mor{objBP}{objCP}{$g'$}
\mor{objBP}{objPCKP}{$\pi'$}[\atright,\solidarrow]
\mor{objPCKP}{objCP}{$n_{\pcoker(f')}$}[\atright,\solidarrow]
\mor{objPCKP}{obj0PCKP}{}
\mor{objCP}{obj0LR}{}
\mor{objA}{objAP}{$a$}
\mor{objB}{objBP}{$b$}
\mor{objC}{objCP}{$c$}
\mor{objPCK}{objPCKP}{$\widetilde{c}$}[\atleft,\dashArrow]
\enddc\]
Take any open set $U \subset X$ and any element $p \in \pcoker(f) (U)$, which since $\pi \colon B(U) \rightarrow \pcoker(f)(U)$ is surjective (see lemma \ref{l-pres-shea}), is given by $p = \pi (\beta)$ for some $\beta \in B(U)$. Define $\widetilde{c} (p) = \pi' b (\beta)$.

$\widetilde{c}$ is a homomorphism, and is well defined, for suppose $\pi (\beta') = \pi (\beta)$, then $\beta' = \beta + f(\alpha)$ for some $\alpha \in A(U)$, and
\begin{align}
\pi' b (\beta') &= \pi' b (\beta) + \pi' b f(\alpha) \notag \\
&= \pi' b (\beta) + \pi' f' a (\alpha) \notag \\
&= \pi' b (\beta). \notag
\end{align}
With $\widetilde{c}$ added to the diagram, it still commutes, because
\begin{align}
n_{\pcoker(f')} \widetilde{c} (p) &= n_{\pcoker(f')} \widetilde{c} \pi (\beta) \notag \\
&= n_{\pcoker(f')} \pi' b (\beta) \notag \\
&= g' b (\beta) \notag \\
&= c g (\beta) \notag \\
&= c n_{\pcoker(f)} \pi (\beta) \notag \\
&= c n_{\pcoker(f)} (p). \notag
\end{align}

Since this is true for any open $U$, $\widetilde{c}$ is a morphism of presheaves.  By functoriality of \v{C}ech cohomology for presheaves, the commutative diagram of presheaves induces one in cohomology.  Using $\partial$ from the short exact sequence of presheaves to define $\partial_{Shv} = \partial (n_{\pcoker(f)}^*)^{-1}$ and likewise $\partial_{Shv}' = \partial' (n_{\pcoker(f')}^*)^{-1}$, we have
\begin{align}
a^* \partial_{Shv} &= a^* \partial (n_{\pcoker(f)}^*)^{-1} \notag \\
&= \partial' \widetilde{c}^* (n_{\pcoker(f)}^*)^{-1} \notag \\
&= \partial_{Shv}' n_{\pcoker(f')}^* \widetilde{c}^* (n_{\pcoker(f)}^*)^{-1} \notag \\
&= \partial_{Shv}' c^*. \notag
\end{align}
The exactness of the long sequence using $C, C'$ follows from that of the long sequence using the presheaf cokernels, because
\begin{align}
\partial_{Shv} &= \partial (n_{\pcoker(f)}^*)^{-1} \notag \\
g &= \pi (n_{\pcoker(f)}^*)^{-1}, \notag
\end{align}
and similarly for $\partial_{Shv}', g'$.
\end{proof}

\begin{lem}\label{l-exp-exac-seq}
\index{exponential exact sequence}
\index{sheaf!exponential exact sequence}
(An Exponential Exact Sequence of Sheaves).
% removed ~ to cure overfull box
\citep[pages 174, 467]{Bred97} Given a topological space $X$, the following sequence of sheaves over $X$ is exact:
\[
0 \rightarrow \underline{\ZZ} \rightarrow \underline{\RR} \rightarrow \underline{\UU(1)} \rightarrow 0 \notag
\]
If $X$ is paracompact this yields the following isomorphism for $p > 0$, namely the connecting homomorphism for the corresponding long exact sequence:
\begin{align}
\CH^p(X;\underline{\UU(1)}) &\isomto \CH^{p+1}(X;\underline{\ZZ}). \notag
\end{align}
\end{lem}
\begin{proof}
The reference states that the sequence is an exact sequence of sheaves, and assuming that $X$ is paracompact, that $\underline{\RR}$ is a soft sheaf, whence by \citet[page~165,~130,~179]{Tayl02} it is acyclic in terms of sheaf cohomology, meaning its sheaf cohomology groups of positive degree vanish, which again assuming that $X$ is paracompact, implies the same of its \v{C}ech cohomology groups: $\CH^p (X; \underline{\RR}) = 0$ for $p > 0$.  Then use the long exact sequence of proposition \ref{p-cech-coho-pres-shea}.
\end{proof}
In general, this short sequence is not exact as a sequence of presheaves.

\begin{lem}\label{l-cech-coho-indu-map-morp-over}
\index{Cech cohomology@\v{C}ech cohomology!morphism over!induced map}
\index{induced map!morphism over}
\index{morphism over!induced map}
\index{morphism!Presh@$Presh$!induced map}
(The Induced Map in \v{C}ech Cohomology from a Morphism in $Presh$).
Given a continuous map $f \colon X \rightarrow Y$ of topological spaces, presheaves $A$ over $X$, $B$ over $Y$, and a morphism $\overline{f} \colon B \rightarrow A$ over $f$, there is an induced map in cohomology for every $p$, $\overline{f}_* \colon \CH^p (Y; B) \rightarrow \CH^p (X; A)$.

Given a sequence of continuous maps $X \xrightarrow{f} Y \xrightarrow{g} Z$, presheaves $A$ over $X$, $B$ over $Y$, $C$ over $Z$, morphisms $\overline{f} \colon B \rightarrow A$ over $f$ and $\overline{g} \colon C \rightarrow B$ over $g$, the induced map of the composition is the composition of the induced maps: $(\overline{f} \circ \overline{g})_* = \overline{f}_* \circ \overline{g}_*$.

Given an exact sequence of presheaves over $X$, $0 \rightarrow A_X \xrightarrow{f_X} B_X \xrightarrow{g_X} C_X \rightarrow 0$, an exact sequence of presheaves over $Y$, $0 \rightarrow A_Y \xrightarrow{f_Y} B_Y \xrightarrow{g_Y} C_Y \rightarrow 0$, a continuous map $\phi \colon X \rightarrow Y$, morphisms over $\phi$, $\overline{a} \colon A_Y \rightarrow A_X$, $\overline{b} \colon B_Y \rightarrow B_X$, $\overline{c} \colon C_Y \rightarrow C_X$, suppose that the following diagram commutes:
\[
\begindc{\commdiag}[5]
\obj(10,30)[obj0UL]{$0$}
\obj(20,30)[objAY]{$A_Y$}
\obj(30,30)[objBY]{$B_Y$}
\obj(40,30)[objCY]{$C_Y$}
\obj(50,30)[obj0UR]{$0$}
\obj(10,20)[obj0LL]{$0$}
\obj(20,20)[objAX]{$A_X$}
\obj(30,20)[objBX]{$B_X$}
\obj(40,20)[objCX]{$C_X$}
\obj(50,20)[obj0LR]{$0$}
\mor{obj0UL}{objAY}{}
\mor{objAY}{objBY}{$f_Y$}
\mor{objBY}{objCY}{$g_Y$}
\mor{objCY}{obj0UR}{}
\mor{obj0LL}{objAX}{}
\mor{objAX}{objBX}{$f_X$}[\atright,\solidarrow]
\mor{objBX}{objCX}{$g_X$}[\atright,\solidarrow]
\mor{objCX}{obj0LR}{}
\mor{objAY}{objAX}{$\overline{a}$}
\mor{objBY}{objBX}{$\overline{b}$}
\mor{objCY}{objCX}{$\overline{c}$}
\enddc\]
Then the corresponding diagrams
\[
\begindc{\commdiag}[5]
\obj(15,30)[objCHCY]{$\CH^p (Y; C_Y)$}
\obj(15,15)[objCHCX]{$\CH^p (X; C_X)$}
\obj(35,30)[objCHAY]{$\CH^{p+1} (Y; A_Y)$}
\obj(35,15)[objCHAX]{$\CH^{p+1} (X; A_X)$}
\mor{objCHCY}{objCHCX}{$\overline{c}*$}[\atright,\solidarrow]
\mor{objCHAY}{objCHAX}{$\overline{a}^*$}
\mor{objCHCY}{objCHAY}{$\partial_Y$}
\mor{objCHCX}{objCHAX}{$\partial_X$}
\enddc
\]
for all $p \ge 0$ commute; i.e., $\partial$ is natural.

If $X$ and $Y$ are paracompact, the same is true given exact sequences of sheaves.
\end{lem}
\begin{proof}
First, the existence of the induced map.  Let $V$ be a cover of $Y$ indexed by $I \subset T_Y$, the topology of $Y$; then its inverse image or pullback, which we will denote $f^* V$, is a cover of $X$ also indexed by $I$.  If $W$ is a cover of $Y$ indexed by $J$, refining $V$, then $f^* W$ is a cover of $X$, also indexed by $J$, refining $f^* V$, and the following diagram of \v{C}ech cochain complexes commutes.  Details of the maps follow.  Hence after further justification, the diagram of \v{C}ech cohomology groups of covers also commutes, for all $p \ge 0$.

For $k \ge 0$ define $\overline{f}_{\#} \colon \check{C}^k (V; B) \rightarrow \check{C}^k (f^* V; A)$ as follows.  Take $s \in \check{C}^k (V; B)$; on any $(k + 1)$-tuple $\sigma$ of elements of $I$, $s(\sigma) \in B (V_{\sigma})$.  Then let $\overline{f}_{\#} (s) (\sigma) = \overline{f} (f^{-1} (V_{\sigma}), V_{\sigma}) (s(\sigma)) \in A (f^{-1} (V_{\sigma}))$, defining $\overline{f}_{\#} (s) \in \check{C}^k (f^* V; A)$.

Use $r$, subscripts pairs of open sets, to denote the restriction homomorphisms for $A$ and $B$.  They induce cochain maps we call $r$ again, subscripts pairs of covers, that depend on the choice of a refining map $\phi \colon J \rightarrow I$.  For example, start with $s \in \check{C}^k (V; B)$, and take $\tau$ any $(k + 1)$-tuple of elements of $J$.  Then $\phi \circ \tau$ is a $(k + 1)$-tuple of elements of $I$, and $s (\phi \circ \tau) \in B (V_{\phi \circ \tau})$.  Referring to definition \ref{d-inde-cove-refi}, since for every $j \in J$, $W_j \subset V_{\phi (j)}$, $W_{\tau} \subset V_{\phi \circ \tau}$, we may let $r_{W,V} (s) (\tau) = r_{W_{\tau}, V_{\phi \circ \tau}} (s (\phi \circ \tau)) \in \check{C}^k (W; B)$, defining $r_{W,V} (s)$.  This is the same definition as would be used if $\overline{f}$ were a morphism of presheaves over one space as in part of the proof of proposition \ref{p-cech-coho-pres-shea}, and no doubt as in that proof's reference if it had given a definition.

The facts that $r$ and $\overline{f}_{\#}$ are cochain maps and that they commute, follow from definitions \ref{d-morp-over}, \ref{d-cech-coho}, and the functorial composition of the restriction homomorphisms $r$ of a presheaf.  That $\overline{f}_{\#}$ preserves the alternating and non-injective zero requirement for a cochain element follows from the definitions, as does the same fact for $r$. For $r$, if $\tau \in J_k$ is not injective, that is true also of $\phi \circ \tau$, so that $s(\phi \circ \tau) = 0$.  If $g \in S_{k + 1}$, $s (\phi \circ (\tau \circ g)) = s ((\phi \circ \tau) \circ g) = \epsilon (g) s (\phi \circ \tau)$.

$h$ denotes the refinement homomorphisms from definition \ref{d-cech-coho}, derived from the cochain maps $r$ by taking homology of cochain complexes, which is a functor; that also explains why the cohomology diagram commutes.  As in note \ref{n-cech-coho}, although the cochain maps $r$ depend on the choice of refinement maps for the covers, the maps $h$ in cohomology do not.
\[
\begindc{\commdiag}[5]
\obj(15,30)[objCCBV]{$\check{C}^* (V; B)$}
\obj(35,30)[objCCAV]{$\check{C}^* (f^* V; A)$}
\obj(15,15)[objCCBW]{$\check{C}^* (W; B)$}
\obj(35,15)[objCCAW]{$\check{C}^* (f^* W; A)$}
\mor{objCCBV}{objCCBW}{$r_{W,V}$}[\atright,\solidarrow]
\mor{objCCAV}{objCCAW}{$r_{f^* W,f^* V}$}
\mor{objCCBV}{objCCAV}{$\overline{f}_{\#}$}
\mor{objCCBW}{objCCAW}{$\overline{f}_{\#}$}

\obj(57,30)[objCHBV]{$\CH^p (V; B)$}
\obj(77,30)[objCHAV]{$\CH^p (f^* V; A)$}
\obj(57,15)[objCHBW]{$\CH^p (W; B)$}
\obj(77,15)[objCHAW]{$\CH^p (f^* W; A)$}
\mor{objCHBV}{objCHBW}{$h_{V,W}$}[\atright,\solidarrow]
\mor{objCHAV}{objCHAW}{$h_{f^* V,f^* W}$}
\mor{objCHBV}{objCHAV}{$\overline{f}_*$}
\mor{objCHBW}{objCHAW}{$\overline{f}_*$}
\enddc
\]
We can take the direct limit of the cohomology groups for $B$ over covers of $Y$ indexed by subsets of $T_Y$ to obtain the \v{C}ech cohomology groups of the space $Y$.  However, if we take the direct limit of the cohomology groups for $A$ over the inverse image covers, which are indexed by subsets of $T_Y$ rather than of $T_X$, the topology of $X$, and even ignoring indexing may not include all covers of $X$, we will not in general get the cohomology groups of the space $X$.  But we don't expect isomorphisms anyway, only homomorphisms of cohomology groups induced by $\overline{f}$, and can get this from morphism of direct systems, which induce direct limit homomorphisms.

Denote by $(\Lambda, \CH^p (V_{(\cdot)}; B), h_{(\cdot, \cdot)})$ the direct system that will be the domain of the morphism, with $\Lambda \subset 2^{T_Y}$.  Since we will refine the inverse image covers by covers with no repetitions, indexed by themselves, subsets of $T_X$, somewhat as for $\Lambda_T$ in note \ref{n-cech-coho}, let $\Lambda' \subset 2^{T_X}$ be the directed set indexing the covers $U$ of $X$, and let $\Psi \colon \Lambda \rightarrow \Lambda'$ take the index for a given cover of $Y$ and hence the corresponding inverse image cover of $X$, to its self-indexed refinement with no repetitions.  That is, from $\alpha \in \Lambda$ we get $f^{-1} (V_{\alpha})$, then considering that just as a set of open sets of $X$, call it $\beta \in \Lambda'$, and set $\Psi(\alpha) = \beta$.

Denote by $(\Lambda', \CH^p (U_{(\cdot)}; A), h_{(\cdot, \cdot)})$ the direct system that will be the codomain of the morphism.  To construct the natural transformation $\psi$ for the morphism of direct systems, $\overline{f}_* \colon \CH^p (V_{(\cdot)}; B) \rightarrow \CH^p (f^{-1} (V_{(\cdot)}); A)$ gets us halfway there and has the correct compatibility with refinement homomorphisms.

For any $\alpha \in \Lambda$, $U_{\beta} = U_{\Psi{\alpha}}$ refines $f^{-1} (V_{\alpha})$, since for every $j \in \beta$, $j$ being an element of the self-indexed cover $U_{\beta}$, there is some $i \in \alpha$, recalling that $V_{\alpha}$ also is self-indexed, such that $(U_{\beta})_j = f^{-1} ((V_{\alpha})_i)$.  Thus we get the rest of the way there with the refinement homomorphism $h_{U_{\beta}, f^{-1} (V_{\alpha})} \colon \CH^p (f^{-1} (V_{\alpha}); A) \rightarrow \CH^p (U_{\beta}; A)$.  That is, let $\psi$ for this $\alpha$ be the composition
\[
\CH^p (V_{\alpha}; B) \xrightarrow{\overline{f}_*} \CH^p (f^{-1} (V_{\alpha}); A) \xrightarrow{h_{U_{\beta}, f^{-1} (V_{\alpha})}} \CH^p (U_{\Psi(\alpha)}; A). \notag
\]
Recall that the refinement homomorphism in cohomology does not depend on the choice of refinement map for covers.  Then by note \ref{n-dire-limi} we get the existence of the induced map:
\[
\begindc{\commdiag}[5]
\obj(57,30)[objCHBV]{$\CH^p (Y; B)$}
\obj(77,30)[objCHAV]{$\CH^p (X; A)$}
\mor{objCHBV}{objCHAV}{$\overline{f}_*$}
\enddc
\]

Second, the assertion about composition of induced maps follows by similar reasoning.  Let $W$ be a cover of $Z$; then $g^* W$ is a cover of $Y$, and $f^* g^* W$ is a cover of $X$.  To start with, the morphisms over $f$ and $g$ induce cochain maps
\[
\check{C}^* (W; C) \xrightarrow{\overline{g}_{\#}} \check{C}^* (g^* W; B) \xrightarrow{\overline{f}_{\#}} \check{C}^* (f^* g^* W; A). \notag
\]
The cochain map induced by $\overline{f} \circ \overline{g}$, $(\overline{f} \circ \overline{g})_{\#} = \overline{f}_{\#} \circ \overline{g}_{\#}$.  The cochain maps induce:
\[
\CH^p (W; C) \xrightarrow{\overline{g}_*} \CH^p (g^* W; B) \xrightarrow{\overline{f}_*} \CH^p (f^* g^* W; A), \notag
\]
which compose according to the functoriality of taking homology of a cochain complex.  Constructing two successive morphisms of direct systems as before, for example choosing the refinement map for covers for going from beginning to end as the composition of the successive refinement maps, but noting again that the refinement homomorphism in cohomology doesn't depend on the particular refinement map chosen for covers, the composition of the morphisms of the direct systems is the morphism that we would construct to go from beginning to end. Then the composition property for induced maps in the statement follows from the functoriality of direct limits as in note \ref{n-dire-limi}.

Third, the entire top row short exact sequence of presheaves is over one space $Y$; and all of the bottom row is over another space $X$.  So we obtain everything except for naturality of $\partial$ directly from proposition \ref{p-cech-coho-pres-shea}.  However, as seen in this proof already, from morphisms of $Presh$ we obtain morphisms of cochain complexes of abelian groups.  At this point it no longer matters whether the morphisms of cochain complexes came from morphisms of $Presh$ or morphisms of $Presh/X$, except for the different covers for the top to bottom parts of the diagram.  Just as in the proof of proposition \ref{p-cech-coho-pres-shea}, we get exact rows of cochain complexes of abelian groups and a morphism of cochain complexes between them, from the morphisms of $Presh$, so that using lemma 2.12 of \citet[page~126]{Tenn75}, we get the naturality of $\partial$ for the long exact sequences of \v{C}ech cohomology for covers $V$ of $Y$ and $\phi^* V$ for $X$, for all $p \ge 0$:
\[
\begindc{\commdiag}[5]
\obj(15,30)[objCHCV]{$\CH^p (V; C_Y)$}
\obj(15,15)[objCHCPM1V]{$\CH^p (\phi^* V; C_X)$}
\obj(40,30)[objCHAV]{$\CH^{p+1} (V; A_Y)$}
\obj(40,15)[objCHAPM1V]{$\CH^{p+1} (\phi^* V; A_X)$}
\mor{objCHCV}{objCHCPM1V}{$\overline{c}*$}[\atright,\solidarrow]
\mor{objCHAV}{objCHAPM1V}{$\overline{a}^*$}
\mor{objCHCV}{objCHAV}{$\partial_Y$}
\mor{objCHCPM1V}{objCHAPM1V}{$\partial_X$}
\enddc
\]
Either passing to the direct limits first for the right side and then the left, using the universal property for the direct limits on the left, or by constructing once again morphisms of direct systems, we obtain the diagram for naturality of $\partial$ for \v{C}ech cohomology of the spaces.

The result for exact sequences of sheaves follows from that for presheaves by almost identical reasoning to that used in corresponding proof in proposition \ref{p-cech-coho-pres-shea}.  Referring to that proof, for the top set of presheaves we evaluate on an open set $V$, and for the bottom, on $\phi^{-1} (V)$; though we still have the homomorphisms of abelian groups as in that proof.  Then $\widetilde{c}$ is a morphism of $Presh$ instead of a morphism of $Presh/X$, but it still induces a map in cohomology, and by the composition property for maps in cohomology induced by morphisms of $Presh$ (of which morphisms of $Shv$ are a special case), we obtain from the commutative diagram of presheaves, one in cohomology.
\end{proof}

\begin{cor}\label{co-cech-coho-indu-map-cont-func}
\index{Cech cohomology@\v{C}ech cohomology!induced map}
\index{cohomology!Cech@\v{C}ech!induced map}
\index{induced map!Cech cohomology@\v{C}ech cohomology}
(The Induced Map in \v{C}ech Cohomology for Sheaves of Continuous Functions).
(Generalized from \citet[page~239]{ES52}) Given an abelian topological group $G$ and a topological space $X$, let $G_X$ denote the presheaf of continuous functions with values in $G$; i.e. for open $U \subset X$, let $G_X (U)$ be the abelian group of continuous functions $U \rightarrow G$.  As in example \ref{e-sheaves}, $G_X$ is a sheaf.

Given a continuous map $\phi \colon X \rightarrow Y$ of topological spaces, define a morphism $\overline{\phi}$ over $\phi$ as in definition \ref{d-morp-over} as follows, where $V \subset Y$ ranges over all open sets:
\begin{align}
\overline{\phi}(\phi^{-1} (V), V) \colon G_Y (V) &\rightarrow G_X(\phi^{-1} (V)) \notag \\
\overline{\phi}(\phi^{-1} (V), V) \colon s &\mapsto s \circ \phi_{|\phi^{-1} (V)}. \notag
\end{align}
Composition of continuous maps is reflected in composition of these morphisms of $Presh$.

By lemma \ref{l-cech-coho-indu-map-morp-over}, $\overline{\phi}$ induces homomorphisms in \v{C}ech cohomology,
\[
\phi^* = \overline{\phi}_* \colon \CH^p (Y; G_Y) \rightarrow \CH^p (X; G_X), \notag
\]
composition of morphisms of $Presh$ are reflected in composition of induced maps in cohomology, and the induced maps in cohomology commute with connecting homomorphisms.

Given abelian topological groups $A$, $B$, $C$, suppose that the rows in the following commutative diagram are exact in categories of presheaves, or if one or both rows are exact in categories of sheaves, then the corresponding spaces are paracompact:
\[
\begindc{\commdiag}[5]
\obj(10,30)[obj0UL]{$0$}
\obj(20,30)[objAY]{$A_Y$}
\obj(30,30)[objBY]{$B_Y$}
\obj(40,30)[objCY]{$C_Y$}
\obj(50,30)[obj0UR]{$0$}
\obj(10,20)[obj0LL]{$0$}
\obj(20,20)[objAX]{$A_X$}
\obj(30,20)[objBX]{$B_X$}
\obj(40,20)[objCX]{$C_X$}
\obj(50,20)[obj0LR]{$0$}
\mor{obj0UL}{objAY}{}
\mor{objAY}{objBY}{$f_Y$}
\mor{objBY}{objCY}{$g_Y$}
\mor{objCY}{obj0UR}{}
\mor{obj0LL}{objAX}{}
\mor{objAX}{objBX}{$f_X$}[\atright,\solidarrow]
\mor{objBX}{objCX}{$g_X$}[\atright,\solidarrow]
\mor{objCX}{obj0LR}{}
\mor{objAY}{objAX}{$\overline{\phi}$}
\mor{objBY}{objBX}{$\overline{\phi}$}
\mor{objCY}{objCX}{$\overline{\phi}$}
\enddc\]
Then the corresponding diagrams
\[
\begindc{\commdiag}[5]
\obj(15,30)[objCHCY]{$\CH^p (Y; C_Y)$}
\obj(15,15)[objCHCX]{$\CH^p (X; C_X)$}
\obj(35,30)[objCHAY]{$\CH^{p+1} (Y; A_Y)$}
\obj(35,15)[objCHAX]{$\CH^{p+1} (X; A_X)$}
\mor{objCHCY}{objCHCX}{$\phi^*$}[\atright,\solidarrow]
\mor{objCHAY}{objCHAX}{$\phi^*$}
\mor{objCHCY}{objCHAY}{$\partial_Y$}
\mor{objCHCX}{objCHAX}{$\partial_X$}
\enddc
\]
for all $p \ge 0$ commute; in particular, $\partial$ is natural.
\end{cor}
\begin{proof}
All that's needed is to show that the $\overline{\phi}$ are maps over $\phi$; that they commute with restriction homomorphisms.  Take open $V' \subset V \subset Y$, then $\phi^{-1} (V') \subset \phi^{-1} (V) \subset X$.  Then for example, for $s \in C_Y(V)$, referring to definition \ref{d-morp-over} and noting that since we define $\overline{\phi}$ on pairs of open sets of the form $(\phi^{-1} (V), V)$ by requiring compatibility with further restrictions to open $U \subset \phi^{-1} (V)$ that may not be inverse images of open sets in $Y$, we don't need to check those:
\begin{align}
r_{\phi^{-1} (V'), \phi^{-1} (V)} \circ \overline{\phi}(\phi^{-1} (V), V) (s) &= r_{\phi^{-1} (V'), \phi^{-1} (V)} (s \circ \phi_{|\phi^{-1} (V)}) \notag \\
&= s \circ \phi_{|\phi^{-1} (V')}, \text{ and } \notag \\
\overline{\phi} (\phi^{-1} (V'), V') \circ r_{V', V} (s) &= s \circ \phi_{|\phi^{-1} (V')}. \notag
\end{align}
\end{proof}

\begin{cor}\label{co-u1-z-diag}
\index{cohomology!Cech@\v{C}ech!induced map U(1) vs Z@induced map $\underline{\UU(1)}$ vs $\underline{\ZZ}$}
\index{induced map U(1) vs Z@induced map $\underline{\UU(1)}$ vs $\underline{\ZZ}$!Cech cohomology@\v{C}ech cohomology}
(The Induced Map in \v{C}ech Cohomology with $\underline{\UU(1)}$ vs. $\underline{\ZZ}$ Coefficients).
Given a continuous map $\phi \colon X \rightarrow Y$ of paracompact topological spaces, the following diagram, where $\partial_Y$ and $\partial_X$ are isomorphisms, commutes for $p > 0$:
\[
\begindc{\commdiag}[5]
\obj(15,30)[objCHCY]{$\CH^p (Y; \underline{\UU(1)}_Y)$}
\obj(15,15)[objCHCX]{$\CH^p (X; \underline{\UU(1)}_X)$}
\obj(35,30)[objCHAY]{$\CH^{p+1} (Y; \underline{\ZZ}_Y)$}
\obj(35,15)[objCHAX]{$\CH^{p+1} (X; \underline{\ZZ}_X)$}
\mor{objCHCY}{objCHCX}{$\phi^*$}[\atright,\solidarrow]
\mor{objCHAY}{objCHAX}{$\phi^*$}
\mor{objCHCY}{objCHAY}{$\partial_Y$}
\mor{objCHCX}{objCHAX}{$\partial_X$}
\enddc
\]
\end{cor}
\begin{proof}
True by lemma \ref{l-exp-exac-seq} and corollary \ref{co-cech-coho-indu-map-cont-func}.
\end{proof}

\begin{lem}\label{l-cech-sing-coho}
\index{Cech cohomology@\v{C}ech cohomology!singular cohomology}
\index{cohomology!Cech@\v{C}ech!singular}
\index{singular cohomology!Cech cohomology@\v{C}ech cohomology}
\index{cohomology!singular!Cech@\v{C}ech}
(When \v{C}ech Cohomology is Isomorphic to Other Cohomology).
Given a locally connected topological space $X$, true if every open cover has a refinement that is a good cover, $\CH^p(X;\underline{\ZZ})$ is naturally isomorphic to the \v{C}ech cohomology group $\CH^p(X;\ZZ)$ with constant $\ZZ$ coefficients defined in \citet[pages~233--237]{ES52} or \citet{Dowk50}, for all $p \ge 0$.  The isomorphism is natural with respect to maps in cohomology induced by maps between spaces.

Given a paracompact topological space $X$ with the property that every open cover has a refinement that is a good cover, there is a natural isomorphism from \v{C}ech cohomology with constant $\ZZ$ coefficients to singular cohomology, $\CH^p(X;\ZZ) \rightarrow \HH^p(X;\ZZ)$ for all $p \ge 0$.  As before, the isomorphism commutes with induced maps.
\end{lem}
\begin{proof}
Local connectedness follows from the requirement on covers because contractible sets are path connected, hence connected.  Since $X$ is locally connected, the connected components of its open sets are open sets \citep[page~113]{Dugu66}.

Thus any indexed open cover used in the definition of $\CH^p(X;\underline{\ZZ})$ can be refined as in note \ref{n-cech-coho} by replacing each member with that member's connected components, indexing the resulting subset of the topology of $X$ by itself.  Thus the \v{C}ech cohomology groups of coverings by connected open sets form a cofinal subset of those without the connected restriction, and so the direct limits using connected open sets are naturally isomorphic to those without the restriction by note \ref{n-dire-limi}.

Since on connected sets, elements of $\underline{\ZZ}$ are constant, there is an isomorphism of direct systems between ours with $\underline{\ZZ}$ coefficients but restricted to connected open sets, and the sources', which are defined using the integral cohomology of the nerve of each covering, resulting in $\ZZ$ coefficients constant on each set of a cover.  Thus we get a natural isomorphism between the corresponding direct limits.

For the relation with singular cohomology groups, the hypothesis of \citet[page~153]{Mard59}, $X$ paracompact, will be satisfied if for every $p \ge 0$, $x \in X$, and open neighborhood $U$ of $x$, there is an open neighborhood $V \subset U$ of $x$ such that the map in $p$-th singular homology induced by the inclusion $V \subset U$ is zero, using reduced singular homology groups for $p = 0$.  Then his result follows, that the singular and \v{C}ech cohomology groups of $X$ are naturally isomorphic, the isomorphism commuting with maps induced in cohomology by continuous maps of spaces.

To meet the hypothesis of \citet[page~153]{Mard59} in our situation, $\{ U, X \setminus \{ x \} \}$ is an open cover of $X$ and hence has a refinement of contractible open sets including some neighborhood $V \subset U$ of $x$.  Since $V$ is contractible, its relevant singular homology (reduced singular homology) groups are all $0$.  Thus we can apply his result.

He identifies his \v{C}ech cohomology groups as following a Spanier-Dowker definition rather than the ``usual'' one, but states that \citet[page~91]{Dowk52}, which calls the Spanier-Dowker groups the Alexander groups and calls the usual ones the \v{Cech} groups, shows that they are isomorphic, and that the isomorphism makes the induced maps correspond also.  \citet{Dowk50} shows that his \v{C}ech groups satisfy the axioms of \citet[pages~13--15]{ES52}.  The definitions in these papers of Dowker, as mentioned in note \ref{n-cech-coho}, and aside from possible set theoretic issues also the definition in \citet[pages~233--237]{ES52}, seem consistent with our definition of \v{C}ech cohomology, when as earlier, we use covers consisting of connected open sets.
\end{proof}
\begin{note}\label{n-cech-sing-coho}
\index{cohomology!Cech@\v{C}ech!diagrams U(1) vs Z@diagrams $\underline{\UU(1)}$ vs $\underline{\ZZ}$}
(Diagrams for \v{C}ech to Singular Cohomology).
For use in the following corollary, here are the diagrams for the naturality conclusions of the lemma for $p \ge 0$, given a continuous map $\phi \colon X \rightarrow Y$ of paracompact topological spaces that have the property that every open cover has a refinement that is a good cover:
\[
\begindc{\commdiag}[5]
\obj(15,30)[objCHCY]{$\CH^p (Y; \underline{\ZZ}_Y)$}
\obj(15,15)[objCHCX]{$\CH^p (X; \underline{\ZZ}_X)$}
\obj(35,30)[objCHAY]{$\CH^p (Y; \ZZ)$}
\obj(35,15)[objCHAX]{$\CH^p (X; \ZZ)$}
\mor{objCHCY}{objCHCX}{$\phi^*$}[\atright,\solidarrow]
\mor{objCHAY}{objCHAX}{$\phi^*$}
\mor{objCHCY}{objCHAY}{$\sim$}
\mor{objCHCX}{objCHAX}{$\sim$}
\obj(55,30)[objCHBY]{$\CH^p (Y; \ZZ)$}
\obj(55,15)[objCHBX]{$\CH^p (X; \ZZ)$}
\obj(75,30)[objHHY]{$\HH^p (Y; \ZZ)$}
\obj(75,15)[objHHX]{$\HH^p (X; \ZZ)$}
\mor{objCHBY}{objCHBX}{$\phi^*$}[\atright,\solidarrow]
\mor{objHHY}{objHHX}{$\phi^*$}
\mor{objCHBY}{objHHY}{$\sim$}
\mor{objCHBX}{objHHX}{$\sim$}
\enddc
\]
\end{note}

\begin{cor}\label{co-u1-z-sing-diag}
\index{cohomology!Cech@\v{C}ech!singular}
(The Induced Map in \v{C}ech Cohomology with $\underline{\UU(1)}$ Coefficients vs. Singular Cohomology with $\ZZ$ coefficients).
Given a continuous map $\phi \colon X \rightarrow Y$ of paracompact topological spaces such that each open cover has a refinement that is a good cover, the following diagram commutes for $p > 0$.  The horizontal maps are isomorphisms.
\[
\begindc{\commdiag}[5]
\obj(15,30)[objCHCY]{$\CH^p (Y; \underline{\UU(1)}_Y)$}
\obj(15,15)[objCHCX]{$\CH^p (X; \underline{\UU(1)}_X)$}
\obj(35,30)[objCHAY]{$\HH^{p+1} (Y; \ZZ)$}
\obj(35,15)[objCHAX]{$\HH^{p+1} (X; \ZZ)$}
\mor{objCHCY}{objCHCX}{$\phi^*$}[\atright,\solidarrow]
\mor{objCHAY}{objCHAX}{$\phi^*$}
\mor{objCHCY}{objCHAY}{$\sim$}
\mor{objCHCX}{objCHAX}{$\sim$}
\enddc
\]
\end{cor}
\begin{proof}
Compose the diagrams of corollary \ref{co-u1-z-diag} and note \ref{n-cech-sing-coho}.
\end{proof}

\begin{lem}\label{l-cech-coho-cove}
\index{Cech cohomology@\v{C}ech cohomology!of a cover}
\index{cohomology!Cech@\v{C}ech!of a cover}
(The \v{C}ech Cohomology of a Good Cover).
Suppose $X$ is a hereditarily paracompact topological space such that each open cover has a refinement that is a good cover, and $U$ is a cover of $X$ for which the singular cohomology in all positive degrees of all the nonempty intersections of elements of $U$ vanishes.  Then their sheaf cohomology with $\underline{\ZZ}$ or $\underline{\UU(1)}$ coefficients vanishes also, and for all $p \ge 0$ the following insertion homomorphisms of direct limits are isomorphisms:
\begin{align}
\CH^p(U;\underline{\ZZ}) &\isomto \CH^p(X;\underline{\ZZ}) \notag \\
\CH^p(U;\underline{\UU(1)}) &\isomto \CH^p(X;\underline{\UU(1)}). \notag
\end{align}
The vanishing singular cohomology condition is true if $U$ is a good cover.

Suppose given a continuous map $\phi \colon X \rightarrow Y$ of hereditarily paracompact topological spaces such that each open cover has a refinement that is a good cover.  Suppose also that $V$ is a good cover of $Y$, and that $\phi^* V$ is a good cover of $X$.  Then the following diagram, where the horizontal isomorphisms are the insertion homomorphisms of the direct limits for the groups on the right, commutes for $p \ge 0$:
\[
\begindc{\commdiag}[5]
\obj(15,30)[objCHCY]{$\CH^p (V; \underline{\UU(1)}_Y)$}
\obj(15,15)[objCHCX]{$\CH^p (\phi^* V; \underline{\UU(1)}_X)$}
\obj(35,30)[objCHAY]{$\CH^p (Y; \underline{\UU(1)}_Y)$}
\obj(35,15)[objCHAX]{$\CH^p (X; \underline{\UU(1)}_X)$}
\mor{objCHCY}{objCHCX}{$\phi^*$}[\atright,\solidarrow]
\mor{objCHAY}{objCHAX}{$\phi^*$}
\mor{objCHCY}{objCHAY}{$\sim$}
\mor{objCHCX}{objCHAX}{$\sim$}
\enddc
\]
The same is true for the analogous diagram with $\underline{\ZZ}$ coefficients.
\end{lem}
\begin{proof}
Since $X$ is hereditarily paracompact, the nonempty intersections of elements of $U$ are paracompact, and hence by theorem 7.8.6 of \citet[page~179]{Tayl02} their sheaf cohomology groups are isomorphic to their \v{C}ech cohomology groups, which because every open cover of $X$ and hence of $U$ can be refined to a good cover, by lemma \ref{l-cech-sing-coho} and note \ref{n-cech-sing-coho}, or for $\underline{\UU(1)}$ coefficients corollary \ref{co-u1-z-sing-diag}, are isomorphic to their singular cohomology groups, which vanish because the nonempty intersections are contractible.

The rest of the proof is shown for $\underline{\UU(1)}$ coefficients, but is the same for $\underline{\ZZ}$ coefficients.  By theorem 7.8.5 of \citet[pages~177,~159]{Tayl02}, fixing any $p \ge 0$, for each cover $U'$ there is a natural map from \v{C}ech to sheaf cohomology groups:
\begin{align}
\CH^p (U', \underline{\UU(1)}) &\xrightarrow{\sigma_{U'}} \HH^p (X, \underline{\UU(1)}). \notag
\end{align}
By the universal property of direct limits, the collection of these maps induces a map $\sigma$ as in \citet[page~179]{Tayl02} from the direct limit, which is the \v{C}ech cohomology of the space, to sheaf cohomology, which composed with the insertion homomorphisms gives the original maps:
\begin{align}
\CH^p (U', \underline{\UU(1)}) &\xrightarrow{h_{X,U'}} \CH^p (X, \underline{\UU(1)}) \xrightarrow{\sigma} \HH^p (X, \underline{\UU(1)}), \notag
\end{align}
with $\sigma_{U'} = \sigma \circ h_{X, U'}$.  Since $X$ is paracompact, $\sigma$ is an isomorphism.  Since the sheaf cohomology vanishes for all positive degrees of all nonempty intersections of elements of any of the covers $U'$ that are good, those $\sigma_{U'}$ are isomorphisms and so the corresponding insertion homomorphisms $h_{X, U'}$, such as $h_{X, U}$, are isomorphisms.

If $U$ is a good cover, the singular cohomology of nonempty intersections vanishes because they are contractible.

The map $\overline{\phi}$ over $\phi$ induces a morphism of direct systems of \v{C}ech cohomology groups of covers, and thus gives commutativity of the commutative diagram in the statement, with the horizontal arrows being the insertion homomorphisms.  The first part of the lemma shows that these are isomorphisms.
\end{proof}
For an example in which the vanishing cohomology condition is true for a cover that is not good, $S^1$ can be covered by two overlapping half-circles, and although the intersection has two connected components and is not contractible, its cohomology of positive degree vanishes.

\section{Bundle Gerbe Triviality and Properties}\label{s-dd-clas-bg-triv}

Now we return to the question of triviality of bundle gerbes, discussed in note \ref{n-bg-triv-fore} at the end of section \ref{s-bg-id-inv-isom}.
\begin{defn}\label{d-bg-delt}
\index{bundle gerbe!delta@$\delta$}
(The Coboundary Map for Principal $\UU(1)$ Bundle Cochains).
For a continuous function $g \colon Y^{[p-1]} \rightarrow \UU(1)$, define $\delta(g) \colon Y^{[p]} \rightarrow \UU(1)$ by letting $p(i) = 1 \text{ if } i \text { is even, } -1 \text{ if } i \text{ is odd}$, and defining:
\[
\delta(g) = \prod_{i = 1}^{p} (g^{*} \widehat{\pi}_{i})^{p(i - 1)} = \prod_{i = 1}^{p} (g \circ \widehat{\pi}_{i})^{p(i - 1)}, \notag
\]
and for a principal $\UU(1)$ bundle $P \rightarrow Y^{[p-1]}$, define a principal $\UU(1)$ bundle $\delta(P) \rightarrow Y^{[p]}$ by
\[
\delta(P) = \Tensor_{i = 1}^{p} (\widehat{\pi}_{i}^{*} P)^{p(i - 1)}. \notag
\]
\end{defn}
$\delta^2(g) = $ constant function $Y^{[p+1]} \rightarrow 1$, and $\delta^2(P) \cong Y^{[p+1]} \cross \UU(1)$.
% The sign convention accords with \citet{Stev00} def. 3.2 and with
% \citet{Murr07}, but for p = 2 at least, opposite to \citet[page~927]{MS99}

\begin{defn}\label{d-cano-triv-bg}
\index{bundle gerbe!canonical trivial}
\index{canonical trivial!bundle gerbe}
(Canonical Trivial Bundle Gerbes).
\citep[page~248]{Murr07} Suppose $X$ and $Y$ are topological spaces, $X$ paracompact, such that each open cover of $X$ has a refinement that is a good cover, and $Y \rightarrow X$ is a continuous surjection with local sections.  Given any principal $\UU(1)$ bundle $R \rightarrow Y$, the bundle gerbe $(\delta(R), Y, X)$ is defined to be a canonical trivial bundle gerbe.  The bundle gerbe multiplication $m_{\delta (R)} \colon \pi_{12}^{*} \delta (R) \tensor \pi_{23}^{*} \delta (R) \isomto \pi_{13}^{*} \delta (R)$ is defined as the composition of the following canonical isomorphisms, using canonical isomorphisms of principal $\UU(1)$ bundles such as that the dual of the pullback is the pullback of the dual (see lemma \ref{l-u1pb-cano-isom}, note \ref{n-u1pb-tens-dual}), \begin{align}
\pi_{12}^{*} \delta (R) \tensor \pi_{23}^{*} \delta (R) &\cong \pi_{12}^{*} (\pi_1^{*} R \tensor \pi_2^{*} R^{*}) \tensor \pi_{23}^{*} (\pi_1^{*} R \tensor \pi_2^{*} R^{*}) \notag \\
&\cong \pi_1^{*} R \tensor \pi_2^{*} R^{*} \tensor \pi_2^{*} R \tensor \pi_3^{*} R^{*} \notag \\
&\cong \pi_{13}^{*} (\pi_1^{*} R \tensor \pi_2^{*} R^{*}) \cong \pi_{13}^{*} \delta (R). \notag
\end{align}
\end{defn}

\begin{eg}\label{e-dist-cano-triv-bg}
\index{bundle gerbe!distinguished canonical trivial}
\index{distinguished canonical trivial!bundle gerbe}
(The Distinguished Canonical Trivial Bundle Gerbe Over a Space).
\citep[page~4]{Wald07} Given a paracompact topological space $X$ such that each open cover of $X$ has a refinement that is a good cover, $(m, \Delta_X \cross \UU(1), X, \ident, X)$, where $X^{[2]} = \Delta_X$, identifying $\UU(1) \tensor \UU(1)^{*} \cong \UU(1)$, and letting $m((x,x,x,z_{12}) \tensor (x,x,x,z_{23})) = (x,x,x,z_{12} z_{23})$, is the distinguished canonical trivial bundle gerbe over $X$, a special case of definition \ref{d-cano-triv-bg}.
\end{eg}

\begin{eg}\label{e-p-as-cocy}
\index{bundle gerbe!multiplication}
($P$ as a Cocycle).
Given a bundle gerbe $(P, Y)$, bundle gerbe multiplication and lemma \ref{l-u1pb-cano-isom} give a canonical isomorphism $\delta(P) = \pi_{23}^{*} P \tensor (\pi_{13}^{*} P)^{*} \tensor \pi_{12}^{*} P \cong \pi_{12}^{*} P \tensor \pi_{23}^{*} P \tensor (\pi_{13}^{*} P)^{*} \xrightarrow{m \tensor \ident} \pi_{13}^{*} P \tensor (\pi_{13}^{*} P)^{*}\cong Y^{[3]} \cross \UU(1)$.  One can think of $P$ as a cocycle; $(P, Y)$ canonically trivial if $P$ is a coboundary.
\end{eg}

\begin{defn}\label{d-bg-mor}
\index{bundle gerbe!morphism}
\index{bundle gerbe!isomorphism}
(Bundle Gerbe Morphisms and Isomorphisms).
Given two continuous bundle gerbes $(m_P, P, p, Y, \pi_Y, X)$, $(m_Q, Q, q, Z, \pi_Z, W)$, a bundle gerbe morphism $(P, Y, X) \rightarrow (Q, Z, W)$ is a triple of continuous maps $(\widehat{g}, \overline{g}, g)$, with $\widehat{g}$ a principal $\UU(1)$ bundle morphism, satisfying the following, giving the name $\overline{g}^{[k]}$ to the map $\overline{g}$ induces on $Y^{[k]}$:
\begin{align}
\pi_Z \circ \overline{g} &= g \circ \pi_Y \text{ (} \overline{g} \text{ over } g \text{)} \notag \\
q \circ \widehat{g} &= \overline{g}^{[2]} \circ p \text{ (} \widehat{g} \text{ over } \overline{g}^{[2]} \text{)} \notag \\
m_Q \circ (\pi_{12}^{*} \widehat{g} \tensor \pi_{23}^{*} \widehat{g}) &= (\pi_{13}^{*} \widehat{g}) \circ m_P \text{ (over } \overline{g}^{[3]} \text{)}. \notag
\end{align}
The morphism may be denoted $(\widehat{g}, \overline{g}, g) \colon (P, Y, X) \rightarrow (Q, Z, W)$.
\end{defn}
\begin{note}\label{n-bg-isom}
\index{bundle gerbe!isomorphism}
(Isomorphisms of Bundle Gerbes).
A morphism of bundle gerbes is an isomorphism if $g$ is a homeomorphism and $\widehat{g}, \overline{g}$ is an isomorphism of principal $\UU(1)$ bundles respecting bundle gerbe multiplication.
\end{note}

\begin{ass}\label{a-bndl-gerb}
\index{bundle gerbes!}
\index{assumptions!bundle gerbes}
(Bundle Gerbe Assumptions).
\citep[page~927]{MS99} As in assumption \ref{a-bndl}, we will assume that isomorphisms of bundle gerbes over the same base space (in the definition, $X = W$) are over the identity ($g = \ident_X$) unless otherwise noted.
\end{ass}

\begin{defn}\label{d-bg-triv}
\index{bundle gerbe!trivial}
\index{trivial!bundle gerbe}
\index{bundle gerbe!canonical trivial}
\index{canonical trivial!bundle gerbe}
(Bundle Gerbe Triviality).
A bundle gerbe $(P, Y, X)$ is trivial when there is a principal $\UU(1)$ bundle $R \rightarrow Y$ such that the canonical trivial bundle gerbe $(\delta(R), Y, X)$ is isomorphic to $(P, Y, X)$, using the identities $Y \rightarrow Y$ and $X \rightarrow X$, and the same projection $Y \rightarrow X$ for both bundle gerbes.  Such an isomorphism is called a trivialization of the bundle gerbe.
\end{defn}
\begin{note}\label{n-bg-triv-mult}
\index{bundle gerbe!triviality!isomorphism}
(The Bundle Gerbe Triviality Isomorphism).
The bundle gerbe isomorphism amounts to an isomorphism $\Phi \colon \delta(R) \rightarrow P$ (the isomorphism $\widehat{g}$ of definition \ref{d-bg-mor}) of the principal $\UU(1)$ bundles $\delta(R) \rightarrow Y^{[2]}$ and $P \rightarrow Y^{[2]}$, covering the identity, such that the following diagram commutes, where $m_P$ is the bundle gerbe multiplication of $(P, Y, X)$ and $m_{\delta (R)}$ is the bundle gerbe multiplication of $(\delta(R), Y, X)$:
\[
\begindc{\commdiag}[5]

\obj(10,20)[objPT]{$\pi_{12}^{*} P \tensor \pi_{23}^{*} P$}
\obj(10,10)[objDRT]{$\pi_{12}^{*} \delta (R) \tensor \pi_{23}^{*} \delta (R)$}
\obj(40,20)[objP]{$\pi_{13}^{*} P$}
\obj(40,10)[objDR]{$\pi_{13}^{*} \delta (R)$}
\mor{objPT}{objP}{$m_P$}
\mor{objDRT}{objPT}{$\pi_{12}^{*} \Phi \tensor \pi_{23}^{*} \Phi$}
\mor{objDR}{objP}{$\pi_{13}^{*} \Phi$}
\mor{objDRT}{objDR}{$m_{\delta (R)}$}

\enddc
\]
\end{note}

\begin{lem}\label{l-bg-two-triv}
\index{bundle gerbe!two trivializations}
\index{trivialization!bundle gerbe}
(Trivializations of a Bundle Gerbe Differ by Principal $\UU(1)$ Bundles).
\citep[page~249]{Murr07} Given a bundle gerbe $(P, Y, \pi, X)$ and two trivializations $R \rightarrow Y$, $\Phi \colon \delta(R) \rightarrow P$ and $R' \rightarrow Y$, $\Phi' \colon \delta(R') \rightarrow P$, there is a principal $\UU(1)$ bundle $Q \rightarrow X$, such that $R' \cong R \tensor \pi^* Q$.
\end{lem}

\begin{lem}\label{l-bg-ib-pb}
\index{bundle gerbe!induced}
\index{bundle gerbe!pullback}
\index{bundle gerbe!morphism}
(A Bundle Gerbe Morphism Factors Through the Pullback). \linebreak
Given $W$ a paracompact topological space such that each open cover of $W$ has a refinement that is a good cover, a continuous map $f \colon W \rightarrow X$, and a continuous bundle gerbe $(m_P, P, p, Y, \pi_Y, X)$, define \citep[page~245]{Murr07} the pullback of the bundle gerbe by $f$ to be the induced bundle gerbe as in the following commutative diagram:
\[
\begindc{\commdiag}[5]

\obj(40,30)[objpbP]{$(\overline{f}^{[2]})^{*} P = (\pi_2^{[2]})^{*} P$}
\obj(40,20)[objZ2]{$Z^{[2]} = f(^{*} Y)^{[2]}$}
\obj(50,10)[objZ]{$Z = f^{*} Y$}
\obj(50,0)[objW]{$W$}

\obj(70,30)[objP]{$P$}
\obj(70,20)[objY2]{$Y^{[2]}$}
\obj(80,10)[objY]{$Y$}
\obj(80,0)[objX]{$X$}

\mor{objpbP}{objZ2}{$\pi_1$}
\mor{objZ2}{objZ}{}
\mor{objZ}{objW}{$\pi_Z = \pi_1$}

\mor{objP}{objY2}{$p$}
\mor{objY2}{objY}{}
\mor{objY}{objX}{$\pi_Y$}

\mor{objpbP}{objP}{$\widehat{f} = \pi_2$}
\mor{objZ2}{objY2}{$\overline{f}^{[2]} = \pi_2^{[2]}$}
\mor{objZ}{objY}{$\overline{f} = \pi_2$}
\mor{objW}{objX}{$f$}

\enddc
\]
If $f$ is a homeomorphism, the bundle gerbe morphism is an isomorphism.

A morphism of continuous bundle gerbes
\[
(\widehat{f}, \overline{f}, f) \colon (Q, q, Z, \pi_Z, W) \rightarrow (P, p, Y, \pi_Y, X)
\]
factors through the pullback $f^{*} (P, Y, X)$ as in the following commutative diagram.  If $f$ is a homeomorphism, the right hand bundle gerbe morphism is an isomorphism.  If $\overline{f}$ is a homeomorphism, the left hand one is.
\[
\begindc{\commdiag}[5]

\obj(10,30)[objQ]{$Q$}
\obj(10,20)[objZ2]{$Z^{[2]}$}
\obj(20,10)[objZ]{$Z$}
\obj(20,0)[objW]{$W$}

\obj(40,30)[objpbP]{$(\pi_2^{[2]})^{*} P$}
\obj(40,20)[objpbY2]{$f(^{*} Y)^{[2]}$}
\obj(50,10)[objpbY]{$f^{*} (Y)$}
\obj(50,0)[objpbX]{$W$}

\obj(70,30)[objP]{$P$}
\obj(70,20)[objY2]{$Y^{[2]}$}
\obj(80,10)[objY]{$Y$}
\obj(80,0)[objX]{$X$}

\mor{objQ}{objZ2}{$q$}
\mor{objZ2}{objZ}{}
\mor{objZ}{objW}{$\pi_Z$}

\mor{objpbP}{objpbY2}{$\pi_1$}
\mor{objpbY2}{objpbY}{}
\mor{objpbY}{objpbX}{$\pi_1$}

\mor{objP}{objY2}{$p$}
\mor{objY2}{objY}{}
\mor{objY}{objX}{$\pi_Y$}

\mor{objQ}{objpbP}{$((\pi_Z \cross \overline{f})^{[2]} \circ q) \cross \widehat{f}$}
\mor{objZ2}{objpbY2}{$(\pi_Z \cross \overline{f})^{[2]}$}
\mor{objZ}{objpbY}{$\pi_Z \cross \overline{f}$}
\mor{objW}{objpbX}{$\ident$}

\mor{objpbP}{objP}{$\pi_2$}
\mor{objpbY2}{objY2}{$\pi_2^{[2]}$}
\mor{objpbY}{objY}{$\pi_2$}
\mor{objpbX}{objX}{$f$}

\enddc
\]
\end{lem}
\citep[page~926]{MS99} If the map $\widehat{f}$ is missing from what otherwise would be a bundle gerbe morphism, one can let $\widehat{f} = \pi_2 \colon (\overline{f}^{[2]})^{*} P \subset Z^{[2]} \cross P \rightarrow P$ to complete the data for a bundle gerbe morphism.

\begin{lem}\label{l-bg-dual-tens}
\index{bundle gerbe!dual}
\index{bundle gerge!tensor product}
(Bundle Gerbe Duals and Tensor Products).
% removed ~ to cure overfull box
\citep[page 245]{Murr07} Given a continuous bundle gerbe $(P, Y)$, its dual $(P, Y)^{*} = (P^{*}, Y)$, topological spaces the same but with inverted $\UU(1)$ action, is a continuous bundle gerbe.

Given two continuous bundle gerbes $(P, Y)$, $(Q, Z)$ over the same base space $X$, their tensor product $(P, Y) \tensor (Q, Z) = (P \tensor Q, Y \cross_X Z)$ is a continuous bundle gerbe, where $P \tensor Q$ is the quotient of $P \cross_X Q$ by the tensor product relation, resulting in $(P \tensor Q)_{((y_1, z_1),(y_2, z_2))} = P_{(y_1, y_2)} \tensor Q_{(z_1, z_2)}$.
\end{lem}

\section{Dixmier-Douady Class Definition and Independence of Choices}\label{s-bg-dd-def-inde-choi}

The Dixmier-Douady class definition involves some technicalities of bundle morphisms.  The definition in terms of torsors, ignoring the essential gluing of these into bundles, is easier; see the note afterwards.
\begin{defn}\label{d-bg-dd-clas}
\index{bundle gerbe!Dixmier-Douady Class}
(The Dixmier-Douady Class of a Bundle Gerbe).
\citep[pages~245--246]{Murr07} Given a bundle gerbe $(P, Y, X)$, choose an indexed cover $U = \{ U_i \}$ of $X$, $i \in I$, such that there are continuous local sections $s_i \colon U_i \rightarrow Y$, $\pi \circ s_i = \ident_{U_i}$, and for each $i, j$ with $U_i \cap U_j \ne \emptyset$, using the map $(s_i, s_j) \colon U_i \cap U_j \rightarrow Y^{[2]}$, it is possible to choose a continuous section $\sigma_{ij}$ of $(s_i, s_j)^{*} P$.  Choose these so that $\sigma_{ji} = \sigma_{ij}^{-1}$ as in definition \ref{d-bg-mult-id-inv}.

For nonempty $U_i \cap U_j \cap U_k$, pull back the bundle gerbe multiplication isomorphism $\pi_{12}^{*} P \tensor \pi_{23}^{*} P \xrightarrow{m} \pi_{13}^{*} P$ over $Y^{[3]}$, by $(s_i, s_j, s_k) \colon U_i \cap U_j \cap U_k \rightarrow Y^{[3]}$ and use canonical identifications of pullbacks such as $(s_i, s_j, s_k)^{*} \pi_{12}^{*} P \cong (\pi_{12} \circ (s_i, s_j, s_k))^{*} P = (s_i, s_j)^{*} P$ to obtain the isomorphism we will still call $m$, $(s_i, s_j)^{*} P \tensor (s_j, s_k)^{*} P \xrightarrow{m} (s_i, s_k)^{*} P$.  Since $m(\sigma_{ij} \tensor \sigma_{jk})$ and $\sigma_{ik}$ are both sections of $(s_i, s_k)^{*} P$, we define $g_{ijk} \colon U_i \cap U_j \cap U_k \rightarrow \UU(1)$ as the continuous unique function such that $m(\sigma_{ij} \tensor \sigma_{jk}) = g_{ijk} \sigma_{ik}$.

From the \v{C}ech cocycle $g_{ijk}$ define the Dixmier-Douady class $DD(P, Y, X) = [g_{ijk}] \in \CH^2(X;\underline{\UU(1)})$, , where the $\underline{\UU(1)}$ denotes the sheaf of continuous $\UU(1)$ valued functions on $X$.  More precisely, $DD(P, Y, X) = h_{XU} ([g_{ijk}])$, with $h_{XU}$ from definition \ref{d-cech-coho}, since $g_{ijk}$ is a cocycle of the cochain complex for $U$.
\[
\begindc{\commdiag}[5]
\obj(10,30)[objP]{$P$}
\obj(10,20)[objYY]{$Y^{[2]}$}
\obj(20,20)[objY]{$Y$}
\obj(20,10)[objX]{$U_i \cap U_j \cap U_k$}
\obj(40,10)[objU]{$\UU(1)$}
\mor{objP}{objYY}{$p$}
\mor(10,21)(20,21){$\pi_1$}
\mor(10,19)(20,19){$\pi_2$}
\mor{objY}{objX}{$\pi$}[\atright,\solidarrow]
\mor(21,10)(21,20){$s_i$}[\atleft,\solidarrow]
\mor(24,10)(24,20){$s_j$}[\atleft,\solidarrow]
\mor(27,10)(27,20){$s_k$}[\atleft,\solidarrow]
\mor(19,10)(13,20){}
\mor(16,10)(10,20){}
\mor(4,10)(4,30){$\sigma_{ik}$}
\mor(0,10)(0,30){$\sigma_{jk}$}
\mor(-4,10)(-4,30){$\sigma_{ij}$}
\mor(13,10)(7,20){$\substack{(s_i, s_j)\\
                             (s_j, s_k)\\
                             (s_i, s_k)}$}
\mor{objX}{objU}{$g_{ijk}$}
\enddc
\]
\end{defn}
\begin{note}\label{n-bg-dd-clas}
\index{bundle gerbe!Dixmier-Douady Class}
\index{good cover}
\index{alternating!cochains}
(The Dixmier-Douady Cocyle is Alternating; Good Covers).
Note that because of the requirement $\sigma_{ji} = \sigma_{ij}^{-1}$, $g_{ijk}$ is alternating with respect to permutations of its indices.

Note that the definition doesn't require that $U$ be a good cover as in the references.  If $U$ is a good cover, though, by lemma \ref{l-cech-coho-cove} $h_{XU}$ is an isomorphism, so it may not be necessary to consider it.
\end{note}

\begin{note}\label{n-bg-dd-mult}
\index{bundle gerbe!Dixmier-Douady Class}
\index{bundle gerbe!multiplication}
\index{identity!element}
\index{inverse!element}
\index{identity!section}
\index{inverse!section}
\index{iota@$\iota$}
(The Dixmier-Douady Cocycle Formulated using Bundle Gerbe Multiplicative Identity and Inverse Sections).
As an exercise preparatory to more complicated calculations shortly, the Dixmier-Douady cocycle definition in \ref{d-bg-dd-clas} can be rearranged using $\sigma_{ik}^{-1}$, the continuous multiplicative inverse section (of $(s_k, s_i)^{*} P$) of the chosen $\sigma_{ik}$, and $\iota_{i i}$, the continuous identity section of $(s_i, s_i)^{*} P$, as in lemma \ref{l-bg-mult-id-inv}.  Over $U_i \cap U_j \cap U_k$,
\[
m(m(\sigma_{ij} \tensor \sigma_{jk}) \tensor \sigma_{ik}^{-1}) = m(g_{ijk} \sigma_{ik} \tensor \sigma_{ik}^{-1}) = g_{ijk} \iota_{ii}, \notag
\]
where as before, the isomorphisms named $m$ involve pullbacks of bundle gerbe multiplication and various canonical isomorphisms for pullbacks that, as stated in note \ref{n-u1pb-tens-dual}, we treat as identifications.

To check that the various pullbacks of $m$ that are all given the same name here are in fact defined, it's convenient to use $s = (s_i, s_j, s_k, s_i) \colon U_i \cap U_j \cap U_k \rightarrow Y^{[4]}$ to pull back the composition of the top and right arrows of the associativity diagram in definition \ref{d-cont-bndl-gerb} (into which we will substitute sections), since various canonical isomorphism identifications have already been made in it, saving time.

Start with the definition of bundle gerbe multiplication as a morphism of principal $\UU(1)$ bundles over $Y^{[3]}$; this is pulled back by $(s_i, s_j, s_k)^{*}$ to $U_i \cap U_j \cap U_k$ in the definition of the Dixmier-Douady cocycle $g_{ijk}$.  Since $(s_i, s_j, s_k) = \pi_{123} \circ (s_i, s_j, s_k, s_i)$, use the last map to pull back the definition of $g_{ijk}$ to $Y^{[4]}$, where it will still be written $m(\sigma_{ij} \tensor \sigma_{jk}) = g_{ijk} \sigma_{ik}$.  Substitute it and $\sigma_{ik}^{-1}$ into the top arrow and then go down the right of the associativity diagram, as pulled back by $(s_i, s_j, s_k)$, to obtain the equation above with $\iota_{ii}$ on the right.
\end{note}

\begin{lem}\label{l-bg-dd-clas}
\index{bundle gerbe!Dixmier-Douady Class}
\index{Dixmier-Douady Class}
\index{good cover}
(The Dixmier-Douady Class is Independent of Choices).
There is a cover such that the $s_i$ and $\sigma_{ij}$ of definition \ref{d-bg-dd-clas} exist.  From whatever cover, $g_{ijk}$ is continuous, and is a $2$-cocycle.  The Dixmier-Douady class does not depend on the choices made for the sections $\sigma_{ij}$, the sections $s_i$, or the cover.
\end{lem}
\begin{proof}
Using definition \ref{d-cont-bndl-gerb}, let $U = \{ U_i \}$ be an open cover with local sections $s_i$, refined to a good cover.  Then the $\sigma_{ij}$ exist since the nonempty intersections $U_i \cap U_j$ are contractible.  $g_{ijk}$ is continuous by lemma \ref{l-pb-tran-func-tors}, composing the translation function with sections, and is a $2$-cocycle because of the bundle gerbe multiplication associativity diagram pulled back by $(s_i, s_j, s_k, s_l)$, and $m$'s $\UU(1)$-equivariance:
\begin{align}
m(m(\sigma_{ij} \tensor \sigma_{jk}) \tensor \sigma_{kl}) &= m(\sigma_{ij} \tensor m(\sigma_{jk} \tensor \sigma_{kl})) \Leftrightarrow \notag \\
m((g_{ijk} \sigma_{ik}) \tensor \sigma_{kl}) &=  m(\sigma_{ij} \tensor (g_{jkl} \sigma_{jl})) \Leftrightarrow \notag \\
g_{ikl} g_{ijk} \sigma_{il} &=  g_{ijl} g_{jkl} \sigma_{il} \Leftrightarrow \notag \\
g_{jkl} g_{ikl}^{-1} g_{ijl} g_{ijk}^{-1} &= \underline{1}, \text{ the constant function}. \notag
\end{align}

The definition does not depend on the choice of sections $\sigma_{ij}$, as follows.  Suppose given another choice $\sigma_{i j}'$.  Since $(s_i, s_j)^{*}P$, $(s_j, s_k)^{*}P$, $(s_i, s_k)^{*}P$ are continuous principal $\UU(1)$ bundles, any two sections of any one of them differ by a factor of a $\UU(1)$ valued function; $\sigma_{ij}' = h_{ij} \sigma_{ij}$, $\sigma_{jk}' = h_{jk} \sigma_{jk}$, $\sigma_{ik}' = h_{ik} \sigma_{ik}$.  Comparing
\begin{align}
m(\sigma_{ij} \tensor \sigma_{jk}) &= g_{ijk} \sigma_{ik} \text{ with} \notag \\
m(\sigma_{ij}' \tensor \sigma_{jk}') &= g_{ijk}' \sigma_{ik}' \text{ or} \notag \\
m((h_{ij} \sigma_{ij}) \tensor (h_{jk} \sigma_{jk})) &= g_{ijk}' h_{ik} \sigma_{ik}, \text{ we get} \notag \\
h_{jk} h_{ik}^{-1} h_{ij} g_{ijk} &= g_{ijk}', \notag
\end{align}
so that $g_{ijk}'$ and $g_{ijk}$ differ by a $1$-coboundary, and hence $[g_{ijk}'] = [g_{ijk}]$ in $\CH^2 (U; \underline{\UU(1)})$ and hence in their images under $h_{XU}$ are equal in $\CH^2 (X; \underline{\UU(1)})$.

Now, suppose $s_i, s_{i'}$ are two sections of $Y$ over the same open set $U_i = U_{i'}$, noticing the notation is $s_{i'}$ instead of $s_i'$, and similarly for $j, k$.  To show that $[g_{i'j'k'}] = [g_{ijk}]$, we need to show that there are some functions $h_{ij}, h_{jk}, h_{ik}$ on $U_i \cap U_j$, $U_j \cap U_k$, $U_i \cap U_k$ respectively, such that on $U_i \cap U_j \cap U_k$, $h_{jk} (h_{ik})^{-1} h_{ij} g_{i'j'k'} = g_{ijk}$.  \citet[pages~17--18]{John02} gives an idea how.  Choosing sections $\sigma$ as needed with $\sigma_{j'j} = \sigma_{jj'}^{-1}$, using bundle gerbe multiplication isomorphisms of the following mixed pullback bundles over $U_i \cap U_j$ and analogous ones over $U_j \cap U_k$, $U_i \cap U_k$, we get $h_{ij} = (g_{i' i j} g_{i' j j'})^{-1} \colon U_i \cap U_j \rightarrow \UU(1)$, $h_{jk} = (g_{j' j k} g_{j' k k'})^{-1} \colon U_j \cap U_k \rightarrow \UU(1)$, $h_{ik} = (g_{i' i k} g_{i' k k'})^{-1} \colon U_i \cap U_k \rightarrow \UU(1)$ satisfying:
\begin{align}
(s_{i'}, s_i)^{*} P \tensor (s_i, s_j)^{*} P &\xrightarrow{m} (s_{i'}, s_j)^{*} P \notag \\
(s_{i'}, s_j)^{*} P \tensor (s_j, s_{j'})^{*} P &\xrightarrow{m} (s_{i'}, s_{j'})^{*} P \notag \\
((s_{i'}, s_i)^{*} P \tensor (s_i, s_j)^{*} P) \tensor (s_j, s_{j'})^{*} P &\xrightarrow{m,m} (s_{i'}, s_{j'})^{*} P \notag \\
(h_{ij}) m(m(\sigma_{i'i} \tensor \sigma_{ij}) \tensor \sigma_{jj'}) &= \sigma_{i'j'} \notag \\
(h_{jk}) m(m(\sigma_{j'j} \tensor \sigma_{jk}) \tensor \sigma_{kk'}) &= \sigma_{j'k'} \notag \\
(h_{ik}) m(m(\sigma_{i'i} \tensor \sigma_{ik}) \tensor \sigma_{kk'}) &= \sigma_{i'k'}. \notag
\end{align}
Substituting in the expression for $g_{i'j'k'}$ from note \ref{n-bg-dd-mult}, using the result twice for the inverse section of a product from lemma \ref{l-bg-mult-id-inv}, temporarily using juxtaposition for bundle gerbe multiplication and omitting parentheses since it is associative, and using properties of inverse and identity sections,
\begin{align}
g_{i'j'k'} \iota_{i'i'} = &m(m(\sigma_{i'j'} \tensor \sigma_{j'k'}) \tensor \sigma_{i'k'}^{-1}) \notag \\
= &m(m((h_{ij}) m(m(\sigma_{i'i} \tensor \sigma_{ij}) \tensor \sigma_{jj'}) \notag \\
&\tensor (h_{jk}) m(m(\sigma_{j'j} \tensor \sigma_{jk}) \tensor \sigma_{kk'})) \notag \\
&\tensor ((h_{ik}) m(m(\sigma_{i'i} \tensor \sigma_{ik}) \tensor \sigma_{kk'}))^{-1}) \notag
\end{align}
\begin{align}
g_{i'j'k'} \iota_{i'i'} = &h_{jk} h_{ik}^{-1} h_{ij} \sigma_{i'i} \sigma_{ij} \sigma_{jj'} \sigma_{j'j} \sigma_{jk} \sigma_{kk'} \sigma_{kk'}^{-1} \sigma_{ik}^{-1} \sigma_{i'i}^{-1} \notag \\
= &h_{jk} h_{ik}^{-1} h_{ij} \sigma_{i'i} \sigma_{ij} \sigma_{jk} \sigma_{ik}^{-1} \sigma_{i'i}^{-1} \notag \\
= &h_{jk} h_{ik}^{-1} h_{ij} g_{ijk} \sigma_{i'i} \sigma_{i'i}^{-1} \notag \\
= &h_{jk} h_{ik}^{-1} h_{ij} g_{ijk} \iota_{i'i'}. \notag
\end{align}
The expanded expression has $8$ bundle gerbe multiplications, and is implicitly over $Y^{[10]}$, pulled back by $(s_{i'}, s_i, s_j, s_{j'}, s_j, s_k, s_{k'}, s_{k}, s_i, s_{i'})$.

Given that the Dixmier-Douady class definition doesn't depend on the choice of sections $s_i$, it doesn't depend on the choice of indexed cover either, similarly to the proof of independence of cover of the first Chern class definition, in \citet[page~66]{Bryl93}.  Given two covers $U$, $V$, we can consider a common refinement $W$.  By restricting the sections for the original cover to the open sets and nonempty intersections in the refinement, as underlies the refinement map of cohomology (see definition \ref{d-cech-coho}), we obtain from the Dixmier-Douady cocycles $g_{ijk,U}$, $g_{ijk,V}$ for the original covers, cocycles $g_{ijk,WU}$, $g_{ijk,WV}$ for the refinement.  We have shown already that up to cohomology the choice of sections for a given cover doesn't matter; $[g_{ijk,WU}] = [g_{ijk,WV}] \in \CH^2(W; \underline{\UU(1)})$.

The Dixmier-Douady class defined by $g_{ijk,U}$ is
\begin{align}
h_{XU} ([g_ijk,U]) &= h_{XW} \circ h_{WU} ([g_ijk,U]) \notag \\
 &= h_{XW} ([g_{ijk,WU}]) = h_{XW} ([g_{ijk,WV}]) \notag \\
 &= h_{XW} \circ h_{WV} ([g_{ijk,V}]) = h_{XV} ([g_{ijk,V}]), \notag
\end{align}
which is the Dixmier-Douady class defined by $g_{ijk,V}$.  Thus the Dixmier-Douady class is well-defined, independent of the cover used to calculate it.
\end{proof}

\section{Dixmier-Douady Class Vanishing Equivalent to Triviality}\label{s-bg-dd-vani-equi-triv}

The Dixmier-Douady class is important in this thesis because its triviality is a necessary and sufficient condition for triviality of the bundle gerbe.  First we define a set of local trivializations of a bundle gerbe and show they exist.
\begin{defn}\label{d-bg-loc-triv}
\index{bundle gerbe!local trivializations}
(Bundle Gerbe Local Trivializations).
A set of local trivializations of a bundle gerbe $(m_P, P, p, Y, \pi, X)$ consist of an indexed open cover $\{ U_i \}$ of $X$, and over each open set $Y_i = \pi^{-1} (U_i) \xrightarrow{\incl} Y$, a principal $\UU(1)$ bundle $R_i$ and a bundle gerbe isomorphism $\Phi_i \colon \delta (R_i) \cong P_i = P_{|Y_i^{[2]}}$ over the identity on $Y_i$ over the identity on $U_i$; i.e., the bundle gerbe morphism as in definition \ref{d-bg-mor}, $(\Phi_i, \ident, \ident)$ making the canonical trivial bundle gerbe as in definition \ref{d-cano-triv-bg}, $(\delta(R_i), Y_i, U_i)$ isomorphic to $(P_i, Y_i, U_i)$.
\end{defn}

\begin{lem}\label{l-bg-loc-triv}
\index{bundle gerbe!local trivializations}
(Bundle Gerbe Local Trivializations Exist).
Every bundle gerbe $(P, Y, X)$ has a set of local trivializations.  A set of local trivializations can be constructed from some of the choices made in definition \ref{d-bg-dd-clas} of the Dixmier-Douady class:  an indexed cover $\{ U_i \}$ of $X$, $i \in I$, and continuous local sections $s_i \colon U_i \rightarrow Y$ (it's not necessary to choose $\sigma_{ij}$).
\end{lem}
\begin{proof}
Define open sets $Y_i \xrightarrow{incl} Y$ as in definition \ref{d-bg-loc-triv}.  To define $R_i$ and $\Phi_i$, we argue first fiberwise with torsors for one kind of understanding, and then with bundles for precise definition to show continuity of isomorphisms across fibers.

Let $y \in Y_i$ and $(y_1, y_2) \in Y^{[2]}$; $\pi(y_1) = \pi(y_2) \Rightarrow (y_1 \in Y_i \Leftrightarrow y_2 \in Y_i)$.  Let $Y^{[2]}_i = \{(y_1, y_2) \in Y^{[2]} \st y_1, y_2 \in Y_i \} = \pi_1^{-1} (Y_i) = \pi_2^{-1} (Y_i)$, an open set in $Y^{[2]}$; $Y_i^{[3]}$ similarly.  Using isomorphisms of lemma \ref{l-bg-mult-cano-isom} and bundle gerbe multiplication, define:
\begin{align}
(R_i)_y &= P_{y, s_i (\pi (y))}, \text{ for } y \in Y_i \notag \\
(\delta (R))_{y_1, y_2} &= (\pi_1^{*} R \tensor (\pi_2^{*} R)^{*})_{y_1, y_2}, \text{ for } (y_1, y_2) \in Y_i^{[2]} \notag \\
 &= (R_i)_{y_1} \tensor ((R_i)_{y_2})^{*} = P_{y_1, s_i(\pi(y_1))} \tensor (P_{y_2, s_i(\pi(y_2))})^{*} \notag \\
 &\cong P_{y_1, s_i(\pi(y_1))} \tensor P_{s_i(\pi(y_2)),y_2} \cong P_{y_1, y_2}; \text{ or} \notag \\
R_i &= (\incl, s_i \circ \pi \circ \incl)^{*} P. \notag
\end{align}
Then
\begin{align}
\pi_{13}^{*} P &\cong \pi_{12}^{*} P \tensor \pi_{23}^{*} P \Rightarrow \notag \\
(\pi_1, s_i \circ \pi \circ \pi_1, \pi_2)^{*} \pi_{13}^{*} P &\cong (\pi_1, s_i \circ \pi \circ \pi_1, \pi_2)^{*} (\pi_{12}^{*} P \tensor \pi_{23}^{*} P ) \Rightarrow \notag \\
P_{|Y^{[2]}_i} \cong ((\pi_1, \pi_2)^{*} P)_{|Y^{[2]}_i}  &\cong ((\pi_1, s_i \circ \pi \circ \pi_1)^{*}P)_{|Y^{[2]}_i} \tensor ((s_i \circ \pi \circ \pi_1, \pi_2)^{*} P)_{|Y^{[2]}_i} \notag \\
 &= ((\pi_1, s_i \circ \pi \circ \pi_1)^{*}P)_{|Y^{[2]}_i} \tensor ((s_i \circ \pi \circ \pi_2, \pi_2)^{*} P)_{|Y^{[2]}_i} \notag \\
 &\cong ((\pi_1, s_i \circ \pi \circ \pi_1)^{*}P)_{|Y^{[2]}_i} \tensor ((\pi_2, s_i \circ \pi \circ \pi_2)^{*} \tau^{*} P)_{|Y^{[2]}_i} \notag \\
 &\cong ((\pi_1, s_i \circ \pi \circ \pi_1)^{*}P)_{|Y^{[2]}_i} \tensor ((\pi_2, s_i \circ \pi \circ \pi_2)^{*} P^{*})_{|Y^{[2]}_i} \notag \\
 &\cong \pi_1^{*} (\incl, s_i \circ \pi \circ \incl)^{*} P \tensor (\pi_2^{*} (\incl, s_i \circ \pi \circ \incl)^{*} P)^{*} \notag \\
 &= \delta (R_i). \notag
\end{align}
$\Phi_i$ is given by the chain of canonical bundle isomorphisms from bottom up, which agrees on fibers with the top down fiberwise definition.  To see that it preserves bundle gerbe multiplication, as in the diagram of note \ref{n-bg-triv-mult}, it suffices to argue about elements of fibers as follows.  Fix $(y_1, y_2, y_3) \in Y_i^{[3]}$.  For arbitrary $r_{12} \in (\delta(R_i))_{y_1, y_2}$,  $r_{23} \in (\delta(R_i))_{y_2, y_3}$, choose $p_2 \in P_{y_2, s_i(\pi(y_2))}$, and let $p_1 \in P_{y_1, s_i(\pi(y_1))}$, $p_3 \in P_{y_3, s_i(\pi(y_3))}$ be such that $r_{12} = p_1 \tensor p_2^{*}$, $r_{23} = p_2 \tensor p_3^{*}$.  Following the definition of $\Phi_i$ and using corollary \ref{co-bg-mult-cano-isom-elem} for the isomorphism involving $\tau$, we see that $\Phi_i (r_{12}) = \Phi_i (p_1 \tensor p_2^{*}) = m_P (p_1 \tensor p_2^{-1})$ and likewise $\Phi_i (r_{23}) = \Phi_i (p_2 \tensor p_3^{*}) = m_P (p_2 \tensor p_3^{-1})$.

Thus, omitting some notation for elements of pullbacks, $m_P ((\pi_{12}^{*} \Phi_i) (r_{12}) \tensor (\pi_{23}^{*} \Phi_i) (r_{23})) = m_P (m_P (p_1 \tensor p_2^{-1}) \tensor m_P (p_2 \tensor p_3^{-1})) = m_P (p_1 \tensor p_3^{-1})$, using associativity and the properties of the multiplicative inverse and identity.

On the other hand, by definition \ref{d-cano-triv-bg}, $m_{\delta (R_i)} (r_{12} \tensor r_{23}) = m_{\delta (R_i)} ((p_1 \tensor p_2^{*}) \tensor (p_2 \tensor p_3^{*})) = p_1 \tensor p_3^{*}$, which $\Phi_i$ maps to $m_P (p_1 \tensor p_3^{-1})$.  Thus $\Phi_i$ preserves bundle gerbe multiplication, and we have a bundle gerbe isomorphism.
\end{proof}

\begin{lem}\label{l-cano-triv-bg-dd-zero}
\index{bundle gerbe!canonical trivial}
\index{canonical trivial!bundle gerbe}
\index{Dixmier-Douady class!zero}
The Canonical Trivial Bundle Gerbe's Dixmier-Douady Class Vanishes.
\end{lem}
\begin{proof}
Use the notation of definition \ref{d-cano-triv-bg} and make choices of good cover $\{ U_i \}$ and local sections $s_i \colon U_i \rightarrow Y$ as for the Dixmier-Douady class definition \ref{d-bg-dd-clas}.  Choose sections $\eta_i \colon U_i \rightarrow s_i^{*} R$.  Then for $U_i \cap U_j \ne \emptyset$ the $\sigma_{ij} = \eta_i \tensor \eta_j^{*}$ are sections of $(s_i, s_j)^{*} \delta(R)$ over $U_i \cap U_j$ as used to define the Dixmier-Douady class $DD(\delta(R))$.  Using the definition of the bundle gerbe multiplication $m_{\delta(R)}$, $g_{ijk} \sigma_{ik} = m_{\delta(R)} (\sigma_{ij} \tensor \sigma_{jk}) = \sigma_{ik}$, so that $g_{ijk} = \underline{1}$, the constant function, and $[g_{ijk}] = 0$.

The bundle gerbe multiplication uses canonical isomorphisms for pullbacks and associativity, from lemma \ref{l-u1pb-cano-isom} for the tensor product of a bundle and its dual that maps $\eta_j^{*} \tensor \eta_j \mapsto \underline{1}$, and for the tensor product of a bundle and the product bundle that maps $\eta_i \tensor \underline{1} \mapsto \eta_i$.  The actions on elements are from lemma \ref{l-u1t-cano-isom}.
\end{proof}

\begin{prop}\label{p-bg-triv-dd-zero}
\index{Dixmier-Douady class!zero}
\index{bundle gerbe!trivial if Dixmier-Douady class zero}
\index{bundle gerbe!trivialization}
\index{good cover}
A Bundle Gerbe is Trivial $\Leftrightarrow$ its Dixmier-Douady Class Vanishes.
\end{prop}
\begin{proof}
``$\Rightarrow$''.  \citep[page~248]{Murr07} From lemma \ref{l-cano-triv-bg-dd-zero}, using the notation of definition \ref{d-bg-triv}, $DD(\delta(R)) = 0$.  Since our bundle gerbe $(P, Y, X)$ is isomorphic to $(\delta(R), Y, X)$ using the identities on $Y$ and $X$, by lemma \ref{l-bg-dd-prop} (which does not depend on the current result) the Dixmier-Douady classes of the bundle gerbes are equal, and $DD(P) = 0$.

``$\Leftarrow$''.  \citep[pages~248--249]{Murr07} Construct a set of local trivializations as in lemma \ref{l-bg-loc-triv}; use its notation and that of its proof.  Use its choices of indexed cover $\{ U_i \}$, refined to a good cover, and local sections $s_i \colon U_i \rightarrow Y$.  Choose sections $\sigma_{ij}$ of $(s_i, s_j)^{*} P$ to define the Dixmier-Douady cocycle $g_{ijk}$ by $m_P (\sigma_{ij} \tensor \sigma_{jk}) = g_{ijk} \sigma_{ik}$ over $U_i \cap U_j \cap U_k$.  Suppose $[g_{ijk}]= 0$; then $g_{ijk} = h_{jk} h_{ik}^{-1} h_{ij}$.  Replace $\sigma_{ij}$ with $\sigma_{ij} h_{ij}^{-1}$ to get $g_{ijk} = \underline{1}$.  The fact that the cover is good both allows the choice of the $\sigma_{ij}$ and ensures that the vanishing of the Dixmier-Douady class implies the vanishing of the class of the cocycle $g_{ijk}$ in the cohomology of the cover.

To see that the $R_i = (\incl, s_i \circ \pi \circ \incl)^{*} P$ fit together to form a principal $\UU(1)$ bundle $R \rightarrow Y$, we first argue fiberwise, then with bundle maps.  For $y \in Y_i \cap Y_j$, construct an isomorphism $\chi_{ij,y} \colon (R_i)_y \rightarrow (R_j)_y$ by noting that bundle gerbe multiplication for $P$ gives the isomorphism of $\UU(1)$ torsors
\[
P_{y, s_i(\pi(y))} \tensor P_{s_i(\pi(y)), s_j(\pi(y))} \xrightarrow{m_P} P_{y, s_j(\pi(y))}, \notag
\]
that $\sigma_{ij}(\pi(y)) \in P_{s_i(\pi(y)), s_j(\pi(y))}$, that bundle gerbe multiplication for $P$ gives $m_P(\sigma_{ij} \tensor \sigma_{jk}) = g_{ijk} \sigma_{ik}$, and for $q \in (R_i)_y$ defining
\begin{align}
\chi_{ij,y}(q) &= m_P(q \tensor \sigma_{ij} (\pi(\pi_{R_i}(q)))), \notag
\end{align}
\begin{align}
\chi_{jk,y} \circ \chi_{ij,y} (q) &= m_P(m_P(q \tensor \sigma_{ij} (\pi(\pi_{R_i}(q)))) \notag \\
&\tensor \sigma_{jk} (\pi(\pi_{R_j}(m_P(q \tensor \sigma_{ij} (\pi(\pi_{R_i}(q)))))))) \notag \\
 &= m_P(m_P(q \tensor \sigma_{ij} (\pi(\pi_{R_i}(q)))) \tensor \sigma_{jk} (\pi(\pi_{R_i}(q)))) \notag \\
 &= m_P(q \tensor m_P(\sigma_{ij} (\pi(\pi_{R_i}(q))) \tensor \sigma_{jk} (\pi(\pi_{R_i}(q))))) \notag \\
 &= g_{ijk} m_P(q \tensor \sigma_{ik} (\pi(\pi_{R_i}(q)))) = m_P(q \tensor \sigma_{ik} (\pi(\pi_{R_i}(q)))) = \chi_{ik,y} (q), \notag
\end{align}
and hence $\chi_{jk,y} \circ \chi_{ij,y} = \chi_{ik,y}$, the necessary cocycle condition to fit together the $R_i$ using the $\chi_{ij}$, given fiberwise.  It is only in the last line of this calculation that the hypothesis $g_{ijk} = \underline{1}$ is used.

The $Y_i$ form an open cover of $Y$, and so the $Y_i^{[2]}$ form one of $Y^{[2]}$.  Thus, if we knew that the $\chi_{ij,y}$ joined together to form isomorphisms of principal bundles, by lemma \ref{l-pb-cons}, we would have a principal bundle $R$ agreeing with each $R_i$ over $Y_i$.  To see this, we work with principal $\UU(1)$ bundles instead of fibers.  Define $\chi_{ij}$ as follows, consistent with the preceding fiberwise calculations.  Starting from bundle gerbe multiplication for $P$, $\pi_{12}^{*} P \tensor \pi_{23}^{*} P \xrightarrow{m_P} \pi_{13}^{*} P$, we pull back to get the isomorphism of principal $\UU(1)$ bundles
\begin{align}
  &(\incl, s_i \circ \pi \circ \incl)^{*} P \tensor (s_i \circ \pi \circ \incl, s_j \circ \pi \circ \incl)^{*} P \notag \\
= &(\incl, s_i \circ \pi \circ \incl, s_j \circ \pi \circ \incl)^{*} \pi_{12}^{*} P \tensor (\incl, s_i \circ \pi \circ \incl, s_j \circ \pi \circ \incl)^{*} \pi_{23}^{*} P \notag \\
\xrightarrow{m_P} &(\incl, s_i \circ \pi \circ \incl, s_j \circ \pi \circ \incl)^{*} \pi_{13}^{*} P \notag \\
= &(\incl, s_j \circ \pi \circ \incl)^{*} P, \notag
\end{align}
then for $q \in R_i$ such that $\pi_{R_i} (q) \in Y_i \cap Y_j$, define
\[
\chi_{ij} (q) = m_P (q \tensor \sigma_{ij} (\pi (\pi_{R_i} (q)))). \notag
 \]
$\chi_{ij}$ is $\UU(1)$-equivariant and covers the identity.  Now, $m_P$ is a continuous function, and the second factor of the tensor product is a continuous function of $q$.  Thus $\chi_{ij}$ will be continuous if we show that the tensor product of elements in two principal $\UU(1)$ bundles is a continuous function of the elements; but locally we may refer this question to the same one for the product $\UU(1)$ bundle, for which it is true.  Thus $\chi_{ij}$ is a morphism, and hence an isomorphism of principal $\UU(1)$ bundles.  Since for every $y \in Y_i$, $\chi_{ij,y} = (\chi_{ij})_{|(R_i)_y}$ and likewise for $jk, ik$, and we know that $\chi_{jk,y} \circ \chi_{ij,y} = \chi_{ik,y}$, therefore $\chi_{jk} \circ \chi_{ij} = \chi_{ik}$; no need to translate that proof into bundle isomorphisms.  Thus we have a principal $\UU(1)$ bundle $R \rightarrow Y$.

As in lemma \ref{l-pb-cons}, if we can show that $\Phi_j \circ \delta(\chi_{ij}) = \Phi_i$, where $\delta(\chi_{ij}) \colon \delta(R_i) \rightarrow \delta(R_j)$ is induced by $\chi_{ij}$, the local $\Phi_i$ will fit together to form a global $\Phi$, which since the $\Phi_i$ are over the identity on $Y^{[2]}$ and respect bundle gerbe multiplication, $\Phi$ also will respect it; we will have the bundle gerbe isomorphism.

To see that $\Phi_j \circ \delta(\chi_{ij}) = \Phi_i$, fix $(y_1, y_2) \in Y_i^{[2]}$.  For arbitrary $r_{12} \in (\delta(R_i))_{y_1, y_2}$, choose $p_2 \in P_{y_2, s_i(\pi(y_2))}$, and let $p_1 \in P_{y_1, s_i(\pi(y_1))}$ be such that $r_{12} = p_1 \tensor p_2^{*}$.  Then $\Phi_i (r_{12}) = \Phi_i (p_1 \tensor p_2^{*}) = m_P (p_1 \tensor p_2^{-1})$.  Substituting $p_1, p_2$ in turn for $q$ above,
\begin{align}
\Phi_j (\delta(\chi_{ij}) (r_{12})) &= \Phi_j (m_P(p_1 \tensor \sigma_{ij} (\pi (\pi_{R_i} (p_1)))) \tensor m_P(p_2 \tensor \sigma_{ij} (\pi (\pi_{R_i} (p_2))))^{*}) \notag \\
&= m_P (m_P(p_1 \tensor \sigma_{ij} (\pi (\pi_{R_i} (p_1)))) \tensor m_P(p_2 \tensor \sigma_{ij} (\pi (\pi_{R_i} (p_2))))^{-1}) \notag \\
&= m_P (m_P(p_1 \tensor \sigma_{ij} (\pi (\pi_{R_i} (p_1)))) \tensor m_P(\sigma_{ij} (\pi (\pi_{R_i} (p_2)))^{-1} \tensor p_2^{-1})) \notag \\
&= m_P(p_1 \tensor p_2^{-1}) = \Phi_i(r_{12}), \notag
\end{align}
since $\pi(\pi_{R_i}(p_1)) = \pi(\pi_{R_i}(p_2))$, using associativity and properties of inverse and identity sections of bundle gerbe multiplication as in lemma \ref{l-bg-mult-id-inv}.
\end{proof}

\section{Dixmier-Douady Class Properties; Stable Isomorphisms}\label{s-bg-dd-clas-prop-stab-isom}

\begin{lem}\label{l-bg-dd-prop}
\index{bundle gerbe!Dixmier-Douady Class}
\index{Dixmier-Douady Class}
(Dixmier-Douady Class Properties).
\citep[page~246]{Murr07} \citep[page~927]{MS99} Given continuous bundle gerbes $(P, Y, X)$, $(Q, Z, X)$ over the same base space,
\begin{align}
DD((P, Y, X)^{*}) &= -DD(P, Y, X) \notag \\
DD((P, Y, X) \tensor (Q, Z, X)) &= DD(P, Y, X) + DD(Q, Z, W). \notag
\end{align}
Given $W$ a paracompact topological space such that each open cover of $W$ has a refinement that is a good cover, a continuous map $f \colon W \rightarrow X$, and a continuous bundle gerbe $(m_P, P, p, Y, \pi_Y, X)$ to be pulled back by $f$; or respectively a morphism of continuous bundle gerbes $(\widehat{f}, \overline{f}, f) \colon (Q, Z, W) \rightarrow (P, Y, X)$:
\begin{align}
DD(f^{*}(P, Y, X)) &= f^{*}(DD(P, Y, X)) \notag \\
DD(Q, Z, W) &= f^{*} (DD(P, Y, X)). \notag
\end{align}
\end{lem}
\begin{proof}
Refer to definitions \ref{d-bg-dd-clas} and \ref{l-bg-dual-tens}.  Regarding the dual, using definition \ref{d-u1pb-dual}, $m_P (\sigma_{ij} \tensor \sigma_{jk}) = g_{ijk} \sigma_{ik}$ translates to $m_{P^{*}} (\sigma_{ij}^{*} \tensor \sigma_{jk}^{*}) = (g_{ijk} \sigma_{ik})^{*} = g_{ijk}^{-1} \sigma_{ik}^{*}$.  The statement about the tensor product follows from tensoring the sections $\sigma_{ij, P}$, $\sigma_{ij, Q}$ of the two bundles; then $(g_{ijk,P} \sigma_{ik,P}) \tensor (g_{ijk,Q} \sigma_{ik,Q}) = g_{ijk,P} (g_{ijk,Q} \sigma_{ik,P} \tensor \sigma_{ik,Q})$.

To consider the case of the pullback via $f$, as in the first diagram of lemma \ref{l-bg-ib-pb}, start with the Dixmier-Douady cocycle for $(P, Y, X)$. Use $f$ to pull back the indexed cover $U = \{ U_i \}$ of $X$ to the cover $f^* U = \{ f^{-1} (U_i) \}$ of $W$.  Pull back the local sections $s_i$ of $Y \rightarrow X$ to local sections $t_i = (\ident \cross (s_i \circ f))$ to $f^{*}(Y) \rightarrow W$.  Similarly, sections $\sigma_{ij}$ of $(s_i, s_j)^{*} P$ pull back to sections $\tau_{ij}$ of $(t_i, t_j)^{*} (\pi_2^{[2]})^{*} P$.  The defining equation for the Dixmier-Douady cocycle for the pullback gives the same function, but pulled back to $f^{-1} (U_i \cap U_j \cap U_k)$, that it gives for the original bundle gerbe on $U_i \cap U_j \cap U_k$.  That is, $g_{ijk,pb} = g_{ijk,orig} \circ f$.  Although $f^* U$ may not be a good cover, since it gives the needed sections, it calculates the Dixmier-Douady class of the pullback.  The map used for the cocycle is the map used in the proof of corollary \ref{co-cech-coho-indu-map-cont-func} to define the induced map $f^{*} \colon \CH^{2} (X; \underline{\UU(1)}) \rightarrow \CH^{2} (W; \underline{\UU(1)})$.

The case of the bundle gerbe morphism factors through the pullback.  Referring to the left side of the last diagram of lemma \ref{l-bg-ib-pb}, we can start with an indexed cover, local sections, and sections for calculating the Dixmier-Douady class of $(Q, Z, W)$ and move everything forward to the pullback since the map on the base space is a homeomorphism (the identity).  Since $\widehat{f}$ is equivariant, the calculations give the same cocycle for the pullback as for $(Q, Z, W)$.
\end{proof}

\begin{defn}\label{d-stab-isom}
\index{bundle gerbe!stable isomorphism}
\index{stable isomorphism!bundle gerbe}
\index{isomorphism!stable}
(When Bundle Gerbes are Stably Isomorphic).
\citep[page~249]{Murr07} Two continuous bundle gerbes $(P, Y)$ and $(Q, Z)$ over the same base space are stably isomorphic when the bundle gerbe $(P, Y)^{*} \tensor (Q, Z)$ is trivial.
\end{defn}
\begin{note}\label{n-stab-isom}
\index{bundle gerbe!stable isomorphism}
(Properties of Bundle Gerbes being Stably Isomorphic).
By \citet[page~249]{Murr07}, two bundle gerbes over the same base space are stably isomorphic if and only if their Dixmier-Douady classes are equal.  In particular, isomorphic bundle gerbes are stably isomorphic.  Stable isomorphism is an equivalence relation.
\end{note}

\begin{lem}\label{l-indu-bg-stab-isom}
\index{bundle gerbe!stable isomorphism}
\index{stable isomorphism!bundle gerbe}
\index{isomorphism!stable}
\index{induced map}
(The Induced Bundle Gerbe from a Map of $Y$ Spaces is Stably Isomorphic to the Original Bundle Gerbe).
\citep[page~928]{MS99}  Given a bundle gerbe $(P, Y, \pi_Y, X)$, a topological space $Z$ with continuous surjection $\pi_Z \colon Z \rightarrow X$ that has local sections, and a continuous map $\psi \colon Z \rightarrow Y$ over the identity on $X$, i.e. $\pi_Y \circ \psi = \pi_Z$, the bundle gerbe $((\psi^{[2]})^* P, Z, X)$ is stably isomorphic to $(P, Y, X)$.
\end{lem}

\chapter{CONSTRUCTION OF THE BUNDLE GERBE}\label{c-cnst-bndl-gerb}

This chapter constructs the bundle gerbe that is at the heart of the thesis.  The bundle gerbe's $Y$ space is the polarization class bundle of definition \ref{d-pol-clas-bndl}, straightforwardly encoding, over every loop $\gamma \in LM$, all Lagrangian subspaces within a chosen polarization class of the Hilbert space built from the fiber $LE_{\gamma}$ above that loop.  The bundle gerbe's $P$ space encodes, over every $\gamma$, all Clifford-linear unitary isomorphisms between pairs of Fock spaces corresponding to pairs of those Lagrangian subspaces.

\section{The Standard Intertwiner Bundle}\label{s-inte-bndl}

Following the established pattern, we define the bundle gerbe's $P$ bundle as a bundle associated to $L\SO(E)$ with fiber the standard intertwiner bundle $T$ total space as follows.  Note that the symbol $T$ is used for other purposes when suitably adorned, and outside of this chapter and theorem \ref{t-bg-triv-clif-alg-mod} is used for operators or torsors also; the context should make things clear.
\begin{prop}\label{p-bg-t}
\index{bundle gerbe!construction!T@$T$}
\index{bundle!standard!intertwiner}
\index{intertwiner!standard bundle}
\index{T@$T$}
(The Standard Intertwiner Bundle Construction).
Suppose given $L \in \Lagr_{res}$.  There is a principal $\UU(1)$ bundle $t \colon T \rightarrow \Lagr_{res} \cross \Lagr_{res}$ with total space $T = \{ ((L_1, L_2), \phi) \st L_1 \text{, } L_2 \in \Lagr_{res} \text{, } \phi \in T(L_1, L_2) \}$, standard fiber $T(L, L)$, and left $L\SO(n)$ action for which $t$ is equivariant, given for $g \in L\SO(n)$ by $(g, (L_1, L_2, \phi)) \mapsto (g L_1, g L_2, \Lambda_g \circ \phi \circ \Lambda_g^{*})$.

$T$ depends on $L$ only for local trivializations.  Another choice $L' \in \Lagr_{res}$ results in local trivializations that are compatible with those for $L$, via $\UU(1)$-equivariant homeomorphisms.  In particular, the topology of $T$ doesn't depend on the choice of $L$.
\end{prop}
\begin{proof}
To define local trivializations of $T$, given $L_1, L_2 \in \Lagr_{res}$, use box neighborhoods of points $(L_1, L_2) \in V_{L_1} \cross V_{L_2} \subset \Lagr_{res} \cross \Lagr_{res}$.  The natural standard fiber is $T(L, L)$, so we will define
\begin{align}
\Psi_{(L_1, L_2)} \colon T_{V_{L_1} \cross V_{L_2}} &\isomto (V_{L_1} \cross V_{L_2}) \cross T(L, L) \label{eq-t-psi}
\end{align}
with $\pi_1 \circ \Psi_{(L_1, L_2)} = t$ and $\pi_2 \circ \Psi_{(L_1, L_2)}$ $\UU(1)$-equivariant when $\UU(1)$ acts on the intertwiner components $\phi$ of the fibers of $P$ and on the intertwiners $T(L, L)$.  Recall from proposition \ref{p-set-intw-u1-tors} that $T(L, L) \cong \UU(1)$, with the $\UU(1)$ torsor topology equal to the operator norm and strong operator topology.  Provided the transition functions (those of definition \ref{d-fb-vb}) are continuous, the local trivializations will induce a topology on $T$.

For $((L_1', L_2'), \phi_{12}) \in T_{V_{L_1} \cross V_{L_2}}$, that is, for $L_1' \in V_{L_1}$, $L_2' \in V_{L_2}$, and $\phi \in T(L_1', L_2')$, we want $\pi_2 \circ \Psi_{(L_1, L_2)} (((L_1', L_2'), \phi)) \in T(L, L)$. The following diagram suggests using for the left and right vertical maps, local trivializations about $L_1$ and $L_2$, $\Theta_{L_1} \colon \pi_F^{-1} (V_{L_1}) \rightarrow V_{L_1} \cross \F(L)$ and $\Theta_{L_2} \colon \pi_F^{-1} (V_{L_2}) \rightarrow V_{L_2} \cross \F(L)$, from proposition \ref{p-std-fock-spac-bndl} notation \ref{n-std-fock-spac-bndl} on the standard Fock space bundle.  This allows us to construct the bottom arrow $\widehat{\phi}$ from the top arrow $\phi$, and thus to define $\Psi_{(L_1, L_2)}$.
\begin{align}
\begindc{\commdiag}[5]
\obj(10,20)[objFL1]{$\F(L_1')$}
\obj(55,20)[objFL2]{$\F(L_2')$}
\obj(10,10)[objFL]{$\F(L)$}
\obj(55,10)[objFLA]{$\F(L)$}
\mor{objFL1}{objFL2}{$\phi$}
\mor{objFL}{objFLA}{$\widehat{\phi} = \pi_2 \circ \Psi_{(L_1, L_2)}(((L_1', L_2'), \phi))$}
\mor{objFL}{objFL1}{$(\pi_2 \circ \Theta_{L_1} (L_1', \cdot))^{-1}$}
\mor{objFL2}{objFLA}{$\pi_2 \circ \Theta_{L_2} (L_2', \cdot)$}
\enddc \label{eq-t-psi-thet}
\end{align}
$\Psi_{(L_1, L_2)}$ is bijective on fibers since it is conjugation with interwiners, thus bijective.  The transition functions can be constructed by stacking two of these diagrams on top of each other.  To distinguish objects from the two diagrams, ``$a$'' is used for the diagram on top and ``$b$'' for that on the bottom.  As for any transition function, the points for the two local trivializations, however named, are the same point.  Thus, for $(L_{1a}', L_{2a}') = (L_{1b}', L_{2b}') \in (V_{L_{1a}} \cross V_{L_{2a}}) \cap (V_{L_{1b}} \cross V_{L_{2b}})$:
\begin{align}
\begindc{\commdiag}[5]
\obj(10,10)[objFL]{$\F(L)$}
\obj(55,10)[objFLA]{$\F(L)$}
\obj(10,25)[objFL1]{$\F(L_{1b}') = \F(L_{1a}')$}
\obj(55,25)[objFL2]{$\F(L_{2b}') = \F(L_{2a}')$}
\obj(10,40)[objFLb]{$\F(L)$}
\obj(55,40)[objFLAb]{$\F(L)$}
\mor{objFL}{objFLA}{$\widehat{\phi}^{b} = \pi_2 \circ \Psi_{(L_{1b}, L_{2b})}(((L_{1b}', L_{2b}'), \phi))$}
\mor{objFL}{objFL1}{$(\pi_2 \circ \Theta_{L_{1b}} (L_{1b}', \cdot))^{-1}$}
\mor{objFL2}{objFLA}{$\pi_2 \circ \Theta_{L_{2b}} (L_{2b}', \cdot)$}
\mor{objFL1}{objFL2}{$\phi$}
\mor{objFLb}{objFL1}{$(\pi_2 \circ \Theta_{L_{1a}} (L_{1b}', \cdot))^{-1}$}[\atright, \solidarrow]
\mor{objFL2}{objFLAb}{$\pi_2 \circ \Theta_{L_{2a}} (L_{2b}', \cdot)$}[\atright, \solidarrow]
\mor{objFLb}{objFLAb}{$\widehat{\phi}^{a} = \pi_2 \circ \Psi_{(L_{1a}, L_{2a})}(((L_{1b}', L_{2b}'), \phi))$}[\atright, \solidarrow]
\enddc \label{eq-t-psi-thet-tran}
\end{align}
\begin{align}
\widehat{\phi}^{a} &= \pi_2 \circ \Psi_{(L_{1a}, L_{2a})} \circ \Psi_{(L_{1b}, L_{2b})}^{-1} (((L_{1b}', L_{2b}'), \widehat{\phi}^{b})) \notag \\
&= (\pi_2 \circ \Theta_{L_{2a}} (L_{2b}', \cdot)) \circ (\pi_2 \circ \Theta_{L_{2b}} (L_{2b}', \cdot))^{-1} \circ \widehat{\phi}^{b} \notag \\
&\circ (\pi_2 \circ \Theta_{L_{1b}} (L_{1b}', \cdot)) \circ (\pi_2 \circ \Theta_{L_{1a}} (L_{1b}', \cdot))^{-1} \notag \\
&= (\pi_2 \circ (\Theta_{L_{2a}} \circ \Theta_{L_{2b}}^{-1} (L_{2b}', \cdot))) \circ \widehat{\phi}^{b} \circ (\pi_2 \circ (\Theta_{L_{1b}} \circ \Theta_{L_{1a}}^{-1} (L_{1b}', \cdot))). \notag
\end{align}
Now $\pi_2 \circ (\Theta_{L_{2a}} \circ \Theta_{L_{2b}}^{-1}) \colon (V_{L_{2b}} \cap V_{L_{2a}}) \cross \F(L) \rightarrow \F(L)$ and $\pi_2 \circ (\Theta_{L_{1b}} \circ \Theta_{L_{1a}}^{-1}) \colon (V_{L_{1b}} \cap V_{L_{1a}}) \cross \F(L) \rightarrow \F(L)$ are continuous transition functions from the proof of proposition \ref{p-std-fock-spac-bndl}, and fixing the first argument of either determines a corresponding element of $T(L, L)$, or equivalently by Schur's Lemma as in the proof of proposition \ref{p-set-intw-u1-tors}, an element of $\UU(1)$ by which to multiply the second argument.  That this element of $\UU(1)$ is a continuous function of the first argument, follows from fixing any nonzero second argument and using the continuity of the transition function.  Thus, since $\widehat{\phi}^{b}$ also is an element of $T(L, L)$, we can write $\widehat{\phi}^{a} = z_{2}(L_{2b}') z_{1}(L_{1b}') \widehat{\phi}^{b}$, where $z_{2} \colon V_{L_{2b}} \cap V_{L_{2a}} \rightarrow \UU(1)$ and $z_{1} \colon V_{L_{1b}} \cap V_{L_{1a}} \rightarrow \UU(1)$ are continuous functions, and conclude that $\widehat{\phi}^{a}$ is a continuous function of $((L_{1b}', L_{2b}'), \widehat{\phi}^{b})$.  Thus the transition functions of $T$ are continuous, and the topology of the total space is defined by the local trivializations.

Now to establish the left action of $L\SO(n)$ on the total and base spaces, with the projection equivariant.  By proposition \ref{p-lson-in-ores}, the inclusion $L\SO(n) \rightarrow \Orth_{res}$ is continuous, so it suffices to find such an action of $\Orth_{res}$.  On the base space, for $g \in \Orth_{res}$, $L_1, L_2 \in \Lagr_{res}$, $g \cdot (L_1, L_2) = (g L_1, g L_2)$ is a continuous action because of definition \ref{d-ores-uvj-homo} and lemma \ref{l-ores-uvj-homo}.  The same action will work for the first component, $(L_1, L_2)$, of elements of the total space; and this will make the projection equivariant.  Since $\phi \in T(L_1, L_2)$, we need $g \cdot \phi \in T(g L_1, g L_2)$.  Since $\phi$ is a map of Fock spaces, thinking of the Fock space construction as a functor (see e.g. lemma \ref{l-lamb-homo}), the natural $\Orth_{res}$ action on Fock spaces is via $g \mapsto \Lambda_g$, which was also used in definition \ref{d-fock-spac-bndl}.  To be compatible with that, the natural action on $\phi$ is $g \cdot \phi = \Lambda_g \phi \Lambda_g^{*} = \Lambda_{g,L_1} \phi \Lambda_{g,L_2}^{*}$ (the decorations of the $\Lambda_g$ showing their domains), which is in $T(g L_1, g L_2)$ by lemma \ref{l-lamb-conj-intw}.

For continuity, use local trivializations that don't vary with the action.  Fix $g \in \Orth_{res}$ and $(L_1, L_2) \in \Lagr_{res} \cross \Lagr_{res}$.  Let $g' \in \Orth_{res}$ vary in an open neighborhood of $g$ and $(L_1', L_2')$ in an open neighborhood of $(L_1, L_2)$, a subset of $V_{L_1} \cross V_{L_2}$ such that $(g' L_1', g' L_2')$ stays in $V_{g L_1} \cross V_{g L_2}$.  Suppose $\widehat{\phi} \in T(L, L)$.  We will show continuity with respect to jointly varying $g', ((L_1', L_2'), \widehat{\phi})$ of
\begin{align}
&\Orth_{res} \cross ((\Lagr_{res} \cross \Lagr_{res}) \cross T(L, L)) \notag \\
&\rightarrow (\Lagr_{res} \cross \Lagr_{res}) \cross T(L, L) \text{ given by} \notag \\
&((L_1', L_2'), \widehat{\phi}) \mapsto ((g' L_1', g' L_2'), \widehat{\phi}^{g'}) \notag \\
&=\Psi_{(g L_1, g L_2)} (((g' L_1', g' L_2'), \Lambda_{g', L_2'} \circ (\pi_2 \circ \Psi_{(L_1, L_2)}^{-1} (((L_1', L_2'), \widehat{\phi}))) \circ \Lambda_{g', L_1'}^{*})). \notag
\end{align}
The symbol $\Lambda$ with two subscripts indicates the group element that induces it and the Lagrangian subspace for the Fock space that is its domain.  As mentioned before, definition \ref{d-ores-uvj-homo} and lemma \ref{l-ores-uvj-homo} imply continuity of the first component, $(g' L_1', g' L_2')$.  The second component is as follows.  The main point is to get the final expression of maps in $\UU(\F(L))$, to make it straightforward to define and prove continuity; whereas the initial expression involves maps between spaces that vary, making even the definition of continuity more complicated.
\begin{align}
\widehat{\phi}^{g'} &= (\pi_2 \circ \Theta_{g L_2} (g' L_2', \cdot)) \circ \Lambda_{g', L_2'} \circ ((\pi_2 \circ \Theta_{L_2} (L_2', \cdot))^{*} \circ \widehat{\phi} \notag \\
&\circ (\pi_2 \circ \Theta_{L_1} (L_1', \cdot))) \circ \Lambda_{g', L_1'}^{*} \circ (\pi_2 \circ \Theta_{g L_1} (g' L_1', \cdot))^{*} \notag \\
&= T_{g' L_2', L} \circ \Lambda_{g', L_2'}  \circ T_{L_2', L}^{*} \circ \widehat{\phi} \circ T_{L_1', L} \circ \Lambda_{g', L_1'}^{*} \circ T_{g' L_1', L}^{*} \notag \\
&= U_{g_{g' L_2'}} \Lambda_{g_{g' L_2'}, L}^{*} \circ \Lambda_{g', L_2'}  \circ (U_{g_{L_2'}} \Lambda_{g_{L_2'}, L}^{*})^{*} \circ \widehat{\phi} \notag \\
&\circ U_{g_{L_1'}} \Lambda_{g_{L_1'}, L}^{*} \circ \Lambda_{g', L_1'}^{*} \circ (U_{g_{g' L_1'}} \Lambda_{g_{g' L_1'}, L}^{*})^{*} \notag \\
&= U_{g_{g' L_2'}} \Lambda_{g_{g' L_2'}^{-1} g' g_{L_2'}} U_{g_{L_2'}}^{*} \circ \widehat{\phi} \circ U_{g_{L_1'}} \Lambda_{g_{L_1'}^{-1} (g')^{-1} g_{g' L_1'}} U_{g_{g' L_1'}}^{*} \notag
\end{align}
where the  $\Theta$'s and $T_{*, L}$ are from the standard Fock space bundle local trivializations of proposition \ref{p-std-fock-spac-bndl} notation \ref{n-std-fock-spac-bndl}, the $U$'s are all in $\UU(\F(L))$.  We have that $T_{L_1', L} = U_{g_{L_1'}} \Lambda_{g_{L_1'}, L}^{*}$, where its first factor is an implementer in $\pi_L$ of $\theta_{g_{L_1'}}$, depending continuously on $g_{L_1'}$, which depends continuously on $L_1'$, similarly for $L_2'$, $g' L_1'$, and $g' L_2'$, and lemma \ref{l-lamb-homo} was used to coalesce the $\Lambda$'s.

Since $g_{g' L_2'}^{-1} g' g_{L_2'}$, reading right to left, maps $L \leftarrow g' L_2' \leftarrow L_2' \leftarrow L$, and $g_{L_1'}^{-1} (g')^{-1} g_{g' L_1'}$ maps $L \leftarrow L_1' \leftarrow g' L_1' \leftarrow L$, the last line $\Lambda$'s, in $\UU(\F(L))$, are by lemma \ref{l-lamb-homo-cont} continuous functions of the group elements, in $\UU(V_J)$, which are continuous functions of $g'$, $L_1'$, and $L_2'$.  The $U$'s are continuous functions of $L_1'$, $L_2'$, $g' L_1'$, and $g' L_2'$  by their construction, and the last two of these are continuous functions of $g'$ and $L_1'$ respectively $L_2'$, jointly, by definition \ref{d-ores-uvj-homo} and lemma \ref{l-ores-uvj-homo}.  Recall that composition and the adjoint are continuous for unitary operators in the strong operator topology on $\F(L)$.  Thus $(g', ((L_1', L_2'), \widehat{\phi})) \mapsto ((g' L_1', g' L_2'), \widehat{\phi}^{g'})$ is a continuous function using the strong operator topology on $\UU(\F(L))$.  As mentioned in proposition \ref{p-set-intw-u1-tors}, the strong operator, operator norm, and $\UU(1)$ torsor topologies on $T(L, L)$ are the same.

Now, suppose we start with some $L' \in \Lagr_{res}$, and use that instead of $L$ for local trivializations $\Theta_K$ etc. of $F$ as in notation \ref{n-std-fock-spac-bndl} of proposition \ref{p-std-fock-spac-bndl}, using the $\Theta_K$ as in \ref{eq-t-psi-thet}, to construct local trivializations $\Psi_{L_1, L_2}'$ as in \ref{eq-t-psi}, with standard fiber $T(L', L')$.  Then these $\Psi_{L_1, L_2}'$  are compatible with the $\Psi_{L_1, L_2}$ constructed starting with $L$, as can be seen using a diagram very like \ref{eq-t-psi-thet-tran}, which is for transition functions between two different $\Psi_{L_1, L_2}$.  Modify the diagram to have $\F(L)$ say at the top, and $\F(L')$ at the bottom; that is, ``$a$'' corresponds to $L$ and ``$b$'' corresponds to $L'$.  This doesn't affect ${}'$ markings other than for $L'$ at the bottom.  Then the same argument that follows the diagram still works, in this case to show that the local trivializations $\Psi_{L_1, L_2}$ and $\Psi_{L_1, L_2}'$ are compatible, because proposition \ref{p-std-fock-spac-bndl} states that its local trivializations, the $\Theta_K$ here, made from $L'$ are compatible with those made from $L$.
\end{proof}

\section{The Identification $\Xi$ for a Fiber Product of an Associated Bundle}\label{s-ident-xi}

We would like to construct the bundle gerbe's $P$ bundle as a bundle associated to $L\SO(E)$, in the same way the $Y$ space is.  This will allow functoriality of the construction.  However, the base of the $P$ bundle needs to be $Y^{[2]}$, which brings two copies of $L\SO(E)$ into the picture.  That is, an element of $Y$ is of the form $[\widetilde{\gamma}, L]$, for some $\widetilde{\gamma} \in L\SO(E)$ and some $L \in \Lagr_{res}$; and so, an element of $Y^{[2]}$ is of the form $([\widetilde{\gamma}_1, L_1], [\widetilde{\gamma}_2, L_2])$, where both $\widetilde{\gamma}_1$ and $\widetilde{\gamma}_2$ lie above the same $\gamma$.

Because the action of $L\SO(n)$ is transitive on the fibers of $L\SO(E)$, it's possible to construct an isomorphism $\Xi$ that gets rid of one of the copies of $L\SO(E)$.  It is a special case of the isomorphism $\Xi$ in the following lemma.  After pointing out some instances where the use of it and its relatives would make definitions more precise, we will generally treat them as identifications.

\begin{lem}\label{l-iden-xi}
\index{Xi@$\Xi$}
(The Identification $\Xi$).
Suppose $G$ is a topological group, $Q \rightarrow X$ is a topological principal $G$ bundle, and $W$ is a topological space on which there is a continuous left $G$ action.  Let
\[
Y = Q \cross_G W \rightarrow X \notag
\]
denote the associated fiber bundle, and $Y^{[2]}$ the fiber product as in definition \ref{d-y-n}.  Given $q_1, q_2 \in Q$ over $\gamma \in X$, $w_1, w_2 \in W$, define
\begin{align}
\Xi \colon Y^{[2]} &\rightarrow Q \cross_G (W \cross W) \notag \\
\Xi ([q_1, w_1], [q_2, w_2]) &= \Xi ([q_1, w_1], [q_1 \tau (q_1, q_2), w_2]) \notag \\
 &= \Xi ([q_1, w_1], [q_1, \tau (q_1, q_2) w_2]) \notag \\
 &= [q_1, w_1, \tau (q_1, q_2) w_2], \notag
\end{align}
where $\tau$ is the continuous translation function of lemma \ref{l-pb-tran-func-tors} for Q.

Then $\Xi$ is an isomorphism of continuous fiber bundles, a homeomorphism over the identity on $X$.
\end{lem}
\begin{proof}
The reason that $\Xi$ exists is that only one copy of $Q$ is needed, because $G$ acts transitively on it, hence different elements of $Q$ lying over the same loop in $X$ can be reflected instead in different elements of $W$.

First define a map $(Q \cross W) \cross_{X} (Q \cross W) \rightarrow Q \cross (W \cross W)$ respecting equivalence relations, each factor of $G \cross G$ acting on the respective factor of the domain, and the first factor of $G \cross G$ acting on the codomain.  Informally, for $g_1, g_2 \in G$,
\begin{align}
((q_1 g_1, g_1^{-1} w_1), (q_2 g_2, g_2^{-1} w_2)) &\mapsto [q_1 g_1, g_1^{-1} w_1, \tau (q_1 g_1, q_2 g_2) g_2^{-1} w_2] \notag \\
 &= [q_1 g_1, g_1^{-1} w_1, g_1^{-1} \tau (q_1, q_2) g_2 g_2^{-1} w_2] \notag \\
 &= [q_1, w_1, \tau (q_1, q_2) w_2] \notag
\end{align}
shows that $\Xi$ is well-defined, and continuity follows because a continuous equivariant map descends to a continuous map of the orbit spaces \citep[page~4]{tomD87}.  Well-definedness and continuity of the inverse defined by
\[
\Xi^{-1} ([q_1, w_1, w_2]) = ([q_1, w_1], [q_1, w_2]) \notag
\]
are similar but easier.
\end{proof}
We can also define, for instance, the analogous canonical homeomorphism $Y^{[3]} \rightarrow Q \cross_{G} (W \cross W \cross W)$.

\section{The Bundle Gerbe $P$ Space Construction}\label{s-bndl-gerb-p-spac-cons}

\begin{prop}\label{p-bg-p}
\index{bundle gerbe!construction!P@$P$}
\index{P@$P$}
(The Bundle Gerbe $P$ Space Construction).
Suppose given $L \in \Lagr_{res}$.  Define the $P$ bundle for the bundle gerbe as follows, omitting $\Xi$:
\[
\begindc{\commdiag}[5]
\obj(30,20)[objPp]{$P = L\SO(E) \cross_{L\SO(n)} T$}
\obj(30,10)[objYY]{$Y^{[2]} = L\SO(E) \cross_{L\SO(n)} (\Lagr_{res} \cross \Lagr_{res})$}
\mor{objPp}{objYY}{}
\enddc
\]
or for once showing $\Xi$ from lemma \ref{l-iden-xi}:
\[
\begindc{\commdiag}[5]
\obj(0,20)[objP]{$P = \Xi^{*} P'$}
\obj(0,10)[objY2]{$Y^{[2]}$}
\obj(30,20)[objPp]{$P' = L\SO(E) \cross_{L\SO(n)} T$}
\obj(30,10)[objYY]{$L\SO(E) \cross_{L\SO(n)} (\Lagr_{res} \cross \Lagr_{res})$}
\mor{objPp}{objYY}{}
\mor{objY2}{objYY}{$\Xi$}
\mor{objP}{objY2}{}
\enddc
\]
The choice of $L$ only affects local trivializations; a different choice results in compatible ones, as in proposition \ref{p-bg-t}. The fibers of $P$ may be identified as follows:
\begin{align}
P_{([\widetilde{\gamma}, L_1], [\widetilde{\gamma}, L_2])} &= P_{[\widetilde{\gamma}, L_1, L_2]} \notag \\
 &= \{ [(\widetilde{\gamma}, L_1, L_2, \phi)] \st \phi \in T(L_1, L_2) \} \notag \\
&\text{or using the bijection of definition \ref{d-pol-clas-bndl},} \notag \\
P_{(\gamma, L_{\gamma, 1}, L_{\gamma, 2})} &= \{ (\gamma, L_{\gamma, 1}, L_{\gamma, 2}, \phi) \st \phi \in T(L_{\gamma, 1}, L_{\gamma, 2}) \}, \notag
\end{align}
where $\widetilde{\gamma} \in L\SO(E)$ lies over $\gamma \in LM$, $L_1$, $L_2$ $\in \Lagr_{res}$, and $L_{\gamma, 1}$, $L_{\gamma, 2}$ are in the chosen polarization class of $\CC \tensor \overline{LE_{\gamma}}$.

Define $p \colon P \rightarrow Y^{[2]}$ by $[\widetilde{\gamma}, L_1, L_2, \phi] \mapsto ([\widetilde{\gamma}, L_1], [\widetilde{\gamma}, L_2])$.  Since $T$ is a principal $\UU(1)$ bundle, $P$ is also; the right $\UU(1)$ action on the right factor of the associated bundle survives that construction.

The bundle gerbe multiplication is given fiberwise by composition of Clifford-linear unitary isomorphisms; that is, intertwiners.  Using relatives of $\Xi$ to omit subscript redundancies, define $\widetilde{m} \colon P_{[(\widetilde{\gamma}, L_1, L_2)]} \cross P_{[(\widetilde{\gamma}, L_2, L_3)]} \rightarrow P_{[(\widetilde{\gamma}, L_1, L_3)]}$ by
\[
\widetilde{m} ([(\widetilde{\gamma}, L_1, L_2, \phi_{12})], [(\widetilde{\gamma}, L_2, L_3, \phi_{23})]) = [(\widetilde{\gamma}, L_1, L_3, \phi_{23} \circ \phi_{12})]. \notag
\]
This uses equivalence class representatives with equal first elements $\widetilde{\gamma}$, similarly to what is done with $\Xi$.  The map $\widetilde{m}$ is $\UU(1)$-bi-equivariant and induces an isomorphism of $\UU(1)$ torsors, $m \colon P_{[(\widetilde{\gamma}, L_1, L_2)]} \tensor P_{[(\widetilde{\gamma}, L_2, L_3)]} \rightarrow P_{[(\widetilde{\gamma}, L_1, L_3)]}$, which defines a principal $\UU(1)$ bundle isomorphism $m \colon \pi_{12}^{*} P \tensor \pi_{23}^{*} P \rightarrow \pi_{13}^{*} P$ that has the required associativity, and induces the bundle gerbe multiplication.
\end{prop}
\begin{proof}
To see why what $\widetilde{m}$ does with the $L_i, L_j$ components is well-defined, note that $[(\widetilde{\gamma}g, K)] = [(\widetilde{\gamma}, K')] \Leftrightarrow K' = g K$, so when equivalence class representatives are chosen with equal first elements $\widetilde{\gamma}$, the second and third elements, in $\Lagr_{res}$, are well-defined.

What $\widetilde{m}$ does with the $\phi_{ij}$ is well-defined because, for $g \in L\SO(n)$, $L_i \in \Lagr_{res}$, $\phi_{ij} \in T(L_i, L_j)$,
\begin{align}
g \cdot ((L_i, L_j), \phi_{ij}) &= ((g L_i, g L_j), \Lambda_{g, L_j} \circ \phi \circ \Lambda_{g, L_i}^{*}) \notag \\
g \cdot \phi_{13} &= g \cdot (\phi_{23} \circ \phi_{12}) \notag \\
&= \Lambda_{g, L_3} \circ \phi_{23} \circ \phi_{12} \circ \Lambda_{g, L_1}^{*} \notag \\
&= \Lambda_{g, L_3} \circ \phi_{23} \circ \Lambda_{g, L_2}^{*} \circ \Lambda_{g, L_2} \circ \phi_{12} \circ \Lambda_{g, L_1}^{*} \notag \\
&= (g \cdot \phi_{23}) \circ (g \cdot \phi_{12}). \notag
\end{align}
The symbol $\Lambda$ with two subscripts indicates the group element that induces it and the Lagrangian subspace for the Fock space that is its domain.  $\widetilde{m}$ is $\UU(1)$-bi-equivariant because the elements of these torsors are linear operators, so $m$ is $\UU(1)$-equivariant, and by corollary \ref{co-g-equi-cont}, is an isomorphism of $\UU(1)$ torsors.  Recall proposition \ref{p-set-intw-u1-tors}; the topology on each of $T(L_1, L_3)$, $T(L_2, L_3)$, and $T(L_1, L_3)$ given by its $\UU(1)$ torsor topology and operator norm topology are the same.  Thus $m$ is a homeomorphism on fibers; i.e. for fixed $L_i$.

The fiberwise $m$ induces a $\UU(1)$-equivariant bijection also named $m$ from the total space of $\pi_{12}^{*} P \tensor \pi_{23}^{*} P$ to that of $\pi_{13}^{*} P$, covering the identity on the base, and associativity follows from the fiberwise definition.  To show that $m$ is a principal $\UU(1)$ bundle isomorphism, by lemma \ref{l-pb-mor-cov-id-isom} it suffices to show that it is continuous.

The continuity of the principal $\UU(1)$ bundle map $m$ follows from continuity of the principal $\UU(1)$ bundle map induced by $\widetilde{m}$, which in turn is induced by and its continuity follows from that of the following map $\widetilde{\widetilde{m}}$, which is $L\SO(n)$-equivariant in the way needed to make it descend (a continuous equivariant map descends to a continuous map of the orbit spaces by \citet[page~4]{tomD87}) to $\widetilde{m}$ by taking associated bundle quotients of the cartesian products $L\SO(E) \cross T$.  The map $\widetilde{\widetilde{m}}$ is also $\UU(1)$-bi-equivariant, acting on the $T$'s.  To define $\widetilde{\widetilde{m}}$ properly, insert symbols $(\widetilde{\gamma}, (L_1, L_2, L_3))$ appropriately.
\begin{align}
\pi_{12}^{*} (L\SO(E) \cross T) \cross \pi_{23}^{*} (L\SO(E) \cross T) &\xrightarrow{\widetilde{\widetilde{m}}} \pi_{13}^{*} (L\SO(E) \cross T) \notag \\
((L_1, L_2, \phi_{12}), (L_2, L_3, \phi_{23})) &\mapsto ((L_1, L_3, \phi_{12} \circ \phi_{23})). \notag
\end{align}
$\widetilde{\widetilde{m}}$ respects the equivalence relations for the associated bundle $L\SO(n)$ quotients, as can be seen from the argument given that $\widetilde{m}$ is well-defined, and also respects the equivalence relations for the tensor product quotient that will be applied to $\widetilde{m}$.  The continuity of $\widetilde{\widetilde{m}}$ can be seen locally as follows.

Because in the definition of $\widetilde{\widetilde{m}}$, $\widetilde{\gamma} \in L\SO(E)$ is unchanged, and neither what happens to the other components nor the local trivializations that will be used for them depend on $\widetilde{\gamma}$, we omit it along with triples of $L_i$ for pullbacks to $Y^{[3]}$ from the argument until the end.  Let $V_{L_1}$, $V_{L_2}$, $V_{L_3}$ be open neighborhoods of $L_1$, $L_2$, $L_3$ respectively, giving box neighborhood domains of local trivializations as in \ref{eq-t-psi}  in the proof of proposition \ref{p-bg-t}, $\Psi_{L_1, L_2}$, $\Psi_{L_2, L_3}$, and $\Psi_{L_1, L_3}$, of $T$.

In these local trivializations, $\widetilde{\widetilde{m}}$ is again given by composition, as follows, for $L_1' \in V_{L_1}$, $L_2' \in V_{L_2}$, $L_3' \in V_{L_3}$, $\phi_{12} \in T(L_1', L_2')$, $\phi_{23} \in T(L_2', L_3')$, $\phi_{13} \in T(L_1', L_3')$.  First define $\widehat{\phi_{ij}}$, the intertwiners in the local trivializations corresponding to the intertwiners $\phi_{ij}$, and show what the multiplication by composition of the $\phi_{ij}$ becomes, in
terms of the $\widehat{\phi_{ij}}$ and $\Psi_{L_i, L_j}$.
\begin{align}
\widehat{\phi_{ij}} &= \pi_2 \circ \Psi_{L_i, L_j} (L_i', L_j', \phi_{ij}) \notag \\
\widehat{\phi_{13}} &= \pi_2 \circ \Psi_{L_1, L_3} \notag \\
 &(L_1', L_3', (\pi_2 \circ (\Psi_{L_1, L_2})^{-1} (L_1', L_2', \widehat{\phi_{12}})) \circ (\pi_2 \circ (\Psi_{L_2, L_3})^{-1} (L_2', L_3', \widehat{\phi_{23}}))). \notag
\end{align}
Then express these formulas, that are in terms of the local trivializations $\Psi_{L_i, L_j}$ of $T$, in terms of the local trivializations $\Theta_{L_i}$ of $F$, as in \ref{eq-t-psi-thet} in the proof of proposition \ref{p-bg-t}.
\begin{align}
\widehat{\phi_{ij}} &= (\pi_2 \circ \Theta_{L_j} (L_j', \cdot)) \circ \phi_{ij} \circ (\pi_2 \circ \Theta_{L_i} (L_i', \cdot))^{-1}  \label{eq-pp-bg-mult} \\
\phi_{ij} &= (\pi_2 \circ \Theta_{L_j} (L_j', \cdot))^{-1} \circ \widehat{\phi_{ij}} \circ (\pi_2 \circ \Theta_{L_i} (L_i', \cdot)) \notag \\
\widehat{\phi_{13}} &= (\pi_2 \circ \Theta_{L_3} (L_3', \cdot)) \notag \\
 &\circ (\pi_2 \circ \Theta_{L_3} (L_3', \cdot))^{-1} \circ \widehat{\phi_{23}} \circ (\pi_2 \circ \Theta_{L_2} (L_2', \cdot)) \notag \\
 &\circ (\pi_2 \circ \Theta_{L_2} (L_2', \cdot))^{-1} \circ \widehat{\phi_{12}} \circ (\pi_2 \circ \Theta_{L_1} (L_1', \cdot)) \notag \\
 &\circ (\pi_2 \circ \Theta_{L_1} (L_1', \cdot))^{-1} \notag \\
 &= \widehat{\phi_{23}} \circ \widehat{\phi_{12}}. \notag
\end{align}
Since the $\widehat{\phi_{ij}}$ are in $T(L, L)$ and composition is continuous in the various equivalent topologies on that set, $\widetilde{\widetilde{m}}$ is locally continuous - the other components such as $\widetilde{\gamma}$ just tag along unchanged - and hence is continuous.
\end{proof}

\section{The Bundle Gerbe Construction}\label{s-bndl-gerb-cons}

With all the pieces in hand, we now construct the bundle gerbe.
\begin{prop}\label{p-bg}
\index{bundle gerbe!construction}
\index{paracompact}
\index{good cover}
(The Bundle Gerbe Construction).
The object given by
\[
\begindc{\commdiag}[5]
\obj(10,30)[objP]{$P$}
\obj(10,20)[objYY]{$Y^{[2]}$}
\obj(20,20)[objY]{$Y$}
\obj(20,10)[objLM]{$LM$}
\mor{objP}{objYY}{$p$}
\mor(10,21)(20,21){$\pi_1$}
\mor(10,19)(20,19){$\pi_2$}
\mor{objY}{objLM}{$\pi$}
\enddc
\]
is a continuous bundle gerbe, where $Y \rightarrow LM$ is the polarization class bundle of definition \ref{d-pol-clas-bndl}, $P$, $p$, and bundle gerbe multiplication come from proposition \ref{p-bg-p}.
\end{prop}
\begin{proof}
Since $M$ is a smooth manifold, by proposition \ref{p-smth-free-loop-spac-frec-mfld} $LM$ is paracompact, and any cover has a refinement that is a good cover. Since $Y \rightarrow LM$ is a locally trivial fiber bundle, it has local sections.
\end{proof}

\chapter{THE BUNDLE GERBE CONSTRUCTION FUNCTOR}\label{c-bndl-gerb-cnst-func}

This chapter shows the functoriality of the bundle gerbe construction, which implies that it is independent of the choice of fiberwise inner product on the vector bundle with which it starts.  Then it gives a stability property of the construction.

\section{Bundle Gerbe Construction Functoriality}\label{s-bndl-gerb-cnst-func}

The basis of the functoriality of the bundle gerbe construction is that the $Y$ and $P$ spaces of the bundle gerbe are bundles associated to $L\SO(E)$.  A vector bundle morphism $E \rightarrow F$ becomes post-composition of elements of $L\SO(E)$ (see section \ref{s-tech-over}) with the morphism, not disturbing the right-hand factors of the associated bundles.  One reason the functoriality is interesting is for use in chapter \ref{c-futu-work}, where, if the original smooth vector bundle $E \rightarrow M$ has a spin structure, the case of a particular universal vector bundle might imply the conjecture generally, that the transgression of the first Pontryagin class of the vector bundle equals plus or minus twice the Dixmier-Douady class of the bundle gerbe.

\begin{defn}\label{d-bg-func}
\index{bundle gerbe!construction!functor}
(The Bundle Gerbe Construction Functor).
Denote the bundle gerbe construction functor by $G \colon \mathcal{V} \rightarrow \mathcal{B}$.  Define the objects of category $\mathcal{V}$ as oriented smooth vector bundles of even rank with fiberwise inner product, over compact connected orientable smooth manifolds with Riemannian metric.  Let the morphisms of $\mathcal{V}$ be smooth orientation-preserving vector bundle morphisms that are isometric linear isomorphisms on the fibers, written as pairs of maps, $(\overline{f}, f)$, the first component being the map on the total spaces, and the second, the map on the base spaces.  However, we will also refer to the pair by the shorthand $f$.  Let the category $\mathcal{B}$ consist of continuous bundle gerbes over topological spaces, and their morphisms.  The action of $G$ on $\Obj(\mathcal{V})$ is as in proposition \ref{p-bg}. Given a morphism $(\overline{f}, f) \colon (E, M) \rightarrow (F, N)$ of $\mathcal{V}$, $G$ produces $(\widehat{G(f)}, \overline{G(f)}, G(f)) \colon (P, Y, LM) = G(E, M) \rightarrow G(F, N) = (Q, Z, LN)$:
\[
\begindc{\commdiag}[5]

\obj(10,20)[objE]{$E$}
\obj(10,10)[objM]{$M$}
\obj(20,20)[objF]{$F$}
\obj(20,10)[objN]{$N$}
\mor{objE}{objM}{}
\mor{objE}{objF}{$\overline{f}$}
\mor{objM}{objN}{$f$}
\mor{objF}{objN}{}

\obj(28,20)[objGs]{}
\obj(38,20)[objGt]{}
\mor{objGs}{objGt}{$G$}[\atleft,\aplicationarrow]

\obj(46,30)[objP]{$P$}
\obj(46,20)[objYY]{$Y^{[2]}$}
\obj(56,20)[objY]{$Y$}
\obj(56,10)[objLM]{$LM$}
\mor{objP}{objYY}{}
\mor(47,21)(56,21){}
\mor(47,19)(56,19){}
\mor{objY}{objLM}{}

\obj(86,30)[objQ]{$Q$}
\obj(86,20)[objf]{$Z^{[2]}$}
\obj(76,20)[objZ]{$Z$}
\obj(76,10)[objLN]{$LN$}
\mor{objQ}{objf}{}
\mor(86,21)(76,21){}
\mor(86,19)(76,19){}
\mor{objZ}{objLN}{}

\mor{objP}{objQ}{$\widehat{G(f)}$}
\mor{objY}{objZ}{$\overline{G(f)}$}
\mor{objLM}{objLN}{$G(f)$}

\enddc
\]
The assumption of orientability the base manifolds of the vector bundles may not be needed; see discussion in the proof of proposition \ref{p-smth-free-loop-spac-frec-mfld}.

Define $G(f) = L(f)$; i.e., $G(f)(\gamma) = L(f) (\gamma) = (f  \circ \gamma)$.

Define $\overline{G(f)}$ as the map $Y \rightarrow Z$ induced, using the associated bundle definitions of $Y$ and $Z$, by the right-$L\SO(n)$-equivariant map $L\SO(E) \rightarrow L\SO(F)$ defined by post-composition with $\overline{f} \in \SO(E, F)$, where $\SO(E, F)$ denotes the smooth orientation-preserving vector bundle morphisms $E \rightarrow F$ that are isometric linear isomorphisms on fibers.  In other words, $\widetilde{\gamma} \in L\SO(E) \mapsto (Lf) (\widetilde{\gamma}) = \overline{f} \circ \widetilde{\gamma} \in L\SO(F)$.

Define $\widehat{G(f)} \colon P \rightarrow Q$ by the same means using the associated bundle definitions of $P$ and $Q$.
\end{defn}

\begin{prop}\label{p-bg-func}
\index{bundle gerbe!construction!functor}
(The Bundle Gerbe Construction Functor).
$G$ is a functor.
\end{prop}
\begin{proof}
$G(f)$ is continuous since $L(f) \colon LM \rightarrow LN$ is smooth by proposition \ref{p-loop-smth-map-smth} and hence continuous by definition \ref{d-frec-deri} of \Frechet differentiability, extended in the usual way to \Frechet manifolds.  For the same reason, $L\overline{f} \colon LE \rightarrow LF$ is continuous.

Since elements of $L\SO(E)$ are maps $\widetilde{\gamma} \colon LE \leftarrow L\RR^n$, $(L\overline{f}) (\widetilde{\gamma}) \in L\SO(F)$.  Post-composition with $\overline{f}$ being equivariant with respect to pre-composition with elements of $L\SO(n)$, the continuous maps $L\overline{f}$ induces of cartesian products with left factor $L\SO(E)$ to those with left factor $L\SO(F)$, descend to continuous maps of the corresponding associated bundles, because a continuous equivariant map descends to a continuous map of the orbit spaces \citep[page~4]{tomD87}.  Thus $\overline{G(f)}$ and $\widehat{G(f)}$ are continuous.  The latter is a principal $\UU(1)$ bundle morphism because post-composition with $\overline{f}$ is $\UU(1)$-equivariant.  We have that $\overline{G(f)}$ covers $G(f)$ and $\widehat{G(f)}$ covers $\overline{G(f)}$:
\begin{align}
(\pi_Z \circ \overline{G(f)}) ([(\widetilde{\gamma}, \dots)]) &= \pi_Z ([((L\overline{f}) (\widetilde{\gamma}), \dots)]) = (L\pi_F) ((L\overline{f}) (\widetilde{\gamma})) \notag \\
&= \pi_F \circ \overline{f} \circ \widetilde{\gamma} \notag \\
&= f \circ \pi_E \circ \widetilde{\gamma} \text{ (} = f \circ \gamma \text{)} \notag \\
&= f \circ L\pi_E (\widetilde{\gamma}) = (f \circ \pi_Y) ([(\widetilde{\gamma}, \dots)]) \notag \\
&= (G(f) \circ \pi_Y) ([(\widetilde{\gamma}, \dots)]). \notag \\
(q \circ \widehat{G(f)}) ([(\widetilde{\gamma}, \dots')]) &= q ([((L\overline{f}) (\widetilde{\gamma}), \dots')]) = ([((L\overline{f}) (\widetilde{\gamma}), \dots'')]) \notag \\
&= \overline{G(f)}_2 ([(\widetilde{\gamma}, \dots'')]) \notag \\
&= (\overline{G(f)}_2 \circ p) ([(\widetilde{\gamma}, \dots')]) \notag.
\end{align}

$\widehat{G(f)}$ commutes with bundle gerbe multiplication over $\overline{G(f)}_3$ because that multiplication is defined by composition of intertwiners that are elements of the right hand factor of the associated bundle construction of $P$ and $Q$, whereas $\widehat{G(f)}$ acts on the left on the left hand factor, covering $\overline{G(f)}_2$.  That is, writing $\widehat{G(f)}$ rather than the more precise $\pi_{13}^{*} \widehat{G(f)}$ that takes into account which fiber product of $Y$ the bundles are over, and similarly for $12, 23$,
\begin{align}
\widehat{G(f)} (m_P ([(\widetilde{\gamma}, \dots, \phi_{12})] \tensor [(\widetilde{\gamma}, \dots, \phi_{23})])) &= \widehat{G(f)} ([(\widetilde{\gamma}, \dots, \phi_{23} \circ \phi_{12})]) \notag \\
&= [(\overline{f} \circ \widetilde{\gamma}, \dots, \phi_{23} \circ \phi_{12})], \notag \\
m_Q ((\widehat{G(f)} ([(\widetilde{\gamma}, \dots, \phi_{12})])) \tensor (\widehat{G(f)} ([(\widetilde{\gamma}, \dots, \phi_{23})]))) &= m_Q ([(\overline{f} \circ \widetilde{\gamma}, \dots, \phi_{12})] \notag \\
&\tensor [(\overline{f} \circ \widetilde{\gamma}, \dots, \phi_{23})]) \notag \\
&= [(\overline{f} \circ \widetilde{\gamma}, \dots, \phi_{23} \circ \phi_{12})]. \notag
\end{align}
Thus $\overline{G(f)}$ is a bundle gerbe morphism.  To conclude that $G$ is a functor, it only remains to note that it maps identity morphisms in $\mathcal{V}$ to identity morphisms in $\mathcal{B}$, and that it preserves composition of morphisms because it more or less maps a morphism to composition on the left by that morphism.
\end{proof}

Now we can show that the bundle gerbe constructed, up to stable isomorphism, doesn't depend on the fiberwise inner product chosen for the vector bundle.  First, a lemma for inner products on a vector space.
\begin{lem}\label{l-hilb-spac-inne-prod}
\index{Hilbert space!two inner products}
(Obtaining an Isometric Linear Isomorphism from Two Inner Products on a Vector Space).
Given a finite-dimensional real or complex Hilbert space $V$ with Hermitian inner product $\langle , \rangle$, there is a bijection between Hermitian inner products $\langle , \rangle_2$ on $V$, and self-adjoint strictly positive linear automorphisms $S$ of $V$, which thus have positive eigenvalues and preserve orientation, such that for every $v, w \in V$,
\[
\langle v, w \rangle_2 = \langle S (v), S (w) \rangle. \notag
\]
By strictly positive is meant that for every $0 \ne v \in V$, $\langle S (v), v \rangle > 0$.

Considering the Hermitian inner products on $V$ as a subset $Inn(V)$ of $Sesq(V)$, the Banach space of bounded sesquilinear forms on $V$, with norm $\norm{\langle , \rangle_2} = \sup_{\norm{v}, \norm{w} \le 1} \abs{\langle v, w \rangle_2}$. Then $Inn(V)$ is an open set in $Sesq(V)$, and the map $Inn(V) \rightarrow \Hom (V)$, the set of bounded operators on $V$, given by $\langle , \rangle_2 \mapsto S$, is smooth as defined using the Banach space structures of $Sesq(V)$ and $\Hom(V)$.
\end{lem}
\begin{proof}
Being finite-dimensional, all Hermitian inner products on $V$ give rise to equivalent norms, all linear operators and sesquilinear forms are bounded.  Lemma 3.2.2 of \citet[page~89]{Pede89} gives an isometric isomorphism $\phi \colon Sesq(V) \rightarrow \Hom(V)$, mapping $b \in Sesq(V)$ to $T$ such that for every $x, y \in V$, $\langle x, y \rangle_2 = \langle T (x), y \rangle$.  Being a linear operator of finite-dimensional normed linear spaces, $\phi$ is smooth.  The set $Inn(V)$ is open in $Sesq(V)$, so the restriction $\phi_{|Inn(V)}$ is smooth.

Given $b \in Inn(V)$, $T = \phi(b)$ is self adjoint, since
\begin{align}
\langle v, T^* (w) \rangle &= \langle T (v), w \rangle \notag \\
 &= \langle v, w \rangle_2 \notag \\
 &= \overline{\langle w, v \rangle_2} \notag \\
 &= \overline{\langle T (w), v \rangle} \notag \\
 &= \langle v, T (w) \rangle \notag
\end{align}
for every $v, w \in V$.  $T$ is positive because $\langle T (v), v \rangle = \langle v, v \rangle_2 \ge 0$; in fact, $\langle T (v), v \rangle > 0$ for $v \ne 0$.  Thus \citep[page~92]{Pede89} there is a unique self-adjoint positive square root $S$ of $T$, and
\begin{align}
\langle S (v), S (w) \rangle &= \langle S^* S (v), w \rangle \notag \\
 &= \langle S^2 (v), w \rangle \notag \\
 &= \langle T (v), w \rangle \notag \\
 &= \langle v, w \rangle_2 \notag
\end{align}
for every $v, w \in V$.

We have given a map $\langle , \rangle_2 \mapsto S$, and now show that it's a bijection from Hermitian inner products to self-adjoint strictly positive operators.  Injectivity is because if two such inner products differed, there would be a pair of vectors $v, w$ for which they differed, resulting in different operators $T$ and hence $S$.  Surjectivity is because given such an $S$, $\langle S^2 (\cdot), \cdot \rangle = \langle S (\cdot), S (\cdot) \rangle$ is a Hermitian inner product, from which the construction again obtains $S$.

Since the Taylor series for $z \mapsto \sqrt{z}$ about $1 \in \CC$ has radius of convergence $1$, scaling as needed and applying \citet[page~59]{Mich46} for $\RR$ or his quotation of the unpublished thesis \citet{Mart32} for $\CC$, $S$ is a smooth function of $T$.
\end{proof}
If the roles of $\langle , \rangle$ and $\langle , \rangle_2$ are reversed, $T$ and $S$ are replaced by their inverses.

Next, a lemma for fiberwise inner products on a vector bundle.
\begin{lem}\label{l-fibe-inne-prod}
\index{fiberwise inner product!vector bundle}
(Obtaining an Isometric Vector Bundle Isomorphism from Two Fiberwise Inner Products on a Vector Bundle).
Let $E$ be a smooth vector bundle with fiberwise inner product, a smooth manifold $M$.  Denote by $E_2$ the same vector bundle with a possibly different fiberwise inner product.  Then there is a vector bundle isomorphism $E \rightarrow E_2$ preserving inner products, over the identity on $M$.  If $E$ is oriented, the isomorphism preserves the orientation.
\end{lem}
\begin{proof}
Denote the fiberwise inner product, for $v, w$ in the fiber of $E$ over $x \in M$, by $\langle v, w \rangle_x$.  Correspondingly, for $E_2$, denote the fiberwise inner product by $\langle v, w \rangle_{x, 2}$.  Apply the construction of lemma \ref{l-hilb-spac-inne-prod} to get for each $x \in M$ a self adjoint positive linear automorphism $S_x \colon E_x \rightarrow (E_2)_x$, resulting in a map $S \colon E \rightarrow E_2$ over $\ident_M$.  To see that $S$ is smooth, choose a vector bundle local trivialization over an open $U \subset M$, of $\pi \colon E \rightarrow M$:
\[
\pi^{-1} (U) \xrightarrow{\phi} U \cross \RR^n, \notag
\]
with the diffeomorphism $\phi$ isometric for each $x$ with respect to $\langle , \rangle_x$ and the standard inner product $\langle , \rangle_{\RR^n}$ on $\RR^n$.  Use $\phi$ to induce from $\langle , \rangle_{x, 2}$ another inner product $\langle , \rangle_{\RR^n, x, 2}$ on $\RR^n$, a smooth function of $x \in U$.  Let $S_{\RR^n, x} \in \Hom(\RR^n)$ be the automorphism of lemma \ref{l-hilb-spac-inne-prod}, a smooth function of $\langle , \rangle_{\RR^n, x, 2}$ and hence of $x$.

Choose any $x \in U$ and $v, w \in E_x$, let $\pi_2 \colon U \cross \RR^n \rightarrow \RR^n$ be the projection, and then
\begin{align}
\langle \pi_2 \circ \phi (S_x^2 (v)), \pi_2 \circ \phi (w) \rangle_{\RR^n} &= \langle S_x^2 (v), w \rangle_x \notag \\
 &= \langle v, w \rangle_{x, 2} \notag \\
 &= \langle \pi_2 \circ \phi (v), \pi_2 \circ \phi (w) \rangle_{\RR^n, x, 2} \notag \\
 &= \langle S_{\RR^n, x}^2 (\pi_2 \circ \phi (v)), \pi_2 \circ \phi (w) \rangle_{\RR^n}, \text{ whence} \notag \\
\pi_2 \circ \phi \circ S_x^2 &= S_{\RR^n, x}^2 \circ \pi_2 \circ \phi, \text{ so} \notag \\
\phi^{-1}(x, \pi_2 \circ \phi \circ S_x^2) &= \phi^{-1} (x, S_{\RR^n, x}^2 \circ \pi_2 \circ \phi), \text{ and} \notag \\
S_x^2 &= \phi^{-1} (x, S_{\RR^n, x}^2 \circ \pi_2 \circ \phi). \notag
\end{align}
Thus $S_x$ is a smooth function of $x$, and hence $S$ is a smooth function.
\end{proof}

\begin{cor}\label{co-inde-fibe-inne-prod}
\index{bundle gerbe!vector bundle!fiberwise inner product}
\index{fiberwise inner product!vector bundle}
(The Bundle Gerbe is Independent of the Fiberwise Inner Product).
Let $E$ be an oriented smooth vector bundle of even rank with fiberwise inner product, over a compact connected orientable smooth manifold $M$ with Riemannian metric.  Denote by $E_2$ the same vector bundle with a possibly different fiberwise inner product.  Then the vector bundle isomorphism $E \rightarrow E_2$ over the identity on $M$ from lemma \ref{l-fibe-inne-prod}, which preserves inner products, induces a bundle gerbe isomorphism $G(E) \rightarrow G(E_2)$ over the identity on $LM$.
\end{cor}
\begin{proof}
Applying the functor of proposition \ref{p-bg-func} to the isometric isomorphism of vector bundles, as with any functor we obtain from the isomorphism an isomorphism, of bundle gerbes.
\end{proof}
\begin{note}\label{n-inde-riem-metr}
\index{bundle gerbe!vector bundle!base manifold assumptions note}
(The Bundle Gerbe may be Independent of the Riemannian Metric).
As noted in the proof of proposition \ref{p-smth-free-loop-spac-frec-mfld}, the Riemannian metric on $M$ is used only to define the \Frechet manifold structure on $LM$, which may not depend on the choice of metric.  As also noted, it may not be necessary that $M$ be orientable.
\end{note}

\section{Bundle Gerbe Construction Stability}\label{s-bndl-gerb-cnst-stab}

Given an oriented smooth vector bundle $E$ of even rank with fiberwise inner product, over a compact connected orientable smooth manifold with Riemannian metric, the bundle gerbe constructed from its Whitney sum $E \dirsum I_{2k}$ with an even rank trivial bundle, is stably isomorphic to the bundle gerbe constructed from $E$.  Thus, if two such vector bundles are stably equivalent \citep[page~117]{Huse94}, so are the bundle gerbes constructed from them.
\begin{lem}\label{l-fock-spac-sum-lagr}
\index{Fock space}
(The Fock Space of a Sum of Lagrangian Subspaces).
Given real Hilbert spaces $V_1$, $V_2$ and Lagrangian subspaces $L_1 \subset H_1$, $L_2 \subset H_2$ of their complexifications, $L_1 \dirsum L_2$ is a Lagrangian subspace of $H_1 \dirsum H_2$, where the sums are orthogonal direct sums, the Hilbert space inner product \citep[page~103]{Pede89} $\F(L_1) \tensor \F(L_2)$ is $\ZZ_2$ graded, as defined on tensor product decomposables by $\partial (w_1 \tensor w_2) = \partial w_1 + \partial w_2$, for $w_1 \in \F(L_1)$, $w_2 \in \F(L_2)$, and there is a $\ZZ_2$ grade preserving $\Cl(V_1 \dirsum V_2)$ linear unitary isomorphism
\begin{align}
\F(L_1) \tensor \F(L_2) &\xrightarrow{\psi} \F(L_1 \dirsum L_2) \notag \\
w_1 \tensor w_2 &\mapsto w_1 \wedge w_2, \notag
\end{align}
with the $\Cl(V_1 \dirsum V_2)$ action on $\F(L_1) \tensor \F(L_2)$ defined for $v_1 \in V_1$, $v_2 \in V_2$ by the self-adjoint Clifford map $\phi_{\tensor} \colon V_1 \dirsum V_2 \rightarrow \F(L_1) \tensor \F(L_2)$, defined using the self-adjoint Clifford maps $\phi_1 \colon V_1 \rightarrow \B(\F(L_1))$, $\phi_2 \colon V_2 \rightarrow \B(\F(L_2))$ for the Fock representations on $\F(L_1)$, $\F(L_2)$, by
\[
\phi_{\tensor} (v_1 + v_2) (w_1 \tensor w_2) = (\phi_1 (v_1) (w_1)) \tensor w_2 + (-1) ^{\partial w_1} w_1 \tensor \phi_2 (v_2) (w_2). \notag
\]
\end{lem}
\begin{proof}
To show that $\psi$ is an isometry, it suffices to evaluate it on tensor product and exterior algebra decomposables made from wedge products of $w_{1 i} \in L_1$ and $w_{2 j} \in L_2$.  Let $w_1 = w_{1 1} \wedge \dots \wedge w_{1 p}$, $w_2 = w_{2 1} \wedge \dots \wedge w_{2 q}$, and similarly for $w_1' \in \Lambda^p L_1$, $w_2' \in \Lambda^q L_2$. Then
\begin{align}
\langle w_1 \tensor w_2 , w_1' \tensor w_2' \rangle &= \langle w_1, w_1' \rangle \langle w_2, w_2' \rangle, \text{ and} \notag \\
\langle w_1 \wedge w_2 , w_1' \wedge w_2' \rangle &= \Det \left[ \langle (w_1 \wedge w_2)_k, (w_1 \wedge w_2)_l \rangle \right] \notag \\
&=
\left[
\begin{matrix}
 \langle w_{1 i}, w_{1 i'}' \rangle & 0 \\
 0 & \langle w_{2 j}, w_{2 j'}' \rangle
\end{matrix}
\right] \notag \\
&= \langle w_1, w_1' \rangle \langle w_2, w_2' \rangle, \notag
\end{align}
since $L_1 \perp L_2$.  It's injective on decomposables because $w_1 \wedge w_2 \Rightarrow (w_1 = 0 \text{ or } w_2 = 0) \Rightarrow w_1 \tensor w_2 = 0$.  It's surjective on decomposables because it has an inverse, defined on decomposables of $\F(L_1 \dirsum L_2)$ ordered so that elements of $L_1$ precede those of $L_2$ in the wedge product, as $w_1 \wedge w_2 \mapsto w_1 \tensor w_2$.  Being a unitary isomorphism on decomposables, by continuity it's a unitary isomorphism.

For Clifford linearity, denote the self-adjoint Clifford map for the Fock representation on $\F(L_1 \dirsum L_2)$ by $\phi_{\dirsum}$.  Then
\begin{align}
\psi (\phi_{\tensor} (v_1 + v_2) (w_1 \tensor w_2)) &= \psi ((\phi_1 (v_1) (w_1)) \tensor w_2 + (-1) ^{\partial w_1} w_1 \tensor \phi_2 (v_2) (w_2)) \notag \\
&= (\phi_1 (v_1) (w_1)) \wedge w_2 + (-1) ^{\partial w_1} w_1 \wedge \phi_2 (v_2) (w_2) \notag
\end{align}
\begin{align}
\phi_{\dirsum} (v_1 + v_2) \psi (w_1 \tensor w_2) &= \phi_{\dirsum} (v_1 + v_2) (w_1 \wedge w_2) \notag \\
&= \phi_{\dirsum} (v_1) (w_1 \wedge w_2) + \phi_{\dirsum} (v_2) (w_1 \wedge w_2) \notag \\
&= \sqrt{2} \thinspace P_{L_1 \dirsum L_2} (v_1) \wedge (w_1 \wedge w_2) \notag \\
 & + \sqrt{2} \thinspace P_{L_1 \dirsum L_2} (v_1) \contract (w_1 \wedge w_2) \notag \\
 & + \sqrt{2} \thinspace P_{L_1 \dirsum L_2} (v_2) \wedge (w_1 \wedge w_2) \notag \\
 & + \sqrt{2} \thinspace P_{L_1 \dirsum L_2} (v_2) \contract (w_1 \wedge w_2) \notag \\
&= (\sqrt{2} \thinspace P_{L_1} (v_1) \wedge w_1) \wedge w_2 + (\sqrt{2} \thinspace P_{L_1} (v_1) \contract w_1) \wedge w_2 \notag \\
 & + (-1)^{\partial w_1} w_1 \wedge (\sqrt{2} \thinspace P_{L_2} (v_2) \wedge w_2) \notag \\
 & + (-1)^{\partial w_1} w_1 \wedge (\sqrt{2} \thinspace P_{L_2} (v_2) \contract w_2) \notag \\
&= (\phi_1 (v_1) (w_1)) \wedge w_2 + (-1) ^{\partial w_1} w_1 \wedge \phi_2 (v_2) (w_2). \notag
\end{align}
\end{proof}

\begin{lem}\label{l-fock-spac-id-tens-intw}
\index{intertwiner!tensor product with identity}
(The Tensor Product of an Intertwiner with the Identity).
Given real Hilbert spaces $V_1$, $V_2$ and Lagrangian subspaces $L_1 \subset H_1$, $L_2, L_2' \subset H_2$ of their complexifications, and a $\Cl(V_2)$ linear unitary isomorphism $\theta \colon \F(L_2) \rightarrow \F(L_2')$, the map $\ident \tensor \theta \colon \F(L_1) \tensor \F(L_2) \rightarrow \F(L_1) \tensor \F(L_2')$ is a $\Cl(V_1 \dirsum V_2)$ linear unitary isomorphism.  If $\theta$ preserves or reverses the $\ZZ_2$ grading, so does $\ident \tensor \theta$.  If $\theta$ is grade preserving, the analogous lemma is true for $\theta \tensor \ident$.
\end{lem}
\begin{proof}
On decomposables $w \tensor x, y \tensor z \in \F(L_1) \tensor \F(L_2)$, since $\theta$ is unitary,
\begin{align}
\langle (\ident \tensor \theta) (w \tensor x), (\ident \tensor \theta) (y \tensor z) \rangle &= \langle x \tensor \theta (w) , z \tensor \theta (y) \rangle \notag \\
&= \langle x, z \rangle \langle \theta (w), \theta (y) \rangle \notag \\
&= \langle x, z \rangle \langle w, y \rangle \notag \\
&= \langle x \tensor w, z \tensor y \rangle, \notag
\end{align}
so $\ident \tensor \theta$ is unitary.  It is an isomorphism on decomposables since $\theta$ is, and thus by continuity is an isomorphism.  To see that it preserves or reverses the $\ZZ_2$ grading as $\theta$ does, let $k \in \ZZ_2$ be such that for $w_2 \in \F(L_2)$, $\partial (\theta (w_2)) = \partial w_2 + k$.  For $w_1 \in \F(L_1)$, $w_2 \in \F(L_2)$:
\begin{align}
\partial ((\ident \tensor \theta) (w_1 \tensor w_2)) &= \partial (w_1 \tensor \theta (w_2)) \notag \\
&= \partial (w_1) + \partial (\theta (w_2)) \notag \\
&= \partial w_1 + \partial w_2 + k \notag \\
&= \partial (w_1 \tensor w_2) + k. \notag
\end{align}
The Clifford linearity of $\ident \tensor \theta$ follows from that of $\theta$: for $v_1 \in V_1$, $v_2 \in V_2$, $w_1 \in \F(L_1)$, $w_2 \in \F(L_2)$, using the definition and notation in \ref{l-fock-spac-sum-lagr}, with analogous notation for ${}'$, in which $\phi_1' = \phi_1$,
\begin{align}
(\ident \tensor \theta) (\phi_{\tensor} (v_1 + v_2) (w_1 \tensor w_2)) &= (\ident \tensor \theta) ((\phi_1 (v_1) (w_1)) \tensor w_2 \notag \\
& + (-1) ^{\partial w_1} w_1 \tensor \phi_2 (v_2) (w_2)) \notag \\
&= \phi_1 (v_1) (w_1) \tensor \theta (w_2) + (-1) ^{\partial w_1} w_1 \tensor \theta (\phi_2 (v_2) (w_2)) \notag \\
&= \phi_1' (v_1) (w_1) \tensor \theta (w_2) + (-1) ^{\partial w_1} w_1 \tensor \phi_2' (v_2) (\theta (w_2)) \notag \\
&= \phi_{\tensor}' (v_1 + v_2) ((\ident \tensor \theta) (w_1 \tensor w_2)). \notag
\end{align}
\end{proof}

\begin{prop}\label{p-bg-cons-stab}
\index{bundle gerbe!construction!stability}
(The Bundle Gerbe Construction is Stable).
Given an oriented smooth vector bundle $E \rightarrow M$ of even rank with fiberwise inner product, $M$ a compact connected orientable smooth manifold with Riemannian metric, as in definition \ref{d-bg-func}, letting $I_2 = M \cross \RR^2 \rightarrow M$ be the trivial real vector bundle of rank 2 over $M$, with standard orientation and fiberwise inner product, the bundle gerbe $G(E \dirsum I_2)$ is stably isomorphic to the bundle gerbe $G(E)$.  By induction, the analogous statement for $I_{2k}$, $k \in \NN$, also is true.
\end{prop}
\begin{proof}
The proof is in four main stages, described by what they concentrate on:
\begin{enumerate}
 \item We demonstrate the continuity of the inclusion $L\SO(E) \rightarrow L\SO(E \dirsum I_2)$. \label{it-lsoe-incl-cont}
 \item We demonstrate the continuity of the inclusion $\Lagr_{res, n} \rightarrow \Lagr_{res, n + 2}$ induced by a choice of $L_2 \in \Lagr_{res, 2}$. \label{it-lagr-incl-cont}
 \item We demonstrate the continuity of the corresponding inclusion $\zeta \colon Y_n \rightarrow Y_{n + 2}$ and reduction of what is to be proved to the fact that $((\zeta^{[2]})^* P_{n+2}, Y_n, LM)$ is isomorphic to $(P_n, Y_n, LM) = G(E)$. \label{it-y-incl-cont-redu-prob}
 \item To demonstrate that fact, we define a map from $\theta \in T(L_{n, 1}, L_{n, 2})$ to $\phi \in T(L_{n, 1} \dirsum L_2, L_{n, 2} \dirsum L_2)$, and show that it is continuous, giving an isomorphism of principal $\UU(1)$ bundles. \label{it-thet-to-phi}
\end{enumerate}
\textbf{Proof Part \ref{it-lsoe-incl-cont}.}
For this proof we consider the bundle gerbe construction functor $G$ from vector bundles of varying even rank.  We take $n$ as the rank of the vector bundle $E$ that we start with, and look first at the inclusion
\begin{align}
\SO(E) \cross_M (M \cross \SO(2)) &\rightarrow \SO(E \dirsum I_2) \notag \\
(\rho, (x, \tau)) &\mapsto \rho \dirsum \tau, \notag
\end{align}
for $x \in M$, $\rho \colon \RR^n \rightarrow E_x$ an orientation-preserving linear isometry, $\tau \in \SO(2)$ hence $\tau \colon \RR^2 \rightarrow \RR^2$, $\rho \dirsum \tau \colon \RR^{n+2} \cong \RR^n \dirsum \RR^2 \rightarrow E_x \dirsum \RR^2$.  This inclusion is right $\SO(n)$-equivariant, acting trivially on the product bundle of the domain, and for the codomain using the continuous map $\SO(n) \cross \SO(2) \rightarrow \SO(n+2)$, $(\sigma, \tau) \mapsto \sigma \dirsum \tau$.  The displayed inclusion is continuous as may be seen from \citet[page~17]{Poor07}, using local trivializations of $E \dirsum I_2$ built from those for $E$ and $I_2$, local trivializations of $\SO(E)$ built from those of $E$, of $\SO(E \dirsum I_2)$ from $E \dirsum I_2$, and local trivializations for the fiber product built from those for the factors.  In these terms, continuity of the inclusion reduces to continuity of the map $\SO(n) \cross \SO(2) \rightarrow \SO(n + 2)$. Continuity of the following inclusion can be seen similarly but more easily, resulting from continuity of the map $\SO(n) \rightarrow \SO(n) \cross \SO(2)$ given by setting the second component to $\ident$:
\begin{align}
\SO(E) &\rightarrow \SO(E) \cross_M (M \cross \SO(2)) \notag \\
\rho &\mapsto (\rho, (x, \ident)); \notag
\end{align}
which, also, is right $\SO(n)$-equivariant.  Composing the inclusions and looping gives the is right $L\SO(n)$-equivariant continuous inclusion
\[
L\SO(E) \rightarrow L\SO(E \dirsum I_2).
\]

\textbf{Proof Part \ref{it-lagr-incl-cont}.}
For even $m \in \NN$, let $\Lagr_m = \{ K \subset L^2 (S^1, \CC^m) \st \overline{K} = K^{\perp} \}$, and let $\Lagr_{res, m} \subset \Lagr_m$ be the polarization class of definition \ref{d-lagr-res}, replacing that definition's $n$ with the $m$ of this context.  Given a choice of $L_2 \in \Lagr_{res, 2}$, there is a left $L\SO(n)$-equivariant inclusion, for $L_n \in \Lagr_{res, n}$:
\begin{align}
\Lagr_{res, n} &\rightarrow \Lagr_{res, n+2} \notag \\
L_n & \mapsto L_n \dirsum L_2, \text{ equivalent to} \notag \\
\US_{res, n} &\rightarrow \US_{res, n+2} \notag \\
J_n & \mapsto J_n \dirsum J_2. \notag
\end{align}
Continuity of each map is equivalent to that of the other, since the topology for each $\Lagr_{res}$ comes from that of the corresponding $\US_{res}$ (see definition \ref{d-ores-uvj-homo}).
Each $\US_{res}$ gets its topology from the corresponding $\Orth_{res} / \UU(V_J)$, and thus ultimately from $\Orth_{res}$. Going through these last two steps, since the topology of each $\US_{res}$ is independent of the particular $J$ used to define it (see lemma \ref{l-ures-top-not-dep-j}), we choose $J$ convenient for showing continuity at $J_n$:  for $\US_{res, n}$, choose $J_n$, and for $\US_{res, n+2}$, choose $J_n \dirsum J_2$, obtaining the following equivariant homeomorphisms (see definition \ref{d-ores-ures-lagr-res} and proposition \ref{p-uv-ov}):
\begin{align}
\US_{res, n} &\cong \Orth_{res, n} / \UU(V_{J_n}) \notag \\
J_n' = g_n J_n g_n^{-1} &\mapsto g_n \UU(V_{J_n}) \notag \\
\US_{res, n + 2} &\cong \Orth_{res, n + 2} / \UU(V_{J_n \dirsum J_2}) \notag \\
J_{n + 2}' = g_{n + 2} (J_n \dirsum J_2) g_{n + 2}^{-1} &\mapsto g_{n + 2} \UU(V_{J_n \dirsum J_2}). \notag
\end{align}
Continuity of $J_n' \mapsto J_n' \dirsum J_2$ at $J_n$ is thus equivalent to continuity of $g_n \UU(V_{J_n}) \rightarrow (g_n \dirsum \ident) \UU(V_{J_n \dirsum J_2})$ at $\ident \in \UU(V_{J_n})$, and continuity of that map in turn follows from continuity of the following $\Orth_{res, n}$-equivariant map that maps $\UU(V_{J_n}) \rightarrow \UU(V_{J_n \dirsum J_2})$:
\begin{align}
\Orth_{res, n} &\rightarrow \Orth_{res, n + 2} \notag \\
g_n &\mapsto g_n \dirsum \ident, \notag
\end{align}
This is continuous using the norms $\norm{}_J$ of proposition \ref{p-ores-top-grp}, with temporarily added subscripts $n$, $n + 2$.  To start with,
\begin{align}
\norm{(g_n \dirsum \ident) - (\ident \dirsum \ident)}_{J_n \dirsum J_2} &= \norm{(g_n - \ident) \dirsum 0}_{n + 2} + \norm{[J_n \dirsum J_2, (g_n - \ident) \dirsum 0]}_{2, n + 2}. \notag
\end{align}
Since for $v_n \in V_n$ and $v_2 \in V_2$, $((g_n - \ident) \dirsum 0) (v_n + v_2) = (g_n - \ident) (v_n) \dirsum 0$, and $V_{n + 2} = V_n \dirsum V_2$ is an orthogonal direct sum, $\norm{(g_n - \ident) \dirsum 0}_{n + 2} = \norm{g_n - \ident}_n$.  Letting $v_n$, $v_2$ represent elements in orthonormal bases for $V_n$, $V_2$, as in the definition in lemma \ref{l-hsop} of $\norm{}_{2, n + 2}$, similar reasoning gives $\norm{[J_n \dirsum J_2, (g_n - \ident) \dirsum 0]}_{2, n + 2} = \norm{[J_n, g_n - \ident]}_{2, n}$.  Thus by making $g_n$ close to $\ident$ in $\Orth_{res, n}$, we can make $g_n \dirsum \ident$ close to $\ident \dirsum \ident$ in $\Orth_{res, n + 2}$.  Thus our $\Lagr_{res, n} \rightarrow \Lagr_{res, n + 2}$ is continuous.

\textbf{Proof Part \ref{it-y-incl-cont-redu-prob}.}
The therefore continuous inclusion
\begin{align}
L\SO(E) \cross \Lagr_{res, n} &\rightarrow L\SO(E \dirsum I_2) \cross \Lagr_{res, n+2} \text{ descends to} \notag \\
Y_n = L\SO(E) \cross_{L\SO(n)} \Lagr_{res, n} &\rightarrow L\SO(E \dirsum I_2) \cross_{L\SO(n)} \Lagr_{res, n+2}, \notag
\end{align}
a continuous map since a continuous equivariant map descends to a continuous map of the orbit spaces \citep[page~4]{tomD87}.  Using the same fact, considering the equivariance of the associated product quotient projection
\[
L\SO(E \dirsum I_2) \cross \Lagr_{res, n+2} \rightarrow L\SO(E \dirsum I_2) \cross_{L\SO(n+2)} \Lagr_{res, n+2}
\]
with respect to the $L\SO(n)$ action on its domain and the trivial action on its codomain, we have a continuous map
\begin{align}
L\SO(E \dirsum I_2) \cross_{L\SO(n)} \Lagr_{res, n+2} &\rightarrow L\SO(E \dirsum I_2) \cross_{L\SO(n+2)} \Lagr_{res, n+2} \notag \\
 &= Y_{n+2}, \text{ that when composed with} \notag \\
Y_n = L\SO(E) \cross_{L\SO(n)} \Lagr_{res, n} &\rightarrow L\SO(E \dirsum I_2) \cross_{L\SO(n)} \Lagr_{res, n+2} \text{ yields} \notag \\
\zeta \colon Y_n &\rightarrow Y_{n+2}, \notag
\end{align}
a continuous map over the identity of $LM$.

Thus we get the following diagram, where, as in definition \ref{d-bg-mor}, $\zeta^{[2]}$ denotes the map induced by $\zeta$ on the second fiber power of the spaces, and in which the leftmost bundle gerbe is $G(E)$ and the rightmost is $G(E \dirsum I_2)$:
\[
\begindc{\commdiag}[5]

\obj(16,30)[objPn]{$P_n$}
\obj(16,20)[objYYn]{$(Y_n)^{[2]}$}
\obj(36,30)[objpbP]{$(\zeta^{[2]})^{*} P_{n+2}$}
\obj(36,20)[objYYnc2]{$(Y_n)^{[2]}$}
\obj(51,20)[objYn]{$Y_n$}
\obj(51,10)[objLM]{$LM$}
\mor{objPn}{objYYn}{}
\mor{objpbP}{objYYnc2}{}
\mor{objYYn}{objYYnc2}{$=$}
\mor(39,21)(50,21){}
\mor(39,19)(50,19){}
\mor{objYn}{objLM}{}

\obj(86,30)[objPnp2]{$P_{n+2}$}
\obj(86,20)[objYYnp2]{$(Y_{n+2})^{[2]}$}
\obj(66,20)[objYnp2]{$Y_{n+2}$}
\obj(66,10)[objLMc2]{$LM$}
\mor{objPnp2}{objYYnp2}{}
\mor(81,21)(69,21){}
\mor(81,19)(69,19){}
\mor{objYnp2}{objLMc2}{}

\mor{objYn}{objYnp2}{$\zeta$}
\mor{objLM}{objLMc2}{$=$}

\enddc
\]
By lemma \ref{l-indu-bg-stab-isom}, the bundle gerbes $((\zeta^{[2]})^* P_{n+2}, Y_n, LM)$ and $G(E \dirsum I_2)$ are stably isomorphic.

We will now show that $((\zeta^{[2]})^* P_{n+2}, Y_n, LM)$, is isomorphic to $(P_n, Y_n, LM) = G(E)$, since the principal $\UU(1)$ bundles are isomorphic over the identity on $(Y_n)^{[2]}$ in a way that respects bundle gerbe multiplication.  Then by note \ref{n-stab-isom}, $G(E \dirsum I_2)$ and $G(E)$ will be stably isomorphic.

\textbf{Proof Part \ref{it-thet-to-phi}.}
An element $[p] \in P_n$ over $[\widetilde{\gamma}, L_{n, 1}, L_{n, 2}] \in (Y (n))^{[2]}$ over $\gamma \in LM$, identifing $(Y_n)^{[2]}$ with $L\SO(E) \cross_{L\SO(n)} (\Lagr_{res, n} \cross \Lagr_{res, n})$ as in proposition \ref{p-bg-p}, is
\[
[p] = [\widetilde{\gamma}, L_{n, 1}, L_{n, 2}, \theta], \notag
\]
with $\theta \in T(L_{n, 1}, L_{n, 2})$.  An element $[q] \in (\zeta^{[2]})^* P_{n+2}$ is
\[
[q] = ([\widetilde{\gamma}, L_{n, 1}, L_{n, 2}], [\widetilde{\gamma} \cross \ident, L_{n, 1} \dirsum L_2, L_{n, 2} \dirsum L_2, \phi]), \notag
\]
with $\phi \in T(L_{n, 1} \dirsum L_2, L_{n, 2} \dirsum L_2)$.  To construct a principal bundle isomorphism, construct $\phi$ from $\theta$ by obtaining intertwiners $\psi_{\phi, 1}$, $\psi_{\phi, 2}$ from lemma \ref{l-fock-spac-sum-lagr}, intertwiner $\ident \tensor \theta$ from lemma \ref{l-fock-spac-id-tens-intw}, and requiring that the following diagram commute.  The intertwiner arrows in the diagram are with respect to $\Cl(V_1 \dirsum V_2)$.
\begin{align}
\begindc{\commdiag}[5]
\obj(10,30)[objFL1]{$\F(L_{n, 1} \dirsum L_2) = \F(L_2 \dirsum L_{n, 1})$}
\obj(50,30)[objFL2]{$\F(L_2 \dirsum L_{n, 2}) = \F(L_{n, 2} \dirsum L_2)$}
\obj(10,10)[objFL]{$\F(L_2) \tensor \F(L_{n, 1})$}
\obj(50,10)[objFLA]{$\F(L_2) \tensor \F(L_{n, 1})$}
\mor{objFL1}{objFL2}{$\phi$}
\mor{objFL}{objFLA}{$\ident \tensor \theta$}
\mor{objFL1}{objFL}{$\psi_{\phi, 1}^{-1}$}[\atright, \solidarrow]
\mor{objFLA}{objFL2}{$\psi_{\phi, 2}$}[\atright, \solidarrow]
\enddc \label{eq-phi-def-via-thet}
\end{align}

Defining $\phi$ in terms of $\theta$ this way defines a $\UU(1)$-equivariant map $P_n \rightarrow P_{n+2}$ over the identity:
\[
[p] = [\widetilde{\gamma}, L_{n, 1}, L_{n, 2}, \theta] \mapsto [\widetilde{\gamma} \cross \ident, L_{n, 1} \dirsum L_2, L_{n, 2} \dirsum L_2, \phi] = [q], \notag
\]
well defined and continuous because it descends from the continuous map that follows, the cartesian product of two $L\SO(n)$-equivariant continuous maps:
\begin{align}
L\SO(E) \cross T_n &\rightarrow L\SO(E \dirsum I_2) \cross T_{n+2} \notag \\
p = (\widetilde{\gamma}, L_{n, 1}, L_{n, 2}, \theta) &\mapsto (\widetilde{\gamma} \cross \ident, L_{n, 1} \dirsum L_2, L_{n, 2} \dirsum L_2, \phi) = q, \notag
\end{align}
where continuity of all components of $q$ but the last, $\phi$, has already been shown.

To see the continuity of the last component of $q$, a function of the last three components of $p$, use local trivializations of proposition \ref{p-bg-t}, specialized.  Proposition \ref{p-std-fock-spac-bndl} allows us to choose for the local trivializations of the standard Fock space bundle $F$, from which local trivializations of the standard intertwiner bundle $T$ are built, the standard fiber $\F(L)$, for any choice of $L \in \Lagr_{res}$.  Let us choose $L_{n, 1}$ for $F(n)$ for $T(n)$, and $L_{n, 1} \dirsum L_2$ for $F(n + 2)$ for $T(n + 2)$.

As in \ref{eq-t-psi} and \ref{eq-t-psi-thet} of proposition \ref{p-bg-t}, we have box neighborhoods $V_{L_{n, 1}} \cross V_{L_{n, 2}}$ of $(L_{n, 1}, L_{n, 2}) \in \Lagr_{res, n} \cross \Lagr_{res, n}$, and $V_{L_{n, 1} \dirsum L_2} \cross V_{L_{n, 2} \dirsum L_2}$ of $(L_{n, 1} \dirsum L_2, L_{n, 2} \dirsum L_2) \in \Lagr_{res, n + 2} \cross \Lagr_{res, n + 2}$, and local trivializations of $T(n)$, $T(n + 2)$
\begin{align}
\Psi_n \colon T_{V_{L_{n, 1}} \cross V_{L_{n, 2}}} &\isomto (V_{L_{n, 1}} \cross V_{L_{n, 2}}) \cross T(L_{n, 1}, L_{n, 1}) \notag \\
\Psi_{n + 2} \colon T_{V_{L_{n, 1} \dirsum L_2} \cross V_{L_{n, 2} \dirsum L_2}} &\isomto (V_{L_{n, 1} \dirsum L_2} \cross V_{L_{n, 2} \dirsum L_2}) \cross T(L_{n, 1} \dirsum L_2, L_{n, 1} \dirsum L_2). \notag
\end{align}
\begin{align}
\begindc{\commdiag}[5]
\obj(10,30)[objFL1]{$\F(L_{n, 1}')$}
\obj(80,30)[objFL2]{$\F(L_{n, 2}')$}
\obj(10,10)[objFL]{$\F(L_{n, 1})$}
\obj(80,10)[objFLA]{$\F(L_{n, 1})$}
\mor{objFL1}{objFL2}{$\theta$}
\mor{objFL}{objFLA}{$\widehat{\theta} = \pi_2 \circ \Psi_n (((L_{n, 1}', L_{n, 2}'), \theta))$}
\mor{objFL}{objFL1}{$\Theta_{n,1'}^{-1} = (\pi_2 \circ \Theta_{L_{n, 1}} (L_{n, 1}', \cdot))^{-1}$}[\atright, \solidarrow]
\mor{objFL2}{objFLA}{$\Theta_{n,2'} = \pi_2 \circ \Theta_{L_{n, 2}} (L_{n, 2}', \cdot)$}[\atright, \solidarrow]
\enddc \notag
\end{align}
\begin{align}
\begindc{\commdiag}[5]
\obj(10,30)[objFL1]{$\F(L_{n, 1}' \dirsum L_2)$}
\obj(80,30)[objFL2]{$\F(L_{n, 2}' \dirsum L_2)$}
\obj(10,10)[objFL]{$\F(L_{n, 1} \dirsum L_2)$}
\obj(80,10)[objFLA]{$\F(L_{n, 1} \dirsum L_2)$}
\mor{objFL1}{objFL2}{$\phi$}
\mor{objFL}{objFLA}{$\widehat{\phi} = \pi_2 \circ \Psi_{n + 2} (((L_{n, 1}' \dirsum L_2, L_{n, 2}' \dirsum L_2), \phi))$}
%\mor{objFL}{objFL1}{$\Theta_{n+2,1'} = (\pi_2 \circ \Theta_{L_{n, 1} %\dirsum L_2} (L_{n, 1}' \dirsum L_2, \cdot))^{-1}$}[\atright, %\solidarrow]
%\mor{objFL2}{objFLA}{$\Theta_{n+2,2'} = \pi_2 \circ \Theta_{L_{n, 2} %\dirsum L_2} (L_{n, 2}' \dirsum L_2, \cdot)$}[\atright, %\solidarrow]
\mor(5,10)(5,30){$\Theta_{n+2,1'}^{-1} = (\pi_2 \circ \Theta_{L_{n, 1} \dirsum L_2} (L_{n, 1}' \dirsum L_2, \cdot))^{-1}$}[\atright, \solidarrow]
\mor(85,30)(85,10){$\Theta_{n+2,2'} = \pi_2 \circ \Theta_{L_{n, 2} \dirsum L_2} (L_{n, 2}' \dirsum L_2, \cdot)$}[\atright, \solidarrow]
\enddc \notag
\end{align}
The long-named $\Theta$, which give Clifford linear unitary isomorphisms, are as in notation \ref{n-std-fock-spac-bndl}; refer now to the proof of proposition \ref{p-std-fock-spac-bndl}.  We specialize them further.

Given the choices of $g_K = g_{L_{n, 1}}$, $g' = g_{L_{n, 1}'}$ for $\Theta_{n, 1'}$ and $g_K = g_{L_{n, 2}}$, $g' = g_{L_{n, 2}'}$ for $\Theta_{n, 2'}$, use the freedom of choice for $g_K$ and $g'$ as afforded by notation \ref{n-std-fock-spac-bndl-g-prim}, to choose $g_K = g_{L_{n, 1}} \dirsum \ident$, $g' = g_{L_{n, 1}'} \dirsum \ident$ for $\Theta_{n+2, 1'}$, and $g_K = g_{L_{n, 2}} \dirsum \ident$, $g' = g_{L_{n, 2}'} \dirsum \ident$ for $\Theta_{n+2, 2'}$.

Given that $g' = g_{L_{n, 1}'}$ for $\Theta_{n, 1'}$ is a continuous function of $L_{n, 1}' \in V \subset \Lagr_{res}$ with $V$ an open neighborhood of $L_{n, 1}$, so is $g' = g_{L_{n, 1}'} \dirsum \ident$ for $\Theta_{n+2, 1'}$ also a continuous function of $L_{n, 1}'$.  It is not necessary for the present purpose to define $\Theta_{L_{n, 1} \dirsum L_2}$ on an open set in $\Lagr_{res, n + 2}$ for its first argument; the use we will make of it goes through the present $V \dirsum L_2$.  The local trivializations we use for the domain of the function $\theta \mapsto \phi$ that we are showing continuous, are defined on open sets, but that's not necessary for those for the codomain.  The analogous statements hold for $\Theta_{n+2, 2'}$.

These choices have the effect that the intertwiners $\Theta_{n+2, 1'}$, $\Theta_{n+2, 2'}$ act as the identity on the $L_2$ portions of the Fock spaces.  Under the inverses of isomorphisms from lemma \ref{l-fock-spac-sum-lagr}, these local trivialization intertwiners for $n + 2$ become $\ident_{L_2}$ tensored with the corresponding intertwiners of the local trivializations for $n$.  That is, the following two diagrams commute:
\[
\begindc{\commdiag}[5]
\obj(10,30)[objFL1]{$\F(L_2 \dirsum L_{n, 1}') = \F(L_{n, 1}' \dirsum L_2)$}
\obj(60,30)[objFL2]{$\F(L_{n, 1} \dirsum L_2) = \F(L_2 \dirsum L_{n, 1})$}
\obj(10,10)[objFL]{$\F(L_2) \tensor \F(L_{n, 1}')$}
\obj(60,10)[objFLA]{$\F(L_2) \tensor \F(L_{n, 1})$}
\mor{objFL1}{objFL2}{$\Theta_{n+2, 1'}$}
\mor{objFL}{objFLA}{$\ident \tensor \Theta_{n, 1'}$}
\mor{objFL1}{objFL}{$\psi_{n + 2, 1'}^{-1}$}[\atright, \solidarrow]
\mor{objFLA}{objFL2}{$\psi_{n + 2, 1}$}[\atright, \solidarrow]
\enddc \notag
\]
\[
\begindc{\commdiag}[5]
\obj(10,30)[objFL1]{$\F(L_2 \dirsum L_{n, 2}') = \F(L_{n, 2}' \dirsum L_2)$}
\obj(60,30)[objFL2]{$\F(L_{n, 1} \dirsum L_2) = \F(L_2 \dirsum L_{n, 1})$}
\obj(10,10)[objFL]{$\F(L_2) \tensor \F(L_{n, 2}')$}
\obj(60,10)[objFLA]{$\F(L_2) \tensor \F(L_{n, 1})$}
\mor{objFL1}{objFL2}{$\Theta_{n+2, 2'}$}
\mor{objFL}{objFLA}{$\ident \tensor \Theta_{n, 2'}$}
\mor{objFL1}{objFL}{$\psi_{n + 2, 2'}^{-1}$}[\atright, \solidarrow]
\mor{objFLA}{objFL2}{$\psi_{n + 2, 1}$}[\atright, \solidarrow]
\enddc \notag
\]
From these diagrams and \ref{eq-phi-def-via-thet} defining $\phi$ from $\theta$, we have the following commutative diagram with leftmost and rightmost vertices equal:
\begin{align}
\begindc{\commdiag}[5]
\obj(10,30)[objNP21R]{$\F(L_{n,1} \dirsum L_2)$}
\obj(32,30)[objNP21P]{$\F(L_{n,1}' \dirsum L_2)$}
\obj(53,30)[objNP22P]{$\F(L_{n,2}' \dirsum L_2)$}
\obj(75,30)[objNP21L]{$\F(L_{n,1} \dirsum L_2)$}
\obj(10,10)[objTNP21R]{$\F(L_2) \tensor \F(L_{n,1})$}
\obj(32,10)[objTNP21P]{$\F(L_2) \tensor \F(L_{n,1}')$}
\obj(53,10)[objTNP22P]{$\F(L_2) \tensor \F(L_{n,2}')$}
\obj(75,10)[objTNP21L]{$\F(L_2) \tensor \F(L_{n,1})$}
\mor{objNP21R}{objTNP21R}{$\psi_{n+2,1}^{-1}$}
\mor{objNP21P}{objTNP21P}{$\psi_{\phi,1}^{-1}$}
\mor{objTNP22P}{objNP22P}{$\psi_{\phi,2}$}
\mor{objTNP21L}{objNP21L}{$\psi_{n+2,1}$}
\mor{objNP21R}{objNP21P}{$\Theta_{n+2,1'}^{-1}$}
\mor{objNP21P}{objNP22P}{$\phi$}
\mor{objNP22P}{objNP21L}{$\Theta_{n+2,2'}$}
\mor{objTNP21R}{objTNP21P}{$\ident \tensor \Theta_{n,1'}^{-1}$}[\atright, \solidarrow]
\mor{objTNP21P}{objTNP22P}{$\ident \tensor \theta$}[\atright, \solidarrow]
\mor{objTNP22P}{objTNP21L}{$\ident \tensor \Theta_{n,2'}$}[\atright, \solidarrow]
\enddc \notag
\end{align}
The bottom edge is $\ident \tensor \widehat{\theta}$ and the top edge is $\widehat{\phi}$.  The maps $\widehat{\theta} \colon \F(L_{n, 1}) \rightarrow \F(L_{n, 1})$, $\ident \tensor \widehat{\theta} \colon \F(L_2) \tensor \F(L_{n, 1}) \rightarrow \F(L_2) \tensor \F(L_{n, 1})$, and $\widehat{\phi} \colon F(L_{n, 1} \dirsum L_2) \rightarrow F(L_{n, 1} \dirsum L_2)$ are all Clifford linear unitary intertwiners, elements of $\UU(1)$ torsors.  The map of fixed $\UU(1)$ torsors $\widehat{\theta} \mapsto \ident \tensor \widehat{\theta} \mapsto \widehat{\phi}$ is $\UU(1)$ equivariant and thus is continuous.  (In fact, in terms of the canonical isomorphisms of these torsors with $\UU(1)$, it maps $1 \mapsto 1 \mapsto 1$, and hence is the identity.)

The principal bundle isomorphism is thus continuous.  It respects bundle gerbe multiplication because that multiplication is just composition of intertwiners.  Composition of two $\theta$'s before applying the isomorphism and composition of two $\phi$'s after applying the isomorphism give equal results.
\end{proof}

\chapter{THE CLIFFORD ALGEBRA MODULE BUNDLE}\label{c-clif-alg-modu-bndl}

We have a Clifford algebra bundle $\Cl(LE) \rightarrow LM$, a polarization class bundle $Y \xrightarrow{\pi} LM$, and a Fock space bundle $FY \rightarrow Y$ that is a $\pi^{*} \Cl(LE)$ module bundle.  Given our bundle gerbe $(P, Y, LM)$ and a trivialization of it via a principal $\UU(1)$ bundle $R \rightarrow Y$ with $\delta(R)$ isomorphic to $P$, we construct from $FY$ and $R$ an irreducible $\Cl(LE)$ module bundle, a spinor bundle $S$ over $LM$.

Before that are preliminary sections on things used in the proof: one on a kind of tensor product such as $FY \tensor_{\UU(1)} R$, and one on inverse limits.

\section{${}^{*}$-Representation Tensor Products}\label{s-star-rep-tens-prod}

This section is used for tensor products of Fock spaces, which are Hilbert spaces, with $\UU(1)$ torsors, and tensor products of Fock space bundles with principal $\UU(1)$ bundles.  A Fock representation of a Clifford algebra on a Fock space gives rise to a unitarily equivalent representation of the Clifford algebra on the tensor product of the Fock space with a $\UU(1)$ torsor.  Thus we speak of tensor products of ${}^{*}$-representations or of ${}^{*}$-modules, with $\UU(1)$ torsors.

It does no harm to consider general ${}^{*}$-representations of Clifford algebras on Hilbert spaces to start with, though the application will be to Fock representations.  Linguistically, we shift from the term representations to the term modules when going from torsors to bundles, but we have defined both equivalently, as in definition \ref{d-c-star-clif-alg-rep}, just as we say that a map is an intertwiner or equivalently a Clifford linear unitary isomorphism, as in definition \ref{d-equi-rep}.

First let us make precise what is meant by the tensor product of a Hilbert space with a $\UU(1)$ torsor.  Before the definition, let us say that generally speaking, it is a Hilbert space isomorphic, though there isn't a preferred isomorphism, with the original Hilbert space.  The torsor doesn't so much affect the Hilbert space as it affects mappings between that Hilbert space and other Hilbert spaces.  A more specific note is that unlike with the tensor product of two modules, there is no need for sums of decomposables.  The most general element is one decomposable, with addition coming mostly from addition in the Hilbert space factor, since there is no addition on the torsor (written multiplicatively), which is only a group, not a ring.

\begin{defn}\label{d-clif-alg-rep-tens-tors}
\index{Clifford algebra!representation!tensor with U(1) torsor@tensor with $\UU(1)$ torsor}
(The Tensor Product of a Clifford Algebra Representation and a $\UU(1)$ Torsor).
Suppose $\pi \colon \Cl(V) \rightarrow \B(F)$ is a ${}^{*}$-representation on a Hilbert space $F$, and $T$ is a $\UU(1)$ torsor.  Letting $\UU(1)$ act on $F$ by $z f = z(1) f$, for $z \in \UU(1)$, $f \in F$, and $1 \in \CC$, define the Hilbert space
\[
F \tensor T = F \tensor_{\UU(1)} T = \{ (f, t) \in F \cross T \} / ((f,t) \sim (z f, z^{-1} t)) \text{ for } z \in \UU(1) \notag
\]
as a set, with scalar multiplication and addition given by:
\begin{align}
\alpha (f \tensor t) &= (\alpha f) \tensor t \notag \\
(f \tensor t) + (g \tensor u) &= (f + w g) \tensor t \notag \\
w &= \langle u, t \rangle \in \UU(1) \text{ such that } u = w t, \notag
\end{align}
for $\alpha \in \CC$, $f \tensor t, g \tensor u \in F \tensor T$, using the natural pairing of lemma \ref{l-u1t-cano-isom}. Give $F \tensor T$ the topology from the inner product defined by
\[
\langle f \tensor t, g \tensor u \rangle = \langle t, u \rangle (\langle f, g \rangle), \notag
\]
where $\langle t, u \rangle \in \UU(1)$ acts on $\langle f, g \rangle \in \CC$.  Using the identification $\UU(1) \subset \CC$, we can also write
\[
\langle f \tensor t, g \tensor u \rangle = \langle f, g \rangle \langle t, u \rangle, \notag
\]
using the product in $\CC$.  Then define
\begin{align}
\Psi \colon \phi \in \B(F) &\mapsto \phi_{\tensor} = \phi \tensor_{\UU(1)} \ident_T \in \B(F \tensor_{\UU(1)} T) \notag \\
\pi_{\tensor} = \Psi \circ \pi \colon \Cl(V) &\rightarrow \B(F \tensor T) \notag \\
\pi_{\tensor}(a)(f \tensor t) &= (\pi(a) (f)) \tensor t. \notag
\end{align}
Call $\pi_{\tensor}$ the tensor product representation.

Analogously, define the tensor product of a representation on the right and a $\UU(1)$ torsor on the left.
\end{defn}
If $T = \UU(1)$, $\langle t, u \rangle = t \overline{u}$, and in general $\langle t, u \rangle = \tau(u, t)$, where $\tau$ is the continuous translation function of lemma \ref{l-pb-tran-func-tors}.

\begin{lem}\label{l-clif-alg-rep-tens-tors}
\index{Clifford algebra!representation!tensor with U(1) torsor@tensor with $\UU(1)$ torsor}
($\pi_{\tensor}$ is a ${}^{*}$-Representation of $\Cl(V)$ on $F \tensor T$).
Suppose $\pi \colon \Cl(V) \rightarrow \B(F)$ is a ${}^{*}$-representation on a Hilbert space $F$, and $T$ is a $\UU(1)$ torsor.  The objects and maps in definition \ref{d-clif-alg-rep-tens-tors} are well-defined.  For $f \in F$, $t \in T$, $\norm{f \tensor t} = \norm{f}$.  The map $\Psi$ is a ${}^{*}$-morphism.  Given $t_0 \in T$, the map
\begin{align}
\psi_{t_0} \colon F &\rightarrow F \tensor_{\UU(1)} T \notag \\
\psi_{t_0} (f) &= f \tensor t_0, \notag
\end{align}
is an isomorphism of Hilbert spaces that intertwines with $\Psi$; i.e., for $\phi \in \B(F)$, $\Psi(\phi) \circ \psi_{t_0} = \psi_{t_0} \circ \phi$.  Suppose $\pi \colon \Cl(V) \rightarrow \B(F)$ is a ${}^{*}$-representation.  Then $\pi_{\tensor}$ is a ${}^{*}$-representation isomorphic to $\pi$ via intertwiner $\psi_{t_0}$, though the choice of intertwiner depends on the choice of $t_0$.

If $\pi$ is $\ZZ_2$ graded, so is $\pi_{\tensor}$, and the intertwiner preserves the grading.  Put simply, the torsor doesn't affect the grading.
\end{lem}
\begin{proof}
Scalar multiplication and addition are well-defined and addition is commutative because, for $\alpha \in \CC$, $w, z, z_1, z_2 \in \UU(1)$, $f, g \in F$, $t, u \in T$,
\begin{align}
\alpha ((z f) \tensor (z^{-1} t)) = (\alpha z f) \tensor (z^{-1} t) &= (z \alpha f) \tensor (z^{-1} t) = (\alpha f) \tensor t \notag \\
u = w t &\Rightarrow z_2^{-1} u = z_1 z_2^{-1} w z_1^{-1}t \text{ and} \notag \\
(z_1 f) \tensor (z_1^{-1}t) + (z_2 g) \tensor (z_2^{-1} u) &= (z_1 f + z_1 z_2^{-1} w z_2 g) \tensor (z_1^{-1} t) \notag \\
&= (z_1 (f + w g)) \tensor (z_1^{-1} t) \notag \\
&= (f + w g) \tensor t = (f \tensor t) + (g \tensor u); \notag \\
(f \tensor t) + (g \tensor u) &= (f + w g) \tensor t \notag \\
(g \tensor u) + (f \tensor t) &= (g + w^{-1} f) \tensor u \notag \\
&= (w (g + w^{-1} f)) \tensor (w^{-1} u) \notag \\
&= (w g + f) \tensor t. \notag
\end{align}
Addition is associative since $(h, v) \in F \cross T$, $x \in \UU(1)$, $v = x t$ $\Rightarrow$ $v = x w^{-1} u$:
\begin{align}
((f \tensor t) + (g \tensor u))+ (h \tensor v) &= ((f + w g) \tensor t) + (h \tensor v) \notag \\
&= (f + w g + x h) \tensor t \notag \\
(f \tensor t) + ((g \tensor u) + (h \tensor v)) &= (f \tensor t) + ((g + x w^{-1} h) \tensor u) \notag \\
&= (f + w g + w x w^{-1} h) \tensor t \notag \\
&= (f + w g + x h) \tensor t. \notag
\end{align}
$0 \tensor t$ is the unit for addition, and the other properties needed for a vector space hold.  We have
\begin{align}
\langle \alpha (f \tensor t), g \tensor u \rangle &= \langle (\alpha f) \tensor t, g \tensor u \rangle \notag \\
&= \alpha \langle f \tensor t, g \tensor u \rangle \notag
\end{align}
\begin{align}
\langle (f \tensor t) + (g \tensor u), h \tensor v \rangle &= \langle (f + w g) \tensor t, h \tensor v \rangle \notag \\
&= \langle f + w g, h \rangle \langle t, v \rangle \notag \\
&= (\langle f, h \rangle + \langle w g, h \rangle) \langle t, v \rangle \notag \\
&= \langle f, h \rangle \langle t, v \rangle + \langle w g, h \rangle \langle t, v \rangle \notag \\
&= \langle f \tensor t,  h \tensor v \rangle + \langle (w g) \tensor t, h \tensor v \rangle \notag \\
&= \langle f \tensor t,  h \tensor v \rangle + \langle  g \tensor u, h \tensor v \rangle \notag \\
\langle
\overline{\langle f \tensor t, g \tensor u \rangle} &= \overline{\langle f, g \rangle} \overline{\langle t, u \rangle} \notag \\
\langle f \tensor t, f \tensor t \rangle &= \langle f, f \rangle \ge 0, \text{} = 0 \text{ only when } f \tensor t = 0. \notag
\end{align}
The Hermitian property holds because $\langle, \rangle$ on $F$ is a Hermitian inner product, and for $z \in \UU(1)$, $t = z u \Leftrightarrow u = \overline{z} t$ so that $\overline{\langle t, u \rangle} = \overline{z} = \langle u, t \rangle$.

Completeness of $F \tensor_{\UU(1)} T$ follows from that of $F$ as follows.  Suppose $f_i \tensor t_i$ is a Cauchy sequence.  Each $f_i \tensor t_i = (f_i u_i) \tensor t_1$ for some $u_i \in \UU(1)$.  Since $f_i \tensor t_i - f_j \tensor t_j = (f_i u_i - f_j u_j) \tensor t_1$, we obtain a Cauchy sequence $f_i u_i$ in $F$, which converges to some $f \in F$.  Then $f_i \tensor t_i = (f_i u_i) \tensor t_1$ converges to $f \tensor t_1$.

$\Psi$ is a morphism of algebras.  For adjoints, look at
\begin{align}
\langle \phi_{\tensor} (x \tensor t), y \tensor u \rangle &= \langle \phi (x) \tensor t, y \tensor u \rangle \notag \\
 &= \langle \phi (x), y \rangle \langle t, u \rangle \notag \\
 &= \langle x, \phi^{*} y \rangle \langle t, u \rangle \notag \\
 &= \langle x \tensor t, \phi^{*} y \tensor u \rangle \notag \\
 &= \langle x \tensor t, (\phi^{*})_{\tensor} (y \tensor u) \rangle . \notag
\end{align}
Comparing with the definition of $(\phi_{\tensor})^{*}$, $\Psi$ is a ${}^{*}$-morphism, and thus $\pi_{\tensor} = \Psi \circ \pi$ is a ${}^{*}$-representation.  Since $t_0 = 1 \cdot t_0$, $\langle t_0, t_0 \rangle = 1$, $\psi_{t_0}$ preserves inner products and is thus injective.  It is also surjective, hence an isomorphism of Hilbert spaces, because each $f \tensor t = (f u) \tensor t_0$ for some $u \in \UU(1)$. The map $\psi_{t_0}$ is an intertwiner between $\pi$ and $\pi_{\tensor}$ because it intertwines with $\Psi$ and $\pi_{\tensor} = \Psi \circ \pi$; i.e., for $a \in \Cl(V)$, $f \in F$,
\[
(\Psi(\pi(a)) \circ \psi_{t_0}) (f) = (\pi(a)(f)) \tensor t_0 = (\psi_{t_0} \circ \pi(a)) (f).
\]
\end{proof}

\begin{note}\label{n-clif-alg-rep-tens-tors}
\index{Clifford algebra!representation!tensor with U(1) torsor@tensor with $\UU(1)$ torsor}
(Properties of the Tensor Product of a Representation and a Torsor).
Suppose $\pi \colon \Cl(V) \rightarrow \B(F)$ is a ${}^{*}$-representation on a Hilbert space $F$, and $T_1$, $T_2$ are $\UU(1)$ torsors.  Then $(F \tensor T_1) \tensor T_2$ is canonically isomorphic to $F \tensor (T_1 \tensor T_2)$, where the tensor products are of a representation and a torsor, except that $T_1 \tensor T_2$ is the tensor product of two torsors.  The isomorphism is a unitary intertwiner of representations, or in module language, Clifford linear and unitary.  We will generally identify the two resulting representations.

A tensor product with representation on the left and torsor on the right, is canonically isomorphic to the product in the reverse order.
\end{note}

\begin{lem}\label{l-clif-alg-rep-tens-tors-top}
\index{Clifford algebra!representation!tensor with U(1) torsor@tensor with $\UU(1)$ torsor}
(The Quotient and Inner Product Topologies on $F \tensor T$ are the Same).
Suppose given a Hilbert space $F$ and a $\UU(1)$ torsor $T$.  Then the quotient topology on the set on $F \tensor T$, not used in definition \ref{d-clif-alg-rep-tens-tors}, and the inner product topology, which is, are the same.
\end{lem}
\begin{proof}
Let $(F \tensor T)_{\sim}$ denote the space with the quotient topology and $F \tensor T$ the same set with the inner product topology, as already defined.  For $t_0 \in T$, define $\widetilde{\phi_{t_0}} \colon F \cross T \rightarrow F$, for $f \in F$, $t \in T$, by $\widetilde{\phi_{t_0}} (f, t) = \tau(t_0, t) f$, where $\tau$ is the continuous translation function of lemma \ref{l-pb-tran-func-tors}, with $\tau(t_0, t) = \langle t, t_0 \rangle$.  This continuous map is $\UU(1)$-equivariant for the tensor product action on the right and the trivial action on the left, and thus descends to a continuous map of the orbit spaces \citep[page~4]{tomD87}, $\phi_{t_0} \colon (F \tensor T)_{\sim} \rightarrow F$.  Then  as maps of sets, $\phi_{t_0}$ and $\psi_{t_0}$ are inverses.  Further, $\psi_{t_0} \circ \phi_{t_0} = \ident \colon (F \tensor T_{\sim} \rightarrow F \tensor T$, as the composition of continuous maps, is continuous.

Its inverse, $\phi_{t_0}^{-1} \circ \psi_{t_0}^{-1} = \ident \colon F \tensor T \rightarrow (F \tensor T)_{\sim}$, also is continuous, as follows.  Since $\psi$ is a homeomorphism, this is equivalent to continuity of $\phi_{t_0}^{-1}$, or openness of $\phi_{t_0}$; that given $V$ open in $(F \tensor T)_{\sim}$, $\phi_{t_0} (V)$ is open in $F$.  It suffices to take any $f \in \phi_{t_0} (V)$ and show that there is some open neighborhood $N$ of $f$ such that $N \subset \phi_{t_0} (V)$, or equivalently since $\phi_{t_0}$ is a bijection, $N \tensor t_0 \subset V$.

Since $V$ being open in $(F \tensor T)_{\sim}$ is equivalent to its inverse image under the quotient projection $\pi$ being open in $F \cross T$, there are open neighborhoods $N_f$ of $f$, $N_{t_0}$ of $t_0$ such that $N_f \cross N_{t_0} \subset \pi^{-1} (V)$, and in particular, $N_f \cross {t_0} \subset \pi^{-1} (V)$.  Applying $\pi$, $N_f \tensor t_0 \subset V$, so take $N = N_f$.
\end{proof}

Rather than speaking of bundles of representations, we talk about bundles of Clifford modules.
\begin{defn}\label{d-clif-alg-rep-tens-bndl}
\index{Clifford algebra!representation!tensor with principal U(1) bundle@tensor with principal $\UU(1)$ bundle}
(The Tensor Product of a Clifford Algebra Module bundle and a Principal $\UU(1)$ Bundle).
Suppose $CM \rightarrow X$ is a Clifford module bundle with standard fiber a Hilbert space, for the Clifford algebra bundle $CA \rightarrow X$, so that there is fiber bundle map $\Pi \colon CA \cross_X CM \rightarrow CM$ that on the fiber over each $x \in X$ is a ${}^{*}$-representation.  Also suppose that $PB \rightarrow X$ is a principal $\UU(1)$ bundle.

Define the Clifford module bundle $CM \tensor_{\UU(1)} PB = CM \tensor PB \rightarrow X$ as the quotient of the topological space $CM \cross_X PB$ by the same $\UU(1)$ action as in definition \ref{d-clif-alg-rep-tens-tors}, using the quotient topology.  Define the Hilbert space structure and Clifford algebra representation on each fiber as in that definition.

Analogously, define the tensor product of a Clifford algebra module bundle on the right and a principal $\UU(1)$ bundle on the left.
\end{defn}

\begin{lem}\label{l-clif-alg-rep-tens-bndl}
\index{Clifford algebra!representation!tensor with principal U(1) bundle@tensor with principal $\UU(1)$ bundle}
(The Tensor Product of a Clifford Algebra Module Bundle and a Principal $\UU(1)$ Bundle).
The objects and maps of definition \ref{d-clif-alg-rep-tens-bndl} are well-defined, the fibers of $CM \tensor PB$ have the same topology, Hilbert space structure and Clifford algebra representation as definition \ref{d-clif-alg-rep-tens-tors} would give them fiberwise, and their Clifford algebra representations cohere into $\Pi_{\tensor} \colon CA \cross_X (CM \tensor PB) \rightarrow CM \tensor PB$.
\end{lem}
\begin{proof}
Since the action for the tensor product quotient respects fibers, and lemma \ref{l-clif-alg-rep-tens-tors-top} implies that the quotient topology for each fiber is the same as that for the inner product, the results for Hilbert spaces and torsors carry through to the fibers here.  Also, the fiber product inherits respectively linear and $\UU(1)$-equivariant local trivializations from its factors, resulting in $\UU(1)$-equivariant local trivializations of the fiber product, that descend to the tensor product.  Likewise, the continuity of $\Pi$ implies continuity of the Clifford algebra action defined on the fiber product that acts trivially on the principal bundle factor.  This Clifford action, which is by linear maps on the Hilbert spaces, is equivariant for the tensor quotient $\UU(1)$ action, and so descends to the continuous fiber bundle map $\Pi_{\tensor}$, with fiberwise action on the tensor product.
\end{proof}

\begin{note}\label{n-clif-alg-rep-tens-bndl}
\index{Clifford algebra!representation!tensor with U(1) torsor@tensor with $\UU(1)$ torsor}
(Properties of the Tensor Product of a Representation and a Principal Bundle).
Analogously to note \ref{n-clif-alg-rep-tens-tors}, using the notation of definition \ref{d-clif-alg-rep-tens-bndl} but with two principal $\UU(1)$ bundles $PB_1$ and $PB_2$, $(CM \tensor PB_1) \tensor PB_2$ and $CM \tensor (PB_1 \tensor PB_2)$ are canonically isomorphic via an isomorphism that is Clifford linear and unitary on fibers.  We will generally identify the two resulting module bundles.

A tensor product with the Clifford algebra module bundle on the left and the principal $\UU(1)$ bundle on the right, is canonically isomorphic to the product in the reverse order.

The tensor product of a Clifford algebra module bundle with the product bundle base space $\cross \UU(1)$, is canonically isomorphic to the original Clifford algebra module bundle.
\end{note}

\section{Inverse Limit of a Functor}\label{s-inv-lim-func}

Since we will use an inverse limit, here is a definition and a lemma.

\begin{defn}\label{d-inv-lim}
\index{inverse limit}
\index{limit!inverse}
(The Inverse Limit of a Functor).
\citep[pages~75--78]{Frey64} Let $D \colon A \rightarrow B$ be a functor; then an inverse limit of $D$ is a pair $(I, \alpha_I)$ consisting of an object $I \in B$ and a natural transformation $\alpha_{I} \colon C_{I} \rightarrow D$, where $C_{I} \colon A \rightarrow B$ is the constant functor that maps every object to $I$, such that the following property holds.  (Recall that a constant functor maps every object to one object, and every morphism to the identity.)  For any constant functor $C_b \colon A \rightarrow B$ and natural transformation $\alpha_b \colon C_b \rightarrow D$, there is a unique natural transformation $\beta \colon C_b \rightarrow C_{I}$ such that $\alpha_{I} \circ \beta = \alpha_b$.
\end{defn}
An illustrative commutative diagram follows. Since $\beta(a_1) = \beta(a_2)$, we abuse notation to call it $\beta \in \Mor(b, I)$.  We only use the special case after the illustration.
\[
\begindc{\commdiag}[5]
\obj(-10,25)[obja1]{$a_1$}
\obj(-10,10)[obja2]{$a_2$}
\obj(5,25)[objCba1]{$b = C_b(a_1)$}
\obj(5,10)[objCba2]{$b = C_b(a_2)$}
\obj(30,25)[objCIa1]{$I = C_I(a_1)$}
\obj(30,10)[objCIa2]{$I = C_I(a_2)$}
\obj(50,25)[objDa1]{$D(a_1)$}
\obj(50,10)[objDa2]{$D(a_2)$}
\mor{obja1}{obja2}{$\phi$}
\mor{objCba1}{objCIa1}{$\beta = \beta(a_1)$}[\atleft,\dashArrow]
\mor{objCba2}{objCIa2}{$\beta = \beta(a_2)$}[\atleft,\dashArrow]
\mor{objCba1}{objCba2}{$\ident = C_b(\phi)$}
\mor{objCIa1}{objCIa2}{$\ident = C_I(\phi)$}
\mor{objCIa1}{objDa1}{$\alpha_I(a_1)$}
\mor{objCIa2}{objDa2}{$\alpha_I(a_2)$}
\mor{objDa1}{objDa2}{$D(\phi)$}
\cmor((7,27)(30,33)(48,27)) \pright(30,35){$\alpha_b (a_1)$}
\cmor((7,8)(30,2)(48,8)) \pright(30,0){$\alpha_b (a_2)$}
\enddc
\]

\begin{lem}\label{l-inv-lim-spec}
\index{inverse limit}
\index{limit!inverse}
(A Concrete Realization of the Inverse Limit of a Special Kind of Functor).
Suppose $A$ is a category in which for every $a_1, a_2 \in \Obj(A)$ there is exactly one morphism $m_{a_1, a_2} \in \Mor(a_1, a_2)$, that $B$ is the category $Set$, and that $D \colon A \rightarrow B$ is a functor.  Then using the convention that $x_{\alpha} \in D(\alpha)$, defining
\begin{align}
I &= \varprojlim_{a \in A} D(a) \notag \\
 &= \{ (x_a) \in \Pi_{a \in A} D(a) \st \forall a_1, a_2 \in A, x_{a_2} = D(m_{a_1, a_2}) (x_{a_1}) \} \notag \\
\alpha_{I} (a) &= (\pi_a)_{|I}, \text{ or just } \pi_a \text{ } \forall a \in A, \notag
\end{align}
$(I, \alpha_{I})$ is an inverse limit of $D$, constructed without arbitrary choice, and all the $m_{a_1, a_2}$, $D(m_{a_1, a_2})$, and $\alpha_I (a)$ are isomorphisms.
\end{lem}
\begin{proof}
$\alpha_{I}$ is a natural transformation since for every $a_1, a_2 \in A$, $D(m_{a_1, a_2}) \circ \alpha_{I} (a_1) = \alpha_{I} (a_2) \circ \ident_{I}$; i.e., $D(m_{a_1, a_2}) \circ \pi_{a_1} = \pi_{a_2}$, or $x_{a_2} = D(m_{a_1, a_2}) (x_{a_1})$ for every $x_{a_1} \in \pi_{a_1} (I)$, true by definition of $I$.

Suppose given any constant functor $C_b$, which maps every object of $A$ to $b \in B$ and maps every morphism of $A$ to $\ident_b$, and any natural transformation $\alpha_b$, a collection of morphisms indexed by $a \in A$, $\alpha_b(a) \colon C_b(a) = b \rightarrow D(a) \in B$, such that for $a_1, a_2 \in A$, $D(m_{a_1, a_2}) \circ \alpha_b (a_1) = \alpha_b (a_2) \circ \ident_b$.  Given such a constant functor and natural transformation, by the universal property for products, there is a unique map $\beta \colon b \rightarrow \Pi_{a \in A} D(a)$ such that for every $a \in A$, $\alpha_I (a) \circ \beta = \pi_a \circ \beta = \alpha_b (a)$, so that $\beta$ is the one morphism of the desired natural transformation, provided we show that the image of $\beta$ is in $I$, which is to say, for every $a_1, a_2 \in A$,  $\pi_{a_2} \circ \beta = D(m_{a_1, a_2}) \circ \pi_{a_1} \circ \beta$.  But that is equivalent to $\alpha_I (a_2) \circ \beta = \alpha_b (a_2) = D(m_{a_1, a_2}) \circ \alpha_b (a_1) = D(m_{a_1, a_2}) \circ \alpha_I (a_1) \circ \beta$, with the middle equality due to the fact that $\alpha_b$ is a natural transformation.  Thus $(I, \alpha_I)$ is an inverse limit of $D$.

For any objects $a_1, a_2$ in $A$, there is exactly one morphism $\ident_{a_1, a_1} \in \Mor (a_1, a_1)$ and exactly one morphism $\ident_{a_2, a_2} \in \Mor (a_2, a_2)$.  Thus, given $m_{a_1, a_2} \in \Mor(a_1, a_2)$ and $m_{a_2, a_1} \in \Mor(a_2, a_1)$, necessarily $m_{a_1, a_2} \circ m_{a_2, a_1} = \ident_{a_1, a_1}$ and $m_{a_2, a_1} \circ m_{a_1, a_2} = \ident_{a_2, a_2}$; hence both $m_{a_1, a_2}$ and $m_{a_2, a_1}$ are isomorphisms.  Since a functor carries isomorphisms to isomorphisms, all the $D(m_{a_1, a_2})$ are isomorphisms.

To see that $\alpha_I (a_1)$ has a right inverse, define the constant functor $C_b$ by $b = D(a_1)$, the natural transformation $\alpha_b$ by $\alpha_b (a_1) = \ident$, and for every $a_2$, since in $A$ there is a morphism $m_{a_1, a_2}$, $\alpha_b (a_2) = D(m_{a_1, a_2})$.  Then reading off the top composition in the diagram, $\alpha_I (a_1) \circ \beta = \ident_{D(a_1)}$.  Thus also, $\alpha_I (a_1)$ is surjective.

To see that $\alpha_I (a_1)$ is injective, suppose $(x_{\alpha}), (y_{\alpha}) \in I$ and $x_{a_1} = \alpha_I (a_1) ((x_{\alpha})) = \alpha_I (a_1) ((y_{\alpha})) = y_{a_1}$.  Then by definition of $A$, for every $a_2$ there is a morphism $m_{a_1, a_2}$, and by definition of $I$, $x_{a_2} = D(m_{a_1, a_2}) (x_{a_1}) = D(m_{a_1, a_2}) (y_{a_1}) = y_{a_2}$.  Thus $(x_{\alpha}) = (y_{\alpha})$.

Since $\alpha_I (a_1)$ is a bijection, it is an isomorphism in $B = Set$.
\end{proof}

\section{The Clifford Module Bundle Construction}\label{s-clif-modu-bndl-cons}

Definition \ref{d-clif-alg-bndl} gives Clifford algebra bundle $\Cl(LE) \rightarrow LM$, \ref{d-pol-clas-bndl} a polarization class bundle $Y \xrightarrow{\pi} LM$, and \ref{d-fock-spac-bndl} a Fock space bundle $FY \rightarrow Y$ that is a $\pi^{*} \Cl(LE)$ module bundle as in lemma \ref{l-clle-act-fy}.  Proposition \ref{p-bg} gives the bundle gerbe $(P, Y, LM)$, and we are given as in definition \ref{d-bg-triv} a trivialization of it:  a principal $\UU(1)$ bundle $R \rightarrow Y$ and an isomorphism $\Phi \colon \delta(R) \rightarrow P$ that respects bundle gerbe multiplication.

We want an irreducible $\Cl(LE)$ module bundle, a spinor bundle $S$ over $LM$.  For a given $\gamma \in LM$, each $y \in Y_{\gamma}$ gives the module $FY_y$ for $\pi^{*} \Cl(LE)_y$, which we can identify with $(\Cl(LE))_{\gamma}$; but in general, we can't choose continuously one such $y$ for each $\gamma$.

Now, $FY \tensor R$ as in definition \ref{d-clif-alg-rep-tens-bndl}, like $FY$, is a $\pi^{*} \Cl(LE)$ module bundle over $Y$, but it differs from $FY$ in that $R$, because $\delta(R)$ is isomorphic to $P$, contains information related to Clifford linear unitary isomorphisms between fibers of $FY$.

We take advantage of the difference, not by choosing one $y \in Y$ and thence one fiber $(FY \tensor R)_y$ as the Clifford algebra module, but by using the information from $R$ to obtain a canonical isomorphism between any two fibers $(FY \tensor R)_{y_1}$ and $(FY \tensor R)_{y_2}$ with $y_1, y_2$ over $\gamma$.  We identify all these fibers over $\gamma$ via an inverse limit to get one Clifford algebra module $S_{\gamma}$ for $(\Cl(LE))_{\gamma}$; and all the $S_{\gamma}$ fit together into a fiber bundle over $LM$.

\begin{thm}\label{t-bg-triv-clif-alg-mod}
\index{bundle gerbe!triviality!Clifford module bundle}
\index{Clifford module bundle}
\index{Clifford algebra!module bundle}
(A Trivialization of the Bundle Gerbe Gives a Clifford Module Bundle).
From a trivialization $R \rightarrow Y$ with isomorphism $\Phi \colon \delta(R) \rightarrow P$, of the bundle gerbe $(P, p, Y, \pi, LM)$ of proposition \ref{p-bg}, we construct a fiber bundle $S \rightarrow LM$, each fiber of which is an irreducible module for the corresponding fiber of the Clifford algebra bundle $\Cl(LE)$ of definition \ref{d-clif-alg-bndl}, resulting in a continuous map $\Cl(LE) \cross_{LM} S \rightarrow S$.  That is, $S$ is a bundle of irreducible Clifford algebra modules over $LM$; and it is a Hilbert space bundle with local trivializations that are Clifford linear unitary isomorphisms on fibers.

The construction of $S$ is made without arbitrary choices.

Given another trivialization $R' \rightarrow Y$, $\Phi' \colon \delta(R') \rightarrow P$, there is a principal $\UU(1)$ bundle $Q \rightarrow LM$ such that $R' \cong R \tensor \pi^* Q$, and the Clifford module bundle $S'$ constructed using $R'$ differs from $S$ by $Q$; that is, $S' \cong S \tensor Q$, a homeomorphism over the identity that is a Clifford linear unitary isomorphism on fibers.
\end{thm}
\begin{proof}
Most of the work is in defining and proving properties of the canonical isomorphism we'll construct below, $\nu \colon \pi_1^{*} (FY \tensor R) \rightarrow \pi_2^{*} (FY \tensor R)$, where $\pi_1, \pi_2 \colon Y^{[2]} \rightarrow Y$ are the projections.  As in the following lemma \ref{l-nu-ev}, omitting irrelevant data for pullbacks and associated bundle constructions, for $(y_1, y_2) = ([\widetilde{\gamma}, L_1] ,[\widetilde{\gamma}, L_2]) \in Y^{[2]}$, $w_1 \in \F(L_1)$, $r_1 \in R_{y_1}$, and arbitrarily chosen $r_2 \in R_{y_2}$,
\begin{align}
(\pi_1^{*} (FY \tensor R))_{y_1, y_2} &\xrightarrow{\nu} (\pi_2^{*} (FY \tensor R))_{y_1, y_2} \notag \\
FY_{y_1} \tensor R_{y_1} &\rightarrow FY_{y_2} \tensor R_{y_2} \notag \\
w_1 \tensor r_1 &\mapsto \Phi(r_1 \tensor r_2^{*}) (w_1) \tensor r_2. \notag
\end{align}

To use bundle maps to define $\nu$, the following lemma \ref{l-nu-ev} defines one of the main pieces of $\nu$, $\ev \colon P \tensor \pi_1^{*} FY \rightarrow \pi_2^{*} FY$.  Other maps involved are two inverses of the canonical isomorphisms we will use from lemma \ref{l-u1pb-cano-isom}, temporarily named $\epsilon$ and $\eta$, which do the following on elements, given a principal $\UU(1)$ bundle $Q \rightarrow B$:
\begin{align}
\epsilon \colon Q &\rightarrow (B \cross \UU(1)) \tensor Q \notag \\
 q &\mapsto (b, 1) \tensor q, \text{ for all } q \in Q_b \notag \\
\eta \colon B \cross \UU(1) &\rightarrow Q \tensor Q^{*} \notag \\
(b, 1) &\mapsto q^{*} \tensor q, \text{ choosing any } q \in Q_b. \notag
\end{align}
We use the same name $\eta$ for the analogous canonical isomorphism with $Q$ a Clifford module bundle.  Other canonical isomorphisms, including these maps in their forward direction, are used without being named.  $\pi_1, \pi_2 \colon Y^{[2]} \rightarrow Y$ are the projections.  Then define $\nu$ as the composition:
\begin{align}
\pi_1^{*} (FY \tensor R) &\xrightarrow{\epsilon} (Y^{[2]} \cross \UU(1)) \tensor \pi_1^{*} (FY \tensor R) \notag \\
&\xrightarrow{\eta \tensor \ident} \pi_2^{*} R \tensor \pi_2^{*} R^{*} \tensor \pi_1^{*} (FY \tensor R) \notag \\
&\xrightarrow{canon.} \pi_1^{*} R \tensor \pi_2^{*} R^{*} \tensor \pi_1^{*} FY \tensor \pi_2^{*} R \notag \\
&\xrightarrow{\Phi \tensor \ident} P \tensor \pi_1^{*} FY \tensor \pi_2^{*} R \notag \\
&\xrightarrow{\ev \tensor \ident} \pi_2^{*} FY \tensor \pi_2^{*} R \notag \\
&\xrightarrow{canon.} \pi_2^{*} (FY \tensor R). \notag
\end{align}

All the maps composing $\nu$ are isomorphisms of fiber bundles, Clifford linear unitary isomorphisms on fibers, either mentioned in note \ref{n-clif-alg-rep-tens-bndl}, being principal $\UU(1)$ bundle isomorphisms tensored with $\ident$ on the Clifford module bundle part, or in the case of $\ev$, due to the following lemma \ref{l-nu-ev}, which we use now but state and prove later, after the end of this, the main line of the proof.

To fill in more steps on elements, but coalescing some of the bundle maps, omitting irrelevant pullback and associated bundle data,
\begin{align}
w_1 \tensor r_1 &\mapsto (y_1, y_2, 1) \tensor w_1 \tensor r_1 \notag \\
&\mapsto r_1 \tensor r_2^{*} \tensor w_1 \tensor r_2 \notag \\
&\mapsto \Phi(r_1 \tensor r_2^{*}) \tensor w_1 \tensor r_2 \notag \\
&\mapsto \Phi(r_1 \tensor r_2^{*}) (w_1) \tensor r_2. \notag
\end{align}

Now, $\nu$ gives a canonical Clifford linear unitary isomorphism between any two fibers of $FY \tensor R$ over $\gamma \in LM$, but to make an equivalence relation from this we need transitivity and symmetry of the relation:  over $Y^{[3]}$, the cocycle condition $\pi_{23}^{*} \nu \circ \pi_{12}^{*} \nu = \pi_{13}^{*} \nu$, and over $Y^{[2]}$, $\Delta_2^{11} \nu = \ident_{\pi_1^{*} (FY \tensor R)}$.  These combine to give $\pi_{21}^{*} \nu \circ \pi_{12}^{*} \nu = \pi_{11} \nu = \ident$.

For the cocycle condition, for $y_3 = [\widetilde{\gamma}, L_3]$, choosing arbitrarily some $r_3 \in R_{y_3}$, substituting the expression $\nu (w_1 \tensor r_1) = \Phi(r_1 \tensor r_2^{*}) (w_1) \tensor r_2$ into the expression $\nu (w_2 \tensor r_3) = \Phi(r_2 \tensor r_3^{*}) (w_2) \tensor r_3$,
\begin{align}
\pi_{23}^{*} \nu \circ \pi_{12}^{*} \nu (w_1 \tensor r_1) &= \Phi(r_2 \tensor r_3^{*}) (\Phi(r_1 \tensor r_2^{*}) (w_1)) \tensor r_3 \notag \\
&= \Phi(r_2 \tensor r_3^{*}) \circ \Phi(r_1 \tensor r_2^{*}) (w_1) \tensor r_3 \notag \\
&= \Phi(r_1 \tensor r_3^{*}) (w_1) \tensor r_3 \notag \\
&= \pi_{13}^{*} \nu (w_1 \tensor r_1), \notag
\end{align}
as before omitting irrelevant pullback and associated bundle data.  The penultimate equality shows the composition of two applications of $\ev$ turning into one application.  Letting $\Phi(r_1 \tensor r_2^{*})$ give $\phi_{12} \in T(L_1, L_2)$, $\Phi(r_2 \tensor r_3^{*})$ give $\phi_{23} \in T(L_2, L_3)$,  and $\Phi(r_1 \tensor r_3^{*})$ give $\phi_{13} \in T(L_1, L_3)$, the composition of applications of $\ev$ amounts to one application with the composition of intertwiners $\phi_{23} \circ \phi_{12}$.  Because $m_{\delta (R)} ((r_1 \tensor r_2^{*}) \tensor (r_2 \tensor r_3^{*})) = r_1 \tensor r_3^{*}$ by definition \ref{d-cano-triv-bg}, because $\Phi$ preserves bundle gerbe multiplication, and the multiplication $m_P$ for our bundle gerbe $(P, Y, LM)$ is given by composition of intertwiners of Fock spaces in reverse order, by note \ref{n-bg-triv-mult} and proposition \ref{p-bg-p}, $\phi_{23} \circ \phi_{12} = \phi_{13}$.  Thus the cocycle condition holds.

To get $\Delta_2^{11} \nu = \ident$, look at the expression for evaluation at a point and take the arbitrary element of $R_{y_1}$ as $r_1$.  It follows from the definition of multiplication for $\delta(R)$, which gives $m_{\delta} ((r_1 \tensor r_1)^{*} \tensor (r_1 \tensor r_1)^{*}) = (r_1 \tensor r_1)^{*}$, $\Phi$'s preserving the multiplication on the bundle gerbes it relates, definition \ref{d-bg-mult-id-inv} and lemma \ref{l-bg-mult-id-inv} concerning the uniqueness of the identity, and the fact that $\ident \in T(L_1, L_1)$ is an identity for composition and hence gives an identity for $m_P$, that $\phi_{11} = \ident$.

Arriving at this point with the properties of $\nu$ in hand, we use $\nu$ to give for every $(y_1, y_2)$ over $\gamma$, a unique Clifford linear unitary isomorphism $\nu_{12} \colon (FY \tensor R)_{y_1} \rightarrow (FY \tensor R)_{y_2}$, then use these to identify via an inverse limit, all the $(FY \tensor R)_y$ for $y$ over $\gamma$, giving one Clifford algebra module over $\gamma$.

To use lemma \ref{l-inv-lim-spec} let $A$ be the (abstract) category with $\Obj(A) = Y_{\gamma}$, the points $y$ over $\gamma$, and exactly one morphism $m_{12}$ for every pair of objects $(a_1, a_2) = (y_1, y_2)$.  Let the category $B = Set$, the category of sets.  The objects and morphisms we will actually use in $B$ come from forgetting $C^{*}$-Clifford algebra and Hilbert space structure.

Let the functor $D \colon A \rightarrow B$ on objects $y \in Y_{\gamma}$ be $y \mapsto (FY \tensor R)_y$ and on morphisms $m_{12} \mapsto \nu_{12}$.  The cocycle condition on the $\nu$'s is equivalent to one requirement for a functor, to preserve composition of morphisms.  The other requirement, that $D(\ident) = \ident$, is true since $\nu_{11} = \ident$.  Then define $S_{\gamma} = I$, a particular concrete realization of $\varprojlim_{a \in A} D(a)$.  The natural transformation $\alpha_I$ from $C_I$ to $D$ consists of the inverse limit projections, which are the projections $\pi_{\gamma, y} \colon S_{\gamma} \rightarrow (FY \tensor R)_y$ of the direct product of which the inverse limit of the lemma is a subset.  As a consequence of the lemma, the $\pi_{\gamma, y}$ are bijections.

We now put back the forgotten structure.  Fix some $y_1 \in Y_{\gamma}$ and define the Hilbert space structure and action of $\Cl(LE)_{\gamma}$ on $S_{\gamma}$ by requiring that the bijection $\pi_{\gamma, y_1}$ be a $\Cl(LE)_{\gamma}$ linear unitary isomorphism.  Then because the $\nu_{12}$ are $\Cl(LE)_{\gamma}$ linear unitary isomorphisms, from commutativity of the right hand square of definition \ref{d-inv-lim}, all the $\pi_{\gamma, y_2}$ are $\Cl(LE)_{\gamma}$ linear unitary isomorphisms.

The result doesn't depend on arbitrary choices such as that of $y_1$.  The inverse limit set from \ref{l-inv-lim-spec} doesn't depend on arbitrary choice.  The added structure doesn't either, for taking any $y_2 \in Y_{\gamma}$, the following diagram commutes, where all the maps are bijections and the map $\nu_{12} \colon (FY \tensor R)_{y_1} \rightarrow (FY \tensor R)_{y_2}$ is a Clifford linear unitary isomorphism.
\[
\begindc{\commdiag}[5]
\obj(10,20)[objS]{$S_{\gamma}$}
\obj(25,30)[objFYR1]{$(FY \tensor R)_{y_1}$}
\obj(25,10)[objFYR2]{$(FY \tensor R)_{y_2}$}
\mor{objS}{objFYR1}{$\pi_{\gamma, y_1}$}
\mor{objS}{objFYR2}{$\pi_{\gamma, y_2}$}
\mor{objFYR1}{objFYR2}{$\nu_{12}$}
\enddc
\]
The diagram shows that the choice of $y_1$ has no effect on the structures defined by it.  For instance, to find the sum of $u, v \in S_{\gamma}$:
\begin{align}
\pi_{\gamma, y_1}^{-1} (\pi_{\gamma, y_1} (u) + \pi_{\gamma, y_1} (v)) &= \pi_{\gamma, y_1}^{-1} \nu_{12} (\nu_{12}^{-1} \pi_{\gamma, y_2} (u) + \nu_{12}^{-1} \pi_{\gamma, y_2} (v)) \notag \\
 &= \pi_{\gamma, y_2}^{-1} (\pi_{\gamma, y_2} (u) + \pi_{\gamma, y_2} (v)). \notag
\end{align}

Because Fock representations are irreducible (see proposition \ref{p-fock-rep}) and the tensor product representations are isomorphic to Fock representations (see lemma \ref{l-clif-alg-rep-tens-tors}), our representation or $\Cl(LE)_{\gamma}$ module $S_{\gamma}$ is irreducible.

Now let the total space of our $\Cl(LE)$ module be the disjoint union $S = \amalg_{\gamma \in LM} S_{\gamma}$ as a set.  Let the projection $\pi_S$ be defined by $S_{\gamma} \rightarrow \{ \gamma \}$.  Considering $S = \amalg_{\gamma \in LM} S_{\gamma} \subset LM \cross \bigcup_{\gamma \in LM} S_{\gamma}$, $\pi_S$ is the projection on the first factor.

We have all of $S$ except a topology.  For this, let $\{ U_i \}$ be an indexed, good cover of $LM$ with local sections $s_i \colon U_i \rightarrow Y$ with $\pi \circ s_i = \ident_{U_i}$.  Define
\begin{align}
\psi_i \colon S_{|U_i} &\rightarrow s_i^{*} (FY \tensor R) \notag \\
s &\mapsto (\pi_S (s), \pi_{\pi_S (s), s_i \circ \pi_S (s)} (s)) = (\gamma, \pi_{\gamma, s_i (\gamma)} (s)), \notag
\end{align}
for $s \in S$ over $\gamma$.  By construction, the $\psi_i$ on fibers are Clifford linear unitary isomorphisms.  For $U_i \cap U_j \ne \emptyset$, suppose $s \in S_{|U_i \cap U_j}$, $\psi_i (s) = (\gamma, f) \in (s_i^{*} (FY \tensor R))_{\pi_S (s)} = (s_i^{*} (FY \tensor R))_{\gamma}$.  Then
\begin{align}
s &= \pi_{\gamma, s_i (\gamma)}^{-1} (\gamma, f), \text{ and} \notag \\
\psi_j \circ \psi_i^{-1} \colon (s_i^{*} (FY \tensor R))_{|U_i \cap U_j} &\rightarrow (s_j^{*} (FY \tensor R)_{|U_i \cap U_j} \text{ is} \notag \\
(\gamma, f) &\mapsto (\gamma, \pi_{\gamma, s_j (\gamma)} \circ \pi_{\gamma, s_i (\gamma)}^{-1} (f)) \notag \\
&= ((s_i, s_j)^{*} \nu) (\gamma, f), \notag
\end{align}
so that
\begin{align}
\psi_j \circ \psi_i^{-1} &= (s_i, s_j)^{*} \nu. \notag
\end{align}
Since $\nu$ is continuous, and a Clifford linear unitary isomorphism on fibers of $\pi_1^{*} (FY \tensor R)$, using the same Clifford algebra fiber for all $(y_1, y_2)$ above $\gamma$, so is $\psi_j \circ \psi_i^{-1}$ continuous and is a Clifford linear unitary isomorphism on fibers (above each $\gamma$) of $s_i^{*} (FY \tensor R)$.

Since the $s_i^{*} (FY \tensor R)$ are fiber bundles with standard fiber a Hilbert space over the contractible $U_i$, they have local trivializations that are unitary isomorphisms on fibers.  For all $i$, compose these with the respective $\psi_i$ to get bijections that satisfy a cocycle condition and determine a topology on $S$ that makes the compositions local trivializations of $S$ that are unitary isomorphisms on fibers.  Since the local trivializations of the $s_i^{*} (FY \tensor R)$ are already homeomorphisms, with the new topology so are the $\psi_i$.

By construction, each fiber $S_{\gamma}$ is a $(\Cl(LE))_{\gamma}$ module.  What remains is to show that the action $\Cl(LE) \cross_{LM} S \rightarrow S$ is continuous. Using the newly constructed local trivializations, because the homeomorhisms $\psi_i$ are Clifford linear unitary isomorphisms on fibers and the Clifford action on $FY$ (see lemma \ref{l-clle-act-fy}), hence $FY \tensor R$, hence $s_i^{*} (FY \tensor R)$ is continuous, the Clifford action on each $S_{U_i}$ and thus on $S$ is continuous.

If different local trivializations of $s_i^{*} (FY \tensor R)$ are chosen, their compatibility with the original choices implies the compatibility of the resulting local trivializations of $S$, resulting in the same topology for $S$.

Suppose given another set of local sections $s_i' \colon U_i \rightarrow Y$.  By adding duplicate $U_i$ to the indexed cover, the reasoning above, starting with the choice of good cover and local sections, comparing $s_i$ with $s_j'$ instead of with $s_j$, shows that the topology of $S$ is the same, regardless of the choice of local sections.

Further, given a different open cover $\{ U_i' \}$, there is a good cover of $LM$ that is a a refinement of both $\{ U_i \}$ and $\{ U_i' \}$, and reasoning in the same vein shows that the topology of $S$ is the same, regardless of the choices of open cover of $LM$ and local sections of $Y$.

For the third statement of the theorem, lemma \ref{l-bg-two-triv} gives from the two trivializations $R$, $R'$, a principal $\UU(1)$ bundle $Q \rightarrow LM$, such that $R' \cong R \tensor \pi^* Q$.  For each $\gamma \in LM$ and $y \in Y_{\gamma}$, since $(\pi^* Q)_y = \{ y \} \cross Q_{\gamma}$,
\begin{align}
S_{\gamma}' &\cong \varprojlim_{y \in \pi^{-1} (\gamma)} (FY \tensor R')_y \notag \\
 &\cong \varprojlim_{y \in \pi^{-1} (\gamma)} (FY \tensor (R \tensor \pi^* Q))_y \notag \\
 &\cong \varprojlim_{y \in \pi^{-1} (\gamma)} ((FY \tensor R)_y \tensor (\pi^* Q)_y) \notag \\
 &\cong \big(\varprojlim_{y \in \pi^{-1} (\gamma)} (FY \tensor R)_y \quad \big) \tensor Q_{\gamma}, \text{ so that} \notag \\
S' &\cong S \tensor Q. \notag
\end{align}
The isomorphism is a Clifford linear unitary isomorphism on fibers as in note \ref{n-clif-alg-rep-tens-tors}.  Above an open set in $LM$ small enough for all three bundles to be trivializable, the isomorphism is realized by multiplication by a $\UU(1)$ valued function, and thus is a homeomorphism.
\end{proof}

Having finished the main line of the proof, we come back to state and prove the lemma mentioned and used in it.
\begin{lem}\label{l-nu-ev}
\index{nu@$\nu$}
\index{ev@$\ev$}
($\ev$ is an Isomorphism).
The map $\ev$, part of the map $\nu$ of theorem \ref{t-bg-triv-clif-alg-mod}, is an isomorphism of fiber bundles, on fibers a Clifford linear unitary isomorphism.
\end{lem}
\begin{proof}
This lemma is not made to stand alone; we adopt wholesale the context of the preceding theorem and of its proof at the point it references this lemma.

To define $\ev$, note that both $P$ and $\pi_1^{*} FY$ are bundles associated to $L\SO(E)$:  $P$ is identified with $L\SO(E) \cross_{L\SO(n)} T$ (see proposition \ref{p-bg-p}), and $\pi_1^{*} FY = \pi_1^{*} (L\SO(E) \cross_{L\SO(n)} F)$.  Let $[\widetilde{\gamma}, L_1, L_2, \phi_{12}]$ denote an element of $P$ over $(y_1, y_2) = ([\widetilde{\gamma}, L_1], [\widetilde{\gamma}, L_2]) \in Y^{[2]}$, where $\widetilde{\gamma} \in L\SO(E)$, $L_1, L_2 \in \Lagr_{res}$, and $\phi_{12} \in T(L_1, L_2)$; $\phi_{12} \colon \F(L_1) \rightarrow \F(L_2)$ is a $\Cl(V)$ linear unitary isomorphism.  Let $[\widetilde{\gamma}, L_1, L_2, w]$ denote an element of $\pi_1^{*} FY$ over $(y_1, y_2)$, where $w \in F(L_1)$.  Define
\[
\ev \colon [\widetilde{\gamma}, L_1, L_2, \phi_{12}] \tensor [\widetilde{\gamma}, L_1, L_2, w] \mapsto [\widetilde{\gamma}, L_1, L_2, \phi_{12}(w)],
\]
with $\phi_{12}(w) \in \F(L_2)$.  That $\ev$ is well defined and continuous follows from the fact that the corresponding $\widetilde{\ev} \colon P \cross_{Y^{[2]}} \pi_1^{*} FY \rightarrow \pi_2^{*} FY$ is well defined, continuous, and $\UU(1)$-equivariant.  Those properties of $\widetilde{\ev}$ follow from the fact that the corresponding $\widetilde{\widetilde{\ev}} \colon T \cross_{(\Lagr_{res} \cross \Lagr_{res})} \pi_1^{*} F \rightarrow \pi_2^{*} F$, $((L_1, L_2, \phi_{12}), ((L_1, L_2), (L_1, w))) \mapsto ((L_1, L_2), (L_2, \phi_{12} (w)))$, where we have written in the pullback details, has the same properties, also is $L\SO(n)$-equivariant, and that the action of $L\SO(n)$ is $\UU(1)$-equivariant.  To summarize the domains and codomains of the maps,
\begin{align}
P \tensor \pi_1^{*} FY &\xrightarrow{\ev} \pi_2^{*} FY \notag \\
P \cross_{Y^{[2]}} \pi_1^{*} FY &\xrightarrow{\widetilde{\ev}} \pi_2^{*} FY \notag \\
T \cross_{(\Lagr_{res} \cross \Lagr_{res})} \pi_1^{*} F &\xrightarrow{\widetilde{\widetilde{\ev}}} \pi_2^{*} F. \notag
\end{align}

The $L\SO(n)$-equivariance of $\widetilde{\widetilde{\ev}}$ can be seen by taking $g \in L\SO(n)$; then referring to definition \ref{d-fock-spac-bndl} for the actions of $L\SO(n)$ on $\Lagr_{res}$ and $\F(L_1)$, and to proposition \ref{p-bg-t} for the action of $L\SO(n)$ on $T(L_1, L_2)$, $(g \cdot \phi_{12}) (g \cdot w) = (\Lambda_g \circ \phi_{12} \circ \Lambda_g^{*}) (\Lambda_g (w)) = \Lambda_g (\phi_{12} (g)) = g \cdot (\phi_{12} (g))$.

To obtain continuity of $\widetilde{\widetilde{\ev}}$ we use local trivializations of the bundles $T$ and $F$, as in the proof of continuity of bundle gerbe multiplication surrounding \ref{eq-pp-bg-mult}, specifically continuity of $\widetilde{\widetilde{m}}$ in the proof of proposition \ref{p-bg-p}, except that in this proof we have not just gone from associated bundle to cartesian product, but have removed $L\SO(E)$ entirely from the definition of $\widetilde{\widetilde{\ev}}$.  We show continuity of $\widetilde{\widetilde{\ev}}$ at $(L_1, L_2, \phi_{12}, w)$.

Let $V_{L_1}, V_{L_2} \subset \Lagr_{res}$ be open neighborhoods of $L_1$ and $L_2$ respectively, such that over $V_{L_1} \cross V_{L_2}$ is the domain of a local trivialization of the standard intertwiner principal $\UU(1)$ bundle $T$ as in \ref{eq-t-psi}, \ref{eq-t-psi-thet} of the proof of proposition \ref{p-bg-t}, and such that over $V_{L_1}$ is the domain of a local trivialization of the standard Fock space bundle $F$ as in note \ref{n-std-fock-spac-bndl}.

For $L_1' \in V_{L_1}$, $L_2' \in V_{L_2}$, $\phi \in T(L_1', L_2')$, $w' \in \F(L_1')$, omitting redundant or irrelevant symbols $L_1'$, $L_2'$ appropriately, with $\widehat{\phi_{12}}, \widehat{\phi_{12}'} \in T(L, L)$, the $\UU(1)$ torsor of intertwiners of the Fock space $\F(L)$ of the standard Lagrangian subspace $L$, and $\widehat{w}, \widehat{w'} \in \F(L)$, the local trivializations map as follows:
\begin{align}
w &\mapsto \widehat{w} = \pi_2 \circ \Theta_{L_1} (L_1', w) \notag \\
w' &\mapsto \widehat{w'} = \pi_2 \circ \Theta_{L_1} (L_1', w') \notag \\
\phi_{12} &\mapsto \widehat{\phi_{12}} = (\pi_2 \circ \Theta_{L_2} (L_2', \cdot)) \circ \phi_{12} \circ (\pi_2 \circ \Theta_{L_1} (L_1', \cdot))^{-1} \notag \\
\phi_{12}' &\mapsto \widehat{\phi_{12}'} = (\pi_2 \circ \Theta_{L_2} (L_2', \cdot)) \circ \phi_{12}' \circ (\pi_2 \circ \Theta_{L_1} (L_1', \cdot))^{-1} \notag
\end{align}
\begin{align}
\phi_{12}(w) &\mapsto \pi_2 \circ \Theta_{L_2} (L_2', \phi_{12}(w)) \notag \\
&= (\pi_2 \circ \Theta_{L_2} (L_2', \cdot)) \circ \phi_{12} \circ (\pi_2 \circ \Theta_{L_1} (L_1', \cdot))^{-1} (\pi_2 \circ \Theta_{L_1} (L_1', w)) \notag \\
&= \widehat{\phi_{12}} (\widehat{w}) \text{;} \notag \\
\phi_{12}'(w') &\mapsto \pi_2 \circ \Theta_{L_2} (L_2', \phi_{12}'(w')) \notag \\
&= (\pi_2 \circ \Theta_{L_2} (L_2', \cdot)) \circ \phi_{12}' \circ (\pi_2 \circ \Theta_{L_1} (L_1', \cdot))^{-1} (\pi_2 \circ \Theta_{L_1} (L_1', w')) \notag \\
&= \widehat{\phi_{12}'} (\widehat{w'}) \text{;} \notag
\end{align}
that is, the salient results of the map in the local trivializations are $\widehat{\phi_{12}} \cross \widehat{w} \mapsto \widehat{\phi_{12}} (\widehat{w})$ and $\widehat{\phi_{12}'} \cross \widehat{w'} \mapsto \widehat{\phi_{12}'} (\widehat{w'})$.

Recall that for $T(L, L)$ the operator norm and strong operator topologies are equal.  Consider continuity at $\widehat{\phi_{12}}$ and $\widehat{w}$, fixing those for the moment.  In $\F(L)$ we have $\widehat{\phi_{12}}(\widehat{w}) - \widehat{\phi_{12}'}(\widehat{w'}) = (\widehat{\phi_{12}} - \widehat{\phi_{12}'}) (\widehat{w}) + \widehat{\phi_{12}'} (\widehat{w} - \widehat{w'})$.  The norm in $\F(L)$ of the first term goes to $0$ as $\widehat{\phi_{12}'} \rightarrow \widehat{\phi_{12}}$ with either topology, and the norm of the second term goes to $0$ as $\widehat{w'} \rightarrow \widehat{w}$ since the norm of $\widehat{\phi_{12}'}$ is bounded by $1$.

Thus $\widetilde{\widetilde{\ev}}$, and hence $\widetilde{\ev}$ and $\ev$ are continuous.

That $\ev$ is Clifford linear on each fiber follows from the same property of $\widetilde{\widetilde{\ev}}$.  On its domain $T \cross_{(\Lagr_{res} \cross \Lagr_{res})} \pi_1^{*} F$, define the action of $\pi_1^{*} (\Lagr_{res} \cross \Cl(L\RR^n))$ trivially on the first factor and, decoding the pullbacks, by the action of lemma \ref{l-clif-act-std-fock-spac-bndl} on the second factor.  On the codomain of $\widetilde{\widetilde{\ev}}$ also, let $\pi_2^{*} (\Lagr_{res} \cross \Cl(L\RR^n))$ act by the lemma.  Then $\widetilde{\widetilde{\ev}}$ is Clifford linear by definition \ref{d-equi-rep} applied to $\phi_{12}$.

Thence from $L\SO(n)$-equivariance of $\widetilde{\widetilde{\ev}}$, taking the associated product with $L\SO(E)$ for domain and codomain of $\widetilde{\widetilde{\ev}}$ and for the Clifford algebra $\Cl(L\RR^n)$, follows the Clifford linearity of $\widetilde{\ev}$ with respect to the Clifford actions of $\pi_1^{*} \pi^{*} \Cl(LE)$ on the domain and $\pi_2^{*} \pi^{*} \Cl(LE)$ on the codomain.  These are canonically isomorphic Clifford algebra bundles; the algebra part only changes when $\gamma$ does.  Finally, because in definition \ref{d-clif-alg-rep-tens-bndl} the Clifford action is defined by descent from the action on the associated product defined as for $\widetilde{\ev}$, $\ev$ itself is Clifford linear on fibers.

To see that $\ev$ is a unitary isomorphism on fibers we argue directly for $\ev$, having not thought of a sensible Hilbert space structure for the associated or cartesian products for $\widetilde{\ev}$ or $\widetilde{\widetilde{\ev}}$.  Suppose $\phi_{12a}, \phi_{12b} = \overline{z} \phi_{12a} \in T(L_1, L_2)$ for some $z \in \UU(1)$, and $w_{1a}, w_{1b} \in \F(L_1)$.  Then before applying $\ev$, omitting irrelevant data for the pullbacks and referring all the associated product equivalence classes to the same $\widetilde{\gamma}$, $\langle \phi_{12a} \tensor w_{1a}, \phi_{12b} \tensor w_{1b} \rangle =  \langle \phi_{12a}, \phi_{12b} \rangle \langle w_{1a}, w_{1b} \rangle = z \langle w_{1a}, w_{1b} \rangle$.  On the other hand, after applying $\ev$, $\langle \phi_{12a} (w_{1a}), \phi_{12b} (w_{1b}) \rangle = z \langle \phi_{12a} (w_{1a}), \phi_{12a} (w_{1b}) \rangle = z \langle w_{1a}, w_{1b} \rangle$, since $\phi_{12a}$ is unitary; the same after as before $\ev$.
\end{proof}

\chapter{FUTURE WORK}\label{c-futu-work}
\index{transgression!Pontryagin class!Dixmier-Douady class}
\index{Pontryagin class!transgression!Dixmier-Douady class}
\index{Dixmier-Douady class!transgression!Pontryagin class}

\section{Introduction to Future Work}\label{s-intr-futu-work}

Apart from proposition \ref{p-dd-bg-susp-c1-pb} about a bundle gerbe over a suspension, which could stand alone, this chapter is mainly concerned with three ways the work in the thesis could be continued.  Two are described very briefly, and one in some detail, with supporting results, which also give examples (e.g. lemma \ref{l-eq-bg}) of working concretely with a bundle gerbe constructed from a particular vector bundle.

The first subject for future work can be described in one paragraph.
\begin{note}\label{n-spin-stru}
\index{spin structure}
(Spin Structures for the Vector Bundles).
Note \ref{n-equi-rep-z2-grad} implies that a Clifford linear unitary isomorphism of Fock representations either preserves or reverses the natural $\ZZ_2$ grading.  It seems likely possible, in the case that the vector bundle $E$ has a spin structure, to alter the bundle gerbe construction to use instead of $\Lagr_{res}$, one connected component of that space (see proposition 1.2 (i), (ii) and the start of section 2.1 of \citet[pages~808,~814]{SW07}).  Instead of $L\SO(E)$, use $L\Spin(E)$; and make the associated products now over $L\SO(n)$, over $L\Spin(n)$, which would act on the left via the projection to $L\SO(n)$, by which it covers only the connected component of the identity, preserving the connected components of $\Lagr_{res}$.  Then the intertwiners in the bundle gerbe would preserve the gradings, so the spinor bundle $S$ would be $\ZZ_2$ graded as a bundle.

Part of this is similar to the situation in lemma \ref{l-eq-bg}, where there is a reduction of the structure group of the vector bundle from $\SO(n)$ to $\Symp(n)$, and the latter is used to construct the bundle gerbe.
\end{note}

The second subject for future work will make up the bulk of the chapter, and concerns the conjecture that the transgression of the first Pontryagin class of the original smooth vector bundle is plus or minus twice the Dixmier-Douady class of the constructed bundle gerbe.  If $E$ has a spin structure so that $\frac{p_1}{2} (E)$ exists, it may be that if it vanishes, so does $DD(G(E))$, whence using local trivializations of the polarization class bundle $Y$ and the bundle gerbe principal bundle $P$, proposition \ref{p-bg-triv-dd-zero} constructs a trivialization of the bundle gerbe.  We will see some results that could possibly be part of a proof.  It turns out that the conjecture first stated may depend on a the existence of a spin structure on the vector bundle, and thus on the work discussed in note \ref{n-spin-stru}.

The third subject for future work, again, can be described in one paragraph.  It is to consider a bundle two-gerbe associated to the first Pontryagin class of the original smooth vector bundle, perhaps from the associated frame bundle, constructing the bundle two-gerbe as in \citet{Stev00} Proposition 9.3 (see also Proposition 11.4).  From it could perhaps be obtained a bundle gerbe stably isomorphic to the bundle gerbe constructed here from the vector bundle.  The bundle two-gerbe would be trivial if the Pontryagin class were trivial as in the reference's Proposition 12.2, and it would be nice if it were possible to construct from a trivialization of the bundle two-gerbe, a trivialization of the bundle gerbe.

Returning to the subject for future work that will be described at length, section \ref{s-tran-h4m-h3lm} defines the transgression $\tau \colon \HH^4(M;\ZZ) \rightarrow \HH^3(LM;\ZZ)$, and shows that the transgression $\tau \colon \HH^4(S^4;\ZZ) \rightarrow \HH^3(LS^4;\ZZ)$ is an isomorphism.

Section \ref{s-quat-line-bndl} defines the particular vector bundle for what will be called, hopefully, the universal case of the conjecture:  the real vector bundle underlying the tautological quaternionic line bundle $E_Q \rightarrow \QH P^1 \cong S^4$.

Section \ref{s-gene-univ} shows that given a doubtful hypothesis concerning cohomology of a classifying space, the particular case of $E_Q$ implies the general case of the conjecture; that is, that $\pm 2 DD (G(E_Q)) = \tau (p_1(E_Q)) \Rightarrow \pm 2 DD (G(E)) = \tau (p_1(E))$ for all vector bundles $E$ satisfying our hypotheses.  If the vector bundles have spin structures and the work described in note \ref{n-spin-stru} turns out as expected, then it seems likely that the analogous cohomological hypothesis would be true, and that the particular, thus universal, case would imply the general case.  Further, since then $p_1 (E)$ is divisible by $2$, it's possible that $DD (G(E)) = \tau (\frac{p_1}{2} (E_Q))$ could then be implied by the same statement for $E_Q$.

Section \ref{s-dixm-doua-quat} gives an overview of an approach to proving the conjecture that the universal case of $E_Q$ is true.  The following sections after \ref{s-dixm-doua-quat} flesh this out.

Section \ref{s-dd-clas-bg-susp} defines a particular suspension isomorphism, and to relate the Dixmier-Douady class of any continuous bundle gerbe over the unreduced suspension of a given locally connected metric space $X$, to the first Chern class of a constructed principal bundle over $X$.  This was written originally to aid in the proof concerning $G(E_Q)$, which can be pulled back to $S^3$, the suspension of $S^2$, but the isomorphism could be of more general use.

Section \ref{s-map-s3-ls4-ster-proj-homo-coor} discuss the map used for the pullback of $G(E_Q)$ from $L\QH P^1 \cong LS^4$ to $S^3$, records isomorphisms related to stereoscopic projection for $S^4$ and homogeneous coordinates for $\QH P^1$.

Section \ref{s-bg-cons-pb-loc-sect-pb} constructs the bundle gerbe $G(E_Q)$, pulls it back to $S^3$, and defines local sections of the pullback.  Also it applies the suspension result of section \ref{s-dd-clas-bg-susp} to the bundle gerbe over $S^3$, obtaining a principal bundle over $S^2$.

In addition to these three lines of work that might give further results, it also might be possible to strengthen and simplify some of the arguments in the thesis, particularly those concerning the \Frechet manifold topology on loop spaces.  Perhaps the proof of proposition \ref{p-smth-free-loop-spac-frec-mfld} could be simplified by clarifying the relation between definitions in \citet{Hami82} and \citet{Stac05}.  In the proof of proposition \ref{p-loop-fb-is-frec-fb}, where only the existence of a topological isomorphism is shown between $L(P \cross_G F)$ and $LP \cross_{LG} LF$, it might be possible to show there is a diffeomorphism.  Another example is in the proof of proposition \ref{p-lm-c-coho-isom}, which would be simpler if it could be shown that $LM$ and $C(S^1, M)$ were homotopy equivalent through comparing definitions in \citet{Hami82} and \citet{Hirs76}.  Then it might be possible also to simplify some things this proof depends on.

\begin{ass}\label{a-sing-tran}
\index{assumptions!tau@$\tau$}
\index{assumptions!singular cohomology}
\index{transgression}
\index{tau@$\tau$}
Throughout this chapter we will use singular cohomology with $\ZZ$ coefficients except as otherwise noted, though sometimes $\ZZ$ will be made explicit.  Throughout this chapter $\tau$ will denote a transgression $\HH^{k+1}(M) \rightarrow \HH^k(LM)$, usually for $k = 3$, except in the proof of proposition \ref{p-dd-bg-susp-c1-pb}, where $\tau$ is used for another purpose.
\end{ass}

% Using $\HH^4(S^4)$, etc., didn't work; hence the approximation.
\section{Transgression Isomorphism $\text{H}^4(S^4)$ to $\text{H}^3(LS^4)$}\label{s-tran-h4m-h3lm}
First we define the transgression $\tau \colon \HH^{k+1}(M;\ZZ) \rightarrow \HH^k(LM;\ZZ)$ and show that it is an isomorphism in the case of $M = S^4$ and $k = 3$.  Along the way we show generally that the inclusion $LM \rightarrow C(S^1, M)$ induces isomorphisms in cohomology, by finding corresponding indexed good covers of the two spaces for use in \v{C}ech cohomology.

\begin{defn}\label{d-tran-hkp1-m-to-hk-lm}
\index{transgression!H{k+1}(M) to Hk(LM)@$\HH^{k+1}(M;\ZZ) \rightarrow \HH^k(LM;\ZZ)$}
\index{evaluation map}
\index{transgression}
\index{tau@$\tau$}
(The Transgression $\HH^{k+1}(M) \rightarrow \HH^k(LM)$).
Let $\ev \colon C(S^1, M) \cross S^1 \rightarrow M$ be the evaluation map $\ev(\gamma, t) = \gamma(t)$, and let $\ev^{*} \colon \HH^{k+1}(M) \rightarrow \HH^{k+1}(C(S^1, M) \cross S^1)$ denote the induced map in cohomology.  Let $/$ denote the slant product \citep[page~287]{Span66} and let $s \colon \HH^{k+1} (C(S^1, M) \cross S^1) \rightarrow \HH^k (C(S^1, M))$ be defined by $s \colon h \mapsto h / [S^1]$, where $[S^1]$ is the fundamental class of $S^1$ corresponding to the standard orientation \citep[pages~303--304]{Span66}.  Let $i_{LM} \colon LM \rightarrow C(S^1, M)$ be the inclusion, and $i_{LM}^{*} \colon \HH^k(C(S^1, M)) \rightarrow \HH^k(LM)$ be the induced map in cohomology.  Then define the transgression maps $\tau_C$, $\tau$ to $\HH^k (C(S^1, M))$ and $\HH^k (LM)$ respectively:
\begin{align}
\tau_C \colon \HH^{k+1}(M;\ZZ) &\rightarrow \HH^k (C(S^1, M);\ZZ) \notag \\
\tau_C = s \circ \ev^{*} \colon c &\mapsto (\ev^{*}(c)) / [S^1] \notag \\
\tau \colon \HH^{k+1}(M;\ZZ) &\rightarrow \HH^k(LM;\ZZ) \notag \\
\tau = i_{LM}^{*} \circ \tau_C \colon c &\mapsto i_{LM}^{*} ((\ev^{*}(c)) / [S^1]). \notag
\end{align}
\end{defn}

\begin{prop}\label{p-tran-hkp1-m-to-hk-lm}
\index{transgression!H{k+1}(M) to Hk(LM)@$\HH^{k+1}(M;\ZZ) \rightarrow \HH^k(LM;\ZZ)$}
\index{evaluation map}
\index{transgression!naturality}
(The Transgression $\HH^{k+1}(M;\ZZ) \rightarrow \HH^k(LM;\ZZ)$ is Well-Defined and Natural).
$\ev$ and $i_{LM}$ are continuous, and so $\tau_C$ and $\tau$ are well-defined.  The transgression $\tau$ is a natural transformation of degree $-1$ from the degree $k+1$ singular cohomology functor on topological spaces to the degree $k$ singular cohomology functor composed with the smooth loop functor (see corollary \ref{co-smth-loop-func}) of topological spaces.  That is, given a continuous map $f \colon M' \rightarrow M$ it follows that $(Lf)^{*} \circ \tau = \tau \circ f^{*} \colon \HH^{k+1}(M;\ZZ) \rightarrow \HH^k(LM';\ZZ)$.
\end{prop}
\begin{proof}
$S^1$ is compact, hence by lemma \ref{l-co-top} the evaluation map $\ev \colon C(S^1, M) \cross S^1 \rightarrow M$ is continuous.  The inclusion $i_{LM} \colon LM \rightarrow C(S^1, M)$ is continuous by proposition \ref{p-smth-free-loop-spac-frec-mfld}.

From definitions, $\ev \circ (Lf \cross \ident) = f \circ \ev$, so $(Lf \cross \ident)^{*} \circ \ev^{*} = \ev^{*} \circ f^{*}$; i.e., the bottom square in the following diagram commutes.  The middle square commutes by naturality of the slant product as in property $1$, \cite[page~287]{Span66}; his $f$, our $Lf$; and his $g$, our $\ident$.  The top square commutes because the square of maps of spaces that induces it commutes.  The compositions forming the sides of the rectangle each equal $\tau$ for the respective spaces.
\[
\begindc{\commdiag}[5]
\obj(10,10)[objHkp1M]{$\HH^{k+1}(M; \ZZ)$}
\obj(10,30)[objCSMxS]{$\HH^{k+1}(C(S^1,M) \cross S^1; \ZZ)$}
\obj(10,50)[objCSM]{$\HH^k(C(S^1,M); \ZZ)$}
\obj(10,70)[objLM]{$\HH^k(LM; \ZZ)$}
\obj(50,10)[objHkp1Mp]{$\HH^{k+1}(M'; \ZZ)$}
\obj(50,30)[objCSMpxS]{$\HH^{k+1}(C(S^1,M') \cross S^1; \ZZ)$}
\obj(50,50)[objCSMp]{$\HH^k(C(S^1,M'); \ZZ)$}
\obj(50,70)[objLMp]{$\HH^k(LM'; \ZZ)$}
\mor{objHkp1M}{objCSMxS}{$\ev^{*}$}
\mor{objCSMxS}{objCSM}{$s$}
\mor{objCSM}{objLM}{$i_{LM}^{*}$}
\mor{objHkp1Mp}{objCSMpxS}{$\ev^{*}$}
\mor{objCSMpxS}{objCSMp}{$s$}
\mor{objCSMp}{objLMp}{$i_{LM'}^{*}$}
\mor{objHkp1M}{objHkp1Mp}{$f^{*}$}
\mor{objCSMxS}{objCSMpxS}{$(Lf \cross \ident)^{*}$}
\mor{objCSM}{objCSMp}{$(Lf)^{*}$}
\mor{objLM}{objLMp}{$(Lf)^{*}$}
\enddc
\]
\end{proof}

\begin{lem}\label{l-s4-oc-grps}
\index{cohomology!C(S1,S4), Omega S4@$C(S^1, S^4)$, $\Omega S^4$}
\index{homotopy!C(S1,S4), Omega S4@$C(S^1, S^4)$, $\Omega S^4$}
(Various Facts about $\Omega S^4$ and $C(S^1, S^4)$).
Let $0$ be the base point of $S^1$, $*$ the base point of topological spaces $X$ for which we don't wish to choose an explicit base point, and define as usual $\Omega X = \{ \gamma \in C(S^1, X) \st \gamma(0) = * \}$.
\begin{enumerate}
    \item $\Hot_k (\Omega S^4) \cong 0, 0, 0, \ZZ, \ZZ / 2 \ZZ$, $k = 0, 1, 2, 3, 4$, because $\Hot_k (\Omega X) = \Hot_{k+1} (X)$, by \citet[pages~437,~445]{ Bred93}, $S^1$ being locally compact Hausdorff, and using the table in \citet[page~339]{Hatc02} for $k = 4$.
    \item $\HH_k (\Omega S^4) \cong \ZZ, 0, 0, \ZZ, 0$, $k = 0, 1, 2, 3, 4$, viewing $S^4$ as the reduced suspension $S S^3$, weak homotopy equivalent to the James reduced product $J(S^3)$ \citep[page~471]{Hatc02}, which has homology given by the tensor algebra over $\ZZ$ generated by $\widetilde{\HH}_{*} (S^3; \ZZ)$ \citep[page~288]{Hatc02}.
    \item $\HH^k (\Omega S^4) \cong \ZZ, 0, 0, \ZZ, 0$, $k = 0, 1, 2, 3, 4$, by the universal coefficient theorem for singular cohomology \citep[page~282]{Bred93}.
    \item $\Hot_k (C(S^1, S^4)) \cong 0, 0, 0, \ZZ$, $k = 0, 1, 2, 3$, by the homotopy exact sequence \citep[page~453]{Bred93} of the fibration \citep[pages~393--394]{Dugu66} $\Omega S^4 \rightarrow C(S^1, S^4) \rightarrow S^4$.
    \item $\HH_k (C(S^1, S^4)) \cong \ZZ, 0, 0, \ZZ$, $k = 0, 1, 2, 3$, using the Hurewicz homomorphism $\Hot_k (C(S^1, S^4)) \isomto \widetilde{\HH}_k (C(S^1, S^4))$ \citep[page~185]{Swit75}, an isomorphism for $k \le 3$.
    \item $\HH^k (C(S^1, S^4)) \cong \ZZ, 0, 0, \ZZ$, $k = 0, 1, 2, 3$, reasoning as for $\Omega S^4$.
\end{enumerate}
\end{lem}
\begin{proof}
Regarding $\Hot_k (C(S^1, S^4))$, $0 \xrightarrow{\incl} S^1$ is a closed cofibration by \citet[page~435]{Bred93}, hence, defining $\ev_0 \colon C(S^1, S^4) \rightarrow S^4$ as the restriction map $C(S^1, S^4) \rightarrow C(\{ 0 \}, S^4)$ followed by the identification $C(\{ 0 \}, S^4) \cong S^4$, $\ev_0$ is a fibration by \citet[page~455]{Bred93}, with fiber $\ev_0^{-1} (*) = \Omega S^4$.  Then the homotopy exact sequence of the fibration \citep[page~453]{Bred93} gives for $k \ge 0$
\[
\Hot_{k + 1} (S^4) \rightarrow \Hot_k(\Omega S^4) \rightarrow \Hot_k(C(S^1, S^4)) \rightarrow \Hot_k (S^4). \notag
\]
Since $\ev_0$ has a section $s \colon S^4 \rightarrow C(S^1, S^4)$ mapping each point in $S^4$ to the constant loop at that point, $\ev_0 \circ s = \ident_{S^4}$, $(\ev_0)_{\sharp} \circ s_{\sharp} = \ident_{\pi_k (S^4)}$, and the consequent surjectivity of $(\ev_0)_{\sharp}$ implies that the subsequent connecting homomorphisms are zero.  Thus the maps $\Hot_k(\Omega S^4) \rightarrow \Hot_k(C(S^1, S^4))$ induced by the inclusion are isomorphisms for $k \le 3$, and since $\Hot_k(\Omega S^4) \cong \Hot_{k + 1}(S^4)$, we obtain $\Hot_k (C(S^1, S^4)) \cong 0, 0, 0, \ZZ$, $k = 0, 1, 2, 3$.
\end{proof}

\begin{lem}\label{l-ev-star-h4s4-h4os}
\index{evaluation map!Omega(S4)xS1->S4@$\Omega S^4 \cross S^1 \rightarrow S^4$}
% (A title for this lemma seemed redundant).
The evaluation map $\Omega S^4 \cross S^1 \rightarrow S^4$ induces an isomorphism in $\HH^4$.
\end{lem}
\begin{proof}
The discussion we will use from \citet{Bred93} in this proof assumes that all the spaces are Hausdorff; we know this for $S^1$, $S^3$, and $S^4$, and by \citet[pages~258--259]{Dugu66} or otherwise, since $S^4$ is Hausdorff, so is $\Omega S^4$.  Let the base point of $\Omega S^4$ be $\gamma_{*}$, the constant loop with value $* \in S^4$.  Letting $\ev_{\Omega}$ denote the restriction to $\Omega S^4 \cross S^1$ of the previously defined $\ev$, then $\Omega S^4 \cross S^1 \xrightarrow{\ev_{\Omega}} S^4$ is a continuous map of pointed spaces.  Consider the smash product $\Omega S^4 \wedge S^1 = (\Omega S^4 \cross S^1) / (\Omega S^4 \vee S^1)$.  Since $\ev_{\Omega}$ maps $\Omega S^4 \vee S^1 = (\{ \gamma_{*} \} \cross S^1) \cup (\Omega S^4 \cross \{ 0 \})$ to $*$, the base point of $S^4$, $\ev_{\Omega}$ descends to a continuous \citep[pages~40,~437--439]{Bred93} map of pointed spaces, $\Omega S^4 \wedge S^1 \xrightarrow{\widetilde{\ev_{\Omega}}} S^4$.  This evaluation map is closely tied to an ``exponential law'' or adjoint relation for maps from a smash product, and to the relation between spheres of successive dimensions via the reduced suspension or smash product with $S^1$, as we will see.

Let $S^3 \wedge S^1 \xrightarrow{\phi} S^4$ be a homeomorphism of pointed spaces \citep[page~435]{Bred93}.  By \citet[pages~437--439]{Bred93}, since $S^1$, $S^3$ are locally compact pointed spaces, the map $\widetilde{\phi} \colon S^3 \rightarrow \Omega S^4$ defined by $\widetilde{\phi} (x) (t) = \phi (x \wedge t)$, i.e. the map induced by the corresponding adjoint map for cartesian products, is a continuous map of pointed spaces, where $x \wedge t$ is a shorthand for the image of $(x, t)$ under the projection from the cartesian product to its quotient, the smash product.  Let $\widetilde{\phi} \wedge \ident \colon S^3 \wedge S^1 \rightarrow \Omega S^4 \wedge S^1$ denote the continuous map of pointed spaces induced by the corresponding map $(x, t) \mapsto (\widetilde{\phi}(x), t)$ for cartesian products \citep[page~68]{Jame84}.  Then $(\widetilde{\ev_{\Omega}} \circ (\widetilde{\phi} \wedge \ident)) (x \wedge t) = \widetilde{\ev_{\Omega}} ((\widetilde{\phi} (x)) \wedge t) = \widetilde{\phi} (x) (t) = \phi (x \wedge t)$, and so the following diagram of continuous maps of pointed spaces commutes:
\[
\begindc{\commdiag}[5]
\obj(10,10)[objS1S3]{$S^3 \wedge S^1$}
\obj(30,10)[objS1OS4]{$\Omega S^4 \wedge S^1$}
\obj(55,10)[objS4]{$S^4$}
\mor{objS1S3}{objS1OS4}{$\widetilde{\phi} \wedge \ident$}
\mor{objS1OS4}{objS4}{$\widetilde{\ev_{\Omega}}$}
\cmor((10,8)(30,2)(53,8)) \pright(30,0){$\phi$}
\enddc
\]
The homeomorphism $\phi$ induces an isomorphism in cohomology.  Because $\HH^4 (S^4) \cong \ZZ$, or for other reasons, $\HH^4 (S^3 \wedge S^1) \cong \ZZ$.  Because of the homotopy equivalence between the unreduced and reduced suspension of a well-pointed space (one for which the inclusion of the base point is a cofibration) in \citet[page~436]{Bred93}, and because of the (unreduced) suspension isomorphism of problem 1 in \citet[page~190]{Bred93}, $\HH^4 (\Omega S^4 \wedge S^1) = \HH^4 (S \Omega S^4) \cong  \HH^4 (\Sigma \Omega S^4) \cong \HH^3 (\Omega S^4) \cong \ZZ$; we noted above that $\{ 0 \} \xrightarrow{\incl} S^4$ is a cofibration, and thus $\gamma_{*} \xrightarrow{\incl} \Omega S^4$ is also, by exercise 7 on \citet[page~457]{Bred93}.  The composition of the maps in $4$-th cohomology induced by the maps in the top of the diagram gives an isomorphism $\ZZ \rightarrow \ZZ$ that factors through another copy of $\ZZ$, whence both factors of the composition must be isomorphisms, and so $\widetilde{\ev_{\Omega}}^{*}$ is an isomorphism in $4$-th cohomology.

Consider now the commutative diagrams
\[
\begindc{\commdiag}[5]
\obj(10,20)[objS1xOS4]{$\Omega S^4 \cross S^1$}
\obj(10,0)[objS1wOS4]{$\Omega S^4 \wedge S^1$}
\obj(35,10)[objS4]{$S^4$}
\mor{objS1xOS4}{objS1wOS4}{$\pi$}
\mor{objS1xOS4}{objS4}{$\ev_{\Omega}$}
\mor{objS1wOS4}{objS4}{$\widetilde{\ev_{\Omega}}$}
\obj(50,20)[objHS1xOS4]{$\HH^4 (\Omega S^4 \cross S^1)$}
\obj(50,0)[objHS1wOS4]{$\HH^4 (\Omega S^4 \wedge S^1)$}
\obj(75,10)[objHS4]{$\HH^4 (S^4)$}
\mor{objHS1wOS4}{objHS1xOS4}{$\pi^{*}$}
\mor{objHS4}{objHS1xOS4}{$\ev_{\Omega}^{*}$}
\mor{objHS4}{objHS1wOS4}{$\widetilde{\ev_{\Omega}}^{*}$}
\enddc
\]
where $\pi$ is the smash product projection or quotient map.  To show that $\ev_{\Omega}^{*}$ is an isomorphism in $4$-th cohomology, is equivalent to showing the same for $\pi^{*}$.  Consider the long exact sequence of the pair $(X, A)$, where $A = \Omega S^4 \vee S^1 = (\{ \gamma_{*} \} \cross S^1) \cup (\Omega S^4 \cross \{ 0 \}) \xrightarrow{\incl} \Omega S^4 \cross S^1 = X$ is a closed cofibration by \citet[page~114]{tomD08}, so by \citet[page~434]{Bred93}, the map $\pi_{pair} \colon (X, A) \rightarrow (X / A, *)$ induced by $\pi$ induces in cohomology $\HH^4 (X, A) \isomfrom \HH^4 (X / A, *) \cong \HH^4 (X / A) = \HH^4 (\Omega S^4 \wedge S^1) \cong \ZZ$ since $\widetilde{\ev_{\Omega}}^{*}$ is an isomorphism in $4$-th cohomology.

By the K\"{u}nneth theorem \citet[page~219]{Hatc02}, $\HH^4 (X) = \HH^4 (\Omega S^4 \cross S^1) \cong \HH^4 (\Omega S^4) \tensor \HH^0 (S^1) \dirsum \HH^3 (\Omega S^4) \tensor \HH^1 (S^1) \cong \HH^3 (\Omega S^4) \cong \ZZ$ as abelian groups.  By \citet[page~~180]{Swit75}, $\HH^4 (A) = \HH^4 (\Omega S^4 \vee S^1) \cong \HH^4 (\Omega S^4) \cross \HH^4 (S^1) \cong 0 \cross 0 \cong 0$.  Thus from $(A, \emptyset) \xrightarrow{i_A} (X, \emptyset) \xrightarrow{j_A} (X, A)$,
\begin{align}
0 \cong \HH^4 (A) \xleftarrow{i_A^{*}} \ZZ \cong \HH^4 (X) \xleftarrow{j_A^{*}} \HH^4 (X, A) \cong \ZZ \notag
\end{align}
is exact, whence $j_A^{*}$, equivalent to a surjective homomorphism $\ZZ \leftarrow \ZZ$, is an isomorphism.  Since all the maps in $4$-th cohomology induced by the maps other than $\pi$ in the following commutative diagram are isomorphisms, $\pi^{*}$ is also.

\[
\begindc{\commdiag}[5]
\obj(10,20)[objXA]{$(X, A)$}
\obj(10,10)[objXE]{$(X, \emptyset)$}
\obj(30,20)[objXAS]{$(X / A, *)$}
\obj(30,10)[objXAE]{$(X / A, \emptyset)$}
\mor{objXE}{objXA}{$j_A$}
\mor{objXAE}{objXAS}{$j_{X/A}$}
\mor{objXA}{objXAS}{$\pi_{pair}$}
\mor{objXE}{objXAE}{$\pi$}
\enddc
\]
\end{proof}

\begin{prop}\label{p-tran-h4s4-to-h3c}
\index{transgression!H4(S4) to H3(C(S1,S4))@$\HH^4(S^4;\ZZ) \rightarrow \HH^3(C(S^1, S^4);\ZZ)$}
\index{transgression!isomorphism}
% (A title for this lemma seemed redundant).
The transgression $\tau_C \colon \HH^4(S^4;\ZZ) \isomto \HH^3(C(S^1, S^4);\ZZ)$ is an isomorphism.
\end{prop}
\begin{proof}
Denote the evaluation map on $\Omega S^4 \cross S^1$ by $\ev_{\Omega}$, and the inclusion $\Omega S^4 \rightarrow C(S^1, S^4)$ by $i_{\Omega}$.  Denote the evaluation map on $C(S^1, S^4) \cross S^1$ by $\ev$.  In the following diagrams, commutativity of the left triangle implies commutativity of the right triangles.  The maps $s$ and $s_{\Omega}^{*}$ are the slant product with $[S^1]$.  Naturality of the slant product as in the proof of proposition \ref{p-tran-hkp1-m-to-hk-lm} implies commutativity of the left square.
\[
\begindc{\commdiag}[5]
\obj(10,60)[objS1xOS4]{$\Omega S^4 \cross S^1$}
\obj(10,40)[objS1xCS4]{$C(S^1, S^4) \cross S^1$}
\obj(35,50)[objS4]{$S^4$}
\mor{objS1xOS4}{objS1xCS4}{$i_{\Omega} \cross \ident$}
\mor{objS1xOS4}{objS4}{$\ev_{\Omega}$}
\mor{objS1xCS4}{objS4}{$\ev$}
\obj(50,60)[objHS1xOS4]{$\HH^4 (\Omega S^4 \cross S^1)$}
\obj(50,40)[objHS1xCS4]{$\HH^4 (C(S^1, S^4) \cross S^1)$}
\obj(75,50)[objHS4]{$\HH^4 (S^4)$}
\mor{objHS1xCS4}{objHS1xOS4}{$(i_{\Omega} \cross \ident)^{*}$}
\mor{objHS4}{objHS1xOS4}{$\ev_{\Omega}^{*}$}
\mor{objHS4}{objHS1xCS4}{$\ev^{*}$}
\obj(25,30)[objHOS4]{$\HH^3 (\Omega S^4)$}
\obj(25,10)[objHCS4]{$\HH^3 (C(S^1, S^4))$}
\obj(50,30)[objHS1xOS4a]{$\HH^4 (\Omega S^4 \cross S^1)$}
\obj(50,10)[objHS1xCS4a]{$\HH^4 (C(S^1, S^4) \cross S^1)$}
\obj(75,20)[objHS4a]{$\HH^4 (S^4)$}
\mor{objHCS4}{objHOS4}{$i_{\Omega}^{*}$}
\mor{objHS1xCS4a}{objHS1xOS4a}{$(i_{\Omega} \cross \ident)^{*}$}
\mor{objHS1xOS4a}{objHOS4}{$s_{\Omega}^{*}$}
\mor{objHS1xCS4a}{objHCS4}{$s$}
\mor{objHS4a}{objHS1xOS4a}{$\ev_{\Omega}^{*}$}
\mor{objHS4a}{objHS1xCS4a}{$\ev^{*}$}
\enddc
\]
Lemma \ref{l-ev-star-h4s4-h4os} showed that $\ev_{\Omega}^{*}$ is an isomorphism, so that both $\HH^4 (S^4)$ and $\HH^4 (\Omega S^4 \cross S^1)$ are isomorphic to $\ZZ$.  As in the proof of lemma \ref{l-s4-oc-grps}, $i_{\Omega}^{*}$ is an isomorphism.  By the K\"{u}nneth theorem,
\begin{align}
\HH^4 (\Omega S^4 \cross S^1) &\cong \HH^4 (\Omega S^4) \tensor \HH^0 (S^1) \dirsum \HH^3 (\Omega S^4) \tensor \HH^1 (S^1) \notag \\
&= \HH^3 (\Omega S^4) \tensor \HH^1 (S^1), \notag
\end{align}
so using property $2$ of the slant product in \cite[page~287]{Span66}, $s_{\Omega}^{*}$ is an isomorphism.  The proposition follows by commutativity.
\end{proof}

\begin{prop}\label{p-lm-c-coho-isom}
\index{cohomology!smooth vs. smooth free loop space}
\index{exponential map}
% (A title for this lemma seemed redundant).
The inclusion $i_{LM} \colon LM \rightarrow C(S^1, M)$ induces isomorphisms in cohomology.
\end{prop}
\begin{proof}
We will use \v{C}ech cohomology with constant integral coefficients, because open sets are prominent in its definition and we are concerned with topologies.  In fact, integral \v{C}ech cohomology is defined depending only on nerves of indexed covers in \citet[pages~233--237]{ES52} and \citet[page~282]{Dowk50}, aside from the ordering of covers by refinement; and by the first part of lemma \ref{l-cech-sing-coho}, this kind of cohomology is naturally isomorphic to our \v{C}ech cohomology with $\underline{\ZZ}$ coefficients.  We calculate \v{C}ech cohomology with $\underline{\ZZ}$ coefficients of $LM$ and $C(S^1, M)$ using the nerve of particular good covers; this is justified by appeal to lemma \ref{l-cech-coho-cove}, which implies that if the inclusion $i_{LM}$ induces an isomorphism in the cohomologies of the covers, then it induces an isomorphism in the cohomologies of the spaces.  Then again using lemma \ref{l-cech-sing-coho}, both parts this time, after showing that the two loop spaces have the desired point set topological properties, we show that fact that the inclusion $i_{LM}$ induces an isomorphism in \v{C}ech cohomology with $\underline{\ZZ}$ coefficients, implies the same in integral singular cohomology.

There are two spaces of loops of interest here.  As a shorthand in this proof, let's name them $L$ (not to be confused with the standard Lagrangian subspace $L$) for $LM$ with the \Frechet manifold topology on $LM$, hence $i_L$ for $i_{LM}$, and $C$ for $C(S^1, M)$ with the compact-open topology.

We know by proposition \ref{p-smth-free-loop-spac-frec-mfld} that $L = LM$ is metrizable, hence hereditarily paracompact, any open cover has a refinement that is a good cover (actually all nonempty intersections are contractible), and $i_L$ is continuous.

$C = C(S^1, M)$ is naturally a metric space choosing some Riemannian metric on $M$ \citep[pages~121--123]{Pete06}, and we will use the same Riemannian metric used for the definition of $L$.  Denote by $d_M$ the metric on $M$ corresponding to its Riemannian metric, and by $d_M^{+}$ the metric on $C$ made from $d_M$ by definition \ref{d-lspc-metr-top}.  The metric space topology agrees with the compact-open topology by lemma \ref{l-co-metr-top}.  Thus $C$ also is hereditarily paracompact \citep[page~979]{Ston48}.  Next we will see that $C$ is locally contractible, and in fact has a good cover.

By \citet[pages~85--87]{GHL04} and \citet[page~134]{Pete06} (also see \citet[page~278]{Berg03}), for any $0 < \delta$ less than the convexity radius of $M$, a number such that geodesic balls of smaller radius are geodesically convex, which is positive since $M$ is compact, for every $x \in M$ and $0 < \delta_1 \le \delta$, $\exp_x$ gives a diffeomorphism between $\B_{\delta_1} (0_{T_x M})$ and $\B_{\delta_1} (x)$, and for any $y, z \in \B_{\delta} (x)$, there is a unique $v \in T_y M$ with $\exp_y (v) = z$ such that $c_{y, z} \colon [0,1] \rightarrow \B_{\delta} (x)$, $c_{y, z} (s) = \exp_y (s v)$ is the unique geodesic of length (equal to the distance given by $d$) less than $2 \delta$ from $y$ to $z$, depending smoothly on $y$, $z$, and $s$, not depending on $x$ or $\delta$ except for its domain of definition.

This allows us to contract any open ball of radius $\delta$ in $C$ to any point in the ball. Given $\gamma_0 \in C$, the center of the ball that will be the domain of definition of the contraction, and $\gamma_1 \in \B_{\delta} (\gamma_0)$ to which we want to contract the ball, there is a homotopy $H_{\gamma_1} \colon [0, 1] \cross \B_{\delta} (\gamma_0) \rightarrow \B_{\delta} (\gamma_0)$, that for $\gamma \in \B_{\delta} (\gamma_0)$ has the properties $H_{ \gamma_1} (1, \gamma) = \gamma$, $H_{\gamma_1} (0, \gamma) = \gamma_1$, $H_{\gamma_1} (s, \gamma_1) = \gamma_1$, given by $H_{ \gamma_1} (s, \gamma) (t) = c_{\gamma_1 (t), \gamma (t)} (s)$.  Considered as a map to $C$, $H_{\gamma_1}$ is continuous because $\gamma_1 (t)$ and $\gamma (t)$ are continuous functions of $t$, $c_{y, z} (s)$ is a continuous function of $y$, $z$, $s$, so $H_{\gamma_1} (\cdot, \cdot) (\cdot) \colon [0, 1] \cross \B_{\delta} (\gamma_0) \cross S^1 \rightarrow M$ is continuous, and by \citet[page~261]{Dugu66} its adjoint $H_{\gamma_1} (\cdot, \cdot) \colon [0, 1] \cross \B_{\delta} (\gamma_0) \rightarrow C(S^1,M)$ is continuous.  Because for each $t \in S^1$, $H_{\gamma_1} (s, \gamma) (t) \in \B_{\delta} (\gamma_0 (t))$, $H_{\gamma_1} (s, \gamma) \in \B_{\delta} (\gamma_0)$.  Thus $\B_{\delta} (\gamma_0)$ is contractible to $\gamma_1$. Notice that $H_{\gamma_1}$ is defined independently of $\gamma_0$ and $\delta$ except to define its domain, and that for $0 < \delta_1 \le \delta$, it restricts to a deformation retraction of $\B_{\delta_1} (\gamma_0)$.  Note also that analogous results hold for open balls of any smaller radius.

We can also contract any intersection of open balls in $C$ to any point in the intersection.  Suppose $\gamma_1 \in \bigcap_{k = 2}^r \B_{\delta_k} (\gamma_k)$.  Since the intersection is contained in all the $\B_{\delta_k} (\gamma_k)$, $H_{\gamma_1}$ defined using any of the $\B_{\delta_k} (\gamma_k)$ restricts on their intersection to a function to their intersection, which is thus contractible.

Since the $\B_{\delta_1} (x)$, $0 < \delta_1 \le \delta$, $x \in M$ form a basis for the topology of $C$, any open cover of that space can be refined to a cover consisting of open balls; all of which, and all of their nonempty finite intersections, are contractible; a good cover. Thus $C$ has the property that any open cover can be refined to a good cover.

The two following lemmas will finish the proof.  Lemma \ref{l-good-cove-c-lm} will construct an indexed good cover $U_C$ of $C$, and a corresponding good cover $U_L$ of $L$, indexed by the same set.  Lemma \ref{l-lm-c-nerv-bij-non-empt} will show that the bijection between the covers due to their being indexed by the same set, takes nonempty intersections to nonempty intersections; i.e. is a bijection between the nerves of the two indexed covers.

For \v{C}ech cohomology with $\underline{\ZZ}$ coefficients, a cochain is a constant function on any nonempty intersection of good cover elements, since the intersection is contractible and thus connected.  Referring to definition \ref{d-orde-cech-coch}, for each injective $k + 1$ tuple $\sigma$ of indices corresponding to a nonempty intersection, the corresponding factor of the product $\CH^k (U_C, \underline{\ZZ})$ is $\ZZ$; for other $\sigma$, it is $\{ 0 \}$.  Similarly for $U_L$.  Thus, the bijection between covers gives an isomorphism of \v{C}ech cochain complexes, hence isomorphisms of \v{C}ech cohomologies of the respective covers, hence isomorphisms of the \v{C}ech cohomologies of the respective spaces, and thence of their singular cohomologies.
\end{proof}

\begin{lem}\label{l-good-cove-c-lm}
\index{good cover}
\index{loop spaces}
(Constructing of Corresponding Indexed Good Covers of $C(S^1, M)$ and $LM$).
Let $\delta = \epsilon / 2$, where $\epsilon$, less than the convexity radius of $M$, is as used in lemma \ref{l-loc-addn} by proposition \ref{p-smth-free-loop-spac-frec-mfld} (see notation \ref{n-smth-free-loop-spac-frec-mfld}) to define the \Frechet manifold structure on $LM$. Define $U_C = \{ \B_{\delta} (\alpha) \st \alpha \in i_{LM} (LM) \}$; this is an indexed good cover for $C(S^1, M)$.  Define
\begin{align}
U_L &= i_{LM}^* U_C = \{ i_{LM}^{-1} (U) \st U \in U_C \} = \{ U_{\alpha, 2} \st \alpha \in LM \}, \text{ where} \notag \\
U_{\alpha, 2} &= i_{LM}^{-1} (\B_{\delta} (\alpha)) \subset U_{\alpha} = i_{LM}^{-1} (\B_{\epsilon} (\alpha)), \notag
\end{align}
a chart domain for $LM$.  The set $U_L$ is an indexed good cover for $LM$.
\end{lem}
\begin{proof}
Use the notations $C = C(S^1, M)$, $L = LM$, and $i_L = i_{LM}$, and facts about the convexity radius, from the proof of proposition \ref{p-lm-c-coho-isom}.  The set $U_C$ is a cover for $C$ because $i_L (L)$ is dense in $C$ by \citet[pages~34--35,~49]{Hirs76}, and is a good cover because proposition \ref{p-lm-c-coho-isom} showed that nonempty finite intersections of open balls in $C$ are contractible.  Each $U_{\alpha, 2}$ is open since $i_L$ is continuous by proposition \ref{p-smth-free-loop-spac-frec-mfld} (see notation \ref{n-smth-free-loop-spac-frec-mfld}), and they cover $L$.

It remains to show that nonempty finite intersections $\bigcap_{k = 1}^r U_{\alpha_k, 2}$ are contractible.  This is true for the same basic reasons given in proposition \ref{p-lm-c-coho-isom} for elements of $U_C$.  There, the exponential map from the neighborhood of the zero section of $TM$ to the neighborhood of the diagonal in $M \cross M$ took open balls in $M$ that were time slices of open balls in $C$, to open balls in fibers of $TM$, where there is a linear structure that can be used in a homotopy.  The same exponential maps are used for both loop spaces, but we must show that the deformation retraction map yields smooth loops and is continuous in the \Frechet manifold topology on $L$.  The topology of $C$ was given, but the topology of $L$ is defined by using something like the loop of the exponential map as part of the chart map from what is defined to be an open set in $L$ (a $U_{\alpha}$) to a \Frechet space related to $TM$.

The deformation retraction is constructed in the same way as for $C$. Start by considering one $U_{\alpha, 2}$. Given some $\beta \in U_{\alpha, 2}$ to contract to, with the idea that when we consider multiple $U_{\alpha_k, 2}$, $\beta$ will be in their intersection, there is a homotopy $H_{\beta} \colon [0, 1] \cross U_{\alpha, 2} \rightarrow U_{\alpha, 2}$, given by $H_{\beta} (s, \gamma) (t) = c_{\beta (t), \gamma (t)} (s)$, for $\gamma \in U_{\alpha, 2}$.  We have that $H_{\beta} (1, \gamma) = \gamma$, $H_{\beta} (0, \gamma) = \beta$, $H_{\beta} (s, \beta) = \beta$.  For a given $s$ and $\gamma$, $H_{\beta} (s, \gamma)$ is a smooth loop (here we mean smooth as a map $S^1 \rightarrow M$, no \Frechet smooth concept involved) because $\beta (t)$ and $\gamma (t)$ are smooth functions of $t$, $c_{y, z} (s)$ is a smooth function of $y$, $z$, $s$, and so $H_{\beta} (s, \gamma) (t) = c_{\beta(t), \gamma(t)} (s)$ is a smooth function of $t$.  Because for each $t \in S^1$, $H_{\beta} (s, \gamma) (t) \in \B_{\delta} (\alpha (t))$, $H_{\beta} (s, \gamma) \in U_{\alpha, 2}$ .

We need to check that $H_{\beta} (\cdot, \cdot) \colon [0, 1] \cross U_{\alpha, 2} \rightarrow U_{\alpha, 2}$ is continuous, using the \Frechet manifold topology on $L$.  We have defined $H_{ \beta}$ in a way that makes the chart map for $U_{\beta}$ natural, since both use, for $t \in S^1$, $\exp_{\beta(t)}$.  Since $\beta$ is in the intersection of open balls of radius $\delta$, the distance from it to any other point in the intersection is less than $2 \delta = \epsilon$, the radius of $U_{\beta}$; hence the entire intersection, within which the contraction takes place, is in the domain of the chart map for $U_{\beta}$.

Let us review chart maps and transition functions for $L$.  For $\gamma \in U_{\beta}$, the first two steps in the chart map of proposition \ref{p-smth-free-loop-spac-frec-mfld} (see notation \ref{n-smth-free-loop-spac-frec-mfld}) for $U_{\beta}$ take $\gamma$ to $L(\pi, \eta)^{-1} (\beta, \gamma) \in LTM$, where $\eta = \sigma \circ \psi$, $(\pi, \sigma) \colon N \rightarrow V \subset M \cross M$ being the diffeomorphism given by the exponential map, from an $\epsilon$-neighborhood of the zero section of $TM$ to the corresponding neighborhood $V$ of the diagonal of $M \cross M$.  The map $\pi \colon TM \rightarrow M$ is the projection.  For $v \in N$, $\sigma (v) = \exp_{\pi(v)} (V)$.  The map $\psi \colon TM \rightarrow N$ is a diffeomorphism such that $\pi \circ \psi = \pi$, $\psi_{|Z} = \ident_Z$.  Thus
\begin{align}
L(\pi, \eta) &= L(\pi, \sigma) \circ \psi, \notag \\
(L(\pi, \eta))^{-1} &= \psi^{-1} \circ L(\pi, \sigma)^{-1}, \text{ and} \notag \\
L(\pi, \sigma)^{-1} ((\beta, \gamma)) &= (\pi, \sigma)^{-1} \circ (\beta, \gamma) \notag
\end{align}
in the chart map for $U_{\beta}$.

Checking what happens when contracting to $\beta$ an arbitrary $\gamma \in U_{\beta}$, substituting $H_{\beta} (s, \gamma)$ for $\gamma$ in the chart map expression, recalling that $v$ is defined by $\exp_{\beta (t)} (v) = \gamma (t)$,
\begin{align}
(\pi, \sigma)^{-1} ((\beta, H_{\beta} (s, \gamma)) (t)) &= (\pi, \sigma)^{-1} (\beta (t), \exp_{\beta (t)} (sv)) \notag \\
&= sv \notag \\
&= s (\pi, \sigma)^{-1} ((\beta, \gamma) (t)), \text{ so} \notag \\
(\pi, \sigma)^{-1} \circ (\beta, H_{\beta} (s, \gamma)) &= s (\pi, \sigma)^{-1} \circ (\beta, \gamma), \text{ or}  \notag \\
L(\pi, \sigma)^{-1} ((\beta, H_{\beta} (s, \gamma))) &= s L(\pi, \sigma)^{-1} ((\beta, \gamma)). \notag
\end{align}
Recalling that $(\pi, \sigma)^{-1} (\beta, \beta) = 0$, the homotopy contracts $\gamma$ to $\beta$ by multiplying the loop in $TM$ corresponding to $\gamma$ by the homotopy parameter $s$.  Scalar multiplication being a (jointly) continuous map of $LTM = TLM$, just as it is in a \Frechet space, and $L(\pi, \sigma)^{-1}$ being continuous by definition on chart domains, the homotopy $H_{ \beta}$ is continuous in the \Frechet manifold topology.  As in the similar proof for $C$, the homotopy doesn't depend on $\alpha$ except for its domain; it is defined in terms of $\beta$

Thus individual $U_{\alpha, 2}$ are contractible to any $\beta \in U_{\alpha, 2}$.  Any $\beta \in \bigcap_{k = 1}^r U_{\alpha_k, 2}$ has the property that $\bigcap_{k = 1}^r U_{\alpha_k, 2} \subset U_{\beta}$.  Since the intersection is contained in all the $U_{\alpha_k}$, $H_{\beta}$ defined using any of the $\alpha_k$, say $\alpha_1$, restricted to the intersection, is the same function as defined using any other of the $\alpha_k$, and maps the intersection to the intersection, which is thus contractible to $\beta$.
\end{proof}

\begin{lem}\label{l-lm-c-nerv-bij-non-empt}
\index{nerve}
\index{good cover}
\index{loop spaces}
(The Bijection between Corresponding Good Covers of $C(S^1, M)$ and $LM$).
Define $\delta$, $U_C$, and $U_L$ as in lemma \ref{l-good-cove-c-lm}.  Since the indexed good covers $U_C$, $U_L$ are indexed by the same set $LM$ (see definition \ref{d-inde-cove}), there is a natural bijection between them.  Under this bijection, injective $k + 1$ tuples of indices corresponding to nonempty intersections of elements of $U_C$ correspond also to nonempty intersections of elements of $U_L$.
\end{lem}
\begin{proof}
Use the notations $C = C(S^1, M)$, $L = LM$, and $i_L = i_{LM}$, and facts about the convexity radius, from the proof of proposition \ref{p-lm-c-coho-isom}.  Given $\alpha \in LM$, lemma \ref{l-good-cove-c-lm} defines corresponding elements $\B_{\delta} (i_L (\alpha)) \in U_C$ and $i_L^{-1} (\B_{\delta} (i_L (\alpha))) \in U_L$.

From properties of the inverse, as a map of sets, of a function,
\[
\bigcap_{k=1}^r i_L^{-1} (\B_{\delta} (i_L (\alpha_k))) = i_L^{-1} (\bigcap_{k=1}^r \B_{\delta} (i_L (\alpha_k))). \notag
\]
Thus, given an nonempty intersection of elements of $U_L$, $\bigcap_{k=1}^r i_L^{-1} (\B_{\delta} (i_L (\alpha_k))) \ne \emptyset$, the corresponding intersection of elements of $U_C$, $\bigcap_{k=1}^r \B_{\delta} (i_L (\alpha_k))$ is necessarily also nonempty, since otherwise its inverse image, the first intersection, couldn't be nonempty.

For the other direction, suppose given a nonempty intersection of elements of $U_C$, $\emptyset \ne \bigcap_{k=1}^r \B_{\delta} (i_L (\alpha_k)) \subset C$.  As noted in the proof of lemma \ref{l-good-cove-c-lm}, $i_L (L)$ is dense in $C$ by \citet[pages~34--35,~49]{Hirs76}, so there is some $\beta \in L$ such that $i_L (\beta) \in \bigcap_{k=1}^r \B_{\delta} (i_L (\alpha_k))$.  Then $\beta \in \bigcap_{k=1}^r i_L^{-1} (\B_{\delta} (i_L (\alpha_k)))$, the corresponding intersection of elements of $U_L$.
\end{proof}

\begin{lem}\label{l-tran-h4s4-to-h3ls}
\index{transgression!H4(S4) to H3(LS4)@$\HH^4(S^4;\ZZ) \rightarrow \HH^3(LS^4;\ZZ)$}
\index{transgression!isomorphism}
% (A title for this lemma seemed redundant).
The transgression $\tau \colon \HH^4(S^4;\ZZ) \isomto \HH^3(LS^4;\ZZ)$ is an isomorphism.
\end{lem}
\begin{proof}
Apply propositions \ref{p-tran-h4s4-to-h3c} and \ref{p-lm-c-coho-isom}.
\end{proof}

\section{Quaternionic Line Bundle}\label{s-quat-line-bndl}
This section discusses the tautological quaternionic line bundle $E_Q$.  The following definition and note are background for the definition of $E_Q$.
\begin{defn}\label{d-q-qu-su2}
\index{quaternions!R4@$\RR^4$}
\index{quaternions!C2@$\CC^2$}
\index{quaternions!unit!SU(2)@$\SU(2)$}
\index{SU(2)@$\SU(2)$}
(Quaternions).
\citep[page~17]{Laws85} The quaternions are defined as a free real algebra over four elements $1 \in \RR, \qi, \qj, \qk$ with relations:
\begin{align}
\QH &= \qone \RR \dirsum \qi \RR \dirsum \qj \RR \dirsum \qk \RR \notag \\
 &= \{ q = a + b \qi + c \qj + d \qk \st a, b, c, d \in \RR, \notag \\
 & \qi^2 = \qj^2 = \qk^2 = -1, \qi \qj = - \qj \qi, \qj \qk = - \qk \qj, \qk \qi = - \qi \qk, \qi \qj \qk = -1 \}, \text{ with involution} \notag \\
 & q = a + b \qi + c \qj + d \qk \mapsto \overline{q} = \overline{a + b \qi + c \qj + d \qk} = a - b \qi - c \qj - d \qk, \text{ and norm} \notag \\
 &\norm{q} = \sqrt{q \overline{q}}. \notag
\end{align}

$\QH$ can be identified as a normed division algebra with $\RR^4$ with standard norm via $a + b \qi + c \qj + d \qk \mapsto (a, b, c, d)$.

$\QH$ can be identified also with $\CC^2$ with standard norm via $a + b \qi + c \qj + d \qk \mapsto (a + bi, c + di) = (z, w)$, with induced multiplication $(z_1, w_1) (z_2, w_2) = (z_1 z_2 -w_1 \overline{w_2}, z_1 w_2 + w_1 \overline{z_2})$ and involution $(z, w) \mapsto (\overline{z}, -w)$.
\end{defn}
\begin{note}\label{n-q-qu-su2}
\index{quaternions!R4@$\RR^4$}
\index{quaternions!unit!S3@$S^3$}
\index{quaternions!unit!Sp1@$\Symp(1)$}
\index{S3@$S^3$}
\index{Sp1@$\Symp(1)$}
(The Inclusion of $\Symp(1)$ into $\SO(4)$).
\citep[page~17]{Laws85} The group of unit quaternions, $S^3$ or $\Symp(1)$, can be considered a subgroup of $\SO(4)$ by identifying $\QH$ with $\RR^4$ and letting $\Symp(1)$ act by left or right multiplication.  The standard inner product on $\RR^4$ applied to $q_1, q_2 \in \QH \cong \RR^4$ is equal to $\Re(q_1 \overline{q_2})$.  Note that $\Re(q_1 \overline{q_2})$ is unchanged if both $q_1$ and $q_2$ are multiplied on the left by $u \in \Sym(1) \subset \QH$ such that $u \overline{u} = 1$; and $q_1 \overline{q_2}$ itself is unchanged if both are multiplied on the right.

If $\QH$ is identified with $\CC^2$, $\Symp(1)$ can be identified with $\SU(2)$.
\end{note}

The following definition and lemma give the construction of the vector bundle $E_Q$.  The key points are that its $\frac{p_1}{2}$ is a generator of the fourth cohomology of its base space, and that it's a relatively simple bundle, making proofs practical.
\begin{defn}\label{d-q-vb}
\index{quaternionic line bundle}
\index{EQ@$E_Q$}
(The Quaternionic Tautological Line Bundle).
\citep[pages~243--244]{MS74} \citep[pages~17--18]{Laws85} Define $E_Q \rightarrow \QH P^1$ as the rank $4$ smooth real vector bundle underlying the tautological quaternionic line bundle, with fiberwise inner product defined by identifying $\QH^2$ with $\RR^8$.  In more detail, considering $\QH^2$ as a right $\QH$ module, let the total space $E_Q = \{ (l, v) \st v \in l, \text{ a } 1-\dim \text{ subspace } \in \QH^2 \}$, with projection $(l, v) \mapsto l$.  Given $(l, v_1), (l, v_2) \in (E_Q)_l$, with $v_i = (v_{i1}, v_{i2})$, define their inner product as $((l, v_1), (l, v_2)) = (v_1, v_2) = \Re(v_{11} \overline{v_{21}} + v_{12} \overline{v_{22}})$.
\end{defn}

\begin{lem}\label{l-q-vb}
\index{quaternionic line bundle}
(Quaternionic Tautological Line Bundle Properties).
\citep[pages~243--244]{MS74} $E_Q \rightarrow \QH P^1$ is indeed a rank $4$ smooth real vector bundle with fiberwise inner product.  We will identify $\QH P^1$ with $S^4$ via stereographic projection from $S^4 \subset \RR^5 \cong \QH \dirsum \RR$ to $\QH$ and thence by homogeneous coordinates to $\QH P^1$.  We give $\QH$ the Riemannian metric from this identification, with the standard Riemannian metric on $S^4$.  The Euler class of the underlying real bundle equals the second Chern class of the underlying complex bundle; call this $u_4$, a generator of $\HH^4 (\QH P^1; \ZZ)$.  Its total Chern class is $1 + u_4$ and its total Pontryagin class is $1 - 2 u_4$; $p_1 (E_Q) = - 2 u_4$.  The bundle $E_Q$ is orientable.

The symplectic frame bundle of $E_Q$, $\Symp(E_Q)$, is a principal bundle with group $\Symp(1) \subset \SO(4)$.  We will identify it with $S^7 \rightarrow \QH P^1$, formed by taking the unit sphere sub-bundle of $E_Q$ \citep[pages~11,18]{Laws85}, dropping the now unnecessary lines in the pairs forming the total space, and using as projection the quotient map for projective space, $\QH^2 \setminus \{ (0, 0) \} \xrightarrow{\pi_{\QH P^1}} \QH P^1$.  The smooth manifold $S^7$ has the standard Riemannian metric from inclusion in $\RR^8$.  The group of unit quaternions acts freely and transitively by right multiplication on the fibers of $S^7 \rightarrow S^4$.  This is called a Hopf bundle.

The transition function (see definition \ref{d-fb-vb}) for $E_Q$, for $S^4 = D^4_{+} \cup D^4_{-}$, the union of open neighborhoods of the upper and lower hemispheres, is determined over $D^4_{+} \cap D^4_{-} \cong S^3 \cross (-1, 1)$ by the map $g_{+-}$ to $S^3$, the unit
quaternions, such that $g_{+-} (x, t) = x$, acting by left multiplication on $\QH \cong \RR^4$
\citep[page~18]{Laws85}.
\end{lem}
The Riemannian metric on $\QH P^1$ is used to define the \Frechet manifold structure on $L\QH P^1$, as in proposition \ref{p-smth-free-loop-spac-frec-mfld}, for the bundle gerbe construction functor $G$.

\section{Generalization from Universal Case}\label{s-gene-univ}
Here we attempt to show that the conjectured transgression result for the particular vector bundle $E_Q$ implies the conjecture for the general bundles $E \rightarrow M$ we consider, justifying our calling $E_Q$ the universal case.  It may not work for the bundle gerbe construction of the thesis as it stands, but it seems likely it would if the bundles $E$ had spin structures, and the bundle gerbe were built from looping $\Spin(E)$ rather than $\SO(E)$, as discussed in note \ref{n-spin-stru}.

Classifying spaces are involved in the proof, and because the bundle gerbe construction functor requires a vector bundle with a compact base space, the proof uses spaces $SG_k (\RR^s)$ that are classifying for base spaces only up to a certain CW complex dimension \citep[pages~50,~96]{Huse94}; though this can be made arbitrarily large.  The $2$-torsion related to the cohomology of the special orthogonal groups may present a problem.  It appears that if the vector bundles $E$ had spin structures, and the work outlined in note \ref{n-spin-stru} were carried out as expected, then the spin groups would be involved instead, and there would be no torsion that would interfere with the proof.

Note that to avoid stating a false proposition, we state one that is possibly vacuously true due to its initial cohomological hypothesis being false.  The idea of spin groups enters only in a comment in the proof.
\begin{prop}\label{p-p-to-dd-tran}
\index{transgression!Pontryagin class!Dixmier-Douady class}
\index{Pontryagin class!transgression!Dixmier-Douady class}
\index{Dixmier-Douady class!transgression!Pontryagin class}
\index{EQ@$E_Q$!universal case}
\index{universal case}
(The $E_Q$ Case Implies Generally that the First Pontryagin Class Transgresses to Plus or Minus Twice the Dixmier-Douady Class).
Suppose that for all sufficiently large even $k \in \NN$, and given $k$, all sufficiently large $s \ge k + 6$, $\HH^3(LSG_k (\RR^s);\ZZ) \cong \ZZ$, where $SG_k (\RR^s) \cong \SO(s) / (\SO(s - k) \cross \SO(k))$.  Then if $\tau(p_1(E_Q)) = \pm 2 DD(G(E_Q))$, then for every vector bundle meeting our hypotheses, i.e. $E \rightarrow M$ an oriented smooth vector bundle of even rank with a fiberwise inner product over a compact connected oriented smooth manifold with Riemannian metric, $\tau(p_1(E)) = \pm 2 DD(G(E))$.
\end{prop}
\begin{proof}
When $E$ has rank $0$ the statement is trivially satisfied, so we may assume positive even rank.

We will use various facts about vector bundles universal up to a certain dimension, constructed as the canonical vector bundles over finite-dimensional Grassmann manifolds of oriented planes.  These are common knowledge to some, but references are: \citet[pages~54,~87--97]{Huse94}, \citet[pages~61--62,~145]{MS74}, \citet[pages~219--221]{NR04}, and \citet[pages~100--105]{Stee99}.

\citep[page~13--14,~25,~30--32,~54,~89--96]{Huse94} connects the ``$n$-universal'' principal bundles he mostly discusses, with vector bundles; for orientable ones, analogous logic applies.  He discusses oriented plane bundles at greater length in this context than do \citet[page~145]{MS74}.  Since the vector bundles (universal and other) we use are over compact base spaces, they are numerable \citep[page~50]{Huse94}, satisfying that hypothesis of his.  His ``$n$-universal'' means universal for bundles over CW complexes of dimension no greater than $n$ (this is not our $n = \rank (E)$).  Our vector bundles are over compact smooth manifolds, which by \citet[page~124]{Whit57} or \citet[pages~1--2]{Qin11} are CW complexes.  Also, the classifying maps in the references are continuous, but by \citet[page~213]{BT82}, they are continuously homotopic to smooth maps, suitable for the bundle gerbe construction functor.

For us, the most salient result is from \citep[page~96]{Huse94}, that for vector bundles associated to $\SO(k)$, the bundle
\begin{align}
E_{k, s} &= V_k (\RR^s) \cross_{\SO(k)} \RR^k \notag \\
 &\cong (\SO(s) / \SO(s - k)) \cross_{\SO(k)} \RR^k \notag \\
 &\rightarrow \SO(s) / (\SO(s - k) \cross \SO(k)) \notag \\
 &\cong SG_k (\RR^s) \notag
\end{align}
is $n$-universal for $s > k + n$.  This is a smooth vector bundle by \citet[pages~84--85]{Proc07}.  \citet[pages~219--221]{NR04}, who call $SG_k (\RR^s)$, $G_{+}(s,k)$, say it's simply connected and orientable.

Starting with a vector bundle $E$ of positive even rank $n$, choose even $k \ge n + 4$, sufficiently large that the proposition's hypothesis concerning $LSG_k (\RR^s)$ is true.  This allows including the rank $4$ bundle $E_Q$ in the following construction since $k \ge n + 4 > 4$, and also ensures by \citet[page~95]{Huse94} that $\pi_i (\SO(k)) \cong \pi_i (\SO)$ for $i \in [0, 4]$ since $k \ge n + 4 \ge 6 = 4 + 2 \ge i + 2$.

We require first that $s \ge k + 6$, because using the exact sequence of the fibration $\SO(k) \rightarrow V_k (\RR^s) \rightarrow SG_k (\RR^s)$  \citep[page~453]{Bred93}, the $+ 6$ ensures that $\pi_i (SG_k (\RR^s)) \cong \pi_{i - 1} (\SO(k)) \cong \pi_i (\SO) \cong \pi_i (B\SO)$ for $i \in [0, 4]$.  Further, choose $s$ sufficiently large that $E_{k, s} \rightarrow SG_k (\RR^s)$ is universal for both $E \dirsum I_{k - n}$ and $E_Q \dirsum I_{k - 4}$, where the $I$ are the product bundles with ranks as subscripts, over the base spaces of the vector bundles to which they're added.  Choose $s$ also sufficiently large that the proposition's hypothesis concerning $LSG_k (\RR^s)$ is true.  Then we have smooth orientation-preserving vector bundle morphisms that are linear isomorphisms on the fibers, from $E \dirsum I_{k - n}$ and $E_Q \dirsum I_{k - 4}$ to $E_{k, s} \rightarrow SG_k (\RR^s)$.

Since by corollary \ref{co-inde-fibe-inne-prod}, different fiberwise inner products on $E$ result in stably isomorphic bundle gerbes, we may induce fiberwise inner products on $E \dirsum I_{k - n}$ and $E_Q \dirsum I_{k - 4}$ from that on $E_{k, s}$, ensuring that those vector bundle morphisms are isometric on fibers, as needed to apply the functor $G$.

Thus we have the following commutative diagram,
\[
\begindc{\commdiag}[5]
\obj(10,20)[objEn]{$E \dirsum I_{k - n}$}
\obj(10,10)[objM]{$M$}
\obj(30,20)[objESOk]{$E_{k, s}$}
\obj(30,10)[objBSOk]{$SG_k (\RR^s)$}
\obj(50,20)[objEQ]{$E_Q  \dirsum I_{n - 4}$}
\obj(50,10)[objS4]{$S^4$}
\mor{objEn}{objM}{}
\mor{objESOk}{objBSOk}{}
\mor{objEQ}{objS4}{}
\mor{objEn}{objESOk}{}
\mor{objM}{objBSOk}{$f$}
\mor{objEQ}{objESOk}{}
\mor{objS4}{objBSOk}{$g$}[\atright, \solidarrow]
\obj(60,15)[objDiagram]{*}
\enddc
\]
and applying the bundle gerbe construction functor $G$, obtain also the following commutative diagram, with horizontal arrows belonging to bundle gerbe morphisms:
\[
\begindc{\commdiag}[5]
\obj(10,20)[objGEn]{$G(E \dirsum I_{k - n})$}
\obj(10,10)[objLM]{$LM$}
\obj(30,20)[objGESOk]{$G(E_{k, s})$}
\obj(30,10)[objLBSOk]{$LSG_k (\RR^s)$}
\obj(50,20)[objGEQ]{$G(E_Q \dirsum I_{k - 4})$}
\obj(50,10)[objLS4]{$LS^4$}
\mor{objGEn}{objLM}{}
\mor{objGESOk}{objLBSOk}{}
\mor{objGEQ}{objLS4}{}
\mor{objGEn}{objGESOk}{}
\mor{objLM}{objLBSOk}{$Lf$}
\mor{objGEQ}{objGESOk}{}
\mor{objLS4}{objLBSOk}{$Lg$}[\atright, \solidarrow]
\obj(60,15)[objDiagram]{**}
\enddc
\]
Recalling proposition \ref{p-bg-cons-stab} and note \ref{n-stab-isom}, $DD(G(E \dirsum I_{k - n})) \cong DD(G(E))$ and $DD(G(E_Q \dirsum I_{k - 4})) \cong DD(G(E_Q))$.

In the following diagram, proposition \ref{p-tran-hkp1-m-to-hk-lm} provides the transgression homomorphisms, all denoted $\tau$, and the commutativity.  Proposition \ref{p-tran-h4s4-to-h3c} says that the rightmost $\tau$ is an isomorphism.
\[
\begindc{\commdiag}[5]
\obj(10,20)[objM]{$\HH^4(M;\ZZ)$}
\obj(10,10)[objLM]{$\HH^3(LM;\ZZ)$}
\obj(40,20)[objBSOk]{$\HH^4(SG_k (\RR^s);\ZZ)$}
\obj(40,10)[objLBSOk]{$\HH^3(LSG_k (\RR^s);\ZZ)$}
\obj(70,20)[objS4]{$\HH^4(S^4;\ZZ)$}
\obj(70,10)[objLS4]{$\HH^3(LS^4;\ZZ)$}
\mor{objM}{objLM}{$\tau$}
\mor{objBSOk}{objLBSOk}{$\tau$}
\mor{objS4}{objLS4}{$\tau$}
\mor{objBSOk}{objM}{$f^{*}$}[\atright, \solidarrow]
\mor{objLBSOk}{objLM}{$(Lf)^{*}$}[\atright, \solidarrow]
\mor{objBSOk}{objS4}{$g^{*}$}
\mor{objLBSOk}{objLS4}{$(Lg)^{*}$}
\obj(80,15)[objDiagram]{***}
\enddc
\]

First we show that
\[
\tau(p_1(E_Q)) = \pm 2 DD(G(E_Q)) \Rightarrow \pm 2 DD(G(E_{k, s})) = \tau (p_1 (E_{k, s})). \notag
\]
By naturality of the Pontryagin class \citep[page~174]{MS74} and commutativity of diagram (***),
\[
\tau (p_1(E_Q)) = \tau (g^{*} p_1 (E_{k, s})) = (Lg)^{*} \tau (p_1 (E_{k, s})). \notag
\]
By commutativity of diagram (**) and naturality of the Dixmier-Douady class (lemma \ref{l-bg-dd-prop}),
\[
(Lg)^{*} 2 DD(G(E_{k, s})) = 2 (Lg)^{*} DD(G(E_{k, s})) = 2 DD(G(E_Q)). \notag
\]
Thus, using our hypothesis,
\[
(Lg)^{*} \tau (p_1 (E_{k, s})) = (Lg)^{*} (\pm 2 DD(G(E_{k, s}))), \notag
\]
and so $\pm 2 DD(G(E_{k, s})) = \tau (p_1 (E_{k, s}))$ will be true if $(Lg)^{*}$ is injective.

In diagram (***), the cohomology groups in the rightmost column are isomorphic to $\ZZ$ and the rightmost $\tau$ is an isomorphism, by lemma \ref{l-tran-h4s4-to-h3ls}.   From \citet[page~229]{NR04} the free part of $\HH^4(SG_k (\RR^s);\ZZ)$ is isomorphic to $\ZZ$.  Thus, since $0 \ne p_1(E_Q) = g^{*} p_1(E_{k, s}) \in g^{*}(\HH^4(SG_k (\RR^s);\ZZ))$, and maps of torsion into $\ZZ$ are $0$, $g^{*}$ is injective on the free part of $\HH^4(SG_k (\RR^s);\ZZ)$, and so is $\tau \circ g^{*} = (Lg)^{*} \circ \tau$.

If $\HH^3(LSG_k (\RR^s);\ZZ) \cong \ZZ$, then $(Lg)^{*}$ itself would be injective on the free part of $\HH^3(LSG_k (\RR^s);\ZZ)$, which would complete the proof that $\pm 2 DD(G(E_{k, s})) = \tau (p_1 (E_{k, s}))$.

As a comment, the hypothesis of the proposition states that $\HH^3(LSG_k (\RR^s);\ZZ) \cong \ZZ$, but it seems from brief preliminary investigation that there may be $2$-torsion making the hypothesis false.  On the other hand, from an even more cursory look, it appears that the corresponding statement likely would be true, for the third cohomology of the loop of a compact space suitably classifying for spin bundles, which we would be dealing with if the work in note \ref{n-spin-stru} were carried out.   \citet[pages~103--104]{Stee99} gives a way to define such a space.  In this case, by \citet[page~148]{McLa92}, $p_1 (E)$ is divisible by $2$, and we could reason that $DD (G(E_Q)) = \tau (\frac{p_1}{2} (E_Q))$ implies the same for the general case.

Assuming the hypothesis and its consequence just shown, we then show that
\[
\pm 2 DD(G(E_{k, s})) = \tau (p_1 (E_{k, s})) \Rightarrow \pm 2 DD(G(E)) = \tau (p_1 (E)). \notag
\]
By naturality of the Pontryagin class \citep[page~174]{MS74}, commutativity of diagrams (***) and (**), and naturality of the Dixmier-Douady class (lemma \ref{l-bg-dd-prop}),
\begin{align}
\tau (p_1(E)) &= \tau (f^{*} p_1 (E_{k, s})) \notag \\
 &= (Lf)^{*} \tau (p_1 (E_{k, s})) \notag \\
 &= (Lf)^{*} (\pm 2 DD(G(E_{k, s}))) \notag \\
 &= \pm 2 (Lf)^{*} DD(G(E_{k, s})) \notag \\
 &= \pm 2 DD(G(E)). \notag
\end{align}
\end{proof}

\section{Dixmier-Douady Class From Quaternionic Line Bundle}\label{s-dixm-doua-quat}

This section gives an overview of an approach to proving the conjecture that $\tau (p_1(E_Q)) = \pm 2 DD(G(E_Q))$.  The idea is to compute the two sides of the equation, reducing the problem to the determination of the first Chern class $c_1 (Q)$ of a principal $\UU(1)$ bundle $Q \rightarrow S^2$.

Not detailed in the thesis, this was reduced again to finding the degree of an explicit map of $\UU(1)$ torsors, a calculation that became difficult.

Also not detailed in the thesis is an approach that might view $Q \rightarrow S^2$, or the related complex line bundle, as the pullback of the Pfaffian line bundle over the restricted isotropic Grassmannian as in \citet[page~808]{SW07}, via a map whose image is in the connected component of the identity, at least if a spin structure is used as in note \ref{n-spin-stru}.  Since the Pfaffian line bundle freely generates the group of complex line bundles over the that component of the Grassmannian, the pullback map would determine $c_1 (Q)$.

Returning to the start, the following takes care of the left hand side of the equation, so that what remains is to show that $DD(G(E_Q))$ is a generator.
\begin{lem}\label{l-tau-isom-qhp1-tau-p1eq-2gen}
\index{transgression}
\index{tau@$\tau$}
\index{transgression!p12geq@$\frac{p_1}{2}(G(E_Q))$}
The transgression $\tau \colon \HH^4(\QH P^1; \ZZ) \rightarrow \HH^3(L\QH P^1; \ZZ)$ is an isomorphism.  Also, $\tau (p_1 (E_Q))$ is $2$ times a generator of $\HH^3(L\QH P^1)$.
\end{lem}
\begin{proof}
For the first statement, $\tau' \colon \HH^4(S^4;\ZZ) \isomto \HH^3(LS^4;\ZZ)$ is an isomorphism by lemma \ref{l-tran-h4s4-to-h3ls}.  Lemma \ref{l-diff-s4-hp1} gives the diffeomorphism $\theta \colon S^4 \rightarrow \QH P^1$; $L\theta$ is also a diffeomorphism by lemma \ref{p-loop-smth-map-smth}.  By naturality of the transgression in proposition \ref{p-tran-hkp1-m-to-hk-lm}, $\tau \colon \HH^4(\QH P^1; \ZZ) \rightarrow \HH^3(L\QH P^1; \ZZ)$ equals $((L\theta)^{*})^{-1} \circ \tau' \circ \theta^{*}$; hence $\tau$ is an isomorphism.

For the second statement, from lemma \ref{l-q-vb} we know that $\frac{p_1}{2} (E_Q)$ is a generator of $\HH^4(\QH P^1;\ZZ) \cong \ZZ$.  Thus $\tau (\frac{p_1}{2} (E_Q))$ is a generator of $\HH^3(L\QH P^1)$.
\end{proof}

\begin{cor}\label{co-dd-geq-gen-imply-univ-case}
\index{EQ@$E_Q$!universal case}
\index{universal case}
$DD(G(E_Q))$ is a generator $\Rightarrow \tau(p_1(E_Q)) = \pm 2 DD(G(E_Q))$.
\end{cor}

Then, to simplify the calculation of $DD(G(E_Q))$, the bundle gerbe over the \Frechet manifold $L\QH P^1 \cong LS^4$ is pulled back to $S^3$.  To further simplify, lemma \ref{l-eq-bg-pb} uses proposition \ref{p-dd-bg-susp-c1-pb} to construct from the pullback of the bundle gerbe to $S^3$, the principal $\UU(1)$ bundle $Q \rightarrow S^2$, giving the suspension isomorphism of corollary \ref{co-susp-isom} relating the Dixmier-Douady class of the pullback bundle gerbe, in $\HH^3(S^3;\ZZ) \cong \CH^2(S^3;\underline{\UU(1)})$, to the first Chern class of the principal $\UU(1)$ bundle, $c_1 (Q) \in \HH^2(S^2;\ZZ) \cong \CH^1(S^2;\underline{\UU(1)})$.

The first Chern class $c_1 (Q)$ corresponds to the degree of the transition function of $Q$, a continuous map of $\UU(1)$ torsors.  In work not shown here, a fairly explicit form was found for this map.  It started with finite but unwieldy Fourier series, after which quadratic exponentials (see proposition \ref{p-quad-exp}) would yield intertwiners of Fock representations.  After a while, this approach was discontinued.

The approach using the Pfaffian line bundle to see that $c_1 (Q)$ is a generator was not pursued but might go like this.  Something reminiscent of that line bundle, in principal $\UU(1)$ bundle form, could be obtained by taking our $T$ bundle over $\Lagr_{res} \cross \Lagr_{res}$, choosing a fixed standard Lagrangian subspace $L$, and constructing the sub-bundle of $T$ over $\{ L \} \cross \Lagr_{res}$.  Also, the maps in following sections constructed to pull back the bundle gerbe and get sections of the pullback's $Y$ space, could perhaps be used to give an isomorphism between $Q$, and the pullback via
\[
S^2 \rightarrow LS^3 \rightarrow L\Symp(1) \rightarrow L\SO(4) \rightarrow \Orth_{res} \rightarrow \Lagr_{res} \rightarrow \{ L \} \cross \Lagr_{res} \notag
\]
of the sub-bundle of $T$.  This might require a spin structure on $E_Q$ as in note \ref{n-spin-stru}.  There are many items to be filled in, yet this approach could be workable.

However, proved by whatever means, if $c_1 (Q)$ is a generator, the following proposition shows that the conjecture for $E_Q$ is true.
\begin{prop}\label{p-tran-p12-eq-to-dd-1}
\index{transgression!H4(LHP1) to H3(LHP1)@$\HH^4(L\QH P^1;\ZZ) \rightarrow \HH^3(L\QH P^1;\ZZ)$}
If the first Chern class $c_1 (Q)$ of the principal $\UU(1)$ bundle $Q \rightarrow S^2$ constructed by lemma \ref{l-eq-bg-pb} from the bundle gerbe $G(E_Q)$ is a generator, then the transgression of $p_1 (E_Q)$ is $\pm 2 DD(G(E_Q))$.
\end{prop}
\begin{proof}
By corollary \ref{co-dd-geq-gen-imply-univ-case}, the statement about the trangression, which is equivalent to commutativity up to sign of square \textcircled{1} in the following diagram, is implied by $DD(G(E_Q))$ being a generator of $\HH^3(L\QH P^1;\ZZ) \cong \ZZ$.  If we can show commutativity of squares \textcircled{2} and \textcircled{3}, and that the map $i^*$ in cohomology is an isomorphism, $DD(G(E_Q))$ being a generator of $\HH^3(L\QH P^1;\ZZ) \cong \ZZ$ is a consequence of $c_1 (Q)$ being a generator of $\HH^2(S^2)$.
\[
\begindc{\commdiag}[5]

\obj(0,20)[objPBS2]{$Q$}
\obj(25,20)[objBGS3]{$i^* G(E_Q)$}
\obj(45,20)[objBGLHP1]{$G(E_Q)$}
\obj(65,20)[objVBHP1]{$E_Q$}

\obj(0,10)[objH2S2]{$\HH^2(S^2)$}
\obj(13,15)[obj5]{\textcircled{3}}
\obj(25,10)[objH3S3]{$\HH^3(S^3)$}
\obj(36,15)[obj6]{\textcircled{2}}
\obj(45,10)[objH3LHP1]{$\HH^3(L\QH P^1)$}
\obj(56,15)[obj5]{\textcircled{1}}
\obj(65,10)[objH4HP1]{$\HH^4(\QH P^1)$}

\mor{objBGS3}{objPBS2}{de-clutch}[\atright, \solidarrow]
\mor{objBGLHP1}{objBGS3}{$i^*$}[\atright, \solidarrow]
\mor{objVBHP1}{objBGLHP1}{$G$}[\atright, \solidarrow]

\mor{objH2S2}{objH3S3}{$\cong_{\Sigma}$}[\atright, \solidarrow]
\mor{objH3LHP1}{objH3S3}{$i^*$}
\mor{objH4HP1}{objH3LHP1}{$\tau$}

\mor{objPBS2}{objH2S2}{$-c_1$}[\atright, \solidarrow]
\mor{objBGS3}{objH3S3}{$DD$}
\mor{objBGLHP1}{objH3LHP1}{$DD$}
\mor{objVBHP1}{objH4HP1}{$\frac{p_1}{2}$}

\enddc
\]
In this diagram some of the arrows are maps of cohomology groups; others, labelled $G$, $c_1$, $DD$, and $\frac{p_1}{2}$, are applications of functors; and de-clutch is part of the construction that proves commutativity of its square.  ``$c_1$'' means first Chern class, ``$DD$'' Dixmier-Douady class, ``$p_1$'' first Pontryagin class, ``$\cong_{\Sigma}$'' (unreduced) suspension isomorphism, ``bg'' bundle gerbe, and ``$/$'' over.

Square \textcircled{2} commutes because the pullback of the Dixmier-Douady class of a bundle gerbe is naturally isomorphic to the Dixmier-Douady class of the pullback of the bundle gerbe, by lemma \ref{l-bg-dd-prop}.  The pullback is via the map here temporarily denoted $i = L\theta \circ LR_{S^4} \circ i_3 \colon S^3 \rightarrow L\QH P^1$, which by lemma \ref{l-incl-s3-ls4-cont} and the facts that $R_{S^4}$ (lemma \ref{l-loc-sect-lspe}) and $\theta$ (lemma \ref{l-diff-s4-hp1}) are diffeomorphisms, is continuous and induces an isomorphism in $\HH^3$.

The (unreduced) suspension isomorphism of \textcircled{3} is given by corollary \ref{co-susp-isom}, and commutativity of that square is shown in proposition \ref{p-dd-bg-susp-c1-pb}.
\end{proof}

The following section \ref{s-dd-clas-bg-susp} gives the lemma and proposition for square \textcircled{3}, then sections \ref{s-map-s3-ls4-ster-proj-homo-coor} and \ref{s-bg-cons-pb-loc-sect-pb} give the lemmas for square \textcircled{2}.

\section{Dixmier-Douady Class of Bundle Gerbe over Suspension}\label{s-dd-clas-bg-susp}

In the following definitions, lemma, and corollary, used in the subsequent proposition \ref{p-dd-bg-susp-c1-pb}, we develop from scratch an exact sequence, the connecting homomorphism of the long exact sequence of which is an isomorphism, following to an extent the Mayer-Vietoris sequence (26) and (unreduced) suspension isomorphism example in \citet[pages~94--95]{Bred97}, quoting results from section \ref{s-cech-coho} to relate the cohomology groups of good covers to the \v{C}ech cohomology of the spaces.  First some definitions to avoid a long lemma statement.

\begin{defn}\label{d-susp-isom-cov}
\index{suspension!unreduced!cover}
(The Cover for the Suspension Isomorphism).
Suppose $X$ is a topological space with a good cover $U$ indexed by a set $I$ disjoint from $\{ \pm 1 \}$.  Define the closed cones and (unreduced) suspension of $X$:
\begin{align}
CX^{-1} &= (X \cross [{-1}, 0]) / (X \cross \{ {-1} \}) \notag \\
CX^1 &= (X \cross [0, 1]) / (X \cross \{ 1 \}) \notag \\
X &= X \cross \{ 0 \} = CX^{-1} \cap CX^1 \notag \\
\Sigma X &= (X \cross [{-1}, 1]) / (X \cross \{ {-1} \} \cup X \cross \{ 1 \}) = CX^{-1} \cup CX^1. \notag
\end{align}
Use the following notation for the inclusion maps:
\begin{align}
&X \xrightarrow{s_k} CX^k \xrightarrow{r_k} \Sigma X, \text{ } k \in \{ \pm 1 \}. \notag
\end{align}
Define a good cover $U^{\Sigma}$ of the suspension by extruding the good cover of $U$ halfway to either end, and adding a cap to each end, overlapping with the extruded cover.  This cover is indexed by the set $J = \{ \pm 1 \} \cup I$:
\begin{align}
U^{\Sigma}_{-1} &= (X \cross [-1, -\frac{1}{4})) / (X \cross \{ {-1} \}) \notag \\
U^{\Sigma}_{1} &= (X \cross (\frac{1}{4}, 1]) / (X \cross \{ 1 \}) \notag \\
U^{\Sigma}_j &= U_j \cross (-\frac{1}{2}, \frac{1}{2}), \text{ } j \in I \notag
\end{align}
Pull back $U^{\Sigma}$ by the inclusion maps to get indexed good covers of the cones,
\[
r_k^* U^{\Sigma}, \text{ } k \in \{ \pm 1 \} \notag
\]
the cover for each cone including one empty set from the end cap for the other cone with index $+1$ or $-1$, and then pull back again to get an indexed good cover of $X$, which we will denote differently since we won't mention chain maps induced by $r_k \circ s_k$ (equal for either $k$).
\[
U^{\Sigma} \cap X = (r_k \circ s_k)^* U^{\Sigma}, \notag
\]
equal to $U$ plus two empty sets from the end caps with indices $\pm 1$.  The pullbacks of the covers are given by intersection of the elements of the covers with the sets the covers are being pulled back to.
\end{defn}

\begin{defn}\label{d-susp-isom-chai-comp}
\index{suspension!unreduced!chain complexes}
(Chain Complexes for the Suspension Isomorphism).
Continuing within the context of definition \ref{d-susp-isom-cov}, as a shorthand, define the following chain complexes of abelian groups as the alternating \v{C}ech cochain complexes of definition \ref{d-orde-cech-coch}:
\begin{align}
C^{*} (U^{\Sigma}) &= \check{C}^* (U^{\Sigma}, \underline{\UU(1)}) \notag \\
C^{*} (r_k^* U^{\Sigma}) &= \check{C}^* (r_k^* U^{\Sigma}, \underline{\UU(1)}) \text{ for } k \in \{ \pm 1 \} \notag \\
C^{*} (U^{\Sigma} \cap X ) &= \check{C}^* (U^{\Sigma} \cap X , \underline{\UU(1)}). \notag
\end{align}
\end{defn}
Elements of $C^p (U^{\Sigma})$ are alternating $p$-cochains $c$ defined on ordered $(p+1)$-tuples $\sigma \colon [0,p] \rightarrow J$ given by $\sigma = (i_0,\dots,i_p)$, with value $c(\sigma)$ a continuous $U(1)$-valued function on
\[
U^{\Sigma}_{\sigma} = U^{\Sigma}_{i_0,\dots,i_p} = \bigcap_{k \in [0,p]} U^{\Sigma}_{\sigma(k)} = U^{\Sigma}_{i_0} \cap \cdots \cap U^{\Sigma}_{i_p}. \notag
\]
Similarly for the other cochain complexes; their good covers share the same index set $J$ used for $U^{\Sigma}$.

\begin{lem}\label{l-susp-isom-exac-seq}
\index{suspension!unreduced!isomorphism}
\index{Mayer-Vietoris}
\index{good cover}
(The Exact Sequence for the Suspension Isomorphism).
Suppose $X$ is a topological space with a good cover $U$ indexed by a set $I$ disjoint from $\{ \pm 1 \}$.  Using the good cover $U^{\Sigma}$ and inclusion maps $r_k$, $s_k$ of \ref{d-susp-isom-cov}, and the cochain complexes of definition \ref{d-susp-isom-chai-comp}, the following sequence of chain complexes, similar to what is used for a Mayer-Vietoris sequence, is exact.
\[
0 \rightarrow C^{*} (U^{\Sigma}) \xrightarrow{(r_{-1}^{\sharp}, r_1^{\sharp})} C^{*} (r_{-1}^* U^{\Sigma}) \dirsum C^{*} (r_1^* U^{\Sigma}) \xrightarrow{s_{-1}^{\sharp} - s_1^{\sharp}} C^{*} (U^{\Sigma} \cap X ) \rightarrow 0.  \notag
\]
\end{lem}
\begin{proof}
Referring to definition \ref{d-orde-cech-coch}, for each degree $p$, the abelian groups in the cochain complexes are products, and the chain maps induced by the inclusions are the products of the maps induced by the inclusions, for the component abelian groups.  These products are over the same set of $p + 1$-tuples of elements of the same index set $J$ for each complex.  Thus, exactness of the short exact sequence of cochain complexes is equivalent to exactness for each degree $p \ge 0$, for each component of the product abelian groups; that is, for each fixed ordered index $p + 1$-tuple $\sigma = (i_0 \dots i_p)$.

Thus for the proof of exactness, fix $p$ and $\sigma$, and denote the intersection of elements of $U^{\Sigma}$ corresponding to $\sigma$ by $W = (\bigcap_{j = 0}^p U^{\Sigma}_j)$.  The cover pullback map $r_k^*$ maps $W$ to $W \cap CX^k$, $k \in \{ \pm 1 \}$, and $s_k^*$ maps $W \cap CX^k$ to $W \cap X$.  We will show exactness of
\begin{align}
0 \rightarrow \underline{\UU(1)}_{U^{\Sigma}} (W) &\xrightarrow{(r_{-1}^{\sharp}, r_1^{\sharp})} \underline{\UU(1)}_{r_{-1}^* U^{\Sigma}} (W \cap CX^{-1}) \dirsum \underline{\UU(1)}_{r_1^* U^{\Sigma}} (W \cap CX^1) \notag \\ &\xrightarrow{s_{-1}^{\sharp} - s_1^{\sharp}} \underline{\UU(1)}_{U^{\Sigma} \cap X } (W \cap X) \rightarrow 0 \text{, or in other words,} \notag \\
0 \rightarrow C(W, \UU(1)) &\xrightarrow{(r_{-1}^{\sharp}, r_1^{\sharp})} C (W \cap CX^{-1}, \UU(1)) \dirsum C (W \cap CX^1, \UU(1)) \notag \\
&\xrightarrow{s_{-1}^{\sharp} - s_1^{\sharp}} C (W \cap X, \UU(1)) \rightarrow 0, \notag
\end{align}
where for this purpose $C(Y, \UU(1))$ denotes the abelian group of continuous functions $Y \rightarrow \UU(1)$.

To see surjectivity of $s_{-1}^{\sharp} - s_1^{\sharp}$, consider $z \in C (W \cap X, \UU(1))$.  If $z = 0$ there is no problem.  Otherwise, $z \ne 0 \Rightarrow \{ \pm 1 \} \cap \im (\sigma) = \emptyset$, where $\sigma$ is thought of as a function $\{ 0 \dots p \} \rightarrow J$.  Thus we can set $w_{-1} \colon W \cap CX^{-1} \rightarrow \UU(1)$, $w_{-1} ([x, t]) = z ([x,0]) = z (x)$ according to our identification of $X \cross \{ 0 \}$ with $X$, $x \in W \cap X$, $t \in (-\frac{1}{2}, 0]$, and $w_1 = 0$, resulting in $s_{-1}^{\sharp} (w_{-1}) = z$ and $s_1^{\sharp} (w_1) = 0$.

To see that $(r_{-1}^{\sharp}, r_1^{\sharp})$ is injective, recall that $\Sigma X = CX^{-1} \cup CX^1$.  Suppose $y \in C (W, \UU(1))$.  That $r_{-1}^{\sharp} (y) = 0$ means that $y = 0$ on $W \cap CX^{-1}$ or $W \cap CX^{-1}$ is empty, and $r_1^{\sharp} (y) = 0$ means similarly that $y(\sigma) = 0$ or $W \cap CX^1$ is empty.  Thus $y = 0$ on all of $W$.

To see exactness in the middle, suppose that $(w_{-1}, w_1) \in C (W \cap CX^{-1}, \UU(1)) \dirsum C (W \cap CX^1, \UU(1))$ and $s_{-1}^{\sharp} (w_{-1}) - s_1^{\sharp} (w_1) = 0$; i.e. $w_{-1} = w_1$ on $W \cap X$.  Define $y \in C (W, \UU(1))$ by setting $y(\sigma)$ equal to $w_{-1}$ on $W \cap CX^{-1}$, and $y(\sigma) = w_1$ on $W \cap CX^1$; a continuous function because the intersection of those domains is $W \cap X$, the two functions agree there, and the union of the domains is $W$.
\end{proof}
\begin{note}\label{n-susp-isom}
\index{suspension!isomorphism!notes}
(Notes for the Use of the Exact Sequence for the Suspension Isomorphism).
For convenience when using the lemma, we record the following.  The connecting map $\partial \colon \CH^p (U^{\Sigma} \cap X ) \rightarrow \CH^{p+1} (U^{\Sigma})$ of the corresponding long exact sequence in cohomology is as follows, following the standard diagram chase but being specific about choices.  Differently from in the proof of the lemma, the symbols $z, w_{-1}, w_1, y$ here will indicate cochains rather than their components on particular $\sigma$.

Given a nonzero cocycle $z \in C^p (U^{\Sigma} \cap X )$, since $s_{-1}^{\sharp} - s_1^{\sharp}$ is surjective, there is a cochain $(w_{-1}, w_1) \in C^p (r_{-1}^* U^{\Sigma}) \dirsum C^p (r_1^* U^{\Sigma})$ that $z$ comes from; i.e. $s_{-1}^{\sharp} (w_{-1}) - s_1^{\sharp} (w_1) = z$.  The choice of $(w_{-1}, w_1)$ in the proof was made by letting $w_{-1}$ be an extension of $z$, and $w_1 = 0$.  In using this lemma, any convenient means may be used to obtain $(w_{-1}, w_1)$ such that $(s_{-1}^{\sharp}, s_1^{\sharp}) (w_{-1}, w_1) = z$.  We have that $z$ is a cocycle and the square below is commutative, so $(s_{-1}^{\sharp} - s_1^{\sharp}) (\delta (w_{-1}), \delta (w_1)) = 0$, whence $(\delta (w_{-1}), \delta (w_1)) = (r_{-1}^{\sharp} (y), r_1^{\sharp} (y))$ for what, given choices of $w_{-1}$ and $w_1$, is a unique $y \in C^{p+1} (U^{\Sigma})$, which as in the standard argument, is a cocycle.
\[
\begindc{\commdiag}[5]
\obj(70,20)[objCpUX]{$C^p (U^{\Sigma} \cap X )$}
\obj(35,20)[objCpU01]{$C^p (r_{-1}^* U^{\Sigma}) \dirsum C^p (r_1^* U^{\Sigma})$}
\obj(70,10)[objCpp1UX]{$C^{p+1} (U^{\Sigma} \cap X )$}
\obj(35,10)[objCpp1U01]{$C^{p+1} (r_{-1}^* U^{\Sigma}) \dirsum C^{p+1} (r_1^* U^{\Sigma})$}
\obj(0,10)[objCpp1US]{$C^{p+1} (U^{\Sigma})$}
\mor{objCpU01}{objCpUX}{$s_{-1}^{\sharp} - s_1^{\sharp}$}
\mor{objCpp1U01}{objCpp1UX}{$s_{-1}^{\sharp} - s_1^{\sharp}$}
\mor{objCpUX}{objCpp1UX}{$\delta$}
\mor{objCpU01}{objCpp1U01}{$\delta$}
\mor{objCpp1US}{objCpp1U01}{$(r_{-1}^{\sharp}, r_1^{\sharp})$}
\enddc
\]
\end{note}

\begin{cor}\label{co-susp-isom}
\index{suspension!unreduced!isomorphism}
\index{Mayer-Vietoris}
\index{good cover}
(The Unreduced Suspension Isomorphism).
Suppose $X$ is a locally connected metric space with a good cover $U$.  Given the short exact sequence of cochain complexes of lemma \ref{l-susp-isom-exac-seq}, the connecting homomorphism in cohomology of the covers induces the following homomorphism of cohomology of the spaces:
\[
\CH^p (X; \underline{\UU(1)}) \xrightarrow{\partial} \CH^{p+1} (\Sigma X; \underline{\UU(1)}), \notag
\]
an isomorphism for $p > 0$, called the (unreduced) suspension isomorphism.
\end{cor}
\begin{proof}
Exactness of the short exact sequence of cochain complexes of abelian groups of lemma \ref{l-susp-isom-exac-seq} gives the long exact sequence \citep[page~333]{Rotm09} of cohomology of the covers, preserving the direct sum of the middle term of the short exact sequence \citep[page~339]{Rotm09}.  Then, exactness of direct limits \citep[page~247]{Rotm09} and the fact that they preserve direct sums, give the long exact sequence of cohomology groups of the spaces, using definition \ref{d-cech-coho}.  Using the universal property of direct limits, the direct limit insertion homomorphisms form a morphism of sequences between the long exact sequence of cohomology of the covers and the long exact sequence of cohomology of the spaces.

Since $X$ is locally connected, so are $\Sigma X$, $CX^{-1}$, and $CX^1$.  That $X$ is a metric space ensures that it, the suspension, and the cones are hereditarily paracompact \citep[page~979]{Ston48}.  By lemma \ref{l-cech-coho-cove}, the direct limit insertion homomorphisms are isomorphisms.

The contractible spaces $CX^{-1}$, $CX^1$ have zero singular cohomology for $p > 0$, and so by corollary \ref{co-u1-z-sing-diag} have zero \v{C}ech cohomology as well.  Thus the connecting homomorphisms for cohomology of the spaces and for the cohomology of the covers are isomorphisms for $p > 0$.
\end{proof}

\begin{prop}\label{p-dd-bg-susp-c1-pb}
\index{Chern class}
\index{Dixmier-Douady class}
\index{bundle gerbe!over unreduced suspension}
\index{suspension!unreduced!isomorphism}
(The Dixmier-Douady Class of a Bundle Gerbe Over an Unreduced Suspension is Isomorphic to the First Chern Class of a Constructed Principal Bundle).
A continuous bundle gerbe $(P, Y, \Sigma X)$ over the unreduced suspension of a locally connected metric space $X$ gives, by choice of sections of a naturally constructed two set open cover of $\Sigma X$, a principal $\UU(1)$ bundle $Q \rightarrow X$.  The inverse of the suspension isomorphism of corollary \ref{co-susp-isom} takes the Dixmier-Douady class of the bundle gerbe, $DD(P, Y, \Sigma X) \in \CH^2 (\Sigma X;
\underline{\UU(1)})$, to minus (or the inverse of, if $\UU(1)$ is written multiplicatively) the first Chern class of $Q$, $-c_1 (Q) \text{ or } (c_1 (Q))^{-1} \in \CH^1 (X; \underline{\UU(1)})$.
\end{prop}
\begin{proof}
We use two related covers of $\Sigma X$, one used to obtain a cocycle for the first Chern class of $Q$, and one used to obtain a cocycle for the Dixmier-Douady class of the bundle gerbe.  These are defined so that as much as possible of the Dixmier-Douady cocycle definition is trivial.  The first cover is one of two open sets, used to define $Q$\:
\begin{align}
V_{-1} &= (X \cross [-1,\frac{1}{2})) / (X \cross \{ -1 \}) \notag \\
V_1 &= (X \cross (-\frac{1}{2},1]) / (X \cross \{ 1 \}), \text{ with intersection } \notag \\
V_{-1} \cap V_1 &= X \cross (-\frac{1}{2}. \frac{1}{2}), \notag
\end{align}
The second cover is made by first choosing a good cover $U$ of $X$, which is possible because the definition of a continuous bundle gerbe requires that each open cover of $X$ has a refinement that is a good cover.  Then the good cover $U^{\Sigma}$ of $\Sigma X$, used for the Dixmier-Douady class, is constructed as in definition \ref{d-susp-isom-cov}.  We will use the symbols $a$ for indices $\pm 1$, and $i$, $j$, $k$ for indices $\ne \pm 1$.

The $V_a$ are contractible sets over which, because $\Sigma X$ is hereditarily paracompact (see \citet[pages~66,~90,~96]{Span66} and use a homotopy between the identity and the map of a contractible space to one of its points), we can choose sections $t_a \colon V_a \rightarrow Y$, $a = \pm 1$.  Let $Q = ((t_{-1}, t_1)^{*} P)_{|X}$, a principal $\UU(1)$ bundle over $X$.

Now we will construct a \v{C}ech 1-cocycle representing the first Chern class of $Q$.  Since the $U^{\Sigma}_i$ are contractible, we can choose sections $\tau_{i} \colon  V_{-1} \cap V_1 \cap U^{\Sigma}_i = U^{\Sigma}_i \rightarrow (t_{-1}, t_1)^{*} P$.  These give local trivializations of $(t_{-1}, t_1)^{*} P$ defined using the continuous translation function of lemma \ref{l-pb-tran-func-tors}, with transition functions (see lemma \ref{l-tran-func-pb}) given by $\tau_{j} = h_{ji} \tau_{i}$ on $U^{\Sigma}_i \cap U^{\Sigma}_j$.  The $h_{ij}$ satisfy $h_{jk} h_{ik}^{-1} h_{ij} = 1$ because they are defined as transition functions of local trivializations of a principal $\UU(1)$ bundle.  The bundle $Q$ has the same transition functions $h_{ji}$ but restricted to $X$, and these define a \v{C}ech 1-cocycle $h$ since the $U^{\Sigma} \cap X _i = U_i$ cover $X$.  Theorem 2.1.3 of \citet[pages~62--68]{Bryl93}, adapted to argue for principal $\UU(1)$ bundles instead of complex line bundles, implies that $h^{-1}$ is a representative of the Chern class $c_1(Q)$.  For the following, denote by $\tau_{a} \colon V_a \rightarrow (t_a, t_a)^{*} P$, $a \in \{ \pm 1 \}$, the canonical identity sections of lemma \ref{l-bg-mult-id-inv}.

Since $U^{\Sigma}$ is a good cover of $\Sigma X$, its cohomology is isomorphic to that of $\Sigma X$, and by note \ref{n-bg-dd-clas}, the Dixmier-Douady class of the bundle gerbe is given directly by the class of the cocycle constructed from the good cover.  To use Dixmier-Douady class definition \ref{d-bg-dd-clas}, using lemma \ref{l-bg-mult-id-inv} define:
\begin{align}
s_a &= {t_{a}}_{|U^{\Sigma}_a} \colon U^{\Sigma}_a \rightarrow Y \notag \\
s_i &= {t_{1}}_{|U^{\Sigma}_i} \colon U^{\Sigma}_i \rightarrow Y \notag \\
\sigma_{-1 i} &= {\tau_{i}}_{| U^{\Sigma}_{-1}} \colon U^{\Sigma}_{-1} \cap U^{\Sigma}_i \rightarrow (s_{-1}, s_i)^{*} P \notag \\
\sigma_{i -1} &= \sigma_{-1 i}^{-1} \text{ (two subscripts, not subtraction)} \notag \\
\sigma_{1 i} &= {\tau_{1}}_{| U^{\Sigma}_1 \cap U^{\Sigma}_i} \colon U^{\Sigma}_1 \cap U^{\Sigma}_i \rightarrow (s_1, s_i)^{*} P \notag \\
\sigma_{i 1} &= (\sigma_{1 i})^{-1} = \sigma_{1 i} \notag \\
\sigma_{i j} &= {\tau_{1}}_{| U^{\Sigma}_i \cap U^{\Sigma}_j} \colon U^{\Sigma}_i \cap U^{\Sigma}_j \rightarrow (s_i, s_j)^{*} P. \notag
\end{align}
The $\sigma$ after the first two are identity sections as in lemma \ref{l-bg-mult-id-inv}.

With our notation, the possible injective (no repetition) $3$-tuples of indices for the Dixmier-Douady cocycle are $(-1,i,j)$, $(1,i,j)$, and $(i,j,k)$, plus permutations that can be calculated using the fact that the cocycle is an alternating function of its $3$-tuple of indices.  By definition \ref{d-orde-cech-coch}, the cocycle can be nonzero only for injective tuples.  (We use $0$ and $1$ somewhat interchangeably in this context, the former referring to a general abelian group written additively, and the latter to one incarnation of $\UU(1)$.) The cocycle is zero for the last two tuples, because all the $\sigma$'s involved are identity sections; e.g., $m(\sigma_{ij} \tensor \sigma_{jk}) = \sigma_{ik}$ suitably restricted.

Thus, the essence of the Dixmier-Douady class now lies in $g_{-1ij}$, defined by
\begin{align}
\sigma_{-1i} &= m(\sigma_{-1i} \tensor \sigma_{ij}) = g_{-1ij} \sigma_{-1j} \text{ on } U^{\Sigma}_{-1} \cap U^{\Sigma}_i \cap U^{\Sigma}_j, \text{ hence } \notag \\
g_{-1ij} &= h_{ij}, \text{ on that set.} \notag
\end{align}

Now we relate the Dixmier-Douady class of the bundle gerbe over $\Sigma X$ and the first Chern class of the principal bundle over $X$ by following backward at the co-chain level the isomorphism $\partial$ of corollary \ref{co-susp-isom} (refer to note \ref{n-susp-isom}): $\CH^2 (\Sigma X; \underline{\UU(1)}) \cong \CH^2 (U^{\Sigma}; \underline{\UU(1)}) \xrightarrow{\partial^{-1}} \CH^1 (U^{\Sigma} \cap X; \underline{\UU(1)}) \cong \CH^1 (X; \underline{\UU(1)})$.

Starting with the cocycle $g$, its image via $r_1^{\sharp}$ is zero, since the only intersections for which $g$ is nonzero are intersected with the set $U^{\Sigma}_{-1}$, and the intersection of that with $CX^1$, for $r_1^{\sharp}$, is empty.  On the other hand the domain of $g_{-1ij}$ is unaltered by intersection with $CX^{-1}$, for $r_{-1}^{\sharp}$.

To find $w_{-1} \in C^1 (r_{-1}^* U^{\Sigma})$ such that $\delta (w_{-1}) = r_{-1}^{\sharp} (g)$, use bundle gerbe multiplication to define $g^{\tau}$ similarly to the way $g_{-1ij}$ is defined:
\begin{align}
\tau_{ij} &= {\tau_{1}}_{|U^{\Sigma}_i \cap U^{\Sigma}_j} \notag \\
\tau_{i} &= m(\tau_{i} \tensor \tau_{ij}) = g^{\tau}_{ij} \tau_{j} \text{ on } U^{\Sigma}_i \cap U^{\Sigma}_j, \text{ hence } \notag \\
g^{\tau}_{ij} &= h_{ij}, \notag
\end{align}
defining $g^{\tau}$ as zero on all other $2$-tuples not permutations of $(i,j)$, and noting that its definition makes it an alternating function of $2$-tuples $(i, j)$, so that it is a $1$ co-chain.  Let $w_{-1} = r_{-1}^{\sharp} (g^{\tau})$; i.e. restrict $g^{\tau}$ to $CX^{-1}$ for the present purpose.  To evaluate $\delta (w_{-1}) \in C^2 (r_{-1}^* U^{\Sigma})$, note that since it's an alternating function of its $3$-tuple, $\delta (w_{-1})$ will be zero for all $3$-tuples up to permutation except $(-1,i,j)$, $(1,i,j)$, and $(i,j,k)$.  We have that $\delta (w_{-1})$ is zero on $(1,i,j)$ because it's a cochain in $C^2 (r_{-1}^* U^{\Sigma})$ and $U^{\Sigma}_1 \cap U^{\Sigma}_i \cap r_{-1}^* U^{\Sigma}_j = \emptyset$.  For the remaining two $3$-tuples,
\begin{align}
(\delta (w_{-1}))_{ijk} &= (w_{-1})_{jk} (w_{-1})_{ik}^{-1} (w_{-1})_{ij} \notag \\
 &= h_{jk} h_{ik}^{-1} h_{ij} = \underline{1} \text{ on } U^{\Sigma}_i \cap U^{\Sigma}_j \cap r_{-1}^* U^{\Sigma}_k \notag
\end{align}
and
\begin{align}
(\delta (w_{-1}))_{-1jk} &= (w_{-1})_{jk} (w_{-1})_{-1k}^{-1} (w_{-1})_{-1j} \notag \\
 &= h_{jk} = g_{-1jk} \text{ on } U^{\Sigma}_{-1} \cap U^{\Sigma}_j \cap r_{-1}^* U^{\Sigma}_k, \notag
\end{align}
because $(w_{-1})_{-1j} = 1 = (w_{-1})_{-1k}$ by definition of $w_{-1}$.  Thus $\delta (w_{-1}) = r_{-1}^{\sharp} (g)$.

$s_{-1}^{\sharp} (w_{-1}) = {g^{\tau}}_{|X} = h$.  The steps can be reversed to see that $\partial ([h]) = [g]$.
\end{proof}

\section{Map $S^3$ to $LS^4$; Stereoscopic Projection; Homogeneous Coordinates}\label{s-map-s3-ls4-ster-proj-homo-coor}

Now come a sequence of definitions and lemmas preparatory to pulling back the bundle gerbe over $L\QH P^1$ to $S^3$.  This is done using the adjoint to the diffeomorphism from the reduced suspension $S S^3 \rightarrow S^4$, the stereoscopic projection for $S^4$, the homogeneous coordinates for $\QH P^1$, and by constructing a diffeomorphism $S^4 \rightarrow \QH P^1$.  These lemmas include \ref{l-incl-s3-ls4-cont} and \ref{l-diff-s4-hp1}, which were referenced in the proof of proposition \ref{p-tran-p12-eq-to-dd-1}.

The following definition concerning unreduced suspensions is intended as motivation for definition \ref{d-redu-susp-sphe-home} concerning reduced suspensions, which also directly refers to the map $\alpha_x$ in its own motivation.
\begin{defn}\label{d-unre-susp-sphe-home}
\index{suspension!unreduced!sphere}
\index{loops!suspension!unreduced}
(A Homeomorphism $\Sigma S^n \rightarrow S^{n+1}$; Loops in $S^{n+1}$).
Define the map $\phi_{\Sigma}$ from the unreduced suspension of $S^n$ to $S^{n+1}$ as follows.

We define $S^n$ as the unit sphere in $\RR^{n+1}$, with basepoint $1 = (1, 0, \dots, 0)$.  We will sometimes use the same notation $x$ for a point in $S^n \subset \RR^{n+1}$ and what more properly might be called $(x, 0) \in S^{n+1} \subset \RR^{n+2}$.

Define the unreduced suspension as $(S^n \cross [-1,1]) / (S^n \cross \{-1\}) / (S^n \cross \{1\})$, denoting the points of the suspension by $[x, t]$, $x \in S^n$, $t \in [-1, 1]$, with basepoint $1 = [1, 0]$.  Then define
\begin{align}
\phi_{\Sigma} \colon \Sigma S^n &\rightarrow S^{n+1} \notag \\
[x, t] &\mapsto ((\sqrt{1 - t^2}) x, t), \text{ } t \in [-1, 1] \notag \\
[x, \cos(s)] &\mapsto (\sin(s) x, \cos(s)), \text{ } s \in [0, \pi]. \notag
\end{align}
Use $\phi_{\Sigma}$ to define the smooth structure on $\Sigma S^n$ making the map a diffeomorphism.

Define the paths
\begin{align}
\alpha_{x,\text{cyl}} \colon [0, \pi] &\rightarrow S^n \cross [-1, 1], \text{ } x \in S^n, \notag \\
s &\mapsto (x, \cos(s)) \notag \\
\alpha_{x,\Sigma} \colon [0, \pi] &\rightarrow \Sigma S^n, \text{ } x \in S^n, \notag \\
s &\mapsto [x, \cos(s)] \notag \\
\alpha_{x} = \phi_{\Sigma} \circ \alpha_x \colon [0, \pi] &\rightarrow S^{n+1}, \text{ } x \in S^n, \notag \\
s &\mapsto (\sin(s) x, \cos(s)). \notag
\end{align}
\end{defn}
\begin{note}
\index{suspension!unreduced!sphere}
\index{loops!suspension!unreduced}
(Properties of $\phi_{\Sigma}$).
The map $\phi_{\Sigma}$ is basepoint preserving; it is a homeomorphism as follows.  For $t \in (-1, 1)$, $\phi_{\Sigma}^{-1} (x, t) = [x / \sqrt{1 - t^2}, t]$, and for $t \in \{ -1, 1 \}$, $\phi_{\Sigma}^{-1} (0, t) = [1, t]$.  By looking at inverse images of open balls around points, $\phi_{\Sigma}$ is continuous, and since $\Sigma S^n$ is compact and $S^{n+1}$ Hausdorff, is a homeomorphism.  The paths are smooth.
\end{note}

\begin{defn}\label{d-redu-susp-sphe-home}
\index{suspension!reduced!sphere}
\index{loops!suspension!reduced}
(A Homeomorphism $S S^n \rightarrow S^{n+1}$; Loops in $S^{n+1}$). Define the reduced suspension $S S^n$ as $(S^n \cross [-1,1]) / ((S^n \cross \{ -1, 1 \}) \cup (\{1\} \cross [-1, 1])) \cong \Sigma S^n / (\{1\} \cross [-1, 1])$, denoting the points of the reduced suspension by $[x, t]$, $x \in S^n$, $t \in [-1, 1]$, with basepoint $1 = [1, 0]$.

Note that except for the two poles $N = (0,\dots,0,1)$, $S = (0,\dots,0,-1)$, $S^{n+1}$ as a set is the disjoint union of the images of the paths $\alpha_x$, $x \in S^n$; the two poles are in the image of every path (see definition \ref{d-unre-susp-sphe-home}).  Imagine a continuous collapsing map $S^{n+1} \rightarrow S^{n+1}$ corresponding via $\phi_{\Sigma}$ and the as yet undefined $\phi_S$ to the quotient map $\Sigma S^n \rightarrow S S^n$.  The paths collapse to loops.  We define $\phi_S$ so that $\alpha_x (s) \mapsto \beta_x (2 s)$ gives a bijection of paths $\alpha_x$ to loops $\beta_x$, the union of whose images is $S^{n+1}$, disjoint except at $1$.  These images are the circles defined by intersecting $S^{n+1}$ with the two-dimensional plane defined by the line through $1 = (1,0,\dots,0)$ and $(1,0,\dots,1)$, and the line through $1$ and $x$ interpreted as a point in $S^{n+1}$.  The point $1$ is in every such circle.  Thus define
\begin{align}
\beta_x \colon [0, 2 \pi] &\rightarrow S^{n+1}, \text{ } x \in S^n, \notag \\
s &\mapsto \frac{1 + x}{2} + \cos(s) \frac{1 - x}{2} + \sin(s) \norm{\frac{1 - x}{2}} N \notag \\
\phi_S \colon S S^n &\rightarrow S^{n+1} \notag \\
[x, \cos(s)] &\mapsto \frac{1 + x}{2} + \cos(2 s) \frac{1 - x}{2} + \sin(2 s) \norm{\frac{1 - x}{2}} N, \text{ } s \in [0, \pi], \notag \\
[x, t] &\mapsto \frac{1 + x}{2} + (2 t^2 - 1) \frac{1 - x}{2} + 2 t \sqrt{1 - t^2} \norm{\frac{1 - x}{2}} N, \text{ } t \in [-1, 1]. \notag
\end{align}
Use $\phi_S$ to define the smooth structure on $S S^n$ making the map a diffeomorphism.
\end{defn}
\begin{note}\label{n-redu-susp-sphe-home}
\index{loop!reparametrization}
(Properties of $\phi_S$).
The map $\phi_S$ is basepoint preserving and a homeomorphism.  The loops are smooth.  The three terms of the sum for $\phi_S$ are mutually orthogonal.  The map $x \mapsto \beta_x$ is the adjoint $\widehat{\phi_S}$ of $\phi_S$ \citep[pages~12--13]{Swit75}, except for transforming the loop parameter from $[-1, 1]$ to $[0, 2 \pi]$.  That is, $\widehat{\phi_S} = (x \mapsto \beta_x \circ (2 \arccos))$, or $\beta_x = \widehat{\phi_S}(x) \circ \cos(\frac{1}{2} \medspace \cdot)$.  The reparametrization of the adjoint loops is useful conceptually, to make them smooth, and to limit their Fourier series to frequencies $0, \pm 1$, which could help in calculations.
\end{note}

\begin{lem}\label{l-incl-s3-ls4-cont}
\index{inclusion!S3 LS4@$S^3 \rightarrow LS^4$!continuous}
\index{inclusion!S3 LS4@$S^3 \rightarrow LS^4$!isomorphism H3@isomorphism $\HH^3$}
(The Inclusion $S^3 \xrightarrow{i_3} LS^4$ is Continuous, and Induces an Isomorphism in $\HH^3$).
The inclusion $S^3 \xrightarrow{i_3} LS^4$ is defined as the adjoint $\widehat{\phi_S}$ of the homeomorphism $\phi_S$ of \ref{d-redu-susp-sphe-home} for $n = 3$, with the reparametrization of the loops such that for $x \in S^3$, $i_3 (x) = \beta_x$, followed by the inclusion into $LS^4$.  The map $i_3$ is continuous and induces an isomorphism in $\HH^3$ with $\ZZ$ coefficients.
\end{lem}
\begin{proof}
By lemma \ref{l-co-top}, since $\phi_S \colon S S^3 \rightarrow S^4$ is continuous, $\widehat{\phi_S} \colon S^3 \rightarrow C(S^1, S^4)$ is continuous.

$\phi_S$ is basepoint preserving and its image lies in $\Omega S^4$, the continuous loops in $S^4$ that begin and end at the basepoint $1$ of $S^4$.  The constant loop at $1$ is the basepoint of $\Omega S^4$. By \citet[page~20]{Swit75}, the adjoint correspondence of based homotopy classes of basepoint preserving maps
\[
\pi_{k+1} (S^4) = [S S^k, S^4]_{*} \rightarrow [S^k, \Omega S^4]_{*} = \pi_k (\Omega S^4) \notag
\]
is an isomorphism of groups for all $k >= 0$.  Thus $\pi_k (\Omega S^4) = 0 = \pi_k (S^3)$ for $k \in \{ 0, 1, 2 \}$, and $\pi_3 (\Omega S^4) = \ZZ = \pi_3 (S^3)$.

$\phi_S$ induces the map $\phi_S^{\sharp} \colon C(S^4, S^4) \rightarrow C(S S^3, S^4)$, a homeomorphism by lemma \ref{l-co-top}, in turn inducing a group isomorphism we will denote by $\phi_S^{*} \colon \pi_4 (S^4) = [S^4, S^4]_{*} \rightarrow [S S^3, S^4]_{*}$.  Since $[\ident]$ is a generator of $\pi_4 (S^4) \cong \ZZ$, $\phi_S^{*} ([\ident]) = [\phi_S]$ is a generator of $C(S S^3, S^4)$.  Therefore, using the isomorphism for the adjoint correspondence, $[\widehat{\phi_S}]$ is a generator of $\pi_3 (\Omega S^4) \cong \ZZ$.

$\widehat{\phi_S}_{*} \colon \pi_k (S^3) \rightarrow \pi_k (\Omega S^4)$ is an isomorphism for $k \in \{ 0, 1, 2 \}$ because its domain and codomain are both zero.  For $k = 3$, since $[\ident] \in \pi_3 (S^3) \cong \ZZ$ is a generator and $\widehat{\phi_S}_{*} ([\ident]) = [\widehat{\phi_S} \circ \ident] = [\widehat{\phi_S}] \in \pi_3 (\Omega S^4)$ is a generator, $\widehat{\phi_S}_{*}$ is once again an isomorphism.

Thus using the Hurewicz theorem and naturality of the Hurewicz maps \citep[pages~114,~185]{Swit75} $\widehat{\phi_S}$ induces an isomorphism in $\HH_k$, $k \in \{ 0, 1, 2, 3 \}$, and thus using naturality of the Universal Coefficient Theorem exact sequence \citep[page~241]{Swit75} in $\HH^k$ also.  That is, $\HH^3 (\Omega S^4) \xrightarrow{\widehat{\phi_S}^{*}} \HH^3 (S^3)$ is an isomorphism.

Now we need to connect the result for $\widehat{\phi_S}$ and $\Omega S^4$ with the desired result for $i_3$ and $LS^4$. First, note that the reparametrization to get from the loops of $\widehat{\phi_S}$ to those of $i_3$ is, by note \ref{n-redu-susp-sphe-home}, composition with $\cos(\frac{1}{2} \medspace \cdot) \colon [0, 2 \pi] \rightarrow [-1, 1]$, a homeomorphism.  By lemma \ref{l-co-top}, the map of $C(S^1, S^4)$ to itself induced by reparametrizing loops in this way, is a homeomorphism, which induces a homeomorphism of its subspace $\Omega S^4$ onto itself, and thus induces isomorphisms in cohomology.  Thus the map we denote here $\widetilde{i_3} = \cos(\frac{1}{2} \medspace \cdot)^{\sharp} (\widehat{\phi_S})$ induces an isomorphism $\HH^3 (\Omega S^4) \xrightarrow{\widetilde{i_3}^{*}} \HH^3 (S^3)$.  The loops from $\widetilde{i_3}$ are smooth.

If we compose the inclusion $\Omega S^4 \xrightarrow{\incl_{\Omega S^4}} C(S^1, S^4)$ with $\widetilde{i_3}$ to get $\widetilde{\widetilde{i_3}} = \incl_{\Omega S^4} \circ \widetilde{i_3}$, we still get an isomorphism in $\HH^3$, because that inclusion induces an isomorphism in $\HH^3$, as follows.  From items 4 and 7 of lemma \ref{l-s4-oc-grps}, $\pi_3 (C(S^1, S^4)) \cong \ZZ$.  As in the proof of that lemma, $\Omega S^4 \xrightarrow{\incl_{\Omega S^4}} C(S^1, S^4) \xrightarrow{\ev_1} S^4$ (the other lemma's basepoint name $0$ is this one's $1$) is a fibration and we have its homotopy exact sequence \citep[pages~435,~453,~455]{Bred93} including the surjection $\ZZ \rightarrow \ZZ$, hence isomorphism $\pi_3 (\Omega S^4) \xrightarrow{{\incl_{\Omega S^4}}_{*}} \pi_3 (C(S^1, S^4))$.  Since the lower homotopy groups of both spaces are zero, we can reason as above to conclude that $\HH^3 (C(S^1, S^4)) \xrightarrow{\incl_{\Omega S^4}^{*}} \HH^3 (\Omega S^4)$ is an isomorphism, and thus so is $\HH^3 (C(S^1, S^4)) \xrightarrow{\widetilde{\widetilde{i_3}}^{*}} \HH^3 (S^3)$.

Continuity of $S^3 \xrightarrow{i_3} LS^4$ is a special case of \citet[page~22]{Stac05} theorem 3.28, since the map $S^1 \cross S^3 \rightarrow S^4$, $(x, s) \mapsto i_3 (x) (s)$ is smooth.  It follows that $i_3$ induces a map $\HH^3 (LS^4) \xrightarrow{i_3^{*}} \HH^3 (S^3)$.  Proposition \ref{p-lm-c-coho-isom} gives an isomorphism $\HH^3 (C(S^1, S^4)) \xrightarrow{i_{LS^4}^{*}} \HH^3 (LS^4)$.  Since $\widetilde{\widetilde{i_3}} = i_{LS^4} \circ i_3$, $\widetilde{\widetilde{i_3}}^{*} = i_3^{*} \circ i_{LS^4}^{*}$, and since the left hand side of the equation and the right factor of the right hand side are isomorphisms, so is $i_3^{*}$.
\end{proof}

\begin{lem}\label{l-ster-s4}
\index{stereographic projection!S4@$S^4$}
\index{S4@$S^4$!stereographic projection}
\index{manifold!charts!S4@$S^4$}
\index{S4@$S^4$!manifold charts}
(Stereographic Projection for $S^4$).
Define
\begin{align}
N &= (0, 1) \in \QH \dirsum \RR \cong \RR^5 \notag \\
S &= (0, -1) \notag \\
\phi_N \colon S^4 \setminus \{ N \} &\rightarrow \QH \notag \\
 (z, y) &\mapsto \frac{z}{1 - y} \notag \\
\phi_S \colon S^4 \setminus \{ S \} &\rightarrow \QH \notag \\
 (z, y) &\mapsto \frac{\overline{z}}{1 + y}. \notag
\end{align}
$\phi_N$, $\phi_S$ are diffeomorphisms, and thought of as smooth manifold charts for $S^4$, have transition function
\begin{align}
\phi_N \circ \phi_S^{-1} \colon \QH \setminus \{ 0 \} &\rightarrow \QH \setminus \{ 0 \} \notag \\
 w &\mapsto w^{-1}. \notag
\end{align}
\end{lem}

\begin{lem}\label{l-homo-coor-hp1}
\index{homogeneous coordinates!HP1@$\QH P^1$}
\index{HP1@$\QH P^1$!homogeneous coordinates}
\index{manifold!charts!HP1@$\QH P^1$}
\index{HP1@$\QH P^1$!manifold charts}
(Homogeneous Coordinates for $\QH P^1$).
Define
\begin{align}
U_0 &= \QH P^1 \setminus \{ [0, 1] \} = \{ [1, w] \st w \in \QH \} \notag \\
U_1 &= \QH P^1 \setminus \{ [1, 0] \} = \{ [w, 1] \st w \in \QH \} \notag \\
\psi_0 \colon U_0 &\rightarrow \QH \notag \\
 [1, w] &\mapsto w \notag \\
\psi_1 \colon U_1 &\rightarrow \QH \notag \\
 [w, 1] &\mapsto w. \notag
\end{align}
 $\psi_0$, $\psi_1$ are diffeomorphisms, and thought of as smooth manifold charts for $\QH P^1$, have transition function
\begin{align}
\psi_0 \circ \psi_1^{-1} \colon \QH \setminus \{ 0 \} &\rightarrow \QH \setminus \{ 0 \} \notag \\
 w &\mapsto w^{-1}. \notag
\end{align}
\end{lem}

\begin{lem}\label{l-diff-s4-hp1}
\index{S4@$S^4$!diffeomorphism to HP1@diffeomorphism to $\QH P^1$}
\index{HP1@$\QH P^1$!diffeomorphism to S4@diffeomorphism to $S^4$}
(A Diffeomorphism $S^4 \rightarrow \QH P^1$.)
Define via lemmas \ref{l-ster-s4} and \ref{l-homo-coor-hp1}
\begin{align}
\theta_N &= \psi_0^{-1} \circ \phi_N \colon S^4 \setminus \{ N \} = S^4 \setminus \{ (0, 1) \} \rightarrow U_0 = \QH P^1 \setminus \{ [0, 1] \} \notag \\
(z, y) &\mapsto [1, \frac{z}{1 - y}] = [1 - y, z] \notag \\
\theta_S &= \psi_s^{-1} \circ \phi_S \colon S^4 \setminus \{ S \} = S^4 \setminus \{ (0, -1) \} \rightarrow U_1 = \QH P^1 \setminus \{ [1, 0] \} \notag \\
(z, y) &\mapsto [\frac{\overline{z}}{1 + y}, 1] = [\overline{z}, 1 + y] \notag \\
\theta_N &= \theta_S \text{ on } S^4 \setminus \{ N, S \} \notag \\
\theta &\colon S^4 \rightarrow \QH P^1 \notag \\
\theta_{|S^4 \setminus \{ N \}} &= \theta_N \notag \\
\theta_{|S^4 \setminus \{ S \}} &= \theta_S. \notag
\end{align}
$\theta_N$, $\theta_S$, and $\theta$ are diffeomorphisms.
\end{lem}

\section{Construct, Pull Back, Local Sections of the Bundle Gerbe}\label{s-bg-cons-pb-loc-sect-pb}

Having discussed the map that will be used to pull the bundle gerbe back to $S^3$, and having obtained diffeomorphisms involving $S^4$ and $\QH P^1$, we now proceed to construct the bundle gerbe, pull it back, construct local sections of the pullback, and use them in proposition \ref{p-dd-bg-susp-c1-pb} to get principal $\UU(1)$ bundle $Q \rightarrow S^2$.

In the bundle gerbe construction we use the symplectic frame bundle $\Symp(E_Q)$ whose total space we identify with $S^7$, rather than using its orthonormal frame bundle $\SO(E_Q)$.  This can be called the reduction of the structure group from $\SO(4)$ to $\Symp(1)$, but would be more expressively worded as an enlargement of the structure group in the other direction, since we are naturally given the $\Symp(1)$ group.  This change results in an isomorphic bundle gerbe. \begin{lem}\label{l-eq-bg}
\index{bundle gerbe!quaternionic line bundle}
\index{quaternionic line bundle!bundle gerbe}
(The Construction of the Bundle Gerbe $G(E_Q)$ from the Tautological Quaternionic Line Bundle).
Using in the bundle gerbe construction the symplectic frame bundle $\Symp(E_Q)$, with group $\Symp(1)$, rather than the orthonormal frame bundle $\SO(E_Q)$, with group $\SO(4)$, results in a bundle gerbe isomorphic to $G(E_Q)$.

For the construction of $G(E_Q)$ we obtain $Y = LS^7 \cross_{\Symp(1)} \Lagr_{res}$ (see definition \ref{d-pol-clas-bndl}), will identify $Y^{[2]} = LS^7 \cross_{\Symp(1)} (\Lagr_{res} \cross \Lagr_{res})$ (see proposition \ref{p-bg}), and $P = LS^7 \cross_{\Symp(1)} T$ (see propositions \ref{p-bg-t} and \ref{p-bg-p}).
\end{lem}
\begin{proof}
From proposition \ref{p-bg-func}, starting with the tautological quaternionic line bundle $E_Q$ of definition \ref{d-q-vb}, we obtain the bundle gerbe $G(E_Q)$ over $L\QH P^1$.  The transition functions of $E_Q$, by lemma \ref{l-q-vb}, are determined by maps with values in $S^3 \cong \Symp(1)$, which are the principal bundle transition functions of definition \ref{l-tran-func-pb} of $\SO(E_Q)$.  Hence \citep[page~53]{KN63} we can reduce the structure group of the orthonormal frame bundle $\SO(E_Q)$ from $\SO(4)$ to $\Symp(1)$ and obtain $\Symp(E_Q)$; there is a $\Symp(1)$-equivariant map $f: \Symp(E_Q) \rightarrow \SO(E_Q)$ over the identity.  Thus there is an principal $\SO(4)$ bundle isomorphism $\Symp(E_Q) \cross_{\Symp(1)} \SO(4) \rightarrow \SO(E_Q)$, given by $[q,g] \mapsto f(q) g$.

In general, let $G$ be a topological group, $H$ a subgroup, $Q$ a principal $H$ bundle, and $F$ a space on which there is a left $G$ action.  Then $Q \cross_H G$ is a principal $G$ bundle, and $(Q \cross_H G) \cross_G F \cong Q \cross_H F$.

Putting together the principal $\SO(4)$ bundle isomorphism and the one just given for associated products, $\SO(E_Q) \cross_{\SO(4)} F \cong (\Symp(E_Q) \cross_{\Symp(1)} \SO(4)) \cross_{\SO(4)} F \cong \Symp(E_Q) \cross_{\Symp(1)} F $.  Moreover, these products and isomorphisms can be looped and are natural with respect to maps $F \rightarrow F'$.  Thus we obtain isomorphic bundle gerbes.
\end{proof}

The following lemma was referenced earlier, in the proof of proposition \ref{p-tran-p12-eq-to-dd-1}, and also is used by lemma \ref{l-eq-bg-pb}.
\begin{lem}\label{l-loc-sect-lspe}
\index{section!local!LSpE@$L\Symp(E_Q)$}
\index{bundle gerbe!over S3@over$S^3$}
(The Construction of Local Sections of the Pullback by $S^3 \xrightarrow{i_3} L\QH P^1$ of $L\Symp(E_Q)$).
The composition of $i_3$ with the rotation $R_{S^4}$ by $\frac{\pi}{2}$, of $S^4 \subset \RR^5 \cong \QH \dirsum \RR$, $\QH = \qone \RR \dirsum \qi \RR \dirsum \qj \RR \dirsum \qk \RR$, that leaves fixed all elements of summands of $\QH$ except for those of $\qi \RR$, and rotates the $\qi \RR \dirsum \RR$ plane counterclockwise (viewed with positive $\qi \RR$ to the right and positive $\RR$ up), moving $(\qi, 0) \mapsto (0,1) \mapsto (-\qi, 0) \mapsto (0,-1)$, produces for each $x \in S^3 \setminus \{ -1 \}$ a loop in $S^4$ whose image lies in the domains of both $\phi_N$ and $\phi_S$ of lemma \ref{l-ster-s4} except that for $x = \qi$, the image lies wholly only in the domain of $\phi_S$, and for $x = -\qi$, wholly only in the domain of $\phi_N$.  Defining $\pi_{\qone}, \pi_{\qi}, \pi_{\qj}, \pi_{\qk} \colon \QH = \qone \RR \dirsum \qi \RR \dirsum \qj \RR \dirsum \qk \RR \rightarrow \RR$ as projections onto the respective real-valued coordinates, so that $\pi_{\qone} \qone + \pi_{\qi} \qi + \pi_{\qj} \qj + \pi_{\qk} \qk = \ident_{\QH}$, these loops are given as follows, for $x \in S^3$ and $s \in [0, 2 \pi]$:
\begin{alignat}{2}
R_{S^4} \circ i_3 (x) (s) = (&\binphantom{+} \pi_{\qone} (&\frac{1 + x}{2} + \cos (s) \mspace{1mu} &\frac{1 - x}{2}) \qone \notag \\
&-                        &\sin (s) \norm{&\frac{1 - x}{2}} \qi \notag \\
&+ \pi_{\qj} (&\frac{1 + x}{2} + \cos (s) &\frac{1 - x}{2}) \qj \notag \\
&+ \pi_{\qk} (&\frac{1 + x}{2} + \cos (s) &\frac{1 - x}{2}) \qk, \notag \\
 &\pi_{\qi} (&\frac{1 + x}{2} + \cos (s) &\frac{1 - x}{2})). \notag
\end{alignat}
Define
\begin{align}
\widetilde{\theta_N} \colon S^4 \setminus \{ N \} &\rightarrow S^7 \subset \QH^2 \notag \\
(z, y) &\mapsto \frac{(1 - y, z)}{\sqrt{(1 - y)^2 + \norm{z}^2}} \notag \\
\widetilde{\theta_S} \colon S^4 \setminus \{ S \} &\rightarrow S^7 \subset \QH^2 \notag \\
(z, y) &\mapsto \frac{(\overline{z}, 1 + y)}{\sqrt{(1 + y)^2 + \norm{z}^2}} \notag \\
\eta_{\qi} \colon S^3 \setminus \{ \qi \} &\rightarrow LS^7 \cong L\Symp(E_Q) \notag \\
x &\mapsto \widetilde{\theta_N} \circ R_{S^4} \circ i_3 (x) \notag \\
\eta_{-\qi} \colon S^3 \setminus \{ -\qi \} &\rightarrow LS^7 \cong L\Symp(E_Q) \notag \\
x &\mapsto \widetilde{\theta_S} \circ R_{S^4} \circ i_3 (x) \notag
\end{align}
$\pi_{\QH P^1} \circ \widetilde{\theta_N} = \theta_N$ and $\pi_{\QH P^1} \circ \widetilde{\theta_S} = \theta_S$, so $L\pi_{\QH P^1} \circ \eta_{\qi} = L\theta_N \circ LR_{S^4} \circ i_3$ and $L\pi_{\QH P^1} \circ \eta_{-\qi} = L\theta_S \circ LR_{S^4} \circ i_3$.  Since $\eta_{\qi}$ and $\eta_{-\qi}$ are maps over the restrictions to their domains of $L\theta \circ LR_{S^4} \circ i_3$, they yield local sections of $(L\theta \circ LR_{S^4} \circ i_3)^{*} (L\Symp(E_Q) \rightarrow L\QH P^1)$.

These local sections give local trivializations, and the corresponding transition function $r$ is defined as follows on $S^3 \setminus \{ \pm \qi \}$.  For $(z, y) \in S^4 \subset \QH \dirsum \RR$, $(1 + y) (1 - y) = z \overline z$.  Letting $(z, y) \in S^4 \setminus \{ N, S \}$, we have also $(1 + y) \overline{z}^{-1} = \frac{z}{1 - y}$ (we refrain from writing $\frac{1 + y}{\overline{z}}$ to avoid slighting questions of commutativity).
\begin{align}
\widetilde{\theta_N} (z, y) &= \frac{(1 - y, z)}{\sqrt{(1 - y)^2 + z \overline{z}}} \notag \\
&= \frac{(1 - y) \overline{z}^{-1} \sqrt{(1 + y)^2 + z \overline{z}}}{\sqrt{(1 - y)^2 + z \overline{z}}} \frac{(\overline{z}, 1 + y)}{\sqrt{(1 + y)^2 + z \overline{z}}} \notag \\
&= \frac{(1 - y) \overline{z}^{-1} \sqrt{(1 + y)^2 + z \overline{z}}}{\sqrt{(1 - y)^2 + z \overline{z}}} \widetilde{\theta_S} (z, y), \text{ with} \notag \\
\frac{(1 - y) \overline{z}^{-1} \sqrt{(1 + y)^2 + z \overline{z}}}{\sqrt{(1 - y)^2 + z \overline{z}}} &= \frac{z \sqrt{(1 + y)^2 + z \overline{z}}}{(1 + y) \sqrt{(1 - y)^2 + z \overline{z}}} \notag \\
&= \frac{z \sqrt{(1 + y)^2 + z \overline{z}}}{\sqrt{(1 - y^2)^2 + z \overline{z} (1 + y)^2}} \notag \\
&= \frac{z \sqrt{(1 + y)^2 + z \overline{z}}}{\sqrt{(z \overline{z})^2 + z \overline{z} (1 + y)^2}} \notag \\
&= \frac{z}{\sqrt{z \overline{z}}}. \notag
\end{align}
So, defining
\begin{align}
\rho \colon \QH \setminus \{ 0 \} &\rightarrow S^3 \notag \\
z &\mapsto \frac{z}{\sqrt{z \overline{z}}}, \notag \\
\eta_{\qi} &= (L\rho \circ LR_{S^4} \circ i_3) (L\cdot) (\eta_{-\qi}) \notag \\
&= (r) (L\cdot) (\eta_{-\qi}) = r \eta_{-\qi}, \text{ letting} \notag \\
r &= L\rho \circ L\pi_{\QH} \circ LR_{S^4} \circ i_3 \colon S^3 \setminus \{ \pm \qi \} \rightarrow LS^3, \notag
\end{align}
where $\pi_{\QH} \colon \QH \dirsum \RR \rightarrow \QH$ is the projection, $L\cdot$ is the loop of the left quaternionic multiplication action of $S^3$ on $S^7$, which we also indicate by juxtaposition.  The unit quaternion $\rho(z)$ commutes with $\widetilde{\theta_N} (z, y)$ and $\widetilde{\theta_S} (z, y)$, so we could reverse the order of the factors and use right quaternionic multiplication.  Note that $z \in S^3 \Leftrightarrow y = 0 \Leftrightarrow \rho(z) = z$, and that $x \in S^2 = \{ x \in S^3 | \pi_{\qi} (x) = 0 \} \Leftrightarrow L\pi_{\QH} \circ LR_{S^4} \circ i_3 \in LS^3$; in this case $r(x) = L\pi_{\QH} \circ LR_{S^4} \circ i_3 (x)$.
\end{lem}

Now we obtain the principal $\UU(1)$ bundle $Q \rightarrow S^2$ from the bundle gerbe for $E_Q$.  The local sections used in the step of the construction performed by proposition \ref{p-dd-bg-susp-c1-pb}, are given in the proof below.
\begin{lem}\label{l-eq-bg-pb}
\index{principal bundle!pullback G(EQ) to suspension@pullback $G(E_Q)$ to suspension}
(The Construction of the Principal $\UU(1)$ Bundle $Q \rightarrow S^2$ from the Bundle Gerbe $G(E_Q)$).
First pull back $G(E_Q)$ to $S^3$, obtaining $(L\theta \circ LR_{S^4} \circ i_3)^{*} G(E_Q)$ using the maps $i_3$ of lemma \ref{l-incl-s3-ls4-cont}, $R_{S^4}$ of lemma \ref{l-loc-sect-lspe}, and $\theta$ of lemma \ref{l-diff-s4-hp1}.  Thence by proposition \ref{p-dd-bg-susp-c1-pb}, using the local sections from lemma \ref{l-loc-sect-lspe}, construct the principal $\UU(1)$ bundle $Q \rightarrow S^2$.
\end{lem}
\begin{proof}
Before the pullback, lemma \ref{l-eq-bg} constructs of $G(E_Q) = (P, Y, L\QH P^1)$.  Lemma \ref{l-loc-sect-lspe} details the composition $LR_{S^4} \circ i_3$.

To construct the local sections of the pullback $Y$ space over the two-set cover of $S^3 = \Sigma S^2$ used by proposition \ref{p-dd-bg-susp-c1-pb} to define the principal $\UU(1)$ bundle over $S^2$, note that the elements of this cover are subsets respectively of $S^3 \setminus \{ \pm \qi \}$.  Then use the restrictions of the local sections of the pullback $Y$ space induced by the local sections of $L\Symp(E_Q)$ of lemma \ref{l-loc-sect-lspe}.  In symbols, the pullback $Y$ space is $(L\theta \circ LR_{S^4} \circ i_3)^{*} (L\Symp(E) \cross_{L\Symp(1)} \Lagr_{res})$, which has local sections $t_{\qi} \colon x \mapsto (x, [\eta_{\qi} (x), L])$, $t_{-\qi} \colon x \mapsto (x, [\eta_{-\qi} (x), L])$, over $S^3 \setminus \{ \qi \}$, $S^3 \setminus \{ -\qi \}$, respectively.

Viewing $S^3$ as the unreduced suspension of the copy of $S^2 = \{ x \in S^3 | \pi_{\qi} (x) = 0 \}$, so that the suspension dimension is the $\qi$ dimension, we obtain $Q$ from proposition \ref{p-dd-bg-susp-c1-pb}.  \end{proof}

\backmatter

\bibliographystyle{nddiss2e}
\bibliography{sambler-diss}

\newpage
\phantomsection
\addcontentsline{toc}{chapter}{INDEX}
\singlespacing
\printindex
\normalspacing

\end{document}